\documentclass{report}
\usepackage{classicthesis}
\usepackage[T1]{fontenc}
\usepackage{amsmath}
\usepackage{amssymb}

\usepackage[all]{xy}

\usepackage{appendix}
\usepackage{standalone}%
\usepackage{amsmath}%
\usepackage{amssymb}
\usepackage{amsthm}
\newtheorem{thm}{Theorem}[section]
\newtheorem{lemma}[thm]{Lemma}
\newtheorem{cor}[thm]{Corollary}
\newtheorem{prop}[thm]{Proposition}
\newtheorem{remark}[thm]{Remark}
\newtheorem{defn}[thm]{Definition}
\newtheorem{notes}[thm]{Notes and remarks}
 \newtheorem{obs}[thm]{Observation}
 \newtheorem{example}[thm]{Example}
\theoremstyle{definition}
\newtheorem{spec}[thm]{Speculations}

 \numberwithin{equation}{section}

 \newcommand{\Span}{\operatorname{Span}}
 \newcommand{\Real}{\operatorname{Re}}
 \newcommand{\Imag}{\operatorname{Im}}
  
  \newcommand{\dist}{\operatorname{Dist}}
   \newcommand{\Tr}{\operatorname{Tr}}
   \newcommand{\tr}{\operatorname{T}}
  \newcommand{\Aut}{\operatorname{Aut}}  
      
   \newcommand{\supp}{\operatorname{supp}}  
    \newcommand{\co}{\operatorname{co}} 
    \newcommand{\Ad}{\operatorname{Ad}} 
     \newcommand{\KMS}{\operatorname{KMS}}
    
     \newcommand{\id}{\operatorname{id}}
      
\newcommand{\Aff}{\operatorname{Aff}}
\newcommand{\Hom}{\operatorname{Hom}}

\newcommand{\Sp}{\operatorname{Sp}}

\usepackage{mdframed,xcolor}

\begin{document}


 \begin{huge} 

\begin{center}

An introduction to KMS weights 
\end{center}
\end{huge}

\begin{center}

\end{center}

\bigskip
\bigskip

\bigskip
\begin{center}
\begin{large}
Klaus Erik Thomsen
\end{large}
\end{center}

\bigskip

\bigskip

\bigskip

\bigskip

\bigskip

\bigskip

\bigskip

\bigskip

\begin{center}
 \emph{Abstract}
\end{center} 
 
  The theory of KMS weights is based on a theorem of Combes and a theorem of Kustermans. In applications to KMS states for flows on a unital $C^*$-algebra the relation to KMS weights of the stabilized algebra has proved useful and this relation hinges on a theorem of Laca and Neshveyev. The first three chapters present proofs of these fundamental results that require a minimum of prerequisites; in particular, they do not depend on the modular theory of von Neumann algebras. In contrast, starting with chapter four the presented material draws heavily on the modular theory of von Neumann algebras. Most results are known from the work of N. V. Pedersen, J. Quaegebeur, J. Verding, J. Kustermans, S. Vaes, A. Kishimoto, A. Kumjian and J. Christensen, but new ones begin to surface. In chapter nine and the Appendices D and E the reader can find a presentation of results obtained recently by the author, partly in collaboration with G. A. Elliott and Y. Sato. This material is a natural culmination of methods developed around 1980 by Bratteli, Elliott, Herman and Kishimoto. Finally, in chapter ten there is a short presentation of the notion of factor types for KMS weights and states.

\newpage  

\begin{huge}
\begin{center}
Preface 
\end{center}
\end{huge}  

This book grew out of an attempt to write down the details of the arguments underlying the use of KMS weights in the study of KMS states for flows on $C^*$-algebras, and what has resulted is, I hope, a text which can serve as an introduction to both KMS weights and states. In the process I gradually found new proofs of some old theorems and discovered new facts which I think are noteworthy. The present text is therefore not only a relatively self-contained presentation of the theory of KMS weights, but contains also several new results. In the Notes and Remarks following many of the sections I explain what is what, and identify the sources of the material. The reason I emphasize that a part of the material presents new results is not that I want to 'plant my flag', but because I want to warn the reader that those parts have not gone through any peer review, and are therefore more likely to contain mistakes. While I welcome all comments and remarks on the content of this text, I am particularly interested in all mistakes, big or small, that the reader may find. Please make me aware of such observations by sending an e-mail to

\begin{center}
matkt@math.au.dk
\end{center}
or
\begin{center}
klausethomsen@gmail.com
\end{center}

\bigskip
 
\ \ \ \ \ \ \ \ \ \ \ \ \ \ \ \ \ \ \ \  \ \ \ \ \ \ \ \ \ \ \ \ \ \ \ \ \ \ \ \ \ \ \ \ \ \ \ \ \ \ \ \ \ \ \ \  \ \ \ \ \ \ Klaus Thomsen

\ \ \ \ \ \ \ \ \ \ \ \ \ \ \ \ \ \ \ \  \ \ \ \ \ \ \ \ \ \ \ \ \ \ \ \ \ \ \ \ \ \ \ \ \ \ \ \ \ \ Aarhus, Denmark, December 2023

  \tableofcontents



\chapter{Weights}

In this chapter we introduce weights on $C^*$-algebras and prove a fundamental theorem of Combes on which the theory is based.

\bigskip

Let $A$ be a $C^*$-algebra and let $A^+$ denote the cone of positive elements in $A$. A map $\psi : A^+ \to [0,\infty]$ is a \emph{semi-weight} on $A$ when
\begin{itemize}
\item[(a)] $\psi(a+b) \leq \psi(a) + \psi(b) \ \ \forall a,b \in A^+$,
\item[(b)] $a \leq b \Rightarrow \psi(a) \leq \psi(b) \ \ \forall a,b \in A^+$,
\item[(c)] $\psi( t a) = t \psi(a) \ \ \forall a \in A^+, \ \forall t \in \mathbb R^+$, with the convention that $0 \cdot \infty = 0$, and  
\item[(d)] $\psi$ is lower semi-continuous; i.e. $\{a \in A^+: \ \psi(a) > t\}$ is open in $A^+$ for all $t \in \mathbb R$.
\end{itemize}

It is easy to see that the lower semi-continuity of $\psi$ is equivalent to the following condition which is often useful and easier to verify:
\begin{itemize}
\item $\psi(a) \leq \liminf_n \psi(a_n)$ when $\lim_{n \to \infty}a_n = a$ in $A^+$.
\end{itemize}

A semi-weight $\psi$ is \emph{densely defined} when $\left\{ a \in A^+ :   \psi(a) < \infty\right\}$ is dense in $A^+$.

Given a semi-weight $\psi$ on $A$ we set
$$
\mathcal M^+_\psi := \left\{ a \in A^+: \ \psi(a) < \infty\right\},
$$ 
$$
\mathcal N_\psi := \left\{ a \in A: \ \psi(a^*a) < \infty\right\},
$$
and
$$
\mathcal M_\psi := \Span \mathcal M^+_\psi.
$$

\begin{lemma}\label{04-11-21n} Let $\psi$ be a semi-weight on $A$. The sets $\mathcal M_\psi$, $\mathcal M_\psi^+$ and $\mathcal N_\psi$ have the following properties.
\begin{itemize}
\item[(a)] $\mathcal N_\psi$ is a left-ideal in $A$, and it is dense in $A$ when $\psi$ is densely defined.
\item[(b)] $\mathcal M_\psi = \Span \mathcal N_\psi^* \mathcal N_\psi$.
\item[(c)] $\mathcal M_\psi \subseteq \mathcal N_\psi \cap \mathcal N_\psi^*$.
\item[(d)]  $\mathcal M_\psi \cap A^+ = \mathcal M_\psi^+$.
\item[(e)] $\mathcal M_\psi$ is a $*$-subalgebra of $A$, and it is dense in $A$ when $\psi$ is densely defined.
\end{itemize}
\end{lemma}
\begin{proof} For (a) note that $(\lambda a + \mu b)^*(\lambda a + \mu b) \leq 2|\lambda|^2a^*a + 2 |\mu|^2b^*b$ for all $a,b \in A$ and all $\lambda, \mu \in \mathbb C$. This shows that $\mathcal N_\psi$ is a subspace of $A$. Since $a^2 \leq \|a\|a$ for all $a \in A^+$ it follows that $\mathcal M_\psi^+ \subseteq \mathcal N_\psi$. In particular, $\mathcal N_\psi$ is dense in $A$ when $\psi$ is densely defined. $\mathcal N_\psi$ is a left ideal because $b^*a^*ab \leq \|a\|^2b^*b$ for all $a,b \in A$. For (b) note that the inclusion $\mathcal M_\psi \subseteq \Span \mathcal N_\psi^* \mathcal N_\psi$ follows from the observation that $\mathcal M^+_\psi \subseteq \mathcal N_\psi^* \mathcal N_\psi$ since $\sqrt{a} \in \mathcal N_\psi$ when $a \in \mathcal M_\psi^+$. The reverse inclusion follows from (a) and the polarisation identity
\begin{equation}\label{polar}
b^*a = \frac{1}{4} \sum_{k=1}^4 i^k (a+ i^kb)^*(a+i^kb).
\end{equation}
(c) follows from (a) and (b) because $\mathcal N_\psi^* \mathcal N_\psi \subseteq \mathcal N_\psi^* \cap \mathcal N_\psi$ since $\mathcal N_\psi$ is a left ideal and $\mathcal N_\psi^*$ is a right ideal. For (d), let $a \in \mathcal M_\psi \cap A^+$. Since $a =a^*$ and $a \in \Span \mathcal M_\psi^+$ it follows that $a = a_+-a_-$, where $a_\pm \in \mathcal M^+_\psi$. Then $a\leq a_+$ and $\psi(a) \leq \psi(a_+) < \infty$. This shows that $\mathcal M_\psi \cap A^+ \subseteq \mathcal M^+_\psi$, and the reverse inclusion is trivial. For (e) note that it follows from the definition that $\mathcal M_\psi^* = \mathcal M_\psi$. Let $a,b \in \mathcal M^+_\psi$. Since $a^2 \leq \|a\|a$, it follows that $a^2 \in \mathcal M^+_\psi$, and hence that $a \in \mathcal N_\psi$. Similarly, $b \in \mathcal N_\psi$ and it follows from (b) that $ab \in \mathcal M_\psi$. Hence $\mathcal M_\psi \mathcal M_\psi \subseteq \mathcal M_\psi$ and we conclude that $\mathcal M_\psi$ is a $*$-subalgebra of $A$. Since every element of $A$ is a linear combination of elements from $A^+$ it follows that $\mathcal M_\psi$ is dense in $A$ when $\psi$ is densely defined. 
\end{proof}

\section{Combes' theorem}

In the following we denote by $A^*_+$ the set of bounded positive linear functionals on $A$. If $\mathcal F$ is a subset of $A^*_+$ we can define a semi-weight by the formula $\sup_{\omega \in \mathcal F}\omega(a)$. The fundamental result on semi-weights is the following theorem of Combes, showing that all semi-weights arise like this.

\begin{thm}\label{04-11-21k} Let $\psi$ be a semi-weight on $A$. Set 
$$
\mathcal F_\psi := \left\{ \omega \in A^*_+ : \ \omega(a) \leq \psi(a) \ \ \forall a \in A^+\right\}.
$$
Then 
$$
\psi(a) = \sup_{\omega \in \mathcal F_\psi} \omega(a)
$$
for all $a \in A^+$.
\end{thm}

The proof of Combes' theorem requires some preparation. We denote by $A^\dagger$ the $C^*$-algebra obtained by adjoining a unit to $A$. In particular, $A^\dagger = A \oplus \mathbb C$ when $A$ already has a unit. For $\alpha > 0$ and $a \in A^+_1$, set
$$
\rho_\alpha(a) := \inf \left\{ \psi(s) + t\alpha : \ s \in \mathcal M^+_\psi , \ t \in \mathbb R^+, \  a \leq s+t1\right\}.
$$
The following properties of $\rho_\alpha: (A^\dagger)^+ \to \mathbb R^+$ are straightforward to establish.
\begin{equation}\label{04-11-21}
\rho_\alpha(a+b) \leq \rho_\alpha(a) + \rho_\alpha(b) \ \ \forall a,b \in (A^\dagger)^+.
\end{equation}
\begin{equation}\label{04-11-21a}
\rho_\alpha(\lambda a) = \lambda \rho_\alpha(a) \ \ \forall a\in (A^\dagger)^+, \ \forall \lambda \in \mathbb R^+.
\end{equation}
\begin{equation}\label{04-11-21b}
\rho_\alpha(a) \leq \psi(a) \ \ \forall a\in A^+.
\end{equation}
\begin{equation}\label{04-11-21c}
a,b \in (A^\dagger)^+, \ a \leq b \Rightarrow \rho_\alpha(a) \leq \rho_\alpha(b).
\end{equation}
\begin{equation}\label{04-11-21f}
\rho_\alpha(a) \leq \alpha \|a\| \ \ \forall a \in (A^\dagger)^+.
\end{equation}
\begin{equation}\label{04-11-21d}
\rho_\alpha(1) \leq \alpha .
\end{equation}

\begin{lemma}\label{04-11-21g}
\begin{equation*}\label{04-11-21e}
\rho_\alpha( a + \lambda 1) =  \rho_\alpha(a) + \lambda\alpha \ \ \forall a \in (A^\dagger)^+, \ \forall \lambda \in \mathbb R^+.
\end{equation*}
\end{lemma}
\begin{proof} This follows from the equality
\begin{equation}\label{05-11-21c}
\begin{split}
&
\left\{\psi(s) + t\alpha : s \in \mathcal M_\psi^+, \ t \in \mathbb R^+ , \ a+ \lambda 1 \leq s+ t1\right\}\\
& = \left\{\psi(s) + t\alpha : s \in \mathcal M_\psi^+, \ t \in \mathbb R^+ , \ a \leq s+ t1\right\} \ + \ \lambda \alpha .
\end{split}
\end{equation}
To establish this equality, let $ s \in \mathcal M_\psi^+, \ t \in \mathbb R^+ , \ a+ \lambda 1 \leq s+ t1$. Then $a \leq s + (t-\lambda)1$, and if we let $\chi : A^\dagger \to \mathbb C$ be the canonical character with $\ker \chi = A$, we conclude that $t-\lambda = \chi(s+(t-\lambda) 1) \geq \chi(a) \geq 0$. Hence
$$
\psi(s) + (t-\lambda)\alpha \in   \left\{\psi(s) + t'\alpha : s \in \mathcal M_\psi^+, \ t' \in \mathbb R^+ , \ a \leq s+ t'1\right\}.
$$
This proves that the first set in \eqref{05-11-21c} is contained in the second. The reverse inclusion is trivial.
\end{proof}

\begin{lemma}\label{04-11-21i} 
$$
\sup_{\alpha > 0} \rho_\alpha(a) = \psi(a) \  \ \forall a \in A^+.
$$
\end{lemma}
\begin{proof} Let $a \in A^+$. It follows from \eqref{04-11-21b} that $\sup_{\alpha > 0} \rho_\alpha(a) \leq \psi(a)$. Assume for a contradiction that $\sup_{\alpha > 0} \rho_\alpha(a) < \psi(a)$. There is then a $k > 0$ such that $\rho_\alpha(a) < k < \psi(a)$ for all $\alpha > 0$. Applied with $\alpha = kn$ we get for each $n \in \mathbb N$ an element $s_n \in \mathcal M^+_\psi$ and a $t_n \in \mathbb R^+$ such that $a \leq s_n + t_n 1$ and $\psi(s_n) + t_n kn < k$. In particular, $t_n \leq \frac{1}{n}$. Set
$$
b_n := a^{\frac{1}{2}} - a^{\frac{1}{2}}\left( \frac{1}{n} + s_n\right)^{-1}s_n = \frac{1}{n}a^{\frac{1}{2}} \left( \frac{1}{n} + s_n\right)^{-1} .
$$
Then
\begin{align*}
& b_n^* b_n = \frac{1}{n^2} \left( \frac{1}{n} + s_n\right)^{-1}a \left( \frac{1}{n} + s_n\right)^{-1} \\
& \leq \frac{1}{n^2} \left( \frac{1}{n} + s_n\right)^{-1}\left( s_n + \frac{1}{n} \right)\left( \frac{1}{n} + s_n\right)^{-1} = \frac{1}{n^2} \left( \frac{1}{n} + s_n\right)^{-1} \leq \frac{1}{n},
\end{align*}
proving that $\lim_{n \to \infty} b_n = 0$. Thus $a = \lim_{n \to \infty} c_n$ where
$$
c_n := s_n \left( \frac{1}{n} + s_n\right)^{-1}a\left( \frac{1}{n} + s_n\right)^{-1}s_n.
$$
Note that
\begin{align*}
&  c_n \leq  s_n \left( \frac{1}{n} + s_n\right)^{-1}\left(s_n + \frac{1}{n} \right)\left( \frac{1}{n} + s_n\right)^{-1}s_n\\
& =  s_n\left( \frac{1}{n} + s_n\right)^{-1}s_n \leq s_n,
\end{align*}
implying that $\psi(c_n) \leq \psi(s_n) < k$. Since $\psi$ is lower semi-continuous it follows that $\psi(a) \leq \liminf_n \psi(c_n)$ and we conclude therefore that $\psi(a) \leq k$ which contradicts that $k < \psi(a)$.
\end{proof}

 Let $A^\dagger_{sa}$ denote the set of self-adjoint elements of $A^\dagger$ and define $\rho'_\alpha : A^\dagger_{sa} \to \mathbb R^+$ by
$$
\rho'_\alpha(a) := \inf \left\{\rho_\alpha(b+c) : \ b,c \in (A^\dagger)^+, \ a = b-c\right\}.
$$

\begin{lemma}\label{04-11-21j} $\rho'_\alpha$ has the following properties.
\begin{itemize}
\item $\rho'_\alpha(a+b) \leq \rho'_\alpha(a) + \rho'_\alpha(b) \  \ \forall a,b \in A^\dagger_{sa}$,
\item $\rho'_\alpha(\lambda a) = |\lambda| \rho'_\alpha(a) \  \ \forall a \in  A^\dagger_{sa}, \ \forall \lambda \in \mathbb R$,
\item $\rho'_\alpha(a) \leq \alpha \|a\| \ \ \forall a \in  A^\dagger_{sa}$, and
\item $\rho'_\alpha(a) = \rho_\alpha(a) \ \ \forall a \in  (A^\dagger)^+$.
\end{itemize}
\end{lemma}
\begin{proof}  The first two items follow easily from \eqref{04-11-21} and \eqref{04-11-21a}. Let $a \in A^\dagger_{sa}$. Recall that there is a (unique) decomposition $a = a_+ -a_-$ where $a_\pm \in (A^\dagger)^+$ and $a_+a_- = 0$. Since $\left\|a_++a_-\right\| =\|a\|$ we get the third item from \eqref{04-11-21f}. If $a \in (A^\dagger)^+$ and $a = b -c$ with $b,c \in (A^\dagger)^+$ we have that $a \leq b+c$ and hence $\rho_\alpha(a) \leq \rho_\alpha(b+c)$ thanks to \eqref{04-11-21c}. This shows that $\rho_\alpha(a) \leq \rho'_\alpha(a)$ in this case. The reverse inequality is trivial. 
\end{proof}

\emph{Proof of Theorem \ref{04-11-21k}:} Let $a \in A^+$ and $\alpha > 0$. By Lemma \ref{04-11-21i} it suffices to find $\omega \in \mathcal F_\psi$ such that $\omega(a) = \rho_\alpha(a)$. Let $V$ be the real vector subspace of $A^\dagger_{sa}$ generated by $a$ and $1$. We can then define a linear map $\omega : V \to \mathbb R$ such that 
$$
\omega(s a + t1) = s\rho_\alpha(a) + t\rho_\alpha(1)
$$ 
for all $s,t \in \mathbb R$. We claim that
 \begin{equation}\label{04-11-21l}
 \left|\omega(x)\right| \leq \rho'_\alpha(x) \ \ \forall x \in V .
 \end{equation}
Assume that $t \geq 0$. Then \eqref{04-11-21d} and Lemma \ref{04-11-21g} imply that
\begin{align*}
& \left|\omega(a + t 1)\right| =\rho_\alpha(a) + t\rho_\alpha(1) \leq \rho_\alpha(a) + t\alpha = \rho_\alpha(a + t1).
\end{align*}
Since $\rho_\alpha(a + t1) = \rho'_\alpha(a + t1)$ by the last item in Lemma \ref{04-11-21j} we have established \eqref{04-11-21l} when $x = a + t1$. Assume that $t \leq 0$. Using the fourth and the second item in Lemma \ref{04-11-21j} we find that
\begin{equation}\label{04-11-21m}
\omega(a +t1) = \rho_\alpha(a) + t\rho_\alpha(1) = \rho_\alpha(a) -  \rho'_\alpha(t1) .
\end{equation}
The first two items in Lemma \ref{04-11-21j} imply that
$$
\rho'_\alpha(t1) = \rho'_\alpha(-t1) = \rho'_\alpha(a - (a+t1)) \leq \rho'_\alpha(a) + \rho'_\alpha(a+t1) 
$$
and, since $\rho_\alpha(a) = \rho'_\alpha(a)$,
$$
\rho_\alpha(a) = \rho'_\alpha( a + t1 - t1) \leq \rho'_\alpha(a + t1) + \rho'_\alpha(t1).
$$
Thus
$$
-\rho'_\alpha(a+t1) \leq \rho_\alpha(a) -  \rho'_\alpha(t1) \leq \rho'_\alpha(a+t1) ,
$$ 
which combined with \eqref{04-11-21m} shows that \eqref{04-11-21l} holds when $x \in a + \mathbb R1$. From the second and the last items in Lemma \ref{04-11-21j} we see that \eqref{04-11-21l} holds when $x \in \mathbb R 1$. Finally, when $s \neq 0$ we can now conclude that
$$
|\omega(sa + t1) = |s|\left|\omega(a + \frac{t}{s}1)\right| \leq |s| \rho'_\alpha(a + \frac{t}{s}1) = \rho'_\alpha(sa +t1).
$$
Thus \eqref{04-11-21l} holds for all $x \in V$ and it follows from the Hahn-Banach extension theorem that there a linear map $\omega : A^\dagger \to \mathbb C$ extending $\omega : V \to \mathbb R$ such that $\omega(x^*) = \overline{\omega(x)}$ for all $x \in A^\dagger$, i.e. $\omega$ is hermitian, and $|\omega(y)| \leq \rho'_\alpha(y)$ for all $y \in A^\dagger_{sa}$. Using that $\omega$ is hermitian together with the third item of Lemma \ref{04-11-21j} we find that
\begin{align*}
&\left\|\omega\right\| = \sup \left\{|\omega(y)|: \ y \in A^\dagger_{sa}, \ \|y\| \leq 1 \right\} 
\\
&\leq \sup \left\{\rho'_\alpha(y): \ y \in A^\dagger_{sa} \ \|y\| \leq 1 \right\} \leq \alpha .
\end{align*}
Since $\omega(1) = \rho_\alpha(1) = \alpha$, it follows that $\omega|_A \in A^*_+$. Furthermore, when $x \in A^+$, we find from \eqref{04-11-21b} that $\omega(x) \leq \rho'_\alpha(x) = \rho_\alpha(x) \leq \psi(x)$. Thus $\omega \in \mathcal F_\psi$. Since $\omega(a) = \rho_\alpha(a)$, this completes the proof.
\qed

\begin{notes}\label{16-12-21} Theorem \ref{04-11-21k} was obtained by Combes in \cite{C1}; see Lemme 1.5 and Remarque 1.6 in \cite{C1}. An alternative proof was given by Haagerup in \cite{Ha}; see Proposition 2.1 and Corollary 2.3 in \cite{Ha}. 
\end{notes}

\section{Combes' theorem in the separable case}\label{defweight}

Let $A$ be a $C^*$-algebra. A map $\psi : A^+ \to [0,\infty]$ is a \emph{weight} on $A$ when
\begin{itemize}
\item $\psi(a+b) = \psi(a) + \psi(b) \ \ \forall a,b \in A^+$,
\item $\psi( t a) = t \psi(a) \ \ \forall a \in A^+, \ \forall t \in \mathbb R^+$, using the convention $0 \cdot \infty = 0$, and 
\item $\psi$ is lower semi-continuous; i.e. $\{a \in A^+: \ \psi(a) > t\}$ is open in $A^+$ for all $t \in \mathbb R$.
\end{itemize}

Any sequence $\{\omega_n\}_{n=1}^\infty$ in $A^*_+$ defines a weight by the formula $\sum_{n=1}^\infty \omega_n(a)$. The content of the next theorem is that they all arise this way when $A$ is separable.

\begin{thm}\label{09-11-21h} Assume that $A$ is separable and let $\psi$ be a weight on $A$. There is a sequence $\omega_n\in A^*_+$, $n = 1,2,\cdots$, such that
\begin{equation}\label{10-11-21x}
\psi(a) = \sum_{n=1}^\infty \omega_n(a)  \ \ \forall a \in A^+
\end{equation}
\end{thm}

The proof requires some preparations starting with the following continuation of the proof of Theorem \ref{04-11-21k}.

\begin{lemma}\label{08-11-21} Assume that $A$ is separable and that $\psi$ is a semi-weight on $A$. There is a countable set $\mathcal C \subseteq \mathcal F_\psi$ with the following property: For all $a \in A^+$ and all $n \in \mathbb N$ there is an element $f \in \mathcal C$ such that $f(a) >  \min \{n ,\psi(a) - \frac{1}{n}\}$.
\end{lemma}
\begin{proof} The function $\alpha \mapsto \rho_\alpha(a)$ which was used in the proof of Theorem \ref{04-11-21k} is non-decreasing for $a \in A^+$. It follows therefore from Lemma \ref{04-11-21i} that $\lim_{k \to \infty} \rho_k(a) = \psi(a)$ for all $a \in A^+$. Let $\{a_i\}_{i=1}^\infty$ be a dense sequence in $A^+$. As shown in the proof of Theorem \ref{04-11-21k} there is for each pair $k,i \in \mathbb N$ an element $\omega_{k,i} \in \mathcal F_\psi$ such that $\omega_{k,i}(a_i) = \rho_k(a_i)$ and $\left|\omega_{k,i}(x)\right| \leq \|x\|k$ for all $x \in A$. Set
$$
\mathcal C := \left\{ \omega_{k,i} : \ k,i \in \mathbb N \right\} .
$$
To check that $\mathcal C$ has the required properties, let $a \in A^+$ and $n \in \mathbb N$ be given. Since $\lim_{k \to \infty} \rho_k(a) = \psi(a)$ there is a $k \in \mathbb N$ such that
$$
\rho_k(a) > \min \left\{ n+1, \psi(a) - \frac{1}{2n}\right\}.
$$
Choose $i \in \mathbb N$ such that $\left\|a-a_i\right\| \leq \frac{1}{4k n}$. Using the properties of $\rho_k$ and $\rho'_k$ we find 
\begin{align*}
&\left|\rho_k(a) - \omega_{k,i}(a)\right| \leq \left|\rho_k(a) - \rho_k(a_i)\right| + \left|\rho_k(a_i) - \omega_{k,i}(a)\right| \\
& \leq \rho'_k(a-a_i) + \left|\omega_{k,i}(a_i) - \omega_{k,i}(a)\right| \leq 2 k \left\|a-a_i\right\| \leq \frac{1}{2n}. 
\end{align*}
Then
\begin{align*}
&\omega_{k,i}(a) \geq \rho_k(a)-\frac{1}{2n} > \min \left\{ n+1, \psi(a) - \frac{1}{2n}\right\}-\frac{1}{2n}\\
& \geq \min \left\{ n, \psi(a) - \frac{1}{n}\right\}.
\end{align*}
\end{proof}

 Note that a weight is also a semi-weight. In the following $\psi : A \to [0,\infty]$ will be a weight on the $C^*$-algebra $A$.

\subsection{The GNS construction for weights}\label{06-02-22}

Let $A$ be a $C^*$-algebra.

\begin{defn}\label{06-02-22a} A \emph{$GNS$ representation} $(H,\Lambda,\pi)$ of $A$ consists of a Hilbert space $H$, a linear map $\Lambda: D(\Lambda) \to H$ defined on a dense left ideal $D(\Lambda) \subseteq A$ and a representation $\pi : A \to B(H)$ of $A$ by bounded operators on $H$ such that
\begin{itemize}
\item[(a)] $\Lambda(D(\Lambda))$ is dense in $H$, and
\item[(b)] $\pi(a)\Lambda(b) = \Lambda(ab)$ for all $a \in A$ and $b \in D(\Lambda)$.
\end{itemize}
\end{defn}

A densely defined weight $\psi : A^+ \to [0,\infty]$ gives rise to a $GNS$ representation of $A$ in a way we now describe.

\begin{lemma}\label{10-11-21} Let $\psi : A \to [0,\infty]$ be a weight. There is a unique linear map $\psi : \mathcal M_\psi \to \mathbb C$ extending $\psi : \mathcal M^+_\psi \to \mathbb R^+$.
\end{lemma}
\begin{proof} An element $a \in \mathcal M_\psi$ can be decomposed as a sum $a = a_1 -a_2 + i(a_3 -a_4)$ where $a_i \in \mathcal M^+_\psi$, and we are forced to define the extension such that 
$$
\psi(a) = \psi(a_1) - \psi(a_2) + i(\psi(a_3)-\psi(a_4)) .
$$
This is well-defined and gives a linear map because $\psi$ is additive and homogeneous on $\mathcal M^+_\psi$.
\end{proof}

It follows from Lemma \ref{04-11-21n} and Lemma \ref{10-11-21} that we can define a sesquilinear form
$$
\left< \cdot \ , \ \cdot \right> : \mathcal N_\psi \times \mathcal N_\psi \to \mathbb C
$$
such that
$$
\left<a,b\right> := \psi(b^*a).
$$
Set 
$$
\ker_\psi := \left\{ a \in \mathcal N_\psi : \psi(a^*a) = 0\right\}.
$$
Arguments from the proof of Lemma \ref{04-11-21n} show that $\ker \psi$ is a left-ideal in $A$.
We let $H_\psi$ denote the Hilbert space obtained by completing the quotient $\mathcal N_\psi/\ker_\psi$ with respect to the norm on $\mathcal N_\psi/\ker_\psi$ induced by $\left< \cdot \ , \ \cdot \right>$. Let $\Lambda_\psi : \mathcal N_\psi \to H_\psi$ be the quotient map coming with this construction. We can then define a representation $\pi_\psi : A \to B(H_\psi)$ of $A$ by bounded operators such that 
$$
\pi_\psi(a) \Lambda_\psi(b) = \Lambda_\psi (ab)
$$
for all $b\in \mathcal N_\psi$. When $\psi$ is densely defined, $\mathcal N_\psi$ is a dense left ideal in $A$ by (a) of Lemma \ref{04-11-21n}. By construction $(H_\psi,\Lambda_\psi,\pi_\psi)$ is then a GNS representation of $A$. We shall refer to $(H_\psi,\Lambda_\psi,\pi_\psi)$ as the \emph{GNS-triple of $\psi$.}

\subsection{The weight of a GNS representation}\label{GNStriple}

Let $(H,\Lambda,\pi)$ be a GNS representation of $A$. Set
$$
\mathcal F_\Lambda := \left\{\omega\in A^*_+ : \ \omega(a^*a) \leq \left<\Lambda(a),\Lambda(a)\right> \ \ \forall a \in D(\Lambda)\right\}.
$$

We denote by $\pi(A)'$ the commutant of $\pi(A)$ in $B(H_\psi)$.

\begin{lemma}\label{08-11-21bx} For every $\omega \in \mathcal F_\Lambda$ there is an operator $T_\omega \in \pi(A)'$ such that $0 \leq T_\omega \leq 1$ and
\begin{equation}\label{08-11-21cx}
\omega(b^*a) = \left< T_\omega \Lambda(a), \Lambda(b)\right> \ \ \forall a,b \in D(\Lambda).
\end{equation}
\end{lemma}
\begin{proof} Since $\omega \in \mathcal F_\Lambda$, the Cauchy-Schwarz inequality gives the estimate
\begin{equation}\label{08-11-21dx}
\omega(b^*a) \leq \left\|\Lambda(a)\right\|\left\|\Lambda(b)\right\| \ \ \forall a,b \in D(\Lambda).
\end{equation}
It follows that there is a bounded sesquilinear form $\left< \cdot, \ \cdot \right>_\omega$ on $H$ such that
$$
\left< \Lambda(a), \Lambda(b)\right>_{\omega} = \omega(b^*a)
$$
for all $a,b \in D(\Lambda)$. There is therefore a bounded operator $T_\omega \in B(H)$ such that \eqref{08-11-21cx} holds. (See for example Theorem 2.4.1 in \cite{KR}.) It follows from \eqref{08-11-21dx} that $\left\|T_\omega \right\| \leq 1$ and from the positivity of $\omega$ that $T_\omega \geq 0$. Finally, since 
\begin{align*}
&\left< \pi(x)T_\omega \Lambda_\psi(a),\Lambda(b)\right> = \left< T_\omega \Lambda(a),\pi(x)^* \Lambda(b)\right> = \left< T_\omega \Lambda(a), \Lambda(x^*b)\right> \\
&=\omega(b^*xa) = \left<T_\omega \Lambda(xa),\Lambda(b)\right> = \left< T_\omega \pi(x)\Lambda(a),\Lambda(b)\right>
\end{align*}
for all $x \in A, \ a,b \in D(\Lambda)$, it follows that $T_\omega \in \pi(A)'$.
\end{proof}

\begin{lemma}\label{08-11-21e} Let $B$ be a $C^*$-algebra and $0 \leq T_i \leq 1, \ i = 1,2$, elements of $B$. Let $\lambda_i \in ]0,1[, \ i = 1,2$. There is an element $0 \leq T \leq 1$ in $B$ and numbers $\lambda \in ]0,1[$ and $r > 0$ such that $\lambda_iT_i \leq \lambda T, \ i = 1,2$, and $T \leq r(T_1+T_2)$.
\end{lemma}
\begin{proof} Choose $\max\{\lambda_1,\lambda_2\} < \gamma < 1$. Set 
$$
S_i := \left(1- \gamma T_i\right)^{-1}\gamma T_i, \ \ i = 1,2,
$$
and
$$
T := (1 +S_1 +S_2)^{-1}(S_1+S_2) \in B .
$$
The function $t \mapsto \frac{t}{1+t} = 1 - (1+t)^{-1}$ is operator monotone by Proposition 1.3.6 in \cite{Pe} or Proposition 2.2.13 in \cite{BR}, and hence 
$$
T \geq (1+S_i)^{-1}S_i = \gamma T_i , \ \ i =1,2.
$$
Put $\lambda := \max \left\{\frac{\lambda_1}{\gamma}, \frac{\lambda_2}{\gamma}\right\} \in ]0,1[$ and note that $\lambda T \geq \lambda \gamma T_i \geq \lambda_i T_i, \ i =1,2$.
On the other hand, it follows directly from the definitions that 
$$
T \leq S_1+S_2 \leq \frac{\gamma}{1 - \gamma} (T_1+ T_2).
$$
Set $r =  \frac{\gamma}{1-\gamma}$.
\end{proof}

\begin{lemma}\label{09-11-21ax} Let $(H,\Lambda,\pi)$ be a GNS representation of $A$. Let $\omega_1,\omega_2 \in \mathcal F_\Lambda$ and $\lambda_1,\lambda_2 \in ]0,1[$ be given. There is an $\omega \in \mathcal F_\Lambda$ and a $\lambda \in ]0,1[$ such that $\lambda_i \omega_i \leq \lambda \omega$ in $A^*_+$, $i =1,2$.
\end{lemma}
\begin{proof} By Lemma \ref{08-11-21bx} there are operators $T_{\omega_i} \in \pi(A)'$ such that $0 \leq T_{\omega_i} \leq 1$ and $\omega_i(b^*a ) = \left<T_{\omega_i}\Lambda(a),\Lambda(b)\right>, \ \forall a,b \in D(\Lambda)$, $i =1,2$. 
From Lemma \ref{08-11-21e} we get an operator $0 \leq T \leq 1$ in $\pi(A)'$ and numbers $\lambda \in ]0,1[$ and $r > 0$ such that $\lambda_i T_{\omega_i} \leq \lambda T, \ i =1,2$, and $T \leq W$ where $W :=r(T_{\omega_1} + T_{\omega_2})$. We define a sesquilinear form $s : D(\Lambda) \times D(\Lambda) \to \mathbb C$ such that
$$
s(a,b) := \left<T\Lambda(a),\Lambda(b)\right> .
$$
Note that
\begin{equation}\label{09-11-21g}
\begin{split}
& |s(a,b)|  = \left|\left<T^{\frac{1}{2}}\Lambda(a),T^{\frac{1}{2}}\Lambda(b)\right>\right|  
 \leq \left\|T^{\frac{1}{2}}\Lambda(a)\right\|\left\|T^{\frac{1}{2}}\Lambda(b)\right\| \\
&= \left<T\Lambda(a),\Lambda(a)\right>^{\frac{1}{2}}\left<T\Lambda(b),\Lambda(b)\right>^{\frac{1}{2}}  \\
&\leq  \left<W\Lambda(a),\Lambda(a)\right>^{\frac{1}{2}}\left<W\Lambda(b),\Lambda(b)\right>^{\frac{1}{2}} \\
&= r\sqrt{(\omega_1+\omega_2)(a^*a)(\omega_1+\omega_2)(b^*b)}\leq r\left\|\omega_1 +\omega_2\right\|\|a\|\|b\|.
\end{split}
\end{equation} 
It follows that $s$ extends by continuity to a sesquilinear form $s : A \times A \to \mathbb C$ such that the estimate
\begin{equation}\label{09-11-21f}
\left|s(a,b)\right| \leq r\left\|\omega_1 +\omega_2\right\|\|a\|\|b\|
\end{equation}
holds for all $a,b \in A$. Using that $T \in \pi(A)'$ we find 
\begin{align*}
&s(xa,b) = \left<T\Lambda(xa),\Lambda(b)\right> = \left<T\pi(x)\Lambda(a),\Lambda(b)\right> \\
&= \left<\pi(x)T\Lambda(a),\Lambda(b)\right> =\left<T\Lambda(a),\Lambda(x^*b)\right> = s(a,x^*b)
\end{align*}
when $x \in A, \ a,b \in D(\Lambda)$. It follows by continuity that $s(xa,b) = s(a,x^*b)$ for all $x,a,b \in A$. Let $\{v_i\}_{i \in I}$ be a net in $A$ such that $0 \leq v_i \leq v_j \leq 1$ in $A$ when $i \leq j$ and such that $\lim_{i \to \infty} v_ia = a$ for all $a \in A$. It follows from Proposition 2.3.11 (d) of \cite{BR} that $\lim_{i \to \infty} r(\omega_1 + \omega_2)(v_i) = r\left\|\omega_1 + \omega_2\right\|$. Let $a \in A$. We can then choose a sequence $\{i_n\}_{n=1}^\infty$ in $I$ such that
\begin{itemize}
\item[(i)] $i_n \leq i_{n+1}$ for all $n$,
\item[(ii)] $r(\omega_1 + \omega_2)(v_{i}) \geq  r\left\|\omega_1 + \omega_2\right\| - \frac{1}{n} \ \ \forall i \geq i_n$, and
\item[(iii)] $\left\| v_ia - a \right\| \leq \frac{1}{n} \ \ \forall i \geq i_n$.
\end{itemize}
The estimate
$$
|s(a,b)| \leq  r\sqrt{(\omega_1+\omega_2)(a^*a)}\sqrt{(\omega_1+\omega_2)(b^*b)},
$$
which was established in \eqref{09-11-21g} when $a,b \in D(\Lambda)$, extends by continuity to all $a,b \in A$. It follows that for $i \geq j$ in $I$,
\begin{equation}\label{06-02-22b}
\begin{split}
&|s(a,v_{i}) - s(a,v_{j})| \leq  r\sqrt{(\omega_1+\omega_2)(a^*a)}\sqrt{(\omega_1+\omega_2)((v_{i}-v_{j})^2)} \\
&\leq r\sqrt{(\omega_1+\omega_2)(a^*a)}\sqrt{(\omega_1+\omega_2)(v_{i}-v_{j})} .
\end{split}
\end{equation}
Since $\lim_{n \to \infty}  r(\omega_1 + \omega_2)(v_{i_n}) = r\left\|\omega_1 + \omega_2\right\|$ this estimate implies that $\left\{s(a,v_{i_n})\right\}$ is a Cauchy sequence and we set
\begin{equation}\label{06-02-22d}
\omega(a) := \lim_{n \to \infty} s(a,v_{i_n}) .
\end{equation}
To see that the limit \eqref{06-02-22d} is independent of the choice of sequence $\{i_n\}$ satisfying (i), (ii) and (iii), let $\{j_n\}_{n=1}^\infty$ be another sequence in $I$ with these three properties. Since $I$ is directed we can then find a sequence $\{k_n\}_{n=1}^\infty$ in $I$ such that $i_n \leq k_n$ and $j_n\leq k_n$ in $I$. Inserting $k_n$ for $i$, and first $i_n$ and next $j_n$ for $j$ in \eqref{06-02-22b} we find that 
$$
\lim_{n \to \infty} s(a,v_{i_n}) = \lim_{n \to \infty} s(a,v_{k_n})  = \lim_{n \to \infty} s(a,v_{j_n}).
$$
Thus $\omega(a)$ does not depend on which sequence $\{i_n\}$ in $I$ with the properties (i), (ii) and (iii) we use. It follows therefore that \eqref{06-02-22d} defines a linear functional $\omega : A \to \mathbb C$ and we conclude from \eqref{09-11-21f} that $\left\|\omega\right\| \leq r\left\|\omega_1 + \omega_2\right\|$. Note that for any $a \in A$ we can find a sequence $\{i_n\}$ in $I$ such that
$$
\omega(a^*a) = \lim_{n \to \infty} s(a^*a,v_{i_n}) = \lim_{n \to \infty} s(a,av_{i_n}) = s(a,a).
$$
When $a \in D(\Lambda)$ we find that
\begin{align*}
&\lambda \omega(a^*a) = \lambda s(a,a) = \lambda \left< T\Lambda(a),\Lambda(a)\right> \\
&\geq \lambda_i \left< T_{\omega_i}\Lambda(a),\Lambda(a)\right> = \lambda_i \omega_i(a^*a), \ i =1,2,
\end{align*}
and
$$
\omega(a^*a) = s(a,a) = \left< T\Lambda(a),\Lambda(a)\right>  \leq \left< \Lambda(a),\Lambda(a)\right> .
$$
The first inequality implies that $\omega \in A^*_+$ and $\lambda \omega \geq \lambda_i \omega_i, \ i =1,2$, and the last inequality implies that $\omega \in \mathcal F_\Lambda$.
 
\end{proof}

Define ${\phi} : A^+ \to [0,\infty]$ by
$$
{\phi}(a) := \sup_{\omega \in \mathcal F_\Lambda} \omega(a).
$$

\begin{lemma}\label{30-01-22a} ${\phi}$ is a weight.
\end{lemma}
\begin{proof} It is clear that ${\phi}$ is a semi-weight but it remains to show that ${\phi}$ is additive. Let $a,b \in A^+$. Then ${\phi}(a+b) \leq {\phi}(a) + {\phi}(b)$ holds by definition of ${\phi}$. To conclude that $$
{\phi}(a) + {\phi}(b) \leq {\phi}(a+b)
$$ 
we may assume that ${\phi}(a+b) < \infty$, and it follows then that ${\phi}(a) < \infty$ and ${\phi}(b) < \infty$. Let $\epsilon > 0$. There are $\omega_1,\omega_2 \in \mathcal F_\Lambda$ such that $\omega_1(a) >{\phi}(a) - \epsilon$ and $\omega_2(b) > {\phi}(b) -\epsilon$. There is also a number $\lambda \in ]0,1[$ such that $\lambda \omega_1(a) > {\phi}(a) - \epsilon$ and $\lambda \omega_2(b) > {\phi}(b) -\epsilon$. By Lemma \ref{09-11-21ax} there is an element $\omega \in \mathcal F_\Lambda$ such that $\lambda \omega_i \leq \omega,  \ i = 1,2$. It follows that ${\phi}(a+b) \geq \omega(a+b) \geq \lambda \omega_1(a) + \lambda \omega_2(b) \geq {\phi}(a) + {\phi}(b) - 2\epsilon$. This shows that ${\phi}(a+b) \geq {\phi}(a) + {\phi}(b)$, and we conclude that ${\phi}$ is a weight. 
\end{proof}

 We will refer to ${\phi}$ as \emph{the weight of the GNS representation}. We note that all densely defined weights arise in this way:
 
\begin{lemma}\label{06-02-22e} Let $\psi$ be a densely defined weight on $A$. Then $\psi$ is the weight of its GNS representation $(H_\psi,\Lambda_\psi,\pi_\psi)$.
\end{lemma}
 \begin{proof} This is a re-formulation of Combes' theorem, Theorem \ref{04-11-21k}, for densely defined weights.
 \end{proof}

\subsection{Proof of Combes' theorem in the separable case}

\emph{Proof of Theorem \ref{09-11-21h}:} Assume first that $\psi$ is densely defined. Let $\mathcal C$ be the countable subset of $\mathcal F_{\psi}$ from Lemma \ref{08-11-21}, and set
$$
\mathcal C' := \left\{ \lambda \omega : \ \lambda \in ]0,1[ \cap \mathbb Q, \ \omega \in \mathcal C \right\} \subseteq \mathcal F_{\psi},
$$
where $\mathbb Q$ denotes the countable set of rational numbers.
Let $\mu_i, i \in \mathbb N$, be a numbering of the elements of $\mathcal C'$. We construct by induction a sequence $\mu'_1 \leq \mu'_2 \leq \mu'_3 \leq \cdots$ in $\mathcal F_{\psi}$ such that $\mu_i \leq \mu'_n, \ i \leq n$, and $\mu'_n \in \lambda_n \mathcal F_\psi$ for some $\lambda_n \in ]0,1[$. To start the induction set $\mu'_1 = \mu_1$. When $\mu'_1 \leq \mu'_2 \leq \cdots \leq \mu'_n$ have been constructed we construct $\mu'_{n+1}$ by using Lemma \ref{09-11-21ax} to find $\mu'_{n+1} \in \lambda_{n+1}\mathcal F_\psi$ for some $\lambda_{n+1} \in ]0,1[$ such that $\mu'_{n+1} \geq \mu'_n$ and $\mu'_{n+1} \geq \mu_{n+1}$. Then $\{\mu'_n\}_{n=1}^\infty$ has the stated properties and we claim that
$$
\lim_{n \to \infty} \mu'_n(a) = \psi(a)
$$
for all $a \in A^+$. Indeed, when $a\in A^+$ and $N \in \mathbb N$ are given we can use Lemma \ref{08-11-21} to get an element $f \in \mathcal C$ such that $f(a) \geq \min \{N+1, \psi(a)-\frac{1}{N+1}\}$. Choose $\lambda \in ]0,1[\cap \mathbb Q$ such that $\lambda f(a) \geq f(a) - \left(\frac{1}{N} - \frac{1}{N+1}\right)$. Then $\lambda f \in \mathcal C'$ and 
$$
\lambda f(a) \geq \min \{N, \psi(a)-\frac{1}{N}\} .
$$ 
There is an $i \in \mathbb N$ such that $\mu_i = \lambda f$, and then 
$$
\psi(a) \geq \mu'_k(a) \geq \mu_i(a) = \lambda f(a) \geq \min \{N, \psi(a)-\frac{1}{N}\}
$$
for all $k \geq i$, proving the claim. We define the sequence $\{\omega_n\}_{n =1}^{\infty}$ in $A^*_+$ such that $\omega_1 = \mu'_1$ and $\omega_n = \mu'_{n} - \mu'_{n-1}, \ n \geq 2$. Then \eqref{10-11-21x} holds for all $a \in A^+$.

Consider then the case where $\psi$ is not densely defined. Set $B := \overline{\mathcal M_\psi}$, the closure of $\mathcal M_\psi$ in $A$, which is a $C^*$-subalgebra of $A$ by (e) of Lemma \ref{04-11-21n}. Let $b \in B^+$. Then $\sqrt{b} \in B$ and we can approximate $\sqrt{b}$ by an element $x \in \mathcal M_\psi$. Then $x^*x$ approximates $b$ and $x^*x \in \mathcal M_\psi^+$ by (e) and (d) of Lemma \ref{04-11-21n}. This shows that $\psi|_{B^+}$ is a densely defined weight on $B$. It follows then from the first part of the proof that there is a sequence $\{\omega''_n\}_{n=1}^\infty$ in $B^*_+$ such that $\sum_{n=1}^\infty \omega''_n(b) = \psi(b)$ for all $b \in B^+ = B \cap A^+$. Choose $\omega'_n \in A^*_+$ such that $\omega'_n|_B = \omega''_n$. The closure $\overline{\mathcal N_\psi}$ of $\mathcal N_\psi$ is a closed left ideal in $A$ by Lemma \ref{04-11-21n} and we set $F := \left\{\omega \in A^*_+ : \ \omega(\overline{\mathcal N_\psi}) = \{0\}, \ \|\omega\| \leq 1 \right\}$, which is a weak* closed face in the quasi-state space of $A$. It follows from a result of Effros, reproduced in Theorem 3.10.7 of \cite{Pe}, that for every element $a \in A^+\backslash \overline{\mathcal N_\psi}$ there is $\omega \in F$ such that $\omega(a) >  0$. Indeed, if not the closure of $Aa + \mathcal N_\psi$ would be a closed left ideal $\mathcal I$ of $A$ strictly larger than $\overline{\mathcal N_\psi}$ with
$$
\left\{\omega\in A^*_+ :  \ \|\omega \| \leq 1, \ \omega(\mathcal I) = \{0\} \right\} = F,
$$
contradicting Effros' theorem. Note that $F$ is a compact metrizable space in the weak* topology because $A$ is separable. There is therefore a countable subset $\mathcal C'' \subseteq F$ which is dense in $F$. It follows that for all $a \in A^+ \backslash \overline{\mathcal N_\psi}$ there is an element $\omega \in \mathcal C''$ such that $\omega(a) > 0$. Let $\{\kappa_n\}_{n=1}^\infty$ be a sequence in $F$ which contains infinitely many copies of each element from $\mathcal C''$. Then 
$\sum_{n=1}^\infty \kappa_n(a) = 0$
for all $a \in \overline{\mathcal N_\psi} \cap A^+$ and 
$\sum_{n=1}^\infty \kappa_n(a) = \infty$
for all $a \in A^+ \backslash \overline{\mathcal N_\psi}$. Set $\omega_n := \omega'_n + \kappa_n \in A^*_+$. It follows from (c) of Lemma \ref{04-11-21n} that $\mathcal M_\psi \subseteq \mathcal N_\psi$ and hence $B \subseteq \overline{\mathcal N_\psi}$. Therefore
$$
\sum_{n=1}^\infty \omega_n(b) = \sum_{n=1}^\infty \omega'_n(b)= \psi(b) \ \ \forall b \in B \cap A^+.
$$
Consider an element $a \in \overline{\mathcal N_\psi} \cap A^+$. There is a sequence $\{x_n\}$ in $\mathcal N_\psi$ such that $\lim_{n \to \infty} x_n^*x_n = a^2$ so (b) of Lemma \ref{04-11-21n} implies that $a^2 \in B$. Since $B$ is a $C^*$-algebra it follows that $a = \sqrt{a^2} \in B$. This shows that $ \overline{\mathcal N_\psi} \cap A^+ \subseteq B \cap A^+$, implying that $\psi(a) = \infty$ when $a \in A^+ \backslash B$. For comparison we note that 
$$
\sum_{n=1}^\infty \omega_n(a) \geq \sum_{n=1}^\infty \kappa_n(a) = \infty
$$
when $a \in A^+ \backslash B \subseteq A^+ \backslash \overline{\mathcal N_\psi}$. It follows that the equality in \eqref{10-11-21x} holds for all $a \in A^+$.
\qed

\begin{notes} A version of Theorem \ref{09-11-21h}, where \eqref{10-11-21x} is only established for $a \in \mathcal M^+_\psi$, was obtained by Combes in Proposition 1.11 of \cite{C1}. A version for non-separable algebras was obtained by Pedersen and Takesaki in Corollary 7.3 of \cite{PT}. The GNS construction and Lemma \ref{08-11-21bx} was introduced in \cite{C1} while Lemma \ref{08-11-21e} and Lemma \ref{09-11-21ax} are taken from \cite{Ku1} and \cite{QV}. 
\end{notes}



\chapter{KMS weights}

After a considerable amount of preparation we prove in this chapter a fundamental theorem of Kustermans on which we base the definition of a KMS weight.

\bigskip

\section{Flows and holomorphic extensions}\label{holextensions}

Let $X$ be a complex Banach space. A \emph{flow} on $X$ is a representation $\sigma = (\sigma_t)_{t \in \mathbb R}$ of the group of real numbers $\mathbb R$ by linear isometries of $X$ such that $\sigma_0 = \id_X$ and
$$
\lim_{t \to t_0} \left\|\sigma_t(a) - \sigma_{t_0}(a)\right\| = 0
$$
for all $t_0 \in \mathbb R$ and all $a \in X$.

\subsection{Entire elements}\label{Section-entire}

Let $X$ be a complex Banach space and let $\sigma$ be a flow on $X$. An element $a \in X$ is \emph{entire analytic} for $\sigma$ when the function $t \mapsto \sigma_t(a)$ is entire analytic in the sense that there is a sequence $\{c_n(a)\}_{n=0}^\infty$ in $X$ such that
\begin{equation}\label{13-11-21}
\sum_{n=0}^\infty c_n(a)t^n = \sigma_t(a),
\end{equation}
with norm-convergence for all $t \in \mathbb R$. Let $\mathcal A_\sigma$ denote the set of elements of $X$ that are entire analytic for $\sigma$.  

\begin{lemma}\label{13-11-21a} Let $a \in \mathcal A_\sigma$. The sequence $\{c_n(a)\}_{n=0}^\infty$ of $X$ for which \eqref{13-11-21} holds with norm-convergence for all $t \in \mathbb R$ is unique, and the series 
$$
\sum_{n=0}^\infty \left\|c_n(a)\right\| z^n
$$
converges for all $z \in \mathbb C$.
\end{lemma}

\begin{proof} Let $\varphi \in X^*$. Then $\varphi(\sigma_t(a)) = \sum_{n=0}^\infty \varphi(c_n(a))t^n$ for all $t \in \mathbb R$, which implies that the radius of convergence of the power series
\begin{equation}\label{13-11-21b}
\sum_{n=0}^\infty \varphi(c_n(a))z^n 
\end{equation}
is infinite, and hence the formula \eqref{13-11-21b} defines an entire holomorphic function $f_\varphi : \mathbb C  \to \mathbb C$. If $\{c'_n(a)\}_{n=0}^\infty$ is another sequence for which \eqref{13-11-21} holds we get in the same way an entire holomorphic function $f'_\varphi : \mathbb C \to \mathbb C$ such that $f'_\varphi(z) =\sum_{n=0}^\infty \varphi(c'_n(a))z^n$. Since $f_\varphi(t) = f'_\varphi(t) = \varphi(\sigma_t(a))$ for all $t \in \mathbb R$ the two entire holomorphic functions must be identical, implying that $\varphi(c_n(a)) = \varphi(c'_n(a))$ for all $n$. Since $\varphi \in X^*$ was arbitrary we conclude that $c_n(a) = c'_n(a)$ for all $n$.

Let $\epsilon > 0$. That the radius of convergence of the series \eqref{13-11-21b} is infinite means that 
$$
\limsup_{n\to \infty} \left|\varphi(c_n(a))\right|^{\frac{1}{n}} = 0.
$$
It follows that for each $\varphi \in X^*$ there is a $K_\varphi > 0$ such that
$$
\left|\varphi(c_n(a))\right| \leq K_\varphi \epsilon^n
$$
for all $n\in \mathbb N$. We can therefore apply the Banach-Steinhaus theorem, Theorem 5.8 of \cite{Ru0}, to the family of functionals
$$
X^* \ni \varphi \mapsto \frac{\varphi(c_n(a))}{\epsilon^n}, \ n \in \mathbb N,
$$
to conclude that 
$$
\sup_{n \in \mathbb N} \frac{\|c_n(a)\|}{\epsilon^n}  < \infty .
$$
It follows first that $\limsup_{n\to \infty} \left\|c_n(a)\right\|^{\frac{1}{n}} \leq \epsilon$, and then since $\epsilon > 0$ was arbitrary that $\limsup_{n\to \infty} \left\|c_n(a)\right\|^{\frac{1}{n}} = 0$. This implies that $\sum_{n=0}^\infty \left\|c_n(a)\right\| z^n$
converges for all $z \in \mathbb C$.
 
\end{proof}

A function $f : \mathbb C \to X$ is \emph{entire holomorphic} when the limit
$$
f'(z) : = \lim_{h \to 0} \frac{f(z+h) - f(z)}{h}
$$
exists in $X$ for all $z \in \mathbb C$. An element $a \in X$ is \emph{entire holomorphic for $\sigma$} when there is an entire holomorphic function $f : \mathbb C \to X$ such that $f(t) = \sigma_t(a)$ for all $t \in \mathbb R$. Since entire holomorphic functions that agree on $\mathbb R$ are identical, the function $f$ is uniquely determined by $a$ when it exists. 

We aim now to show that an element $a \in X$ is entire holomorphic for $\sigma$ if and only if it is entire analytic for $\sigma$, and at the same time establish some apparently weaker but equivalent conditions.

\begin{lemma}\label{17-02-23} Let $\{x_n\}_{n=0}^\infty$ be a sequence in $X$ with the property that
$$
\sum_{n=0}^\infty x_n t^n
$$
converges in $X$ for all $t \in \mathbb R$. It follows that
\begin{itemize}
\item[(1)] the series $\sum_{n=0}^\infty x_n z^n$ and $\sum_{n=1}^\infty nx_n z^{n-1}$ converge for all $z \in \mathbb C$,
\item[(2)] the function $f(z) := \sum_{n=0}^\infty x_n z^n$ is  entire holomorphic, and
\item[(3)] $f'(z) = \sum_{n=1}^\infty nx_n z^{n-1}
$
for all $z \in \mathbb C$.
\end{itemize}
\end{lemma}
\begin{proof} For each $\varphi \in X^*$ the series
$$
 \sum_{n=0}^\infty \varphi(x_n) t^n
$$
converges for all $t \in \mathbb R$. As in the proof of Lemma \ref{13-11-21a} it follows from this and an application of the Banach-Steinhaus theorem that $\limsup_{n \to \infty} \left\|x_n\right\|^{\frac{1}{n}} = 0$. Thus
\begin{equation}\label{18-02-23}
\sum_{n=0}^\infty \|x_n\||z|^n < \infty 
\end{equation}
and
\begin{equation}\label{17-02-23b}
\sum_{n=1}^\infty \|x_n\|n|z|^{n-1}< \infty
\end{equation}
for all $z \in \mathbb C$. It follows from \eqref{18-02-23} and \eqref{17-02-23b} that (1) holds and we set $f(z) := \sum_{n=0}^\infty x_n z^n$. Note that for $n \geq 1$,
$$
(z+h)^n - z^n = \sum_{j=0}^{n-1} \binom{n}{j} z^j h^{n-j} = h  \sum_{j=0}^{n-1} \binom{n}{j} z^j h^{n-j-1} .
$$
For $0 < |h|\leq 1$ we find that
\begin{equation}\label{18-02-23a}
\begin{split}
&  \left|\frac{(z+h)^n -z^n}{h} - nz^{n-1}\right| \leq  \left|\sum_{j=0}^{n-1} \binom{n}{j} z^j h^{n-j-1} - nz^{n-1}\right| \\
& \leq   n |z|^{n-1} + \sum_{j=0}^n \binom{n}{j} |z|^j  \leq n|z|^{n-1} + (|z| +1)^n 
\end{split}
\end{equation}
for all $n \geq 1$. It follows from \eqref{18-02-23} and \eqref{17-02-23b} that  
$$
\sum_{n=0}^\infty \|x_n\|\left( n|z|^{n-1} +(|z|+1)^n \right) < \infty 
$$
and hence the estimate \eqref{18-02-23a} justifies that we interchange limit with summation in the following calculation:
\begin{align*}
&  \lim_{h \to 0} \left(\frac{ f(z+h) - f(z)}{h} -\sum_{n=1}^\infty x_n nz^{n-1} \right)  =  \lim_{h \to 0}  \sum_{n=1}^\infty x_n \left( \frac{(z+h)^n - z^n}{h} - nz^{n-1} \right) \\
& =    \sum_{n=1}^\infty x_n\lim_{h \to 0} \left( \frac{(z+h)^n - z^n}{h} -nz^{n-1}\right) =  0 .
\end{align*}
\end{proof}

\begin{remark}\label{31-08-22h} \rm{Let $Y \subseteq X$ be a closed subspace of $X$ such that $\sigma_t(Y) = Y$ for all $t \in \mathbb R$. Then $\sigma$ is also a flow on $Y$. Let $a \in Y$. It follows immediately from the definition that if $a$ is entire analytic for $\sigma$ considered as a flow on $Y$, it is also entire analytic for $\sigma$ as a flow on $X$. By Lemma \ref{17-02-23} the converse is also true; if $a$ is entire analytic for $\sigma$ as a flow on $X$ it is also entire analytic for $\sigma$ considered as a flow on $Y$. In fact, the sequence $\{c_n(a)\}_{n=0}^\infty$ from \eqref{13-11-21}, which a priori consists of element from $X$, is in fact a sequence in $Y$, given by the formula
$$
n! c_n(a) = \frac{\mathrm d^n}{\mathrm d t^n} \sigma_t(a)|_{t=0}.
$$
This follows from repeated applications of  Lemma \ref{17-02-23}.}
\end{remark}

\begin{thm}\label{17-02-23e} For $a \in X$ the following conditions are equivalent.
\begin{itemize}
\item[(1)] $a$ is entire analytic for $\sigma$.
\item[(2)] For each $\varphi \in X^*$ there is an entire holomorphic function $f_\varphi : \mathbb C \to \mathbb C$ such that $f_\varphi(t) = \varphi(\sigma_t(a))$ for all $t \in \mathbb R$.
\item[(3)] There is a function $f : \mathbb C \to X$ such that $f(t) = \sigma_t(a)$ for all $t \in \mathbb C$ and such that the function $\mathbb C \ni z \mapsto \varphi(f(z))$ is entire holomorphic for every $\varphi \in X^*$.
\item[(4)] $a$ is entire holomorphic for $\sigma$. 
\end{itemize}
\end{thm}
\begin{proof} The implications (4) $\Rightarrow$ (3) $\Rightarrow$ (2) are trivial and the implication (1) $\Rightarrow$ (4) follows directly from Lemma \ref{17-02-23}. We prove (2) $\Rightarrow$ (1) : Let $\varphi \in X^*$. By assumption there is a sequence $\{c_\varphi(n)\}_{n=0}^\infty$ in $\mathbb C$ such that 
$$
\sum_{n=0}^\infty c_\varphi(n)z^n
$$
converges for all for $z \in \mathbb C$ and
$$
\sum_{n=0}^\infty c_\varphi(n)t^n = \varphi(\sigma_t(a))
$$
for all $t \in \mathbb R$. By using that holomorphic functions that agree on $\mathbb R$ are identical and that the coefficients in the power series expansion of a holomorphic function are unique, it follows that the resulting maps $X^* \ni \varphi \mapsto c_\varphi(n)$ are well-defined and linear. Let $\epsilon > 0$. Since $\limsup_{n \to \infty} \left|c_\varphi(n)\right|^{\frac{1}{n}} = 0$ there is a $K_\varphi > 0$ such that 
$$
\left|c_\varphi(n)\right| \leq K_\varphi \epsilon^n
$$
for all $n$. By the Banach-Steinhaus theorem this implies that there is $K > 0$ such that
$$
\sup_{\varphi \in X^*, \ \|\varphi\| \leq 1} \ \frac{|c_\varphi(n)|}{\epsilon^n} \leq K
$$
for all $n$. In particular, the functionals $\varphi \mapsto c_\varphi(n)$ are all bounded and define elements $L_n \in X^{**}$ such that $L_n(\varphi) = c_\varphi(n)$ for all $\varphi,n$. Furthermore,
$$
\left\|L_n\right\| \leq K \epsilon^n
$$
for all $n$. Since $\epsilon > 0$ is arbitrary, it follows that $\limsup_{n \to \infty} \left\|L_n\right\|^{\frac{1}{n}} = 0$ implying that the series 
$$
\sum_{n=0}^\infty z^nL_n
$$
converges in the norm of $X^{**}$ for all $z \in \mathbb C$. Let $\iota : X \to X^{**}$ be the canonical embedding. Since
$$
\varphi(\sigma_t(a)) =  \sum_{n=0}^\infty c_\varphi(n)t^n = \sum_{n=0}^\infty t^nL_n(\varphi) = \left(\sum_{n=0}^\infty t^nL_n\right)(\varphi) 
$$
for all $\varphi \in X^*$ we have that
\begin{equation*}\label{17-02-23d}
\iota(\sigma_t(a)) = \sum_{n=0}^\infty t^nL_n 
\end{equation*}
for all $t \in \mathbb R$. In particular, $L_0 = \iota(a) \in \iota(X)$. Since $\iota(X)$ is norm closed in $X^{**}$ it follows from Lemma \ref{17-02-23} that $\sum_{n=1}^\infty n L_nt^{n-1} \in \iota(X)$ for all $t \in \mathbb R$, and hence, in particular, $L_1 \in \iota(X)$. By repeated application of Lemma \ref{17-02-23} it follows that $L_n \in \iota(X)$ for all $n$. Define $c_n(a) \in X$ such that $\iota(c_n(a)) = L_n$. The embedding $\iota$ is a linear isometry and hence the series $\sum_{n=0}^\infty c_n(a)t^n$ converges in norm for all $t \in \mathbb R$ since $\sum_{n=0}^\infty t^nL_n$ does.
Since $\varphi(\sigma_t(a)) = \sum_{n=0}^\infty t^nL_n(\varphi) =  \sum_{n=0}^\infty t^n\varphi(c_n(a)) = \varphi ( \sum_{n=0}^\infty c_n(a)t^n)$ for all $\varphi \in X^*$ it follows that $\sum_{n=0}^\infty c_n(a)t^n = \sigma_t(a)$, and we conclude that $a$ is entire analytic for $\sigma$.
\end{proof}

\begin{lemma}\label{13-11-21d} Let $\{a_n\}_{n=0}^\infty$ be a sequence in $\mathbb C$ and $\{b_n\}_{n=0}^\infty$ a sequence in $X$ such that
$\sum_{n=0}^\infty |a_n||z|^n < \infty$
and $\sum_{n=0}^\infty \|b_n\||z|^n < \infty$
for all $z \in \mathbb C$. Set 
$$
c_n = \sum_{i=0}^n a_ib_{n-i}, \ n \in \mathbb N.
$$ 
Then
\begin{equation}\label{14-11-21a}
\sum_{n=0}^\infty c_n z^n = \left(\sum_{n=0}^\infty  a_n z^n\right) \left(\sum_{n=0}^\infty  b_n z^n\right) ,
\end{equation}
with norm-convergence, for all $z \in \mathbb C$.
\end{lemma}

\begin{proof} Let $z \in \mathbb C$. It follows easily from the assumptions on $\{a_n\}_{n=0}^\infty$ and $\{b_n\}_{n=0}^\infty$ that the series $\sum_{n=0}^\infty  a_n z^n$ and $\sum_{n=0}^\infty  b_n z^n$ converge in $\mathbb C$ and $X$, respectively. It remains therefore only to establish \eqref{14-11-21a}. Let $N \in \mathbb N$. Then
\begin{align*}
&\left\|(\sum_{n=0}^N a_nz^n) (\sum_{n=0}^N b_nz^n) - \sum_{n=0}^N c_nz^n\right\| \\
&= \left\|\sum_{ 0 \leq i,j \leq N, \ i+j \geq N+1 }  a_iz^i b_jz^j\right\| \\
& \ \leq \sum_{n=\left[\frac{N}{2}\right]}^N |a_n||z|^n\sum_{n=0}^\infty \left\|b_n\right\||z|^n + \sum_{n=\left[\frac{N}{2}\right]}^N \left\|b_n\right\||z|^n\sum_{n=0}^\infty |a_n||z|^n
\end{align*}
where $\left[\frac{N}{2}\right]$ denotes the integer part of $\frac{N}{2}$. This estimate proves the lemma. 
\end{proof}

We are going to use some integration theory for Banach space valued functions and in Appendix \ref{integration} the reader may find the proofs of the few facts that we shall need.

\begin{lemma}\label{18-11-21} Let $n \in \mathbb N$ and $a \in X$. There is a sequence $\{b_k\}_{k=0}^\infty$ in $X$ such that $\sum_{k=0}^\infty \left\|b_k\right\| |z|^k < \infty$ and
$$
\sum_{k=0}^\infty b_kz^k = \sqrt{\frac{n}{\pi}}\int_\mathbb R e^{-n (s-z)^2} \sigma_{s}(a) \ \mathrm{d} s
$$
for all $z \in \mathbb C$.
\end{lemma}
\begin{proof}
Write 
\begin{align*}
& e^{-n(s-z)^2} = e^{-nz^2}e^{2nsz} e^{-ns^2} = e^{-nz^2} \sum_{k=0}^\infty z^k (2n)^k \frac{s^k e^{- ns^2}}{k!}   .
\end{align*}
Since $s \mapsto e^{n(2|z||s| -s^2)}$ is in $L^1(\mathbb R)$ it follows that 
$$
\int_\mathbb R e^{-n (s-z)^2} \sigma_{s}(a) \ \mathrm{d} s
$$ 
exists, for example by using Proposition 2.5.18 in \cite{BR} or Lemma \ref{12-02-22} in Appendix \ref{integration}. Since $ e^{n(2|z||s| -s^2)}$ dominates $\sum_{k=0}^N z^k (2n)^k \frac{s^k e^{- ns^2}}{k!}$ for all $N$ it follows from Lemma \ref{28-02-22} in Appendix \ref{integration} that we can interchange integration with summation to get
\begin{equation}\label{24-11-23}
\begin{split}
& \sqrt{\frac{n}{\pi}} \int_\mathbb R e^{-n (s-z)^2} \sigma_s(a) \ \mathrm{d} s  = \sqrt{\frac{n}{\pi}} e^{-nz^2} \sum_{k=0}^\infty \frac{z^k (2n)^k}{k!}\int_\mathbb R s^k e^{- ns^2} \sigma_s(a) \ \mathrm ds.
\end{split}
\end{equation}
Set
$$
b_k := \frac{(2n)^k}{k!} \int_\mathbb R s^k e^{- ns^2} \sigma_s(a) \ \mathrm ds \in X.
$$
Note that
$$
\left\| \int_\mathbb R s^k e^{- ns^2} \sigma_s(a) \ \mathrm ds\right\| \leq 2\|a\|  \int_0^\infty t^k e^{-nt^2} \ \mathrm d t = 2\|a\| n^{- \frac{k+1}{2}}\int_0^\infty t^k e^{-t^2} \ \mathrm d t.
$$
Set $I_k :=  \int_0^\infty t^k e^{-t^2} \ \mathrm d t$ so that
\begin{equation}\label{15-11-21}
\left\|b_k\right\| \leq \lambda_k
\end{equation}
where 
$$
\lambda_k :=  2^{k+1} \|a\| n^{\frac{k -1}{2}} \frac{I_k}{k!}$$ 
for all $k \in \mathbb N$. By using the relation
$$
I_k = \int_0^\infty t^{k-1}\frac{\mathrm d}{\mathrm d t} \left[-\frac{1}{2} e^{-t^2}\right] \ \mathrm dt  = \frac{k-1}{2} \int_0^\infty t^{k-2} e^{-t^2} \ \mathrm dt =  \frac{k-1}{2}I_{k-2}
$$
for $k \geq 2$, it follows that
$$
I_{k} = 2^{-\frac{k}{2}} (k-1)(k-3)(k-5) \cdots 3 I_0
$$
when $k$ is even and
$$
I_k = 2^{-\frac{k-1}{2}} (k-1)(k-3)(k-5) \cdots 4\cdot 2 I_1
$$
when $n$ is odd. Hence
$$
\frac{\lambda_{k+1}}{\lambda_k} = 2\sqrt{n} \frac{k(k-2)(k-4) \cdots 4\cdot 2}{(k+1)(k-1)(k-3) \cdots 3} \frac{I_1}{I_0}
$$
when $k$ is even while
$$
\frac{\lambda_{k+1}}{\lambda_k} = 2 \sqrt{n} \frac{k(k-2)(k-4) \cdots 3}{(k+1)(k-1)(k-3) \cdots 4 \cdot 2} \frac{I_0}{I_1}
$$
when $k$ is odd. By using that $\frac{k}{k+1} \leq e^{-\frac{1}{k+1}}$ we find that
$$
 \frac{k(k-2)(k-4) \cdots 4\cdot 2}{(k+1)(k-1)(k-3) \cdots 3 } \leq \exp \left( - \sum_{j=1}^{\frac{k}{2}} \frac{1}{2j+1}\right)
 $$
when $k$ is even and
$$
 \frac{k(k-2)(k-4) \cdots 3}{(k+1)(k-1)(k-3) \cdots 4 \cdot 2} \leq \exp \left( - \sum_{j=1}^{\frac{k+1}{2}} \frac{1}{2j}\right)
 $$
 when $k$ is odd. It follows that $\lim_{k \to \infty} \frac{\lambda_{k+1}}{\lambda_k} = 0$ and hence that $\sum_{k=0}^\infty \lambda_k |t|^k < \infty$ for all $t$. Combined with \eqref{15-11-21} this shows that $\sum_{k=0}^\infty \left\|b_k\right\| |z|^k < \infty$ for all $z \in \mathbb C$. It follows then from Lemma \ref{13-11-21d} that there is a sequence $\{c_k\}_{k=0}^\infty$ in $X$ such that
\begin{align*}
&\sqrt{\frac{n}{\pi}} e^{-nz^2} \sum_{k=0}^\infty \frac{z^k (2n)^k}{k!}\int_\mathbb R s^k e^{- ns^2} \sigma_s(a) \ \mathrm ds \\
& = \sqrt{\frac{n}{\pi}} \left(\sum_{k=0}^\infty \frac{(-nz^2)^k}{k!}\right) \left(\sum_{k=0}^\infty b_k z^k\right) = \sum_{k=0}^\infty c_kz^k,
\end{align*}
with norm-convergence, for all $z \in \mathbb C$. By \eqref{24-11-23} this proves the lemma.
\end{proof}

For $n \in \mathbb N$ and $a \in X$, set
\begin{equation}\label{17-11-21f}
R_n(a) :=  \sqrt{\frac{n}{\pi}}\int_\mathbb R e^{-n s^2} \sigma_{s}(a) \ \mathrm{d} s .
\end{equation}
Then $R_n : X \to X$ is a linear operator, and a contraction since $ \sqrt{\frac{n}{\pi}}\int_\mathbb R e^{-n s^2}  \ \mathrm{d} s =1$. These operators will be used diligently in the following, and we shall refer to them as \emph{smoothing operators}.

\begin{lemma}\label{24-11-21} $R_n(a) \in \mathcal A_\sigma$ for all $n \in \mathbb N, \ a\in X$, and $\lim_{n \to \infty} R_n(a) = a$ for all $a\in X$.
\end{lemma}
\begin{proof} Using Lemma \ref{18-11-21} we find that
\begin{align*}
&\sigma_t(R_n(a)) = \sqrt{\frac{n}{\pi}}\int_\mathbb R e^{-n s^2} \sigma_{t+s}(a) \ \mathrm{d} s \\
&= \sqrt{\frac{n}{\pi}}\int_\mathbb R e^{-n (s-t)^2} \sigma_{s}(a) \ \mathrm{d} s = \sum_{k=0}^\infty b_kt^k
\end{align*}
for all $t \in \mathbb R$, proving that $R_n(a) \in  \mathcal A_\sigma$.
 Let $\epsilon > 0$. Note that
 \begin{align*}
& \left\| R_n(a) -a\right\| = \left\|\sqrt{\frac{n}{\pi}} \int_\mathbb R e^{-n s^2} (\sigma_s(a) -a) \ \mathrm{d} s\right\| \\&\leq\sqrt{\frac{n}{\pi}} \int_\mathbb R e^{-n s^2} \left\|\sigma_s(a) -a\right\| \ \mathrm{d} s.
\end{align*} 
 Choose $\delta > 0$ such that $\left\|\sigma_s(a) -a\right\| \leq \frac{\epsilon}{2}$ when $|s| \leq \delta$. Since $\sqrt{n}e^{-ns^2}$ decreases with $n$ when $|s| \geq \delta$ and $n \geq (2 \delta^2)^{-1}$, and converges to $0$ for $n \to \infty$, it follows from Lebesgue's theorem that there is an $N \in \mathbb N$ such that
$$
\sqrt{\frac{n}{\pi}} \int_{|s| \geq \delta} e^{-n s^2} \left\|\sigma_s(a) -a\right\| \mathrm{d} s \leq \frac{\epsilon}{2}
$$
when $n \geq N$. Then, for $n \geq N$,  
\begin{align*}
&  \left\| R_n(a) -a\right\| \\
& \leq\sqrt{\frac{n}{\pi}} \int_{|s| \geq \delta} e^{-n s^2} \left\|\sigma_s(a) -a\right\|\mathrm{d} s  + \sqrt{\frac{n}{\pi}} \int_{|s| \leq \delta} e^{-n s^2} \left\|\sigma_s(a) -a\right\|\mathrm{d} s \\
& \leq  \frac{\epsilon}{2}\sqrt{\frac{n}{\pi}} \int_{|s| \leq \delta} e^{-n s^2} \mathrm{d} s + \frac{\epsilon}{2}  \leq \epsilon.
\end{align*} 
\end{proof}

For every $z \in \mathbb C$ we define an operator $\sigma_z : \mathcal A_\sigma \to X$
such that 
$$
\sigma_z(a) = \sum_{n=0}^\infty c_n(a) z^n,
$$
where $\{c_n(a)\}_{n=0}^\infty$ is the sequence from \eqref{13-11-21} and Lemma \ref{13-11-21a}. Alternatively, $\sigma_z(a)$ can be defined as $f(z)$, where $f: \mathbb C \to X$ is entire holomorphic and has the property that $f(t) = \sigma_t(a)$ when $t \in \mathbb R$. 

In the remaining part of the section we prove a series of facts about the operators $\sigma_z$ that we shall need later on.

\begin{lemma}\label{25-11-21} Let $a \in \mathcal A_\sigma$. Then
\begin{itemize}
\item[(a)] $\sigma_z(a) \in \mathcal A_\sigma$ for all $z \in \mathbb C$, and
\item[(b)] $\sigma_w \circ \sigma_z(a) = \sigma_z \circ \sigma_w(a) = \sigma_{z+w}(a)$ for all $z,w \in \mathbb C$.
\end{itemize}
\end{lemma}
\begin{proof} Let $t \in \mathbb R$. We prove first that 
\begin{itemize}
\item[(a')] $\sigma_t(a) \in \mathcal A_\sigma$, and
\item[(b')] $\sigma_t \circ \sigma_z(a) = \sigma_z \circ \sigma_t(a) = \sigma_{z+t}(a)$ for all $t \in \mathbb R$ and all $z \in \mathbb C$.
\end{itemize}
For all $s \in \mathbb R$,
$$
\sigma_s(\sigma_t(a)) = \sigma_t(\sigma_s(a)) = \sigma_t\left( \sum_{n=0}^\infty c_n(a)s^n\right) = \sum_{n=0}^\infty \sigma_t(c_n(a)) s^n,
$$
with norm-convergence. This proves (a'). To prove (b') note that the three functions $\mathbb C \to X$ given by
$z \mapsto \sigma_z(\sigma_t(a))$, $z \mapsto \sigma_t(\sigma_z(a))$, and $z \mapsto \sigma_{t+z}(a)$,
are all entire holomorphic and agree on $\mathbb R$. They are therefore identical.

Using (a') and (b') we obtain (a) and (b) as follows. By definition there is an entire holomorphic function $f : \mathbb C \to X$ such that $f(t) = \sigma_t(a)$ for all $t \in \mathbb R$ and $\sigma_z(a) = f(z)$. It follows from (b') that $\sigma_t(\sigma_z(a)) = \sigma_{t+z}(a) = f(t+z)$. Since $\mathbb C \ni w \mapsto f(w+z)$ is entire holomorphic it follows that $\sigma_z(a) \in \mathcal A_\sigma$ and that $\sigma_w(\sigma_z(a)) = f(w+z) = \sigma_{w+z}(a)$.
\end{proof}

\begin{lemma}\label{24-09-23x} Let $\mathcal D \subseteq \mathbb C$ such that $\left\{ \Imag z: \ z \in \mathcal D\right\}$ is a bounded subset of $\mathbb R$. Let $a \in \mathcal A_\sigma$. Then
$$
\sup_{z \in \mathcal D} \left\| \sigma_z(a)\right\| < \infty .
$$
\end{lemma}
\begin{proof} Since $\sigma_t$ is an isometry for all $t \in \mathbb R$ it follows from (b) of Lemma \ref{25-11-21} that $\left\| \sigma_z(a)\right\| = \left\| \sigma_{i \Imag z}(a)\right\|$ for all $z \in \mathbb C$. The lemma follows from this because the assumptions imply that
$\sup_{z \in \mathcal D} \left\| \sigma_{i \Imag z}(a)\right\| < \infty$.
\end{proof}

\begin{lemma}\label{15-11-21d} The operator $\sigma_z$ is closable for all $z \in \mathbb C$.
\end{lemma}
\begin{proof} When $z \in \mathbb R$ there is nothing to prove, so assume that $z \notin \mathbb R$. Let $\{a_n\}$ be a sequence in $\mathcal A_\sigma$ such that $\lim_{n \to \infty} a_n = 0$ and $\lim_{n\to \infty} \sigma_z(a_n) = b$ in $X$. We must show that $b =0$. It follows from Lemma \ref{25-11-21} that $\lim_{n \to \infty} \sigma_{i \Imag z}(a_n) = \sigma_{-\Real z}(b)$. Let $\mathcal D_{\Imag z}$ denote the strip in $\mathbb C$ consisting of the elements $u \in \mathbb C$ such that $\Imag u$ is between $0$ and $\Imag z$. For $n,m \in \mathbb N$ and $u \in  \mathcal D_{\Imag z}$, it follows from Lemma \ref{24-09-23x} and the Phragmen-Lindel{\"o}f theorem, Proposition 5.3.5 in \cite{BR}, that
$$
\left|\varphi(\sigma_u(a_n)) - \varphi(\sigma_u(a_m))\right| \leq \max \left\{ \left\|a_n-a_m\right\|, \ \left\|\sigma_{i \Imag z}(a_n) - \sigma_{i \Imag z}(a_m)\right\|\right\}\|\varphi\|
$$
for all $\varphi \in X^*$. It follows that the sequence of functions $u \mapsto \varphi(\sigma_u(a_n)), \ n \in \mathbb N$, converge uniformly on $\mathcal D_{\Imag z}$ to a continuous function $f$ which is holomorphic in the interior of $\mathcal D_{\Imag z}$. Since $\lim_{n \to \infty} \sigma_t(a_n) = 0$, the function $f$ must vanish on $\mathbb R$ and it follows therefore from Proposition 5.3.6 in \cite{BR}, applied to $z \mapsto \overline{f(\overline{z})}$ if $\Imag z < 0$, that $f =0$. Since 
$$
f(i \Imag z)= \varphi(\sigma_{-\Real z}(b)),
$$
it follows that $\varphi(\sigma_{-\Real z}(b)) =0$, and since $\varphi \in A^*$ is arbitrary, that $b =0$.
\end{proof}

In the following the symbol $\sigma_z$ will denote the closure of the operator $\sigma_z : \mathcal A_\sigma \to X$. Notice that for $z \in \mathbb R$ the two meanings of $\sigma_z$ agree.

\begin{lemma}\label{18-11-21b} $D(\sigma_{t+z}) = D(\sigma_z)$, $\sigma_t(D(\sigma_z)) = D(\sigma_z)$ and $\sigma_{t+z}(a) = \sigma_{t}(\sigma_z(a)) = \sigma_z(\sigma_t(a))$ for all $t \in \mathbb R$, $z \in \mathbb C$ and $a \in D(\sigma_z)$.
\end{lemma}
\begin{proof} Let $a \in D(\sigma_z)$. There is a sequence $\{a_n\}$ in $\mathcal A_\sigma$ such that $\lim_{n \to \infty} a_n = a$ and 
$$
\lim_{n \to \infty} \sigma_z(a_n) = \sigma_z(a).
$$ 
Using Lemma \ref{25-11-21} this implies that 
$$
\lim_{n \to \infty} \sigma_{t+z}(a_n) =\lim_{n \to \infty} \sigma_{z}(\sigma_t(a_n)) = \lim_{n \to \infty} \sigma_{t}(\sigma_z(a_n))) = \sigma_t(\sigma_z(a)).
$$ 
Since $\sigma_z$ and $\sigma_{t+z}$ are closed operators it follows that $a \in D(\sigma_{t+z}), \ \sigma_t(a) \in D(\sigma_z)$ and $\sigma_{t+z}(a) = \sigma_z(\sigma_t(a)) = \sigma_t(\sigma_z(a))$. The reverse inclusions $D(\sigma_{t+z}) \subseteq D( \sigma_z)$ and $D(\sigma_z) \subseteq \sigma_t(D(\sigma_z))$ follow by replacing $z$ with $z+t$ and $t$ with $-t$.
\end{proof}

\begin{lemma}\label{23-11-21} $\sigma_z\left(D(\sigma_z)\right) \subseteq D(\sigma_{-z})$ and $\sigma_{-z}\circ \sigma_z(a) = a \ \ \forall a \in D(\sigma_z)$.
\end{lemma}
\begin{proof} Let $a \in D(\sigma_z)$. There is a sequence $\{a_n\}$ in $\mathcal A_\sigma$ such that $\lim_{n \to \infty} a_n =a$ and $\lim_{n\to \infty} \sigma_z(a_n) = \sigma_z(a)$. It follows from Lemma \ref{25-11-21} that $\sigma_z(a_n) \in \mathcal A_\sigma$ and $\sigma_{-z}\left(\sigma_z(a_n)\right) = a_n$. Since $\sigma_{-z}$ is closed it follows that $\sigma_z(a) \in D(\sigma_{-z})$ and $ \sigma_{-z}\circ \sigma_z(a) = a$.

\end{proof}

\begin{lemma}\label{17-11-21i} Let $a \in X$. Then
\begin{equation}\label{12-05-22}
\sigma_z(R_n(a)) = \sqrt{\frac{n}{\pi}} \int_\mathbb R e^{-n(s-z)^2}\sigma_s(a) \ \mathrm ds 
\end{equation}
for all $n \in \mathbb N$ and $z \in \mathbb C$. If $a \in D(\sigma_z)$, then
\begin{equation}\label{12-05-22a}
\sigma_z(R_n(a)) = \sqrt{\frac{n}{\pi}} \int_\mathbb R e^{-ns^2}\sigma_s(\sigma_z(a)) \ \mathrm ds = R_n(\sigma_z(a))
\end{equation}
for all $n \in \mathbb N$. 
\end{lemma}
\begin{proof} It follows from Lemma \ref{18-11-21} that $z \mapsto f(z) : = \sqrt{\frac{n}{\pi}} \int_\mathbb R e^{-n (s-z)^2} \sigma_s(a) \ \mathrm{d} s$ is entire holomorphic and since 
$$
f(t) = \sqrt{\frac{n}{\pi}} \int_\mathbb R e^{-n(s-t)^2}\sigma_s(a) \ \mathrm ds  = \sqrt{\frac{n}{\pi}} \int_\mathbb R e^{-n s^2} \sigma_{s+t}(a) \ \mathrm{d} s = \sigma_t(R_n(a)) 
$$
for all $t \in \mathbb R$, equality \eqref{12-05-22} is a consequence of how $\sigma_z$ is defined. For \eqref{12-05-22a}, note that $\sigma_z(\sigma_s(a)) = \sigma_s(\sigma_z(a))$ depends continuously of $s$ and that $\sigma_z$ is closed. It follows therefore from Lemma \ref{12-02-22b} that $R_n(a) \in D(\sigma_z)$ and 
$$
\sigma_z(R_n(a)) = \sqrt{\frac{n}{\pi}} \int_\mathbb R e^{-ns^2} \sigma_s(\sigma_z(a)) \ \mathrm ds.
$$

\end{proof}

\begin{lemma}\label{11-02-22x} $\bigcap_{z \in \mathbb C} D(\sigma_z) = \mathcal A_\sigma$.
\end{lemma}
\begin{proof} Let $ a \in \bigcap_{z \in \mathbb C} D(\sigma_z)$ and let $\varphi \in X^*$. We will show that $\mathbb C \ni z \mapsto \varphi(\sigma_z(a))$ is entire holomorphic. Let $n \in \mathbb N$. Thanks to Lemma \ref{24-11-21} it follows from Proposition 5.3.5 in \cite{BR} that
\begin{align*}
&\sup_{|\Imag z| \leq n}\left|\varphi(\sigma_z(R_k(a))) - \varphi(\sigma_z(R_m(a))) \right| \\
&\leq \left\|\varphi\right\|\max \left\{\left\|\sigma_{in}(R_k(a)) - \sigma_{in}(R_m(a))\right\|, \left\|\sigma_{-in}(R_k(a)) - \sigma_{-in}(R_m(a))\right\|\right\} 
\end{align*}
for all $k,m \in \mathbb N$.
Since  $\lim_{k \to \infty} \sigma_{z}(R_k(a)) = \sigma_{z}(a)$ for all $z \in \mathbb C$ by Lemma \ref{17-11-21i} and Lemma \ref{24-11-21}, it follows that the sequence of functions 
$$
\varphi(\sigma_z(R_k(a)), k \in \mathbb N,
$$ 
converges uniformly on the strip $\left\{z \in \mathbb C: \ \left|\Imag z\right| \leq n \right\}$ to a function $f$ which is holomorphic in the interior of the strip since $\mathbb C \ni z \mapsto \varphi(\sigma_z(R_k(a)))$ is for each $k$. Since $f(z) = \lim_{k \to \infty} \varphi(\sigma_{z}(R_k(a))) = \varphi(\sigma_{z}(a))$ for all $z$ in the strip it follows that $\mathbb C \ni z \mapsto \varphi(\sigma_z(a))$ is holomorphic in the interior of the strip. Since $n$ was arbitrary we conclude that $\mathbb C \ni z \mapsto \varphi(\sigma_z(a))$ is entire holomorphic. Since $\varphi \in X^*$ was arbitrary this means that $a \in \mathcal A_\sigma$ by Theorem \ref{17-02-23e}. This completes the proof because the inclusion $\mathcal A_\sigma \subseteq \bigcap_{z \in \mathbb C} D(\sigma_z)$ holds by construction.
\end{proof} 

\begin{example}\label{15-06-22}

\textnormal{Let $\mathbb H$ be a Hilbert space and $U = (U_t)_{t \in \mathbb R}$ a strongly continuous unitary representation of $\mathbb R$ on $\mathbb H$; a one-parameter unitary group on $\mathbb H$ in the sense of \cite{KR}. This is an example of a flow of isometries as those considered more generally in this section and one can therefore define the operators $U_z$ for each $z \in \mathbb C$ as we have done here. By Stone's theorem there is a self-adjoint operator $H$ on $\mathbb H$ such that $U_t = e^{i t H}$, cf. Theorem 5.6.36 of \cite{KR}. Then 
\begin{equation}\label{15-06-22a}
U_z = e^{ i zH} 
\end{equation}
for all $z \in \mathbb C$, where the latter is the operator defined by spectral theory applied to the self-adjoint operator $H$, cf. e.g. Theorem 5.6.26 in \cite{KR}. To see this, let $\{E_\lambda\}$ be the spectral resolution of $H$, and set
$$
E_n := E[-n,n] .
$$
By spectral theory $\bigcup_n E_n\mathbb H$ is a core for $e^{i z H}$ and we claim that it is also a core for $U_z$. To see this, note first of all that $E_k \mathbb H \subseteq \mathcal A_U$ since 
$$
U_t E_k\psi = \sum_{n=0}^\infty (iE_k H)^nE_k t^n
$$
with convergence in $\mathbb H$ for all $\psi \in \mathbb H$. Consider then an element $w \in \mathbb H$ from the set $\mathcal A_U$ of entire analytic elements for the flow $U$. By Lemma \ref{13-11-21a} there is a sequence $\{v_k\}$ in $\mathbb H$ such that
$$
U_t w = \sum_{k=0}^\infty v_k t^k,
$$
with convergence in $\mathbb H$ for all $t \in \mathbb R$. Then
$$
U_tE_nw= E_nU_tw =\sum_{k=0}^\infty E_n v_k t^k
$$
for all $t \in \mathbb R$ and hence $U_zE_nw = \sum_{k=0}^\infty E_n v_k z^k$. Since $\sum_{k=0}^\infty \left\|v_k\right\| |z|^k < \infty$ by Lemma \ref{13-11-21a}, and $\lim_{n \to \infty} E_n v_k = v_k$ for all $k$, it follows that
$$
\lim_{n \to \infty} U_zE_nw = \sum_{k=0}^\infty v_kz^k = U_z w .
$$
Since $\mathcal A_U$ is a core for $U_z$ by definition, this shows that so is $\bigcup_n E_n\mathbb H$. Now note that for $w \in \bigcup_n E_n\mathbb H$,
$$
U_tw = \sum_{k=0}^\infty \frac{(i H)^k}{k!} t^k w 
$$
with convergence in the norm of $\mathbb H$, implying that
$$
U_zw =  \sum_{k=0}^\infty \frac{(i zH)^k}{k!} w  = e^{izH}w .
$$
Hence $U_z$ and $e^{i zH}$ agree on the common core $\bigcup_n E_n\mathbb H$. Since both operators are closed, they agree, i.e. \eqref{15-06-22a} holds.}
\end{example}

\subsection{Flows on $C^*$-algebras}

Let $A$ a $C^*$-algebra. A \emph{flow} on $A$ is a representation $\sigma = (\sigma_t)_{t \in \mathbb R}$ of  $\mathbb R$ by automorphisms of $A$ such that
$$
\lim_{t \to t_0} \left\|\sigma_t(a) - \sigma_{t_0}(a)\right\| = 0
$$
for all $t_0 \in \mathbb R$ and all $a \in A$. This is a particular case of what we have considered above, but we shall need a few facts that involve the additional structure of a $C^*$-algebra. Let therefore now $\sigma$ be a flow on the $C^*$-algebra $A$.

\begin{lemma}\label{24-11-21k} Let $z \in \mathbb C$. Then $D(\sigma_z)^* = D(\sigma_{\overline{z}})$ and $\sigma_{\overline{z}}(a^*) = \sigma_z(a)^*$ for all $a \in D(\sigma_z)$. 
\end{lemma} 
\begin{proof} Let $a \in \mathcal A_\sigma$. If $f : \mathbb C \to A$ is entire holomorphic and $f(t) = \sigma_t(a)$ for $t \in \mathbb R$ we have that $\sigma_z(a) = f(z)$. Note that $z \mapsto g(z) := f(\overline{z})^*$ is also entire holomorphic and $g(t) = \sigma_t(a^*)$ for all $t \in \mathbb R$. It follows that $a^* \in \mathcal A_\sigma$ and $\sigma_{\overline{z}}(a^*) = g(\overline{z}) = f(z)^* = \sigma_z(a)^*$. Let then $a \in D(\sigma_z)$. There is a sequence $\{a_n\}$ in $\mathcal A_\sigma$ such that $\lim_{n \to \infty} a_n =a$ and $\lim_{n \to \infty} \sigma_z(a_n) = \sigma_z(a)$. Since $\lim_{n \to \infty} a_n^* = a^*$, $\lim_{n \to \infty} \sigma_{\overline{z}}(a_n^*) =  \lim_{n \to \infty} \sigma_{{z}}(a_n)^* = \sigma_z(a)^*$ and $\sigma_{\overline{z}}$ is closed, it follows that $a^* \in D(\sigma_{\overline{z}})$ and $\sigma_{\overline{z}}(a^* ) = \sigma_z(a)^*$. We have shown that $D(\sigma_z)^* \subseteq D(\sigma_{\overline{z}})$ for all $z \in \mathbb C$. Therefore $D(\sigma_{\overline{z}}) = D(\sigma_{\overline{z}})^{**} \subseteq D(\sigma_z)^*$.
\end{proof}

\begin{lemma}\label{11-11-21} $\mathcal A_\sigma$ is a dense $*$-subalgebra of $A$.
\end{lemma}
\begin{proof} Let $a,b \in \mathcal A_\sigma$. By Theorem \ref{17-02-23e} there are entire holomorphic functions $f,g : \mathbb C \to A$ such that $f(t) = \sigma_t(a)$ and $g(t) = \sigma_t(b)$ for all $t \in \mathbb R$. Then $\mathbb C \ni z \mapsto f(z)g(z)$ is entire holomorphic, with $\frac{\mathrm d}{\mathrm d z} f(z)g(z) = \left[\frac{\mathrm d}{\mathrm d z}f(z)\right] g(z) +  f(z)\left[\frac{\mathrm d}{\mathrm d z}g(z)\right]$ and $f(t)g(t) = \sigma_t(ab)$ for all $t \in \mathbb R$. Thus $ab \in \mathcal A_\sigma$. It follows that $\mathcal A_\sigma$ is a subalgebra of $A$. It is a invariant under $*$ because $\mathbb C \ni z \mapsto f(\overline{z})^*$ is holomorphic and $f(t)^* = \sigma_t(a^*)$ for all $t \in \mathbb R$. It follows from Lemma \ref{24-11-21} that $\mathcal A_\sigma$ is dense in $A$.
\end{proof}

\begin{lemma}\label{18-11-21g} Let $a \in \mathcal A_\sigma$ and $b \in D(\sigma_z)$. Then $ab \in D(\sigma_z)$ and $\sigma_z(ab) = \sigma_z(a)\sigma_z(b)$.
\end{lemma}
\begin{proof} Choose a sequence $\{b_n\}$ in $\mathcal A_\sigma$ such that $\lim_{n \to \infty} b_n = b$ and 
$$
\lim_{n \to \infty} \sigma_z(b_n) = \sigma_z(b).
$$ 
For each $n$ the function $w \mapsto \sigma_w(a)\sigma_w(b_n)$ is entire holomorphic with
$$
\frac{\mathrm d}{\mathrm d w} \sigma_w(ab_n) =\left[\frac{\mathrm d}{\mathrm d w} \sigma_w(a)\right] b_n + a  \left[\frac{\mathrm d}{\mathrm d w} \sigma_w(b_n)\right] .
$$ 
It follows from Lemma \ref{11-11-21} that $t \mapsto \sigma_t(ab_n)$ has an entire holomorphic extension $w \mapsto \sigma_w(ab_n)$. Since this extension agrees with $\sigma_w(a)\sigma_w(b_n)$ when $w \in \mathbb R$, it follows that $\sigma_w(a)\sigma_w(b_n) = \sigma_w(ab_n)$ for all $w\in \mathbb C$. The desired conclusion follows now because $\sigma_z$ is closed.
\end{proof}

\begin{notes} The main content in this section is a slightly modified version of material from Bratteli and Robinsons monograph \cite{BR}. The material about $\sigma_z$ for non-real $z$ is gleaned from Kustermans, \cite{Ku1}. I have selected facts that will be needed in the following. More can be found in \cite{Ku1} and \cite{Ku2}.
\end{notes}


\section{Flows and invariant weights}\label{invariantweights} In this section we consider a flow $\sigma$ on the $C^*$-algebra $A$ and a densely defined weight $\psi$ on $A$ which we assume is $\sigma$-invariant in the sense that $\psi \circ \sigma_t = \psi$ for all $t \in \mathbb R$. Let $H_\psi$ be the Hilbert space and $\Lambda_\psi : \mathcal N_\psi \to H_\psi$ the linear map from the GNS construction associated to $\psi$, cf. Section \ref{06-02-22}. 

\begin{lemma}\label{17-11-21a} $\Lambda_\psi : \mathcal N_\psi \to H_\psi$ is closed.
\end{lemma}
\begin{proof} Let $\lim_{n \to \infty} a_n = a$ in $A$ with $a_n \in \mathcal N_\psi$ for all $n$, and assume that $\lim_{n \to \infty} \Lambda_\psi(a_n) = v$ in $H_\psi$. We must show that $a \in \mathcal N_\psi$ and that $\Lambda_\psi(a) = v$. Let $\omega \in \mathcal F_\psi$. By Lemma \ref{08-11-21bx} there is an operator $T_\omega$ on $H_\psi$ such that $0 \leq T_\omega \leq 1$ and
$$
\omega(c^*d) = \left<T_\omega \Lambda_\psi(d),\Lambda_\psi(c)\right> \  \ \forall c,d \in \mathcal N_\psi .
$$
Then
$$
\omega(a^*a) = \lim_{k \to \infty} \omega(a_k^*a_k) = \lim_{k \to \infty} \left<T_\omega \Lambda_\psi(a_k),\Lambda_\psi(a_k)\right> = \left<T_\omega v,v\right> \leq \|v\|^2 .
$$
Since $\omega \in \mathcal F_\psi$ was arbitrary it follows from Combes' theorem, Theorem \ref{04-11-21e}, that $\psi(a^*a) \leq \|v\|^2$, and hence $a \in \mathcal N_\psi$. Let $\epsilon >0$ and let $b \in \mathcal N_\psi$. It follows also from Combes' theorem that there is $\omega \in \mathcal F_\psi$ such that 
$$
\psi(b^*b) - \epsilon \leq \omega(b^*b) \leq \psi(b^*b).
$$
Then
\begin{align*}
& \left\| T_\omega \Lambda_\psi(b) - \Lambda_\psi(b)\right\|^2 \\
&= \left<T_\omega^2\Lambda_\psi(b),\Lambda_\psi(b)\right> + \left< \Lambda_\psi(b),\Lambda_\psi(b)\right> - 2\left<T_\omega \Lambda_\psi(b),\Lambda_\psi(b)\right> \\
& \leq \omega(b^*b) + \psi(b^*b) - 2\omega(b^*b) \\
&\leq 2\psi(b^*b) - 2(\psi(b^*b) - \epsilon) = 2\epsilon .
\end{align*}
We can choose $k$ so large that $\left\|v-\Lambda_\psi(a_k)\right\| \leq \epsilon$ and $\left|\omega(b^*a_k) - \omega(b^*a)\right| \leq \epsilon$.
 By the calculation above,
 \begin{align*}
 &\left|\left< v,\Lambda_\psi(b)\right> - \left<\Lambda_\psi(a),\Lambda_\psi(b)\right>\right| \\
 &\leq (\|v\| + \left\|\Lambda_\psi(a)\right\|) \sqrt{2\epsilon} + \left|\left<v,T_\omega\Lambda_\psi(b)\right> - \left<\Lambda_\psi(a), T_\omega\Lambda_\psi(b)\right> \right| \\
 & \leq (\|v\| + \left\|\Lambda_\psi(a)\right\|)\sqrt{2\epsilon} + \left\|\Lambda_\psi(b)\right\| \epsilon \\ 
 & \ \ \ \ \ \ \ \ \ \ \ \ \ \ \  +  \left| \left<\Lambda_\psi(a_k),T_\omega\Lambda_\psi(b)\right> - \left<\Lambda_\psi(a), T_\omega\Lambda_\psi(b)\right>\right| \\
 & = (\|v\| + \left\|\Lambda_\psi(a)\right\|)\sqrt{2\epsilon} + \left\|\Lambda_\psi(b)\right\| \epsilon + \left|\omega(b^*a_k) - \omega(b^*a)\right|\\
 & \leq (\|v\| + \left\|\Lambda_\psi(a)\right\|)\sqrt{2\epsilon} + \left\|\Lambda_\psi(b)\right\| \epsilon + \epsilon .
 \end{align*}
Since both $\epsilon > 0$ and $b\in \mathcal N_\psi$ were arbitrary we conclude that $v = \Lambda_\psi(a)$.
\end{proof}

\begin{cor}\label{24-06-22a} $\mathcal N_\psi$ is a Banach algebra when equipped with the algebra structure from $A$ and the norm
$$
|||a||| := \|a\| + \left\|\Lambda_\psi(a)\right\|  .
$$
\end{cor}
\begin{proof} $\mathcal N_\psi$ is a left ideal by (a) of Lemma \ref{04-11-21n} and hence a subalgebra of $A$. When $a,b \in \mathcal N_\psi$,
\begin{align*}
&|||ab||| = \|ab\| + \left\|\Lambda_\psi(ab)\right\| = \|ab\| + \left\|\pi_\psi(a) \Lambda_\psi(b)\right\| \\
& \leq \|a\|\|b\| + \|a\| \left\|\Lambda_\psi(b)\right\| \leq |||a||| \ |||b||| .
\end{align*}
That $\mathcal N_\psi$ is complete in the norm $||| \ \cdot \ |||$ follows because $A$ and $H_\psi$ are complete, and $\Lambda_\psi$ is closed by Lemma \ref{17-11-21a}.
\end{proof}

Since $\psi$ is $\sigma$-invariant it follows that $\sigma_t(\mathcal N_\psi) = \mathcal N_\psi$ and we obtain for each $t \in \mathbb R$ a unitary $U^\psi_t \in B(H_\psi)$ such that
$$
U^\psi_t\Lambda_\psi(a) = \Lambda_\psi(\sigma_t(a)) \ \ \forall a \in \mathcal N_\psi .
$$
It is straightforward to check that
\begin{equation}\label{08-03-22a}
U^\psi_t\pi_\psi(a) U^\psi_{-t} = \pi_\psi(\sigma_t(a))
\end{equation}
for all $a \in A$.

\begin{lemma}\label{17-11-21e} $\left\{U^\psi_t\right\}_{t \in \mathbb R}$ is a strongly continuous unitary representation of $\mathbb R$; that is, ${U^\psi_t}^* = U^\psi_{-t}$ for all $t \in \mathbb R$, $U^\psi_{t+s} = U^\psi_tU^\psi_s$ for all $s,t\in \mathbb R$, $U_0 = 1$ and $\mathbb R \ni t \mapsto U^\psi_tv$ is continuous for each $v \in H_\psi$.
\end{lemma}
\begin{proof} The algebraic statements are easily established. By using them it follows that in order to establish continuity it suffices to show that 
\begin{equation}\label{24-11-21e}
\lim_{t \to 0}\left<U^\psi_t\Lambda_\psi(a),\Lambda_\psi(b)\right> = \left< \Lambda_\psi(a),\Lambda_\psi(b)\right>
\end{equation}
for all $a,b \in \mathcal N_\psi$. Let $\epsilon > 0$. It follows from Combes' theorem, Theorem \ref{04-11-21k}, that there is $\omega \in \mathcal F_\psi$ such that 
$$
\psi(a^*a) - \epsilon \leq \omega(a^*a) \leq \psi(a^*a).
$$
By Lemma \ref{08-11-21bx} there is an operator $T_\omega$ on $H_\psi$ such that $0 \leq T_\omega \leq 1$ and
$$
\omega(c^*d) = \left<T_\omega \Lambda_\psi(d),\Lambda_\psi(c)\right> \  \ \forall c,d \in \mathcal N_\psi .
$$
The same calculation as in the proof of Lemma \ref{17-11-21a} shows that 
$$
\left\| T_\omega \Lambda_\psi(a) - \Lambda_\psi(a)\right\|^2 \leq 2 \epsilon
$$ 
and we find therefore that
\begin{align*}
& \left|\left<U^\psi_t\Lambda_\psi(a), \Lambda_\psi(b)\right> - \left<\Lambda_\psi(a),\Lambda_\psi(b)\right>\right| \\
&= \left|\left<\Lambda_\psi(a), U^\psi_{-t}\Lambda_\psi(b)\right> - \left<\Lambda_\psi(a),\Lambda_\psi(b)\right>\right| \\
& \leq 2\left\|\Lambda_\psi(b)\right\| \sqrt{2\epsilon} + \left|\left<T_\omega\Lambda_\psi(a), U^\psi_{-t}\Lambda_\psi(b)\right> - \left<T_\omega \Lambda_\psi(a),\Lambda_\psi(b)\right>\right| \\
& = 2\left\|\Lambda_\psi(b)\right\| \sqrt{2\epsilon} + \left|\omega(\sigma_{-t}(b)^*a) - \omega(b^*a)\right|
\end{align*}
for all $t \in \mathbb R$. Since $\lim_{t \to 0}\omega(\sigma_{-t}(b)^*a) =\omega(b^*a)$ we get \eqref{24-11-21e}.
\end{proof}

Motivated by the last two lemmas we fix a GNS representation $(H,\Lambda,\pi)$ of $A$ such that
\begin{itemize}
\item $\sigma_t(D(\Lambda)) = D(\Lambda)$ for all $t \in \mathbb R$,
\item there is a continuous unitary representation $(U_t)_{t \in \mathbb R}$ of $\mathbb R$ on $H$ such that $U_t\Lambda(a) = \Lambda(\sigma_t(a))$ for $t \in \mathbb R, \ a \in D(\Lambda)$, and
\item $\Lambda : D(\Lambda) \to H$  is closed.
\end{itemize}

Many of the lemmas we shall need hold for GNS representations with no mention of weights, and it will be useful that they do, but for some of them it is necessary that the GNS representation arises from a weight. We fix therefore also a densely defined $\sigma$-invariant weight on $A$ and note that its GNS-triple $(H_\psi,\Lambda_\psi,\pi_\psi)$ is a GNS representation with the properties stipulated above. Note that $D(\Lambda_\psi) = \mathcal N_\psi$.

\begin{lemma}\label{18-11-21ex}  Let $v \in H$ and $n \in \mathbb N$. There is a sequence $\{v_k\}_{k=0}^\infty$ in $H$ such that $\sum_{k=0}^\infty \left\|v_k\right\| |z|^k < \infty$ and
$$
\sum_{k=0}^\infty v_kz^k = \sqrt{\frac{n}{\pi}}\int_\mathbb R e^{-n (s-z)^2} U_{s}v \ \mathrm{d} s
$$
for all $z \in \mathbb C$.
\end{lemma}
\begin{proof} This follows from Lemma \ref{18-11-21}.
\end{proof}

\begin{lemma}\label{17-11-21m} Let $a \in D(\Lambda)$. Then $\sigma_z( R_n(a)) \in D(\Lambda)$ and
$$
\Lambda(\sigma_z( R_n(a))) = \sqrt{\frac{n}{\pi}} \int_\mathbb R e^{-n (s-z)^2} \Lambda(\sigma_s(a)) \ \mathrm{d} s
$$
for all $n \in \mathbb N$ and all $z \in \mathbb C$.
\end{lemma}
\begin{proof} It follows from Lemma \ref{17-11-21i} that 
$$
\sigma_z(R_n(a)) = \sqrt{\frac{n}{\pi}} \int_\mathbb R e^{-n (s-z)^2} \sigma_s(a) \ \mathrm{d} s.
$$
The stated assertions follow therefore from Lemma \ref{12-02-22b} because $\Lambda$ is closed.
\end{proof}

\begin{lemma}\label{24-11-21g} Let $a \in D(\Lambda)$. Then 
\begin{itemize}
\item $\left\|\Lambda(R_n(a))\right\| \leq \left\|\Lambda(a)\right\|$ for all $n$, and
\item $\lim_{n \to \infty} \Lambda(R_n(a)) = \Lambda(a)$ in $H$.
\end{itemize}
\end{lemma}
\begin{proof} As a special case of the formula from Lemma \ref{17-11-21m} we get 
$$
\Lambda(R_n(a)) = \sqrt{\frac{n}{\pi}} \int_\mathbb R e^{-n s^2} U_s\Lambda(a) \ \mathrm{d} s.
$$
It follows therefore first from \eqref{28-02-22a} in Appendix \ref{integration} that $\left\|\Lambda(R_n(a))\right\| \leq \left\|\Lambda(a)\right\|$ and then from Lemma \ref{24-11-21} that $\lim_{n \to \infty} \Lambda(R_n(a)) = \Lambda(a)$.
\end{proof}

\begin{lemma}\label{08-12-21a} Let $a \in D( \Lambda) \cap D(\sigma_z)$. Assume that $\sigma_z(a) \in D(\Lambda)$. Then
$$
\lim_{k \to \infty} \Lambda(\sigma_z(R_k(a))) = \Lambda(\sigma_z(a))
$$
in $H$.
\end{lemma}
\begin{proof} By Lemma \ref{17-11-21i} $\sigma_z(R_k(a)) = R_k(\sigma_z(a))$ and hence the conclusion follows from Lemma \ref{24-11-21g}.
\end{proof}



\begin{lemma}\label{02-12-21ax} $\psi(R_k(a)) = \psi(a)$ for all $k \in \mathbb N$ and all $a \in A^+$.
\end{lemma}
\begin{proof} The $C^*$-subalgebra $S$ of $A$ generated by $\left\{\sigma_t(a) : \ t \in \mathbb R\right\}$ is separable and the restriction $\psi|_S$ of $\psi$ to $S$ is a weight. By Theorem \ref{09-11-21h} there is a sequence $\omega_n \in S^*_+$ such that $\psi(b) = \sum_{n=1}^\infty \omega_n(b)$ for all $b \in S^+$. By using that $\psi(\sigma_t(a)) = \psi(a)$ and Lebesgue's theorem on monotone convergence we find 

\begin{align*}
&\psi(R_k(a)) = \sum_{n=1}^\infty \omega_n(R_k(a)) = \sum_{n=1}^\infty \sqrt{\frac{k}{\pi}} \int_\mathbb R e^{-kt^2} \omega_n(\sigma_t(a)) \ \mathrm d t \\
& = \sqrt{\frac{k}{\pi}} \int_\mathbb R e^{-kt^2} \sum_{n=1}^\infty \omega_n(\sigma_t(a)) \ \mathrm d t  = \sqrt{\frac{k}{\pi}} \int_\mathbb R e^{-kt^2} \psi(\sigma_t(a)) \ \mathrm d t = \psi(a).
\end{align*}

\end{proof}

Set
\begin{equation}\label{31-01-22e}
\mathcal M_\psi^\sigma := \left\{ a \in \mathcal M_\psi \cap \mathcal A_\sigma: \ \sigma_z(a) \in \mathcal M_\psi  \ \ \forall z \in \mathbb C\right\}.
\end{equation}
It follows from Lemma \ref{25-11-21} that
$$
\sigma_z\left(\mathcal M^\sigma_\psi\right) = \mathcal M^\sigma_\psi
$$
for all $z \in \mathbb C$.
\begin{lemma}\label{07-12-21} $R_k(\mathcal M_\psi) \subseteq \mathcal M_\psi^\sigma$ for all $k \in \mathbb N$.
\end{lemma}
\begin{proof} It follows from Lemma \ref{02-12-21ax} that $R_k(\mathcal M_\psi) \subseteq \mathcal M_\psi$. Hence $R_k(\mathcal M_\psi) \subseteq \mathcal M_\psi \cap \mathcal A_\sigma$ by Lemma \ref{24-11-21}. Let $z \in \mathbb C, \ a \in \mathcal M_\psi$. To show that $\sigma_z(R_k(a)) \in \mathcal M_\psi$ we may assume, by linearity, that $a \in \mathcal M_\psi^+$. Since
$$
e^{-k(s-z)^2} = e^{-ks^2} e^{-kz^2}e^{2k sz} ,
$$
we can write $e^{-k(s-z)^2}$ as a sum 
$$
e^{-k(s-z)^2}  = e^{-k s^2}\sum_{j=1}^4 i^j f_j(s)
$$
where $f_j : \mathbb R \to [0,\infty), j = 1,2,3,4$, are continuous non-negative functions such that
$$
\left|f_j(s)\right| \leq |e^{-kz^2}| e^{2k|z||s|}
$$
for all $j,s$. It follows from Lemma \ref{17-11-21i} that
\begin{align*}
& \sigma_z(R_k(a)) = \sqrt{\frac{k}{\pi}} \sum_{j=1}^4 i^j \int_\mathbb R e^{-ks^2} f_j(s) \sigma_s(a) \ \mathrm d s .
\end{align*}
By using Theorem \ref{09-11-21h} and Lebesgue's theorem on monotone convergence in the same way as in the proof of Lemma \ref{02-12-21ax} it follows that
$$
\psi\left(\int_\mathbb R e^{-ks^2} f_j(s) \sigma_s(a) \ \mathrm d s\right) = \psi(a)\int_\mathbb R e^{-ks^2} f_j(s) \ \mathrm d s  < \infty,
$$
for each $j$. Thus $\sigma_z(R_k(a))\in \mathcal M_\psi$.
\end{proof}

\begin{lemma}\label{07-12-21a} $\mathcal M_\psi^\sigma$ is a dense $*$-subalgebra of $A$ which is invariant under $\sigma_z$ for all $z \in \mathbb C$. 
\end{lemma} 
\begin{proof} It follows from (e) of Lemma \ref{04-11-21n} and Lemma \ref{11-11-21} that $\mathcal M_\psi \cap \mathcal A_\sigma$ is a $*$-subalgebra and then from Lemma \ref{25-11-21}, Lemma \ref{18-11-21g} and Lemma \ref{24-11-21k} that $\mathcal M_\psi^\sigma$ is a $\sigma_z$-invariant $*$-subalgebra of $A$. Since $\psi$ is densely defined it follows from (e) of Lemma \ref{04-11-21n} that $\mathcal M_\psi$ is dense in $A$ and then from Lemma \ref{07-12-21} and Lemma \ref{24-11-21} that  $\mathcal M_\psi^\sigma$ is dense in $A$.
\end{proof}

Recall that a representation $\pi : A \to B(H)$ is \emph{non-degenerate} when 
$$
\left\{\pi(a)\eta : \ a \in A, \ \eta \in H\right\}
$$ 
spans a dense subspace of $H$.

\begin{lemma}\label{08-02-22} $\Lambda_\psi : \mathcal N_\psi\to H_\psi$ is closed and $\pi_\psi$ is non-degenerate.
\end{lemma}
\begin{proof} That $\Lambda_\psi$ is closed follows from Lemma \ref{17-11-21a}. To show that $\pi_\psi$ is non-degenerate, let $a \in \mathcal N_\psi$. There is a sequence $\{u_n\}$ in $A$ such that $0 \leq u_n \leq u_{n+1} \leq 1$ for all $n$ and $\lim_{n \to \infty} u_na = a$ in $A$. Note that
\begin{align*}
&\left\|\pi_\psi({u_n})\Lambda_\psi(a) - \Lambda_\psi(a)\right\|^2 = \psi(a^*(1-u_n)^2a) \\
&\leq \psi(a^*(1-u_n)a) = \psi(a^*a) - \psi(a^*u_na).
\end{align*} 
 Since $\psi$ is lower semi-continuous, $\lim_{n \to \infty}\psi(a^*u_na) = \psi(a^*a)$ and it follows that $\lim_{n \to \infty}\pi_\psi({u_n})\Lambda_\psi(a) = \Lambda_\psi(a)$ in $H_\psi$. By definition $\left\{\Lambda_\psi(a) : \ a \in \mathcal N_\psi\right\}$ is dense in $H_\psi$ and we conclude therefore that $\pi_\psi$ is non-degenerate.  
\end{proof}

\begin{lemma}\label{16-12-21a} Assume that $\pi : A \to B(H)$ is non-degenerate. There is a net $\{e_m\}_{m \in \mathcal I}$ in $D(\Lambda)$ such that $0 \leq e_m \leq 1$ for all $m$, $e_m \leq e_{m'}$ when $m \leq m'$, $\lim_{m \to \infty} e_m a = a$ for all $a \in A$ and $\lim_{m \to \infty} \pi(e_m) =1$ in the strong operator topology of $B(H)$.
\end{lemma}
\begin{proof} Let $\mathcal I$ be the collection of finite subsets of $D(\Lambda)^*$ which we consider as a directed set ordered by inclusion. For each $m \in \mathcal I$, set
$$
f_m := \sum_{a \in m} aa^* , 
$$
and
$$
e_m  := (\# m) f_m(1 + (\# m)f_m)^{-1} \in D(\Lambda)^*.
$$
Note that $e_m =e_m^*\in D(\Lambda)$ since $D(\Lambda)^*$ is a right ideal. As shown in the proof of Proposition 2.2.18 in \cite{BR} the net $\{e_m\}_{m \in \mathcal I}$ consists of positive contractions, i.e. $0 \leq e_m \leq 1$, it increases with $m$, i.e. $m \subseteq m' \Rightarrow e_m \leq e_{m'}$ and it has the property that $\lim_{m \to \infty} e_m b = b$ for all $b \in D(\Lambda)^*$. Since $D(\Lambda)^*$ is dense in $A$ it follows that $\lim_{m \to \infty} e_m a = a$ for all $a \in A$. Since $\pi$ is non-degenerate this implies that $\lim_{m \to \infty} \pi(e_m) = 1$ in the strong operator topology.
\end{proof}

\begin{lemma}\label{07-12-21c} Let $F \subseteq \mathbb C$ be a finite set of complex numbers. Let $a \in \mathcal N_\psi \cap (\bigcap_{z \in F} D(\sigma_z))$. There is a sequence $\{a_n\}$ in $\mathcal M_\psi^\sigma$ such that 
\begin{itemize}
\item $\lim_{n \to \infty} a_n = a$, 
\item $\lim_{n \to \infty} \sigma_{z}(a_n) =  \sigma_{z}(a)$ for all $z \in F$, and
\item $\lim_{n \to \infty} \Lambda_\psi(a_n) = \Lambda_\psi(a)$.
\end{itemize}
If $\sigma_z(a) \in \mathcal N_\psi$ for all $z \in F$, we can also arrange that
\begin{itemize}
\item $\lim_{n \to \infty} \Lambda_\psi\left(\sigma_{z}(a_n)\right) =  \Lambda_\psi(\sigma_{z}(a))$ for all $z \in F$. 
\end{itemize}
\end{lemma}
\begin{proof} We will construct $a_n \in \mathcal M^\sigma_\psi$ such that
\begin{equation}\label{22-12-21a}
\left\|a_n - a\right\| \leq \frac{2}{n} ,
\end{equation}
\begin{equation}\label{17-03-22c}
\left\|\Lambda_\psi(a_n) - \Lambda_\psi(a)\right\| \leq \frac{2}{n},
\end{equation}
\begin{equation}\label{22-12-21c}
\left\|\sigma_z(a_n) - \sigma_z(a)\right\| \leq \frac{2}{n} ,
\end{equation}
for all $z \in F$, and when $\sigma_z(a) \in \mathcal N_\psi$ for all $z \in F$, also such that
\begin{equation}\label{22-12-21d}
\left\|\Lambda_\psi(\sigma_{z}(a_n)) - \Lambda_\psi(\sigma_{z}(a))\right\| \leq \frac{2}{n} 
\end{equation}
when $z \in F$.
Let $\{e_m\}_{m \in \mathcal I}$ be the approximate unit\footnote{In this text an approximate unit $\{u_i\}_{i \in I}$ in a $C^*$-algebra $D$ is a net of elements in $D$ with the property that $\lim_{i \to \infty} u_id =d$ for all $d \in D$. In most cases the net will have other properties too and they will then be specified explicitly.} of Lemma \ref{16-12-21a} applies to the triple $(H_\psi, \Lambda_\psi,\pi_\psi)$. The element $a_n$ we seek will be 
$$
a_n := R_k(e_ma)
$$ 
for an appropriate choice of $m \in \mathcal I$ and $k \in \mathbb N$. Note that since $a \in \mathcal N_\psi$ and $e_m \in \mathcal N_\psi^*$ it follows from (b) of Lemma \ref{04-11-21n} that $e_ma \in \mathcal M_\psi$ and hence $R_k(e_ma) \in \mathcal M^\sigma_\psi$ by Lemma \ref{07-12-21}. In order to choose $m$ and $k$ note first of all that the properties of $\{e_m\}_{m \in \mathcal I}$ ensure the existence of $m_n \in \mathcal I$ such that
\begin{equation}\label{22-12-21e}
\left\| e_m a - a \right\| \leq \frac{1}{n},
\end{equation}
and
\begin{equation}\label{22-12-21f}
\left\| \Lambda_\psi(e_m a)  - \Lambda_\psi(a) \right\| = \left\| \pi_\psi(e_m)\Lambda_\psi(a)  - \Lambda_\psi(a) \right\|\leq \frac{1}{n} 
\end{equation}
when $m_n \leq m$.
It follows from Lemma \ref{17-11-21m} that we then also have the estimate
\begin{equation}\label{22-12-21j}
\begin{split}
&\left\| \Lambda_\psi(R_k(e_ma)) - \Lambda_\psi(R_k(a))\right\| = \left\| \sqrt{\frac{k}{\pi}} \int_\mathbb R e^{-k s^2} U^\psi_s\left(\Lambda_\psi(e_ma) - \Lambda_\psi(a)\right) \ \mathrm ds \right\| \\
& \leq \left\|\Lambda_\psi(e_ma) - \Lambda_\psi(a)\right\| \leq \frac{1}{n} \ \ \forall k \ \text{when} \ m_n \leq m.
\end{split}
\end{equation}
We will choose $k \in \mathbb N$ first and then subsequently choose $m \geq m_n$. It follows from Lemma \ref{24-11-21} that we can choose $k$ such that
\begin{equation}\label{22-12-21g}
\left\| R_k(a) - a \right\| \leq \frac{1}{n}
\end{equation}
and from Lemma \ref{24-11-21g} such that also
\begin{equation}\label{22-12-21i}
\left\| \Lambda_\psi(R_k(a)) - \Lambda_\psi(a)\right\| \leq \frac{1}{n}.
\end{equation}
It follows from Lemma \ref{17-11-21i} and Lemma \ref{24-11-21} that 
$$
\lim_{k \to \infty}\sigma_{z}(R_k(a)) = \lim_{k \to \infty} R_k(\sigma_z(a)) = \sigma_{z}(a) ,
$$
for all $z \in F$ and we can therefore arrange that
\begin{equation}\label{22-12-21k}
\left\|\sigma_{z}(R_k(a)) - \sigma_{z}(a)\right\| \leq \frac{1}{n}
\end{equation}
for all $z \in F$. Similarly, thanks to Lemma \ref{08-12-21a} and assuming that $\sigma_z(a) \in \mathcal N_\psi$ for all $z \in F$, we can also arrange that
\begin{equation}\label{22-12-21m}
\left\|\Lambda_\psi(\sigma_{z}(R_k(a)) - \Lambda_\psi(\sigma_{z}(a))\right\| \leq \frac{1}{n}
\end{equation}
for all $z \in F$.
Now we fix $k \in \mathbb N$ such that \eqref{22-12-21g}, \eqref{22-12-21i}, \eqref{22-12-21k} and \eqref{22-12-21m} all hold; the latter of course only under the stated additional assumption. We go on to choose $m$.

Note that it follows from Lemma \ref{17-11-21i} that
\begin{align*}
& \left\| \sigma_{z} (R_k(e_ma)) - \sigma_{z} (R_k(a))\right\| \\
& \leq \sqrt{\frac{k}{\pi}} \int_\mathbb R \left|e^{-k(s-z)^2}\right| \left\| \sigma_s(e_m a) - \sigma_s(a)\right\| \ \mathrm{d} s \\
& \leq \|e_m a -a \| \sqrt{\frac{k}{\pi}} \int_\mathbb R \left|e^{-k(s-z)^2}\right|  \ \mathrm{d} s ,
\end{align*}
and from Lemma \ref{17-11-21m} that
\begin{align*}
&\left\|\Lambda_\psi(\sigma_{z}(R_k(e_m a))) - \Lambda_\psi(\sigma_{z}(R_k(a))) \right\|\\
&  = \left\| \sqrt{\frac{k}{\pi}} \int_\mathbb R e^{-k(s -z)^2} U^\psi_s(\Lambda_\psi(e_ma) - \Lambda_\psi(a)) \ \mathrm d s \right\|\\
& \leq \left\|\Lambda_\psi(e_ma) - \Lambda_\psi(a)\right\|  \sqrt{\frac{k}{\pi}} \int_\mathbb R \left|e^{-k(s -z)^2}\right| \ \mathrm d s \\
& = \left\|\pi_\psi(e_m)\Lambda_\psi(a)-\Lambda_\psi(a)\right\|   \sqrt{\frac{k}{\pi}} \int_\mathbb R \left|e^{-k(s -z)^2}\right| \ \mathrm d s 
\end{align*}
for all $z \in F$.
Thanks to Lemma \ref{16-12-21a} we can therefore choose $m \in \mathcal I$ such that $m_n \leq m$ and
\begin{equation}\label{22-12-21l} 
 \left\| \sigma_{z} (R_k(e_ma)) - \sigma_{z} (R_k(a))\right\|\leq \frac{1}{n}
\end{equation}
and
\begin{equation}\label{22-12-21n}
\left\|\Lambda_\psi(\sigma_{z}(R_k(e_m a))) - \Lambda_\psi(\sigma_{z}(R_k(a))) \right\| \leq \frac{1}{n}
\end{equation}
for all $z \in F$. It remains to check that $a_n := R_k(e_ma)$ satisfies the desired estimates \eqref{22-12-21a} \ -\ \eqref{22-12-21d}. Using \eqref{22-12-21g}, \eqref{22-12-21e} and that $R_k$ is a linear contraction we find that
\begin{align*}
& \left\| a_n -a \right\| \leq \left\| R_k(e_m a) - R_k(a)\right\| + \left\|R_k(a) -a\right\| \leq \|e_ma-a\| + \frac{1}{n} \leq \frac{2}{n}.
\end{align*}
Using \eqref{22-12-21l} and \eqref{22-12-21k} we find that
\begin{align*}
&\left\|\sigma_{z}(a_n) - \sigma_{z}(a)\right\|\\
&  \leq \left\|\sigma_{z}(R_k(e_ma)) - \sigma_{z}(R_k(a))\right\| + \left\| \sigma_{z}(R_k(a)) - \sigma_{z}(a)\right\| \leq\frac{2}{n} ,
\end{align*}
for all $z \in F$. It follows from \eqref{22-12-21j} and \eqref{22-12-21i} that
$$
\left\|\Lambda_\psi(a_n) - \Lambda_\psi(a)\right\| \leq \frac{2}{n} .
$$
Finally, when $\sigma_z(a) \in \mathcal N_\psi$ for all $z \in F$ it follows from \eqref{22-12-21m} and \eqref{22-12-21n} that
\begin{equation*}\label{22-12-21h}
\left\|\Lambda_\psi(\sigma_{z}(a_n)) - \Lambda_\psi(\sigma_{z}(a)) \right\| \leq \frac{2}{n}
\end{equation*}
for all $z \in F$.
\end{proof}

\begin{cor}\label{15-02-22} $\mathcal M_\psi^\sigma$ is a core for $\Lambda_\psi$; that is, for all $a \in \mathcal N_\psi$ there is a sequence $\{a_n\}$ in $\mathcal M^\sigma_\psi$ such that $\lim_{n \to \infty} a_n =a$ and $\lim_{n \to \infty} \Lambda_\psi(a_n) = \Lambda_\psi(a)$.
\end{cor}
\begin{proof} Apply Lemma \ref{07-12-21c} with $F = \{0\}$.
\end{proof}

\begin{lemma}\label{18-11-21kx} Let $a,b \in \mathcal N_\psi \cap \mathcal A_\sigma, \ z \in \mathbb C$. Assume that $\sigma_z(a) \in \mathcal N_\psi$ and $\sigma_{\overline{z}}(b) \in \mathcal N_\psi$. Then
$$
\psi(\sigma_z(b^*a)) = \psi(b^*a).
$$
\end{lemma}
\begin{proof} By using Lemma \ref{24-11-21k} and Lemma \ref{18-11-21g} we find that
\begin{align*}
&\psi(\sigma_z(b^*a)) = \psi(\sigma_{\overline{z}}(b)^*\sigma_z(a))  = \left< \Lambda_\psi(\sigma_z(a)),\Lambda_\psi(\sigma_{\overline{z}}(b))
\right>.
\end{align*}
It follows from Lemma \ref{08-12-21a} that $\lim_{n \to \infty} \Lambda_\psi(\sigma_z(R_n(a))) = \Lambda_\psi(\sigma_z(a))$ and $\lim_{n \to \infty} \Lambda_\psi(\sigma_{\overline{z}}(R_n(b))) = \Lambda_\psi(\sigma_{\overline{z}}(b))$ and hence
$$
 \left< \Lambda_\psi(\sigma_z(a)),\Lambda_\psi(\sigma_{\overline{z}}(b))\right> = \lim_{n \to \infty} \left<  \Lambda_\psi(\sigma_z(R_n(a))),  \Lambda_\psi(\sigma_{\overline{z}}(R_n(b)))\right> .
 $$ 
 It follows from Lemma \ref{17-11-21m} that
 $$
 \Lambda_\psi(\sigma_w(R_n(a))) = \sqrt{\frac{n}{\pi}} \int_\mathbb R e^{-n(s-w)^2} U^\psi_s \Lambda_\psi(a) \ \mathrm d s 
 $$
 for all $w \in \mathbb C$, and then from Lemma \ref{18-11-21ex} that $\mathbb C \ni w \mapsto \Lambda_\psi(\sigma_w(R_n(a)))$ is entire holomorphic. The same is true when $a$ is replaced by $b$, and consequently
 \begin{equation}\label{24-09-23}
 \mathbb C \ni w \mapsto  \left<  \Lambda_\psi(\sigma_w(R_n(a))) ,  \Lambda_\psi(\sigma_{\overline{w}}(R_n(b))) \right>
 \end{equation}
 is entire holomorphic. Since
\begin{align*}
& \left<  \Lambda_\psi(\sigma_t(R_n(a)))
,  \Lambda_\psi(\sigma_{{t}}(R_n(b)))
\right>  = \left<U^\psi_t \Lambda_\psi(R_n(a)), U^\psi_t\Lambda_\psi(R_n(b))\right> \\
&=  \left<\Lambda_\psi(R_n(a)), \Lambda_\psi(R_n(b))\right> 
\end{align*}
 when $t \in \mathbb R$, we conclude that the function \eqref{24-09-23} is constant. In particular,
 $$ \left<  \Lambda_\psi(\sigma_z(R_n(a)),  \Lambda_\psi(\sigma_{\overline{z}}(R_n(b))\right> =  \left<\Lambda_\psi(R_n(a)), \Lambda_\psi(R_n(b))\right>.
 $$
Using Lemma \ref{24-11-21g} we find therefore that
$$
\psi(\sigma_z(b^*a)) = \lim_{n \to \infty} \left<\Lambda_\psi(R_n(a)), \Lambda_\psi(R_n(b))\right> = \left<\Lambda_\psi(a),\Lambda_\psi(b)\right> = \psi(b^*a) .
$$ 
 \end{proof}

\section{The definition of a KMS weight}\label{KMSdefn}

Let $\sigma$ be a flow on the $C^*$-algebra $A$. The definition of a KMS weight for $\sigma$ will be based on the following theorem. To formulate it, consider a real number $\beta \in \mathbb R$. Set
$$
\mathcal D_\beta = \left\{z \in \mathbb C: \ \Imag z \in [0,\beta]\right\}
$$
when $\beta \geq 0$, and
$$
\mathcal D_\beta = \left\{z \in \mathbb C: \ \Imag z \in [\beta,0]\right\}
$$
when $\beta \leq 0$, and let $\mathcal D_\beta^0$ denote the interior of $\mathcal D_\beta$ in $\mathbb C$.

\begin{thm}\label{24-11-21d} (Kustermans, \cite{Ku1}) Let $\psi$ be a non-zero densely defined weight on $A$ which is $\sigma$-invariant in the sense that $\psi \circ \sigma_t = \psi$ for all $t \in \mathbb R$, and let $\beta \in \mathbb R$ be a real number. The following conditions are equivalent.
\begin{enumerate}
\item[(1)] $\psi(a^*a) = \psi\left( \sigma_{-i \frac{\beta}{2}}(a) \sigma_{-i \frac{\beta}{2}}(a)^*\right) \ \ \forall a \in D\left( \sigma_{-i \frac{\beta}{2}}\right)$.
\item[(2)]  $\psi(a^*a) = \psi\left( \sigma_{-i \frac{\beta}{2}}(a) \sigma_{-i \frac{\beta}{2}}(a)^*\right) \ \ \forall a \in \mathcal M_\psi^\sigma$. 
\item[(3)] $\psi(ab) = \psi(b\sigma_{i\beta}(a)) \ \ \forall a,b \in \mathcal M_\psi^\sigma$.
\item[(4)] For all $a,b \in \mathcal N_\psi \cap \mathcal N_\psi^*$ there is a continuous function $f: {\mathcal D_\beta} \to \mathbb C$ which is holomorphic in the interior $\mathcal D_\beta^0$ of ${\mathcal D_\beta}$ and has the property that
\begin{itemize}
\item $f(t) = \psi(b \sigma_t(a)) \ \ \forall t \in \mathbb R$, and
\item $f(t+i\beta) = \psi(\sigma_t(a)b) \ \ \forall t \in \mathbb R$.
\end{itemize}
\end{enumerate}
\end{thm}

We make the following definition.

\begin{defn}\label{24-11-21c} For $\beta \in \mathbb R$ we define a \emph{$\beta$-KMS weight} for $\sigma$ to be a non-zero densely defined weight on $A$ which is $\sigma$-invariant and satisfies one and hence all of the conditions in Definition \ref{24-11-21d}. 
\end{defn}

When $\psi$ is a $\beta$-KMS weight, but the value of $\beta$ is unimportant, we say that $\psi$ is a KMS weight.

It is apparent that the case $\beta =0$ is exceptional since $\mathcal D^0_\beta = \emptyset$ when $\beta =0$, and we make the following definition.

\begin{defn}\label{03-02-22f} A \emph{trace} on a $C^*$-algebra $A$ is a non-zero map $\psi : A^+ \to [0,\infty]$ with the following properties:
\begin{itemize}
\item $\psi(a+b) = \psi(a) + \psi(b) \ \ \forall a,b \in A^+$,
\item $\psi( t a) = t \psi(a) \ \ \forall a \in A^+, \ \forall t \in \mathbb R^+$, using the convention $0 \cdot \infty = 0$, 
\item  $\psi(a^*a) = \psi(aa^*)$ for all $a \in A$, and
\item $\left\{ a \in A^+ :  \psi(a) < \infty\right\}$ is dense in $A^+$.
\end{itemize}
\end{defn}

Thus a $0$-KMS weight is a $\sigma$-invariant lower semi-continuous trace on $A$.

\section{Proof of Kustermans' theorem} 
The case $\beta = 0$ is slightly exceptional because the interior $\mathcal D_\beta^0$ of $\mathcal D_\beta$ is empty in this case. In the following formulations we focus on the case $\beta \neq 0$ and leave the reader to make the necessary interpretations when $\beta =0$. It is worthwhile because Theorem \ref{24-11-21d} carries non-trivial information also in this case. 

The implication (1) $\Rightarrow$ (2) is trivial so it suffices to prove (2) $\Rightarrow$ (3) $\Rightarrow$ (4) $\Rightarrow$ (1).

(2) $\Rightarrow$ (3):  Let $a,b \in \mathcal M_\psi^\sigma$. Using polarization and Lemmas \ref{04-11-21n}, \ref{10-11-21}, \ref{18-11-21g}, \ref{24-11-21k}, \ref{23-11-21}, \ref{25-11-21} and \ref{18-11-21kx} we find that
\begin{equation}\label{08-12-21e}
\begin{split}
& \psi(b^*a) = 
\frac{1}{4} \sum_{k=1}^4 i^k \psi((a+ i^kb)^*(a+i^kb)) \\
& = \frac{1}{4} \sum_{k=1}^4 i^k \psi\left(\sigma_{-i \frac{\beta}{2}}(a+i^kb) \sigma_{-i \frac{\beta}{2}}(a+i^kb)^*  \right) \\
& =  \frac{1}{4} \sum_{k=1}^4 i^k \psi\left((\sigma_{-i \frac{\beta}{2}}(a)+i^k \sigma_{-i\frac{\beta}{2}}(b)) (\sigma_{-i \frac{\beta}{2}}(a)+i^k \sigma_{-i\frac{\beta}{2}}(b))^*  \right) \\
& = \psi(\sigma_{-i \frac{\beta}{2}}(a)\sigma_{-i \frac{\beta}{2}}(b)^*)\\
& = \psi\left( \sigma_{-i \frac{\beta}{2}}\left(a\sigma_{i\frac{\beta}{2}}(\sigma_{-i \frac{\beta}{2}}(b)^*)\right)\right)\\
& = \psi(a \sigma_{i\beta}(b^*)) .
\end{split}
\end{equation}

(3) $\Rightarrow$ (4): Let $a,b \in \mathcal N_\psi \cap \mathcal N_\psi^*$. We claim that there are sequences $\{a_n\}$ and $\{b_n\}$ in $\mathcal M_\psi^\sigma$ such that $\lim_{n \to \infty} a_n = a, \ \lim_{n \to \infty} \Lambda_\psi(a_n) = \Lambda_\psi(a)$, 
$$   \lim_{n \to \infty} \Lambda_\psi(a_n^*) = \Lambda_\psi(a^*)
$$ 
and  $\lim_{n \to \infty} b_n = b, \ \lim_{n \to \infty} \Lambda_\psi(b_n) = \Lambda_\psi(b), \   \lim_{n \to \infty} \Lambda_\psi(b_n^*) = \Lambda_\psi(b^*)$. To construct $\{a_n\}$ note that $\lim_{k \to \infty} R_k(a) = a$ by Lemma \ref{24-11-21}, and 
$$
\lim_{k \to \infty} \Lambda_\psi(R_k(a))  = \Lambda_\psi(a)
$$ 
and $\lim_{k \to \infty} \Lambda_\psi(R_k(a)^*)  = \Lambda_\psi(a^*)$ by Lemma \ref{24-11-21g}. It suffices therefore to construct the sequences $\{a_n\}$ and $\{b_n\}$ with $a$ and $b$ replaced $R_k(a)$ and $R_k(b)$. Note that $R_k(a) \in \mathcal N_\psi \cap \mathcal N_\psi^* \cap \mathcal A_\sigma$ by Lemma \ref{24-11-21} and Lemma \ref{17-11-21m}. Furthermore, $\sigma_z(R_k(a)) \in \mathcal N_\psi$ for all $z \in \mathbb C$ by Lemma \ref{17-11-21m}. We may therefore assume that $a \in   \mathcal N_\psi \cap \mathcal N_\psi^* \cap \mathcal A_\sigma$ and that $\sigma_z(a) \in \mathcal N_\psi$ for all $z \in \mathbb C$. It follows then from Lemma \ref{07-12-21c} that there is a sequence $\{a_n\}$ in $\mathcal M_\psi^\sigma$ such that $\lim_{n \to \infty} a_n = a$, $\lim_{n \to \infty} \Lambda_\psi(a_n) = \Lambda_\psi(a)$ and $\lim_{n \to \infty} \Lambda_\psi\left(\sigma_{i \frac{\beta}{2}}(a_n)\right) =  \Lambda_\psi\left(\sigma_{i \frac{\beta}{2}}(a)\right)$. Since we assume (3) we  have $\psi(x^*y) =   \psi(y\sigma_{i {\beta}}(x^*))$ for all $x,y \in \mathcal M^\sigma_\psi$, so by repeating the last two steps in the calculation \eqref{08-12-21e} we find that $\psi(x^*y) = \psi(\sigma_{-i \frac{\beta}{2}}(y) \sigma_{-i \frac{\beta}{2}}(x)^*)$ for all $x,y \in \mathcal M^\sigma_\psi$. In particular,
\begin{align*}
& \left\|\Lambda_\psi(a_n^*) - \Lambda_\psi(a_m^*)\right\|^2  = \psi\left((a_n -a_m)(a_n-a_m)^*\right) \\
& = \psi((a_n^*-a_m^*)^*(a_n^*-a_m^*)) = \psi\left( \sigma_{ -i \frac{\beta}{2}}(a_n^* -a_m^*)  \sigma_{ -i \frac{\beta}{2}}(a_n^* -a_m^*)^* \right)\\
& =\psi((\sigma_{i \frac{\beta}{2}}(a_n) - \sigma_{i \frac{\beta}{2}}(a_m))^*(\sigma_{i \frac{\beta}{2}}(a_n) - \sigma_{i \frac{\beta}{2}}(a_m)) \\
& = \left\|\Lambda_\psi( \sigma_{ i \frac{\beta}{2}}(a_n)) - \Lambda_\psi( \sigma_{ i \frac{\beta}{2}}(a_m))\right\|^2 
\end{align*}
for all $n,m$.
It follows that $\{\Lambda_\psi(a_n^*)\}$ converges in $H_\psi$ and since $\Lambda_\psi$ is closed by Lemma \ref{17-11-21a} it follows that $\lim_{n \to \infty} \Lambda_\psi(a_n^*) = \Lambda_\psi(a^*)$. The sequence $\{b_n\}$ is constructed in the same way. We need to arrange that the functions $z \mapsto \Lambda_\psi(\sigma_z(a_n))$ are entire analytic. To ensure this note that we can choose a sequence $\{k_n\}$ in $\mathbb N$ such that $\left\|R_{k_n}(a_n) - a_n\right\|  \leq \frac{1}{n}, \ \left\|\Lambda_\psi(R_{k_n}(a_n)) - \Lambda_\psi(a_n)\right\| \leq \frac{1}{n}$ and $\left\|\Lambda_\psi(R_{k_n}(a_n^*)) - \Lambda_\psi(a_n^*)\right\| \leq \frac{1}{n}$. This follows from Lemma \ref{24-11-21} and Lemma \ref{24-11-21g}. Note that $R_{k_n}(a_n) \in \mathcal M^\sigma_\psi$ by Lemma \ref{07-12-21}. The sequence $\{R_{k_n}(a_n)\}$ therefore has the same properties as $\{a_n\}$. Since $z \mapsto \Lambda_\psi\left(\sigma_z\left(R_{k_n}(a_n)\right)\right)$ is entire analytic by Lemma \ref{18-11-21ex} and Lemma \ref{17-11-21m}, we may assume that $z \mapsto \Lambda_\psi(\sigma_z(a_n))$ is entire analytic for all $n$. Furthermore, it follows from Lemma \ref{17-11-21m} and Lemma \ref{17-11-21i} that
$$
\Lambda_\psi\left(\sigma_z(R_{k_n}(a_n))\right) = U^\psi_z R_{k_n}\left(\Lambda_\psi(a_n)\right)
$$
for all $z \in \mathbb C$, and hence by Lemma \ref{24-09-23x} that
$$
\sup_{z \in \mathcal D_\beta} \left\|\Lambda_\psi\left(\sigma_z(R_{k_n}(a_n))\right)\right\| < \infty .
$$
We may therefore also assume that $\sup_{z \in \mathcal D_\beta} \left\|\Lambda_\psi(\sigma_z(a_n))\right\| < \infty$ for all $n$.

For each $n \in \mathbb N$ define $f_n : \mathbb C \to \mathbb C$ such that
$$
f_n(z) := \psi(b_n\sigma_z(a_n)) = \left<\Lambda_\psi(\sigma_z(a_n)),\Lambda_\psi(b_n^*)\right>.
$$
Note that $f_n$ is entire holomorphic since $z \mapsto \Lambda_\psi(\sigma_z(a_n))$ is entire analytic, and bounded on $\mathcal D_\beta$ since $\Lambda_\psi(\sigma_z(a_n))$ is norm-bounded for $z \in \mathcal D_\beta$. For $z \in \mathcal D_\beta$ we get therefore the estimate
\begin{equation}\label{18-11-21d}
\left|f_n(z) - f_m(z)\right| \leq \max \left\{ \sup_{t \in \mathbb R} |f_n(t) -f_m(t)|, \ \sup_{t \in \mathbb R} |f_n(t+i\beta) -f_m(t+i\beta)|\right\}
\end{equation}
from Proposition 5.3.5 in \cite{BR} (Phragmen-Lindel\"of) for all $n,m \in \mathbb N$.
We note that 
\begin{align*}
&f_n(t)  =\left< U^\psi_t\Lambda_\psi(a_n),\Lambda_\psi(b_n^*)\right>,
\end{align*}
and that $ U^\psi_t\Lambda_\psi(a_n)$ converges to $U^\psi_t\Lambda_\psi(a)$ uniformly in $t$ since $\lim_{n \to \infty} \Lambda_\psi(a_n) = \Lambda_\psi(a)$. As $\lim_{n \to \infty} \Lambda_\psi(b_n^*)= \Lambda_\psi(b^*)$ it follows therefore that 
$$
\lim_{n \to \infty} f_n(t) = \left<U^\psi_t\Lambda_\psi(a),\Lambda_\psi(b^*)\right> = \psi(b\sigma_t(a))
$$ 
uniformly in $t$. It follows from (3) that
\begin{equation*}
f_n(t+i\beta) =\psi(b_n \sigma_{i\beta}(\sigma_t(a_n))) = \psi(\sigma_t(a_n)b_n) = \left<\Lambda_\psi(b_n), U^\psi_t\Lambda_\psi(a_n^*)\right> .
\end{equation*}
Since $\lim_{n \to \infty} \Lambda_\psi(b_n)= \Lambda_\psi(b)$ and $\lim_{n \to \infty} \Lambda_\psi(a_n^*)= \Lambda_\psi(a^*)$, it follows that
$$
\lim_{n \to \infty} f_n(t+i\beta) = \left<\Lambda_\psi(b), U^\psi_t\Lambda_\psi(a^*)\right>  = \psi(\sigma_t(a)b)
$$
uniformly in $t$. It follows now from the estimate \eqref{18-11-21d} that the sequence $\{f_n\}$ converges uniformly on ${\mathcal D_\beta}$ to a continuous function $f : {\mathcal D_\beta} \to \mathbb C$ which is holomorphic in $\mathcal D_\beta^0$ and has the required properties. \footnote{Note that $f$ is also bounded. This additional property could therefore be added to (4) in the statement of Kustermans' theorem
, but the property will not be needed for the proof of (4) $\Rightarrow$ (1).}
(4) $\Rightarrow$ (1): Let $a\in \mathcal N_\psi \cap D(\sigma_{-i \frac{\beta}{2}})$. It follows from Lemma \ref{07-12-21c} that there is a sequence $\{a_n\}$ in $\mathcal M_\psi^\sigma$ such that $\lim_{n \to \infty} a_n = a$, $\lim_{n \to \infty} \Lambda_\psi(a_n) = \Lambda_\psi(a)$ and $\lim_{n \to \infty} \sigma_{-i \frac{\beta}{2}}(a_n) =  \sigma_{-i \frac{\beta}{2}}(a)$. As above we can arrange, by exchanging $a_n$ by $R_{k_n}(a_n)$ for some large $k_n \in \mathbb N$, that $z \mapsto \Lambda_\psi(\sigma_z(a_n^*))$ is entire analytic. Fix $n,m \in \mathbb N$, and set
$$
H(z) := \psi(a_n\sigma_z(a_m^*)) = \left<\Lambda_\psi(\sigma_z(a_m^*)), \Lambda_\psi(a_n^*)\right> \ \ \forall z \in \mathbb C.
$$
Then $H$ is entire holomorphic. Since we assume (4) there is a continuous function $f : {\mathcal D_\beta} \to \mathbb C$ such that $f$ is holomorphic on $\mathcal D^0_\beta$,
$$
f(t) = \psi(a_n \sigma_t(a_m^*)) \ \ \forall t \in \mathbb R, \ \text{and}
$$
$$
f(t+i\beta) = \psi(\sigma_t(a_m^*)a_n)) \ \ \forall t \in \mathbb R.
$$ 
Note that $f(t) = H(t)$ for all $t \in \mathbb R$. Applying Proposition 5.3.6 in \cite{BR} with $\mathcal O = \left\{z \in \mathbb C: \ \Imag z < \beta\right\}$ and $F(z) = H(z) -f(z)$ when $\beta > 0$ and with $\mathcal O = \left\{z \in \mathbb C: \ \Imag z < -\beta\right\}$ and $F(z) = \overline{H(\overline{z})} - \overline{f(\overline{z})}$ when $\beta < 0$, we conclude that $H(z) = f(z)$ for all $z \in {\mathcal D_\beta}$. Hence
\begin{align*}
\psi(\sigma_t(a_m^*)a_n) = f(t+ i \beta) = H(t+i\beta) = \psi(a_n\sigma_{t+i \beta}(a_m^*)) \
\end{align*} 
for all $t \in \mathbb R$.
By combining with Lemma \ref{24-11-21k}, Lemma \ref{18-11-21g} and Lemma \ref{18-11-21kx} we find
\begin{align*}
& \psi(a_m^*a_n) \\
& =\psi(a_n\sigma_{i \beta}(a_m^*)) =  \psi\left(a_n\sigma_{i \frac{\beta}{2}}\left( \sigma_{i \frac{\beta}{2}}(a_m^*)\right)\right)\\
& =  \psi\left(a_n\sigma_{i \frac{\beta}{2}}\left( \sigma_{-i \frac{\beta}{2}}(a_m)^*\right)\right)\\
& = \psi\left( \sigma_{i \frac{\beta}{2}}\left( \sigma_{-i \frac{\beta}{2}}(a_n)  \sigma_{-i \frac{\beta}{2}}(a_m)^*\right)\right) \\
& =  \psi\left( \sigma_{-i \frac{\beta}{2}}(a_n)  \sigma_{-i \frac{\beta}{2}}(a_m)^*\right) .
\end{align*} 
Thus
\begin{equation}\label{25-11-21e}
\left<\Lambda_\psi(a_n),\Lambda_\psi(a_m)\right> = \left< \Lambda_\psi( \sigma_{-i \frac{\beta}{2}}(a_m)^*), \Lambda_\psi( \sigma_{-i \frac{\beta}{2}}(a_n)^*)\right> ,
\end{equation}
and hence 
\begin{align*}
& \left\|\Lambda_\psi( \sigma_{-i \frac{\beta}{2}}(a_n)^*) - \Lambda_\psi( \sigma_{-i \frac{\beta}{2}}(a_m)^*)\right\|^2 \\
& = \left< \Lambda_\psi( \sigma_{-i \frac{\beta}{2}}(a_n)^*), \Lambda_\psi( \sigma_{-i \frac{\beta}{2}}(a_n)^*)\right> +  \left< \Lambda_\psi( \sigma_{-i \frac{\beta}{2}}(a_m)^*), \Lambda_\psi( \sigma_{-i \frac{\beta}{2}}(a_m)^*)\right> \\
& \ \ \ \ - \left< \Lambda_\psi( \sigma_{-i \frac{\beta}{2}}(a_n)^*), \Lambda_\psi( \sigma_{-i \frac{\beta}{2}}(a_m)^*)\right> -  \left< \Lambda_\psi( \sigma_{-i \frac{\beta}{2}}(a_m)^*), \Lambda_\psi( \sigma_{-i \frac{\beta}{2}}(a_n)^*)\right> \\
&  = \left< \Lambda_\psi( a_n), \Lambda_\psi(a_n)\right> +  \left<\Lambda_\psi(a_m), \Lambda_\psi(a_m)\right> - \left<\Lambda_\psi(a_n), \Lambda_\psi(a_m)\right> \\
& \ \ \ \ \ \ \ \ \ \ \ \ \ \ \ \ \ \ \ \ \ \ \ \ \ \ \ \ \ \ \ \ \ \ \ \ \ \ \ \ \ \ \ \ \ \ \ \ \ \ \ \ \ \ \ \ \ \ \ \ \ \ \ \ \ \ \ \ \ \ \ \ \ \ \ \ \ \ \ \ \ \ \ -  \left< \Lambda_\psi( a_m), \Lambda_\psi(a_n)\right> \\
& = \left\| \Lambda_\psi(a_n) - \Lambda_\psi(a_m)\right\|^2.
\end{align*}
It follows that $\{\Lambda_\psi( \sigma_{-i \frac{\beta}{2}}(a_n)^*)\}$ converges in $H_\psi$. Since $\lim_{n \to \infty} \sigma_{-i \frac{\beta}{2}}(a_n)^* = \sigma_{-i \frac{\beta}{2}}(a)^* = \sigma_{i\frac{\beta}{2}}(a^*)$ and $\Lambda_\psi$ is closed by Lemma \ref{17-11-21a}, it follows that $\sigma_{i\frac{\beta}{2}}(a^*) \in \mathcal N_\psi$ and
$$
\lim_{n \to \infty} \Lambda_\psi( \sigma_{-i \frac{\beta}{2}}(a_n)^*) = \Lambda_\psi\left(\sigma_{i \frac{\beta}{2}}(a^*)\right) .
$$
Combined with \eqref{25-11-21e} and Lemma \ref{24-11-21k} we get
\begin{align*}
&\psi(a^*a) = \left<\Lambda_\psi(a),\Lambda_\psi(a)\right> \\
& = \lim_{n \to \infty} \left<\Lambda_\psi(a_n),\Lambda_\psi(a_n)\right> \\
& = \lim_{n \to \infty} \left< \Lambda_\psi( \sigma_{-i \frac{\beta}{2}}(a_n)^*), \Lambda_\psi( \sigma_{-i \frac{\beta}{2}}(a_n)^*)\right> 
\\
& = \left<\Lambda_\psi\left(\sigma_{i \frac{\beta}{2}}(a^*)\right), \Lambda_\psi\left(\sigma_{i \frac{\beta}{2}}(a^*)\right)\right> \\
&
= \psi(\sigma_{i \frac{\beta}{2}}(a^*)^*\sigma_{i \frac{\beta}{2}}(a^*)) = \psi(\sigma_{-i \frac{\beta}{2}}(a)\sigma_{-i \frac{\beta}{2}}(a)^*) .
\end{align*}
We have now established the desired equality in (1) when $a \in \mathcal N_\psi \cap D(\sigma_{-i\frac{\beta}{2}})$. If $a \in D(\sigma_{-i \frac{\beta}{2}}) \backslash \mathcal N_\psi$ it follows that $ \psi(\sigma_{-i \frac{\beta}{2}}(a) \sigma_{-i \frac{\beta}{2}}(a)^*) = \infty$. Indeed, if 
$$
\psi(\sigma_{-i \frac{\beta}{2}}(a) \sigma_{-i \frac{\beta}{2}}(a)^*)< \infty
$$ 
we have that $\sigma_{-i \frac{\beta}{2}}(a)^* \in \mathcal N_\psi$ and hence $\sigma_{i \frac{\beta}{2}}(a^*) \in \mathcal N_\psi$ by Lemma \ref{24-11-21k}. But $\sigma_{i \frac{\beta}{2}}(a^*) \in D(\sigma_{-i \frac{\beta}{2}})$ by Lemma \ref{23-11-21} and hence, by what we have just established and Lemma \ref{23-11-21},
$$
\infty > \psi( \sigma_{i \frac{\beta}{2}}(a^*)^* \sigma_{i \frac{\beta}{2}}(a^*)) = \psi(\sigma_{-i \frac{\beta}{2}}(\sigma_{i \frac{\beta}{2}}(a^*)) \sigma_{-i \frac{\beta}{2}}(\sigma_{i \frac{\beta}{2}}(a^*))^*)  = \psi(a^*a) ;
$$
a contradiction. Thus the equality in (1) is valid in all cases where it makes sense, namely for all $a \in  D(\sigma_{-i\frac{\beta}{2}})$.
\qed

\bigskip
In the proof of (4) $\Rightarrow$ (1) we only used condition (4) for elements $a,b \in \mathcal M^\sigma_\psi$. Therefore conditions (1)-(4) in Theorem \ref{24-11-21d} are also equivalent to 

\begin{enumerate}
\item[(4')] For all $a,b \in \mathcal M_\psi^\sigma$ there is a continuous function $f: {\mathcal D_\beta} \to \mathbb C$ which is holomorphic in the interior $\mathcal D_\beta^0$ of ${\mathcal D_\beta}$ and has the property that
\begin{itemize}
\item $f(t) = \psi(b \sigma_t(a)) \ \ \forall t \in \mathbb R$, and
\item $f(t+i\beta) = \psi(\sigma_t(a)b) \ \ \forall t \in \mathbb R$.
\end{itemize}
\end{enumerate}

 Some choices were made in the formulation of condition (2),(3) and now (4') of Theorem \ref{24-11-21d}. Specifically, the algebra $\mathcal M^\sigma_\psi$ was chosen because it has the nice properties that it is a $*$-algebra which is invariant under $\sigma_z$ for all $z \in \mathbb C$ on which $\psi$ is an everywhere defined linear functional. However, in specific cases it can be hard to identify $\mathcal M^\sigma_\psi$, and in applications it is often important to deduce that a weight is a $\beta$-KMS weight because it satisfies the identity in (2) or (3) of Theorem \ref{24-11-21d} for elements in a smaller set. The next section is devoted to the proof of a result which seems optimal in this respect.

 \subsection{The GNS triple of a KMS weight}\label{GNS-KMS}
Let $\sigma$ be a flow on the $C^*$-algebra $A$.
In order to determine when a GNS representation $(H,\Lambda,\pi)$ of $A$ is isomorphic to the GNS-triple of a KMS weight for $\sigma$, we consider the following conditions:
\begin{itemize}
\item[(A)] $\pi$ is non-degenerate, i.e $\Span \left\{\pi(a)\Lambda(b) : \ a \in A, b \in D(\Lambda)\right\}$ is dense in $H$,
\item[(B)] $\Lambda : D(\Lambda) \to H$ is closed, 
\item[(C)] $\sigma_t(D(\Lambda)) = D(\Lambda)$ for all $t \in \mathbb R$,
\item[(D)] there is a strongly continuous unitary group representation $(U_t)_{t \in \mathbb R}$ of $\mathbb R$ on $H$ such that $U_t\Lambda(a) = \Lambda(\sigma_t(a))$ for all $t \in \mathbb R$ and all $a \in D(\Lambda)$, and
\item[(E)] there is a conjugate linear isometry $J : H \to H$ such that $J\Lambda(a) = \Lambda(\sigma_{-i \frac{\beta}{2}}(a)^*)$ for all $a \in \mathcal M_\Lambda^\sigma$, where
$$
\mathcal M_\Lambda^\sigma := \left\{ a \in D(\Lambda) \cap  D(\Lambda)^* \cap \mathcal A_\sigma: \ \sigma_z(a) \in D(\Lambda) \cap  D(\Lambda)^* \ \forall z \in \mathbb C \right\}. 
$$
\end{itemize}
Note that $\mathcal M^\sigma_\Lambda$ is a $*$-subalgebra of $A$, invariant under $\sigma_z$ for all $z \in \mathbb C$.

\begin{lemma}\label{17-03-22d} Assume that (A) holds. Then $\mathcal M^\sigma_\Lambda$ is a core for $\Lambda : D(\Lambda) \to H$.
\end{lemma}
\begin{proof} Let $a \in D(\Lambda)$ and let $\{e_j\}_{j \in I}$ be the approximate unit for $A$ in $D(\Lambda)$ from Lemma \ref{16-12-21a}. Then $e_ja \in D(\Lambda) \cap D(\Lambda)^*$ and $R_k(e_ja) \in D(\Lambda) \cap D(\Lambda)^* \cap \mathcal A_\sigma$ by Lemma \ref{17-11-21m} and Lemma \ref{24-11-21}. Then $\sigma_z(R_k(e_ja)) \in \mathcal A_\sigma$ by Lemma \ref{25-11-21} and $\sigma_z(R_k(e_ja)) \in D(\Lambda) \cap D(\Lambda)^* $ for all $z \in \mathbb C$ by Lemma \ref{17-11-21m}. Hence $R_k(e_ja) \in \mathcal M_\Lambda^\sigma$. It follows from Lemma \ref{24-11-21} that $\lim_{k \to \infty} R_k(e_ja)= e_ja$ and from Lemma \ref{16-12-21a} that $\lim_{j \to \infty} e_ja = a$. By Lemma \ref{24-11-21g} 
$$
\lim_{k \to \infty} \Lambda(R_k(e_ja)) = \Lambda(e_ja) = \pi(e_j)\Lambda(a).
$$ 
This completes the proof because $\lim_{j \to \infty} \pi(e_j)\Lambda(a) = \Lambda(a)$ by Lemma \ref{16-12-21a}.
\end{proof}

\begin{lemma}\label{09-02-22x} Let $(H,\Lambda,\pi)$ be a GNS representation of $A$ with the properties (A), (B), (C) and (D). Let $\beta \in \mathbb R$. There is a net $\{u_j\}_{j \in I}$ in $A$ such that
\begin{itemize}
\item[(a)] $u_j \in \mathcal M_\Lambda^\sigma$ for all $j \in I$,
\item[(b)] $\sup_{j \in I} \left\|u_j\right\| < \infty$,
\item[(c)] $\lim_{j \to \infty} u_j a = \lim_{j \to \infty} a u_j = a$ for all $a \in A$, and
\item[(d)] $\left\{\sigma_{-i \frac{\beta}{2}}(u_j)\right\}_{j \in I}$ is an increasing approximate unit for $A$ and $0 \leq \sigma_{-i \frac{\beta}{2}}(u_j) \leq 1$ for all $j$.
\end{itemize}
If (E) also holds there is, in addition, a net $\{\rho_j\}_{j \in I}$ in $B(H)$, indexed by the same directed set $I$, such that
\begin{itemize}
\item[(e)] $\left\|\rho_j\right\| \leq 1$ for all $j \in I$,
\item[(f)] $\rho_j\Lambda(a) = \pi(a) \Lambda(u_j)$ for all $j \in I$ and all $ a \in D(\Lambda)$, and
\item[(g)] $\lim_{j \to \infty} \rho_j = 1$ in the strong operator topology.
\end{itemize}
\end{lemma}
\begin{proof} By Lemma \ref{16-12-21a} there is an increasing approximate unit $\{e_j\}_{j \in I}$ for $A$ contained in $D(\Lambda)$ such that $0 \leq e_j \leq 1$ for all $j$. In particular, $e_j \in D(\Lambda) \cap D(\Lambda)^*$. Set 
$$
u_j := \sigma_{i\frac{\beta}{2}}(R_1(e_j)).
$$
It follows from Lemma \ref{24-11-21} and Lemma \ref{17-11-21m} that $\sigma_z(u_j) \in D(\Lambda) \cap \mathcal A_\sigma$ for all $z \in \mathbb C$. Since $u_j^* = \sigma_{-i\frac{\beta}{2}}(R_1(e_j))$ it follows in the same way that $\sigma_z(u_j) \in D(\Lambda)^*  \cap \mathcal A_\sigma$ for all $z \in \mathbb C$. Thus (a) holds. It follows from \eqref{12-05-22} that $\left\|u_j\right\| \leq \frac{1}{\sqrt{\pi}} \int_\mathbb R e^{-s^2 + \frac{\beta^2}{4}} \ \mathrm d s$, showing that (b) holds. It follows from \eqref{12-05-22} and Lemma \ref{28-02-22} in Appendix \ref{integration} that
\begin{equation}\label{18-01-22}
\lim_{j \to \infty} \sigma_z(R_1(e_j)) a = \lim_{j \to \infty} a\sigma_{z}(R_1(e_j)) = a \frac{1}{\sqrt{\pi}} \int_\mathbb R e^{-(s-z)^2} \ \mathrm d s 
\end{equation}
for all $z \in \mathbb C$ since $\{e_j\}$ is an
approximate unit for $A$. The function $z \mapsto \frac{1}{\sqrt{\pi}} \int_\mathbb R e^{-(s-z)^2} \ \mathrm d s $ is entire analytic\footnote{This is a fact from calculus, but it follows also by applying Lemma \ref{18-11-21} to the trivial flow on $\mathbb C$.} and since
$$
\frac{1}{\sqrt{\pi}} \int_\mathbb R e^{-(s-t)^2} \ \mathrm d s = \frac{1}{\sqrt{\pi}} \int_\mathbb R e^{-s^2} \ \mathrm d s = 1
$$
for all $t \in \mathbb R$, we conclude that $\frac{1}{\sqrt{\pi}} \int_\mathbb R e^{-(s-z)^2} \ \mathrm d s =1$ for all $z \in \mathbb C$. Applied with $z = i \frac{\beta}{2}$ we obtain (c). For (d) note
that $\sigma_{-i \frac{\beta}{2}}(u_j) = R_1(e_j)$ increases with $j$ because $R_1$ is a positive linear operator. The fact that $\lim_{j \to \infty}\sigma_{-i \frac{\beta}{2}}(u_j)a= \lim_{j \to \infty} R_1(e_j)a = a$ follows as above.

Assume now that also (E) holds. For $a \in \mathcal M_\Lambda^\sigma$ we find by using condition (E) that
\begin{align*}
&  \left\|\Lambda(au_j)\right\| = \left\|\Lambda\left( \sigma_{-i \frac{\beta}{2}}(au_j)^*\right)\right\|\\
& = \left\|\Lambda\left(R_1(e_j)\sigma_{-i \frac{\beta}{2}}(a)^*\right) \right\| = \left\|\pi(R_1(e_j)) \Lambda\left(\sigma_{-i \frac{\beta}{2}}(a)^*\right) \right\| \\
& \leq \left\|\Lambda\left(\sigma_{-i \frac{\beta}{2}}(a)^*\right) \right\| = \left\|\Lambda(a)\right\|.
\end{align*}
It follows from Lemma \ref{17-03-22d} that $\Lambda\left(\mathcal M_\Lambda^\sigma\right)$ is dense in $H$. Hence the estimate above shows that we can define $\rho_j \in B(H)$ such that 
\begin{equation}\label{09-02-22b}
\rho_j\Lambda(a) := \Lambda(au_j), \ \ \forall a \in \mathcal M^\sigma_\Lambda, 
\end{equation}
and that $\left\|\rho_j\right\| \leq 1$. Thus (e) holds and (f) holds for $a \in \mathcal M^\sigma_\Lambda$. That (f) also holds when $a \in D(\Lambda)$ follows from Lemma \ref{17-03-22d}. For $a \in  \mathcal M^\sigma_\Lambda $, by condition (E)
\begin{align*}
& \left\|\rho_j \Lambda(a) - \Lambda(a)\right\|^2 = \left\|\Lambda\left(au_j - a\right)\right\|^2  \\
&= \left\|\Lambda\left(\sigma_{-i \frac{\beta}{2}}(au_j)^* -\sigma_{-i \frac{\beta}{2}}(a)^*\right)\right\| =  \left\|\Lambda\left(R_1(e_j)\sigma_{-i \frac{\beta}{2}}(a)^* -\sigma_{-i \frac{\beta}{2}}(a)^*\right)\right\| \\
&=  \left\|\pi(R_1(e_j))\Lambda(\sigma_{-i \frac{\beta}{2}}(a)^*)  - \Lambda(\sigma_{-i \frac{\beta}{2}}(a)^*)\right\|.
\end{align*}
Since $\pi$ is non-degenerate and since $\lim_{j \to \infty} R_1(e_j)b = b$ for all $b \in A$ we conclude first that $\lim_{j \to \infty} \pi(R_1(e_j)) =1$ strongly, and then from the calculation above that (g) holds since $\Lambda(\mathcal M_\Lambda^\sigma)$ is dense in $H$.
\end{proof}

 Two GNS representations, $(H_i,\Lambda_i,\pi_i), \ i =1,2$, of $A$ are \emph{isomorphic} when $D(\Lambda_1) = D(\Lambda_2)$ and there is a unitary $W : H_1\to H_2$ such that 
\begin{itemize} 
\item $ W\Lambda_1(a) = \Lambda_2(a), \ \ \forall a \in  D(\Lambda_1)$, and
\item $W\pi_1(a) = \pi_2(a)W,  \ \ \forall a \in A$.
\end{itemize}

\begin{prop}\label{08-02-22a} Let $(H,\Lambda,\pi)$ be a GNS representation of $A$ and let $\beta \in \mathbb R$. Then $(H,\Lambda,\pi)$ is isomorphic to the GNS-triple of a $\beta$-KMS weight for $\sigma$ if and only if (A), (B), (C), (D) and (E) all hold.
\end{prop}
\begin{proof} Neccesity of the five conditions: (A) and (B) follow from Lemma \ref{08-02-22} and (C) follows because a $\beta$-KMS weight is $\sigma$-invariant. (D) follows from Lemma \ref{17-11-21e}. Let $\psi$ be a $\beta$-KMS weight for $\sigma$. Then $\psi(a^*a) = \psi\left(\sigma_{-i \frac{\beta}{2}}(a)  
\sigma_{-i \frac{\beta}{2}}(a)^*\right)$ for $a \in \mathcal A_\sigma$ by (1) of Theorem \ref{24-11-21d} and we can therefore define a conjugate linear isometry $J' : \Lambda_\psi( \mathcal M^\sigma_{\Lambda_\psi})  \to  \Lambda_\psi(\mathcal M^\sigma_{\Lambda_\psi})$ such that $J'\Lambda_\psi(a) = \Lambda_\psi\left(\sigma_{-i\frac{\beta}{2}}(a)^*\right)$. Since $\Lambda_\psi(\mathcal M^\sigma_{\Lambda_\psi})$ is dense in $H$ by Lemma \ref{17-03-22d}, it follows that $J'$ extends by continuity to a conjugate linear isometry $J': H_\psi \to H_\psi$. Hence if $(H,\Lambda,\pi)$ is isomorphic to the GNS-triple of $\psi$, and $W : H \to H_\psi$ is the associated unitary, the map $J := W^*J'W$ will have the property required in (E). 

Sufficiency: Assume (A) through (E) all hold. Let $\{u_j\}_{j \in I}$ and $\{\rho_j\}_{j \in I}$ be the two nets from Lemma \ref{09-02-22x}.
Define $\omega_j : A \to \mathbb C$ such that
$$
\omega_j(a) := \left< \pi(a)\Lambda(u_j),\Lambda(u_j)\right> .
$$ 
Then $\omega_j \in A^*_+$. When $a \in D(\Lambda)$,
\begin{align*}
&\omega_j(a^*a)  = \left< \pi(a)\Lambda(u_j),\pi(a)\Lambda(u_j)\right> =  \left< \Lambda(au_j),\Lambda(au_j)\right>
 \\
 & = \left<\rho_j^2\Lambda(a),\Lambda(a)\right> \leq \left<\Lambda(a),\Lambda(a)\right>,
\end{align*}
showing that $\omega_j \in \mathcal F_\Lambda$. Let $\psi : A^+ \to [0,\infty]$ be the weight of the GNS-triple $(H,\Lambda,\pi)$, cf. Section \ref{GNStriple}.
 Let $a \in D(\Lambda)$. Then 
 \begin{equation}\label{08-02-22d}
 \psi(a^*a) = \sup_{\omega \in \mathcal F_\Lambda}\omega(a^*a) \leq \left<\Lambda(a),\Lambda(a)\right> < \infty ,
 \end{equation} 
 showing that $ D(\Lambda) \subseteq\mathcal N_{\psi}$.  Let $a \in \mathcal N_{\psi}$. Then $au_{j} \in D(\Lambda)$ since $u_{j} \in D(\Lambda)$ and $D(\Lambda)$ is a left ideal. It follows from (c) of Lemma \ref{09-02-22x} that $\lim_{j \to \infty} au_{j} = a$.  Note that
 \begin{equation}\label{14-09-23}
\left\|\Lambda(au_{j})\right\|^2 = \omega_j(a^*a) \leq {\psi}(a^*a) 
\end{equation} 
for all $j$. Since $\lim_{j \to \infty} au_{j} = a$ we can pick out a sequence $j_1 \leq j_2 \leq  \cdots $ in $I$ such that the sequence $a_n := au_{j_n}, n \in \mathbb N$, has the property that $\left\|a_n - a\right\| \leq \frac{1}{k}$ for all $n \geq k$ and all $k \in \mathbb N$. It follows from \eqref{14-09-23} that
$$
\sup_{n \geq k}\left\|\Lambda(a_n)\right\| \leq {\psi}(a^*a)^{\frac{1}{2}}.
$$
The weak* closure $\overline{\co} \{ \Lambda(a_n) : \ n \geq k\}$ of $\{ \Lambda(a_n) : \ n \geq k\}$ is therefore compact in the weak* topology of $H^* = H$, and the intersection
$$
\bigcap_k \overline{\co} \{ \Lambda(a_n) : \ n \geq k\}
$$
is not empty. We let $\eta$ be an element of this intersection. By the self-duality of a Hilbert space the weak* topology is the same as the weak topology and hence $\overline{\co} \{ \Lambda(a_n) : \ n \geq k\}$ is also the norm closure of $\co \{ \Lambda(a_n) : \ n \geq k\}$. 
Using the linearity of $\Lambda$ we can therefore construct $b_k \in \co \{a_n : \ n \geq k\} \subseteq D(\Lambda)$ such that $\|b_k - a\| \leq \frac{1}{k}$ and $\left\|\Lambda(b_k) - \eta\right\| \leq \frac{1}{k}$. Since $\Lambda$ is closed this implies that $a \in D(\Lambda)$, and we have therefore shown that $\mathcal N_\psi = D(\Lambda)$. For $a \in D(\Lambda)$ it follows from (g) of Lemma \ref{09-02-22x} that
\begin{align*}
&\lim_{j \to \infty} \omega_j(a^*a) = \lim_{j \to \infty} \left< \Lambda(au_j),\Lambda(au_j)\right> \\
&= \lim_{j \to \infty} \left<\rho_j\Lambda(a), \rho_j\Lambda(a)\right> = \left<\Lambda(a),\Lambda(a)\right> ,
\end{align*}
proving that 
\begin{equation}\label{29-03-23}
\psi(a^*a) = \sup_{\omega \in \mathcal F_\Lambda} \omega(a^*a) = \left< \Lambda(a),\Lambda(a)\right>. 
\end{equation}
Hence
$$	
\left\|\Lambda_\psi(a)\right\|^2 = \psi(a^*a) = \left<\Lambda(a),\Lambda(a)\right>  = \left\|\Lambda(a)\right\|^2 .
$$
It follows that there is a unitary $W : H  \to H_\psi$ such that $W\Lambda(a) = \Lambda_\psi(a)$ for all $a \in \mathcal N_\psi = D(\Lambda)$. Since
$$
W\pi(a) \Lambda(b) = W \Lambda(ab) = \Lambda_\psi(ab) = \pi_{\psi}(a) \Lambda_\psi(b) = \pi_{\psi}(a)W\Lambda(b)
$$
for all $b \in D(\Lambda)$ we conclude that $W\pi(a) = \pi_\psi(a)W$ for all $a \in A$, and hence that $(H,\Lambda,\pi)$ is isomorphic to the GNS-triple $(H_\psi,\Lambda_\psi,\pi_\psi)$ of $\psi$. It remains to show that $\psi$ is a $\beta$-KMS weight for $\sigma$. First of all, $\psi$ is densely defined because $D(\Lambda)$ is dense in $A$ and $\psi(a^*a) = \left<\Lambda(a),\Lambda(a)\right> < \infty$ when $a \in D(\Lambda)$. Let $t \in \mathbb R$ and $\omega\in \mathcal F_\Lambda$. It follows from condition (C) and (D) that
\begin{align*}
&\omega\circ \sigma_t(a^*a) = \omega(\sigma_t(a)^*\sigma_t(a)) \leq \left< \Lambda(\sigma_t(a)),\Lambda(\sigma_t(a))\right>\\
& = \left<U_t\Lambda(a),U_t\Lambda(a)\right> = \left<\Lambda(a),\Lambda(a)\right>
\end{align*}
for all $a \in D(\Lambda)$, showing that $\omega \circ \sigma_t \in \mathcal F_\Lambda$. Since this holds for all $t \in \mathbb R$ and all $\omega \in \mathcal F_\Lambda$ we conclude that 
$$
\psi(\sigma_t(a)) = \sup_{\omega \in \mathcal F_\Lambda} \omega \circ \sigma_t(a) = \sup_{\omega \in \mathcal F_\Lambda} \omega(a) = \psi(a)
$$
for all $a \in A^+$ and all $t$. That is, $\psi$ is $\sigma$-invariant. Let $a \in \mathcal M_{\Lambda_\psi}^\sigma=  \mathcal M_\Lambda^\sigma$. Then
\begin{align*}
& {\psi}(a^*a) = \left<\Lambda_\psi(a),\Lambda_\psi(a)\right>
 =\left<\Lambda(a),\Lambda(a)\right> = \left\|\Lambda(a)\right\|^2 \\
& = \left\| \Lambda\left( \sigma_{-i \frac{\beta}{2}}(a)^*\right)\right\|^2 = \left< \Lambda\left( \sigma_{-i \frac{\beta}{2}}(a)^*\right),\Lambda\left( \sigma_{-i \frac{\beta}{2}}(a)^*\right)\right> \\
& = \left< \Lambda_\psi\left( \sigma_{-i \frac{\beta}{2}}(a)^*\right),\Lambda_\psi\left( \sigma_{-i \frac{\beta}{2}}(a)^*\right)\right>  = {\psi}\left(  \sigma_{-i \frac{\beta}{2}}(a) \sigma_{-i \frac{\beta}{2}}(a)^*\right).
\end{align*}
Since $\mathcal M^\sigma_\psi \subseteq  \mathcal M_{\Lambda_\psi}^\sigma$ it follows from Theorem \ref{24-11-21d} that $\psi$ is a $\beta$-KMS weight.
\end{proof}

 \begin{lemma}\label{09-02-22} Let $(H,\Lambda,\pi)$ be a GNS representation of $A$ with the properties (B), (C) and (D). The map $\Lambda \circ \sigma_z : \mathcal M^\sigma_\Lambda \to H$ is closable for all $z \in \mathbb C$.
 \end{lemma}

 \begin{proof} Let $\{a_n\}$ be a sequence in $\mathcal M^\sigma_\Lambda$ such that $\lim_{n \to \infty} a_n = 0$ and 
 $$
 \lim_{n \to \infty} \Lambda\left( \sigma_z(a_n)\right) = \eta \in H.
 $$
 Let $k \in \mathbb N$. Then $\lim_{n \to \infty} R_k(a_n) = 0$ and by using Lemma \ref{17-11-21i} and Lemma \ref{28-02-22} we find that
\begin{align*}
&\lim_{n \to \infty} \sigma_z\left(R_k(a_n)\right) =\lim_{n \to \infty}  \sqrt{\frac{k}{\pi}} \int_\mathbb R e^{-k(s-z)^2} \sigma_s(a_n) \ \mathrm d s = 0 .
\end{align*}
Similarly, by using first \eqref{12-05-22a} and then Lemma \ref{17-11-21m} and Lemma \ref{28-02-22}, we find that 
\begin{align*}
&\lim_{n \to \infty} \Lambda\left(  \sigma_z\left(R_k(a_n)\right)\right) = \lim_{n \to \infty} \sqrt{\frac{k}{\pi}} \int_\mathbb R e^{-ks^2} U_s \Lambda\left(\sigma_z(a_n)\right) \ \mathrm d s
 \\
 &=  \sqrt{\frac{k}{\pi}} \int_\mathbb R e^{-ks^2} U_s \eta  \ \mathrm d s ,
\end{align*}
and conclude therefore that $\sqrt{\frac{k}{\pi}} \int_\mathbb R e^{-ks^2} U_s \eta  \ \mathrm d s  = 0$ since $\Lambda$ is closed. Since
$$
\lim_{k \to \infty}  \sqrt{\frac{k}{\pi}} \int_\mathbb R e^{-ks^2} U_s \eta  \ \mathrm d s = \eta
$$
by Lemma \ref{24-11-21}, it follows that $\eta =0$.
 \end{proof}

\begin{lemma}\label{09-02-22d} Let $(H,\Lambda,\pi)$ be a GNS representation of $A$ with the properties (A), (B), (C) and (D). Let $\beta \in \mathbb R$. Condition (E) is equivalent to the following:
\begin{itemize}
\item[(F)] There is a subspace $S\subseteq \mathcal M_\Lambda^\sigma$ such that
\begin{itemize}
\item for all $ a \in \mathcal M_\Lambda^\sigma$ there is a sequence $\{s_n\}$ in $S$ such that $\lim_{n \to \infty} s_n =a$ and $\lim_{n \to \infty} \Lambda(s_n) = \Lambda(a)$, and
\item $\left\| \Lambda(s)\right\| = \left\|\Lambda\left(\sigma_{-i \frac{\beta}{2}}(s)^*\right)\right\|, \ \ \forall s \in S$.
\end{itemize}
\end{itemize}
\end{lemma}
\begin{proof} Assume (F). Let $a \in \mathcal M_\Lambda^\sigma$. By assumption there is a sequence $\{s_n\}$ in $S$ such that $\lim_{n \to \infty} s_n = a$ and $\lim_{n \to \infty} \Lambda(s_n) =\Lambda(a)$. Since
\begin{align*}
&\left\| \Lambda \left( \sigma_{i \frac{\beta}{2}}(s_n^*)\right) - \Lambda \left( \sigma_{i \frac{\beta}{2}}(s_m^*)\right)\right\|= \left\| \Lambda\left(\sigma_{-i \frac{\beta}{2}}(s_n-s_m)^*\right)\right\| \\
&= \left\|\Lambda(s_n) - \Lambda(s_m)\right\|,
\end{align*}
it follows that $\left\{ \Lambda \left(\sigma_{i \frac{\beta}{2}}(s_n^*)\right)\right\}$ is a Cauchy sequence and hence convergent in $H$. Since $\lim_{n \to \infty} s_n^* = a^*$ it follows from Lemma \ref{09-02-22} that $\lim_{n \to \infty} \Lambda\left(\sigma_{i \frac{\beta}{2}}(s_n^*)\right) = \Lambda\left(\sigma_{i \frac{\beta}{2}}(a^*)\right)$. Therefore
\begin{align*}
&\left\|\Lambda(a)\right\| = \lim_{n \to \infty} \left\|\Lambda(s_n)\right\| = \lim_{n \to \infty} \left\|\Lambda\left(\sigma_{-i \frac{\beta}{2}}(s_n)^*\right)\right\| \\
&= \lim_{n \to \infty} \left\|\Lambda\left(\sigma_{i \frac{\beta}{2}}(s_n^*)\right)\right\| = \left\|\Lambda\left(\sigma_{i \frac{\beta}{2}}(a^*)\right)\right\| = \left\|\Lambda\left(\sigma_{-i \frac{\beta}{2}}(a)^*\right)\right\| .
\end{align*}
Since $\Lambda(\mathcal M_\Lambda^\sigma)$ is dense in $H$ by Lemma \ref{17-03-22d}, this gives us a conjugate linear isometry $J : H \to H$ defined 
such that $J \Lambda(a) = \Lambda\left(\sigma_{-i \frac{\beta}{2}}(a)^*\right)$ for all $a \in \mathcal M_\Lambda^\sigma$; that is, (E) holds. That (E) implies (F) is trivial.
\end{proof}

 \begin{thm}\label{12-12-13} (Kustermans, \cite{Ku1}) Let $\psi$ be a non-zero densely defined weight on $A$ which is $\sigma$-invariant in the sense that $\psi \circ \sigma_t = \psi$ for all $t \in \mathbb R$, and let $\beta \in \mathbb R$ be a real number. Let $S$ be a subspace of $\mathcal M^\sigma_\psi$ with the property that for any element $a \in \mathcal M^\sigma_\psi$ there is a sequence $\{a_n\}$ in $S$ such that $\lim_{n \to \infty} a_n = a$ and $ \lim_{n \to \infty} \Lambda_\psi(a_n) = \Lambda_\psi(a)$.
The following conditions are equivalent:
\begin{enumerate}
\item[(1)] $\psi$ is a $\beta$-KMS weight for $\sigma$.
\item[(2)]  $\psi(a^*a) = \psi\left( \sigma_{-i \frac{\beta}{2}}(a) \sigma_{-i \frac{\beta}{2}}(a)^*\right) \ \ \forall a \in S$. 
\item[(3)] $\psi(b^*a) = \psi(a\sigma_{i\beta}(b^*)) \ \ \forall a,b \in S$.
\end{enumerate}
\end{thm}
\begin{proof} The implication (1) $\Rightarrow$ (2) follows immediately from Theorem \ref{24-11-21d} and the implication (2) $\Rightarrow$ (3) results from the following careful inspection of the polarization argument in the proof of the same implication of Theorem \ref{24-11-21d}: Let $a,b \in S$. Then $b^*a$ and $(a+ i^kb)^*(a+ i^kb)  \in \mathcal M_\psi$ for all $k \in \mathbb N$ by (b) of Lemma \ref{04-11-21n}, and \eqref{polar} and Lemma \ref{10-11-21} can therefore be applied to conclude that
\begin{align*}
& \psi(b^*a) = \frac{1}{4} \sum_{k=1}^4 i^k \psi((a+ i^kb)^*(a+i^kb)) .
\end{align*}
Since we assume (2) it follows that
\begin{align*}
& \psi(b^*a) = \frac{1}{4} \sum_{k=1}^4 i^k \psi\left(\sigma_{-i \frac{\beta}{2}}(a+i^kb) \sigma_{-i \frac{\beta}{2}}(a+i^kb)^*  \right) .
\end{align*}
Now, (2) implies also that $\sigma_{-i \frac{\beta}{2}}(a+i^kb)\sigma_{-i \frac{\beta}{2}}(a+i^kb)^* \in \mathcal M_\psi$ and we get 
\begin{align*}
& \psi(b^*a) = \psi\left(\frac{1}{4} \sum_{k=1}^4 i^k \sigma_{-i \frac{\beta}{2}}(a+i^kb) \sigma_{-i \frac{\beta}{2}}(a+i^kb)^*  \right) .
\end{align*}
Note that 
\begin{align*}
&\frac{1}{4} \sum_{k=1}^4 i^k (\sigma_{-i \frac{\beta}{2}}(a+i^kb) \sigma_{-i \frac{\beta}{2}}(a+i^kb)^* = \sigma_{-i \frac{\beta}{2}}(a)\sigma_{-i \frac{\beta}{2}}(b)^* .\\ 
\end{align*}
It follows from Lemmas \ref{24-11-21k}, \ref{25-11-21}, \ref{23-11-21} and \ref{25-11-21} that
\begin{equation}\label{09-02-22h}
 \sigma_{-i \frac{\beta}{2}}(a)\sigma_{-i \frac{\beta}{2}}(b)^* =   \sigma_{-i \frac{\beta}{2}}\left(a\sigma_{i\frac{\beta}{2}}(\sigma_{-i \frac{\beta}{2}}(b)^*)\right) = \sigma_{-i \frac{\beta}{2}}\left(a\sigma_{i\beta}(b^*)\right) ,
\end{equation}
and we conclude therefore from Lemma \ref{18-11-21kx} that $\psi(b^*a) =  \psi\left(a\sigma_{i\beta}(b^*)\right)$; that is, (3) holds. The implication (3) $\Rightarrow$ (2) follows by reversing the arguments just given: Assuming (3), Lemma \ref{18-11-21kx} and \eqref{09-02-22h} gives (2) by taking $b = a$. Finally, (2) $\Rightarrow$ (1): Let $(H_\psi,\Lambda_\psi,\pi_\psi)$ be the GNS-triple of $\psi$. This GNS representation of $A$ satisfies conditions (A) and (B) by Lemma \ref{08-02-22}, (C) because $\psi$ is $\sigma$-invariant and (D) by Lemma \ref{17-11-21e}. If $a \in \mathcal N_\psi$ it follows from Corollary \ref{15-02-22} that there is a sequence $\{a_n\}$ in $\mathcal M^\sigma_\psi$ such that $\lim_{n \to \infty} a_n = a$ and $\lim_{n \to \infty} \Lambda_\psi(a_n) = \Lambda_\psi(a)$. By using the assumption about $S$ we can arrange that $a_n \in S$ for all $n$. Since we assume (2) this shows that the triple $(H_\psi,\Lambda_\psi,\pi_\psi)$ has property (F) of Lemma \ref{09-02-22d} and it follows therefore from Lemma \ref{09-02-22d} that it also has property (E). By Proposition \ref{08-02-22a} this means that it is isomorphic to the GNS-triple of a $\beta$-KMS weight for $\sigma$. Since the weights of isomorphic GNS  representations are the same, this means that the weight of $(H_\psi,\Lambda_\psi,\pi_\psi)$ is a $\beta$-KMS weight and hence $\psi$ is a $\beta$-KMS weight by Lemma \ref{06-02-22e}.
\end{proof}

For later use we record the following lemmas.

\begin{lemma}\label{24-06-22b} Let $\sigma$ be flow on $A$ and $\psi$ a $\beta$-KMS weight for $\sigma$. Let $a,b \in \mathcal N_\psi \cap D(\sigma_{-i \frac{\beta}{2}})$. Then $ \sigma_{-i \frac{\beta}{2}}(a)\sigma_{-i \frac{\beta}{2}}(b)^* \in \mathcal M_\psi$ and 
$$
\psi(b^*a) = \psi(  \sigma_{-i \frac{\beta}{2}}(a)\sigma_{-i \frac{\beta}{2}}(b)^*) .
$$
\end{lemma}
\begin{proof} This is a repetition of some of the arguments from the last proof. Since $b^*a, \ (a+ i^kb)^*(a+ i^kb)  \in \mathcal M_\psi$ for all $k \in \mathbb N$ by (b) of Lemma \ref{04-11-21n}, it follows from  \eqref{polar} and Lemma \ref{10-11-21} that
\begin{align*}
& \psi(b^*a) = \frac{1}{4} \sum_{k=1}^4 i^k \psi((a+ i^kb)^*(a+i^kb)) .
\end{align*}
Since  $a+ i^kb \in D(\sigma_{-i\frac{\beta}{2}})$ we get from (1) in  Theorem \ref{24-11-21d} that
\begin{align*}
& \psi(b^*a) = \frac{1}{4} \sum_{k=1}^4 i^k \psi\left(\sigma_{-i \frac{\beta}{2}}(a+i^kb) \sigma_{-i \frac{\beta}{2}}(a+i^kb)^*  \right) .
\end{align*}
It follows also from (1) in Theorem \ref{24-11-21d} that $\sigma_{-i \frac{\beta}{2}}(a)^*, \ \sigma_{-i \frac{\beta}{2}}(b)^* \in \mathcal N_\psi$ and hence  $\sigma_{-i \frac{\beta}{2}}(a)\sigma_{-i \frac{\beta}{2}}(b)^* \in \mathcal M_\psi$ by (b) of Lemma \ref{04-11-21n}. We can therefore use Lemma \ref{10-11-21} to conclude that
\begin{align*}
&\frac{1}{4} \sum_{k=1}^4 i^k \psi\left(\sigma_{-i \frac{\beta}{2}}(a+i^kb) \sigma_{-i \frac{\beta}{2}}(a+i^kb)^*  \right) \\
&= \frac{1}{4} \sum_{k=1}^4 i^k \psi\left((\sigma_{-i \frac{\beta}{2}}(a) +i^k \sigma_{-i \frac{\beta}{2}}(b)) (\sigma_{-i \frac{\beta}{2}}(a) +i^k \sigma_{-i \frac{\beta}{2}}(b))^*\right)\\
& = \psi\left( \sigma_{-i \frac{\beta}{2}}(a)\sigma_{-i \frac{\beta}{2}}(b)^*\right) .
\end{align*}
\end{proof}

\begin{lemma}\label{01-10-23a}  Let $\sigma$ be flow on $A$ and $\sigma'$ a flow on $B$. Let $\psi$ be a $\beta$-KMS weight for $\sigma$ and $\alpha : B \to A$ a $*$-isomorphism such that $\sigma_t\circ \alpha = \alpha \circ \sigma'_t$ for all $t \in \mathbb R$. Then $\psi \circ \alpha$ is a $\beta$-KMS weight for $\sigma'$.  
\end{lemma}
\begin{proof} It is easy to see that $\alpha(\mathcal A_{\sigma'}) = \mathcal A_\sigma$ and that $\sigma_z \circ \alpha = \alpha \circ \sigma'_z$ for all $z \in \mathbb C$. It follows that
\begin{align*}
&\psi\circ \alpha(a^*a) = \psi(\sigma_{- i \frac{\beta}{2}}(\alpha(a)) \sigma_{- i \frac{\beta}{2}}(\alpha(a))^*) \\
&= \psi( \alpha(\sigma'_{-i \frac{\beta}{2}}(a)  \alpha(\sigma'_{-i \frac{\beta}{2}}(a)^*)  = \psi \circ \alpha\left(\sigma'_{-i \frac{\beta}{2}}(a) \sigma'_{-i \frac{\beta}{2}}(a)^*\right)
\end{align*}
for all $a \in \mathcal A_{\sigma'}$. Hence $\psi \circ  \alpha$ satisfies condition (2) of Theorem \ref{24-11-21d} and is therefore $\beta$-KMS weight for $\sigma'$.
\end{proof}

\begin{cor}\label{01-10-23}  Let $\sigma$ be flow on $A$ and $\psi$ a $\beta$-KMS weight for $\sigma$. Let $\alpha \in \Aut A$ be an automorphism of $A$ which commutes with $\sigma$ in the sense that $\sigma_t \circ \alpha = \alpha \circ \sigma_t$ for all $t \in \mathbb R$. Then $\psi \circ \alpha$ is a $\beta$-KMS weight for $\sigma$.
\end{cor}

\begin{notes} The crucial Lemma \ref{17-11-21a} appears in the work of Quaegebeur and Verding, \cite{QV}, as does some of the arguments in the proof of Proposition \ref{08-02-22a}. But otherwise most of the material in this section is taken from Kustermans' paper \cite{Ku1} where he shows that Combes' original definition of a KMS weight, \cite{C2}, which is condition (4) in Theorem \ref{24-11-21d}, is equivalent to condition (1). The proof here of (4) $\Rightarrow$ (1) is the same as Kustermans', but the proof that (1) implies (4) is more direct and it allows us to include the intermediate conditions (2) and (3). Section \ref{GNS-KMS} here is an interpretation of Section 5.1 in \cite{Ku1}. 
\end{notes}

\subsection{KMS states}\label{KMSstatesx}

\begin{defn}\label{21-10-23} A weight $\psi$ on a $C^*$-algebra $A$ is \emph{bounded} when $\psi(a) < \infty$ for all $a \in A^+$ and a \emph{state} when 
$$
\sup \left\{ \psi(a) : \ 0 \leq a \leq 1 \right\} = 1.
$$
\end{defn}

\begin{lemma}\label{21-10-23a} Let $\psi$ be a bounded weight on $A$. There is a positive linear functional $\psi' \in A^*_+$ such that $\psi =\psi'|_{A^+}$.
\end{lemma}
\begin{proof} Note that $\mathcal M_\psi^+ = A^+$ and hence that $\mathcal M_\psi = A$. It follows therefore from Lemma \ref{10-11-21} that there is a linear map $\psi': A \to \mathbb C$ such that $\psi'|_{A^+} = \psi$. Since a positive linear functional on a $C^*$-algebra is continuous, cf. Proposition 2.3.11 in \cite{BR}, it follows that $\psi'\in A^*_+$.
\end{proof}

For a bounded weight $\psi$ on $A$ we will not distinguish between $\psi$ and its unique linear extension to $A$. We note that when $A$ is unital all densely defined weights are bounded:

\begin{lemma}\label{24-09-23b} Let $\psi$ be a non-zero densely defined weight on the unital $C^*$-algebra $A$. Then $0 < \psi(1) < \infty$ and $\psi(1)^{-1}\psi$ is a state.
\end{lemma}
\begin{proof} Since $\psi$ is densely defined there is an $ a \in A^+$ such that $\psi(a) < \infty$ and $\left\|1-a \right\| \leq \frac{1}{2}$. Then $a$ is invertible and $a^{-1} \leq 2$. Hence $1 = a^{-1}a \leq 2a$ and $\psi(1) \leq 2 \psi(a) < \infty$. If $b \in A^+$, $0 \leq b \leq \|b\| 1$ and hence $\psi(1) > 0$ since $\psi$ is not zero.
\end{proof}

Let $\sigma$ be a flow on $A$. A bounded $\beta$-KMS weight for $\sigma$ will be called a \emph{$\beta$-KMS functional}, or just a \emph{KMS functional} when we don't need to specify $\beta$. Similarly, a $\beta$-KMS weight which is also a state will be called a \emph{$\beta$-KMS state} or a \emph{KMS state}. It follows from Lemma \ref{21-10-23a} that a KMS weight $\psi$ for $\sigma$ is bounded if and only if there is a KMS state $\omega$ for $\sigma$ and a positive scalar $\lambda >0$ such that $\lambda \omega|_{A^+} = \psi$, and then from Lemma \ref{24-09-23b} that, when $A$ is unital every KMS weight for $\sigma$ is a scalar multiple of a KMS state.

For positive functionals and states much of the general theory simplifies considerably. For example, since $\mathcal M^\sigma_\psi = \mathcal A_\sigma$ and $\mathcal N_\psi =A$ when $\psi$ is a positive linear functional, we have the following version of Kusterman's theorem, Theorem \ref{24-11-21d}.

\begin{thm}\label{21-11-23b} Let $\sigma$ be a flow on the $C^*$-algebra $A$ and $\omega \in A^*_+$ a non-zero positive linear functional on $A$. Let $\beta\in \mathbb R$. Then $\omega$ is a $\beta$-KMS functional for $\sigma$ if and only if the following equivalent conditions (1)-(4) hold.
\begin{enumerate}
\item[(1)] $\omega(a^*a) = \omega\left( \sigma_{-i \frac{\beta}{2}}(a) \sigma_{-i \frac{\beta}{2}}(a)^*\right) \ \ \forall a \in D\left( \sigma_{-i \frac{\beta}{2}}\right)$.
\item[(2)]  $\omega(a^*a) = \omega\left( \sigma_{-i \frac{\beta}{2}}(a) \sigma_{-i \frac{\beta}{2}}(a)^*\right) \ \ \forall a \in \mathcal A_\sigma$. 
\item[(3)] $\omega(ab) = \omega(b\sigma_{i\beta}(a)) \ \ \forall a,b \in \mathcal A_\sigma$.
\item[(4)] For all $a,b \in A$ there is a continuous function $f: {\mathcal D_\beta} \to \mathbb C$ which is holomorphic in the interior $\mathcal D_\beta^0$ of ${\mathcal D_\beta}$ and has the property that
\begin{itemize}
\item $f(t) = \omega(b \sigma_t(a)) \ \ \forall t \in \mathbb R$, and
\item $f(t+i\beta) = \omega(\sigma_t(a)b) \ \ \forall t \in \mathbb R$.
\end{itemize}
\end{enumerate}
\end{thm}



\chapter{Laca-Neshveyev' theorem}\label{LN}

This chapter is devoted to the proof of a fundamental theorem on the extension of a KMS weight from an invariant hereditary and full sub-$C^*$-algebra.

\bigskip

Let $A$ be a $C^*$-algebra and $B \subseteq A$ a $C^*$-subalgebra of $A$. Recall that $B$ is \emph{hereditary} in $A$ when 
$$
 a \in A, \ b \in B, \ 0 \leq a \leq b  \Rightarrow a \in B,
 $$ 
 and \emph{full} when $\Span ABA$ is dense in $A$. Given weights $\psi$ on $A$ and $\phi$ on $B$ we say that $\psi$ \emph{extends} $\phi$ when $\psi(b) = \phi(b)$ for all $b \in B^+$. In this chapter we prove the following theorem due to M. Laca and S. Neshveyev.

\begin{thm}\label{02-12-21} (Theorem 3.2 iv) in \cite{LN}.) Let $\sigma$ be a flow on the $C^*$-algebra $A$ and let $B \subseteq A$ be a $C^*$-subalgebra of $A$ such that $\sigma_t(B) = B$ for all $t \in \mathbb R$. Assume that $B$ is hereditary and full in $A$. Let $\beta \in \mathbb R$. For every $\beta$-KMS weight $\phi$ for $\sigma$ on $B$ there is a unique $\beta$-KMS weight for $\sigma$ on $A$ which extends $\phi$.
\end{thm}

\section{Proof of Laca-Neshveyev' theorem}\label{LNproof}

In the proof we shall many times use results from the first chapters without explicit reference.

 We fix the setting from Theorem \ref{02-12-21}. To simplify the notation we set
$$
\xi := -i\frac{\beta}{2}  .
$$
Write $\mathcal A_\sigma  \mathcal M_\phi^\sigma$ for the set of elements of $A$ of the form $ab$, where $a \in A$ is entire analytic for $\sigma$ and $b \in B$ is an element of the algebra $\mathcal M^\sigma_\phi \subseteq B$ obtained from the restriction of $\sigma$ to $B$. Note that $\mathcal A_\sigma \mathcal M^\sigma_\phi \subseteq \mathcal A_\sigma$. Let $\mathcal I$ denote the collection of finite subsets $I \subseteq  \mathcal A_\sigma  \mathcal M_\phi^\sigma$ with the property that $\sum_{a \in I} aa^* \leq 1$. When $I \in \mathcal I$ we set
$$
w_I :=  \sum_{a \in I} aa^*.
$$

\begin{lemma}\label{01-12-21x} The collection $\left\{w_I : \ I \in \mathcal I\right\}$ is an approximate unit in $A$ in the following sense: For every finite set $S \subseteq A$ and every $\epsilon > 0$ there is an element $J \in \mathcal I$ such that $\left\|w_Ja -a\right\| \leq \epsilon$ for all $a \in S$.
\end{lemma}
\begin{proof}
 For each finite set $I\subseteq AB$, set
$$
f_I := \sum_{a \in I} aa^* , 
$$
and
$$
e_I  := (\# I) f_I(1 + (\# I)f_I)^{-1}.
$$
As shown in the proof of Proposition 2.2.18 in \cite{BR} the net $\{e_I\}$ consists of positive contractions; i.e. $0 \leq e_I \leq 1$, it increases with $I$; i.e. $I \subseteq J \Rightarrow e_I \leq e_{J}$, and it has the property that $\lim_{I \to \infty} e_Ib  = b$ for all $b \in  AB$. Since $\Span ABA$\ is dense in $A$ it follows that $\lim_{I \to \infty} e_I b = b$ for all $b \in A$. Thus, given a finite set $S \subseteq A$ and an $\epsilon > 0$ there is a finite set $I \subseteq AB$ such that $\left\|e_I b - b\right\| < \frac{\epsilon}{2}$ for all $b \in S$. Note that
$$
e_I = \sum_{x \in I} \sqrt{\# I} \left( 1+ (\# I)f_I\right)^{-\frac{1}{2}} xx^*\sqrt{\# I} \left( 1+ (\# I)f_I\right)^{-\frac{1}{2}} ,
$$
where 
\begin{equation}\label{dragoer}
\sqrt{\# I} \left( 1+ (\# I)f_I\right)^{-\frac{1}{2}} x \in AB.
\end{equation}
Choose $t \in ]0,1[$ such that $\left\|te_Ib-b\right\| \leq \frac{\epsilon}{2}$ for all $b \in S$. Since $\mathcal M_\phi^\sigma$ is dense in $B$ and $\mathcal A_\sigma$ is dense in $A$ it follows that $\mathcal A_\sigma \mathcal M^\sigma_\phi$ is dense in $AB$ and we can therefore approximate the elements from \eqref{dragoer} by elements from $\mathcal A_\sigma \mathcal M_\phi^\sigma$ to get a finite set $J \subseteq \mathcal A_\sigma \mathcal M_\phi^\sigma$, such that 
$$
\left\|te_Ib - w_Jb\right\| \leq \frac{\epsilon}{2}
$$
for all $b \in S$ and
$$
\left\| te_I  - w_J\right\| \leq 1-t.
$$
 Then
$$
w_J=  w_J - te_I +te_I \leq 1-t + t = 1,
$$
showing that $J \in \mathcal I$. This completes the proof since 
$$
\left\|w_Jb-b\right\| \leq \left\|w_Jb-te_Ib\right\| +\|te_Ib -b\| \leq \epsilon
$$
for all $b \in S$.
\end{proof}

Note that $\sigma_z(x)^*a\sigma_z(x) \in \mathcal M_\phi^+$ for all $x \in \mathcal A_\sigma \mathcal M_\phi^\sigma $, all $a \in A^+$ and all $z \in \mathbb C$. For each $I \in \mathcal I$ we can therefore define an element $L_I \in A^*_+$ such that
$$
L_I(a) := \sum_{x \in I} \phi(\sigma_{-\xi}(x)^* a \sigma_{-\xi}(x))  \  \ \ \forall a \in A.
$$
For $a \in A^+$ set
\begin{equation}\label{03-01-22}
\psi(a) := \sup_{I \in \mathcal I} L_I(a) =  \sup_{I \in \mathcal I}\sum_{x \in I} \phi(\sigma_{-\xi}(x)^*a \sigma_{-\xi}(x)) .
\end{equation}

\begin{lemma}\label{01-12-21a} $\psi$ is a densely defined weight on $A$.
\end{lemma}
\begin{proof} Let $a,b \in A^+$. It is clear that $\psi(a+b) \leq \psi(a) + \psi(b)$. To show that $\psi(a) +\psi(b) \leq \psi(a+b)$ we may assume that $\psi(a+b) < \infty$. Since $\psi(a) \leq \psi(a+b)$ and $\psi(b) \leq \psi(a+b)$ this implies that $\psi(a)$ and $\psi(b)$ are both finite. 

 Let $\epsilon > 0$ and choose $I_1,I_3 \in \mathcal I$ such that $\psi(a) \leq L_{I_1}(a) + \epsilon$ and $\psi(b) \leq L_{I_3}(b)+\epsilon$. Let $k \in \mathbb N$ and set $c = R_k(\sqrt{a})$ and $d = R_k(\sqrt{b})$. Thanks to Lemma \ref{24-11-21} we can choose $k $ so large that $L_{I_1}(a) \leq L_{I_1}(c^2) + \epsilon$ and $L_{I_3}(b) \leq L_{I_3}(d^2) + \epsilon$. Note that $c,d \in \mathcal A_\sigma$ by Lemma \ref{24-11-21}.

Let $h \in \mathcal A_\sigma$ and consider an element $I_2 \in \mathcal I$. Then
 \begin{equation}\label{01-12-21d}
\begin{split}
&\sum_{x \in I_1,y \in I_2} \phi(\sigma_{-\xi}(x)^* h^*yy^*h\sigma_{-\xi}(x))\\
& = \sum_{x \in I_1,y \in I_2} \phi(\sigma_\xi(y^*h\sigma_{-\xi}(x)) \sigma_\xi(y^* h\sigma_{-\xi}(x))^*)\\
& = \sum_{x \in I_1,y \in I_2} \phi(\sigma_\xi(y^*)\sigma_\xi(h)xx^*\sigma_{\xi}(h)^*\sigma_\xi(y^*)^*)\\
& = \sum_{x \in I_1,y \in I_2} \phi(\sigma_{-\xi}(y)^*\sigma_\xi(h)xx^*\sigma_{\xi}(h)^*\sigma_{-\xi}(y)).
\end{split}
\end{equation}
Using Lemma \ref{01-12-21x} we choose $I_2$ such that
\begin{align*}
& L_{I_1}(c^2) \leq \sum_{x \in I_1,y \in I_2} \phi(\sigma_{-\xi}(x)^* cyy^*c\sigma_{-\xi}(x)) + \epsilon
\end{align*}
and
\begin{align*}
& L_{I_3}(d^2) \leq \sum_{x \in I_3,y \in I_2} \phi(\sigma_{-\xi}(x)^* dyy^*d\sigma_{-\xi}(x)) + \epsilon .
\end{align*}
It follows from calculation \eqref{01-12-21d} that
\begin{align*}
& \sum_{x \in I_1,y \in I_2} \phi(\sigma_{-\xi}(x)^* cyy^*c\sigma_{-\xi}(x)) = \sum_{x \in I_1,y \in I_2} \phi(\sigma_{-\xi}(y)^*\sigma_{\xi}(c)xx^*\sigma_{\xi}(c)^*\sigma_{-\xi}(y))\\
&\leq  \sum_{y \in I_2} \phi(\sigma_{-\xi}(y)^*\sigma_{\xi}(c)\sigma_{\xi}(c)^*\sigma_{-\xi}(y)).
\end{align*}
Hence
\begin{align*}
& L_{I_1}(c^2) 
\leq \sum_{y \in I_2} \phi(\sigma_{-\xi}(y)^*\sigma_{\xi}(c)\sigma_{\xi}(c)^*\sigma_{-\xi}(y)) + \epsilon \\
\end{align*}
and similarly,
\begin{align*}
& L_{I_3}(d^2) 
\leq \sum_{y \in I_2} \phi(\sigma_{-\xi}(y)^*\sigma_{\xi}(d)\sigma_{\xi}(d)^*\sigma_{-\xi}(y)) + \epsilon .
\end{align*}
Using Lemma \ref{01-12-21x} again we choose $I_4 \in  \mathcal I$ such that
\begin{align*}
&\sum_{y \in I_2} \phi(\sigma_{-\xi}(y)^*\sigma_{\xi}(c)\sigma_{\xi}(c)^*\sigma_{-\xi}(y)^*) \\
&\leq \sum_{u \in I_4,y \in I_2} \phi(\sigma_{-\xi}(y)^*\sigma_{\xi}(c)uu^*\sigma_{\xi}(c)^*\sigma_{-\xi}(y)) + \epsilon 
\end{align*}
and
\begin{align*}
&\sum_{y \in I_2} \phi(\sigma_{-\xi}(y)^*\sigma_{\xi}(d)\sigma_{\xi}(d)^*\sigma_{-\xi}(y)) \\
&\leq \sum_{u \in I_4,y \in I_2} \phi(\sigma_{-\xi}(y)^*\sigma_{\xi}(d)uu^*\sigma_{\xi}(d)^*\sigma_{-\xi}(y)) + \epsilon  .
\end{align*}
The calculation \eqref{01-12-21d} shows that 
\begin{align*}
&\sum_{u \in I_4,y \in I_2} \phi(\sigma_{-\xi}(y)^*\sigma_{\xi}(c)uu^*\sigma_{\xi}(c)^*\sigma_{-\xi}(y)) \\
& =\sum_{u \in I_4,y \in I_2} \phi(\sigma_{-\xi}(u)^*cyy^*c\sigma_{-\xi}(u)) \\
& \leq \sum_{u \in I_4} \phi(\sigma_{-\xi}(u)^*c^2\sigma_{-\xi}(u)) \\
\end{align*}
and
\begin{align*}
&\sum_{u \in I_4,y \in I_2} \phi(\sigma_{-\xi}(y)^*\sigma_{\xi}(d)uu^*\sigma_{\xi}(d)^*\sigma_{-\xi}(y)) \\
& =\sum_{u \in I_4,y \in I_2} \phi(\sigma_{-\xi}(u)^*dyy^*d\sigma_{-\xi}(u)) \\
& \leq \sum_{u \in I_4} \phi(\sigma_{-\xi}(u)^*d^2\sigma_{-\xi}(u)) .
\end{align*}
We conclude therefore that
\begin{align*}
\psi(a) + \psi(b) \leq L_{I_4}(c^2 + d^2) + 4\epsilon.
\end{align*}
It follows from Kadison's inequality, Proposition 3.2.4 in \cite{BR}, that $c^2 + d^2 = R_k(\sqrt{a})^2 + R_k(\sqrt{b})^2 \leq R_k(a+b)$, and hence
\begin{align*}
\psi(a) + \psi(b) \leq L_{I_4}(R_k(a+b)) + 4 \epsilon.
\end{align*}
By definition of $R_k$ there are finite sets of numbers, $\lambda_i \in [0,1], \ t_i \in \mathbb R, \ i =1,2,\cdots, n$, such that $\sum_{i=1}^n\lambda_i = 1$, and the sum
$$
\sum_{i=1}^n \lambda_i \sigma_{t_i}(a+b) 
$$
approximates $R_k(a+b)$, cf. Lemma \ref{01-03-22} in Appendix \ref{integration}. We can therefore choose these sets of numbers such that
$$
L_{I_4}(R_k(a+b)) \leq \sum_{i=1}^n \lambda_i L_{I_4}\circ \sigma_{t_i}(a+b) + \epsilon .
$$
 Since $\phi$ is $\sigma$-invariant we find that
\begin{align*}
&  \sum_{i=1}^n \lambda_i L_{I_4}\circ \sigma_{t_i}(a+b) = \sum_{i=1}^n \lambda_i \sum_{y \in I_4} \phi(\sigma_{-\xi}(y)^*\sigma_{t_i}(a+b) \sigma_{-\xi}(y))\\
& = \sum_{i=1}^n \lambda_i \sum_{y \in I_4} \phi\circ \sigma_{t_i}(\sigma_{-\xi}(\sigma_{-t_i}(y))^*(a+b) \sigma_{-\xi}(\sigma_{-t_i}(y))) \\
& = \sum_{i=1}^n \lambda_i \sum_{y \in I_4} \phi(\sigma_{-\xi}(\sigma_{-t_i}(y))^*(a+b) \sigma_{-\xi}(\sigma_{-t_i}(y)) ) .
\end{align*}
 Note that $\sqrt{\lambda_i} \sigma_{-t_i}(u) \in \mathcal A_\sigma \mathcal M^\sigma_\phi$ for all $i$ and $u$, and
$$
\sum_{i=1}^n \sum_{u \in I_4} \left( \sqrt{\lambda_i}\sigma_{-t_i}(u) \right)\left( \sqrt{\lambda_i}\sigma_{-t_i}(u) \right)^* \leq 1 .
$$
Choose numbers $s_{i,u} \in  ]0,1[$, $u\in I_4$, $i = 1,2,\cdots , n$, close to $1$ such that
$$
s_{i,u}  \sqrt{\lambda_i}\sigma_{-t_i}(u) \neq s_{i',u'}  \sqrt{\lambda_{i'}}\sigma_{-t_{i'}}(u')
$$
when $(i,u) \neq (i',u')$ and $(\sqrt{\lambda_i}\sigma_{-t_i}(u), \sqrt{\lambda_{i'}}\sigma_{-t_{i'}}(u')) \neq (0,0)$. Set
$$
I_5 := \left\{s_{i,u} \sqrt{\lambda_i} \sigma_{-t_i}(u) : \ u \in I_4, \ i = 1,2,\cdots, n\right\} .
$$
Then $I_5 \in \mathcal I$ and if the numbers $s_{i,u}$ are close enough to $1$ we have that
$$
\sum_{i=1}^n \lambda_i \sum_{y \in I_4} \phi(\sigma_{-\xi}(\sigma_{-t_i}(y))^*(a+b) \sigma_{-\xi}(\sigma_{-t_i}(y)) ) \leq L_{i_5}(a+b) + \epsilon .
$$
We conclude therefore that
 $$
 \psi(a) + \psi(b) \leq L_{I_5}(a+b) + 6 \epsilon \leq \psi(a+b) +  6\epsilon .
 $$
It follows that $\psi$ is additive on $A^+$. Since $\psi$ is clearly homogeneous and lower semi-continuous, we have shown that $\psi$ is a weight. To see that $\psi$ is densely defined, let $u,v \in \mathcal A_\sigma, \ b \in \mathcal M_\phi^\sigma$. For each $I \in \mathcal I$ we have that
\begin{align*}
& \sum_{x \in I} \phi(\sigma_{-\xi}(x)^* (ubv)^*(ubv) \sigma_{-\xi}(x)) \\
& \leq \|u\|^2\sum_{x \in I} \phi(\sigma_{-\xi}(x)^*v^*b^*bv \sigma_{-\xi}(x))\\
& = \|u\|^2\sum_{x \in I} \phi(\sigma_\xi(bv \sigma_{-\xi}(x))\sigma_\xi(bv \sigma_{-\xi}(x))^*)\\
& = \|u\|^2\sum_{x \in I} \phi(\sigma_\xi(b)\sigma_\xi(v)xx^*\sigma_\xi(v)^*\sigma_\xi(b)^*)\\
& \leq  \|u\|^2\phi(\sigma_\xi(b)\sigma_\xi(v)\sigma_\xi(v)^*\sigma_\xi(b)^*) ,
\end{align*}
proving that 
$$
\psi((ubv)^*(ubv)) \leq \|u\|^2\phi(\sigma_\xi(b)\sigma_\xi(v)\sigma_\xi(v)^*\sigma_\xi(b)^*)<\infty,
$$ 
and hence that $\Span \mathcal A_\sigma \mathcal M_\phi^\sigma \mathcal A_\sigma \subseteq \mathcal N_\psi$. Since $\Span \mathcal A_\sigma \mathcal M_\phi^\sigma \mathcal A_\sigma $ is dense in $A$, every element of $A^+$ can be approximated by elements of the the form $\{x^*x: \ x \in \mathcal N_\psi\} \subseteq \mathcal M^+_\psi$.
\end{proof}

\begin{lemma}\label{01-12-21e} $\psi(b^*b) = \phi(b^*b)$ for all $b \in \mathcal A_\sigma \cap B$.
\end{lemma}
\begin{proof} Let $b \in \mathcal A_\sigma \cap B$. For $F \in \mathcal I$ we find that 
\begin{align*}
&\sum_{x \in F} \phi(\sigma_{-\xi}(x)^* b^*b \sigma_{-\xi}(x)) = \sum_{x \in F} \phi(\sigma_\xi( b\sigma_{-\xi}(x))\sigma_\xi(b\sigma_{-\xi}(x))^*) \\
& = \sum_{x \in F} \phi(\sigma_\xi(b)xx^*\sigma_\xi(b)^*) \leq \phi(\sigma_\xi(b)\sigma_\xi(b)^*) = \phi(b^*b).
\end{align*}
It follows that $\psi(b^*b) \leq \phi(b^*b)$. In the other direction, let $\epsilon >0$ and $N > 0$ be given. It follows from Lemma \ref{01-12-21x} and the lower semi-continuity of $\phi$ that there is $F\in \mathcal I$ such that
\begin{align*}
&\min\{\phi(b^*b) -\epsilon, N\} \leq \sum_{x \in F} \phi(b^*xx^*b) \\
&=\sum_{x \in F} \phi(\sigma_\xi(x^*b) \sigma_\xi(x^*b)^*)= L_F(\sigma_\xi(b)\sigma_\xi(b)^*) \\
&  \leq \psi(\sigma_\xi(b)\sigma_\xi(b)^*) = \psi(b^*b) .
\end{align*}
We conclude that $\phi(b^*b) \leq \psi(b^*b)$ and hence that $\psi(b^*b) = \phi(b^*b)$.
\end{proof}

\begin{prop}\label{01-12-21} $\psi$ is a $\beta$-KMS weight for $\sigma$.
\end{prop}
\begin{proof} It follows from Lemma \ref{01-12-21e} that $\psi$ is non-zero since $\phi$ is. By using that $\phi$ is $\sigma$-invariant we find that
$$
L_I(\sigma_t(a)) = L_{I_t}(a) ,
$$
where $I_t = \sigma_{-t}(I)$, for all $a \in A^+$, all $t \in \mathbb R$ and all $I \in \mathcal I$. It follows that $\psi$ is $\sigma$-invariant. Let $h \in \mathcal A_\sigma$ and let $F_1 \in \mathcal I$. Let $\epsilon > 0$. It follows from Lemma \ref{01-12-21x} that there is an element $F_2 \in \mathcal I$ such that 
$$
 L_{F_1}(h^*h) - \epsilon \leq \sum_{x \in F_1,y \in F_2} \phi(\sigma_{-\xi}(x)^* h^*yy^*h\sigma_{-\xi}(x)).
$$ 
It follows from the calculation \eqref{01-12-21d} that 
$$
\sum_{x \in F_1,y \in F_2} \phi(\sigma_{-\xi}(x)^* h^*yy^*h\sigma_{-\xi}(x)) \leq L_{F_2}(\sigma_\xi(h)\sigma_\xi(h)^*) \leq \psi(\sigma_\xi(h)\sigma_\xi(h)^*).
$$
Since $\epsilon > 0$ is arbitrary if follows that $L_{F_1}(h^*h) \leq \psi(\sigma_\xi(h)\sigma_\xi(h)^*)$, and since $F_1$ is arbitrary it follows that $\psi(h^*h) \leq \psi(\sigma_\xi(h)\sigma_\xi(h)^*)$. The reverse inequality follows by inserting $\sigma_{-\xi}(h^*) = \sigma_\xi(h)^*$ for $h$ in the last inequality and we conclude therefore that $\psi(h^*h) = \psi(\sigma_\xi(h)\sigma_\xi(h)^*)$ for all $h \in \mathcal A_\sigma$. 
  Hence $\psi$ has property (2) in Kustermans' theorem, Theorem \ref{24-11-21d}, implying that $\psi$ is a $\beta$-KMS weight for $\sigma$.
\end{proof}

\begin{lemma}\label{01-03-22b} Let $\sigma$ be a flow on the $C^*$-algebra $D$. Let $\rho_i : D^+ \to [0,\infty]$, $i =1,2$, be densely defined $\sigma$-invariant weights on $D$. Assume that $\rho_1(d^*d) = \rho_2(d^*d)$ for all $d \in \mathcal A_\sigma$. Then $\rho_1 = \rho_2$.
\end{lemma}
\begin{proof} Assume first that $a \in D^+$ and that $\rho_i(a) < \infty, i = 1,2$. It follows from Lemma \ref{24-11-21g} that 
$$
\rho_i(a) = \left<\Lambda_{\rho_i}(\sqrt{a}), \Lambda_{\rho_i}(\sqrt{a})\right> = \lim_{k \to \infty} \rho_i\left(R_k(\sqrt{a})^2\right), \ \ i = 1,2.
$$
Since $R_k(\sqrt{a}) \in \mathcal A_\sigma \cap D^+$ we conclude that $\rho_1(a) = \rho_2(a)$. It remains now only to show that $\mathcal N_{\rho_1} = \mathcal N_{\rho_2}$. Let $x \in \mathcal N_{\rho_1}$. It follows from Lemma \ref{24-11-21g}, Lemma \ref{17-11-21m} and Lemma \ref{24-11-21} and $x_n := R_n(x)$ gives a sequence $\{x_n\}$ in $\mathcal A_\sigma \cap \mathcal N_{\rho_1}$ such that $\lim_{n \to \infty} x_n =x$ and $\lim_{n \to \infty} \Lambda_{\rho_1}(x_n) = \Lambda_{\rho_1}(x)$. In particular,
$$\rho_1\left((x_n-x_m)^*(x_n-x_m)\right) = \left\|\Lambda_{\rho_1}(x_n) - \Lambda_{\rho_1}(x_m)\right\|^2
$$
is arbitrarily small if $n,m$ both are large enough. By assumption, $\rho_2(x_n^*x_n) = \rho_1(x_n^*x_n) < \infty$, i.e. $x_n \in \mathcal N_{\rho_2}$, and 
\begin{align*}
&\left\|\Lambda_{\rho_2}(x_n) - \Lambda_{\rho_2}(x_m)\right\|^2 =\rho_2\left((x_n-x_m)^*(x_n-x_m)\right)\\
&  = \rho_1\left((x_n-x_m)^*(x_n-x_m)\right)
\end{align*}
is arbitrarily small if $n,m$ both are large enough; i.e. $\left\{\Lambda_{\rho_2}(x_n)\right\}$ is a Cauchy sequence in $H_{\rho_2}$. Since $\Lambda_{\rho_2}$ is closed by Lemma \ref{17-11-21a} it follows that $x \in \mathcal N_{\rho_2}$. Thus $\mathcal N_{\rho_1} \subseteq \mathcal N_{\rho_2}$. By symmetry we must have equality; $\mathcal N_{\rho_1} =\mathcal N_{\rho_2}$.
\end{proof}

\begin{lemma}\label{01-12-21h} $\psi|_B = \phi$.
\end{lemma}
\begin{proof} This follows from Lemma \ref{01-12-21e} and Lemma \ref{01-03-22b}.

\end{proof}

\begin{lemma}\label{02-12-21g} 
Let $\psi'$ be a $\beta$-KMS weight for $\sigma$ such that $\psi'|_B = \phi$. Then $\psi' = \psi$. 
\end{lemma}
\begin{proof} Assume $a \in \mathcal A_\sigma$. For each $F \in \mathcal I$,
\begin{align*}
& \sum_{x \in F} \phi(\sigma_{-\xi}(x)^* a^*a \sigma_{-\xi}(x)) =\sum_{x \in F} \psi'(\sigma_{-\xi}(x)^* a^*a \sigma_{-\xi}(x))\\
& = \sum_{x \in F} \psi'(\sigma_\xi(a)xx^*\sigma_\xi(a)^*) \leq   \psi'(\sigma_\xi(a)\sigma_\xi(a)^*) = \psi'(a^*a),
\end{align*}
proving that $\psi(a^*a) \leq \psi'(a^*a)$. In the other direction the lower semi-continuity of $\psi'$ combined with Lemma \ref{01-12-21x} shows that for arbitrary $N,\epsilon  > 0$ there is $F\in \mathcal I$ such that 
\begin{align*}
&\min \{\psi'(a^*a)-\epsilon, N\} \leq \sum_{x \in F}\psi'(\sigma_\xi(a)xx^*\sigma_\xi(a)^*) \\
&= \sum_{x \in F} \psi'(\sigma_{-\xi}(x)^* a^*a\sigma_{-\xi}(x)) =\sum_{x \in F} \phi(\sigma_{-\xi}(x)^* a^*a \sigma_{-\xi}(x)) \leq \psi(a^*a).
\end{align*}
We conclude therefore now that $\psi(a^*a) = \psi'(a^*a)$. It follows from Lemma \ref{01-03-22b} that $\psi = \psi'$.
\end{proof}

Laca-Neshveyev' theorem follows by combining Proposition \ref{01-12-21}, Lemma \ref{01-12-21h} and Lemma \ref{02-12-21g}.

We now strengthen the statement in Laca-Neshveyev's theorem as follows.

\begin{thm}\label{07-06-22e} Let $\sigma$ be a flow on the $C^*$-algebra $A$ and let $B \subseteq A$ be a $C^*$-subalgebra of $A$ such that $\sigma_t(B) = B$ for all $t \in \mathbb R$. Assume that $B$ is hereditary and full in $A$. The restriction map $\psi \mapsto \psi|_B$ is a bijection from the set of $\beta$-KMS weights for $\sigma$ on $A$ onto the set of $\beta$-KMS weights for the restriction of $\sigma$ to $B$.
\end{thm}

For the proof, as well as later purposes, we need the following lemma.

\begin{lemma}\label{05-12-21a} Let $\sigma$ be a flow on $A$ and $\psi$ a $\beta$-KMS weight for $\sigma$. Then
$$
\ker_\psi : =\left\{ a \in A : \psi(a^*a) = 0\right\} 
$$
is a closed two-sided $\sigma$-invariant ideal.
\end{lemma}
\begin{proof} That $\ker_\psi$ is a left-ideal follows from the same reasoning as in the proof of (a) of Lemma \ref{04-11-21n}. It is closed because $\psi$ is lower semi-continuous and it is $\sigma$-invariant because $\psi$ is. To show that $\ker_\psi$ is a right-ideal it suffices to show that $a \mathcal A_\sigma \subseteq \ker_\psi$ when $a \in \ker_\psi$. For this let $b \in \mathcal A_\sigma$ and $k \in \mathbb N$. Then
\begin{align*}
& \psi(b^*R_k(a)^*R_k(a)b) = \psi\left(\sigma_{-i \frac{\beta}{2}}(R_k(a)b) \sigma_{-i \frac{\beta}{2}}(R_k(a)b)^*\right)\\
& = \psi(\sigma_{-i \frac{\beta}{2}}( R_k(a))\sigma_{-i \frac{\beta}{2}}(b)\sigma_{-i \frac{\beta}{2}}(b)^* \sigma_{-i \frac{\beta}{2}}(R_k(a))^*)\\
& \leq \left\|\sigma_{-i \frac{\beta}{2}}(b)\right\|^2  \psi(\sigma_{-i \frac{\beta}{2}}( R_k(a))\sigma_{-i \frac{\beta}{2}}(R_k(a))^*) \\
&= \left\|\sigma_{-i \frac{\beta}{2}}(b)\right\|^2  \psi(R_k(a)^*R_k(a)) \leq \left\|\sigma_{-i \frac{\beta}{2}}(b)\right\|^2  \psi(R_k(a^*a)) ,
\end{align*}
where the last step used that $R_k(a)^*R_k(a) \leq R_k(a^*a)$ since $R_k$ is a completely positive contraction. See for example Exercise 11.5.17 in \cite{KR}. It follows from Lemma \ref{02-12-21ax} that $\psi(R_k(a^*a)) = \psi(a^*a) =0$ and we conclude therefore from the estimate above that $ \psi(b^*R_k(a)^*R_k(a)b) = 0$. Since $\lim_{k \to \infty} b^*R_k(a)^*R_k(a)b = b^*a^*ab$, the lower semi-continuity of $\psi$ implies that $\psi(b^*a^*ab) = 0$.
\end{proof}

\emph{Proof of Theorem \ref{07-06-22e}:} In view of Laca-Neshveyev's theorem, it suffices to show that the restriction map is well-defined, i.e. that $\psi|_B$ is a $\beta$-KMS weight for the restriction of $\sigma$ to $B$ when $\psi$ is a $\beta$-KMS weight for $\sigma$. It is clear that $\psi|_B$ is $\sigma|_B$-invariant and has property (1) of Kustermans' theorem, Theorem \ref{24-11-21d}, since $\psi$ has, so what remains is only to show that $\psi|_B$ is non-zero and densely defined. If $\psi(B^+) = \{0\}$ it follows from Lemma \ref{05-12-21a} that $\ker_ \psi = A$ since $B$ is full in $A$. This contradicts that $\psi \neq 0$ and shows therefore that $\psi|_B$ is not zero. 

To see that $\psi|_B$ is densely defined, let $\{e_j\}_{j \in I}$ be an approximate unit for $B$ such that $0 \leq e_j \leq 1$ for all $j \in I$. As in the proof of Lemma \ref{09-02-22x}, set
$$
u_j := \sigma_{i\frac{\beta}{2}}(R_1(e_j)) .
$$
Then $u_j \in \mathcal A_\sigma$ and as shown in the proof of Lemma \ref{09-02-22x}, $ \lim_{j \to \infty} b u_j =b$ for all $b \in B$. Let $b \in B$. Then $\lim_{j \to\infty} u_j^*b^*bu_j = b^*b$ and since $\mathcal M^\sigma_\psi$ is dense in $A$ each $u_j^*b^*bu_j$ can be approximated by elements of the form $u_j^*a^*au_j$ where $a \in \mathcal M^\sigma_\psi$. Note that $u_j^*a^*au_j \in B$ since $u_j^*a^*au_j \leq \|a\|^2 u_j^*u_j \in B$. Furthermore,
\begin{align*}
& \psi(u_j^*a^*au_j) = \psi\left( \sigma_{-i\frac{\beta}{2}}(a) R_1(e_j)^2 \sigma_{-i\frac{\beta}{2}}(a)^*\right) \leq \psi\left( \sigma_{-i\frac{\beta}{2}}(a) \sigma_{-i\frac{\beta}{2}}(a)^*\right) < \infty .
\end{align*}
 It follows that $\psi|_B$ is densely defined.
\qed

\section{Restriction to corners}

Let $\sigma$ be a flow on the $C^*$-algebra $A$, and let $\beta \in \mathbb R$. A \emph{ray} of $\beta$-KMS weights for $\sigma$ is a set of $\beta$-KMS weights for $\sigma$ of the form
$$
\left\{ t \psi : \ t\in \mathbb R, \ t > 0\right\}
$$
for a some $\beta$-KMS weight $\psi$ for $\sigma$. We say that this is the ray generated by $\psi$.

\begin{thm}\label{01-03-22d} Let $\sigma$ be a flow on the $C^*$-algebra $A$ and $p =p^*=p^2$ a projection in the fixed point algebra $A^\sigma$ of $\sigma$. Then $\psi(p) < \infty$ for all KMS weights $\psi$ for $\sigma$.

Assume in addition that $pAp$ is full in $A$. Then $\psi(p) > 0$ for all KMS weights $\psi$ for $\sigma$, and the map
$$
\psi \mapsto \psi(p)^{-1}\psi|_{pAp}
$$
is a bijection from the set of rays of $\beta$-KMS weights for $\sigma$ onto the set of $\beta$-KMS states for the restriction of $\sigma$ to $pAp$.
\end{thm}

Besides the theorem of Laca and Neshveyev the main ingredient in the proof of Theorem \ref{01-03-22d} is the following result by Christensen, \cite{Ch}.

\begin{prop}\label{02-12-21x}
Let $\sigma$ be a flow on $A$ and $\psi$ a $\beta$-KMS weight for $\sigma$. If $a \in A^+\cap D(\sigma_{-i\frac{\beta}{2}})$ and $b \in A^+$ are such that $ab = a$, then $a \in \mathcal N_\psi$.
\end{prop}
\begin{proof} We assume, as we can, that $0 \leq a \leq 1$. The limit $q = \lim_{n \to \infty} \pi_\psi(a)^{\frac{1}{n}}$ exists in the strong operator topology of $B(H_\psi)$ and $q$ is the range projection of $\pi_\psi(a)$. Since $\pi_\psi(a)^{\frac{1}{n}}\pi_\psi(b) = \pi_\psi(b)\pi_\psi(a)^{\frac{1}{n}} = \pi_\psi(a)^{\frac{1}{n}}$ for all $n$, we find that $q \pi_\psi(b) = \pi_\psi(b)q = q$. Since $\psi$ is densely defined there is a $c \in A^+$ such that $\psi(c^2) < \infty$ and $\left\|b^2-c^2\right\| < \frac{1}{2}$. Note that $\lim_{k \to \infty} R_k(c) = c$ and $R_k(c) \in \mathcal A_\sigma$ by Lemma \ref{24-11-21}. Furthermore, $\psi(R_k(c)^2) \leq \psi(R_k(c^2)) = \psi(c^2)$ by Proposition 3.2.4 in \cite{BR} and Lemma \ref{02-12-21ax} above. By substituting $R_k(c)$ for $c$ for some large $k$ we can therefore arrange that $c \in \mathcal A_\sigma \cap A^+$. Note that
\begin{align*}
&\left\|q\pi_\psi(c^2)q-q\right\| = \left\|q(\pi_\psi(c^2) - \pi_\psi(b^2))q\right\| \leq\left\|c^2 -b^2\right\| < \frac{1}{2}.
\end{align*}
It follows that the spectrum of $q\pi_\psi(c^2)q$ in $qB(H_\psi)q$ is contained in $]\frac{1}{2},\frac{3}{2}[$ (unless $q=0$), implying that $q \leq 2q\pi_\psi(c^2)q$. Thus
\begin{align}\label{02-12-21b}
& \pi_\psi(a^2) = \pi_\psi(a)q\pi_\psi(a) \leq 2\pi_\psi(a)q\pi_\psi(c^2)q\pi_\psi(a) =2\pi_\psi(ac^2a).
\end{align}
Let $\omega \in \mathcal F_\psi$, cf. Theorem \ref{04-11-21k}, and let $T_\omega \in \pi_\psi(A)'$ be the operator from Lemma \ref{08-11-21bx}.  It follows from \eqref{02-12-21b} that for any $x \in \mathcal N_\psi$ we have
\begin{align*}
& \omega(x^*a^2x) = \left<T_\omega \Lambda_\psi(ax), \Lambda_\psi(ax)\right> = \left<T_\omega \pi_\psi(a^2)\Lambda_\psi(x),\Lambda_\psi(x)\right> \\
& \leq 2 \left<T_\omega \pi_\psi(ac^2a)\Lambda_\psi(x),\Lambda_\psi(x)\right>  = 2\omega(x^*ac^2ax) .
\end{align*}
Since $\mathcal N_\psi$ is dense in $A$ and $\omega$ is continuous it follows that $\omega(a^2) \leq 2\omega(ac^2a)$. 
By Combes' theorem, Theorem \ref{04-11-21k}, it follows that $\psi(a^2) \leq 2\psi(ac^2a)$. By using Lemma \ref{18-11-21g} and that $\psi$ is a $\beta$-KMS weight we get
\begin{align*}
&\psi(ac^2a) = \psi(\sigma_{-i\frac{\beta}{2}}(ca)\sigma_{-i\frac{\beta}{2}}(ca)^*)  =\psi(\sigma_{-i\frac{\beta}{2}}(c)\sigma_{-i\frac{\beta}{2}}(a)\sigma_{-i\frac{\beta}{2}}(a)^*\sigma_{-i\frac{\beta}{2}}(c)^*) \\
& \leq \left\|\sigma_{-i\frac{\beta}{2}}(a)\right\|^2\psi(\sigma_{-i\frac{\beta}{2}}(c)\sigma_{-i\frac{\beta}{2}}(c)^*) = \left\|\sigma_{-i\frac{\beta}{2}}(a)\right\|^2 \psi(c^2). 
\end{align*}
Hence $\psi(a^2) \leq 2\left\|\sigma_{-i\frac{\beta}{2}}(a)\right\|^2 \psi(c^2) < \infty$.
\end{proof}

\emph{Proof of Theorem \ref{01-03-22d}:} It follows from Proposition \ref{02-12-21x} that $\psi(p) < \infty$ for every KMS weight $\psi$. The rest of the theorem is an immediate consequence of Theorem \ref{07-06-22e}.
\qed

\begin{notes} Laca-Neshveyev's theorem, Theorem \ref{02-12-21}, appears in \cite{LN} in a slightly different context. The formula \eqref{03-01-22} for the extension is taken from \cite{LN}, but otherwise the proof given here is different from that in \cite{LN}. When $\beta =0$ the theorem is about traces and it appears in work by Cuntz and Pedersen, \cite{CP1}, albeit without a flow. Theorem \ref{01-03-22d} here is Theorem 2.4 in \cite{Th1}, while its generalization in Theorem \ref{07-06-22e} is new. Proposition \ref{02-12-21x} here is Proposition 3.1 in \cite{Ch}.
\end{notes}



\chapter[Modular theory]{Modular theory and KMS weights}

In this chapter we relate KMS weights to the modular theory of von Neumann algebras.

\section{The modular theory of von Neumann algebras}\label{modular1}

We first summarize the main results of a theory which is usually referred to as the modular theory or the Tomita-Takesaki theory of von Neumann algebras. The presentation here is brief and highly selective. It is intended to facilitate the application of the theory in the following sections. No proofs are supplied here; they can be found in \cite{KR}, \cite{SZ} or \cite{Ta2}, for example.

Let $M$ be a von Neumann algebra. The set of $\sigma$-weakly continuous linear functionals on $M$, as defined for example in Section 2.4.1 of \cite{BR}, constitute a norm closed subspace $M_*$ of the dual $M^*$, and $M$ can be identified in a natural with the dual of $M_*$, cf. e.g. Proposition 2.4.18 in \cite{BR}. The set $M_*$ is therefore often referred to as the \emph{predual} of $M$.

A weight $\phi : M^+ \to [0,\infty]$ is \emph{normal} when $\phi$ is lower semi-continuous with respect to the $\sigma$-weak topology; that is, when $\left\{ m \in M^+ : \ \phi(m) > t\right\}$ is open in the $\sigma$-weak topology for all $t \in \mathbb R$. It is \emph{semi-finite} when 
$$
\left\{m \in M^+ : \ \phi(m) < \infty\right\}
$$ 
is $\sigma$-weakly dense in $M^+$ and it is \emph{faithful} when $\phi(m) > 0$ for all $m \in M^+ \backslash \{0\}$. A \emph{normal faithful semi-finite trace} on $M$ is a normal faithful semi-finite weight with the property that $\phi(mm^*) = \phi(m^*m)$ for all $m \in M$. In particular, it need not be densely defined with respect to the norm-topology as we require of a trace on a $C^*$-algebra. The GNS representation $(H_\phi,\Lambda_\phi,\pi_\phi)$ of a normal faithful semi-finite weight $\phi$ on $M$ has the property that $\pi_\phi : M \to B(H_\phi)$ is a normal isomorphism of $M$ onto $\pi_\phi(M)$, and the latter is a von Neumann algebra, cf. Theorem 7.5.3 in \cite{KR}. 

A representation $\alpha =\{\alpha_t\}_{t \in \mathbb R}$ of $\mathbb R$ by automorphisms of $M$ will be called a \emph{normal flow} when the map $\mathbb R \ni t \mapsto \alpha_t(m)$ is $\sigma$-weakly continuous for all $m \in M$. Given a normal faithful semi-finite weight $\phi$ on $M$ there is a normal flow $\alpha$ on $M$ such that
\begin{itemize}
\item[(a)] $\phi \circ \alpha_t = \phi$ for all $t \in \mathbb R$, and
\item[(b)] for all pairs $a,b \in \mathcal N_\phi \cap \mathcal N_\phi^*$ there is a continuous bounded function 
$$
f : \left\{ z \in \mathbb C : \ 0 \leq \Imag z \leq 1\right\} \to \mathbb C
$$ 
which is holomorphic in $\left\{ z \in \mathbb C : \ 0 < \Imag z < 1\right\}$, such that
$$
f(t) = \phi(\alpha_t(a)b) \ \ \text{and} \ \ f(t+i) = \phi(b\alpha_t(a))
$$
for all $t \in \mathbb R$. 
\end{itemize}
This normal flow is unique; if $\alpha'$ is a normal flow with the properties (a) and (b) then $\alpha'_t = \alpha_t$ for all $t \in \mathbb R$, and it is called \emph{the modular automorphism group} of $M$ relative to $\phi$. The modular automorphism group 
\begin{itemize}
\item is trivial on the center of $M$,
\item is inner, i.e. of the form $\sigma_t = \Ad u_t$, where $\{u_t\}$ is a strongly continuous group of unitaries in $M$, if and only if $M$ admits a faithful normal semi-finite trace.
\end{itemize}

Thanks to condition (a) and the normality of $\phi$, there is a strongly continuous unitary representation $u = \{u_t\}_{t \in \mathbb R}$ of $\mathbb R$ on $H_\phi$ defined such that
$$
u_t \Lambda_\phi(m) := \Lambda_\phi(\alpha_t(m))
$$
for all $m \in \mathcal N_\phi \cap \mathcal N_\phi^*$. Then
$$
\pi_\phi \circ \alpha_t(m) = u_t \pi_\phi(m)u_{-t}
$$
for all $m \in M$ and $t \in \mathbb R$. By Stone's theorem there is a self-adjoint operator $H$ on $H_\phi$ such that $u_t = e^{itH}$. This unitary group is usually written
$$
u_t = \nabla^{it} ,
$$
where $\nabla := e^H$ is \emph{the modular operator} of $\phi$.

The conjugate linear map $S_0 : \Lambda_\phi(\mathcal N_\phi \cap \mathcal N_\phi^*) \to H_\phi$ defined such that
$$
S_0\Lambda_\psi(a) = \Lambda_\phi(a^*), \ \ a \in \mathcal N_\phi \cap \mathcal N_\phi^* ,
$$
is closable and its closure $S$ is related to the modular operator $\nabla$ via the polar decomposition of $S$:
$$
S = J\nabla^{\frac{1}{2}} ,
$$
where $J$ is a conjugate linear unitary, called the \emph{modular conjugation}. It has the property that $J^* = J$, $J^2 = 1$ and
\begin{equation}\label{03-03-22h}
J \pi_\phi(M)J = \pi_\phi(M)' ;
\end{equation}
the latter being the commutant of $\pi_\phi(M)$ in $B(H_\phi)$.

\section{The modular theory of a KMS weight}\label{modular2}

Let $\sigma$ be a flow on the $C^*$-algebra $A$. In this section we show how a KMS weight $\psi$ for $\sigma$ gives rise to a faithful normal semi-finite weight $\psi''$ on the von Neumann algebra $\pi_\psi(A)''$ generated by $\pi_\psi(A)$ and we begin to investigate the ingredients of the modular theory associated with $\psi''$.

Fix $\beta \in \mathbb R$ and let $\psi$ be a $\beta$-KMS weight for $\sigma$.

\begin{lemma}\label{14-02-22} The following inclusions hold.
$$
\mathcal A_\sigma \mathcal M^\sigma_\psi \subseteq \mathcal N_\psi \cap \mathcal N_\psi^* \cap \mathcal A_\sigma .
$$
and
$$
\mathcal M^\sigma_\psi \mathcal A_\sigma  \subseteq \mathcal N_\psi \cap \mathcal N_\psi^* \cap \mathcal A_\sigma .
$$
\end{lemma}
\begin{proof} For the proof of the first inclusion, let $a \in  \mathcal M^\sigma_\psi$, $h \in \mathcal A_\sigma$. Since $ \mathcal M^\sigma_\psi \subseteq \mathcal A_\sigma \cap \mathcal N_\psi$, and $\mathcal N_\psi$ is a left ideal in $A$ while $\mathcal A_\sigma$ is a subalgebra of $A$ it suffices to show that $h a \in \mathcal N_\psi^*$. Observe therefore that 
\begin{align*}
& \psi((ha)(ha)^*) = \psi\left( \sigma_{-i \frac{\beta}{2}}\left(\sigma_{i \frac{\beta}{2}}(ha)\right)  \sigma_{-i \frac{\beta}{2}}\left(\sigma_{i \frac{\beta}{2}}(ha)\right)^*\right) \\
& = \psi\left(\sigma_{i \frac{\beta}{2}}(ha)^*\sigma_{i \frac{\beta}{2}}(ha)\right) =\psi\left(\sigma_{i \frac{\beta}{2}}( a)^*\sigma_{i \frac{\beta}{2}}(h)^*\sigma_{i \frac{\beta}{2}}(h)\sigma_{i \frac{\beta}{2}}( a)\right) \\
& \leq \|\sigma_{i \frac{\beta}{2}}(h)\|^2 \psi\left(\sigma_{i \frac{\beta}{2}}( a)^*\sigma_{i \frac{\beta}{2}}( a)\right) < \infty ,
\end{align*}
since $\sigma_{i \frac{\beta}{2}}( a) \in \mathcal M^\sigma_\psi \subseteq \mathcal N_\psi$ by definition of $\mathcal M^\sigma_\psi$. This establishes the first inclusion, and the second follows by taking adjoints.
\end{proof}

From Section \ref{GNS-KMS} we take the notation $\mathcal M_{\Lambda_\psi}^\sigma$ for the set
$$
\mathcal M_{\Lambda_\psi}^\sigma := \left\{ a \in \mathcal N_\psi \cap \mathcal N_\psi^* \cap \mathcal A_\sigma : \ \sigma_z(a) \in  \mathcal N_\psi \cap \mathcal N_\psi^* \cap \mathcal A_\sigma \ \ \forall z \in \mathbb C\right\} .
$$
Since $\sigma_z(\mathcal A_\sigma \mathcal M^\sigma_\psi) \subseteq \mathcal A_\sigma \mathcal M^\sigma_\psi$ it follows from Lemma \ref{14-02-22} that
\begin{equation}\label{09-02-23}
\mathcal A_\sigma \mathcal M_\psi^\sigma \subseteq \mathcal M_{\Lambda_\psi}^\sigma.
\end{equation}

Let $\mathcal J$ denote the collection of finite subsets $I \subseteq  \mathcal M_{\Lambda_\psi}^\sigma$ with the property that $\sum_{a \in I} aa^* \leq 1$. When $I \in \mathcal J$ we set
$$
w_I :=  \sum_{a \in I} aa^*.
$$

\begin{lemma}\label{01-03-22e} The collection $\left\{w_I : \ I \in \mathcal J\right\}$ is an approximate unit in $A$ in the following sense: For every finite set $S \subseteq A$ and every $\epsilon > 0$ there is an element $J \in \mathcal J$ such that $\left\|w_Ja -a\right\| \leq \epsilon$ for all $a \in S$.
\end{lemma}
\begin{proof} Thanks to \eqref{09-02-23} this follows from Lemma \ref{01-12-21x} applied with $B =A$ and $\phi = \psi$.
\end{proof}

To simplify the notation in the following, set
$$
N := \pi_{\psi}(A)''
$$
and 
$$
\xi := - i\frac{\beta}{2} .
$$
Define
$$
\psi'' : N^+ \to [0,\infty]
$$
by
$$
\psi''(m) := \sup_{I \in \mathcal J} \sum_{x \in I} \left<m \Lambda_\psi(\sigma_{-\xi}(x)), \Lambda_\psi(\sigma_{-\xi}(x))\right> .
$$

\begin{lemma}\label{06-02-23} $\psi''$ is additive; i.e. $\psi''(n+m)= \psi''(n) + \psi''(m)$ when $n,m \in N^+$.
\end{lemma}
\begin{proof} Let $\omega \in \mathcal F_\psi$. It follows from Lemma \ref{08-11-21bx} that there is a operator $0 \leq T_\omega \leq 1$ in $\pi_{\psi}(A)'$ such that $\omega(b^*a) = \left<T_\omega \Lambda_\psi(a), \Lambda_\psi(b)\right>$ for all $a,b \in \mathcal N_\psi$. By Lemma \ref{16-12-21a} there is an increasing net $0 \leq e_m \leq 1$ in $\mathcal N_\psi$ which is an approximate unit in $A$. For $m \leq m'$  we find that
\begin{align*}
&\left\|T_\omega^{\frac{1}{2}} \Lambda_\psi(e_{m'})  -T_\omega^{\frac{1}{2}} \Lambda_\psi (e_m)\right\|^2 = \left< T_\omega \Lambda_\psi(e_{m'} -e_m), \Lambda_\psi(e_{m'} -e_m)\right> \\
& = \omega((e_{m'} -e_m)^2)  \leq \omega(e_{m'} - e_m) .
\end{align*}
Since $\lim_{m \to \infty} \omega(e_m) = \|\omega\|$ by Proposition 2.3.11 of \cite{BR}, this estimate shows that $\{T_\omega^{\frac{1}{2}} \Lambda_\psi(e_{m})\}$ converges in $H_\psi$, and we set $\xi_\omega := \lim_{m \to \infty} T_\omega^{\frac{1}{2}} \Lambda_\psi(e_{m})$. Then
$$
\omega(a) = \lim_{m \to \infty} \omega(e_mae_m) = \lim_{m \to \infty}\left<T_\omega \Lambda_\psi(e_m), \Lambda_\psi(a^*e_m)\right> = \left<\pi_\psi(a)\xi_\omega,\xi_\omega\right> 
$$
for all $a \in A$. Define $\tilde{\omega} \in N_*^+$ such that $\tilde{\omega}(m) := \left< m\xi_\omega,\xi_\omega\right>$. Then $\tilde{\omega}$ is the unique normal positive functional on $N$ with the property that
$$
\tilde{\omega} \circ \pi_\psi = \omega .
$$
To prove that $\psi''$ is additive we will show that
\begin{equation}\label{06-02-23a}
 \psi''(m) = \sup \left\{ \tilde{\omega}(m) : \ \omega \in \mathcal F_\psi \right\} .
\end{equation}
Let $I \in \mathcal J$ and $a \in \mathcal A_\sigma$. Then
\begin{align*}
& \sum_{x \in I} \left<\pi_\psi(a^*a) \Lambda_\psi(\sigma_{-\xi}(x)), \Lambda_\psi(\sigma_{-\xi}(x))\right> = \sum_{x \in I} \left< \Lambda_\psi(a\sigma_{-\xi}(x)), \Lambda_\psi(a\sigma_{-\xi}(x))\right> \\
& = \sum_{x \in I} \psi(\sigma_{-\xi}(x)^*
a^*a\sigma_{-\xi}(x))
 =
 \sum_{x \in I} \psi(\sigma_\xi(a)xx^*\sigma_\xi(a)^*) \\
 &\leq \psi(\sigma_\xi(a)\sigma_\xi(a)^*) = \psi(a^*a) .
\end{align*}
Let $a \in \mathcal N_\psi$. Inserting $R_l(a)$ for $a$ above we find that
$$
\sum_{x \in I} \left<\pi_\psi(R_l(a)^*R_l(a)) \Lambda_\psi(\sigma_{-\xi}(x)), \Lambda_\psi(\sigma_{-\xi}(x))\right>
\leq \psi(R_l(a)^*R_l(a)) .
$$
Taking the limit $l \to \infty$ it follows now from Lemma \ref{17-11-21m} and Lemma \ref{24-11-21g} that
$$
\sum_{x \in I} \left<\pi_\psi(a^*a) \Lambda_\psi(\sigma_{-\xi}(x)), \Lambda_\psi(\sigma_{-\xi}(x))\right>
\leq \psi(a^*a) .
$$
 This shows that the functional $\omega \in A^*$ given by
$$
\omega(b) := \sum_{x \in I} \left<\pi_\psi(b) \Lambda_\psi(\sigma_{-\xi}(x)), \Lambda_\psi(\sigma_{-\xi}(x))\right> 
$$
is an $\mathcal F_\psi$. Since $\left<\pi_\psi(a) \Lambda_\psi(\sigma_{-\xi}(x)), \Lambda_\psi(\sigma_{-\xi}(x))\right>  = \tilde{\omega} \circ \pi_\psi(a)$ for all $a \in A$ it follows that $\left<m \Lambda_\psi(\sigma_{-\xi}(x)), \Lambda_\psi(\sigma_{-\xi}(x))\right>  = \tilde{\omega}(m)$ for all $m \in N$, and hence from the definition of $\psi''$ that 
$$
\psi''(m) \leq \sup \left\{ \tilde{\omega}(m) : \ \omega \in \mathcal F_\psi \right\}.
$$ 
To establish the reverse inequality, let $\omega \in \mathcal F_\psi$ and $m \in N^+$. Let $\epsilon > 0$. We are going to use the net $\{u_j\}$ of Lemma \ref{09-02-22x}. Since $\pi_\psi$ is non-degenerate it follows from the properties (b) and (c) in Lemma \ref{09-02-22x} that $\lim_{j \to \infty} \pi_\psi(u_j)^*m\pi_\psi(u_j) = m$ strongly. It follows that there is a $j$ such that
$$
\tilde{\omega}(m) - \epsilon \leq \tilde{\omega}(\pi_{\psi}(u_j)^*m\pi_\psi(u_j)) .
$$
Let $a \in A$. 
Then
\begin{align*}
& \tilde{\omega}(\pi_{\psi}(u_j)^*\pi_\psi(a^*a)\pi_{\psi}(u_j)) = \omega(u_j^*a^*au_j) = \left<T_\omega \Lambda_\psi(au_j),\Lambda_\psi(au_j)\right> \\
& = \left<T_\omega \pi_\psi(a^*a)\Lambda_\psi(u_j), \Lambda_{\psi}(u_j)\right> \leq \left< \pi_\psi(a^*a)\Lambda_\psi(u_j), \Lambda_{\psi}(u_j)\right> .
\end{align*}
Since $\pi_\psi(A^+)$ is $\sigma$-weakly dense in $N^+$ it follows that
\begin{equation}\label{07-02-23}
\tilde{\omega}(\pi_{\psi}(u_j)^*m\pi_{\psi}(u_j)) \leq \left< m\Lambda_\psi(u_j), \Lambda_{\psi}(u_j)\right>  .
\end{equation}
From the proof of Lemma \ref{09-02-22x} we see that
$$
\Lambda_\psi(u_j) = \Lambda_\psi(\sigma_{-\xi}(R_1(e_j))) ,
$$
where $0 \leq e_j \leq 1$. It follows from (a) of Lemma \ref{09-02-22x} that $R_1(e_j) \in \mathcal M^\sigma_{\Lambda_\psi}$. Using Kadisons inequality, Proposition 3.2.4 in \cite{BR}, we find that $R_1(e_j)R_1(e_j)^* = R_1(e_j)R_1(e_j) \leq R_1(e_j^2) \leq 1$. Hence $\{R_1(e_j)\} \in \mathcal J$ and it follows therefore from \eqref{07-02-23} that $\tilde{\omega}(\pi_{\psi}(u_j)^*m\pi_{\psi}(u_j)) \leq \psi''(m)$, implying that $\tilde{\omega}(m) - \epsilon \leq \psi''(m)$. It follows that \eqref{06-02-23a} holds.

Set $\mathcal G = \left\{ t\omega: \ \omega \in \mathcal F_\psi, t \in ]0,1[ \right\}$. It follows from Lemma \ref{09-11-21ax} that $\mathcal G$ is a directed set with respect to the natural order in the hermitian part of $A^*$ and hence 
$\left\{ \tilde{\omega}\right\}_{\omega \in \mathcal G}$
is a net of positive normal functionals on $N$. It follows from \eqref{06-02-23a} that 
$$
\psi''(m) = \lim_{\omega \to \infty} \tilde{\omega}(m)
$$
for all $m \in N^+$, and the additivity of $\psi''$ is a consequence of this because $\psi''(n+m) = \lim_{\omega \to \infty} \tilde{\omega}(n+m) =  \lim_{\omega \to \infty} \tilde{\omega}(n)  + \lim_{\omega \to \infty} \tilde{\omega}(m) = \psi''(n) + \psi''(m)$.
\end{proof}

\begin{lemma}\label{01-03-22f} $\psi''$ is a  normal faithful semi-finite weight on $N$.
\end{lemma}
\begin{proof} $\psi''$ is additive by Lemma \ref{06-02-23}. For $I \in \mathcal J$ the map 
$$
N \ni m \mapsto \sum_{x \in I} \left<m \Lambda_\psi(\sigma_{-\xi}(x)), \Lambda_\psi(\sigma_{-\xi}(x))\right> 
$$
is $\sigma$-weakly continuous so it follows that $\psi''$ is lower semi-continuous with respect to the $\sigma$-weak topology; that is, $\psi''$ is normal. To see that $\psi''$ is semi-finite, let $ a \in \mathcal M_\psi^\sigma$. Then 
\begin{align*}
& \sum_{x \in I} \left<\pi_\psi(aa^*) \Lambda_\psi(\sigma_{-\xi}(x)), \Lambda_\psi(\sigma_{-\xi}(x))\right> = \sum_{x \in I} \left< \Lambda_\psi(a^*\sigma_{-\xi}(x)), \Lambda_\psi(a^*\sigma_{-\xi}(x))\right> \\
& = \sum_{x \in I} \psi(\sigma_{-\xi}(x)^*aa^*\sigma_{-\xi}(x)) = \sum_{x \in I} \psi(\sigma_\xi(a^*)xx^*\sigma_\xi(a^*)^*) \\
& \leq \psi(\sigma_\xi(a^*)\sigma_\xi(a^*)^*) = \psi(\sigma_{-\xi}(a)^* \sigma_{-\xi}(a)). 
\end{align*}
It follows that $\psi''(\pi_\psi(aa^*)) \leq \psi(\sigma_{-\xi}(a)^* \sigma_{-\xi}(a))$, and hence that $\psi''(\pi_\psi(aa^*)) < \infty$ since $\sigma_{-\xi}(a) \in \mathcal N_\psi$. Note that $\mathcal M_\psi^\sigma$ is dense in $A$ by Lemma \ref{07-12-21a} and that an application of Kaplansky's density theorem shows that $\pi_\psi(A^+)$ is $\sigma$-strongly and hence also $\sigma$-weakly dense in $N^+$. Thus
$$
\left\{\pi_\psi(aa^*) : \ a \in \mathcal M_\psi^\sigma \right\}
$$
is $\sigma$-weakly dense in $N^+$ and we conclude therefore that $\psi''$ is semi-finite. To show that $\psi''$ is faithful, let $m \in N^+$ and assume that $\psi''(m) = 0$. Then 
 $$
 \sum_{x \in I} \left<m \Lambda_\psi(\sigma_{-\xi}(x)), \Lambda_\psi(\sigma_{-\xi}(x))\right> = 0
 $$ 
 for all $I \in \mathcal J$, and hence in particular $m \Lambda_\psi(\sigma_{-\xi}(x)) = 0$ for all $x \in \mathcal M_{\Lambda_\psi}^\sigma$. Since $\mathcal M_\psi^\sigma \subseteq \mathcal M_{\Lambda_\psi}^\sigma$ it follows that $m\Lambda_\psi(x) = 0$ for all $x \in \mathcal M^\sigma_\psi$. It follows from Lemma \ref{07-12-21c} (applied with $F= \{0\}$) that $\left\{\Lambda_\psi(x) : \ x \in \mathcal M^\sigma_\psi \right\}$ is dense in $H_\psi$ and we conclude that $m =0$; i.e. $\psi''$ is faithful.
\end{proof}

Consider the unitary group representation $U^\psi$ on $H_\psi$, defined such that $U^\psi_t \Lambda_\psi(a) := \Lambda_\psi(\sigma_t(a))$ when $a \in \mathcal N_\psi$, cf. Section \ref{invariantweights}. It follows from \eqref{08-03-22a} that we can define a normal flow $\sigma''$ on $N$ such that
$$
\sigma''_t(m) = U^\psi_t m {U^\psi_t}^* .
$$
We note that
\begin{equation}\label{08-03-22b}
\pi_\psi \circ \sigma_t = \sigma''_t \circ \pi_\psi .
\end{equation}

\begin{lemma}\label{01-03-22g} $\psi''$ is $\sigma''$-invariant; that is, $\psi'' \circ \sigma''_t = \psi''$ for all $t \in \mathbb R$.
\end{lemma}
\begin{proof} Let $I \in \mathcal J$ and $m \in N^+$. Then
\begin{align*}
&  \sum_{x \in I} \left< \sigma''_t(m) \Lambda_\psi(\sigma_{-\xi}(x)), \Lambda_\psi(\sigma_{-\xi}(x))\right> = 
\sum_{x \in I} \left< m U^\psi_{-t}\Lambda_\psi(\sigma_{-\xi}(x)), U^\psi_{-t}\Lambda_\psi(\sigma_{-\xi}(x))\right> \\
& = \sum_{x \in I} \left< m \Lambda_\psi(\sigma_{-\xi}(\sigma_{-t}(x))), \Lambda_\psi(\sigma_{-\xi}(\sigma_{-t}(x)))\right> \\
& = \sum_{y \in I'} \left< m \Lambda_\psi(\sigma_{-\xi}(y)), \Lambda_\psi(\sigma_{-\xi}(y))\right> 
\end{align*}
where $I' := \sigma_{-t}(I) \in \mathcal J$. The lemma follows.
\end{proof}

As in Section \ref{KMSdefn}, set
$$
\mathcal D_\beta := \left\{z \in \mathbb C: \ \Imag z \in [0,\beta]\right\}
$$
when $\beta \geq 0$, and
$$
\mathcal D_\beta := \left\{z \in \mathbb C: \ \Imag z \in [\beta,0]\right\}
$$
when $\beta \leq 0$, and let $\mathcal D_\beta^0$ denote the interior of $\mathcal D_\beta$ in $\mathbb C$. We are aiming to prove the following theorem.

\begin{thm}\label{21-02-22dx} $\psi''$ is a faithful normal semi-finite and $\sigma''$-invariant weight on $\pi_\psi(A)''$, and for each pair $a,b \in \mathcal N_{\psi''} \cap \mathcal N_{\psi''}^*$ there is a continuous bounded function $f: {\mathcal D_\beta} \to \mathbb C$ which is holomorphic in the interior $\mathcal D_\beta^0$ of ${\mathcal D_\beta}$ and has the property that
\begin{itemize}
\item $f(t) = \psi''(b \sigma''_t(a)) \ \ \forall t \in \mathbb R$, and
\item $f(t+i\beta) = \psi''(\sigma''_t(a)b) \ \ \forall t \in \mathbb R$.
\end{itemize}
\end{thm}

\bigskip

At this point it is only the existence of the function $f$ we need to establish, and for this we need a considerable amount of preparation.

\subsection{Proof of Theorem \ref{21-02-22dx}} 

\subsubsection{Smoothing operators for normal flows}\label{smoothing} In this section we consider first an arbitrary von Neumann algebra $M$ and a normal flow $\alpha$ on $M$. For $\omega \in M_*$ the function $\mathbb R \ni t \mapsto \omega(\alpha_t(m))$ is continuous and bounded for every $m \in M$. We can therefore define the integral
$$
\sqrt{\frac{k}{\pi}} \int_\mathbb R e^{-k (t -z)^2} \omega(\alpha_t(m)) \ \mathrm dt
$$
for every $k \in \mathbb N$ and every $z \in \mathbb C$.

\begin{lemma}\label{02-03-22x} Let $m \in M$, $k \in \mathbb N$ and $z \in \mathbb C$. There is a unique element 
$$
\sqrt{\frac{k}{\pi}} \int_\mathbb R e^{-k (t-z)^2}\alpha_t(m) \ \mathrm dt \in M
$$
with the property that
$$
\omega\left( \sqrt{\frac{k}{\pi}} \int_\mathbb R e^{-k (t-z)^2}\alpha_t(m) \ \mathrm dt\right)  = \sqrt{\frac{k}{\pi}} \int_\mathbb R e^{-k (t - z)^2} \omega(\alpha_t(m)) \ \mathrm dt
$$
for all $\omega \in M_*$.
\end{lemma}
\begin{proof} Note that
\begin{align*}
& \left|\sqrt{\frac{k}{\pi}} \int_\mathbb R e^{-k (t -z)^2} \omega(\alpha_t(m)) \ \mathrm dt\right|  \leq \sqrt{\frac{k}{\pi}} \int_\mathbb R \left|e^{-k (t -z)^2}\right| \left|\omega(\alpha_t(m))\right| \ \mathrm d t \\
& \leq \|m\|\left\|\omega\right\| \sqrt{\frac{k}{\pi}} \int_\mathbb R \left|e^{-k (t -z)^2}\right| \ \mathrm d t .
\end{align*}
This shows that the map
$$
M_* \ni \omega \mapsto \sqrt{\frac{k}{\pi}} \int_\mathbb R e^{-k (t-z)^2} \omega(\alpha_t(m)) \ \mathrm dt
$$
is continuous for the norm on $M_*$, and since $M$ is the dual of $M_*$ this proves the existence of the element of $M$ with the stated property. Uniqueness follows because $M_*$ separates points in $M$.
\end{proof}

To simplify notation write 
$$
\overline{R}_k(m) := \sqrt{\frac{k}{\pi}} \int_\mathbb R e^{-k t^2}\alpha_t(m) \ \mathrm dt.
$$
We note that $\overline{R}_k : M \to M$ is a unital norm-continuous linear positive contraction. 


\begin{defn}\label{02-03-22a} Let $M_c$ denote the set of elements of $M$ with the property that $\mathbb R \ni t \mapsto \alpha_t(m)$ is norm-continuous, and let $M_a$ denote the set of elements $m \in M$ with the property that there is entire holomorphic function $f : \mathbb C \to M$ such that $f(t) = \alpha_t(m)$ for all $t \in \mathbb R$.
\end{defn}

We emphasize that the function $f$ in this definition should be entire holomorphic in the same sense we have used so far; that is, the limit 
$$
\lim_{h \to 0} \frac{f(z+h) - f(z)}{h}
$$
must exist in norm for all $z \in \mathbb C$. Thus
$$
M_a \subseteq M_c.
$$
It is straightforward to verify that $M_c$ is a $C^*$-subalgebra of $M$. Thus $M_a$ is the set of entire analytic elements for the restriction of the flows $\alpha$ to $M_c$. In symbols,
$$
M_a = \mathcal A_{\alpha|_{M_c}}.
$$
In particular, the operator $\alpha_z,  z \in \mathbb C$, are defined as they were introduced in Section \ref{Section-entire}, and
$$
D(\alpha_z) \subseteq M_c.
$$

\begin{lemma}\label{02-03-22bx} $M_a$ is $\sigma$-weakly dense in $M$. In fact, 
\begin{itemize}
\item $\overline{R}_n(m) \in M_a$ for all $m \in M$,
\item $\alpha_z\left( \overline{R}_n(m)\right) = \sqrt{\frac{k}{\pi}} \int_\mathbb R e^{-n (t-z)^2}\alpha_t(m) \ \mathrm dt$ for all $z \in \mathbb C$ and all $m \in M$, and
\item $\lim_{n \to \infty}\overline{R}_n(m) = m$ in the $\sigma$-weak topology for all $m \in M$.
\end{itemize} 
\end{lemma}
\begin{proof} Let $m \in M, \ \omega \in M_*$. Then $\omega \circ \alpha_t \in M_*$ by Theorem 2.4.23 in \cite{BR} and hence
$$
\omega\left(\alpha_t \left( \overline{R}_n(m)\right) \right) = \sqrt{\frac{n}{\pi}} \int_\mathbb R e^{-n s^2}\omega(\alpha_{t+s}(m)) \ \mathrm ds = \sqrt{\frac{n}{\pi}} \int_\mathbb R e^{-n (s -t)^2}\omega(\alpha_{s}(m)) \ \mathrm ds .
$$
It follows that
$$
\alpha_t( \overline{R}_n(m)) = \sqrt{\frac{n}{\pi}} \int_\mathbb R e^{-n (s -t)^2} \alpha_{s}(m) \ \mathrm ds
$$
for all $t \in \mathbb R$, and to establish the first two items it suffices therefore to show that
$$
H(z) : =  \sqrt{\frac{n}{\pi}} \int_\mathbb R e^{-n (s-z)^2}\alpha_t(m) \ \mathrm d s
$$
is entire analytic. This is done by repeating arguments from the proof of Lemma \ref{18-11-21} as follows. In fact, the only differences come from the difference in how the integrals are defined: For each $\omega \in M_*$ an application of Lebesques theorem on dominated convergence shows that
\begin{align*}
& \omega(H(z)) =  \sqrt{\frac{n}{\pi}} \int_\mathbb R e^{-n (s-z)^2}\omega(\alpha_s(m)) \ \mathrm ds \\
& = \sqrt{\frac{n}{\pi}} e^{-nz^2} \int_\mathbb R \sum_{k=0}^\infty z^k(2n)^k \frac{s^ke^{-ns^2}}{k!} \omega(\alpha_s(a)) \ \mathrm d s  \\
& = \sqrt{\frac{n}{\pi}} e^{-nz^2}\sum_{k=0}^\infty \frac{z^k(2n)^k}{k!}  \int_\mathbb R s^ke^{-ns^2} \omega(\alpha_s(a)) \ \mathrm d s  \\
& = \sqrt{\frac{n}{\pi}} e^{-nz^2}\sum_{k=0}^\infty \omega(b_k)z^k 
\end{align*}
where $b_k = \frac{(2n)^k}{k!} \int_\mathbb R s^ke^{-ns^2} \sigma_s(a) \ \mathrm d s$. Note that
$$
\left\|b_k\right\| = \sup_{\mu \in M_*, \|\mu\| \leq 1} \left|\frac{(2n)^k}{k!} \int_\mathbb R s^ke^{-ns^2} \mu(\sigma_s(a)) \ \mathrm d s\right| \leq 2\|a\| n^{-\frac{k+1}{2}} \int_0^\infty s^k e^{-s^2} \ \mathrm d s .
$$ 
A verbatim repetition of the arguments from the proof of Lemma \ref{18-11-21} shows that 
$$
\sum_{k=0}^\infty \|b_k\| |z|^k < \infty
$$
for all $z \in \mathbb C$. It follows therefore from the calculation above that
$$
H(z) =  \sqrt{\frac{n}{\pi}} e^{-nz^2}\sum_{k=0}^\infty b_k z^k 
$$
for all $z \in \mathbb C$ and an application of Lemma \ref{13-11-21d} shows then that $H$ is entire analytic. \footnote{A less pedestrian proof of this can be based on Proposition 2.5.21 in \cite{BR}.}

To show that $\lim_{k \to \infty}\overline{R}_k(m) = m$ in the $\sigma$-weak topology we must show that
$$
\lim_{k \to \infty} \omega(\overline{R}_k(m)) = \lim_{k \to \infty}\sqrt{\frac{k}{\pi}} \int_\mathbb R e^{-k t^2} \omega(\alpha_t(m)) \ \mathrm dt = \omega(m)
$$
for all $\omega \in M_*$. The proof of this is identical to the proof of Lemma \ref{24-11-21}. 
\end{proof}

Now we return to the setting of Theorem \ref{21-02-22dx}. In particular, the roles of $M$ and $\alpha$ are now taken by $N = \pi_\psi(A)''$ and $\sigma''$, respectively. Since $\psi''$ is lower semi-continuous with respect to the $\sigma$-weak topology it is also lower semi-continuous with respect to the norm-topology, and it is therefore a weight in the sense of Section \ref{defweight}. We consider the GNS-triple of the weight $(H_{\psi''},\Lambda_{\psi''},\pi_{\psi''})$ as defined in Section \ref{06-02-22}. Since $\mathcal M_{\psi''}^+$ is $\sigma$-weakly dense in $N^+$ by Lemma \ref{01-03-22f} it follows from Lemma \ref{04-11-21n} that $\mathcal N_{\psi''} \cap \mathcal N_{\psi''}^*$ is $\sigma$-weakly dense in $N$.

\begin{lemma}\label{22-02-22fx} $\Lambda_{\psi''} : \mathcal N_{\psi''} \to H_{\psi''}$ is closed with respect to the $\sigma$-weak topology.
\end{lemma}
\begin{proof}  Thanks to the way $\psi''$ is defined we can repeat the proof of Lemma \ref{17-11-21a} almost ad verbatim. Let $\{a_i\}_{i \in J}$ be a net $\mathcal N_{\psi''}$ such that $\lim_{i \to \infty} a_i = a$ in the $\sigma$-weak topology of $N$ and $\lim_{i \to \infty} \Lambda_{\psi''}(a_i) = v$ in $H_{\psi''}$. We must show that $a \in \mathcal N_{\psi''}$ and that $\Lambda_{\psi''}(a) = v$. Consider an element $I \in \mathcal J$ and set
$$
\omega ( \ \cdot \ ) = \sum_{x \in I} \left< \ \cdot \ \Lambda_\psi(\sigma_{-\xi}(x)), \Lambda_\psi(\sigma_{-\xi}(x))\right>.
$$
The triple $(H_{\psi''},\Lambda_{\psi''},\pi_{\psi''})$ is not a GNS representation of $N$ as defined in Section \ref{06-02-22} since $\mathcal N_{\psi''}$ may not be norm-dense in $N$, but nonetheless the proof of Lemma \ref{08-11-21bx} shows that there is an operator $T_\omega$ on $H_{\psi''}$ such that $0 \leq T_\omega \leq 1$ and
$$
\omega(b^*a) = \left<T_\omega \Lambda_{\psi''}(a),\Lambda_{\psi''}(b)\right> \ \ \forall a,b \in \mathcal N_{\psi''} .
$$
Since $\omega \in N_*$ we have that
\begin{align*}
&\omega(a^*a) =\lim_{j \to \infty} \lim_{k \to \infty} \omega(a_k^*a_j) = \lim_{j \to \infty} \lim_{k \to \infty} \left<T_\omega \Lambda_{\psi''}(a_j),\Lambda_{\psi''}(a_k)\right> \\
&= \left<T_\omega v,v\right> \leq \|v\|^2 .
\end{align*}
Since $I \in \mathcal J$ was arbitrary it follows that $\psi''(a^*a) \leq \|v\|^2 < \infty$, showing that $a \in \mathcal N_{\psi''}$.
Let $b \in \mathcal N_{\psi''}$ and $\epsilon > 0$. By definition of $\psi''$ there is an element $\mu \in N_* \cap \mathcal F_{\psi''}$ of the form
$$
\mu ( \ \cdot \ ) = \sum_{x \in I} \left< \ \cdot \ \Lambda_\psi(\sigma_{-\xi}(x)), \Lambda_\psi(\sigma_{-\xi}(x))\right>
$$
such that
$$
\psi''(b^*b) - \epsilon \leq \mu(b^*b) \leq \psi''(b^*b).
$$
As above it follows from the proof of Lemma \ref{08-11-21bx} that there is an operator $T_\mu$ on $H_{\psi''}$ such that $0 \leq T_\mu \leq 1$ and
$$
\mu(c^*d) = \left<T_\mu \Lambda_{\psi''}(d),\Lambda_{\psi''}(c)\right> \  \ \forall c,d \in \mathcal N_{\psi''} .
$$
As in the proof of Lemma \ref{17-11-21a} we find that
\begin{align*}
 \left\| T_\mu\Lambda_{\psi''}(b) - \Lambda_{\psi''}(b)\right\|^2 \leq 2 \epsilon .
\end{align*}
We can choose $k \in J$ such that $\left\|v-\Lambda_{\psi''}(a_k)\right\| \leq \epsilon$ and $\left|\mu(b^*a_k) - \mu(b^*a)\right| \leq \epsilon$. The same calculation as in the proof of Lemma \ref{17-11-21a} yields
 \begin{align*}
 &\left|\left< v,\Lambda_{\psi''}(b)\right> - \left<\Lambda_{\psi''}(a),\Lambda_{\psi''}(b)\right>\right| \\
 &\leq (\|v\| + \left\|\Lambda_{\psi''}(a)\right\|)\sqrt{2\epsilon} + \left\|\Lambda_{\psi''}(b)\right\| \epsilon + \epsilon .
 \end{align*}
Since both $\epsilon > 0$ and $b\in \mathcal N_{\psi''}$ were arbitrary we conclude that $v = \Lambda_{\psi''}(a)$.
\end{proof}

Since $\psi''$ is $\sigma''$-invariant by Lemma \ref{01-03-22g} we can define $U^{\psi''}_t \in B(H_{\psi''})$ such that
$$
U^{\psi''}_t\Lambda_{\psi''}(a):= \Lambda_{\psi''}(\sigma''_t(a))  \ \ \forall a \in \mathcal N_{\psi''} . 
$$

\begin{lemma}\label{22-02-22e} $\left\{U^{\psi''}_t\right\}_{t \in \mathbb R}$ is a strongly continuous unitary representation of $\mathbb R$.
\end{lemma}
\begin{proof} The proof is essentially identical to the proof of Lemma \ref{17-11-21e}, and it holds more generally as pointed out in Section \ref{modular1}. We leave the details to the reader.
\end{proof}

It follows from \eqref{08-03-22a} that
$$
U^{\psi''}_t \pi_{\psi''}(m)U^{\psi''}_{-t} = \pi_{\psi''}(\sigma''_t(m))
$$ for all $t$ and all $m \in N$.

\begin{lemma}\label{02-03-22d} Let $m \in \mathcal N_{\psi''}$ and $z \in \mathbb C$. Then $\sigma''_z\left(\overline{R}_k(m)\right) \in \mathcal N_{\psi''}$ and
$$
\Lambda_{\psi''}\left(\sigma''_z\left(\overline{R}_k(m)\right)\right) = \sqrt{\frac{k}{\pi}} \int_\mathbb R e^{-k(t-z)^2} U^{\psi''}_t \Lambda_{\psi''}(m) \ \mathrm d t .
$$
\end{lemma}
\begin{proof} We use the formula for $\sigma''_z\left(\overline{R}_k(m)\right)$ from the second item in Lemma \ref{02-03-22bx}. Let $F \subseteq N_*$ be a finite subset of the predual and let $\epsilon >0$. Define $H : \mathbb R \to \mathbb C^F \oplus H_{\psi''}$ such that
$$
H(t) := \left( \left( \sqrt{\frac{k}{\pi}} e^{-k(t-z)^2} \omega(\sigma''_t(m)) \right)_{\omega \in F}, \ \sqrt{\frac{k}{\pi}} e^{-k(t-z)^2} U^{\psi''}_t \Lambda_{\psi''}(m) \right).
$$
This is a continuous map and $\int_\mathbb R \left\|H(t)\right\| \ \mathrm d t < \infty$. By Lemma \ref{12-02-22} in Appendix \ref{integration} there are therefore finite sets of numbers $s_i \geq 0$ and $t_i \in \mathbb R$, $i =1,2,\cdots, n$, such that the element
$$
a(F,\epsilon) := \sum_{i=1}^n \sqrt{\frac{k}{\pi}} e^{-k(t_i-z)^2}\sigma''_{t_i}(m)s_i \in \mathcal N_{\psi''}
$$
has the property that
$$
\left|\omega(a(F,\epsilon)) - \omega\left(\sigma''_z\left(\overline{R}_k(m)\right)\right)\right| \leq \epsilon
$$
for all $\omega \in F$ and
$$
\left\| \Lambda_{\psi''}(a(F,\epsilon)) - \sqrt{\frac{k}{\pi}} \int_\mathbb R e^{-k(t-z)^2} U^{\psi''}_t \Lambda_{\psi''}(m) \ \mathrm d t \right\| \leq \epsilon .
$$ 
When we order the pairs $(F,\epsilon)$ such that $(F,\epsilon) \leq (F',\epsilon')$ means $F \subseteq F'$ and $\epsilon' \leq \epsilon$, we get a directed set $J$ such that $(a(F,\epsilon))_{(F,\epsilon) \in J}$ is a net with the property that 
$$
\lim_{(F,\epsilon) \to \infty} a(F,\epsilon) = \sigma''_z(\overline{R}_k(m))
$$ 
in the $\sigma$-weak topology, and
$$
\lim_{(F,\epsilon) \to \infty} \Lambda_{\psi''}\left(a(F,\epsilon)\right) = \sqrt{\frac{k}{\pi}} \int_\mathbb R e^{-k(t-z)^2} U^{\psi''}_t \Lambda_{\psi''}(m) \ \mathrm d t .
$$
In this way the desired conclusion follows from Lemma \ref{22-02-22fx}.
\end{proof}

\begin{cor}\label{02-03-22e} Let $m \in \mathcal N_{\psi''}$. Then 
$$
\lim_{k \to \infty} \Lambda_{\psi''}\left(\overline{R}_k(m)\right) = \Lambda_{\psi''}(m).
$$
\end{cor}
\begin{proof} This follows from Lemma \ref{02-03-22d} because
$$
\lim_{k \to \infty} \sqrt{\frac{k}{\pi}} \int_\mathbb R e^{-kt^2} U^{\psi''}_t \Lambda_{\psi''}(m) \ \mathrm d t  = \Lambda_{\psi''}(m)
$$
by Lemma \ref{24-11-21}.
\end{proof}

\begin{lemma}\label{02-03-22hx} Let $k \in \mathbb N$ and $z \in \mathbb C$. Let $\{a_i\}_{i \in J}$ be a net in $A$ and $m \in N$ an element such that $\sup_{i \in J} \left\|\pi_{\psi}(a_i)\right\| < \infty$ and $\lim_{i \to \infty} \pi_\psi(a_i) = m$ in the strong operator topology. It follows that
$$
\lim_{i \to \infty} \pi_\psi\left( \sigma_z\left(R_k(a_i)\right)\right) = \sigma''_z(\overline{R}_k(m))
$$
in the strong operator topology.
\end{lemma}
\begin{proof} It follows from Lemma \ref{17-11-21i} and \eqref{08-03-22b} that
$$
\pi_\psi\left(\sigma_z\left(R_k(a_i)\right)\right) =  \sqrt{\frac{k}{\pi}} \int_\mathbb R e^{-k(t-z)^2} \sigma''_t(\pi_\psi(a_i)) \ \mathrm d t 
$$
and from Lemma \ref{02-03-22bx} that
$$
\sigma''_z(\overline{R}_k(m)) = \sqrt{\frac{k}{\pi}} \int_\mathbb R e^{-k(t-z)^2} \sigma''_t(m) \ \mathrm d t .
$$
Let $\chi \in H_\psi$. Since $\mathbb R \ni t \mapsto \sigma''_t(m)\chi = U^\psi_tmU^\psi_{-t}\chi$ is continuous the integral
$$
\sqrt{\frac{k}{\pi}} \int_\mathbb R e^{-k(t-z)^2} \sigma''_t(m)\chi \ \mathrm d t 
$$
exists in $H_\psi$, cf. Appendix \ref{integration},  and hence
$$
\sigma''_z(\overline{R}_k(m))\chi = \sqrt{\frac{k}{\pi}} \int_\mathbb R e^{-k(t-z)^2} \sigma''_t(m)\chi \ \mathrm d t .
$$
Similarly, 
$$
\pi_\psi\left(\sigma_z\left(R_k(a_i)\right)\right)\chi =  \sqrt{\frac{k}{\pi}} \int_\mathbb R e^{-k(t-z)^2} \sigma''_t(\pi_\psi(a_i))\chi \ \mathrm d t 
$$
for all $i$. Let $\epsilon > 0$. Since $\sup_{i,t} \left\|\sigma''_t(\pi_\psi(a_i))\chi\right\| < \infty$ by assumption, there is an $R > 0$ such that
$$
\left\| \pi_\psi\left(\sigma_z\left(R_k(a_i)\right)\right)\chi - \sqrt{\frac{k}{\pi}} \int_{-R}^R e^{-k(t-z)^2} \sigma''_t(\pi_\psi(a_i))\chi \ \mathrm d t \right\| \leq \epsilon 
$$
for all $i$. By increasing $R$ we can arrange that 
$$
\left\| \sigma''_z(\overline{R}_k(m))\chi - \sqrt{\frac{k}{\pi}} \int_{-R}^Re^{-k(t-z)^2} \sigma''_t(m)\chi \ \mathrm d t \right\| \leq \epsilon .
$$
Set $K = \sqrt{\frac{k}{\pi}} \int_{-R}^R \left|e^{-k(t-z)^2}\right|  \ \mathrm d t$. Since $\{U^\psi_t \chi : \ t \in [-R,R]\}$ is a compact set in $H_\psi$ and $\lim_{i \to \infty} \pi_\psi(a_i) = m$ strongly, there is an $i_0 \in I$ such that
$$
K\left\|\pi_\psi(a_i)U^{\psi}_{-t}\chi - mU^\psi_{-t}\chi\right\| \leq \epsilon
$$
for all $t \in [-R,R]$ when $i \geq i_0$. Then
$$
\left\| \sqrt{\frac{k}{\pi}} \int_{-R}^Re^{-k(t-z)^2} \sigma''_t(m)\chi \ \mathrm d t - \sqrt{\frac{k}{\pi}} \int_{-R}^R e^{-k(t-z)^2} \sigma''_t(\pi_\psi(a_i))\chi \ \mathrm d t \right\| \leq \epsilon
$$
and hence
$$
\left\| \sigma''_z(\overline{R}_k(m))\chi - \pi_\psi\left(\sigma_z\left(R_k(a_i)\right)\right)\chi\right\| \leq 3 \epsilon
$$
when $i \geq i_0$.

\end{proof}

\begin{lemma}\label{02-03-22g} Let $m \in N_a$. Then $\psi''(m^*m) = \psi''\left(\sigma''_\xi(m)\sigma''_\xi(m)^*\right)$.
\end{lemma}
\begin{proof}  Let $\epsilon > 0$ and $F_1 \in \mathcal J$. Since $\pi_\psi$ is a non-degenerate representation it follows from Lemma \ref{01-03-22e} that there is an element $F_2 \in \mathcal J$ such that 
\begin{align*}
& \sum_{x \in F_1} \left<m\Lambda_\psi(\sigma_{-\xi}(x)), m\Lambda_\psi(\sigma_{-\xi}(x)) \right> - \epsilon \\
&\leq \sum_{x \in F_1,y \in F_2} \left<\pi_\psi(y^*)m\Lambda_\psi(\sigma_{-\xi}(x)), \pi_{\psi}(y^*)m\Lambda_\psi(\sigma_{-\xi}(x)) \right> .
\end{align*} 
Since $\pi_\psi(A)$ is $\sigma$-weakly dense in $N$ and $\mathcal M^\sigma_\psi$ is a norm-dense $*$-subalgebra of $A$ by Lemma \ref{07-12-21a} it follows from Kaplansky's density theorem that there is a net $\{a_i\}_{i \in J'}$ in $\mathcal M^\sigma_\psi$ such that $\left\|\pi_\psi(a_i)\right\| \leq \|m\|$ for all $i$ and $\lim_{i \to \infty}\pi_\psi(a_i) = m$ in the $\sigma$-strong* topology, cf. Theorem 2.4.16 of \cite{BR}. Since $m \in N_a$ it follows from Lemma \ref{17-11-21i} and Lemma \ref{24-11-21} applied to the restriction of $\sigma''$ to $N_c$ that
$$
\lim_{k \to \infty} \sigma''_z\left(\overline{R}_k(m)\right) = \sigma''_z(m)
$$
in norm for all $z \in \mathbb C$. It follows from Lemma \ref{02-03-22hx} that 
$$
\lim_{i \to \infty} \pi_\psi(\sigma_z(R_k(a_i^*))) = \sigma''_z\left(\overline{R}_k(m^*)\right)
$$ 
in the strong operator topology for all $z \in \mathbb C$. We can therefore construct from 
$$
\left\{{R}_k(a_i) : \ i \in J', k \in \mathbb N\right\}$$
 a net $\{b_i\}_{i \in J}$ in $\mathcal M^\sigma_\psi$ with the property that $\lim_{i \to \infty} \pi_\psi\left( b_i\right) = m$ and 
 $$
 \lim_{i \to \infty} \pi_\psi\left( \sigma_\xi(b_i)\right)^* = \sigma''_\xi(m)^*
 $$ 
 in the strong operator topology. Then
\begin{align*}
& \sum_{x \in F_1,y \in F_2} \left<\pi_\psi(y^*)m\Lambda_\psi(\sigma_{-\xi}(x)), \pi_{\psi}(y^*)m\Lambda_\psi(\sigma_{-\xi}(x)) \right> \\
& =\lim_{i \to \infty} \sum_{x \in F_1,y \in F_2} \left<\pi_\psi(y^*)\pi_{\psi}(b_i)\Lambda_\psi(\sigma_{-\xi}(x)), \pi_{\psi}(y^*)\pi_\psi(b_i)\Lambda_\psi(\sigma_{-\xi}(x)) \right> .
\end{align*}
We note that
\begin{align*}
&\sum_{x \in F_1,y \in F_2} \left<\pi_\psi(y^*)\pi_{\psi}(b_i)\Lambda_\psi(\sigma_{-\xi}(x)), \pi_{\psi}(y^*)\pi_\psi(b_i)\Lambda_\psi(\sigma_{-\xi}(x)) \right> \\
&   =  \sum_{x \in F_1,y \in F_2} \psi( \sigma_{-\xi}(x)^* b_i^* y y^* b_i\sigma_{-\xi}(x)) \\
& =  \sum_{x \in F_1,y \in F_2} \psi( \sigma_\xi(y^*) \sigma_\xi(b_i) xx^* \sigma_\xi(b_i)^*\sigma_\xi(y^*)^*) \\
& \leq  \sum_{y \in F_2} \psi( \sigma_{-\xi}(y)^* \sigma_\xi(b_i)\sigma_\xi(b_i)^*\sigma_{-\xi}(y)) \\
& = \sum_{y \in F_2} \left< \pi_\psi(\sigma_\xi(b_i))^* \Lambda_\psi\left(\sigma_{-\xi}(y)\right) , \pi_\psi(\sigma_\xi(b_i))^* \Lambda_\psi\left(\sigma_{-\xi}(y)\right)\right> \\
\end{align*}
for all $i$. Since
$\lim_{i \to \infty} \pi_\psi(\sigma_\xi(b_i))^* = \sigma''_\xi(m)^*$ it follows that
\begin{align*}
&\sum_{x \in F_1,y \in F_2} \left<\pi_\psi(y^*)m\Lambda_\psi(\sigma_{-\xi}(x)), \pi_{\psi}(y^*)m\Lambda_\psi(\sigma_{-\xi}(x)) \right>\\
& \leq \sum_{y \in F_2} \left< \sigma''_\xi(m)^* \Lambda_\psi\left(\sigma_{-\xi}(y)\right) ,  \sigma''_\xi(m)^*\Lambda_\psi\left(\sigma_{-\xi}(y)\right)\right> .
\end{align*} 
Hence
\begin{align*}
& \sum_{x \in F_1} \left<m\Lambda_\psi(\sigma_{-\xi}(x)), m\Lambda_\psi(\sigma_{-\xi}(x)) \right> - \epsilon \leq \psi''( \sigma_\xi(m)\sigma_\xi(m)^*).
\end{align*} 
Since $F_1 \in \mathcal J$ and $\epsilon > 0$ were arbitrary it follows that $\psi''(m^*m) \leq \psi''( \sigma_\xi(m)\sigma_\xi(m)^*)$. Since $\sigma''_{-\xi}(m^*) \in N_a$, we can insert it in the place of $m$ to get
$$
 \psi''( \sigma_\xi(m)\sigma_\xi(m)^*) = \psi''\left( \sigma''_{-\xi}(m^*)^* \sigma''_{-\xi}(m^*)\right) \leq \psi''(m^*m) ,
 $$
and the lemma follows.

\end{proof}

\begin{lemma}\label{03-03-22fx} $\psi'' \circ \pi_\psi(a) = \psi(a)$ for all $a \in A^+$.
\end{lemma}
\begin{proof} Let $b \in \mathcal A_\sigma$. For $F \in \mathcal J$ we find that 
\begin{align*}
&\sum_{x \in F} \left<\pi_{\psi}(b^*b) \Lambda_{\psi}(\sigma_{-\xi}(x)),\Lambda_{\psi}(\sigma_{-\xi}(x))\right> = \sum_{x \in F} \psi(\sigma_{-\xi}(x)^*b^*b\sigma_{-\xi}(x)) \\
& = \sum_{x \in F} \psi(\sigma_\xi(b)xx^*\sigma_\xi(b)^*) \leq \psi(\sigma_\xi(b)\sigma_\xi(b)^*) = \psi(b^*b).
\end{align*}
It follows that $\psi'' \circ \pi_{\psi}(b^*b) \leq \psi(b^*b)$. To obtain the reverse inequality, let $\epsilon >0$ and $N > 0$ be given. It follows from Lemma \ref{01-03-22e} and the lower semi-continuity of $\psi$ that there is $F\in \mathcal J$ such that
\begin{align*}
&\min\{\psi(b^*b) -\epsilon, N\} \leq \sum_{x \in F} \psi(b^*xx^*b) \\
&=\sum_{x \in F} \psi(\sigma_\xi(x^*b) \sigma_\xi(x^*b)^*)=\sum_{x \in F} \psi(\sigma_{-\xi}(x)^*\sigma_\xi(b) \sigma_\xi(b)^* \sigma_{-\xi}(x)) .
\end{align*}
Since $\pi_\psi(b) \in N_a$ and $\sigma''_t\left( \pi_\psi(b)\right) = \pi_\psi(\sigma_t(b))$ for all $t$ by \eqref{08-03-22b}, it follows that $\sigma''_z\left( \pi_\psi(b)\right) = \pi_\psi(\sigma_z(b))$ for all $z \in \mathbb C$. In particular,
\begin{align*}
&\sum_{x \in F} \psi(\sigma_{-\xi}(x)^*\sigma_\xi(b) \sigma_\xi(b)^* \sigma_{-\xi}(x)) \\
& = \sum_{x \in F} \left<\sigma''_\xi(\pi_\psi(b))^*\Lambda_\psi(\sigma_{-\xi}(x)) , \sigma''_\xi(\pi_\psi(b))^*\Lambda_\psi(\sigma_{-\xi}(x))\right>\\
& \leq \psi''\left(\sigma''_\xi(\pi_\psi(b))\sigma''_\xi(\pi_\psi(b))^*\right) = \psi''\left(\pi_\psi(b)^*\pi_\psi(b)\right)  = \psi''\left(\pi_\psi(b^*b)\right) ,
\end{align*}
where we have used Lemma \ref{02-03-22g} for the penultimate step. We conclude that $\psi(b^*b) \leq \psi''\left(\pi_\psi(b^*b)\right)$ and hence that $\psi(b^*b) = \psi''\left(\pi_\psi(b^*b)\right)$. It follows therefore from Lemma \ref{01-03-22b} that $\psi'' \circ \pi_\psi|_A = \psi$.
\end{proof}

\subsubsection{Proof of Theorem \ref{21-02-22dx}:}
 Set $B := N_a \cap \mathcal N_{\psi''} \cap \mathcal N_{\psi''}^*$. Since $\sigma''_t \circ \pi_\psi = \pi_\psi \circ \sigma_t$ we have that $\pi_\psi(\mathcal A_\sigma) \subseteq N_a$ and it follows from Lemma \ref{03-03-22fx} that $\pi_\psi(\mathcal N_\psi \cap \mathcal N_\psi^*) \subseteq \mathcal N_{\psi''}\cap \mathcal N_{\psi''}^*$. Hence $\pi_\psi(\mathcal M^\sigma_\psi) \subseteq \pi_\psi(\mathcal A_\sigma \cap \mathcal N_\psi \cap \mathcal N_\psi^*) \subseteq B$. In particular, $B \neq 0$; in fact, as shown in the proof of Lemma \ref{02-03-22g} $\pi_\psi(\mathcal M^\sigma_\psi)$, and hence also $B$ is strongly dense in $N$. Since $B$ is $*$-subalgebra of $N$ its norm-closure $\overline{B}$ is a $C^*$-algebra. Note that $\sigma''_t(\overline{B}) = \overline{B}$ for all $t \in \mathbb R$ and that $\mathbb R \ni t \mapsto \sigma''_t(b)$ is norm-continuous for $b \in \overline{B}$ since this is true for elements of $N_a$. Since $\psi''$ is lower semi-continuous with respect to the $\sigma$-weak topology, it is also lower semi-continuous with respect to the norm-topology and hence $\psi''|_{\overline{B}}$ is a weight on $\overline{B}$, and it is $\sigma''$-invariant since $\psi''$ is. Since $\psi''(b^*b) < \infty$ for all $b \in B$ it follows that $\psi''|_{\overline{B}}$ is densely defined on $\overline{B}$. Note also that since $\psi''$ is faithful by Lemma \ref{01-03-22f} it follows that $\psi''|_B$ is not zero because $B$ is not. It follows from Lemma \ref{02-03-22g} that $\psi''|_{\overline{B}}$ has the property (2) of Theorem \ref{24-11-21d} and we conclude therefore that $\psi''|_{\overline{B}}$ is a $\beta$-KMS weight for the restriction of $\sigma''$ to $\overline{B}$.

Let $a,b \in \mathcal N_{\psi''} \cap \mathcal N_{\psi''}^*$. Set $a_k := \overline{R}_k(a), \ b_k := \overline{R}_k(b)$. It follows from Lemma \ref{02-03-22bx} and Lemma \ref{02-03-22d} that $a_k,b_k$ are elements of $B$ for all $k$. It follows therefore from condition (4) of Theorem \ref{24-11-21d} that there are continuous functions $f_k: {\mathcal D_\beta} \to \mathbb C$ that are holomorphic in the interior $\mathcal D_\beta^0$ of ${\mathcal D_\beta}$ such that
\begin{itemize}
\item $f_k(t) = \psi''(b_k \sigma''_t(a_k)) \ \ \forall t \in \mathbb R$, and
\item $f_k(t+i\beta) = \psi''(\sigma''_t(a_k)b_k) \ \ \forall t \in \mathbb R$.
\end{itemize}
In order to apply the theorem of Phragmen-Lindel\"of, Proposition 5.3.5 in \cite{BR}, to $f_k$ we must show that $f_k$ is bounded in $\mathcal D_\beta$. \footnote{If one is willing to use the observation from the footnote to the proof of Theorem \ref{24-11-21d} this step is not needed.} To this end note that
$$
f_k(t) = \left< \Lambda_{\psi''}(\sigma''_t(a_k)), \Lambda_{\psi''}(b_k^*)\right>
$$
It follows from Lemma \ref{02-03-22d} that $\sigma''_z(a_k) \in \mathcal N_{\psi''}$ and
$$
\Lambda_{\psi''}(\sigma''_z(a_k)) = \sqrt{\frac{k}{\pi}} \int_\mathbb R e^{-k (t-z)^2} U^{\psi''}_t \Lambda_{\psi''}(a) \ \mathrm d t = U^{\psi''}_z R_k(\Lambda_{\psi''}(a)) 
$$
for all $z \in \mathbb C$. In particular the function $z \mapsto  \left< \Lambda_{\psi''}(\sigma''_z(a_k)), \Lambda_{\psi''}(b_k^*)\right>$ is an entire holomorphic function which agree with $f_k$ on $\mathbb R$. It follows therefore from Proposition 5.3.6 in \cite{BR} that the two functions agree on $\mathcal D_\beta$. Since 
$$
\sup_{z \in \mathcal D_\beta} \left\|U^{\psi''}_z R_k(\Lambda_{\psi''}(a))\right\| < \infty
$$    
by Lemma \ref{24-09-23x} it follows that $f_k$ is indeed bounded on $\mathcal D_\beta$. For $z \in \mathcal D_\beta$ we get therefore the estimate
\begin{equation}\label{18-11-21dx}
\left|f_n(z) - f_m(z)\right| \leq \max \left\{ \sup_{t \in \mathbb R} |f_n(t) -f_m(t)|, \ \sup_{t \in \mathbb R} |f_n(t+i\beta) -f_m(t+i\beta)|\right\}
\end{equation}
from Proposition 5.3.5 in \cite{BR} (Phragmen-Lindel\"of) for all $n,m \in \mathbb N$.
We note that 
\begin{align*}
&f_n(t)  =\left< U^{\psi''}_t\Lambda_{\psi''}(a_n),\Lambda_{\psi''}(b_n^*)\right>,
\end{align*}
and that $ U^{\psi''}_t\Lambda_{\psi''}(a_n)$ converges to $U^{\psi''}_t\Lambda_{\psi''}(a)$ uniformly in $t$ since 
$$
\lim_{n \to \infty} \Lambda_{\psi''}(a_n) = \Lambda_{\psi''}(a)
$$ 
by Corollary \ref{02-03-22e}. The same corollary also shows that $\lim_{n \to \infty} \Lambda_{\psi''}(b_n^*)= \Lambda_{\psi''}(b^*)$ and it follows therefore that 
$$
\lim_{n \to \infty} f_n(t) = \left<U^{\psi''}_t\Lambda_{\psi''}(a),\Lambda_{\psi''}(b^*)\right> = \psi''(b\sigma''_t(a))
$$ 
uniformly in $t$. Since
\begin{equation*}\label{25-11-21c}
f_n(t+i\beta) =\psi''(\sigma''_t(a_n)b_n) = \left<\Lambda_{\psi''}(b_n), U^{\psi''}_t\Lambda_{\psi''}(a_n^*)\right> 
\end{equation*}
we find in the same way that
$$
\lim_{n \to \infty} f_n(t+i\beta) = \left<\Lambda_{\psi''}(b), U^{\psi''}_t\Lambda_{\psi''}(a^*)\right>  = \psi''(\sigma''_t(a)b)
$$
uniformly in $t$. It follows now from the estimate \eqref{18-11-21dx} that the sequence $\{f_n\}$ converges uniformly on ${\mathcal D_\beta}$ to a continuous bounded function $f : {\mathcal D_\beta} \to \mathbb C$ which is holomorphic in $\mathcal D_\beta^0$ and has the required properties.

\qed

With an appropriate definition of KMS weights for normal flows on von Neumann algebras Theorem \ref{21-02-22dx} would say that $\psi''$ is a $\beta$-KMS weight for $\sigma''$. We shall only need the notion of KMS weights for normal flows on von Neumann algebras; only the following which does not require any additional definition.

\begin{lemma}\label{06-10-23c} Let $\psi$ be a lower semi-continuous trace on the $C^*$-algebra $A$. There is a normal faithful semi-finite trace $\psi''$ on $\pi_\psi(A)''$ such that $\psi'' \circ \pi_\psi = \psi$.
\end{lemma}
\begin{proof} Let $x \in \pi_\psi(A)''$. By considering $\psi$ as a $1$-KMS weight for the trivial action on $A$ it follows from Lemma \ref{01-03-22f} that $\psi''$ is a normal faithful semi-finite weight. The identity $\psi'' \circ \pi_\psi = \psi$ follows from Lemma \ref{03-03-22fx}. It remains therefore only to show that $\psi''(x^*x) = \psi''(xx^*)$. This follows from Lemma \ref{02-03-22g} by using that $\sigma''$ is the trivial flow when $\sigma$ is so that $\pi_\psi(A)''_a = \pi_\psi(A)''$ and $\sigma''_\xi = \id$.
\end{proof}

\section{The modular automorphism group and the modular conjugation associated with $\psi''$}

Let $f: \mathcal D_\beta \to \mathbb C$ be the function from Theorem \ref{21-02-22dx}. We define $H : \mathcal D_1 \to \mathbb C$ such that 
$$
H(z) := f(i\beta - \beta z) . 
$$
Then $H$ is continuous, holomorphic in $\mathcal D^0_1$, 
$$
H(t) = f(-\beta t + i \beta) = \psi''(\sigma''_{-\beta t}(a)b)  \ \ \forall t \in \mathbb R,
$$
and
$$
H(t+i) = f(-\beta t) =  \psi''(b \sigma''_{- \beta t}(a))  \ \ \forall t \in \mathbb R.
$$
It follows that $\left\{\sigma''_{-\beta t}\right\}_{t \in \mathbb R}$ is the modular automorphism group relative to $\psi''$, cf. Section \ref{modular1} and Theorem 9.2.38 of \cite{KR}. We can therefore summarize Lemma \ref{01-03-22f}, Lemma \ref{01-03-22g} and Theorem \ref{21-02-22dx} as follows.

\begin{thm}\label{05-03-22x} Let $\psi$ be a $\beta$-KMS weight for the flow $\sigma$ on $A$. There is a $\sigma''$-invariant and faithful normal semi-finite weight $\psi''$ on $\pi_\psi(A)''$ such that $\psi'' \circ \pi_\psi = \psi$ and the modular automorphism group $\alpha$ associated to $\psi''$ is given by 
$$
\alpha_t(m) = \sigma''_{-\beta t}(m) = U^\psi_{-\beta t} m  U^\psi_{\beta t}
$$
for all $t \in \mathbb R$ and all $m \in \pi_\psi(A)''$.
\end{thm}

We proceed towards a description of the modular conjugation operator for $\psi''$.

\begin{lemma}\label{24-02-22x} Let $\nabla$ be the modular operator of $\psi''$. Let $z \in \mathbb C$ and $a \in \mathcal N_{\psi''} \cap N_a$. Assume $\sigma''_{-\beta z}(a) \in \mathcal N_{\psi''}$. Then $\Lambda_{\psi''}(a) \in D\left(  \nabla^{i z} \right)$ and  $ \nabla^{iz}\Lambda_{\psi''}(a) = \Lambda_{\psi''}\left(\sigma''_{-\beta z}(a)\right)$.
\end{lemma}
\begin{proof} In the case $\beta =0$ it follows from Theorem \ref{05-03-22x} that the modular automorphism group of $\psi''$ is trivial and hence so is the modular operator. In this case the statement of the lemma is trivial, and we may therefore assume that $\beta \neq 0$.

Since $a \in N_a \subseteq N_c$, the $C^*$-subalgebra of $N$ where $\sigma''$ is norm-continuous, we can apply  Lemma \ref{24-11-21} and Lemma \ref{17-11-21i} with $\sigma_t:= \sigma''_{-\beta t}$ to find that
$$
\sqrt{\frac{k}{\pi}} \int_\mathbb R e^{-kt^2} \sigma''_{-\beta t}(a) \ \mathrm d t \in N_a
$$
and
\begin{equation}\label{25-09-23}
\begin{split}
&\sigma''_{-\beta z}\left(\sqrt{\frac{k}{\pi}} \int_\mathbb R e^{-kt^2} \sigma''_{-\beta t}(a) \ \mathrm d t \right) = \sqrt{\frac{k}{\pi}} \int_\mathbb R e^{-k (t -z)^2} \sigma''_{-\beta t}(a) \ \mathrm d t \\
& = \sqrt{\frac{k}{\pi}} \int_\mathbb R e^{-kt^2} \sigma''_{-\beta t}(\sigma''_{-\beta z}(a)) \ \mathrm d t .
\end{split} 
\end{equation}
By the arguments that proved Lemma \ref{02-03-22d} we have
$$
 \sqrt{\frac{k}{\pi}} \int_\mathbb R e^{-k(t-z)^2} \sigma''_{-\beta t}(a) \ \mathrm d t \in \mathcal N_{\psi''},
 $$
 and
 \begin{equation}\label{25-09-23a}
 \Lambda_{\psi''}\left( \sqrt{\frac{k}{\pi}} \int_\mathbb R e^{-k(t-z)^2} \sigma''_{-\beta t}(a) \ \mathrm d t\right) =  \sqrt{\frac{k}{\pi}} \int_\mathbb R e^{-k(t-z)^2} \Lambda_{\psi''}(\sigma''_{-\beta t}(a)) \ \mathrm d t .
\end{equation}
Since we assume that $\sigma''_{-\beta z}(a) \in \mathcal N_{\psi''}$ it follows in the same way that
$$
 \sqrt{\frac{k}{\pi}} \int_\mathbb R e^{-kt^2} \sigma''_{-\beta t}(\sigma''_{-\beta z}(a)) \ \mathrm d t \in \mathcal N_{\psi''}
$$
and
$$
 \Lambda_{\psi''}\left( \sqrt{\frac{k}{\pi}} \int_\mathbb R e^{-kt^2} \sigma''_{-\beta t}(\sigma''_{-\beta z}(a)) \ \mathrm d t\right)  =  \sqrt{\frac{k}{\pi}} \int_\mathbb R e^{-k t^2} \Lambda_{\psi''}(\sigma''_{-\beta t}( \sigma''_{-\beta z}(a  ))) \ \mathrm d t .
$$
Since $\{\sigma''_{-\beta t}\}_{t \in \mathbb R}$ is the modular automorphism group of $\psi''$ by Theorem \ref{05-03-22x}, the modular operator $\nabla$ is related to $\sigma''$ by the equation
$$
\nabla^{it}\Lambda_{\psi''}(a) = \Lambda_{\psi''}(\sigma''_{-\beta t}(a)) \ \ \forall a \in \mathcal N_{\psi''}.
$$
Thus
$$
\sqrt{\frac{k}{\pi}} \int_\mathbb R e^{-k t^2} \Lambda_{\psi''}(\sigma''_{-\beta t}( \sigma''_{-\beta z}(a  )) \ \mathrm d t = \sqrt{\frac{k}{\pi}} \int_\mathbb R e^{-k t^2} \nabla^{it}\Lambda_{\psi''}( \sigma''_{-\beta z}(a  )) \ \mathrm d t .
$$
It follows therefore from Lemma \ref{24-11-21} that 
\begin{equation}\label{25-09-23h}
\lim_{k \to \infty} \Lambda_{\psi''}\left( \sqrt{\frac{k}{\pi}} \int_\mathbb R e^{-kt^2} \sigma''_{-\beta t}(\sigma''_{-\beta z}(a)) \ \mathrm d t\right) = \Lambda_{\psi''}(\sigma''_{-\beta z}(a)) .
\end{equation}
By Lemma \ref{17-11-21i} we have
\begin{equation}\label{25-09-23g}
\begin{split}
&\sqrt{\frac{k}{\pi}} \int_\mathbb R e^{-k(t-z)^2} \Lambda_{\psi''}(\sigma''_{-\beta t}(a)) \ \mathrm d t = \sqrt{\frac{k}{\pi}} \int_\mathbb R e^{-k(t-z)^2} \nabla^{it}\Lambda_{\psi''}(a) \ \mathrm d t \\
& = \nabla^{iz}\left( \sqrt{\frac{k}{\pi}} \int_\mathbb R e^{-k t^2}\nabla^{it}\Lambda_{\psi''}\left(a\right) \ \mathrm d t\right)  .
\end{split}
\end{equation}
Note that 
\begin{equation}\label{25-09-23e}
\lim_{k \to \infty} \sqrt{\frac{k}{\pi}} \int_\mathbb R e^{-k t^2}\nabla^{it}\Lambda_{\psi''}\left(a\right) \ \mathrm d t = \Lambda_{\psi''}(a)
\end{equation}
by Lemma \ref{24-11-21}, and that  
\begin{align*}
&\lim_{k \to \infty} \nabla^{iz}\left( \sqrt{\frac{k}{\pi}} \int_\mathbb R e^{-k t^2}\nabla^{it}\Lambda_{\psi''}\left(a\right) \ \mathrm d t\right)  \\
& = \lim_{k \to \infty}\sqrt{\frac{k}{\pi}} \int_\mathbb R e^{-k(t-z)^2} \Lambda_{\psi''}(\sigma''_{-\beta t}(a)) \ \mathrm d t \ \ \ \ \ \ \ \ \ \ \ \ \text{(by \eqref{25-09-23g})}\\
& =  \lim_{k \to \infty} \Lambda_{\psi''}\left( \sqrt{\frac{k}{\pi}} \int_\mathbb R e^{-k(t-z)^2} \sigma''_{-\beta t}(a) \ \mathrm d t\right)\ \ \ \ \ \ \ \ \ \ \ \ \text{(by \eqref{25-09-23a})} \\
& =  \lim_{k \to \infty} \Lambda_{\psi''}\left(  \sqrt{\frac{k}{\pi}} \int_\mathbb R e^{-kt^2} \sigma''_{-\beta t}(\sigma''_{-\beta z}(a)) \ \mathrm d t\right)\ \ \ \ \ \ \ \ \ \ \ \ \text{(by \eqref{25-09-23})}\\
& = \Lambda_{\psi''}(\sigma''_{-\beta z}(a))\ \ \ \ \ \ \ \ \ \ \ \ \text{(by \eqref{25-09-23h})} .
\end{align*}
Since $\nabla^{iz}$ is a closed operator we can combine this with \eqref{25-09-23e} to find that $\Lambda_{\psi''}(a) \in D(\nabla^{iz})$ and $\nabla^{iz}\Lambda_{\psi''}(a) = \Lambda_{\psi''}\left(\sigma''_{-\beta z}(a)\right)$.\footnote{In the statement and the proof of Lemma \ref{24-02-22x} the symbol $\nabla^{iz}$ is used with two apparently different meanings. First in equation \eqref{25-09-23g} as the holomorphic extension defined as in Section \ref{holextensions} from the flow $\nabla^{it}$ on $H_{\psi''}$ and then, in the statement of the lemma and in the closing lines of the proof, as the operator on $H_{\psi''}$ defined by spectral theory from the self-adjoint generator of the flow. It is crucial to the argument that the operator $\nabla^{iz}$ arising in these ways is the same. See Example \ref{15-06-22}. }

\end{proof}

\begin{lemma}\label{03-03-22c} The modular conjugation operator $J$ for $\psi''$ satisfies that 
$$
J \Lambda_{\psi''}(a) = \Lambda_{\psi''}\left( \sigma''_{-i\frac{\beta}{2}}(a)^*\right)
$$ 
for all $a \in \mathcal N_{\psi''} \cap \mathcal N_{\psi''}^* \cap N_a$ such that $\sigma''_{\pm i \frac{\beta}{2}}(a) \in \mathcal N_{\psi''}\cap \mathcal N_{\psi''}^*$.
\end{lemma}
\begin{proof} Fix an element $a \in \mathcal N_{\psi''} \cap \mathcal N_{\psi''}^* \cap N_a$ such that $\sigma''_{\pm i \frac{\beta}{2}}(a) \in \mathcal N_{\psi''}\cap \mathcal N_{\psi''}^*$. Applying Lemma \ref{24-02-22x} with $z= - \frac{i}{2}$ we find that
$$
\nabla^{\frac{1}{2}}\Lambda_{\psi''}(a) = \Lambda_{\psi''}\left(\sigma''_{\frac{i \beta}{2}}(a)\right) .
$$
Thus
$$
\Lambda_{\psi''}(a^*) = J \nabla^{\frac{1}{2}}\Lambda_{\psi''}(a) = J \Lambda_{\psi''}\left(\sigma''_{\frac{i \beta}{2}}(a)\right) .
$$
 Since $J^2 =1$ this implies that
$$
J \Lambda_{\psi''}(a^*) = \Lambda_{\psi''}\left(\sigma''_{\frac{i \beta}{2}}(a)\right) .
$$
Note that we can substitute $a^*$ for $a$. This gives
$$
J \Lambda_{\psi''}(a) = \Lambda_{\psi''}\left(\sigma''_{\frac{i \beta}{2}}(a^*)\right)  =  \Lambda_{\psi''}\left(\sigma''_{-i\frac{ \beta}{2}}(a)^*\right).
$$
\end{proof}

It follows from Lemma \ref{03-03-22fx} that we can define an isometry $W : H_\psi \to H_{\psi''}$ such that $W\Lambda_\psi(a) := \Lambda_{\psi''}\left(\pi_\psi(a)\right)$ for $a \in \mathcal N_\psi$. It is straightforward to check that 
$$
W \pi_\psi(x) \Lambda_\psi(a) = \pi_{\psi''}\left(\pi_\psi(x)\right)W\Lambda_\psi(a)
$$ 
when $x \in A$ and $a \in \mathcal N_\psi$. Since $\pi_{\psi''}$ is $\sigma$-weakly continuous and $\pi_\psi(\mathcal N_\psi)$ is $\sigma$-weakly dense in $N$, it follows that
\begin{equation}\label{23-02-22ex}
Wm = \pi_{\psi''}(m)W \ \ \forall m \in N .
\end{equation}

\begin{lemma}\label{23-02-22dx} $W$ is a unitary.
\end{lemma}
\begin{proof} Let $E \in B(H_{\psi''})$ be the orthogonal projection onto the range 
$$\overline{\left\{\Lambda_{\psi''}\left(\pi_\psi(a)\right): \ a \in \mathcal N_\psi\right\}}
$$ 
of $W$. Note that
$$
\pi_{\psi''}(\pi_\psi(b))\Lambda_{\psi''}\left(\pi_\psi(a)\right) =  \Lambda_{\psi''}\left(\pi_\psi(ba)\right)  ,
$$
for all $b \in A$ and $a \in \mathcal N_\psi$, which shows that $E \in \pi_{\psi''}(N)'$. It follows from Lemma \ref{03-03-22c} that the modular conjugation operator $J$ satisfies
$$
J \Lambda_{\psi''}(m) = \Lambda_{\psi''}\left( \sigma''_{-i\frac{\beta}{2}}(m)^*\right)
$$
when $m \in \mathcal N_{\psi''} \cap \mathcal N_{\psi''}^* \cap N_a$ and  $\sigma''_{\pm i \frac{\beta}{2}}(m) \in \mathcal N_{\psi''}\cap \mathcal N_{\psi''}^*$. When $a \in \mathcal M^\sigma_\psi$, the element $\pi_\psi(a)$ is in $\mathcal N_{\psi''} \cap \mathcal N_{\psi''}^* \cap N_a$ and it follows from Lemma \ref{03-03-22fx} that $\sigma''_{\pm i \frac{\beta}{2}}(\pi_\psi(a)) = \pi_\psi(\sigma_{\pm i \frac{\beta}{2}}(a)) \in \mathcal N_{\psi''}\cap \mathcal N_{\psi''}^*$. Therefore,
\begin{align*}
&JW\Lambda_\psi(a) = J \Lambda_{\psi''}(\pi_\psi(a)) = \Lambda_{\psi''}\left(\sigma''_{-i\frac{\beta}{2}}(\pi_\psi(a))^*\right)\\
& =  \Lambda_{\psi''}\left( \pi_\psi( \sigma_{-i\frac{\beta}{2}}(a)^*)\right) \in WH_{\psi}.
\end{align*}
Since $\left\{\Lambda_\psi(a) : \ a \in \mathcal M^\sigma_\psi\right\}$ is dense in $H_\psi$ by Lemma \ref{07-12-21c} it follows that $EJE = JE$, implying that $J$ commutes with $E$ since $J^* = J$. It follows that $E = JEJ \in J\pi_{\psi''}(N)'J = \pi_{\psi''}(N)$, cf. \eqref{03-03-22h}. That is, $E$ is in the center of $\pi_{\psi''}(N)$. Since $\pi_{\psi''}$ is an isomorphism there is an $E_0$ in the center of $N$ such that $\pi_{\psi''}(E_0) = E$. Using \eqref{23-02-22ex} we find that
$$
W (1-E_0) = \pi_{\psi''}(1-E_0)W = (1-E)W = 0 .
$$
Since $W$ is an isometry this implies that $E_0 = 1$ and hence $E =1$; that is, $W$ is surjective, and hence a unitary. 
\end{proof}

Thanks to Lemma \ref{23-02-22dx} we can move the operators associated with the modular theory of the weight $\psi''$ from the Hilbert space $H_{\psi''}$ back to $H_\psi$. We are particularly interested in the modular conjugation operator. Note that it follows from one of the defining properties of a $\beta$-KMS weight, more specifically from condition (1) in Theorem \ref{24-11-21d}, that there is a conjugate linear isometry 
$$
J_\psi : H_\psi \to H_\psi
$$
given by the formula
$$
J_\psi\Lambda_\psi(a) = \Lambda_\psi\left( \sigma_{-i\frac{\beta}{2}}(a)^*\right)
$$
when $a \in \mathcal N_\psi \cap \mathcal A_\sigma$. This observation was used for condition (E) in Section \ref{GNS-KMS} and in Proposition \ref{08-02-22a}.

\begin{lemma}\label{04-03-22} Let $J$ be the modular conjugation operator of the weight $\psi''$ on $\pi_\psi(A)''$ and $W$ the unitary of Lemma \ref{23-02-22dx}. Then
$$
W^*JW = J_\psi.
$$
\end{lemma}
\begin{proof} Let $a \in \mathcal M^\sigma_\psi$. By using Lemma \ref{03-03-22c} we find
\begin{align*}
& W^*JW\Lambda_\psi(a) = W^*J \Lambda_{\psi''}(\pi_\psi(a)) = W^*\Lambda_{\psi''}\left( \sigma''_{-i\frac{\beta}{2}}(\pi_\psi(a))^*\right) \\
& =  W^*\Lambda_{\psi''}\left( \pi_{\psi}\left({\sigma}_{-i\frac{\beta}{2}}(a)^*\right)\right) = \Lambda_\psi\left( \sigma_{-i\frac{\beta}{2}}(a)^*\right) = J_\psi \Lambda_\psi(a).
\end{align*}
This proves the lemma because $\Lambda_\psi(  \mathcal M^\sigma_\psi)$ is dense in $H_\psi$ by Lemma \ref{07-12-21c}.
\end{proof}

Thanks to Lemma \ref{04-03-22} we get now the following from \eqref{23-02-22ex} and \eqref{03-03-22h}.

\begin{thm}\label{04-03-22a} Let $\sigma$ be a flow on the $C^*$-algebra $A$ and $\psi$ a KMS weight for $\sigma$. Then
$$
J_\psi \pi_\psi(A)' J_\psi = \pi_\psi(A)'' .
$$
\end{thm}

\bigskip

\begin{notes}\label{09-03-22} We refer to \cite{KR}, \cite{SZ} and \cite{Ta2} for further references to the modular theory of von Neumann algebras. The connection to modular theory underlies the first papers on weights on $C^*$-algebras, \cite{C1}, \cite{C2}, and is very explicit in the work of Vigand Pedersen, \cite{VP}, and in the work of Kustermans and Vaes, \cite{Ku1}, \cite{Ku2}, \cite{KV1}. The main results we have presented in the last two sections were obtained by Kustermans and Vaes in Section 2.2 of \cite{KV1}, but we have chosen a different route inspired by the proof of the theorem of Laca and Neshveyev in Chapter \ref{LN}. 
\end{notes}

\section{On the uniqueness of $\psi''$}


We have defined a weight $\phi$ on a von Neumann algebra $M$ to be normal when it is lower semi-continuous for the $\sigma$-weak topology on $M^+$. Thanks to a result by Haagerup this condition is equivalent to several other conditions that appear in the literature. In \cite{Ha} Haagerup proved the  following von Neumann algebra version of Combes' theorem, Theorem \ref{04-11-21k}. 

\begin{thm}\label{05-04-22a} (Haagerup, \cite{Ha}) Let $M$ be a von Neumann algebra and $\lambda : M^+ \to [0,\infty]$ a map such that
\begin{itemize}
\item[(a)] $\lambda(a+b) \leq \lambda(a) + \lambda(b) \ \ \forall a,b \in M^+$,
\item[(b)] $a \leq b \Rightarrow \lambda(a) \leq \lambda(b) \ \ \forall a,b \in M^+$,
\item[(c)] $\lambda( t a) = t \lambda(a) \ \ \forall a \in M^+, \ \forall t \in \mathbb R^+$, with the convention that $0 \cdot \infty = 0$, and  
\item[(d)] $\lambda$ is lower semi-continuous for the $\sigma$-weak topology; i.e. $\{a \in M^+: \ \lambda(a) > t\}$ is open in $M^+$ for the $\sigma$-weak topology and for all $t \in \mathbb R$.
\end{itemize}
There is a set $\mathcal F$ of positive normal functionals on $M$ such that
$$
\lambda(a) = \sup_{\omega \in \mathcal F} \omega(a)
$$ 
for all $a \in M^+$.
\end{thm}

As alluded to in Notes and remarks \ref{16-12-21} Haagerups approach also provided another proof of Combes' theorem, but more importantly it led to a proof of the following theorem.

\begin{thm}\label{06-04-22a} (Haagerup, \cite{Ha}) Let $M$ be a von Neumann algebra and $\phi : M^+ \to [0,\infty]$ a map such that
\begin{itemize}
\item $\phi(a+b) = \phi(a) + \phi(b) \ \ \forall a,b \in M_+$,
\item $\phi( t a) = t \phi(a) \ \ \forall a \in M^+, \ \forall t \in \mathbb R^+$, using the convention $0 \cdot \infty = 0$.
\end{itemize}
The following are equivalent.
\begin{itemize}
\item[(a)] There is a set $I$ of positive normal functionals on $M$ such that $\phi(a) = \sup_{\omega \in I} \omega(a)$ for all $a \in M^+$. 
\item[(b)] There is a set $J$ of positive normal functionals on $M$ such that $\phi(a) = \sum_{\omega \in J} \omega(a)$ for all $a \in M^+$.  
\item[(c)] $\phi$ is lower semi-continuous with respect to the $\sigma$-weak topology.
\end{itemize}
\end{thm}

In fact, Haagerups theorem contains two additional equivalent conditions, but we shall not need them here. We do not give a proof of the two theorems above; instead we use them to prove that the normal faithful semi-finite weight $\psi''$  on $\pi_\psi(A)''$ is the unique $\sigma''$-invariant normal extension of $\psi$.

\begin{thm}\label{05-04-22} Let $\psi$ be a $\beta$-KMS weight for the flow $\sigma$ on $A$, and let $\phi$ be a normal $\sigma''$-invariant weight on $\pi_\psi(A)''$ such that $\phi \circ \pi_\psi|_{A^+} = \psi$. Then $\phi = \psi''$.
\end{thm}

For the proof of Theorem \ref{05-04-22} we need some preparations, the first of which is the following fact which was implicitly used in the proof of Proposition \ref{08-02-22a}.

\begin{lemma}\label{07-04-22e} Let $H$ be a Hilbert space and $C \subseteq H$ a convex subset. The norm closure of $C$ in $H$ is the same as the weak*-closure of $C$ in $H^* = H$.
\end{lemma}
\begin{proof} Since $H^* = H$ the weak* closure of $C$ is the same as the weak closure of $C$. Since $C$ is convex it follows from the Hahn-Banach separation theorem that the weak closure is the same as the norm closure.
\end{proof}

\begin{lemma}\label{09-04-22} Let $\phi$ be a normal weight on the von Neumann algebra $M$ and $(H_\phi,\Lambda_\phi,\pi_\phi)$ its GNS triple. Then $\Lambda_\phi : \mathcal N_\psi \to H_\phi$ is closed with respect to the $\sigma$-weak topology of $M$.
\end{lemma}
\begin{proof} Since $\phi$ has the property (a) of Theorem \ref{06-04-22a} the proof of Lemma \ref{22-02-22fx} can be adapted in the obvious way.
\end{proof}

\begin{lemma}\label{07-04-22d} Let $\phi$ be a normal weight on the von Neumann algebra $M$. Let $\{a_i\}_{i \in I}$ be a net in $\mathcal N_\phi$ and $m \in M$ such that $\lim_{i \to \infty} a_i = m$ in the $\sigma$-weak topology. Assume that
\begin{equation}\label{09-04-22a}
\sup_{i \in I} \left\|\Lambda_\phi(a_i)\right\| < \infty .
\end{equation}
It follows that $m \in \mathcal N_\phi$ and there is a net $\{b_j\}_{j \in J}$ in the convex hull
$$
\co \left\{ a_i: \ i \in I\right\}
$$
such that $\lim_{j \to \infty} b_j = m$ in the $\sigma$-weak topology and $\lim_{j \to \infty} \Lambda_\phi(b_j) = \Lambda_\phi(m)$ in $H_\phi$.
\end{lemma}
\begin{proof}  The set of pairs $(F,\epsilon)$ where $F \subseteq M_*$ is a finite set and $0 < \epsilon \leq 1$ constitute a directed set $\mathcal T$ where $(F,\epsilon) \leq (F',\epsilon')$ means that $F \subseteq F'$ and $\epsilon' \leq \epsilon$. Since closed balls are weak* compact in $H_\phi$ it follows from \eqref{09-04-22a} that the weak* closure $\overline{\co}\left\{ \Lambda_\phi(a_r) :  \ r \geq j \right\}$ of the convex hull  
$$
{\co}\left\{ \Lambda_\phi(a_r) :  \ r \geq j\right\}
$$ 
is compact in the weak* topology for all $j \in I$. Since $I$ is directed, 
$$
\cap_{ j \in F} \overline{\co}\left\{ \Lambda_\phi(a_r) :  \ r \geq j \right\} 
$$
is non-empty for all finite subsets $F \subseteq J$ and hence
$$
\bigcap_{j \in J} \overline{\co}\left\{ \Lambda_\phi(a_r) :  \ r \geq j \right\}  \neq \emptyset .
$$
Let $\eta$ be an element in this intersection.  Let $(L,\delta) \in \mathcal T$.
Since $
\lim_{j \to \infty} a_j= m$ in the $\sigma$-weak topology, there is a $j_0 \in I$ such that
$$
\left|\omega(a_j) - \omega(m)\right| \leq \delta
$$
for all $\omega \in L$ when $j \geq j_0$. By Lemma \ref{07-04-22e} $\overline{\co}\left\{ \Lambda_\phi(a_j) :  \ j \geq j_0\right\}$ is also the norm closure of ${\co}\left\{ \Lambda_\phi(a_j) :  \ j \geq j_0\right\}$. Since
$$
\eta \in \overline{\co}\left\{ \Lambda_\phi(a_j) :  \ j \geq j_0\right\},
$$
we can therefore find an element $b(L,\delta)$ in the convex hull of $\left\{ a_j : \ j \geq j_0\right\}$ such that
$$
\left\|\Lambda_\phi(b(L,\delta)) - \eta \right\| \leq \delta .
$$
Since
$$
\left|\omega(b(L,\delta)) - \omega(m)\right| \leq \delta
$$
for all $\omega \in L$, it follows that $\lim_{(L,\delta) \to \infty} b(L,\delta) = m$ in the $\sigma$-weak topology and $\lim_{(L,\delta) \to \infty} \Lambda_\phi(b(L,\delta)) = \eta$ in $H_\phi$. It follows from Lemma \ref{09-04-22} that $m \in \mathcal N_\phi$ and $\Lambda_\phi(m) = \eta$.

\end{proof}

\begin{lemma}\label{04-04-22} Let $\psi$ and $\phi$ be as in Theorem \ref{05-04-22}. Let $m \in \mathcal N_{\phi}$. There is a net $\{a_i\}_{i \in I}$ in $\mathcal M^\sigma_\psi$ such that 
\begin{itemize}
\item $\left\|\pi_\psi(a_i)\right\| \leq \|m\|$ for all $i \in I$,
\item $\left\| \Lambda_{\phi}(\pi_\psi(a_i))\right\| \leq \left\|\Lambda_\phi(m)\right\| + 1$ for all $i\in I$, 
\item $\lim_{i \to \infty} \pi_\psi(a_i) = m$ in the strong operator topology, and 
\item $\lim_{i \to \infty} \Lambda_{\phi}(\pi_\psi(a_i)) = \Lambda_{\phi}(m)$ in $H_\phi$.
\end{itemize}
\end{lemma}
\begin{proof} The set of pairs $(F,\epsilon)$ where $F \subseteq H_\psi$ is a finite set and $0 < \epsilon \leq 1$ constitute a directed set $\mathcal S$ where $(F,\epsilon) \leq (F',\epsilon')$ means that $F \subseteq F'$ and $\epsilon' \leq \epsilon$. Let $(F,\epsilon) \in \mathcal S$. It suffices to find $a := a(F,\epsilon) \in \mathcal M^\sigma_\psi$ such that $\left\|\pi_\psi(a)\right\| \leq \|m\|$,
\begin{equation*}\label{05-04-22e}
\left\| \pi_\psi(a)\chi -m\chi\right\| \leq \epsilon
\end{equation*}
for all $\chi \in F$, and
\begin{equation*}\label{05-04-22f}
\left\|\Lambda_\phi(\pi_\psi(a)) - \Lambda_\phi(m)\right\| \leq \epsilon ,
\end{equation*}
because then $\{a(F,\epsilon)\}_{(F,\epsilon) \in \mathcal S}$ is a net with the stated properties. 

Since $\phi$ is $\sigma''$-invariant there is a unitary representation $U^\phi$ of $\mathbb R$ by unitaries on $H_\phi$ such that
$$
U^\phi_t \Lambda_\phi(a) = \Lambda_\phi(\sigma''_t(a))
$$
for all $a \in \mathcal N_\phi$ and all $t \in \mathbb R$. Thanks to Haagerups theorem, Theorem \ref{05-04-22a}, the arguments from the proof of Lemma \ref{17-11-21e} apply to show that $\mathbb R \ni t \mapsto U^\phi_t$ is continuous with respect to the strong operator topology. Thanks to this property and Lemma \ref{09-04-22} the proof of Lemma \ref{02-03-22d} shows that
$$
\Lambda_\phi(\overline{R}_k(m)) = \sqrt{\frac{k}{\pi}} \int_\mathbb R e^{-kt^2} U^\phi_t \Lambda_\phi(m) \ \mathrm d t 
$$
for all $k\in \mathbb N$, and combined with Lemma \ref{24-11-21} it follows that
$$
\lim_{k \to \infty} \Lambda_\phi(\overline{R}_k(m)) = \Lambda_\phi(m) .
$$
In addition,
$$
\overline{R}_k(m)\chi = \sqrt{\frac{k}{\pi}} \int_\mathbb R e^{-k t^2} U^\phi_t m U^\phi_{-t}\chi \ \mathrm d t ,
$$
for all $\chi \in H_\psi$, and the arguments from the proof of Lemma \ref{24-11-21} show that 
$$
\lim_{k \to \infty} \overline{R}_k(m) = m
$$
in the strong operator topology. Set $m_k := \overline{R}_k(m)$. It follows that there is an $N \in \mathbb N$ such that
$$
\left\|m_N\chi -m\chi\right\| \leq \frac{\epsilon}{3}
$$
for all $\chi \in F$ and
$$
\left\|\Lambda_\phi(m_N) - \Lambda_\phi(m)\right\| \leq \frac{\epsilon}{3} .
$$
 Since $\pi_\psi(\mathcal M^\sigma_\psi)$ is $\sigma$-weakly dense in $\pi_\psi(A)''$ it follows from Kaplanskys density theorem that there is a net $\{d_k\}_{k \in I}$ in $\mathcal M^\sigma_\psi$ such that $\lim_{k \to \infty} \pi_\psi(d_k) = 1$ in the strong operator topology and $\left\|\pi_\psi(d_k) \right\| \leq 1$ for all $k \in I$. It follows from Haagerups theorem, Theorem \ref{06-04-22a}, that $\phi$ is normal in the sense of \cite{KR} and hence that $\pi_\phi$ is a normal representation, cf. e.g. page 489 in \cite{KR}. Thus
\begin{align*}
& \lim_{k\to \infty} \Lambda_\phi(\pi_\psi(d_k) m_N) =  \lim_{k\to \infty} \pi_\phi( \pi_\psi(d_k))  \Lambda_\phi( m_N) =  \Lambda_\phi( m_N). 
\end{align*}
It follows that there is a $K\in I$ such that
$$
\left\|\pi_\psi(d_K) m_N\chi - m_N \chi \right\| \leq \frac{\epsilon}{3}
$$
for all $\chi \in F$ and
$$
\left\|\Lambda_\phi(\pi_\psi(d_K) m_N) - \Lambda_\phi( m_N)\right\| \leq \frac{\epsilon}{3} .
$$
By using Kaplanskys density theorem again we find a net $\{c_j\}_{j \in J}$ in $\mathcal M^\sigma_\psi$ such that $\left\|\pi_\psi(c_j)\right\| \leq \left\|m\right\|$ for all $j$ and $\lim_{ j \to \infty} \pi_\psi(c_j) = m$ in the strong operator topology. Set
$$
b_j := R_N(c_j). 
$$
Then $b_j \in \mathcal M^\sigma_\psi$ by Lemma \ref{07-12-21}. Note that 
\begin{align*}
& \lim_{j \to \infty} \pi_\psi(b_j) =  \lim_{j \to \infty} \sqrt{\frac{N}{\pi}} \int_\mathbb R e^{-Nt^2} \pi_\psi(\sigma_t(c_j)) \ \mathrm d t \\
& =  \lim_{j \to \infty} \sqrt{\frac{N}{\pi}} \int_\mathbb R e^{-Nt^2} \sigma''_t(\pi_\psi(c_j)) \ \mathrm d t = m_N
\end{align*} 
in the strong operator topology by Lemma \ref{02-03-22hx}. There is therefore a $j_0 \in J$ such that
$$
\left\|\pi_\psi(d_K) \pi_\psi(b_j)\chi - \pi_\psi(d_K)m_N \chi \right\| \leq \frac{\epsilon}{3}
$$
for all $\chi \in F$ when $j \geq j_0$.

 Note also that
\begin{align*}
&\left\|\pi_\psi\left(\sigma_{-i\frac{\beta}{2}}(b_j)\right)\right\| = \left\| \sqrt{\frac{N}{\pi}} \int_\mathbb R e^{-N(t + i \frac{\beta}{2})^2} \pi_\psi(\sigma_t(c_j)) \ \mathrm d t \right\| \\
&\leq  K_N \left\| \pi_\psi(c_j)\right\|  \leq   K_N \|m\|,
\end{align*}
where $K_N =  \sqrt{\frac{N}{\pi}} \int_\mathbb R \left|e^{-N(t + i \frac{\beta}{2})^2}\right|  \ \mathrm d t$. By using $\phi \circ \pi_\psi = \psi$ and that $\psi$ is a $\beta$-KMS weight for $\sigma$, this implies that
\begin{equation*}
\begin{split}
& \left\|\Lambda_\phi(\pi_\psi(d_K)\pi_\psi(b_j))\right\|^2 = \phi(\pi_\psi(b_j^*d_K^*d_Kb_j)) \\
& = \psi( b_j^*d_K^*d_Kb_j) = \psi(\sigma_{-i \frac{\beta}{2}}(d_K)\sigma_{-i \frac{\beta}{2}}(b_j) \sigma_{-i \frac{\beta}{2}}(b_j)^* \sigma_{-i \frac{\beta}{2}}(d_K)^*) \\
&= \left\|\Lambda_\psi(  \sigma_{-i \frac{\beta}{2}}(b_j)^* \sigma_{-i \frac{\beta}{2}}(d_K)^*)\right\|^2 \\
& \leq \left\|\pi_\psi( \sigma_{-i \frac{\beta}{2}}(b_j)^*)\right\|^2 \left\|\Lambda_\psi(\sigma_{-i \frac{\beta}{2}}(d_K)^*)\right\|^2 \leq  K_N \|m\| \left\|\Lambda_\psi(\sigma_{-i \frac{\beta}{2}}(d_K)^*)\right\|^2 
\end{split}
\end{equation*}
for all $j$. Thanks to this estimate and because $\lim_{j \to \infty} \pi_\psi(d_K)\pi_\psi(b_j) =\pi_\psi(d_K)m_N$ in the $\sigma$-weak topology it follows from Lemma \ref{07-04-22d} that there is an element $b \in \co \{b_j : \ j \geq j_0\}$ such that
$$
\left\|\Lambda_\phi(\pi_\psi(d_K)\pi_\psi(b))  - \Lambda_\phi(\pi_\psi(d_K)m_N) \right\| \leq \frac{\epsilon}{3} .
$$ 
Then $a := d_Kb \in \mathcal M^\sigma_\psi$ has the desired properties.

\end{proof}

\emph{Proof of Theorem \ref{05-04-22}:} Let $\phi_i, \ i =1,2$, be normal $\sigma''$-invariant weights on $\pi_\psi(A)''$ such that $\phi_i \circ \pi_\psi|_{A^+} = \psi, \ i =1,2$. By Lemma \ref{01-03-22f}, Lemma \ref{01-03-22g} and Lemma \ref{03-03-22fx} $\psi''$ has these properties so it suffices here to show that $\phi_1 = \phi_2$. Let $m \in \mathcal N_{\phi_1}$. By Lemma \ref{04-04-22} there is a net $\{a_i\}_{i \in I}$ in $\mathcal M^\sigma_\psi$ such that
\begin{itemize}
\item[(i)] $\left\|\pi_\psi(a_i)\right\| \leq \|m\|$ for all $i \in I$,
\item[(ii)] $\left\| \Lambda_{\phi_1}(\pi_\psi(a_i))\right\| \leq \left\|\Lambda_{\phi_1}(m)\right\| + 1$ for all $i\in I$, 
\item[(iii)] $\lim_{i \to \infty} \pi_\psi(a_i) = m$ in the strong operator topology, and 
\item[(iv)] $\lim_{i \to \infty} \Lambda_{\phi_1}(\pi_{\psi}(a_i)) = \Lambda_{\phi_1}(m)$ in $H_\phi$.
\end{itemize}
Let $(F,\epsilon) \in \mathcal S$, where $\mathcal S$ is the directed set from the proof of Lemma \ref{04-04-22}. It follows from (iii) and (iv) that there is an $l \in I$ such that 
$$
\left\|\pi_\psi(a_j)\chi - m\chi\right\| \leq \epsilon
$$
for all $\chi \in F$ and
$$
\left\|\Lambda_{\phi_1}(\pi_\psi(a_j)) - \Lambda_{\phi_1}(m)\right\| \leq \epsilon
$$
when $j \geq l$. It follows from (ii) that
$$
\left\|\Lambda_{\phi_2}(\pi_\psi(a_i))\right\| = \sqrt{\psi(a_i^*a_i)} = \left\|\Lambda_{\phi_1}(\pi_\psi(a_i))\right\| \leq \left\|\Lambda_{\phi_1}(m)\right\| +1
$$
for all $i$. Since $\phi_i \circ\pi_\psi = \psi, \ i =1,2$, it follows therefore from Lemma \ref{07-04-22d} that $m \in \mathcal N_{\phi_2}$ and that there is an element 
$$
b(F,\epsilon) \in \co \left\{ a_j : \ j \geq l \right\}
$$
such that 
$$
\left\|\Lambda_{\phi_2}(b(F,\epsilon)) - \Lambda_{\phi_2}(m)\right\| \leq \epsilon .
$$
The net $\{b(F,\epsilon)\}_{(F,\epsilon) \in \mathcal S}$ has the properties that $\lim_{(F,\epsilon) \to \infty} \Lambda_{\phi_i}(\pi_\psi(b(F,\epsilon))) = \Lambda_{\phi_i}(m)$ in $H_{\phi_i}, \ i =1,2$, and hence
\begin{align*}
& {\phi_1}(m^*m) = \left\|\Lambda_{\phi_1}(m)\right\|^2 = \lim_{(F,\epsilon) \to \infty}  \left\|\Lambda_{\phi_1}\left(b(F,\epsilon)\right)\right\|^2 \\
&=  \lim_{(F,\epsilon) \to \infty} \phi_1\left( \pi_\psi(b(F,\epsilon)^*)\pi_\psi(b(F,\epsilon))\right)  = \lim_{(F,\epsilon) \to \infty} \psi\left(b(F,\epsilon)^*b(F,\epsilon)\right)\\
& = \lim_{(F,\epsilon) \to \infty} \phi_2\left( \pi_\psi(b(F,\epsilon))^*\pi_\psi(b(F,\epsilon))\right)  = \phi_2(m^*m) .
\end{align*}
It follows, in particular, that $m \in \mathcal N_{\phi_2}$, implying that $\mathcal N_{\phi_1} \subseteq \mathcal N_{\phi_2}$.
Since $\mathcal N_{\phi_2} \subseteq \mathcal N_{\phi_1}$ by symmetry, we conclude that $\phi_1 = \phi_2$.
\qed


\begin{notes}\label{07-04-22} Most the material in this chapter is based on the work of Kustermans and Vaes in \cite{Ku1} and \cite{KV1}. In particular, they construct by a different method a normal extension of $\psi$ to $\pi_\psi(A)''$ in \cite{KV1}. It follows from Theorem \ref{05-04-22} that the normal extension $\psi''$ of $\psi$ we have constructed here is the same as the extension constructed by Kustermans and Vaes; a fact which also follows by comparing \eqref{06-02-23a} to Definition 2.10 in \cite{KV1}.

\end{notes}

\section{A faithful KMS weight remembers the flow}

By using the normal extension of a KMS weight we deduce in this section a few consequences of the uniqueness of the modular flow of a faithful normal semi-finite weight on a von Neumann algebra.

\bigskip

\begin{lemma}\label{07-04-22a} Let $\sigma^1$ and $\sigma^2$ be flows on $A$. Let $\beta_i \in \mathbb R, i = 1,2$. Assume that $\psi$ is a $\beta_1$-KMS weight for $\sigma^1$ and a $\beta_2$-KMS weight for $\sigma^2$, and that $\pi_\psi$ is faithful. Then
$$
\sigma^1_{\beta_1 t} = \sigma^2_{\beta_2 t}
$$
for all $t \in \mathbb R$.
\end{lemma}
\begin{proof} Let $\psi_i, i = 1,2$, be the normal extensions of $\psi$ to $\pi_\psi(A)''$ obtained from $\sigma^i, i = 1,2$. Then $\psi_1 \circ \pi_\psi = \psi = \psi_2 \circ \pi_\psi$ on $A^+$. The key point is to check that arguments from the proof of Theorem \ref{05-04-22} work to show that $\psi_1 = \psi_2$ despite that the flows $\sigma^1$ and $\sigma^2$ used to define $\psi_1$ and $\psi_2$ are not a priori related. We leave this to the reader. Once the equality $\psi_1 = \psi_2$ is established the proof is quickly completed: By uniqueness of the modular automorphism group associated to $\psi_1$ it follows from Theorem \ref{05-03-22x} that ${\sigma^1}''_{-\beta_1 t} = {\sigma^2}''_{-\beta_2 t}$ for all $t \in \mathbb R$. Since
\begin{align*}
& \pi_\psi(\sigma^1_{-\beta_1 t}(a)) = {\sigma^1}''_{-\beta_1 t}(\pi_\psi(a)) = {\sigma^2}''_{-\beta_1 t}(\pi_\psi(a)) = \pi_\psi(\sigma^2_{-\beta_1 t}(a))
\end{align*}
it follows that ${\sigma^1}_{-\beta_1 t} = {\sigma^2}_{-\beta_2 t}$ for all $t \in \mathbb R$ because $\pi_\psi$ is faithful by assumption.
\end{proof}

A weight $\psi$ on $A$ is \emph{faithful} when $a \in A \backslash \{0\} \Rightarrow \psi(a^*a) > 0$.

\begin{thm}\label{09-04-22c} Let $\beta, \beta' \in \mathbb R$ and let $\psi$ be a $\beta$-KMS weight for the flow $\sigma$ on $A$. Assume that $\psi$ is faithful. If $\sigma^1$ is a flow such that $\psi$ is a $\beta'$-KMS weight for $\sigma^1$, then $\sigma^1_{\beta' t} = \sigma_{\beta t}$ for all $t \in \mathbb R$. 
\end{thm}
\begin{proof} If $\pi_\psi(a^*a) = 0$, consider a sequence $\{u_n\}_{n=1}^\infty$ in $\mathcal N_\psi$ such that 
$$
\lim_{n \to \infty} au_n = a,
$$ 
cf. Lemma \ref{16-12-21a}. Then
\begin{align*}
& \psi(a^*a) \leq \liminf_{n \to \infty} \psi(u_n^*a^*au_n) =  \liminf_{n \to \infty}\left< \Lambda_\psi(au_n),\Lambda_\psi(au_n)\right> \\
&= \liminf_{n \to \infty}\left<\pi_{\psi}(a^*a) \Lambda_\psi(u_n),\Lambda_\psi(u_n)\right> = 0 .
\end{align*}
It follows that $a =0$ since $\psi$ is faithful by assumption, proving that $\pi_\psi$ is faithful. Hence Lemma \ref{07-04-22a} applies.
\end{proof}



\begin{cor}\label{01-05-22d}  Let $\beta, \beta' \in \mathbb R$ and let $\psi$ be a $\beta$-KMS weight for the flow $\sigma$ on $A$. Assume that the only $\sigma$-invariant ideals in $A$ are $\{0\}$ and $A$. If $\sigma^1$ is a flow such that $\psi$ is a $\beta'$-KMS weight for $\sigma^1$, then $\sigma^1_{\beta' t}= \sigma_{\beta t}$ for all $t \in \mathbb R$.
\end{cor}
\begin{proof} Under the given assumption it follows from Lemma \ref{05-12-21a} that $\psi$ is faithful. Hence Theorem \ref{09-04-22c} applies.
\end{proof}

\begin{cor}\label{31-12-22}  Assume that the only $\sigma$-invariant ideals in $A$ are $\{0\}$ and $A$. Let $\beta,\beta' \in \mathbb R, \ \beta \neq \beta'$. If $\psi$ is a weight on $A$ which is both a $\beta$-KMS weight and a $\beta'$-KMS weight for $\sigma$, then $\sigma_t = \id_A$ for all $t \in \mathbb R$. 
\end{cor}
\begin{proof} Take $\sigma^1 = \sigma$ in Corollary \ref{01-05-22d}.
\end{proof}

\begin{notes} Theorem \ref{09-04-22c} was obtained by Kustermans in Proposition 6.34 of \cite{Ku1}.
\end{notes}



\chapter{The set of $\beta$-KMS weights}

In this chapter we use the modular theory of KMS weights from the last chapter to show that the set of $\beta$-KMS weights for a given flow has a natural structure as a convex cone and a natural partial ordering which turns it into a lattice.

\section{Addition of KMS weights}
 Let $\psi$ and $\varphi$ be weights on the $C^*$-algebra $A$. The sum $\psi + \varphi$ is then defined by
$$
(\psi + \varphi)(a) := \psi(a) + \varphi(a), \ \ a \in A^+ .
$$

\begin{lemma}\label{14-02-22a} $\psi + \varphi$ is a weight.
\end{lemma}
\begin{proof} Only lower semi-continuity requires a proof and this follows from Combes' theorem, Theorem \ref{04-11-21k}. Indeed, it follows from Combes' theorem that 
$$
\psi(a) + \varphi(a) = \sup_{\omega \in \mathcal F_{\psi + \varphi}}\omega(a); 
$$
an equality which implies the lower semi-continuity of $\psi + \varphi$. 
\end{proof}

Let $\sigma$ be a flow on $A$.

\begin{prop}\label{14-02-22b} Let $\psi$ be a $\beta$-KMS weight for $\sigma$ and $\varphi$ a $\beta'$-KMS weight for $\sigma$. Then $\psi + \varphi$ is a densely defined $\sigma$-invariant weight and a $\beta$-KMS weight when $\beta = \beta'$.
\end{prop}
\begin{proof} In view of Lemma \ref{14-02-22a}, for the first statement the only non-trivial fact is that $\psi + \varphi$ is densely defined. It follows from Lemma \ref{14-02-22} that
$$
\mathcal M_\psi^\sigma \mathcal M_\varphi^\sigma \subseteq \mathcal N_\psi \cap \mathcal N_\varphi .
$$
Since $\mathcal M_\psi^\sigma$ and $\mathcal M_\varphi^\sigma$ both are dense in $A$ by Lemma \ref{07-12-21a} it follows that $\mathcal M_\psi^\sigma \mathcal M_\varphi^\sigma$ is dense in $A$ and hence $\mathcal N_\psi \cap \mathcal N_\varphi$ is dense in $A$. Since $\psi + \varphi$ is finite on $a^*a$ when $a \in \mathcal N_\psi \cap \mathcal N_\varphi$, $\psi + \varphi$ is densely defined. That $\psi + \varphi$ is a $\beta$-KMS weight when $\beta = \beta'$ follows immediately from Kustermans' theorem, Theorem \ref{24-11-21d}.
\end{proof}

\section{The $\beta$-KMS weights dominated by another}

Let $\psi$ and $\phi$ be $\beta$-KMS weights for $\sigma$. We write $\phi \leq \psi$ when $\phi(a) \leq \psi(a)$ for all $a \in A^+$.

\begin{lemma}\label{15-02-22dx} Let $\psi$ be a $\beta$-KMS weight for $\sigma$. For every $\beta$-KMS weight $\phi$ for $\sigma$ such that $\phi \leq \psi$ there is an operator $0 \leq T_\varphi \leq 1$ in $\pi_\psi(A)' \cap \pi_\psi(A)''$ which commutes with $U^\psi_t$ for all $t \in \mathbb R$ and satisfies that  
$$
\phi(b^*a) = \left< T_\phi\Lambda_\psi(a),\Lambda_\psi(b)\right>
$$
for all $a,b \in \mathcal N_\psi$.
\end{lemma}
\begin{proof} Since $\phi \leq \psi$ it follows that $\mathcal N_\psi \subseteq \mathcal N_\phi$ and $\ker \Lambda_\psi \subseteq \ker \Lambda_\phi$. We can therefore define a sequilinear form $\left< \ \cdot \ , \ \cdot \ \right>_\phi$ on $H_\psi$ such that
$$
\left< \Lambda_\psi(a)  , \Lambda_\psi(b)\right>_\phi := \phi(b^*a) = \left< \Lambda_\phi(a), \Lambda_\phi(b)\right> 
$$
for all $a,b \in \mathcal N_\psi$. As in the proof of Lemma \ref{08-11-21bx} this gives us an operator $T_\phi \in \pi_\psi(A)'$ such that $0 \leq T_\phi \leq 1$ and 
$$
\left< T_\phi \Lambda_\psi(a),\Lambda_\psi(b)\right> = \phi(b^*a)
$$
for all $a,b \in \mathcal N_\psi$. Let $a,b \in \mathcal M_\psi^\sigma \subseteq \mathcal M_\phi^\sigma$. By using that $\psi$ and $\phi$ are $\beta$-KMS weight for $\sigma$ we find
\begin{equation}\label{30-11-23}
\begin{split}
& \left< T_\phi J_\psi\Lambda_\psi(a),\Lambda_\psi(b)\right> = \left< T_\phi \Lambda_\psi(\sigma_{-i \frac{\beta}{2}}(a)^*),\Lambda_\psi(b)\right> \\
& = \phi\left(b^* \sigma_{i\frac{\beta}{2}}(a^*) \right) = 
\phi\left(  \sigma_{i\frac{\beta}{2}}(a^*)\sigma_{i\beta}(b^*)\right) \\
& = \phi\left( \sigma_{i\frac{\beta}{2}}\left( a^* \sigma_{i \frac{\beta}{2}}(b^*)\right)\right) = \phi\left( a^* \sigma_{i \frac{\beta}{2}}(b^*)\right) ,
\end{split}
\end{equation}
where the third equality uses (3) from Kustermans' theorem, Theorem \ref{24-11-21d}, and last equality follows from Lemma \ref{18-11-21kx}.  To calculate $\left< T_\phi J_\psi\Lambda_\psi(a),\Lambda_\psi(b)\right>$ a different way we use that $J_\psi^* = J_\psi$; a fact which follows from Lemma \ref{04-03-22} since $J^* = J$, or it can be deduced more directly as follows by remembering how to define the adjoint of a conjugate linear map: Let $a, b \in \mathcal M^\sigma_\psi$. Then
\begin{align*}
& \left<J_\psi^* \Lambda_\psi(a),\Lambda_\psi(b)\right> = \left<J_\psi \Lambda_\psi(b),\Lambda_\psi(a)\right>  =\psi(a^*\sigma_{-i \frac{\beta}{2}}(b)^*) \\
& = \psi(\sigma_{i\frac{\beta}{2}}(b^*)\sigma_{i\beta}(a^*)) \ \ \ \ \ \ \ \ \ \ \ \ \ \ \ \text{(using (3) of Theorem \ref{12-12-13})}\\
& =\psi \left(\sigma_{i\frac{\beta}{2}}(b^*\sigma_{i\frac{\beta}{2}}(a^*))\right) = \psi \left(b^*\sigma_{-i\frac{\beta}{2}}(a)^*\right) = \left< J_\psi \Lambda_\psi(a),\Lambda_\psi(b)\right> .
\end{align*}
Since $\Lambda_\psi(\mathcal M^\sigma_\psi)$ is dense in $H_\psi$  by Lemma \ref{07-12-21c}, it follows that $J_\psi^* = J_\psi$ as asserted. By using this we find
\begin{equation}\label{30-11-23a}
\begin{split}
&\left<J_\psi T_\phi \Lambda_\psi(a),\Lambda_\psi(b)\right> = \left<J_\psi^* \Lambda_\psi(b),  T_\psi \Lambda_\psi(a)\right> = \left<J_\psi \Lambda_\psi(b),  T_\phi \Lambda_\psi(a)\right> \\
&= \left< \Lambda_\psi\left( \sigma_{-i \frac{\beta}{2}}(b)^*\right), T_\phi \Lambda_\psi(a)\right> = \left<T_\phi \Lambda_\psi\left( \sigma_{-i \frac{\beta}{2}}(b)^*\right),\Lambda_\psi(a)\right>\\
& = \phi\left(a^*  \sigma_{-i \frac{\beta}{2}}(b)^*\right) = \phi\left(a^*  \sigma_{i \frac{\beta}{2}}(b^*)\right) .
\end{split}
\end{equation}
By comparing the two expressions for $\phi\left(a^*  \sigma_{i \frac{\beta}{2}}(b^*)\right)$ in \eqref{30-11-23} and \eqref{30-11-23a} it follows that $T_\phi$ commutes with $J_\psi$. By using that $J_\psi^2 = J_\psi$ we deduce from Theorem \ref{04-03-22a} that $T_\phi \in \pi_\psi(A)' \cap \pi_{\psi}(A)''$. Since 
\begin{align*}
& \left< T_\phi U^\psi_t\Lambda_\psi(a),\Lambda_\psi(b)\right>  = \phi(b^*\sigma_t(a)) = \phi(\sigma_{-t}(b)^*a) \\
& = \left<T_\phi \Lambda_\psi(a), U^\psi_{-t}\Lambda_\psi(b)\right>  = \left< U_t^\psi T_\phi\Lambda_\psi(a),\Lambda_\varphi(b)\right> 
\end{align*}
for $a,b \in \mathcal M^\sigma_\psi$, we see that $T_\phi$ commutes with $U^\psi_t$.
\end{proof}
Set 
$$
\mathcal Z_\psi := \pi_\psi(A)' \cap \pi_\psi(A)'';
$$
the center of $\pi_\varphi(A)''$, and
$$
\mathcal Z_\psi^{\sigma''} := \left\{ m \in \mathcal Z_\psi : \ \sigma''_t(m) = m \ \forall t \in \mathbb R\right\} .
$$

\begin{lemma}\label{31-03-22} 
$\mathcal Z_\psi^{\sigma''} = \mathcal Z_\psi $ when $\beta \neq 0$.
 \end{lemma}
\begin{proof} This follows from Theorem \ref{05-03-22x} because the modular automorphism group associated to a faithful normal semi-finite weight on a von Neumann algebra acts trivially on the center of the algebra. 
\end{proof}

 For every element $c \in \mathcal Z_\psi^{\sigma''} \backslash \{0\}$, $0 \leq c \leq 1$, we can define a weight $\psi_c: A^+ \to [0,\infty]$ such that
\begin{equation}\label{01-04-22}
\psi_c(a) := \psi''(c\pi_\psi(a)) .
\end{equation}

\begin{lemma}\label{06-03-22} $\psi_c$ is a $\beta$-KMS weight for $\sigma$ such that $\psi_c \leq \psi$. 
\end{lemma}
\begin{proof} To see that $\psi_c$ is a weight the only non-trivial fact is that $\psi_c$ is lower semi-continuous which follows easily from the fact that $\psi''$ is lower semi-continuous with respect to the $\sigma$-weak topology and hence also with respect to the norm-topology. To see that $\psi_c$ is non-zero note that $\psi''(c) > 0$ since $\psi''$ is faithful by Lemma \ref{01-03-22f}. Since $\pi_\psi(A)$ is $\sigma$-weakly dense in $N$ there is a net $\{a_i\}$ in $A^+$ such that $\lim_{i \to \infty} \pi_\psi(a_i) = 1$ in the $\sigma$-weak topology. Since $\{x \in N^+ : \ \psi''(x) > \frac{\psi''(c)}{2} \}$ is open in the $\sigma$-weak topology, there is an $i$ such that $\psi''(c\pi_\psi(a_i)) > \frac{\psi''(c)}{2}$. Then $\psi_c(a_i) > 0$ , showing that $\psi_c$ is not zero. That $\psi_c$ is $\sigma$-invariant follows from Lemma \ref{01-03-22g}:
\begin{align*}
&\psi_c(\sigma_t(a)) = \psi''(c\pi_\psi(\sigma_t(a))) = \psi''( c \sigma''_t(\pi_\psi(a))) = \psi''\left( \sigma''_t( \sigma''_{-t}(c)\pi_\psi(a))\right) \\
&= \psi''\left( \sigma''_{-t}(c)\pi_\psi(a)\right)  = \psi''(c \pi_\psi(a)) = \psi_c(a).
\end{align*}
It remains to show that $\psi_c$ satisfies one of the four equivalent conditions of Theorem \ref{24-11-21d}. Let $a,b \in \mathcal N_{\psi_c} \cap \mathcal N_{\psi_c}^*$. Then $\sqrt{c}\pi_\psi(a),\sqrt{c}\pi_\psi(b) \in \mathcal N_{\psi''} \cap \mathcal N_{\psi''}^*$. By Theorem \ref{21-02-22dx} there is a continuous function $f: {\mathcal D_\beta} \to \mathbb C$ which is holomorphic in the interior $\mathcal D_\beta^0$ of ${\mathcal D_\beta}$ and has the property that
\begin{itemize}
\item $f(t) = \psi''(\sqrt{c} \pi_\psi(b) \sigma''_t(\sqrt{c}\pi_\psi(a))) \ \ \forall t \in \mathbb R$, and
\item $f(t+i\beta) = \psi''(\sigma''_t(\sqrt{c}\pi_\psi(a))\sqrt{c}\pi_\psi(b)) \ \ \forall t \in \mathbb R$.
\end{itemize}
Since $\psi''(\sqrt{c} \pi_\psi(b) \sigma''_t(\sqrt{c}\pi_\psi(a))) = \psi''(c\pi_\psi(b \sigma_t(a))) = \psi_c(b \sigma_t(a))$ and similarly $\psi''(\sigma''_t(\sqrt{c}\pi_\psi(a))\sqrt{c}\pi_\psi(b)) = \psi_c(\sigma_t(a)b)$, it follows that $\psi_c$ satisfies condition (4) of Theorem \ref{24-11-21d} and hence is a $\beta$-KMS weight for $\sigma$. Since $c \leq 1$ it follows from Lemma \ref{03-03-22fx} that $\psi_c(a) = \psi''\left(\pi_\psi(a^\frac{1}{2})c\pi_\psi(a^\frac{1}{2})\right)\leq \psi''(\pi_\psi(a)) = \psi(a)$ for all $a \in A^+$; i.e. $\psi_c \leq \psi$.
\end{proof}

\begin{lemma}\label{31-03-22c} Let $\phi_1, \ \phi_2$ and $\psi$ be $\beta$-KMS weights for $\sigma$. If $\phi_1(a^*a) = \phi_2(a^*a)$ for all $\mathcal N_{\psi}$, then $\phi_1 =\phi_2$.
\end{lemma}
\begin{proof} Let $\{e_j\}_{j \in J}$ be the approximate unit in $\mathcal N_{\psi}$ from Lemma \ref{16-12-21a}. Then $e_j^2 \in \mathcal M_\psi^+$ and
 $$
 f_j := R_1(e_j^2) \in \mathcal M_\psi^\sigma
 $$
 by Lemma \ref{07-12-21}. Let $a \in A$. The $C^*$-algebra generated by $\left\{ \sigma_t(a) : \ t \in \mathbb R\right\}$ is separable and there is therefore a sequence $j_1 \leq j_2 \leq j_3 \leq \cdots$ in $J$ such that $\lim_{n \to \infty} e_{j_n}^2 \sigma_t(a) = \sigma_t(a)$ for all $t \in \mathbb R$ and it follows then from Lemma \ref{28-02-22} in Appendix \ref{integration} that
 $$
 \lim_{n \to \infty} f_{j_n}a = \lim_{n \to \infty} \frac{1}{\sqrt{\pi}}\int_\mathbb R e^{-t^2} \sigma_t(e_{j_n}^2)a \ \mathrm dt = a .
 $$
 Assume then that $a \in \mathcal A_\sigma$. By Lemma \ref{14-02-22},
 $$
 \sigma_{-i\frac{\beta}{2}}(f_ja) = \sigma_{-i\frac{\beta}{2}}(f_j) \sigma_{-i\frac{\beta}{2}}(a) \in \mathcal N_\psi^*
 $$
 and it follows therefore from the assumption that
 $$
 \phi_1(\sigma_{-i\frac{\beta}{2}}(f_ja)\sigma_{-i\frac{\beta}{2}}(f_ja)^*) = \phi_2(\sigma_{-i\frac{\beta}{2}}(f_ja)\sigma_{-i\frac{\beta}{2}}(f_ja)^*)
 $$
 for all $j$. Since $\phi_1$ and $\phi_2$ are both $\beta$-KMS weights for $\sigma$ this implies that
 \begin{equation}\label{31-03-22a}
 \phi_1(a^*f_j^2a) = \phi_2(a^*f_j^2a)
 \end{equation}
 for all $j$. Note that $\left\|f_j^2\right\| \leq 1$ for all $j$ and that $\lim_{n \to \infty}a^*f_{j_n}^2a = a^*a$. It follows therefore from \eqref{31-03-22a} and the lower semi-continuity of $\phi_1$ and $\phi_2$ that $\phi_1(a^*a) = \phi_2(a^*a)$. Since $a \in \mathcal A_\sigma$ was arbitrary it follows from Lemma \ref{01-03-22b} that $\phi_1 = \phi_2$.
\end{proof}

\begin{cor}\label{31-03-22d}Let $\phi_i, i =1,2$, be $\beta$-KMS weights for $\sigma$. If $\phi_1(a^*a) = \phi_2(a^*a)$ for all $a \in \mathcal N_{\phi_1} \cap \mathcal N_{\phi_2}$, then $\phi_1 = \phi_2$.
\end{cor}
\begin{proof} By Proposition \ref{14-02-22b} the sum $\phi_1 + \phi_2$ is a $\beta$-KMS weight.  Since $\mathcal N_{\phi_1+\phi_2} = \mathcal N_{\phi_1} \cap \mathcal N_{\phi_2}$ it follows from Lemma \ref{31-03-22c} that $\phi_1 = \phi_2$.
\end{proof}

\begin{cor}\label{21-06-22e} Let $\phi_i, i =1,2$, be $\beta$-KMS weights for $\sigma$. If $S \subseteq \mathcal N_{\phi_1}$ is a subspace which is a core for $\Lambda_{\phi_1}$ and $\phi_1(a^*a) = \phi_2(a^*a)$ for all $a \in S$, then $\phi_1 = \phi_2$.
\end{cor}
\begin{proof} Let $a \in \mathcal N_{\phi_1}$. Since $S $ is a core for $\Lambda_{\phi_1}$ by assumption there is a sequence $\{s_n\}$ in $S$ such that $\lim_{n \to \infty} s_n = a$ and $\lim_{n \to \infty} \Lambda_{\phi_1}(s_n) = \Lambda_{\phi_1}(a)$. Since $\left\|\Lambda_{\phi_2}(s_n) - \Lambda_{\phi_2}(s_m)\right\| = \left\|\Lambda_{\phi_1}(s_n) - \Lambda_{\phi_1}(s_m)\right\|$ for all $n,m$, it follows from that $\left\{\Lambda_{\phi_2}(s_n)\right\}$ converges and hence that $a \in \mathcal N_{\phi_2}$ since $\Lambda_{\phi_2}$ is closed by Lemma \ref{17-11-21a}. Furthermore, $\phi_1(a^*a) = \phi_2(a^*a)$ and Corollary \ref{31-03-22d} applies.
\end{proof}

\begin{thm}\label{05-03-22} Let $\psi$ be a $\beta$-KMS weight for $\sigma$. The map $c \mapsto \psi_c$ given by \eqref{01-04-22} is a bijection from
$$
\left\{ c \in \mathcal Z_\psi^{\sigma''} \backslash \{0\} : \ 0 \leq c \leq 1\right\}
$$
onto the set of $\beta$-KMS weights $\phi$ for $\sigma$ with the property that $\phi\leq \psi$. Furthermore,
$c \leq c'$ in $\mathcal Z_\psi^{\sigma''}$ if and only if $\psi_c \leq \psi_{c'}$. 
\end{thm}
\begin{proof} Let $\phi \leq \psi$. It follows from Lemma \ref{15-02-22dx} that there is $c \in \mathcal Z_\psi^{\sigma''}$ such that $0 \leq c \leq 1$ and $\phi(a^*a) = \left< c\Lambda_\psi(a),\Lambda_\psi(a)\right> $ for all $a \in \mathcal N_\psi$.
Let $W$ be the unitary from Lemma \ref{23-02-22dx} and $a \in \mathcal N_\psi$. By using \eqref{23-02-22ex} we find that
\begin{align*}
& \phi(a^*a) =\left< c\Lambda_\psi(a),\Lambda_\psi(a)\right>  =  \left< Wc\Lambda_\psi(a),W\Lambda_\psi(a)\right> \\
& = \left< \pi_{\psi''}(c)W\Lambda_\psi(a),W\Lambda_\psi(a)\right> = \left< \pi_{\psi''}(c) \Lambda_{\psi''}(\pi_\psi(a)), \Lambda_{\psi''}(\pi_\psi(a))\right> \\
&= \psi''(c\pi_\psi(a^*a)) = \psi_c(a^*a).
\end{align*}
Thanks to Lemma \ref{31-03-22c} this shows that $\phi = \psi_c$, and we conclude that the map under consideration is surjective. 
 
  Assume $c_i \in \left\{ c \in \mathcal Z_\psi^{\sigma''} \backslash \{0\} : \ 0 \leq c \leq 1\right\}, \ i =1,2$, and $\psi_{c_1} = \psi_{c_2}$. For $a \in \mathcal N_\psi$,
\begin{align*}
&\psi_{c_i}(a^*a) =\psi''(c_i\pi_\psi(a^*a)) = \left<\pi_{\psi''}(c_i) \Lambda_{\psi''}(\pi_\psi(a)), \Lambda_{\psi''}(\pi_\psi(a)) \right> \\
& =  \left<\pi_{\psi''}(c_i) W \Lambda_{\psi}(a), W\Lambda_{\psi}(a) \right>  = \left<Wc_i  \Lambda_{\psi}(a), W\Lambda_{\psi}(a) \right> \\
& = \left<c_i  \Lambda_{\psi}(a), \Lambda_{\psi}(a) \right> , \ \ i = 1,2.
\end{align*}
Hence 
$$ 
\left<c_1  \Lambda_{\psi}(a), \Lambda_{\psi}(a) \right> = \psi_{c_1}(a^*a) =  \psi_{c_2}(a^*a) = \left<c_2  \Lambda_{\psi}(a), \Lambda_{\psi}(a) \right>
$$
for all $a \in \mathcal N_\psi$, implying that $c_1 = c_2$.

Assume that $c \leq c'$ in $\mathcal Z_\psi^{\sigma''}$. Then $c\pi_\psi(a)  \leq c'\pi_\psi(a)$ and hence $\psi_c(a)= \psi''(c\pi_\psi(a))  \leq \psi''(c'\pi_\psi(a)) =\psi_{c'}(a)$ for all $a \in A^+$; i.e., $\psi_c \leq \psi_{c'}$. Conversely, if $\psi_c \leq \psi_{c'}$ the calculation above shows that $\left<c  \Lambda_{\psi}(a), \Lambda_{\psi}(a) \right>  \leq \left<c'  \Lambda_{\psi}(a), \Lambda_{\psi}(a) \right> $ for all $a \in \mathcal N_\psi$, implying that $c \leq c'$.
\end{proof}

Denote by 
$$
\KMS(\sigma,\beta)
$$ 
the set of $\beta$-KMS weights for $\sigma$. 

\begin{defn}\label{01-08-23f}
A $\beta$-KMS weight $\psi$ is \emph{extremal} when the only $\beta$-KMS weights that are dominated by $\psi$ are the scalar multiplies of $\psi$. 
\end{defn}
In symbols $\psi \in \KMS(\sigma,\beta)$ is extremal when 
$$
\phi \in   \KMS(\sigma,\beta) , \ \phi \leq \psi \Rightarrow \phi \in \mathbb R^+ \psi .
$$ 
An immediate consequence of Theorem \ref{05-03-22} and Lemma \ref{31-03-22} is the following

\begin{cor}\label{07-03-22} Let $\sigma$ be a flow and $\beta \in \mathbb R \backslash \{0\}$. A $\beta$-KMS weight for $\sigma$ is extremal if and only if $\pi_\psi(A)''$ is a factor.
\end{cor}

There is also a version of the last corollary for $0$-KMS weights. To formulate it we say that a flow $\alpha$ on a von Neumann algebra $M$ is \emph{ergodic} when the fixed point algebra of $\alpha$ is as small as possible; viz. $M^\alpha = \mathbb C 1$.

\begin{cor}\label{09-03-22b} Let $\sigma$ be a flow on $A$. A $\sigma$-invariant lower semi-continuous trace $\psi$ on $A$ is extremal as a $0$-KMS weight if and only if $\sigma''$ acts ergodically on the center $\mathcal Z_\psi$ of $\pi_\psi(A)''$.
\end{cor}

\begin{notes} Theorem \ref{05-03-22} is new. Corollary \ref{31-03-22d} is a very useful tool, and it was obtained by Kusterman and Vaes in \cite{KV1}. In fact, in their version it is only required that one of the weights is a KMS weight.
\end{notes} 

\section{The lattice of $\beta$-KMS weights}

The set $\KMS(\sigma, \beta) \cup \{0\}$ is a partially ordered set \footnote{See for example Appendix A1 in \cite{Ru2} for the definition of a partially ordered set.} when we define $\psi \leq \phi$ to mean that $\psi(a) \leq \phi(a)$ for all $a \in A^+$. Moreover, it follows from Proposition \ref{14-02-22b} that $\KMS(\sigma,\beta) \cup \{0\}$ has the structure of a convex cone. Specifically, when $\psi ,\varphi \in \KMS(\sigma,\beta) \cup \{0\}$ and $t,s \in \mathbb R^+$, the formula
$$
(t \psi + s \varphi)(a) := t \psi(a) + s\varphi(a), \ \ a \in A^+,
$$
defines an element $t\psi + s\varphi$ of $\KMS(\sigma,\beta) \cup \{0\}$. 

\begin{spec}
Usually a convex cone is a closed subset of an ambient locally convex vector space, but in the case of KMS weights there are no obvious candidate for such a space. This is not to say that it does not exist, and in fact it does, at least when $A$ is separable. See Notes and Remarks \ref{12-06-22b}. Under the additional assumption that the fixed point algebra of the flow $\sigma$ contains an approximate unit for $A$ a direct construction is given in \cite{Ch}.
\end{spec}

\begin{thm}\label{14-02-22d} $(\KMS(\sigma,\beta) \cup \{0\}, \leq )$ is a lattice in the usual sense; every pair $\psi_1,\psi_2 \in \KMS(\sigma,\beta) \cup \{0\}$ has a least upper bound $\psi_1 \vee \psi_2$ and greatest lower bound $\psi_1 \wedge \psi_2$ in $\KMS(\sigma,\beta) \cup \{0\}$.
\end{thm}
\begin{proof} To show that $\psi_1$ and $\psi_2$ have a greatest lower bound we may assume that $\psi_1 \neq 0$ and $\psi_2 \neq 0$; otherwise the greatest lower bound is clearly $0$. Set $\varphi := \psi_1 + \psi_2$, which is a $\beta$-KMS weight for $\sigma$ by Proposition \ref{14-02-22b}. By Theorem \ref{05-03-22} there are elements $c_i \in \mathcal Z_\varphi^{\sigma''}, \ 0 \leq c_i \leq 1$, such that $\psi_i = \varphi_{c_i}, \ i =1,2$. Since $\mathcal Z_\varphi^{\sigma''}$ is an abelian $C^*$-algebra the cone ${\mathcal Z_\varphi^{\sigma''}}^+$ of its positive elements is a lattice. See e.g. Example 4.2.6 in \cite{BR}. Let $c_1 \wedge_\varphi c_2$ be the greatest lower bound for $c_1$ and $c_2$ in ${\mathcal Z_\varphi^{\sigma''}}^+$. Set
$$
\psi_1 \wedge \psi_2 := \varphi_{c_1 \wedge_\varphi c_2} .
$$
It follows from Theorem \ref{05-03-22} that $\psi_1 \wedge \psi_2$ is the greatest lower bound for $\psi_1$ and $\psi_2$ in the set
\begin{equation}\label{07-06-22}
\left\{ \phi \in \KMS(\sigma,\beta) \cup \{0\}: \ \phi \leq \varphi \right\} .
\end{equation}
But then $\psi_1 \wedge \psi_2$ is the greatest lower bound overall; if namely $\mu \leq \psi_i , \ i = 1,2$, it follows that $\mu \leq \varphi$ and hence $\mu \leq \psi_1 \wedge \psi_2$.

 To show that $\psi_1$ and $\psi_2$ have a least upper bound we may assume that $\psi_1$ and $\psi_2$ are not both zero; otherwise the least upper bound is clearly $0$. Let $c_1 \vee_\varphi c_2$ be the least upper bound of $c_1$ and $c_2$ in $\mathcal Z_\varphi^{\sigma''}$ and set
$$
\psi_1 \vee \psi_2 := \varphi_{c_1 \vee_\varphi c_2} .
$$
It follows from Theorem \ref{05-03-22} that $\psi_1 \vee \psi_2$ is the least upper bound for $\psi_1$ and $\psi_2$ in the set \eqref{07-06-22}.
To see that it is in fact the least upper bound overall, let $\mu \in \KMS(\sigma,\beta) \cup \{0\}$ such that $\psi_i \leq \mu, \ i =1,2$. Then $\psi_i \leq \mu \wedge \varphi, \ i =1,2$, and $\mu \wedge \varphi \leq \varphi$. Hence $\psi_1 \vee \psi_2 \leq \mu \wedge \varphi \leq \mu$, showing that $\psi_1 \vee \psi_2$ is the least upper bound for $\psi_1$ and $\psi_2$ in $\KMS(\sigma,\beta) \cup \{0\}$.
\end{proof}

\begin{notes}\label{25-09-23i}
 The main results in this chapter, Theorem \ref{05-03-22} and Theorem \ref{14-02-22d}, are new, but Christensen obtains in Theorem 4.7 of \cite{Ch} a version of Theorem \ref{14-02-22d} when the fixed point algebra of the flow $\sigma$ contains an approximate unit for $A$.
\end{notes}

\section{Passive KMS weights}

Let $\sigma$ be a flow on the $C^*$-algebra $A$. 
Let $I \subseteq A$ be a $\sigma$-invariant closed two-sided ideal in $A$ and let $q_I: A \to A/I$ be the quotient map. Let $\sigma^{A/I}$ be the flow on $A/I$ induced by $\sigma$; viz.
$$
\sigma^{A/I}_t \circ q_I = q_I \circ \sigma_t .
$$

\begin{lemma}\label{01-05-22ax} Let $\psi$ be a $\beta$-KMS weight for $\sigma^{A/I}$. Then $\psi \circ q_I$ is a $\beta$-KMS weight for $\sigma$.
\end{lemma}
\begin{proof} To see that $\psi$ is a weight on $A$ only the lower semi-continuity is not trivial, but follows from Combes' theorem, Theorem \ref{04-11-21k}, which implies that 
$$
\psi \circ q_I = \sup_{\omega \in \mathcal F_\psi} \omega \circ q_I.
$$ 
To see that $\psi \circ q_I$ is a $\beta$-KMS weight, note first of all that $\psi \circ q_I$ is non-zero and densely defined since $\psi$ has these properties. Let $a \in \mathcal A_\sigma$. By definition this means that there is a sequence $\{a_n\}_{n=0}^\infty$ in $A$ such that $\sigma_t(a) = \sum_{n=0}^\infty a_n t^n$ and $\sum_{n=0}^\infty \|a_n\||t|^n < \infty$ for all $t\in \mathbb R$. It follows that $\sigma^{A/I}_t\left(q_I(a)\right) = \sum_{n=0}^\infty q_I(a_n) t^n$ and $\sum_{n=0}^\infty \left\|q_I(a_n)\right\| |t|^n < \infty$ for all $t\in \mathbb R$, and hence that $q_I(a) \in \mathcal A_{\sigma^{A/I}}$ and 
\begin{equation}\label{02-05-22}
\sigma^{A/I}_z \left( q_I(a)\right) = q_I\left(\sigma_z(a)\right)
\end{equation}
for all $z \in \mathbb C$. Since $\psi$ is a $\beta$-KMS weight for $\sigma^{A/I}$ it follows that
\begin{align*}
& \psi\circ q_I\left(a^*a\right) = \psi(q_I(a)^*q_I(a)) \\
& = \psi \left( \sigma^{A/I}_{-i \frac{\beta}{2}} (q_I(a))  \sigma^{A/I}_{-i \frac{\beta}{2}} (q_I(a))^*\right) \\
& = \psi\left( q_I \left(\sigma_{-i \frac{\beta}{2}} (a)\right)q_I \left(\sigma_{-i \frac{\beta}{2}} (a)\right)^*\right) \\
& = \psi \circ q_I \left( \sigma_{-i \frac{\beta}{2}} (a)\sigma_{-i \frac{\beta}{2}} (a)^*\right).
\end{align*}
Thanks to Kustermans' theorem, Theorem \ref{24-11-21d}, this shows that $\psi \circ q_I$ is a $\beta$-KMS weight for $\sigma$.
\end{proof}

 We say that a flow $\sigma$ is trivial when $\sigma_t(a) = a$ for all $t \in \mathbb R$ and all $a \in A$; otherwise $\sigma$ is said to be \emph{non-trivial}. For a trivial flow $\sigma$ the set of $\beta$-KMS weights is the set of all lower semi-continuous traces on $A$ for any $\beta \in \mathbb R$, as one easily deduces from Kustermans theorem, Theorem \ref{24-11-21d}. If follows therefore from Lemma \ref{01-05-22ax} that if, in the setting of Lemma \ref{01-05-22ax}, the flow $\sigma^{A/I}$ is trivial, all lower semi-continuous traces on $A/I$ give rise to $\beta$-KMS weights for $\sigma$ for all $\beta \in \mathbb R$ by composition with the quotient map $q_I$. KMS weights of this sort are relatively easy to understand and we want to be able to push them aside.

 \begin{lemma}\label{02-05-22b} Let $\psi$ be a KMS weight for $\sigma$. The following conditions are equivalent.
\begin{itemize}
\item[(i)] $\psi$ is a $\beta$-KMS weight for $\sigma$ for all $\beta \in \mathbb R$.
\item[(ii)] There are real numbers $\beta \neq \beta'$ such that $\psi$ is both a $\beta$-KMS weight for $\sigma$ and a $\beta'$-KMS weight for $\sigma$.
\item[(iii)] The flow $\sigma^{A/\ker_\psi}$ is trivial and there is a lower semi-continuous trace $\tau$ on $A/\ker_\psi$ such that $\psi = \tau \circ q_{\ker_\psi}$.
\item[(iv)] There is a $\sigma$-invariant ideal $I$ in $A$ such that $\sigma^{A/I}$ is the trivial flow on $A/I$ and a lower semi-continuous trace $\tau$ on $A/I$ such that $\psi = \tau \circ q_I$. 
\end{itemize} 
  \end{lemma}
  \begin{proof} We remind the reader that $\ker_\psi$ is the $\sigma$-invariant ideal in $A$ defined in Lemma \ref{05-12-21a}. The implications (i) $\Rightarrow$ (ii) and (iii) $\Rightarrow$ (iv) are trivial and the implication (iv) $\Rightarrow$ (i) follows from Lemma \ref{01-05-22ax}. It remains therefore only to show that (ii) $\Rightarrow$ (iii). Assume therefore that (ii) holds. By (2) of Theorem \ref{24-11-21d}, $\psi(b^*\sigma_{i\beta}(a)) = \psi(b^*\sigma_{i \beta'}(a))$ for all $a,b \in \mathcal M_\psi^\sigma$, implying that 
$$
\left<\Lambda_\psi(\sigma_{i(\beta - \beta')}(a)), \Lambda_\psi(b) \right> = \left<\Lambda_\psi(a),\Lambda_\psi(b)\right>
$$
for all $a,b \in \mathcal M_\psi^\sigma$. Since $\Lambda_\psi\left(\mathcal M^\sigma_\psi\right)$ is dense in $H_\psi$ by Lemma \ref{07-12-21c} this implies that $\Lambda_\psi(\sigma_{i(\beta - \beta')}(a)) =\Lambda_\psi(a)$, and hence that
$$
\psi\left((\sigma_{i(\beta - \beta')}(a)-a)^*(\sigma_{i(\beta - \beta')}(a)-a)\right) = \left\| \Lambda_\psi(\sigma_{i(\beta - \beta')}(a) -a)\right\|^2 = 0 .
$$
Thus $\sigma_{i(\beta - \beta')}(a) -a \in \ker_\psi$. Using \eqref{02-05-22} we find that
$$
\sigma^{A/{\ker_\psi}}_{it + i(\beta - \beta')}(q_{\ker_\psi}(a)) = \sigma^{A/{\ker_\psi}}_{it }(q_{\ker_\psi}( \sigma_{i(\beta -\beta')}(a))) = \sigma^{A/{\ker_\psi}}_{it }(q_{\ker_\psi}(a)) 
$$ 
for all $t \in \mathbb R$.
Thus the function $\mathbb R \ni t \mapsto \sigma^{A/{\ker_\psi}}_{it}(q_{\ker_\psi}(a))$ is $\beta-\beta'$-periodic and continuous, and hence norm-bounded. Since 
$$
\left\|\sigma^{A/{\ker_\psi}}_{z}(q_{\ker_\psi}(a))\right\| =\left\|\sigma^{A/{\ker_\psi}}_{i \Imag z}(q_{\ker_\psi}(a))\right\|
$$ 
it follows that the function $z \mapsto \sigma^{A/{\ker_\psi}}_{z}(q_{\ker_\psi}(a))$ is norm-bounded. Since it is also entire holomorphic it is therefore constant by Liouville's theorem. This shows that $ \sigma^{A/{\ker_\psi}}_{t}(q_{\ker_\psi}(a))  = q_{\ker_\psi}(a)$ for all $t \in \mathbb R$. The set $\mathcal M^\sigma_\psi$ is dense in $A$ by Lemma \ref{07-12-21a} and consequently $q_{\ker_\psi}(\mathcal M^\sigma_\psi)$ is dense in $A/{\ker_\psi}$. It follows that $\sigma^{A/{\ker_\psi}}$ is the trivial flow. We claim that we can define a lower semi-continuous weight $\tau : (A/{\ker_\psi})^+ \to [0,\infty]$ such that
$$
\tau\left(q_{\ker_\psi}(a)\right) = \psi(a)
$$
when $A^+$. To see this, let $\omega \in \mathcal F_\psi$. Then $\omega(a^*a) = \psi(a^*a) = 0$ for all ${\ker_\psi}$. Let $x \in {\ker_\psi}^+$. Then $\sqrt{x} \in {\ker_\psi}$ and hence $\omega(x) = \omega(\sqrt{x}\sqrt{x}) = 0$. Since $\ker_\psi$ is spanned by ${\ker_\psi}^+$, it follows that $\omega|_{\ker_\psi} = 0$, and we can therefore define $\tilde{\omega} : A/{\ker_\psi} \to \mathbb C$ such that $\tilde{\omega} \circ q_{\ker_\psi} = \omega$. Define $\tau : (A/{\ker_\psi})^+ \to [0,\infty]$ by
$$
\tau(x) := \sup \left\{ \tilde{\omega}(x): \ \omega \in \mathcal F_\psi \right\} .
$$ 
Let $a \in A^+$. Then
\begin{align*}
& \tau(q_{\ker_\psi}(a)) =  \sup \left\{ \tilde{\omega}(q_{\ker_\psi}(a)): \ \omega \in \mathcal F_\psi \right\} =  \sup \left\{ {\omega}(a): \ \omega \in \mathcal F_\psi \right\} = \psi(a)
\end{align*}
by Combes' theorem, implying the claim. Note that $\tau$ is non-zero and densely defined on $A/{\ker_\psi}$. It remains only to show that $\tau$ is a trace. For this purpose note that since $\sigma^{A/{\ker_\psi}}$ is trivial the algebra $\mathcal M^{ \sigma^{A/{\ker_\psi}}}_\tau$ is the set $\mathcal M_\tau$. By Theorem \ref{24-11-21d} it suffices to show that $\tau(x^*x) = \tau(xx^*)$ when $x \in \mathcal M_\tau$. Note that we can choose $a \in \mathcal M_\psi$ such that $q_{\ker_\psi}(a) =x$. Using Lemma \ref{24-11-21g} we find
\begin{align*}
& \tau(x^*x) = \lim_{n \to \infty} \tau\left(R_n(x)^*R_n(x)\right) = \lim_{n \to \infty} \tau\left(R_n(q_{\ker_\psi}(a))^* R_n(q_{\ker_\psi}(a))\right) \\
& =  \lim_{n \to \infty} \tau\circ q_{\ker_\psi}\left(R_n(a)^* R_n(a)\right) =  \lim_{n \to \infty} \psi\left(R_n(a)^* R_n(a)\right)\\
& =  \lim_{n \to \infty} \psi\left(\sigma_{-i \frac{\beta}{2}}(R_n(a))\sigma_{-i \frac{\beta}{2}}(R_n(a))^*\right)\\
&  = \lim_{n \to \infty} \tau \circ q_{\ker_\psi}\left(\sigma_{-i \frac{\beta}{2}}(R_n(a))\sigma_{-i \frac{\beta}{2}}(R_n(a))^*\right)\\
& = \lim_{n \to \infty} \tau \left(q_{\ker_\psi}(\sigma_{-i \frac{\beta}{2}}(R_n(a)))q_{\ker_\psi}\left(\sigma_{-i \frac{\beta}{2}}(R_n(a))^*\right)\right)\\
& = \lim_{n \to \infty} \tau \left(q_{\ker_\psi}(R_n(a))q_{\ker_\psi}\left(R_n(a)^*\right)\right)\\
& =  \lim_{n \to \infty} \tau\left(R_n(x)R_n(x)^*\right)\\
& = \tau(xx^*).
\end{align*}
\end{proof}

\begin{defn}\label{02-05-22a} A KMS weight for $\sigma$ is \emph{passive} when it satisfies one of the equivalent conditions of Lemma \ref{02-05-22b}. A KMS weight which is not passive is said to be \emph{essential}.
\end{defn} 
 
All passive KMS weights are traces, but not conversely; there are many examples of flows with $0$-KMS weights that are essential. We note that all faithful KMS weights of non-trivial flows are essential, which in particular implies that all KMS weights of non-trivial flows are essential when $A$ is $\sigma$-simple in the sense that the only $\sigma$-invariant ideals in $A$ are $0$ and $A$. In contrast when $\sigma$ is a flow on an abelian $C^*$-algebra all $\beta$-KMS weights for $\sigma$ with $\beta \neq 0$ are passive because they are also automatically $0$-KMS weights and hence satisfies condition (ii) in Lemma \ref{02-05-22b}. See Section \ref{abelian} for a little more on KMS weights for flows on abelian $C^*$-algebras.

We note that for a given $\beta \in \mathbb R$ the set of passive $\beta$-KMS weights constitute a face in the set of all $\beta$-KMS weights:

\begin{lemma}\label{03-05-22} Assume that $\psi$ and $\phi$ are $\beta$-KMS weights for $\sigma$, that $\psi$ is passive and that $\psi \geq \phi$. Then $\phi$ is passive. 
\end{lemma}
\begin{proof} It follows from Theorem \ref{05-03-22} that there is an element $c \in \mathcal Z_\psi^{\sigma''}$, $0 \leq c \leq 1$, such that $\phi( \ \cdot \ ) = \psi''(c\pi_\psi( \ \cdot \ ))$. Since $\psi$ is passive it follows that $\psi$ is also a $(\beta+1)$-KMS weight for $\sigma$ and then a second application of Theorem \ref{05-03-22} shows that $\phi$ is also a $(\beta+1)$-KMS weight for $\sigma$. By Lemma \ref{02-05-22b} this implies that $\phi$ is passive. 
\end{proof}

\begin{notes}\label{09-03-22g} The material in this section is new.
\end{notes}



\chapter{Some examples and constructions}

In this chapter we present a series of examples and constructions of KMS weights for various flows. Although some parts of the chapter will be used in the subsequent development the reader can easily skip the chapter and return to it on an ad hoc basis.

\section{Flows on the compact operators}\label{K-flows} Let $\mathbb K$ denote the $C^*$-algebra of compact operators on an infinite-dimensional separable Hilbert space $\mathbb H$, and let $H$ be a possibly unbounded self-adjoint operator on $\mathbb H$. 

\begin{lemma}\label{03-02-22} Let $k \in \mathbb K$. The map $\mathbb R \ni t \mapsto e^{itH}k$ is continuous for the norm topology.
\end{lemma}
\begin{proof} The map $t \mapsto e^{itH}$ is continuous for the strong operator topology by Theorem 5.6.36 in \cite{KR} and since $k$ is compact the image under $k$ of the unit ball in $\mathbb H$ is compact in $\mathbb H$. It follows that the continuity of $t \mapsto e^{itH}k\psi$ is uniform for $\psi$ in the unit ball of $\mathbb H$; i.e. $t \mapsto e^{itH}k$ is continuous for the norm topology.
\end{proof}

Thanks to Lemma \ref{03-02-22} we can define a flow $\sigma$ on $\mathbb K$ such that
\begin{equation}\label{02-01-22b}
\sigma_t(k) := e^{itH}k e^{-itH} \ \ \forall t \in \mathbb R, \ k \in \mathbb K.
\end{equation}
It is a well-known fact that all flows on $\mathbb K$ are of this form. A proof of this can be found in Appendix \ref{compact operators}. 

The aim here is to find the KMS weights for $\sigma$. Let $\Tr : B(\mathbb H)^+ \to [0,\infty]$ denote the usual trace given by the formula
$$
\Tr (m) = \sum_{i=1}^{\infty} \left<m\psi_i,\psi_i\right>, 
$$
where $\{\psi_i\}_{i=1}^\infty$ is an orthonormal basis in $\mathbb H$. In the following we shall rely on properties of trace-class operators and Hilbert-Schmidt operators on $\mathbb H$, and refer to Chapter VI in \cite{RS} for the facts we shall need. In particular, we note that $\mathcal N_{\Tr}$ by definition is the ideal of Hilbert-Schmidt operators and that the combination of (b) of Lemma \ref{04-11-21n} with (h) of Theorem VI.22 in \cite{RS} shows that $\mathcal M_{\Tr}$ is the ideal of trace class operators.

Let $P$ be the projection valued measure on $\mathbb R$ representing $H$;
$$
H = \int_\mathbb R \lambda \ \mathrm d P_\lambda ,
$$
and set $P_n := P[-n,n]$. Note that
\begin{align*}
&\Tr(e^{-\frac{\beta}{2}H}P_n aP_ne^{-\frac{\beta}{2}H}) = \Tr(\sqrt{a}P_ne^{-\beta H}\sqrt{a}) \\
& \leq \Tr(\sqrt{a}P_{n+1}e^{-\beta H}\sqrt{a}) = \Tr(e^{-\frac{\beta}{2}H}P_{n+1} aP_{n+1}ne^{-\frac{\beta}{2}H})
\end{align*}
for $a \in \mathbb K^+$; that is, 
$$
n \mapsto \Tr(e^{-\frac{\beta}{2}H}P_n aP_ne^{-\frac{\beta}{2}H})
$$ is non-decreasing and we can therefore define $\psi_\beta : \mathbb K^+ \to [0,\infty]$ by 
\begin{equation}\label{02-01-22}
\psi_\beta(a) := \lim_{n \to \infty}\Tr(e^{-\frac{\beta}{2}H}P_n aP_ne^{-\frac{\beta}{2}H}) .
\end{equation}
It follows easily from the properties of $\Tr$ that $\psi_\beta : \mathbb K^+ \to [0,\infty]$ is a $\sigma$-invariant weight.

\begin{lemma}\label{stubbekÃ¸bing2} $\psi_\beta$ is densely defined.
\end{lemma}
\begin{proof}  By spectral theory $\lim_{n \to \infty} P_n =1$ in the strong operator topology, implying that $\lim_{n \to \infty} P_nk = k$ in norm for all $k \in \mathbb K$. In particular, $\lim_{k \to \infty} P_kaP_k = a$ when $a$ is a positive trace class operator. Note that
\begin{align*}
&\psi_\beta(P_kaP_k) =  \lim_{n \to \infty}\Tr(e^{-\frac{\beta}{2}H}P_n P_kaP_kP_ne^{-\frac{\beta}{2}H}) \\
&= \Tr(e^{-\frac{\beta}{2}H}P_kaP_ke^{-\frac{\beta}{2}H}) < \infty,
\end{align*}
since $e^{-\frac{\beta}{2}H}P_k$ is bounded and the trace class operators form a two-sided ideal in $B(\mathbb H)$. The lemma follows now because this ideal is dense in $\mathbb K $.
\end{proof}

\begin{lemma}\label{03-02-22d}
\begin{equation}\label{26-12-21}
\psi_\beta(a^*a) = \begin{cases} \Tr(e^{-\frac{\beta}{2}H}a^*a e^{-\frac{\beta}{2}H}) & \ \text{when $a e^{-\frac{\beta}{2}H}$ is Hilbert-Schmidt}, \\ \infty & \ \ \text{otherwise.} \end{cases} 
\end{equation}
\end{lemma}
\begin{proof} Assume $\psi_\beta(a^*a) < \infty$. Then $\Tr(e^{-\frac{\beta}{2}H}P_na^* aP_ne^{-\frac{\beta}{2}H}) < \infty$ and hence $aP_ne^{-\frac{\beta}{2}H}$ is Hilbert-Schmidt for all $n$. Let $\left\| \ \cdot \ \right\|_{HS}$ denote the Hilbert-Schmidt norm;
$$
\left\|b\right\|_{HS} := \Tr(b^*b)^{\frac{1}{2}} .
$$  
Then, if $m \geq n$,
\begin{equation}\label{02-01-22d}
\begin{split}
&\left\|aP_ne^{-\frac{\beta}{2}H} - aP_me^{-\frac{\beta}{2}H} \right\|_{HS}^2 \\
& = \Tr\left( e^{-\frac{\beta}{2}H}P_na^*a P_n e^{-\frac{\beta}{2}H}\right) +  \Tr\left( e^{-\frac{\beta}{2}H}P_ma^*a P_m e^{-\frac{\beta}{2}H}\right) \\
& \ \ \ \ \  \ -  \Tr\left( e^{-\frac{\beta}{2}H}P_ma^*a P_n e^{-\frac{\beta}{2}H}\right) -  \Tr\left( e^{-\frac{\beta}{2}H}P_na^*a P_m e^{-\frac{\beta}{2}H}\right) \\
&  = \Tr\left( e^{-\frac{\beta}{2}H}P_na^*a P_n e^{-\frac{\beta}{2}H}\right) +  \Tr\left( e^{-\frac{\beta}{2}H}P_ma^*a P_m e^{-\frac{\beta}{2}H}\right) \\
& \ \ \ \ \  \ -  \Tr\left( P_me^{-\frac{\beta}{2}H}P_ma^*a P_n e^{-\frac{\beta}{2}H}P_n\right) -  \Tr\left(P_n e^{-\frac{\beta}{2}H}P_na^*a P_m e^{-\frac{\beta}{2}H}P_m\right) \\ 
&  = \Tr\left( e^{-\frac{\beta}{2}H}P_na^*a P_n e^{-\frac{\beta}{2}H}\right) +  \Tr\left( e^{-\frac{\beta}{2}H}P_ma^*a P_m e^{-\frac{\beta}{2}H}\right) \\
& \ \ \ \ \  \ -  \Tr\left(P_nP_m e^{-\frac{\beta}{2}H}P_ma^*a P_n e^{-\frac{\beta}{2}H}\right) -  \Tr\left( e^{-\frac{\beta}{2}H}P_na^*a P_m e^{-\frac{\beta}{2}H}P_mP_n\right) \\ 
&  = \Tr\left( e^{-\frac{\beta}{2}H}P_na^*a P_n e^{\frac{-\beta}{2}H}\right) +  \Tr\left( e^{-\frac{\beta}{2}H}P_ma^*a P_m e^{-\frac{\beta}{2}H}\right) \\
& \ \ \ \ \  \ -  \Tr\left( e^{-\frac{\beta}{2}H}P_na^*a P_n e^{-\frac{\beta}{2}H}\right) -  \Tr\left( e^{-\frac{\beta}{2}H}P_na^*a P_n e^{-\frac{\beta}{2}H}\right) \\ 
&   = \Tr\left( e^{-\frac{\beta}{2}H}P_ma^*a P_m e^{-\frac{\beta}{2}H}\right) - \Tr\left( e^{-\frac{\beta}{2}H}P_na^*a P_n e^{-\frac{\beta}{2}H}\right) .
\end{split}
\end{equation}
Since $\lim_{k \to \infty} \Tr\left( e^{-\frac{\beta}{2}H}P_ka^*a P_k e^{-\frac{\beta}{2}H}\right) = \psi_\beta(a^*a) < \infty$ this equality implies that $\{aP_ne^{-\frac{\beta}{2}H}\}$ is Cauchy in the Hilbert-Schmidt norm and converges therefore to a Hilbert-Schmidt operator in that norm, and therefore also in the operator norm. By spectral theory $\lim_{n \to \infty}aP_ne^{-\frac{\beta}{2}H}\psi =  ae^{-\frac{\beta}{2}H}\psi$ for all $\psi \in D(e^{-\frac{\beta}{2}H})$ and it follows therefore that $ae^{-\frac{\beta}{2}H}$ is not only bounded, but in fact a Hilbert-Schmidt operator. Furthermore, since $\lim_{n \to \infty} \left\|aP_ne^{-\frac{\beta}{2}H} - ae^{-\frac{\beta}{2}H}\right\|_{HS} =0$ it follows that
$$
\psi_\beta(a^*a) = \Tr(e^{-\frac{\beta}{2}H}a^*a e^{-\frac{\beta}{2}H}).
$$
Conversely, assuming that $ae^{-\frac{\beta}{2}H}$ is Hilbert-Schmidt we have that 
 $$
 \lim_{n \to \infty} P_ne^{-\frac{\beta}{2} H} a^* = e^{-\frac{\beta}{2}H}a^*
 $$ 
 in operator norm since $\lim_{n \to \infty} P_n = 1$ in the strong operator topology and $e^{-\frac{\beta}{2} H} a^*$ is compact. Hence $\{aP_ne^{-\beta H}a^*\}$ converges increasingly and in norm to $ae^{-\beta H}a^*$. By using the lower semi-continuity of $\Tr$ we get 
 \begin{align*}
&
\psi_\beta(a^*a) = \lim_{n \to \infty} \Tr(e^{-\frac{\beta}{2}H}P_n a^*aP_ne^{-\frac{\beta}{2}H}) = \lim_{n \to \infty}\Tr (aP_ne^{-\beta H}a^*) \\
& = \Tr (ae^{-\beta H}a^*) = \Tr (e^{-\frac{\beta}{2}H} a^*ae^{-\frac{\beta}{2}H}) .\\
\end{align*}
\end{proof}

\begin{lemma}\label{03-02-22e} $\psi_\beta$ is a $\beta$-KMS weight for $\sigma$.
\end{lemma}
\begin{proof}
It follows from Lemma \ref{03-02-22d} that 
$$
\mathcal N_{\psi_\beta} = \left\{ a \in \mathbb K: \ ae^{-\frac{\beta}{2} H} \ \text{ is Hilbert-Schmidt} \right\} .
$$ 
Since the Hilbert-Schmidt operators form a two-sided ideal in $B(\mathbb H)$ is follows that $P_n e^{iz H} \mathcal N_{\psi_\beta} \subseteq \mathcal N_{\psi_\beta}$ and $\mathcal N_{\psi_\beta}P_n  e^{iz H}\subseteq \mathcal N_{\psi_\beta}$ for all $z \in \mathbb C$ and hence
\begin{equation}\label{25-09-23j}
 e^{iz H}P_n \mathcal M_{\psi_\beta}P_n e^{-iz H} \subseteq \mathcal M_{\psi_\beta}
\end{equation}
for all $z \in \mathbb C$ since $\mathcal M_{\psi_\beta} = \Span \mathcal N_{\psi_\beta}^* \mathcal N_{\psi_\beta}$ by (b) of Lemma \ref{04-11-21n}. Let $a \in \mathbb K$. Then
$$
z \mapsto e^{iz H}P_naP_n e^{-izH}
$$
is entire holomorphic and agrees with $\sigma_t(P_naP_n)$ for $t \in \mathbb R$. It follows therefore that $P_naP_n \in \mathcal A_\sigma$ and 
$$
\sigma_z(P_naP_n) = e^{iz H}P_naP_n e^{-izH} = P_ne^{iz H}a e^{-izH}P_n
$$
for all $z \in \mathbb C$. It follows from this and \eqref{25-09-23j} that
$$
P_n\mathcal M^\sigma_{\psi_\beta} P_n \subseteq  \mathcal M^\sigma_{\psi_\beta} .
$$
In particular $P_n\mathcal M^\sigma_{\psi_\beta} P_n \subseteq P_{n+1}\mathcal M^\sigma_{\psi_\beta} P_{n+1}$, implying that
$$
S := \bigcup_n P_n\mathcal M^\sigma_{\psi_\beta}P_n 
$$
is a subspace of $\mathcal M^\sigma_{\psi_\beta}$. We claim that $S$ has the properties required of $S$ in Theorem \ref{12-12-13}. To see this, let $b \in \mathcal M^\sigma_{\psi_\beta}$. We will show that
\begin{equation}\label{stubbekÃ¸bing}
\lim_{n \to \infty} P_nbP_n = b
\end{equation}
and
\begin{equation}\label{stubbekÃ¸bing1}
\lim_{n \to \infty} \Lambda_{\psi_\beta}(P_nbP_n)  = \Lambda_{\psi_\beta}(b).
\end{equation}
As pointed out in the proof of Lemma \ref{stubbekÃ¸bing2}, $\lim_{n \to \infty} P_nk = k$ in norm for all $k \in \mathbb K$. \eqref{stubbekÃ¸bing} follows from this because $\mathcal M_{\psi_\beta}$ is spanned by its positive elements. To establish \eqref{stubbekÃ¸bing1}, note that
\begin{equation}\label{31-07-22}
\begin{split}
& \left\| \Lambda_{\psi_\beta}(P_nbP_n) -  \Lambda_{\psi_\beta}(b)\right\| \\
&\leq \left\| \Lambda_{\psi_\beta}(P_nbP_n) -  \Lambda_{\psi_\beta}(P_nb)\right\|  + \left\| \Lambda_{\psi_\beta}(P_nb) -  \Lambda_{\psi_\beta}(b)\right\| \\
& = \left\| \Lambda_{\psi_\beta}(P_nb(1-P_n))\right\|+ \left\| \Lambda_{\psi_\beta}((1-P_n)b)\right\| ,
\end{split}
\end{equation}
\begin{equation}\label{31-07-22a}
\begin{split}
& \left\| \Lambda_{\psi_\beta}(P_nb(1-P_n))\right\|^2 = \Tr(e^{-\frac{\beta}{2}H} (1-P_n)b^*P_nb(1-P_n) e^{-\frac{\beta}{2}H} ) \\
& \leq \Tr(e^{-\frac{\beta}{2}H} (1-P_n)b^*b(1-P_n) e^{-\frac{\beta}{2}H} ) \\
& = \Tr(be^{-\beta H}(1-P_n)b^*) \\
& = \Tr(be^{-\beta H}b^*) - \Tr(be^{-\beta H}P_nb^*) 
\end{split}
\end{equation}
and
\begin{equation}\label{31-07-22b}
\begin{split}
& \left\| \Lambda_{\psi_\beta}((1-P_n)b)\right\|^2 = \Tr(e^{-\frac{\beta}{2}H} b^*(1-P_n)be^{-\frac{\beta}{2}H} ) \\
& = \Tr(e^{-\frac{\beta}{2}H} b^*be^{-\frac{\beta}{2}H} ) -\Tr(e^{-\frac{\beta}{2}H} b^*P_nbe^{-\frac{\beta}{2}H} ).
\end{split}
\end{equation}
The sequence $\{be^{-\beta H}P_nb^*\}$ increases with $n$ and converges in norm to $be^{-\beta H}b^*$ because $be^{-\frac{\beta}{2} H}$ and $e^{-\frac{\beta}{2} H}b^*$ are both Hilbert-Schmidt. Similarly, the sequence $\{e^{-\frac{\beta}{2}H} b^*P_nbe^{-\frac{\beta}{2}H}\}$ increases with $n$ and converges in norm to $e^{-\frac{\beta}{2}H} b^*be^{-\frac{\beta}{2}H}$. It follows therefore from the lower semi-continuity of $\Tr$ and \eqref{31-07-22a} and \eqref{31-07-22b} that 
$$\lim_{n \to \infty} \left\| \Lambda_{\psi_\beta}(P_nb(1-P_n))\right\| = \lim_{n \to \infty} \left\| \Lambda_{\psi_\beta}((1-P_n)b)\right\| = 0 .
$$
Inserted into \eqref{31-07-22} it follows that $\lim_{n \to \infty} \Lambda_{\psi_\beta}(P_nbP_n) =  \Lambda_{\psi_\beta}(b)$ as desired.

 To deduce from Theorem \ref{12-12-13} that $\psi_\beta$ is a $\beta$-KMS weight for $\sigma$ it suffices now to check that $\psi_\beta(x^*x) = \psi_{\beta}\left(\sigma_{-i \frac{\beta}{2}}(x)\sigma_{-i \frac{\beta}{2}}(x)^*\right)$ for $x \in S$. Let $x = P_naP_n$, where $a \in \mathcal M_{\psi_\beta}$. Then 
$$
\sigma_{-i \frac{\beta}{2}}(x) = e^{\frac{\beta}{2} P_nH}P_naP_n e^{-\frac{\beta}{2} P_nH}
$$
and hence
\begin{align*}
& \psi_{\beta}\left(\sigma_{-i \frac{\beta}{2}}(x)\sigma_{-i \frac{\beta}{2}}(x)^*\right) \\
&= \Tr (e^{-\frac{\beta}{2}H}e^{\frac{\beta}{2} P_nH} P_naP_n e^{-\beta P_nH} P_na^*P_n e^{\frac{\beta}{2} P_nH}e^{-\frac{\beta}{2}H} ) \\
& = \Tr(P_naP_ne^{-\beta H}a^*P_n) = \Tr(e^{-\frac{\beta}{2}H}P_na^*P_naP_ne^{-\frac{\beta}{2}H})  = \psi_{\beta}(x^*x) .
\end{align*}

\end{proof}

\begin{lemma}\label{stubbekoebing5} $\psi_\beta$ is the only $\beta$-KMS weight for $\sigma$, up to multiplication by scalars.
\end{lemma} 
\begin{proof} Let $\psi$ be a $\beta$-KMS weight for $\sigma$. Choose $n$ so big that $P_n \neq 0$. Then $P_n\mathbb KP_n$ is a $\sigma$-invariant hereditary $C^*$-subalgebra of $\mathbb K$ which is full in $\mathbb K$ since $\mathbb K$ is simple. It is easy to see that the map
 \begin{equation}\label{stubbekoebing4}
 (P_n\mathbb KP_n)^+ \ni T \mapsto \psi(e^{\frac{\beta}{2}H} T e^{\frac{\beta}{2}H})
 \end{equation}
is a $\sigma$-invariant weight on $P_n\mathbb KP_n$ since $\psi$ is a $\sigma$-invariant weight on $\mathbb K$. To see that it is a trace on $P_n\mathbb KP_n$, let $T \in P_n \mathbb K P_n$.  Then $Te^{\frac{\beta}{2}H} \in \mathcal A_\sigma$ and $\sigma_{-i \frac{\beta}{2}}(Te^{\frac{\beta}{2}H}) = e^{\frac{\beta}{2}H} T$. Since $\psi$ is a $\beta$-KMS weight for $\sigma$ this gives
\begin{align*}
&  \psi(e^{\frac{\beta}{2}H} T^*T e^{\frac{\beta}{2}H}) = \psi(\sigma_{-i \frac{\beta}{2}}(Te^{\frac{\beta}{2}H})\sigma_{-i \frac{\beta}{2}}(Te^{\frac{\beta}{2}H})^*)=\psi(e^{\frac{\beta}{2}H}TT^*e^{\frac{\beta}{2}H}).
\end{align*}
This shows that \eqref{stubbekoebing4} is a $\sigma$-invariant lower semi-continuous trace on $P_n\mathbb K P_n$, i.e. a $0$-KMS weight for $\sigma$ on $P_n\mathbb KP_n$.  Applied to $0$-KMS weights, it follows from Theorem \ref{02-12-21} and the essential uniqueness of the trace on $\mathbb K$ that 
$\psi(e^{\frac{\beta}{2}H} \ \cdot \ e^{\frac{\beta}{2}H}) = \lambda \Tr( \ \cdot \ )$
on $P_n\mathbb KP_n$ for some $\lambda > 0$. It follows that
\begin{align*}
&\psi(K) = \psi(e^{\frac{\beta}{2}H} e^{-\frac{\beta}{2}H}K      e^{-\frac{\beta}{2}H} \ e^{\frac{\beta}{2}H}) \\
&= \lambda \Tr(e^{-\frac{\beta}{2}H}K e^{-\frac{\beta}{2}H} )  = \lambda \psi_{\beta}(K) 
\end{align*}
for all $K \in (P_n\mathbb K P_n)^+$. By using the uniqueness statement in Theorem \ref{02-12-21} again, this time applied to $\beta$-KMS weights, it follows that $\psi = \lambda \psi_\beta$. 
\end{proof}




To decide when $\psi_\beta$ is bounded we use the following lemma.

\begin{lemma}\label{03-02-22h} Let $\varphi : A^+ \to [0,\infty]$ be a weight on the $C^*$-algebra $A$ and let $\{u_i\}_{i \in I}$ an approximate unit for $A$ such that $0 \leq u_i \leq 1$ for all $i$. Then $\varphi$ is bounded if and only if there  is $M > 0$ such that $\varphi(u_i) \leq M$ for all $i$.
\end{lemma}
\begin{proof} Assume that such an $M$ exists, and consider $a \in A$ such that $0 \leq a \leq 1$. For each $n \in \mathbb N$ we pick $i_n\in I$ such that $\left\|u_{i_n}\sqrt{a} - \sqrt{a}\right\| \leq \frac{1}{n}$. Then $\lim_{n \to \infty} u_{i_n}au_{i_n} = a$ and hence the lower semi-continuity of $\varphi$ entails that
$$
\varphi(a) \leq \liminf_n \varphi(u_{i_n} au_{i_n}) \leq \liminf_n \varphi(u_{i_n}^2)\leq \liminf_n \varphi(u_{i_n}) \leq M .
$$
It follows that $\varphi$ is bounded. The reverse implication follows from Lemma \ref{21-10-23a}.
\end{proof} 

Let $\{E_k\}$ be an approximate unit in $\mathbb K$ consisting of projections. By Lemma \ref{03-02-22h} $\psi_\beta$ is bounded if and only if there is $M > 0$ such that $\psi_\beta(E_k) \leq M$ for all $k$. Assume that such an $M$ exists. Then $\Tr(e^{-\frac{\beta}{2}H} P_nE_kP_n e^{-\frac{\beta}{2}H}) \leq \psi_\beta(E_k) \leq M$ for all $k$ and by lower semi-continuity of $\Tr$ with respect to the strong operator topology it follows that $\Tr(e^{-\frac{\beta}{2}H} P_ne^{-\frac{\beta}{2}H}) \leq M$ for all $n$. In particular, $P_ne^{-\frac{\beta}{2}H}$ is Hilbert-Schmidt for all $n$ and the calculation \eqref{02-01-22d} with $a =1$ shows that $\{P_ne^{-\frac{\beta}{2}H}\}$ converges in the Hilbert-Schmidt norm, implying that $e^{-\frac{\beta}{2}H}$ is Hilbert-Schmidt, or alternatively, that $e^{-\beta H}$ is trace-class. Conversely, if $e^{-\beta H}$ is trace-class it follows that $\psi_\beta(E_k) \leq \Tr(e^{-\beta H})$ for all $k$, and $\psi_\beta$ is bounded.

We summarize the results of this section in the following

\begin{thm}\label{02-01-22a} For each $\beta \in \mathbb R$ the flow $\sigma$, given by \eqref{02-01-22b}, has up to multiplication by scalars a unique $\beta$-KMS weight $\psi_\beta$ given by \eqref{26-12-21}. It is bounded if and only if $e^{-\beta H}$ is trace-class.
\end{thm}


Set
$$
C_H = \left\{\beta \in \mathbb R: \  \Tr(e^{-\beta H}) < \infty \right\} .
$$
By Theorem \ref{02-01-22a} this is the set of real numbers $\beta$ for which the flow $\sigma$ on $\mathbb K$ defined by $H$ has a bounded $\beta$-KMS weight.

\begin{lemma}\label{13-02-22} The set $C_H$ is one of the following:
\begin{itemize}
\item $C_H = \emptyset$,
\item $C_H = [r,\infty)$ for some $r > 0$,
\item $C_H = ]r,\infty[$ for some $r \geq 0$,
\item $C_H = (-\infty, r[$ for some $r\leq 0$, or
\item $C_H = (-\infty,r]$ for some $r < 0$,
\end{itemize}
and all possibilities occur.
\end{lemma}
\begin{proof} Unless $H =0$ the function $\beta \mapsto \Tr(e^{-\beta H})$ is strictly convex on $C_H$ since the exponential function is strictly convex and $\Tr$ is faithful. It follows therefore that $C_H$ is an interval. Since $0 \notin C_H$, because $\Tr$ is unbounded, the interval $C_H$ is contained in $]0,\infty)$ or $(-\infty,0[$. Since $C_H = -C_{-H}$, to prove the first part of the statement, it suffices to show
$$
0 < \beta \in C_H, \ \beta < \beta' \Rightarrow \beta' \in C_H.
$$ 
As above we let $P$ be the projection valued measure on $\mathbb R$ representing $H$ and set $P_+ = P[0,\infty)$ and $P_- = P(-\infty,0)$. Since $e^{-\beta H}$ is trace class and $e^{-\beta H} = e^{-\beta H}P_+ + e^{-\beta H}P_- \geq P_-$, it follows that $P_-$ is finite dimensional and hence $e^{-\beta H}P_-$ and $e^{-\beta' H}P_-$ are both of trace class. Since $e^{-\beta' H}P_+ \leq e^{-\beta H}P_+$, it follows that $e^{-\beta' H} = e^{-\beta' H}P_+ + e^{-\beta' H}P_-$ is trace class, i.e. $\beta' \in C_H$.

To show that all possibilities occur observe that $C_H = \emptyset$ when $H = 0$. Let $\{a_n\}_{n=1}^\infty$ be a sequence of real numbers and $\{\psi_n\}_{n=1}^\infty$ an orthonormal basis in $\mathbb H$. We can then define a self-adjoint operator $H$ on $\mathbb H$ such that
$$
H\psi_n = a_n\psi_n .
$$
In this case,
$$
C_H = \left\{ \beta \in \mathbb R: \ \sum_{n=1}^\infty e^{-\beta a_n} < \infty \right\}.
$$ 
For each $r \geq 0$ it is possible to construct a sequence $\{b_n\}$ of positive numbers such that $\sum_{n=1}^\infty b_n^\beta < \infty$ if and only if $\beta > r$ so if we use $a_n = \log b_n$ to define $H$ we get $C_H = (r,\infty[$. If $r > 0$ it is also  possible to construct a sequence $\{b'_n\}$ of positive numbers such that $\sum_{n=1}^\infty {b'_n}^\beta < \infty$ if and only if $\beta \geq r$ and when we then use $a_n = \log b'_n$ to define $H$ we get $C_H = [r,\infty[$. The intervals $]-\infty,-r)$ and $]-\infty,r]$ are then realized by using the sequences $\{-a_n\}$ instead of $\{a_n\}$. 
\end{proof}

\section{Flows on abelian $C^*$-algebras}\label{abelian}

Given a locally compact Hausdorff space $X$ we denote by $C_0(X)$ the $C^*$-algebra of continuous functions on $X$ that vanish at infinity. A \emph{flow} on $X$ is a continuous representation $\phi = (\phi_t)_{t \in \mathbb R}$ of $\mathbb R$ by homeomorphisms. In more detail, each $\phi_t$ is a homeomorphism $\phi_t : X \to X$, and
\begin{itemize}
\item $\phi_s(\phi_t(x)) = \phi_{t+s}(x)$ for all $s,t \in \mathbb R$ and all $x \in X$, and
\item the map $\mathbb R \times X \ni (t,x) \mapsto \phi_t(x) \in X$ is continuous.
\end{itemize}
A flow $\phi$ on $X$ gives rise to a flow $\sigma^\phi$ on $C_0(X)$ defined such that $\sigma^\phi_t(f) = f \circ \phi_t$, and all flows on $C_0(X)$ arises in this way. See Section 2.1 of \cite{Wi}.

\begin{lemma}\label{03-01-22a} Let $X$ be a second countable locally compact Hausdorff space and let $\psi : C_0(X)^+ \to [0,\infty]$ be a densely defined weight on $C_0(X)$. There is a regular measure $\mu$ on $X$ such that
\begin{equation}\label{03-01-22c}
\psi(f) = \int_X f \ \mathrm d\mu    
\end{equation}
for all $f \in C_0(X)^+$.
\end{lemma}
\begin{proof} Since we assume that $X$ is second countable the $C^*$-algebra $C_0(X)$ is separable and it follows therefore from Theorem \ref{09-11-21h} that there is a sequence $\{\omega_n\}_{n=1}^\infty$ of positive functionals on $C_0(\mathbb R)$ such that 
$$
\psi(f) =\sum_{n=1}^\infty\omega_n(f)
$$ for all $f \in C_0(X)^+$. By the Riesz representation theorem, \cite{Ru0}, there are bounded Borel measures $\mu_n$ on $\mathbb R$ such that 
$$
\omega_n(f) = \int_\mathbb R f \ \mathrm d\mu_n \ \ \forall f \in C_0(\mathbb R).
$$
Define a Borel measure $\mu$ on $\mathbb R$ such that 
$$
\mu(A) = \sum_{n=1}^\infty \mu_n(A)
$$
for all Borel sets $A \subseteq \mathbb R$. For any $f \in C_0(\mathbb R)^+$, by approximating $f$ by an increasing sequence of simple functions it follows that
$$
\int_\mathbb R f \ \mathrm{d} \mu = \sum_{n=1}^\infty \int_{\mathbb R} f \ \mathrm d\mu_n .
$$
Hence $\int_\mathbb R f \ \mathrm{d} \mu = \sum_{n=1}^\infty \omega_n(f) = \psi(f)$. To show that $\mu$ is regular, let $K \subseteq X$ be a compact subset. By Urysohn's lemma there is a continuous function $f : X \to [0,2]$ of compact support such that $f(x) = 2$ for all $x \in K$. Since $\psi$ is densely defined there is a $g \in C_0(X)^+$ such that $\psi(g) < \infty$ and $\left\|f - g \right\| \leq \frac{1}{2}$. Then $g(x) \geq 1$ for all $x \in K$ and hence 
$$
\mu(K) \leq \int_X g \ \mathrm d\mu = \psi(g) < \infty .
$$
This shows that $\mu$ is regular.
\end{proof}

When $\phi$ is a flow on $X$ and $\sigma^\phi$ is the associated flow on $C_0(X)$, the $\sigma^\phi$-invariant densely defined weights are given by integration with respect to regular Borel measures $\mu$ on $X$ that are $\phi$-invariant in the sense that $\mu \circ \phi_t = \mu$ for all $t \in \mathbb R$. Hence the $0$-KMS weights for $\sigma^\phi$ are in bijective correspondence with the non-zero regular $\phi$-invariant measures on $X$ via the equation \eqref{03-01-22c}. As pointed out after Definition \ref{02-05-22a}, when $\beta \neq 0$ all $\beta$-KMS weights are passive. In terms of the flow $\phi$ they arise from regular measures on the set of fixed points for $\phi$. In particular essential $\beta$-KMS weights correspond to $\phi$-invariant regular measures that are not concentrated on the fixed points of $\phi$. 

To summarize, for $\beta \neq 0$ the $\beta$-KMS weights all arise from integration with respect to a regular Borel measure on $\left\{x \in X: \ \phi_t(x) = x \ \forall t \in \mathbb R\right\}$ while the $0$-KMS weights arise from integration with respect to a general regular $\phi$-invariant Borel measure. In particular, the set of $\beta$-KMS weights is the same for all $\beta \neq 0$ while the set of $0$-KMS weights is a larger set in general.

\section{Exterior equivalence}\label{inner flows}

In this section we shall use the multiplier algebra $M(A)$ of a $C^*$-algebra $A$ and the strict topology on $M(A)$. The reader may find the definitions in Appendix \ref{multipliers}.

Two flows $\sigma$ and $\sigma'$ on the same $C^*$-algebra $A$ are \emph{exterior equivalent} when there is a unitary representation $u = (u_t)_{t \in \mathbb R}$ of $\mathbb R$ by multipliers of $M(A)$ such that
\begin{itemize}
\item[(a)] $\mathbb R \ni t \mapsto u_t$ is strictly continuous, 
\item[(b)] $u_{s+t} = u_s\sigma_s(u_t)$ for all $s,t \in \mathbb R$, and
\item[(c)] $\sigma'_t = \Ad u_t \circ \sigma_t$ for all $t \in \mathbb R$,
\end{itemize}
cf. Section 2.2 of \cite{Co1} or Section 8.11 in \cite{Pe}. 

\begin{thm}\label{06-06-22b} Let $\sigma$ and $\sigma'$ be exterior equivalent flows on $A$. For each $\beta \in \mathbb R$ there is an affine bijection $\KMS(\sigma,\beta) \simeq \KMS(\sigma',\beta)$.
\end{thm}
\begin{proof} Let $M_2(A)$ denote the $C^*$-algebra of $2 \times 2$ matrices over $A$. Define the flow $\gamma$ on $M_2(A)$ such that
$$
\gamma_t\left( \begin{smallmatrix} a & b \\ c & d \end{smallmatrix} \right) = \left(\begin{matrix} \sigma_t(a) & \sigma_t(b)u_{-t} \\ u_t \sigma_t(c) & \sigma'_t(d) \end{matrix} \right) ,
$$
cf. Lemme 2.2.6 of \cite{Co1} or \cite{Pe}. The embeddings $\iota_i : A \to M_2(A), \ i = 1,2$, defined by
$$
\iota_1(a) = \left( \begin{smallmatrix} a & 0 \\ 0 & 0 \end{smallmatrix} \right)
$$
and
$$
\iota_2(a) = \left( \begin{smallmatrix} 0 & 0 \\ 0 & a \end{smallmatrix} \right) 
$$
identify $A$ as a $C^*$-subalgebra of $M_2(A)$ in different ways. Note that $\iota_1(A)$ and $\iota_2(A)$ are both hereditary and full $C^*$-subalgebras of $M_2(A)$. Since $\iota_i$ is equivariant and a $*$-isomorphism onto its image it follows from Lemma \ref{01-10-23a} and Theorem \ref{07-06-22e} that $\psi \mapsto \psi \circ \iota_1$ is a bijection $\KMS(\gamma,\beta) \simeq \KMS(\sigma,\beta)$ and $\psi \mapsto \psi \circ \iota_2$ is a bijection $\KMS(\gamma,\beta) \simeq \KMS(\sigma',\beta)$. Hence $\KMS(\sigma,\beta) \simeq \KMS(\gamma,\beta) \simeq \KMS(\sigma',\beta)$.

\end{proof}

\subsection{A special case of Theorem \ref{06-06-22b}}\label{05-09-22g}
Let $\sigma$ be flow on $A$ and $u = (u_t)_{t \in \mathbb R}$ a strictly continuous unitary representation of $\mathbb R$ by multipliers of $M(A)$ such that $\sigma_t(u_s) = u_s$ for all $s,t$. Then 
$$
\sigma'_t := \Ad u_t \circ \sigma_t
$$
is a flow on $A$ exterior equivalent to $\sigma$. In this setting we can give an explicit description of the bijection $\KMS(\sigma,\beta) \simeq \KMS(\sigma',\beta)$ of Theorem \ref{06-06-22b}. For this we consider $A$ as a non-degenerate $C^*$-subalgebra of $B(\mathbb H)$, the $C^*$-algebra of bounded operators on the Hilbert space $\mathbb H$. Here non-degenerate means that $\left\{a\xi : \ a \in A, \ \xi \in H\right\}$ spans a dense subspace of $\mathbb H$. We can then identify $M(A)$ as
$$
M(A) = \left\{ m \in B(\mathbb H) : \ mA \subseteq A, \ Am \subseteq A \right\} ;
$$ 
cf. Lemma \ref{26-09-22}. Then $u$ is a strongly continuous one-parameter group of unitaries on $\mathbb H$ and by Stone's theorem there is a self-adjoint operator $H$ on $\mathbb H$ such that 
$$
u_t = e^{it H}  .
$$

In the following we denote by $C_0(\mathbb R)$ the algebra of continuous functions on $\mathbb R$ that vanish at infinity and by $C_c(\mathbb R)$ the sub-algebra of continuous functions on $\mathbb R$ with compact support, and we extend each automorphism of $A$ to a unique automorphism of $M(A)$ which is continuous for the strict topology on norm-bounded sets of $M(A)$; see Remark \ref{27-09-23} in Appendix \ref{multipliers}. In particular, $\sigma_t \in \Aut M(A)$ for all $t \in \mathbb R$.

We shall need the Fourier transform and since there are confusingly many different conventions regarding this transform we will be explicit: Given a function $f \in L^1(\mathbb R)$ we define the Fourier transform $\hat{f}$ of $f$ to be the function
$$
\hat{f}(x) := \int_\mathbb R e^{ixt} f(t) \ \mathrm dt .
$$ 
With this definition we follow \cite{KR}. 

\begin{lemma}\label{11-10-23} The set $\left\{ \hat{f} : \ f \in C_c(\mathbb R)\right\}$ is a dense $*$-subalgebra of $C_0(\mathbb R)$.
\end{lemma}
\begin{proof} Observe $f \star g$ and $t \mapsto g(-t)$ are elements of $C_c(\mathbb R)$ when $f,g$ are. Using this, and despite that Rudin uses other conventions, it follows from (c) and (d) of Theorem 9.2 in \cite{Ru0} that $\left\{ \hat{f} : \ f \in C_c(\mathbb R)\right\}$ is a $*$-subalgebra of $C_0(\mathbb R)$. (This follows also (almost) from Corollary 3.2.28 in \cite{KR}.) Thanks to the Stone-Weierstrass theorem it suffices therefore to show that this algebra separates points of $\mathbb R$ and does not vanish at any point. Let $x \neq y$ in $\mathbb R$. There is then a $t_0 \in \mathbb R$ such that $e^{ixt_0} \neq e^{iyt_0}$; say $\left|e^{ixt_0} - e^{iyt_0}\right| \geq  \epsilon > 0$. We can then choose a function $f \in C_c(\mathbb R)$ with a support concentrated around $t_0$ such that $\int_\mathbb R f(s) \ \mathrm ds = 1$ and
$$
\left|\int_\mathbb R e^{ixs} f(s) \ \mathrm ds  - e^{it_0x}\right| \leq \frac{\epsilon}{3} 
$$
and
$$ 
\left|\int_\mathbb R e^{iys} f(s) \ \mathrm ds  - e^{it_0y}\right| \leq \frac{\epsilon}{3} .
$$
Then $\hat{f}(x) \neq \hat{f}(y)$ and $\hat{f}(x) \neq 0$ if $\epsilon < 1$.
\end{proof}

\begin{lemma}\label{23-01-22} $f(H) \in M(A)$ for all $f\in C_0(\mathbb R)$ and $\sigma_t(f(H)) = f(H)$ for all $t \in \mathbb R$.
\end{lemma}
\begin{proof} When $g \in C_c(\mathbb R)$,
$$
\widehat{g}(H) = \int_{\mathbb R} g(t) u_t \ \mathrm d t,
$$
where $\widehat{g}(x) = \int_{\mathbb R} g(t) e^{itx} \ \mathrm d t$
is the Fourier transform of $g$; cf. Theorem 5.6.36 in \cite{KR}. For $a \in A$,
$$
\widehat{g}(H)a = \int_{\mathbb R} g(t) u_t a \ \mathrm d t,
$$
and
$$
a\widehat{g}(H) = \int_{\mathbb R} g(t) a u_t  \ \mathrm d t.
$$
Since $t \mapsto u_ta$ and $t \mapsto au_t$ are continuous in norm these integral representations show that $\widehat{g}(H)a$ and $a\widehat{g}(H)$ are elements of $A$, cf. Appendix \ref{integration}, and hence $\widehat{g}(H) \in M(A)$. 
Since $\left\{ \widehat{g} : \ g \in C_c(\mathbb R)\right\}$ is dense in $C_0(\mathbb R)$ by Lemma \ref{11-10-23}  and $\left\|f(H)\right\| \leq \|f\|$ for all $f \in C_0(\mathbb R)$ it follows that $f(H) \in M(A)$ for all $f\in C_0(\mathbb R)$. 

That $\sigma_t(f(H)) = f(H)$ follows because $\sigma_t(u_s) = u_s$ by assumption: For all $a \in A$ and $g \in L^1(\mathbb R)$,
\begin{align*}
&\sigma_t(\widehat{g}(H))\sigma_t(a) = \sigma_t(\widehat{g}(H)a) = \sigma_t\left(\int_\mathbb R g(s)u_s a \ \mathrm d s\right) \\
&= \int_\mathbb R g(s)\sigma_t(u_s a) \ \mathrm d s = \left(\int_\mathbb R g(s)u_s  \ \mathrm d s \right) \sigma_t(a) = \widehat{g}(H)\sigma_t(a) .
\end{align*}
Since $\left\{ \sigma_t(a) \xi: \ \xi \in \mathbb H, \ a \in A \right\}$ is dense in $\mathbb H$ this identity shows that $\sigma_t(\widehat{g}(H)) = \widehat{g}(H))$, and the equality $\sigma_t(f(H)) = f(H)$ for $f \in C_0(\mathbb R)$ follows from Lemma \ref{11-10-23} by continuity.
\end{proof}

For $ z \in \mathbb C$ let $e^{zH}$ be the normal operator on $\mathbb H$ obtained from the function $t \mapsto e^{zt}$ by spectral theory applied to $H$, e.g. Theorem 5.6.26 in \cite{KR}. Note that $e^{zH}$ is often unbounded, but for $f\in C_c(\mathbb R)$ the operator $f(H)e^{zH} = e^{zH}f(H)$ is bounded and an element of $M(A)$ by Lemma \ref{23-01-22}. Given a $\beta$-KMS weight $\psi$ for $\sigma$, we can therefore define $\psi_\beta : A^+ \to [0,\infty]$ such that
\begin{equation}\label{29-06-22ax}
\psi_\beta(a) := \sup \left\{\psi\left(e^{-\frac{\beta H}{2}} f(H) a f(H) e^{-\frac{\beta H}{2}}\right) : \ f \in C_c(\mathbb R), \ 0 \leq f \leq 1 \right\} .
\end{equation}
We aim to show that $\psi_\beta$ is a $\beta$-KMS weight for $\sigma'$. For this note first of all that $\psi_\beta$ is lower semi-continuous since $\psi$ is.

\begin{lemma}\label{23-01-22ax} Let $(g_n)_{n=1}^\infty$ be a sequence in $C_c(\mathbb R)$ such that $0 \leq g_n \leq g_{n+1} \leq 1$ for all $n \in \mathbb N$ and $\lim_{n \to \infty} g_n(x) = 1$ for all $x \in \mathbb R$. Then $\lim_{n \to \infty} g_n(H) = 1$ in the strict topology of $M(A)$.
\end{lemma}
\begin{proof} Let $b \in A$. It suffices to show that $\lim_{n \to \infty} \left\| b - g_n(H)b\right\| = 0$. Since 
$$
 \left\| b - g_n(H)b\right\|^2 = \left\|b^*(1-g_n(H))^2b\right\| \leq \left\|b^*(1-g_n(H))b\right\|
$$
it suffices to show that $\lim_{n \to \infty} \left\|b^*(1-g_n(H))b\right\| =0$. To this end let $\omega \in B^*_+, \ \|\omega\| \leq 1$. It follows from Kadison's representation theorem, Theorem 3.10.3 in \cite{Pe}, and Dini's theorem that it suffices to show that
\begin{equation}\label{24-01-22ax}
\lim_{n \to \infty} \omega( b^*(1-g_n(H))b) = 0 .
\end{equation}
Let $(\mathcal H_\omega, \pi_\omega, \xi_\omega)$ be the GNS representation of $B$ defined from $\omega$. Being non-degenerate $\pi_\omega$ extends to a $*$-homomorphism $\overline{\pi}_\omega : M(B) \to B(\mathcal H_\omega)$ such that $\overline{\pi}_\omega(m)\pi_\omega(b) = \pi_\omega(mb)$ for all $m \in M(B)$ and $b \in B$, cf. Lemma \ref{26-09-22} in Appendix \ref{multipliers}. Hence the unitary one-parameter group $\overline{\pi}_\omega(u_t), t \in \mathbb R$, is strongly continuous. By Stone's theorem there is a self-adjoint operator $H_\omega$ on $\mathcal H_\omega$ such that
$$
\overline{\pi}_\omega(u_t) = e^{ i t H_\omega} \ \ \forall t \in \mathbb R.
$$
Since $\widehat{f}(H_\omega) = \int_\mathbb R f(t) \overline{\pi}_\omega (u_t) \ \mathrm d t = \overline{\pi}_\omega\left( \int_\mathbb R f(t) u_t \ \mathrm d t\right) = \overline{\pi}_\omega(\widehat{f}(H))$ when $f \in C_c(\mathbb R)$ it follows by continuity that $g(H_\omega) = \overline{\pi}_\omega(g(H))$ for all $g\in C_0(\mathbb R)$, using Lemma \ref{11-10-23}. In particular, $g_n(H_\omega) = \overline{\pi}_\omega(g_n(H))$ and hence
\begin{equation}\label{24-01-22x}
\omega\left( b^*(1-g_n(H))b\right) = \left<(1- g_n(H_\omega))\pi_\omega(b)\xi_\omega, \pi_\omega(b)\xi_\omega \right> .
\end{equation}
It follows from spectral theory, see e.g. Theorem 5.6.26 in \cite{KR}, that 
$$
\lim_{n \to \infty} g_n(H_\omega) = 1
$$ 
in the strong operator topology and hence \eqref{24-01-22ax} follows from \eqref{24-01-22x}. 
\end{proof}

\begin{lemma}\label{29-06-22cx} Let $\phi$ be a $\beta$-KMS weight for the flow $\sigma$ on $A$. Let $u \in M(A)$ be a unitary such that $\sigma_t(u) = u$ for all $t \in \mathbb R$. Then $\phi = \phi \circ \Ad u$.
\end{lemma}
\begin{proof} Let $a \in \mathcal A_\sigma$. Then $au^* \in \mathcal A_\sigma$ and $\sigma_z(au^*) = \sigma_z(a)u^*$ for all $z \in \mathbb C$. Hence
\begin{align*}
&\phi(ua^*au^*) = \phi(\sigma_{-i \frac{\beta}{2}}(au^*)\sigma_{-i \frac{\beta}{2}}(au^*)^*) \\
& = \phi(\sigma_{-i \frac{\beta}{2}}(a)u^*u\sigma_{-i \frac{\beta}{2}}(a)^*) = \phi(\sigma_{-i \frac{\beta}{2}}(a)\sigma_{-i \frac{\beta}{2}}(a)^*) = \phi(a^*a) .
\end{align*}
It follows therefore from Lemma \ref{01-03-22b} that $\phi = \phi \circ \Ad u$.
\end{proof}

\begin{lemma}\label{29-06-22xx} $\psi_\beta$ is a $\beta$-KMS weight for $\sigma'$.
\end{lemma}
\begin{proof} Let $0 \leq f \leq g$ in $C_c(\mathbb R)$ and let $b \in \mathcal A_\sigma$. Let $g_k : \mathbb R \to [0,1]$ be a continuous function such that $g_k(t) = 1$ when $|t| \leq k$ and $g_k(t) = 0$ when $|t| \geq k+1$. Then
\begin{align*}
&\psi\left(e^{\frac{-\beta H}{2}} f(H) b^*b f(H) e^{-\frac{\beta H}{2}}\right)  = \lim_{k \to \infty} \psi\left(e^{-\frac{\beta H}{2}} f(H) {b^*}g_k(H)^2{b} f(H) e^{\frac{-\beta H}{2}}\right)
\end{align*}
by Lemma \ref{23-01-22ax} and the lower semi-continuity of $\psi$. By Lemma \ref{23-01-22} $g_k(H)$ and $f(H)e^{-\frac{\beta H}{2}}$ are fixed by $\sigma$ and since $b$ is entire analytic for $\sigma$ by assumption, it follows that $g_k(H){b} f(H) e^{\frac{-\beta H}{2}}$ is entire analytic for $\sigma$, and
$$
\sigma_{-i \frac{\beta}{2}}(g_k(H){b} f(H) e^{\frac{-\beta H}{2}}) = g_k(H)\sigma_{-i \frac{\beta}{2}}({b}) f(H) e^{\frac{-\beta H}{2}}.
$$ 
It follows therefore that
\begin{align*}
& \psi\left(e^{-\frac{\beta H}{2}} f(H) {b^*}g_k(H)^2{b} f(H) e^{\frac{-\beta H}{2}}\right) \\
& =  \psi\left(\sigma_{-i \frac{\beta}{2}}(g_k(H){b} f(H) e^{\frac{-\beta H}{2}}) \sigma_{-i \frac{\beta}{2}} (g_k(H){b} f(H) e^{\frac{-\beta H}{2}} )^*\right) \\
& = \psi\left(g_k(H)\sigma_{-i \frac{\beta}{2}}({b}) f(H) e^{\frac{-\beta H}{2}}e^{-\frac{\beta H}{2}} f(H) \sigma_{-i\frac{\beta}{2}}(b)g_k(H)\right) \\
& \leq \psi\left(g_k(H)\sigma_{-i \frac{\beta}{2}}({b}) g(H) e^{\frac{-\beta H}{2}}e^{-\frac{\beta H}{2}} g(H) \sigma_{-i\frac{\beta}{2}}(b)g_k(H)\right) \\
& = \psi\left(e^{-\frac{\beta H}{2}} g(H) {b^*}g_k(H)^2{b} g(H) e^{\frac{-\beta H}{2}}\right) ,
\end{align*}
where the last equality follows from the first two equalities applied with $f$ replaced by $g$. Letting $k$ go to infinity and by using Lemma \ref{23-01-22ax} and the lower semi-continuity of $\psi$ again, it follows that
\begin{equation}\label{17-08-22}
\psi\left(e^{\frac{-\beta H}{2}} f(H) b^*b f(H) e^{-\frac{\beta H}{2}}\right)  \leq \psi\left(e^{\frac{-\beta H}{2}} g(H) b^*b g(H) e^{-\frac{\beta H}{2}}\right)  .
\end{equation}
Let $b \in A^+$ be arbitrary. Since $\lim_{n \to \infty} R_n(\sqrt{b}) = \sqrt{b}$ by Lemma \ref{24-11-21} it follows from the lower semi-continuity of $\psi$ that
\begin{align*}
& \psi\left(e^{\frac{-\beta H}{2}} f(H) b f(H) e^{-\frac{\beta H}{2}}\right)\\
& \leq \liminf_n  \psi\left(e^{\frac{-\beta H}{2}} f(H) R_n(\sqrt{b})R_n(\sqrt{b}) f(H) e^{-\frac{\beta H}{2}}\right) ,
\end{align*}
and then from \eqref{17-08-22} that
\begin{equation}\label{17-08-22a}
\begin{split}
& \psi\left(e^{\frac{-\beta H}{2}} f(H) b f(H) e^{-\frac{\beta H}{2}}\right)\\
& \leq \liminf_n  \psi\left(e^{\frac{-\beta H}{2}} g(H) R_n(\sqrt{b})R_n(\sqrt{b}) g(H) e^{-\frac{\beta H}{2}}\right) ,
\end{split}
\end{equation}
since $R_n(\sqrt{b}) \in \mathcal A_\sigma$ by Lemma \ref{24-11-21}. By Kadisons inequality, Proposition 3.2.4 in \cite{BR},
\begin{align*}
&  \psi\left(e^{\frac{-\beta H}{2}} g(H) R_n(\sqrt{b})R_n(\sqrt{b}) g(H) e^{-\frac{\beta H}{2}}\right) \\
& \leq \psi\left(e^{\frac{-\beta H}{2}} g(H) R_n({b}) g(H) e^{-\frac{\beta H}{2}}\right) ,
\end{align*}
while the definition of $R_n$ implies that
$$
e^{\frac{-\beta H}{2}} g(H) R_n({b}) g(H) e^{-\frac{\beta H}{2}} = R_n\left( e^{\frac{-\beta H}{2}} g(H) {b} g(H) e^{-\frac{\beta H}{2}}\right) 
$$
since $e^{\frac{-\beta H}{2}} g(H) = g(H)e^{\frac{-\beta H}{2}}$ is fixed by $\sigma$ by Lemma \ref{23-01-22}. By use of Lemma \ref{02-12-21ax} we find therefore that
$$
\psi\left(e^{\frac{-\beta H}{2}} g(H) R_n({b}) g(H) e^{-\frac{\beta H}{2}}\right) = \psi\left(e^{\frac{-\beta H}{2}} g(H) {b} g(H) e^{-\frac{\beta H}{2}}\right) ,
$$
and conclude therefore now from \eqref{17-08-22a} that
\begin{equation*}
\psi\left(e^{\frac{-\beta H}{2}} f(H) b f(H) e^{-\frac{\beta H}{2}}\right)  \leq \psi\left(e^{\frac{-\beta H}{2}} g(H) b g(H) e^{-\frac{\beta H}{2}}\right)  .
\end{equation*}
This shows that the expression $\psi\left(e^{\frac{-\beta H}{2}} f(H) b f(H) e^{\frac{-\beta H}{2}}\right)$ increases with $f$, implying that
$$
\psi_\beta(b) = \lim_{k \to \infty}  \psi\left(e^{-\frac{\beta H}{2}} g_k(H) b g_k(H) e^{-\frac{\beta H}{2}}\right)
$$
for all $b \in A^+$. In particular, it follows that $\psi_\beta$ is additive and hence a weight. 
To see that $\psi_\beta$ is not zero, let $b \in A^+$ be an element such that $\psi(b) > 0$. Then $x_k .:= e^{\frac{\beta H}{2}}g_k(H)\sqrt{b} \in A$ by Lemma \ref{23-01-22} and 
$$
\psi_\beta(x_kx_k^*) = \lim_{l \to \infty} \psi\left(g_l(H)g_k(H)bg_k(H)g_l(H)\right) = \psi\left(g_k(H)bg_k(H)\right).
$$
Since $\lim_{k \to \infty} g_k(H)bg_k(H)= b$ by Lemma \ref{23-01-22ax}, the lower semi-continuity of $\psi$ ensures that $\psi_\beta(x_kx_k^*) =  \psi\left(g_k(H)bg_k(H)\right) > 0$ for $k$ large enough. Hence $\psi_\beta \neq 0$. To show that $\psi_\beta$ is densely defined let $b \in \mathcal M^\sigma_\psi$ and set $b_n := g_n(H){b}$. Then $g_k(H)b_n = b_n$ when $k > n$ and it follows from Lemma \ref{23-01-22} that
$$
\sigma_{-i\frac{\beta}{2}} (b_n^* e^{-\frac{\beta H}{2}}) = \sigma_{-i\frac{\beta}{2}} (b^*) g_n(H)e^{-\frac{\beta H}{2}}.
$$
Hence
\begin{align*}
& \psi\left(e^{-\frac{\beta H}{2}}g_k(H) b_nb_n^*g_k(H) e^{-\frac{\beta H}{2}}\right)   = \psi\left(e^{-\frac{\beta H}{2}} b_nb_n^* e^{-\frac{\beta H}{2}}\right) \\ 
& = \psi\left( \sigma_{-i\frac{\beta}{2}} (b_n^* e^{-\frac{\beta H}{2}}) \sigma_{-i\frac{\beta}{2}} (b_n^* e^{-\frac{\beta H}{2}})^*\right)\\
& = \psi(\sigma_{-i\frac{\beta}{2}}(b^*) g_n(H)^2e^{\beta H} \sigma_{-i\frac{\beta}{2}}(b^*)^* ) \\
& \leq \left\| g_n(H)^2e^{-\beta H}\right\|\psi(\sigma_{-i\frac{\beta}{2}}(b^*) \sigma_{-i\frac{\beta}{2}}(b^*)^* ) 
\end{align*}
when $k > n$. It follows that $\psi_\beta(b_nb_n^*) < \infty$. Since $\lim_{n \to \infty} b_n = b$ by Lemma \ref{23-01-22ax} and $\mathcal M^\sigma_\psi$ is dense in $A$, it follows that $\psi_\beta$ is densely defined. That $\psi_\beta$ is $\sigma'$-invariant follows from the $\sigma$-invariance of $\psi$ by use of Lemma \ref{29-06-22cx} and Lemma \ref{23-01-22}.

Finally we chack that $\psi_\beta$ satifies condition (2) in Kustermans' theorem, Theorem \ref{24-11-21d}. Let $a \in \mathcal A_{\sigma'}$ and set $a_k = g_k(H)ag_k(H)$. To see that $a_k \in \mathcal A_\sigma$ note that 
$$
\sigma_t(a_k) = g_k(H)\sigma_t(a)g_k(H) =   g_k(H)u_{-t}\sigma'_t(a)u_tg_k(H),
$$
so that 
$$
z \mapsto g_k(H)e^{-izH}\sigma'_z(a)g_k(H)e^{iz H} = e^{-izH}\sigma'_z(a_k)e^{izH}
$$
is an entire holomorphic extension of $t \mapsto \sigma_t(a)$. In particular,
$$
\sigma_{-i \frac{\beta}{2}}(a_k) = e^{-\frac{\beta}{2}H} \sigma'_{-\frac{\beta}{2}}(a_k)e^{\frac{\beta}{2}H} .
$$ 
This is used in the following calculation: 
\begin{align*}
&\psi_\beta \left( a_k^* a_k\right)\\
& = \lim_{n \to \infty} \psi\left( e^{-\frac{\beta H}{2}}g_n(H)a_k^*a_k g_n(H) e^{-\frac{\beta H}{2}}\right)\\
& =  \psi\left( e^{-\frac{\beta H}{2}}a_k^*a_k e^{-\frac{\beta H}{2}}\right) \\
& =  \psi\left( \sigma_{-i\frac{\beta}{2}}(a_k) e^{-{\beta H}}\sigma_{-i\frac{\beta}{2}}(a_k)^*\right) \\
& = \psi\left((e^{-\frac{\beta H}{2}}\sigma'_{-i \frac{\beta}{2}}(a_k) e^{\frac{\beta H}{2}})e^{-\beta H} (e^{-\frac{\beta H}{2}}\sigma'_{-i \frac{\beta}{2}}(a_k) e^{\frac{\beta H}{2}})^*\right)\\
& = \psi\left(e^{-\frac{\beta H}{2}}  \sigma'_{-i\frac{\beta}{2}}( a_k)  \sigma'_{-i\frac{\beta}{2}}( a_k)^* e^{-\frac{\beta H}{2}} \right) \\ 
& =  \psi\left(e^{-\frac{\beta H}{2}} g_k(H) \sigma'_{-i\frac{\beta}{2}}( a) g_k(H)^2 \sigma'_{-i\frac{\beta}{2}}( a)^* g_k(H)e^{-\frac{\beta H}{2}} \right) \\ 
& \leq  \psi\left(e^{-\frac{\beta H}{2}} g_k(H) \sigma'_{-i\frac{\beta}{2}}( a) \sigma'_{-i\frac{\beta}{2}}( a)^* g_k(H)e^{-\frac{\beta H}{2}} \right) .
\end{align*}  
Since $\lim_{k \to \infty} a_k = a$ by Lemma \ref{23-01-22ax} it follows from this estimate and the lower semi-continuity of $\psi_\beta$ that 
\begin{align*}
& \psi_\beta(a^*a) \leq \liminf_{k \to \infty} \psi_\beta(a_k^*a_k) \\
& \leq \lim_{k \to \infty} \psi\left(e^{-\frac{\beta H}{2}} g_k(H) \sigma'_{-i\frac{\beta}{2}}( a) \sigma'_{-i\frac{\beta}{2}}( a)^* g_k(H)e^{-\frac{\beta H}{2}} \right) \\
& = \psi_\beta\left(  \sigma'_{-i\frac{\beta}{2}}( a) \sigma'_{-i\frac{\beta}{2}}( a)^*\right) .
\end{align*}
Since $a \in \mathcal A_{\sigma'}$ was arbitrary, the same estimate holds with $a$ replaced by $\sigma'_{i \frac{\beta}{2}}(a^*)$ and we conclude therefore that 
$$
\psi_\beta(a^*a) =  \psi_\beta\left(  \sigma'_{-i\frac{\beta}{2}}( a) \sigma'_{-i\frac{\beta}{2}}( a)^*\right)
$$
for all $a \in \mathcal A_{\sigma'}$. It follows then from Kustermans' theorem, Theorem \ref{24-11-21d}, that $\psi_{\beta}$ is a $\beta$-KMS weight for $\sigma'$.
\end{proof}

\begin{thm}\label{29-06-22d}
Let $\sigma$ be flow on $A$ and $u = (u_t)_{t \in \mathbb R}$ a strictly continuous unitary representation of $\mathbb R$ by multipliers of $M(A)$ such that $\sigma_t(u_s) = u_s$ for all $s,t$. Consider the flow $\sigma'$ defined by
$$
\sigma'_t := \Ad u_t \circ \sigma_t.
$$
The map $\psi \mapsto \psi_\beta$ given by \eqref{29-06-22ax} is a bijection from $\KMS(\sigma,\beta)$ onto $\KMS(\sigma',\beta)$ for all $\beta \in \mathbb R$.
\end{thm}
\begin{proof} Since $\sigma_t = \Ad u_t^* \circ \sigma'_t$ and $\sigma'_t(u_t^*) = u_t^*$ we get in the same way a map $\phi \mapsto \phi^\beta$ from $\KMS(\sigma',\beta)$ to $\KMS(\sigma,\beta)$ defined such that
$$
\phi^\beta(a)  := \sup \left\{\phi\left(e^{\frac{\beta H}{2}} f(-H) a f(-H) e^{\frac{\beta H}{2}}\right) : \ f \in C_c(\mathbb R), \ 0 \leq f \leq 1 \right\} .
$$
Let $\psi \in \KMS(\sigma,\beta)$ and $a \in \mathcal A_\sigma$. For $f,g \in C_c(\mathbb R)$ with $0 \leq f \leq 1$ and $0 \leq g \leq 1$ we have
$$
\psi(f(-H)g(H)a^*af(-H)g(H)) = \psi(\sigma_{-i\frac{\beta}{2}}(a)f^2(-H)g^2(H)\sigma_{-i\frac{\beta}{2}}(a)^*).
$$
It follows therefore from an application of Lemma \ref{23-01-22ax} that
\begin{align*}
&\left(\psi_\beta\right)^\beta(a^*a) =  \sup \left\{ \psi(\sigma_{-i\frac{\beta}{2}}(a)f^2(-H)g^2(H)\sigma_{-i\frac{\beta}{2}}(a)^*) : \ 0 \leq f,g \leq 1 \right\} \\
& = \psi\left(\sigma_{-i\frac{\beta}{2}}(a)\sigma_{-i\frac{\beta}{2}}(a)^*\right) = \psi(a^*a) .
\end{align*}
Then Lemma \ref{01-03-22b} shows that $\left(\psi_\beta\right)^\beta = \psi$. Similarly, or by symmetry, $\left(\phi^\beta\right)_\beta = \phi$ for all $\phi \in \KMS(\sigma',\beta)$.
\end{proof}

\begin{notes}\label{06-06-22d} When $A$ is unital, Theorem \ref{06-06-22b} follows from a result by Kishimoto, cf. Proposition 2.1 in \cite{Ki1}. For KMS weights both Theorem \ref{06-06-22b} and Theorem \ref{29-06-22d} are new.
\end{notes}

\section{Inner flows}
 A flow $\sigma$ on a $C^*$-algebra $A$ is \emph{inner} when there is a unitary representation $u = (u_t)_{t \in \mathbb R}$ of $\mathbb R$ by unitaries in the multiplier algebra $M(A)$ of $A$ such that $\mathbb R \ni t \mapsto u_t$ is continuous with respect to the strict topology of $M(A)$ and
$$
\sigma_t(a) = u_tau_{-t} 
$$ 
for all $a \in A$ and all $t \in \mathbb R$. Before studying general inner flows we consider first a special case.

\subsection{Inner flows with a bounded generator}\label{veryinner}

Let $A$ be a $C^*$-algebra and $h = h^* \in M(A)$ a self-adjoint element of the multiplier algebra $M(A)$ of $A$. In this section we consider the corresponding inner flow
\begin{equation}\label{04-09-22b}
\sigma_t(a) = \Ad e^{i t h}(a) = e^{it h} a e^{-ith} .
\end{equation}

\begin{lemma}\label{27-09-23a} $A = \mathcal A_\sigma$, $e^{zh} \in M(A)$ and $\sigma_z(a) = e^{zh}a e^{-zh}$ for all $a \in A$ and $z \in \mathbb C$.
\end{lemma}
\begin{proof} Since $\|h\| < \infty$ the sum $e^{zh} =\sum_{n=0}^\infty \frac{h^n}{n!} z^n$ converges in norm and hence $e^{zh} \in M(A)$ since $M(A)$ is a $C^*$-algebra. Furthermore, it follows that
$$
z \mapsto e^{zh}ae^{-zh}
$$
is entire holomorphic for all $a \in A$.
\end{proof}

Let $\beta \in \mathbb R $ and let $\psi$ be a $\beta$-KMS weight for $\sigma$. It follows from Lemma \ref{27-09-23a} that
\begin{align*}
&\psi(e^{\frac{\beta h}{2}}a^*ae^{\frac{\beta h}{2}})  = \psi(\sigma_{-i \frac{\beta}{2}}(ae^{\frac{\beta h}{2}}) \sigma_{-i \frac{\beta}{2}}(ae^{\frac{\beta h}{2}})^*) \\
&= \psi(e^{\frac{\beta h}{2}}a ( e^{\frac{\beta h}{2}}a)^* ) = \psi(e^{\frac{\beta h}{2}}aa^*e^{\frac{\beta h}{2}}) 
\end{align*}
for all $a \in A$. This shows that 
$$
\tau_\psi (a) := \psi(e^{\frac{\beta h}{2}}ae^{\frac{\beta h}{2}})
$$
defines a lower semi-continuous trace $\tau_\psi : A^+ \to [0,\infty]$ such that
\begin{equation*}\label{05-09-22}
 \psi(a) = \tau_\psi( e^{-\frac{\beta h}{2}}ae^{-\frac{\beta h}{2}})
\end{equation*}
for all $a \in A^+$. Conversely, given a lower semi-continuous trace $\tau$ on $A$ we can define a $\beta$-KMS weight $\psi_\tau : A^+ \to [0,\infty]$ such that
\begin{equation}\label{05-09-22b}
\psi_\tau(a) := \tau( e^{-\frac{\beta h}{2}}ae^{-\frac{\beta h}{2}}) ,
\end{equation}
as the reader can easily verify. The two maps $\tau \mapsto \psi_\tau$ and $\psi \mapsto \tau_\psi$ are clearly each others inverses and we get therefore the following.

\begin{prop}\label{05-09-22a} For each $\beta \in \mathbb R$ the map $\tau \mapsto \psi_\tau$ defined by \eqref{05-09-22b} is a bijection from the set of lower semi-continuous traces on $A$ onto the set of $\beta$-KMS weights for the flow \eqref{04-09-22b}.
\end{prop}

\begin{prop}\label{04-09-22e}  For each $\beta \in \mathbb R$ the map $\tau \mapsto \psi_\tau$ defined by \eqref{05-09-22b} is a bijection from the set of bounded traces on $A$ onto the set of bounded $\beta$-KMS weights for the flow \eqref{04-09-22b}.
\end{prop}
\begin{proof} The maps $A \to A$, given by $a \mapsto e^{- \frac{\beta h}{2}}ae^{- \frac{\beta h}{2}}$ and $a \mapsto e^{ \frac{\beta h}{2}}ae^{\frac{\beta h}{2}}$, are both bounded and hence the map $\tau \mapsto \psi_\tau$ and its inverse both take bounded functionals to bounded functionals. 
\end{proof}


\begin{cor}\label{05-09-22c} Assume that $A$ is unital. For each $\beta \in \mathbb R$ the map $\omega \mapsto \psi_\omega$ given by
$$
\psi_\omega(a) := \frac{\omega(e^{-\beta h}a)}{\omega(e^{-\beta h})}, \ \ \forall a \in A,
$$
is a bijection from the set of tracial states $\omega$ of $A$ onto the set of $\beta$-KMS states for the flow \eqref{04-09-22b}. 
\end{cor}

\begin{example}\label{31-07-23} \rm{(Flows on a matrix algebra.) Let $M_n(\mathbb C)$ be the $C^*$-algebra of complex $n$ by $n$ matrices. For every self-adjoint element $h= h^*$ in $M_n(\mathbb C)$ there is a flow $\sigma$ on $M_n(\mathbb C)$ given by
\begin{equation}\label{31-07-23a}
\sigma_t(m) = e^{ith}me^{-ith} .
\end{equation}
It follows from Corollary \ref{05-09-22c} that there is a unique $\beta$-KMS state $\omega_\beta$ for $\sigma$ for every $\beta \in \mathbb R$ and that 
$$
\omega_\beta(m) = \frac{\Tr_n(e^{-\beta h}m)}{\Tr_n(e^{-\beta h})} \ \ \ \ \forall m \in M_n(\mathbb C),
$$
where $\Tr_n$ is the usual trace on $M_n(\mathbb C)$.}

\rm{It is straightforward to adopt the arguments from the proof of Theorem \ref{12-04-22} in Appendix \ref{compact operators} to show that every flow on $M_n(\mathbb C)$ is of the form \eqref{31-07-23a} for some $h=h^* \in M_n(\mathbb C)$.}
\end{example}

\begin{remark}\label{31-07-23b} \rm{The map $\omega \mapsto \psi_\omega$ in Corollary \ref{05-09-22c} is continuous for the weak* topology and hence a homeomorphism since a continuous bijection between compact Hausdorff spaces is a homeomorphism. However, the map is not affine. It fact, in general there may not be an affine homeorphism between the tracial state space of $A$ and the compact convex set of $\beta$-KMS states in the setting of Corollary \ref{05-09-22c}. The reader can see what can go wrong in Remarks 2.3 and 2.4 of \cite{Ki1}.}
\end{remark}

\subsection{General inner flows}\label{geninner}

By using Theorem \ref{29-06-22d} we can partly extend the results of the last subsection to more general inner flows, as follows. Let $u = (u_t)_{t \in \mathbb R}$ be a strictly continuous unitary representation of $\mathbb R$ by unitaries in the multiplier algebra $M(A)$ of $A$ and consider the corresponding inner flow
$$
\sigma'_t(a) = u_tau_{-t} .
$$ 
Then $\sigma'$ is related to the trivial flow $\sigma_t = \id_A$ in the same way as $\sigma'$ was related to the $\sigma$ of Section \ref{05-09-22g}. Since the $\beta$-KMS weights for the trivial flow are the lower semi-continuous traces on $A$ we get the following directly from Theorem \ref{29-06-22d} and its proof:

\begin{thm}\label{17-08-22c} Let $\alpha$ be an inner flow on $A$ such that $\alpha_t = \Ad e^{itH}$, where $u_t = e^{itH}$ is a strictly continuous unitary representation of $\mathbb R$ by multipliers in $M(A)$. For any $\beta \in \mathbb R$ the map $\tau \mapsto \tau_\beta$, where
$$
\tau_\beta(a) := \sup \left\{\tau\left(e^{\frac{-\beta H}{2}} f(H) a f(H) e^{-\frac{\beta H}{2}}\right) : \ f \in C_c(\mathbb R), \ 0 \leq f \leq 1 \right\} ,
$$
is a bijection from the set of lower semi-continuous traces $\tau$ on $A$ onto the set of $\beta$-KMS weights for $\alpha$. The inverse of the map $\tau \mapsto \tau_\beta$ is the map $\psi \mapsto \tau_\psi$ which takes a $\beta$-KMS weight $\psi$ for $\alpha$ to the lower semi-continuous trace $\tau_\psi : A^+ \to [0,\infty]$ given by
$$
\tau_\psi(a) := \sup \left\{\psi\left(e^{\frac{\beta H}{2}} f(H) a f(H) e^{\frac{\beta H}{2}}\right) : \ f \in C_c(\mathbb R), \ 0 \leq f \leq 1 \right\}.
$$
\end{thm}

\begin{notes}\label{27-09-23d} Corollary \ref{05-09-22c} is a folklore fact, but Theorem \ref{17-08-22c} seems to be new. It generalizes Theorem \ref{02-01-22a}.
\end{notes}

\subsection{KMS-states for approximately inner flows}

A flow $\sigma$ on a $C^*$-algebra $A$ is \emph{approximately inner} when there is a sequence $\{h_n\}$ of self-adjoint elements in $A$ such that for all $a \in A$,
$$
\lim_{n \to \infty} e^{ith_n} a e^{-ith_n} = \sigma_t(a)
$$ 
uniformly for $t$ in any compact subset of $\mathbb R$.

\begin{thm}\label{01-08-23d} Let $A$ be a unital $C^*$-algebra and $\sigma$ an approximately inner flow on $A$. If $A$ has a trace state there is a $\beta$-KMS state for $\sigma$ for all $\beta \in \mathbb R$.
\end{thm}
\begin{proof} Let $\tau$ be a trace state on $A$ and let $\beta \in \mathbb R$. For each $n \in \mathbb N$ set
$$
\omega_n(a) := \frac{\tau(e^{-\beta h_n}a)}{\tau(e^{-\beta h_n})} .
$$
Each $\omega_n$ is a state on $A$ and since the state space of $A$ is compact in the weak* topology, the set $\{\omega_n : n \in \mathbb N\}$ has a weak* condensation point $\omega$. We aim to show that $\omega$ is a $\beta$-KMS state for $\sigma$. To show this let $a \in \mathcal A_\sigma$ and $b \in A$. By Theorem \ref{21-11-23b} it suffices to show that
\begin{equation}\label{01-08-23e}
\omega(ab) = \omega(b\sigma_{i\beta}(a)) .
\end{equation}
The $C^*$-algebra $C$ generated by 
$$
\left\{ \sigma_t(x): \ x \in \cup_{n=1}^\infty \{a,b,h_n\}\right\}
$$ 
is $\sigma$-invariant and separable, and there is therefore a sequence $m_1 < m_2 < m_3 < \cdots$ in $\mathbb N$ such that $\lim_{n \to \infty} \omega_{m_n}(c) =c$ for all $c \in C$. To simplify notation, set $\omega'_n := \omega_{m_n}$.

Let $\sigma^k$ denote the flow on $C$ given by $\sigma^k_t(c) = e^{ith_k} ce^{-i th_k}$. Choose $L_n > 0$ such that
\begin{equation}\label{02-10-23d} 
  \sqrt{\frac{n}{\pi}}\int_{|s| \geq L_n} \left| e^{-n(s-i\beta)^2}\right|  \ \mathrm d s \leq \frac{1}{n},
\end{equation}
and
\begin{equation}\label{02-10-23e} 
  \sqrt{\frac{n}{\pi}}\int_{|s| \geq L_n}  e^{-ns^2}  \ \mathrm d s \leq \frac{1}{n},
\end{equation}
and then $k_n \in \mathbb N$ such that
\begin{equation}\label{02-10-23f} 
 \sup_{|t| \leq L_n} \left\| \sigma^{k_n}_t(a) - \sigma_t(a)\right\| \leq \frac{M}{n},
\end{equation}
where
$$
M := \min \left\{  \left(\sqrt{\frac{n}{\pi}}\int_{|s| \leq L_n} \left|e^{-n(s-i\beta)^2}\right|  \ \mathrm d s\right)^{-1}, \  \left(\sqrt{\frac{n}{\pi}}\int_{|s| \leq L_n} e^{-n s^2}  \ \mathrm d s\right)^{-1}\right\}.
$$
By increasing $k_n$ we can arrange that
\begin{equation}\label{01-08-23}
\left|\omega'_{k_n}(R_n(a)b) - \omega(R_n(a)b)\right| \leq \frac{1}{n}
\end{equation}
and
\begin{equation}\label{01-08-23a}
\left|\omega'_{k_n}(b\sigma_{i\beta}(R_n(a))) - \omega(b\sigma_{i\beta}(R_n(a)))\right| \leq \frac{1}{n} .
\end{equation}
Let $R^{(k_n)}_n$ be the smoothing operators for flow $\sigma^{k_n}$. Then \eqref{02-10-23e} and \eqref{02-10-23f} give the estimate
\begin{equation}\label{02-10-23}
\begin{split}
& \left\|R^{(k_n)}_n(a) - R_n(a)\right\| \leq \sqrt{\frac{n}{\pi}} \int_\mathbb R e^{-nt^2} \left\| \sigma^{k_n}_t(a) - \sigma_t(a)\right\|  \ \mathrm d t \\
& \\
& \leq 2\|a\|\sqrt{\frac{n}{\pi}}\int_{|s| \geq L_n}  e^{-ns^2}  \ \mathrm d s   + \sqrt{\frac{n}{\pi}} \int_{|s| \leq L_n} e^{-nt^2} \left\| \sigma^{k_n}_t(a) - \sigma_t(a)\right\|  \ \mathrm d t\\
& \\
& \leq \frac{2\|a\|}{n} + \frac{M}{n} \sqrt{\frac{n}{\pi}} \int_{|s| \leq L_n} e^{-nt^2} \ \mathrm d t  \leq \frac{2\|a\| +1}{n} .
\end{split}
\end{equation} 
It follows from Lemma \ref{17-11-21i} that
\begin{align*}
& \sigma_{i \beta}(R_n(a)) = \sqrt{\frac{n}{\pi}}\int_\mathbb R e^{-n(s-i\beta)^2} \sigma_s(a) \ \mathrm d s 
\end{align*}
and
\begin{align*}
& \sigma^{k_n}_{i \beta}(R^{(k_n)}_n(a)) = \sqrt{\frac{n}{\pi}}\int_\mathbb R e^{-n(s-i\beta)^2} \sigma^{k_n}_s(a) \ \mathrm d s .
\end{align*}
Hence estimates similar to the preceding, using \eqref{02-10-23d} and \eqref{02-10-23f}, show that
\begin{equation}\label{02-10-23a}
\begin{split}
& \left\| \sigma_{i \beta}(R_n(a)) - \sigma^{k_n}_{i\beta}(R^{(k_n)}_n(a)) \right\| \\
&\leq \sqrt{\frac{n}{\pi}}\int_\mathbb R \left|e^{-n(s-i\beta)^2}\right| \left\|\sigma_s(a) - \sigma^{k_n}_s(a)\right\| \ \mathrm d s  \leq   \frac{2\|a\| +1}{n}.
\end{split}
\end{equation}
Then 
\begin{align*}
& \left| \omega'_{k_n}(b\sigma_{i\beta}(R^{(k_n)}_n(a))) - \omega(b\sigma_{i \beta}(R_n(a)))\right| \\
&\leq \|b\|\frac{2\|a\| +1}{n} + \left|\omega'_{k_n}(b\sigma_{i \beta} (R_n(a))) - \omega(b\sigma_{i \beta}(R_n(a)))\right| \ \ \ \ \ \ \text{(by \eqref{02-10-23a})} \\
& \leq  \|b\|\frac{2\|a\| +1}{n} + \frac{1}{n} \ \ \ \ \ \ \text{(by \eqref{01-08-23a})},
\end{align*}
and by using \eqref{02-10-23} and \eqref{01-08-23} in the same way we find also that
\begin{align*}
&\left|\omega'_{k_n}(R^{(k_n)}_n(a)b) - \omega(R_n(a)b) \right| \leq  \|b\|\frac{2\|a\| +1}{n} + \frac{1}{n} .
\end{align*}
Since $\omega'_{k_n}(b\sigma_{i\beta}(R^{(k_n)}_n(a))) = \omega'_{k_n}(R^{(k_n)}_n(a)b))$ this implies that
\begin{equation}\label{02-10-23c}
\left|\omega(b\sigma_{i \beta}(R_n(a)) -\omega(R_n(a)b)\right| \leq  2\|b\|\frac{2\|a\| +1}{n} + \frac{2}{n} .
\end{equation}
It follows from Lemma \ref{17-11-21i} and Lemma \ref{24-11-21} that
$$
\lim_{n \to \infty} b\sigma_{i\beta}(R_n(a)) - R_n(a)b = b\sigma_{i\beta}(a) -  ab,
$$
which combined with \eqref{02-10-23c} shows that \eqref{01-08-23e} holds. Hence $\omega$ is a $\beta$-KMS state.
\end{proof}

\begin{notes} Theorem \ref{01-08-23d} was obtained by Powers and Sakai in 1974 and published in the very influential paper \cite{PS} in which they conjectured that all flows on a UHF algebra are approximately inner. The Powers-Sakai conjecture was not refuted before the work of Matui and Sato in \cite{MS}, following a paper by Kishimoto, \cite{Ki2}, in which he gave an example of a flow on a simple unital AF-algebra which is not approximately inner.

We note that Theorem \ref{01-08-23d} also follows from Proposition 5.3.25 in \cite{BR}.

\end{notes}

\section{Tensor products}\label{tensor}

In this section we consider tensor products of $C^*$-algebras with the aim of defining the tensor product of KMS weights. The construction we present will work regardless of which $C^*$-norm we consider, as long as it is 'reasonable' as specified on page 850 of \cite{KR}. In particular, it works both for the maximal and the minimal tensor product. In the construction we shall also consider tensor products of Hilbert spaces and of operators, in addition to certain algebraic tensor product. We use the same notation $\otimes$ for all tensor products in the hope that the meaning will be clear from the context.

Let $\sigma^1$ be a flow on the $C^*$-algebra $A_1$ and $\sigma^2$ a flow on the $C^*$-algebra $A_2$, and let $\psi_1$ be a $\beta$-KMS wight for $\sigma^1$ and $\psi_2$ a $\beta$-KMS weight for $\sigma^2$. We can then consider the tensor product flow $\sigma$ on $A_1 \otimes A_2$, defined by
$$
\sigma_t := \sigma^1_t \otimes \sigma^2_t.
$$
Thus, on simple tensors $\sigma_t(a_1 \otimes a_2) = \sigma^1_t(a_1) \otimes \sigma^2_t(a_2)$.

\begin{lemma}\label{08-06-22c} $\Lambda_{\psi_1} \otimes \Lambda_{\psi_2} : \ \mathcal N_{\psi_1} \otimes \mathcal N_{\psi_2} \to H_{\psi_1} \otimes H_{\psi_2}$ is closable.
\end{lemma}

\begin{proof} Let $(x_i)_{i \in I}$ be a net in $\mathcal N_{\psi_1} \otimes \mathcal N_{\psi_2}$ such that $\lim_{i \to \infty} x_i = 0$ in $A_1\otimes A_2$ and $\lim_{i \to \infty} \Lambda_{\psi_1} \otimes \Lambda_{\psi_2}(x_i ) = w$ in $H_{\psi_1} \otimes H_{\psi_2}$. Write $x_i$ as a finite sum
$$
x_i = \sum_n x_n^i \otimes y^i_n ,
$$
where $x_n^i \in \mathcal N_{\psi_1}$ and $y_n^i \in \mathcal N_{\psi_2}$ for all $n,i$. For each of $k \in \{1,2\}$ we let $(u^k_j)_{j \in I_k}$ and $(\rho^k_j)_{j \in I_k}$ be the net in $\mathcal M^{\sigma^k}_{\psi_k}$ and the corresponding net of operators in $B(H_{\psi_k})$ obtained by applying Lemma \ref{09-02-22x} to $\psi_k$. Then
 \begin{align*}
&\rho^1_j \otimes \rho^2_{j'}(w)  = \lim_{i \to \infty} \rho^1_j \otimes \rho^2_{j'}(\Lambda_{\psi_1} \otimes \Lambda_{\psi_2}(x_i )) \\
& = \lim_{i \to \infty}\sum_n  \rho^1_j \otimes \rho^2_{j'}( \Lambda_{\psi_1}(x_n^i) \otimes \Lambda_{\psi_2}(y_n^i )) \\
& =  \lim_{i \to \infty}\sum_n  \Lambda_{\psi_1}(x_n^iu^1_j) \otimes \Lambda_{\psi_2}(y_n^iu^2_{j'} ) \\
& = \lim_{i \to \infty} (\pi_{\psi_1} \otimes \pi_{\psi_2})(x_i)(\Lambda_{\psi_1}(u^1_j) \otimes \Lambda_{\psi_2}(u^2_{j'}))  = 0 
 \end{align*}
 for all $(j,j') \in I_1\times I_2$.
 Since $\lim_{(j,j') \to \infty} \rho^1_j \otimes \rho^2_{j'}(w) = w$
thanks to property (g) of Lemma \ref{09-02-22x}, it follows that $w =0$.
\end{proof}

Let $\Lambda : D(\Lambda) \to H_{\psi_1} \otimes H_{\psi_2}$ be the closure of $\Lambda_{\psi_1} \otimes \Lambda_{\psi_2}$. The triple
\begin{equation}\label{08-06-22e}
( H_{\psi_1} \otimes H_{\psi_2}, \Lambda,  \pi_{\psi_1} \otimes \pi_{\psi_2})
\end{equation}
is then a GNS representation of $A_1 \otimes A_2$ as defined in Definition \ref{06-02-22a}. 

\begin{lemma}\label{08-06-22d} 
 With respect to the flow $\sigma$ the GNS-triple \eqref{08-06-22e} has the properties (A), (B), (C), (D) and (F) from Section \ref{GNS-KMS}.
\end{lemma}
\begin{proof} The triple has property (B) by definition. It is straightforward to see that the triple has property (A) because $\pi_{\psi_k}, k =1,2$, are both non-degenerate representations. Let $x \in D(\Lambda)$. By definition there is a sequence $\{x_n\}$ in $ \mathcal N_{\psi_1} \otimes \mathcal N_{\psi_2}$ such that $\lim_{n \to \infty} x_n = x$ and $\lim_{n \to \infty} \Lambda(x) = \lim_{n \to \infty} (\Lambda_{\psi_1} \otimes \Lambda_{\psi_2})(x_n)$. Then $\sigma_t(x_n) \in  \mathcal N_{\psi_1} \otimes \mathcal N_{\psi_2}$ and $\lim_{n \to \infty} \sigma_t(x_n) = \sigma_t(x)$. Furthermore, 
\begin{align*}
&\lim_{n \to \infty} \Lambda( \sigma_t(x_n)) = \lim_{n \to \infty} (U^{\psi_1}_t \otimes U^{\psi_2}_t)(\Lambda(x_n))  = (U^{\psi_1}_t \otimes U^{\psi_2}_t)(\Lambda(x)),  
\end{align*}
showing that $\sigma_t(x) \in D(\Lambda)$ and $\Lambda(\sigma_t(x)) = (U^{\psi_1}_t \otimes U^{\psi_2}_t)(\Lambda(x))$. The properties (C) and (D) follow from this. To establish (F) set
$$
S := \mathcal M^{\sigma^1}_{\psi_1} \otimes \mathcal M^{\sigma^1}_{\psi_1} .
$$
Let
\begin{equation}\label{01-12-23}
s = \sum_n x^1_n \otimes x^2_n ,
\end{equation}
be a finite sum with $x^i_n \in \mathcal M^{\sigma^i}_{\psi_i}$. For each $n$ the map $\mathbb C \ni z \mapsto \sigma^1_z(x_n^1)\otimes \sigma^2_z(x_n^2)$ is entire holomorphic, implying that
$$
\sigma_z(x_n^1 \otimes x_n^2) = \sigma^1_z(x^1_n) \otimes \sigma^2_z(x^2_n)
$$
for all $z \in \mathbb C$. It follows that $S \subseteq \mathcal M^\sigma_\Lambda$. If $a_i \in \mathcal N_{\psi^i}$, there is a sequence $\{s^i_n\}$ in $\mathcal M^{\sigma^i}_{\psi_i}$ such that $\lim_{n \to \infty} s_n^i = a_i$ and $\lim_{n \to \infty} \Lambda_{\psi_i}(s^i_n) = \Lambda_{\psi_i}(a_i)$ by Lemma \ref{07-12-21c}. Then $\{s^1_n \otimes s^2_n\}$ is a sequence in $S$ such that $\lim_{n \to \infty} s^1_n \otimes s^2_n = a_1 \otimes a_2$ and $\lim_{n \to \infty} \Lambda(s^1_n \otimes s^2_n) = \Lambda(a_1 \otimes a_2)$. By linearity, and since $\Lambda$ is defined as the closure of $\Lambda_{\psi_1} \otimes \Lambda_{\psi_2}$, it follows that for every $x \in D(\Lambda)$ there is a sequence $\{s_n\}$ in $S$ such that $\Lambda(x) = \lim_{n \to \infty} \Lambda(s_n)$. In particular, the triple has the property from the first item in condition (F) of Lemma \ref{09-02-22d}. In addition, when $s \in S$ is given as in \eqref{01-12-23}, 
$$
\sigma_{-i \frac{\beta}{2}}(s)^* = \sum_n \sigma^1_{-i \frac{\beta}{2}}(x^1_n)^* \otimes \sigma^2_{-i \frac{\beta}{2}}(x^2_n)^* ,
$$
and hence
\begin{align*}
&\left\|\Lambda\left(\sigma_{-i \frac{\beta}{2}}(s)^*\right)\right\| = \left\| \sum_n \Lambda_{\psi_1}\left(\sigma^1_{-i \frac{\beta}{2}}(x^1_n)^*\right) \otimes \Lambda_{\psi_2}\left(\sigma^2_{-i \frac{\beta}{2}}(x^2_n)^*\right)\right\|\\
& = \left\| J_{\psi_1} \otimes J_{\psi_2}\left(  \sum_n \Lambda_{\psi_1}(x^1_n)\otimes \Lambda_{\psi_2}(x^2_n)\right)\right\| =  \left\| J_{\psi_1} \otimes J_{\psi_2}(\Lambda(s))\right\| = \left\|\Lambda(s)\right\|,
\end{align*}  
since $J_{\psi_1} \otimes J_{\psi_2}$, as the tensor product of two conjugate linear isometries is itself an isometry. This shows that the triple \eqref{08-06-22e} also has the property from the second item in condition (F) of Lemma \ref{09-02-22d}, and completes the proof.
\end{proof}

\begin{thm}\label{08-06-22f} Let $\sigma^i$ be a flow on the $C^*$-algebra $A_i$ and $\psi_i$ a $\beta$-KMS weight for $\sigma^i, \ i = 1,2$. There is a unique $\beta$-KMS weight $\psi_1 \otimes \psi_2$ for $\sigma^1\otimes \sigma^2$ on $A_1 \otimes A_2$ such that $\mathcal N_{\psi_1} \otimes \mathcal N_{\psi_2}$ is a core for $\Lambda_{\psi_1 \otimes \psi_2}$ and
$$
\psi_1 \otimes \psi_2(a_1^*b_1\otimes a_2^*b_2) = \psi_1(a_1^*b_1)\psi_2(a_2^*b_2)
$$
when $a_i,b_i \in \mathcal N_{\psi_i}, \ i = 1,2$. 
\end{thm}
\begin{proof} By Lemma \ref{08-06-22d}, Lemma \ref{09-02-22d} and Proposition \ref{08-02-22a} the triple \eqref{08-06-22e} is the GNS-triple for a $\beta$-KMS weight for $\sigma^1 \otimes \sigma^2$. The resulting $\beta$-KMS weight $\psi_1 \otimes \psi_2$ has the stated properties by construction. The uniqueness is an immediate consequence of Corollary \ref{21-06-22e}.
\end{proof}

\begin{notes}\label{26-09-22x} Theorem \ref{08-06-22f} was obtained by Kustermans in Proposition 7.6 of \cite{Ku1}.
\end{notes}

\subsection{ITPFI flows}\label{ITPFI}

Infinite tensor products of facors of type I (abbreviated to ITPFI) were introduced in the theory of von Neumann algebras in work by R. Powers, H. Araki and E.J. Woods from the sixties. Here we consider the corresponding infinite tensor product flows.

Let $n_1,n_2, n_3, \cdots $ be a sequence of natural numbers. Set
$$
A_k := M_{n_1}(\mathbb C) \otimes  M_{n_2}(\mathbb C) \otimes  M_{n_3}(\mathbb C) \otimes \cdots \otimes  M_{n_k}(\mathbb C).
$$
A unital embedding $\phi_k : A_k \to A_{k+1}$ is given on simple tensors by
$$
\phi_k(a_1\otimes \cdots \otimes a_k) =a_1\otimes \cdots \otimes a_k \otimes 1 .
$$
The inductive limit $C^*$-algebra of the sequence 
\begin{equation*}
\begin{xymatrix}{
A_1 \ar[r]^-{\phi_1} & A_2 \ar[r]^-{\phi_2} & A_3 \ar[r]^-{\phi_3} & \cdots  
}
\end{xymatrix}
\end{equation*}
is a UHF algebra unless $n_j = 1$ for all but finitely many $j$, in which case it is a matrix algebra. We denote it by
$$
\otimes_{k=1}^\infty M_{n_k}(\mathbb C) .
$$
Pick for each $k \in \mathbb N$ a selfadjoint matrix
$$
h_k \in M_{n_k}(\mathbb C)
$$
and define a flow $\sigma^k$ on $\otimes_{k=1}^\infty M_{n_k}(\mathbb C)$ such that
$$
\sigma^k_t(a_1\otimes \cdots \otimes a_k) = e^{ith_1}a_1e^{-ith_1} \otimes  e^{ith_2}a_2e^{-ith_2}\cdots \otimes  e^{ith_k}a_ke^{-ith_k}
$$
on simple tensors $a_1 \otimes a_2 \otimes \cdots \otimes a_k$. Define self-adjoint elements $\hat{h}_i \in A_k$ as simple tensors where $\hat{h}_1 = h_1 \otimes 1 \otimes 1 \otimes \cdots \otimes 1, \ \hat{h}_2 = 1 \otimes h_2 \otimes 1  \otimes \cdots \otimes 1$ and so on. Then $\sigma^k = \Ad e^{it H_k}$, where
$$
H_k := \sum_{i=1}^k \hat{h}_i.
$$
Since $\sigma^{k+1}_t \circ \phi_k = \phi_k \circ \sigma^k_t$ there is a flow $\sigma$ on $\otimes_{k=1}^\infty M_{n_k}(\mathbb C)$ defined such that
\begin{equation}\label{28-09-23x}
\sigma_t \circ \phi_{k,\infty} = \phi_{k,\infty} \circ \sigma^k_t,
\end{equation}
where $\phi_{k,\infty} : A_k \to \otimes_{k=1}^\infty M_{n_k}(\mathbb C)$ is the canonical embedding. Flows of this kind were considered in \cite{KiR} and \cite{Ki1} (with $n_j =2$).

Let $\beta \in \mathbb R$. It follows from Example \ref{31-07-23} that there is a unique $\beta$-KMS state $\omega_{h_j,\beta}$ for the flow $\Ad e^{it h_j}$ on $M_{n_j}(\mathbb C)$ given by
$$
\omega_{h_j,\beta}(m) := \frac{\Tr_{n_j}(e^{-\beta h_j} m)}{\Tr_{n_j}(e^{-\beta h_j})} 
$$ 
when $\Tr_{n_j}$ denote the usual trace on $M_{n_j}(\mathbb C)$.  We can then consider the tensor product state
$$
\otimes_{j=1}^k \omega_{h_j,\beta} 
$$
on $A_k$ defined such that
$$
\left(\otimes_{j=1}^k \omega_{h_j,\beta}\right) (a_1 \otimes a_2 \otimes \cdots \otimes a_k) := \prod_{j=1}^k \omega_{h_j,\beta}(a_j)
$$
on simple tensors. Using Lemma \ref{27-09-23a} we find that
\begin{align*}
&\left(\otimes_{j=1}^k \omega_{h_j,\beta}\right)((a_1 \otimes \cdots \otimes a_k)(b_1 \otimes \cdots \otimes b_k))  = \prod_{j=1}^k \omega_{h_j,\beta}(a_jb_j) \\
& =  \prod_{j=1}^k \omega_{h_j,\beta}(b_je^{-\beta h_j}a_je^{\beta h_j}) \\
&= \left(\otimes_{j=1}^k \omega_{h_j,\beta}\right)(b_1e^{-\beta h_1}a_1e^{\beta h_1} \otimes \cdots \otimes b_k e^{-\beta h_k}a_ke^{\beta h_k})\\
& = \left(\otimes_{j=1}^k \omega_{h_j,\beta}\right) ((b_1\otimes \cdots \otimes b_k)\sigma^k_{i\beta}(a_1\otimes \cdots \otimes a_k)) \\
& =\left(\otimes_{j=1}^k \omega_{h_j,\beta}\right)((b_1\otimes \cdots \otimes b_k)e^{-\beta H_k}(a_1\otimes \cdots \otimes a_k)e^{\beta H_k}),
\end{align*}
and we see therefore that $\otimes_{j=1}^k \omega_{h_j,\beta}$ is a $\beta$-KMS state for $\sigma^k$. As such it is unique by Example \ref{31-07-23} since $A_k \simeq M_m(\mathbb C)$ with $m = \prod_{j=1}^k n_j$. Note that 
$$
\left(\otimes_{j=1}^{k+1} \omega_{h_j,\beta}\right) \circ \phi_k =\otimes_{j=1}^{k} \omega_{h_j,\beta} .
$$
It follows that there is a state $\omega_\beta$ on  $\otimes_{k=1}^\infty M_{n_k}(\mathbb C)$ such that
\begin{equation}\label{28-09-23a}
\omega_\beta \circ \phi_k = \otimes_{j=1}^{k} \omega_{h_j,\beta}
\end{equation}
for all $k$. By Lemma \ref{27-09-23a} every element of $A_k$ is entire analytic for $\sigma^k$ and it follows therefore from \eqref{28-09-23x} that 
$$
\bigcup_{k=1}^\infty \phi_{k,\infty}(A_k) \subseteq \mathcal A_\sigma 
$$
and that $\sigma_z \circ \phi_{k,\infty} = \phi_{k,\infty} \circ \sigma^k_z$ for all $k,z$. In combination with \eqref{28-09-23a} this implies that
$$
\omega_\beta(ab) = \omega_\beta(b\sigma_{i \beta}(a))
$$
for all $a,b \in \bigcup_{k=1}^\infty \phi_{k,\infty}(A_k)$. Since $\Lambda_{\omega_\beta} : A \to H_{\omega_\beta}$ is bounded and $\bigcup_{k=1}^\infty \phi_{k,\infty}(A_k) $ is dense in  $\otimes_{k=1}^\infty M_{n_k}(\mathbb C)$ it follows from Theorem \ref{12-12-13} with $\bigcup_{k=1}^\infty \phi_{k,\infty}(A_k) $ in the role as $S$, that $\omega_\beta$ is a $\beta$-KMS state for $\sigma$. We denote it by
$$
\omega_\beta = \otimes_{k=1}^\infty \omega_{h_k,\beta} .
$$
Conversely, if $\omega$ is a $\beta$-KMS state for $\sigma$ it follows from \eqref{28-09-23x} that $\omega \circ \phi_{k,\infty}$ is a $\beta$-KMS state for $\sigma^k$ and and hence $\omega\circ \phi_{k,\infty} = \otimes_{k=1}^\infty \omega_{h_k,\beta}$ by Example \ref{31-07-23}. It follows that $\omega = \omega_\beta$. In conclusion: The flow $\sigma$ has exactly one $\beta$-KMS state for all $\beta \in \mathbb R$; the infinite tensor product state $\otimes_{k=1}^\infty \omega_{h_k,\beta}$.

\subsection{Matroid flows}

In the same setting as in the last section choose non-zero projections $p_j \in M_{n_j}(\mathbb C), \ j \geq 2$, and define $\kappa_k : A_k \to A_{k+1}$ by
$$
\kappa_k(a_1 \otimes \cdots \otimes a_k) = a_1 \otimes \cdots \otimes a_k \otimes p_{k+1}
$$
for $k=1,2,3, \cdots$. Let $D$ be the $C^*$-algebra inductive limit of the sequence
\begin{equation*}
\begin{xymatrix}{
A_1 \ar[r]^-{\kappa_1} & A_2 \ar[r]^-{\kappa_2} & A_3 \ar[r]^-{\kappa_3} & \cdots  
}
\end{xymatrix}
\end{equation*}
This kind of $C^*$-algebras were introduced by J. Dixmier who called them matroid algebras. We note that $D$ is unital if and only $p_j$ is the unit in $M_{n_j}(\mathbb C)$ for all sufficiently large $j$, in which case $D$ is a UHF algebra.

Choose selfadjoint matrices $h_k \in M_{n_k}(\mathbb C)$, but this time such that $h_k$ commutes with $p_k, \ k \geq 2$; i.e.
$$
h_k=h_k^* \in \{p_k\}' \cap M_{n_k}(\mathbb C) , \ k \geq 2 .
$$
Define the flows $\sigma^k$ on $A_k$ as above using the new $h_k$'s and note that $\sigma^{k+1}_t \circ \kappa_k = \kappa_k \circ \sigma^k_t$ so that we have a flow $\alpha$ on $D$ defined by
$$
\alpha_t \circ \kappa_{k,\infty} = \kappa_{k,\infty} \circ \sigma^k_t
$$
for all $k,t$. Let $1_{n_1}$ be the unit in $A_1 = M_{n_1}(\mathbb C)$. The corner
$$
\kappa_{1,\infty}(1_{n_1})D\kappa_{1,\infty}(1_{n_1})
$$
of $D$ is then $\alpha$-invariant and a full corner of $D$. Since the corner is $*$-isomorphic to
$$
\otimes_{j=1}^\infty M_{m_j}(\mathbb C) ,
$$
where $m_j := \Tr_{n_j}(p_j)$, it follows from Section \ref{ITPFI} and Theorem \ref{07-06-22e} that there is a $\beta$-KMS weight for $\alpha$ for every $\beta \in \mathbb R$ and that it is unique up to multiplication by scalars. 

The images $\kappa_{k,\infty}(1_{A_k})$ of the units $1_{A_k}$ the $A_k$'s in $D$ constitute an approximate unit for $D$ and it follows therefore from the description of the $\beta$-KMS states on $\otimes_{j=1}^\infty M_{m_j}(\mathbb C)$ and Lemma \ref{03-02-22h} that the $\beta$-KMS weights for $\alpha$ will be bounded if and only if 
\begin{equation}\label{02-10-23g}
\prod_{j=2}^\infty \frac{\Tr_{n_j}(e^{-\beta h_j})}{\Tr_{n_j}(e^{-\beta h_j}p_j)} < \infty .
\end{equation}
To see this let $\psi$ be a $\beta$-KMS weight for $\alpha$. Then $\psi \circ \kappa_{k,\infty}$ is a $\beta$-KMS functional for the flow $\sigma^k$ on $A_k$ and hence a scalar multiple of $\otimes_{j=1}^k \omega_{h_j,\beta}$ by Example \ref{31-07-23}. That is, $\psi \circ \kappa_{k,\infty} = \lambda_k \otimes_{j=1}^k \omega_{h_j,\beta}$ for some $\lambda_k > 0$. Note that
$$
\psi \circ \kappa_{k,\infty}(1_{A_k}) = \lambda_k \otimes_{j=1}^k \omega_{h_j,\beta}(1_{A_k}) = \lambda_k
$$
while
\begin{align*}
& \psi\circ \kappa_{1,\infty}(1_{n_1}) =
\psi \circ \kappa_{k,\infty}(1_{n_1} \otimes p_2 \otimes p_3 \otimes \cdots \otimes p_{k-1}) \\
&= \lambda_k \otimes_{j=1}^k \omega_{h_j,\beta} (1_{n_1} \otimes p_2 \otimes p_3 \otimes \cdots \otimes p_{k-1}) = \lambda_k \prod_{k=2}^{k-1} \frac{\Tr_{n_j}(e^{-\beta h_j}p_j)}{\Tr_{n_j}(e^{-\beta h_j})} .
\end{align*}
Combined these equations imply that
$$
\psi \circ \kappa_{k,\infty}(1_{A_k}) = \psi\circ \kappa_{1,\infty}(1_{n_1})\prod_{k=2}^{k-1} \frac{\Tr_{n_j}(e^{-\beta h_j})}{\Tr_{n_j}(e^{-\beta h_j}p_j)} .
$$
Hence $\psi$ is bounded if and only if \eqref{02-10-23g} holds.

We consider two special cases; the first to show that there can be bounded KMS weights in cases where there are no bounded traces on $D$ and the second to show that there can be unbounded KMS weights in cases where the trace(s) on $D$ are bounded.

Consider the case where $m_j = n_j -1$ and $h_j = \mu_j e_j$ for some $\mu_j \in \mathbb R$ and a minimal projection $e_j$ in $M_{n_j}(\mathbb C)$ with $e_j \leq p_j$. Then
$$
\frac{\Tr_{n_j}(e^{-\beta h_j})}{\Tr_{n_j}(e^{-\beta h_j}p_j)} = \frac{n_j -1 +e^{-\beta  \mu_j}}{n_j-2 + e^{-\beta \mu_j}} = 1  +x_j
$$
where $x_j = (n_j-2 + e^{-\beta \mu_j})^{-1}$. Hence, in this case the $\beta$-KMS weights for $\alpha$ are bounded if and only if\footnote{We are here using that when $\{a_n\}_{n=1}^\infty$ is a sequence of positive real numbers, $\sum_{n=1}^N a_n \leq \prod_{n=1}^N (1+a_n) \leq \exp(\sum_{n=1}^N a_n)$ which leads to the well-known conclusion that $\prod_{n=1}^\infty (1+a_n) < \infty$ if and only if $\sum_{n=1}^\infty a_n < \infty$.}
$$
\sum_{j=1}^\infty  (n_j-2 + e^{-\beta \mu_j})^{-1} < \infty.
$$
Thus, if we choose $n_j =3$ for all $j$, the condition becomes
$$
\sum_{j=1}^\infty  (1+ e^{-\beta \mu_j})^{-1} < \infty.
$$
If we then set
$$
\mu_j = - (l-1)
$$
for $7^{l-1} \leq j < 7^l$, $l=1,2,3, \cdots$, we see that the $\beta$-KMS weights for $\alpha$ are bounded if and only if
$$
\sum_{l=1}^{\infty} \frac{7^{l-1}}{1+ e^{\beta(l-1)}} < \infty ;
Ã¸$$
i.e. if and only if $\beta > \log 7$. In particular, the lower semi-continuous traces on $D$ are unbounded.

Consider then the case where $n_j = j!$ and $m_j = \Tr_{n_j}(p_j) = j! -1$. In this case 
\begin{align*}
& \frac{\Tr_{n_j}(e^{-\beta h_j})}{\Tr_{n_j}(e^{-\beta h_j}p_j)} =  \frac{\Tr_{n_j}(e^{-\beta h_j}(1-p_j)) +\Tr_{n_j}(e^{-\beta h_j}p_j)}{\Tr_{n_j}(e^{-\beta h_j}p_j)}\\
& = \frac{e^{-\beta \mu_j} +\Tr_{n_j}(e^{-\beta h_j}p_j)}{\Tr_{n_j}(e^{-\beta h_j}p_j)} = 1 +  \frac{e^{-\beta \mu_j} }{\Tr_{n_j}(e^{-\beta h_j}p_j)}
\end{align*}
where $\mu_j \in \mathbb R$ is the real number such that $h_j(1-p_j) = \mu_j(1-p_j)$. Hence the $\beta$-KMS weight for $\alpha$ are bounded if and only if
$$
\sum_{j=1}^\infty \frac{e^{-\beta \mu_j} }{\Tr_{n_j}(e^{-\beta h_j}p_j)} < \infty .
$$
Since $\Tr_{n_j}(p_j) = j!-1$ this condition is certainly satisfied when $\beta = 0$, regardless of the choice of $h_j$, implying that $D$ has a bounded trace. On the other hand, if we take take $h_j = - \log(j!)(1-p_j)$ we get the value $\mu_j =-\log (j!)$ while $e^{-\beta h_j}p_j = p_j$
 for all $\beta$. Then
 $$
 \sum_{j=1}^\infty \frac{e^{- \mu_j} }{\Tr_{n_j}(e^{- h_j}p_j)} = \sum_{j=1}^{\infty} \frac{j!}{j!-1} = \infty,
 $$
implying that the $1$-KMS weights are unbounded. In fact, in this case the $\beta$-KMS weights are unbounded if and only of $\beta \geq 1$.

\subsection{Stabilization} 
Let $\sigma$ be a flow on $A$ and let $\mathbb H$ be a Hilbert space. Given a flow $W = \{W_t\}_{t \in \mathbb R}$ of unitaries on $\mathbb H$ we define a flow $\sigma \otimes \Ad W$ on $A \otimes \mathbb K(\mathbb H)$ such that
$$
(\sigma \otimes \Ad W)_t(a \otimes k) := \sigma_t(a) \otimes W_t k W_t^*
$$
when $a \in A, \ k \in \mathbb K(\mathbb H)$. By Stone's theorem there is a self-adjoint operator $H$ on $\mathbb H$ such that
$$
W_t = e^{it H}
$$
for all $t \in \mathbb R$. For any $\beta \in \mathbb R$ we denote by $\kappa_\beta$ the $\beta$-KMS weight for $\Ad W$ on $\mathbb K(\mathbb H)$ given by the formula
\begin{equation*}\label{30-06-22}
\kappa_\beta(k) = \sup \left\{\Tr\left(e^{\frac{-\beta H}{2}} f(H) k f(H) e^{-\frac{\beta H}{2}}\right) : \ f \in C_c(\mathbb R), \ 0 \leq f \leq 1 \right\} ,
\end{equation*}
cf. Theorem \ref{17-08-22c}.

\begin{lemma}\label{09-06-22x}  For each $\beta \in \mathbb R$ there is a bijection 
$$
\rho : \KMS(\sigma, \beta) \to \KMS(\sigma \otimes \Ad W, \beta)
$$
 defined by
$$
\rho(\psi) = \psi \otimes \kappa_\beta .
$$
\end{lemma}
\begin{proof} The map $\rho$ is defined thanks to Theorem \ref{08-06-22f}. We note that $\rho$ is the composition $\rho = \rho_2 \circ \rho_1$ of the map $\rho_1 : \KMS(\sigma,\beta) \to \KMS(\sigma \otimes \id_{\mathbb K(\mathbb H)},\beta)$ given by $\rho_1(\psi) := \psi \otimes \Tr$ and the bijection $\rho_2 : \KMS(\sigma \otimes \id_{\mathbb K(\mathbb H)},\beta) \to \KMS(\sigma \otimes \Ad W, \beta)$ coming from Theorem \ref{29-06-22d} by taking $u_t = 1_A \otimes W_t$. It suffices therefore to show that $\rho_1$ is a bijection. For this we pick a 1-dimensional projection $e \in \mathbb K(\mathbb H)$. Then $A \otimes e$ is a hereditary $C^*$-subalgebra of $A \otimes \mathbb K(\mathbb H)$ which is full and $\sigma \otimes \id_{\mathbb K(\mathbb H)}$-invariant.  Let $\phi \in \KMS(\sigma \otimes \id_{\mathbb K(\mathbb H)}, \beta)$. It follows from Theorem \ref{07-06-22e} that there is $\beta$-KMS weight $\psi$ for $\sigma$ such that 
$$
\mathcal N_\psi = \left\{ a \in A: \ a \otimes e \in \mathcal N_\phi\right\}
$$
and $\phi(a^*a \otimes e) = \psi(a^*a)$ for all $a \in \mathcal N_\psi$. Then Corollary \ref{31-03-22d} implies that  $\phi|_{A \otimes e} = \psi \otimes \Tr|_{A \otimes e}$ and hence that $\phi = \psi \otimes \Tr$ by Laca-Neshveyevs theorem or Theorem \ref{07-06-22e} again. This shows that $\rho_1$ is surjective. If $\rho_1(\psi) = \rho_1(\psi')$ it follows that 
$$
\psi(a^*a) = \rho_1(\psi)(a^*a \otimes e) = \rho_1(\psi')(a^*a \otimes e ) = \psi'(a^*a)
$$ 
for all $a \in \mathcal N_\psi \cap \mathcal N_{\psi'}$. Hence $\psi = \psi'$ by Corollary \ref{31-03-22d}. 
\end{proof}

\begin{example}\label{29-11-22} \textnormal{The intention with this example is to illustrate Theorem \ref{17-08-22c} and in particular the significance of the lower semi-continuity assumption in the definition of a weight and a KMS weight. Let $A = C_0(\mathbb N,M_2(\mathbb C))$ be the $C^*$-algebra of sequences $\{a_n\}_{n=1}^\infty$ in $M_2(\mathbb C)$ with the property that $\lim_{n \to \infty} \|a_n\| = 0$. Let $\omega$ be the trace state of $M_n(\mathbb C)$. For any non-zero sequence $t=\{t_n\}_{n=1}^\infty$ of non-negative real numbers we can define a trace $\tau_t$ on $A$ by
$$
\tau_t(\{a_n\}_{n=1}^\infty) = \sum_{n=1}^\infty t_n \omega(a_n)
$$
when $\{a_n\}_{n=1}^\infty \in A^+$. This is a lower semi-continuous trace on $A$, and every lower semi-continuous trace of $A$ arises in this way. For each $h = h^* \in M_2(\mathbb C)$ we can define an inner flow $\sigma$ on $A$ such that
$$
\sigma_t(\{a_n\}_{n=1}^\infty) := \{e^{ith}a_ne^{-ith}\}_{n=1}^\infty .
$$
It follows from Theorem \ref{17-08-22c} that for any $\beta \in \mathbb R$ the set of $\beta$-KMS weights for $\sigma$ is in bijection with the set of lower semi-continuous traces on $A$ and hence with the set of non-zero sequences of non-negative real numbers.}

\textnormal{The point we want to make here is that there are many traces on $A$ that are not lower semi-continuous. Let $\xi \in \beta \mathbb N \backslash \mathbb N$ be a free ultrafilter in $\mathbb N$ identified with an element of the Stone-\v{C}ech compactification of $\mathbb N$. Define $\mu : A^+ \to [0,\infty]$ such that
\begin{align*}
&\mu((a_n)_{n=1}^\infty) \\
&= \begin{cases} \lim_{N \to \xi} \frac{1}{N} \sum_{n=1}^N n \omega(a_n) , & \ \text{when $\{\frac{1}{N} \sum_{n=1}^N n \omega(a_n)\}_{N=1}^\infty$ is bounded} \\ \infty , & \  \text{when $\{\frac{1}{N} \sum_{n=1}^N n \omega(a_n)\}_{N=1}^\infty$ is unbounded.}  \end{cases}
\end{align*}
It is straightforward to see that this is a trace on $A$. To see that it is not lower semi-continuous define $b(k) \in A^+$ such that
$$
b(k)_n = \begin{cases} \frac{1}{n}, & \ n \leq k \\ 0, & \ n > k . \end{cases}
$$
Then $b(k) \leq b(k+1)$ and $\lim_{k \to \infty} b(k) = b$, where $b \in A^+$ is the element $b =\{\frac{1}{n}\}_{n=1}^\infty$. Since $\mu(b(k)) = 0$ for all $k$ while $\mu(b) =1$ we conclude that $\omega$ is not lower semi-continuous.}

\textnormal{For each $\beta\in \mathbb N$ this trace $\mu$ gives rise to a map $\mu_\beta : A^+ \to [0,\infty]$ defined such that
$$
\mu_\beta(\{a_n\}_{n=1}^\infty) = \mu(\{e^{-\beta h}a_n\}_{n=1}^\infty) .
$$
This map $\mu_\beta$ has the first two of the three properties required of a weight in Section \ref{defweight} and it has the properties relative to $\sigma$ specified by (1), (2) and (3) in Kustermans' theorem, Theorem \ref{24-11-21d}. But it is not a $\beta$-KMS weight since it is not lower semi-continuous.}

\textnormal{This observation, that there are weights with all properties required by a KMS weight except for the lower semi-continuity assumption, may seem out of place because all the major results on KMS weights we have presented depend on this continuity assumption. But we note that in the first definition of a KMS weight for flows given by F. Combes in Definition 4.1 of \cite{C1} lower semi-continuity is not mentioned.}
\end{example}



\chapter{Dual KMS weights}\label{20-01-22g}

In this chapter we describe a method to construct KMS weights for flows on a crossed product $C^*$-algebra starting from KMS weight for a flow on the core algebra.

\section{Extending KMS weights to a crossed product}\label{07-08-22e}

 Let $A$ be a $C^*$-algebra, $G$ a locally compact group and $\alpha$ a representation of $G$ by automorphisms $\alpha_x, x \in G$, of $A$ such that $x \mapsto \alpha_x(a)$ is continuous for all $a \in A$. Let $\sigma$ be a flow on $A$ which commutes with $\alpha$ in the sense that
$$
\sigma_t \circ \alpha_x = \alpha_x \circ \sigma_t
$$ 
for all $x \in G$ and $t \in \mathbb R$. Let
$$
B := A \rtimes_{r,\alpha} G
$$
be the reduced crossed product constructed from $(A,G,\alpha)$ as in Section 7.7 of \cite{Pe} or in \cite{Wi}. Let $C_c(G,A)$ denote the set of continuous compactly supported functions $f : G \to A$. Then $C_c(G,A)$ is a $*$-algebra with the convolution-like product
$$
(f \star g)(x) := \int_G   f(y) \alpha_y(g(y^{-1}x)) \ \mathrm d y
$$
and involution
$$
f^\sharp(x) := \Delta(x)^{-1}\alpha_x(f(x^{-1})^*) ;
$$
see \cite{Pe}. Here $\Delta$ is the modular function of $G$ and $dy$ denotes integration with respect to a left-invariant Haar measure which we denote by $\mu$ in the following. We assume that $A$ is realized as a $C^*$-subalgebra of the bounded operators on a Hilbert space $\mathbb H$ in a $\sigma$-covariant way; i.e. we assume that there is a strongly continuous unitary representation $U_t, t \in \mathbb R$, of $\mathbb R$ on $\mathbb H$ such that $\Ad U_t(a) = \sigma_t(a)$ for all $t \in \mathbb R$ and all $a \in A$. Such a covariant representation always exists, cf. Theorem 7.6.6 of \cite{Pe}. We consider the corresponding injective representation
$$
\pi : C_c(G,A) \to B(L^2(G,\mathbb H))
$$  
defined such that
$$
(\pi(f)\xi)(x) = \int_G \alpha_{x^{-1}}(f(y))\xi(y^{-1}x) \ \mathrm d y  
$$
for $f \in C_c(G,A)$ and $\xi \in L^2(G,\mathbb H)$, cf. 7.7.1 and 7.7.5 in \cite{Pe}. The $C^*$ algebra $B= A \rtimes_{r,\alpha} G$ is the closure in $B(L^2(G,\mathbb H))$ of $\pi\left(C_c(G,A)\right)$. We shall work almost exclusively with $C_c(G,A)$ and will therefore mostly suppress the representation $\pi$ in the notation.

 Let $\theta : G \to \mathbb R$ be a continuous group homomorphism. Define unitaries $S_t$ on $L^2(G,\mathbb H)$ by
$$
S_t\xi(x) := e^{i\theta(x)t}U_t\xi(x) .
$$
Then
$$
\Ad S_t (\pi(f)) = \pi(\sigma^\theta_t(f))
$$
for all $t \in \mathbb R$ and $f \in C_c(G,A)$. It follows that there is a flow $\sigma^\theta$ on $B$ defined such that $\sigma^\theta_t(f) \in C_c(G,A)$ and
\begin{equation}\label{30-01-22e}
\sigma^\theta_t(f)(x) = e^{i\theta(x) t}\sigma_{t}(f(x)) 
\end{equation}
 when $f \in C_c(G,A)$.

Let $\beta \in \mathbb R$ be a real number and let $\psi$ be a $\beta$-KMS weight for $\sigma$ which is scaled by $\alpha$ such that
\begin{equation}\label{04-02-22a}
\psi\circ \alpha_x = \Delta(x)^{-1}e^{-\theta(x)\beta}\psi \ \ \ \ \ \forall x \in G.
\end{equation}
We fix the above setting in the rest of this section. The intention is to construct from $\psi$ a $\beta$-KMS weight $\widehat{\psi}$ for $\sigma^\theta$.

The first step is to relate $\alpha$ and $\Lambda_\psi$. Let $a \in \mathcal N_\psi$. It follows from \eqref{04-02-22a} that $\Delta(x)^{\frac{1}{2}}e^{\frac{\theta(x)\beta}{2}}\alpha_x(a) \in \mathcal N_\psi$ and 
\begin{align*}
&\left\| \Lambda_\psi\left( \Delta(x)^{\frac{1}{2}}e^{\frac{\theta(x)\beta}{2}}\alpha_x(a)\right)\right\|^2 = \Delta(x) e^{\theta(x)\beta}  \psi(\alpha_x(a^*a)) \\
& = \psi(a^*a) = \left\|\Lambda_\psi(a)\right\|^2 . 
\end{align*}
We can therefore define $W_x \in B(H_\psi)$ such that
$$
W_x\Lambda_\psi(a) :=  \Delta(x)^{\frac{1}{2}}e^{\frac{\theta(x)\beta}{2}}\Lambda_\psi(\alpha_x(a)) 
$$
for all $a \in \mathcal N_\psi$.

\begin{lemma}\label{16-01-22} $(W_x)_{x \in G}$ is a continuous unitary representation of $G$ on $H_\psi$.
\end{lemma}
\begin{proof} It is straightforward to check that $W$ is a unitary representation. To check that it is continuous it is then sufficient to show that 
\begin{equation}\label{16-01-22a}
\lim_{x \to e} \left< W_x\Lambda_\psi(a),\Lambda_\psi(b)\right> = \left<\Lambda_\psi(a),\Lambda_\psi(b)\right>
\end{equation} 
for $a,b \in \mathcal N_\psi$. Let $\epsilon > 0$. It follows from Combes' theorem, Theorem \ref{04-11-21k}, that there is $\omega \in \mathcal F_\psi$ such that 
$$
\psi(a^*a) - \epsilon \leq \omega(a^*a) \leq \psi(a^*a).
$$
By Lemma \ref{08-11-21bx} there is an operator $T_\omega$ on $H_\psi$ such that $0 \leq T_\omega \leq 1$ and
$$
\omega(c^*d) = \left<T_\omega \Lambda_\psi(d),\Lambda_\psi(c)\right> \  \ \ \ \ \forall c,d \in \mathcal N_\psi .
$$
The same calculation as in the proof of Lemma \ref{17-11-21a} shows that 
$$
\left\| T_\omega \Lambda_\psi(a) - \Lambda_\psi(a)\right\|^2 \leq 2 \epsilon
$$ 
and we find therefore that
\begin{align*}
& \left|\left<W_x\Lambda_\psi(a), \Lambda_\psi(b)\right> - \left<\Lambda_\psi(a),\Lambda_\psi(b)\right>\right| \\
&= \left|\left<\Lambda_\psi(a), W_{x^{-1}}\Lambda_\psi(b)\right> - \left<\Lambda_\psi(a),\Lambda_\psi(b)\right>\right| \\
& \leq 2\left\|\Lambda_\psi(b)\right\| \sqrt{2\epsilon} + \left|\left<T_\omega\Lambda_\psi(a),  W_{x^{-1}}\Lambda_\psi(b)\right> - \left<T_\omega \Lambda_\psi(a),\Lambda_\psi(b)\right>\right| \\
& = 2\left\|\Lambda_\psi(b)\right\| \sqrt{2\epsilon} + \left|\Delta(x)^{-\frac{1}{2}}e^{-\frac{\theta(x)\beta}{2}} \omega(\alpha_{x^{-1}}(b)^*a) - \omega(b^*a)\right|
\end{align*}
for all $x \in G$. Since 
$$\lim_{x \to e} \left|\Delta(x)^{-\frac{1}{2}}e^{-\frac{\theta(x)\beta}{2}}\omega(\alpha_{x^{-1}}(b)^*a) - \omega(b^*a)\right| =0
$$ 
we get \eqref{16-01-22a}.
\end{proof}

\begin{lemma}\label{20-01-22e} Let $f \in C_c(G,A)$ and assume that $f(x) \in \mathcal N_\psi^*$ for all $x \in G$. Then $G \ni x \to \Lambda_\psi(f^\sharp(x))$ is continuous if and only if $G \ni x \to \Lambda_\psi(f(x)^*)$ is continuous.
\end{lemma}
\begin{proof} This follows from Lemma \ref{16-01-22} since
$$
\Lambda_\psi(f^\sharp(x)) = \Delta(x)^{-1}\Lambda_\psi\left(\alpha_x(f(x^{-1})^*)\right) = \Delta(x)^{-\frac{3}{2}}e^{-\frac{\theta(x)\beta}{2}}W_x\Lambda_\psi(f(x^{-1})^*) .
$$
\end{proof}

Recall from Corollary \ref{24-06-22a} that $\mathcal N_\psi$ has a structure as a Banach algebra in the norm
$$
|||a||| := \|a\| + \left\|\Lambda_\psi(a)\right\| .
$$
Denote by $C_0(G,\mathcal N_\psi)$ the space of continuous functions from $G$ to 
$$
(\mathcal N_\psi, ||| \ \cdot \ |||)
$$ 
that vanish at infinity. This is also a Banach algebra; the norm is 
$$
|||f||| := \sup_{x \in G} |||f(x)||| .
$$  
There is a flow $\overline{\sigma}$ of isometries on $C_0(G,\mathcal N_\psi)$ defined by
$$
\overline{\sigma}_t(f)(x) := \sigma_t(f(x)) .
$$
Note that $C_0(G,\mathcal N_\psi) \subseteq C_0(G,A)$.
We denote by $\mathcal A_{\overline{\sigma}}$ the algebra of elements of $C_0(G,\mathcal N_\psi)$ that are entire analytic for $\overline{\sigma}$. Set
$$
B_0 := C_c(G,A) \cap \mathcal A_{\overline{\sigma}}  .
$$

\begin{lemma}\label{16-01-22e} $B_0$ is subalgebra of $B$.
\end{lemma}
\begin{proof} It is obvious that $B_0$ is a subspace of $B$. Let $f,g \in B_0$. To show that $f \star g \in B_0$, note that it follows from from Lemma \ref{12-02-22b} that 
$$
f \star g (x) = \int_G f(y)\alpha_y(g(y^{-1}x)) \ \mathrm dy \in \mathcal N_\psi
$$ 
for all $x \in G$ and
\begin{align*}
& \Lambda_\psi(f  \star g(x)) =  \int_G \Lambda_\psi\left(f(y)\alpha_y(g(y^{-1}x))\right) \ \mathrm dy \\
& =  \int_G \pi_\psi(f(y))\Lambda_\psi\left(\alpha_y(g(y^{-1}x))\right) \ \mathrm dy .
\end{align*}
From this expression it follows that $f \star g \in C_0(C,\mathcal N_\psi)$. To show that $f \star g\in B_0$ we must show that $f\star g \in \mathcal A_{\overline{\sigma}}$. For this note that there are sequences of functions $\{a_n\}$ and $\{b_n\}$ in $C_0(G,\mathcal N_\psi)$ such that
$$
\sigma_t(f(x)) = \sum_{n=0}^\infty a_n(x)t^n
$$
and
$$
\sigma_t(g(x)) = \sum_{n=0}^\infty b_n(x)t^n
$$
both converge in $(\mathcal N_\psi, ||| \ \cdot \ |||)$ uniformly in $x$ for all $t \in \mathbb R$. In fact, 
$$
\sum_{n=0}^\infty \sup_{x \in G} |||a_n(x)||| \ |z|^n < \infty ,
$$
$$
\sum_{n=0}^\infty \sup_{x \in G} ||| b_n(x)||| \  |z|^n < \infty ,
$$
for all $z \in \mathbb C$, and there are compact sets $K$ and $L$ in $G$ such that all $a_n$'s vanish outside of $K$ and all $b_n$'s vanish outside of $L$. Indeed, if $h : G \to [0,1]$ is a compactly supported continuous function such that $h(x) =1$ for all $x \in \supp f$, the sequence $\{ha_n\}$ will have the same properties as $\{a_n\}$ and hence $ha_n = a_n$ for all $n$ by the uniqueness part of Lemma \ref{13-11-21a}, and we can therefore take $K = \supp h$. The set $L$ is constructed in the same way.

Set
$$
c_n(x) := \sum_{k=0}^n \int_G a_k(y)\alpha_y(b_{n-k}(y^{-1}x)) \ \mathrm d y .
$$
An application of Lemma \ref{12-02-22} and \eqref{28-02-22a} shows that $c_n \in C_0(G,\mathcal N_\psi)$. Note that $c_n$ is supported in $KL$ and
\begin{align*}
&\overline{\sigma}_t(f \star g)(x) = \sigma_t((f \star g)(x)) = \int_G \sigma_t(f(y))\alpha_y(\sigma_t(g(y^{-1}x))) \ \mathrm dy \\
&  = \int_G (\sum_{n=0}^\infty a_n(y)t^n) \alpha_y( \sum_{n=0}^\infty b_n(y^{-1}x)t^n) \ \mathrm d y  \\
& = \int_G \sum_{n=0}^\infty \left(\sum_{k=0}^n a_k(y) \alpha_y( b_{n-k}(y^{-1}x))\right)t^n \ \mathrm d y \\
&= \sum_{n=0}^\infty c_n(x) t^n .
\end{align*}
Note also that
$$
\sup_{x \in G} \left\|c_n(x)\right\| \leq \sum_{k=0}^n (\sup_{x \in G} \|a_k(x)\|)(\sup_{x \in G} \left\|b_{n-k}(x)\right\| )\mu(KL) .
$$
Using Lemma \ref{12-02-22b} we find that
\begin{align*}
&\Lambda_\psi(c_n(x)) = \sum_{k=0}^n \int_G \pi_\psi(a_k(y))\Lambda_\psi(\alpha_y(b_{n-k}(y^{-1}x))) \ \mathrm d y \\
& =\sum_{k=0}^n\int_G \Delta(y)^{-\frac{1}{2}}e^{-\frac{\theta(y)\beta}{2}}\pi_\psi(a_k(y)) W_y \Lambda_\psi(b_{n-k}(y^{-1}x)) \ \mathrm d y 
\end{align*}
and hence
\begin{align*}
& \sup_{x \in G}\|\Lambda_\psi(c_n(x))\| \leq M \sum_{k=0}^n (\sup_{x \in G}\left\|a_k(x)\right\|)(\sup_{x \in G} \left\|\Lambda_\psi(b_{n-k}(x))\right\|)
\end{align*}
where $M = \int_{KL} \Delta(y)^{-\frac{1}{2}}e^{-\frac{\theta(y)\beta}{2}} \ \mathrm dy$. It follows that
$$
\sup_{x \in G} |||c_n(x)||| \leq \max\{\mu(KL), M\} \sum_{k=0}^n (\sup_{x \in G} |||a_k(x)|||) (\sup_{x \in G} |||b_{n-k}(x)|||) ,
$$
and we conclude that
\begin{align*}
&\sum_{n=0}^\infty \sup_{x \in G} |||c_n(x)||| \ |z|^n \\
&\leq \max\{\mu(KL), M\}\left(\sup_{x \in G} |||a_n(x)||| \ |z|^n \right) \left(\sup_{x \in G} |||b_n(x)||| \ |z|^n \right) < \infty 
\end{align*}
for all $z \in \mathbb C$. 
It follows that 
$$
\overline{\sigma}_t(f \star g) = \sum_{n=0}^\infty c_n t^n
$$
converges in $C_0(G,\mathcal N_\psi)$ for all $t \in \mathbb R$, proving that $f \star g \in \mathcal A_{\overline{\sigma}}$. Hence $f \star g \in B_0$.
\end{proof}

\begin{lemma}\label{23-06-22}  Let $f \in B_0$. Then
\begin{itemize} 
\item[(a)] $f$ is entire analytic for $\sigma^\theta$ and $f(x) \in \mathcal A_\sigma$ for all $x \in G$,
\item[(b)] $\sigma_z(f(x)) \in \mathcal N_\psi$ for all $z \in \mathbb C$ and all $x \in G$,
\item[(c)] $G \ni x \mapsto \sigma_z(f(x))$ is continuous for all $z \in \mathbb C$,
\item[(d)] $G \ni x \mapsto \Lambda_\psi(\sigma_z(f(x)))$ is continuous for all $z \in \mathbb C$, 
\item[(e)] $\sigma_z^\theta(f) \in B_0$ for all $z \in \mathbb C$,
\item[(f)] $\sigma^\theta_z(f)(x) = e^{i\theta(x)z} \sigma_{z}(f(x))$ for all $x \in G$ and all $z \in \mathbb C$.
\end{itemize}
\end{lemma}
\begin{proof} (a)+(e)+(f): Choose a continuous function $h : G \to [0,1]$ of compact support such that $h(x) =1$ for $x$ in the support of $f$. Then
$$
\sigma^\theta_t(f)(x) =   e^{i \theta(x) h(x) t} \sigma_{t}(f(x)) 
$$
for all $x \in G$. This has the effect that we may substitute $\theta h$ for $\theta$ in the following calculations and hence work as if $\theta$ is bounded. Since $f$ is entire analytic for $\overline{\sigma}$ there is a sequence $\{g_n\}$ in $C_0(G,\mathcal N_\psi)$ such that
$$
\sigma_{t}(f(x)) = \sum_{n=0}^\infty g_n(x) t^n
$$
and 
$$
\sum_{n=0}^\infty \sup_{x \in G} |||g_n(x)||| |t|^n < \infty
$$
for all $t \in \mathbb R$. This implies, in particular, that $f(x) \in \mathcal A_\sigma$ for all $x \in G$. Note that $hg_n = g_n$ by uniqueness of $\{g_n\}$, cf. Lemma \ref{13-11-21a}. Set
$$
h_n := \sum_{k=0}^n \frac{(i h \theta)^k}{k!}g_{n-k} .
$$
Then 
$$
\sigma^\theta_t(f)(x) = \sum_{n=0}^\infty h_n(x)t^n
$$
converges uniformly on $G$ with all terms supported in the compact support of $h$. Therefore the sequence also converges in $L^1(G,A)$ and hence in $B$. 
This shows that $f$ is entire analytic for $\sigma^\theta$, and
\begin{align*}
&\sigma_z^\theta(f)(x) = \sum_{n=0}^\infty h_n(x)z^n \\
& = \sum_{n=0}^\infty \sum_{k=0}^n \frac{(ih(x)\theta(x))^k}{k!}z^k g_{n-k}(x)z^{n-k}= e^{i\theta(x)z} \sigma_{z}(f(x) )
\end{align*}
for all $z,x$. Note that
$$
x \mapsto e^{i\theta(x)z}\sigma_z(f(x)) = e^{i \theta(x)z} \sum_{n=0}^\infty g_n(x)z^n
$$
is continuous, showing that $\sigma^\theta_z(f) \in C_c(G,A)$. Note also that
\begin{align*}
&\sigma_t( \sigma^\theta_z(f)(x)) = e^{i \theta(x)z} \sum_{n=0}^\infty g_n(x) (t + z)^n \\
&=  e^{i \theta(x)z} \sum_{n=0}^\infty g_n(x) \sum_{k=0}^n {n \choose k}  t^kz^{n-k} = \sum_{n=0}^\infty \kappa_k(x)t^k
\end{align*}  
where 
$$
\kappa_k(x) = e^{ih(x)\theta(x)z} \sum_{n \geq k}  {n \choose k}  g_n(x) z^{n-k} ,
$$
and that 
$$
\sum_{k=0}^\infty |||\kappa_k||| \leq K \sum_{k=0}^\infty \sum_{n\geq k} {n \choose k} |||g_n||| \ |z|^{n-k} ,
$$
where $K = \sup_{x \in G} \left|e^{ih(x)\theta(x)z}\right| < \infty$. Since
$$
\sum_{k=0}^\infty \sum_{n\geq k} {n \choose k} |||g_n||| \ |z|^{n-k} \leq \sum_{n=0}^\infty |||g_n||| (1+|z|)^n < \infty,
$$
this shows that $\sigma^\theta_z(f) \in B_0$.

(b) + (c) + (d):  Since $f$ is analytic for $\overline{\sigma}$ it follows that $\overline{\sigma}_z(f)(x) = \sigma_z(f(x)) \in \mathcal N_\psi$ for all $z,x$. Since $\overline{\sigma}_z(f)\in C_0(G,\mathcal N_\psi)$ it follows that $G \ni x \mapsto \sigma_z(f(x))$ and $G \ni x \mapsto \Lambda_\psi(\sigma_z(f(x))$ are continuous. 
\end{proof}

Let $C_c(G,H_\psi)$ denote the vector space of continuous compactly supported functions from $G$ to $H_\psi$, and $L^2(G,H_\psi)$ the Hilbert space obtained by completing $C_c(G,H_\psi)$ with respect to the inner product
$$
\left< f,g\right> = \int_G \left<f(x),g(x)\right> \ \mathrm d x .
$$
Let $f \in  B_0$. Then
\begin{equation}\label{24-06-22}
\Lambda_\psi(\alpha_{x^{-1}}(f(x))) = \Delta(x)^{\frac{1}{2}}e^{\frac{\theta(x)\beta}{2}}W_{x^{-1}}\Lambda_\psi(f(x)) ,
\end{equation}
and it follows from Lemma \ref{16-01-22} and (d) of Lemma \ref{23-06-22} that \eqref{24-06-22} depends continuously on $x$. We can therefore define 
$$
\Lambda_0 : B_0 \to L^2(G,H_\psi)
$$ 
such that
$$
\Lambda_0(f)(x) := \Lambda_\psi(\alpha_{x^{-1}}(f(x))).
$$

Let $C_c(G,A)^\psi$ be the subspace of $C_c(G,A)$ consisting of the functions $h \in C_c(G,A)$ with the properties that
\begin{itemize}
\item $h(s) \in \mathcal N_\psi \cap \mathcal N_\psi^*$ for all $s \in G$, and
\item  the functions $G \ni s \mapsto \Lambda_\psi(h(s))$ and $G \ni s \mapsto \Lambda_\psi(h(s)^*)$ are both continuous. 
\end{itemize}
Note that $C_c(G,A)^\psi \subseteq C_0(G,\mathcal N_\psi)$.

\begin{lemma}\label{24-06-22d} Let $f \in C_c(G,A)^\psi$ and define $\overline{R}_n(f) \in C_c(G,A)$ such that
$$
\overline{R}_n(f)(x) = \sqrt{\frac{n}{\pi}} \int_\mathbb R e^{-n s^2} \sigma_s(f(x)) \ \mathrm d s .
$$
Then $\overline{R}_n(f) \in B_0 \cap B_0^\sharp$, $\lim_{n \to \infty} \overline{R}_n(f) = f$ in $B$ and
$$
\lim_{n \to \infty} \Lambda_0(\overline{R}_n(f)) = \Lambda_0(f)
$$
in $L^2(G,H_\psi)$.
\end{lemma}
\begin{proof} Since
$$
\overline{R}_n(f) =  \sqrt{\frac{n}{\pi}}\int_\mathbb R e^{-n s^2} \overline{\sigma}_s(f) \ \mathrm d s
$$
it follows from Lemma \ref{24-11-21} that $\overline{R}_n(f) \in B_0$ and that $\lim_{n \to \infty} \overline{R}_n(f) = f$ in $C_0(G,\mathcal N_\psi)$. Thus $\lim_{n \to \infty} \overline{R}_n(f)(x) = f(x)$ uniformly and compactly supported in $x$. Hence $\lim_{n \to \infty} \overline{R}_n(f) = f$ in $L^1(G,A)$ and therefore also in $B$. Since $\sigma$ and $\alpha$ commute we find that
$$
\overline{R}_n(f)^\sharp = \overline{R}_n(f^\sharp) .
$$
It follows from Lemma \ref{20-01-22e} that $f^\sharp \in C_c(G,A)^\psi$, and we conclude therefore that $\overline{R}_n(f)\in B_0^\sharp$. By using Lemma \ref{12-02-22b} we find that 
\begin{equation}\label{24-06-22f}
\begin{split}
& \Lambda_0(\overline{R}_n(f))(x) = \Lambda_\psi\left( \alpha_{x^{-1}}(\overline{R}_n(f)(x))\right) \\
& = \Lambda_\psi\left( \sqrt{\frac{n}{\pi}} \int_\mathbb R e^{-n s^2} \sigma_s(\alpha_{x^{-1}}(f(x))) \ \mathrm d s \right)\\
& = \sqrt{\frac{n}{\pi}}  \int_\mathbb Re^{-n s^2} \Lambda_{\psi}(\sigma_s(\alpha_{x^{-1}}(f(x)))) \ \mathrm d s\\
& = \sqrt{\frac{n}{\pi}}  \int_\mathbb Re^{-n s^2} U^\psi_s\Lambda_{\psi}(\alpha_{x^{-1}}(f(x))) \ \mathrm d s   . \\
\end{split}
\end{equation}
By Lemma \ref{24-11-21}  
\begin{equation}\label{24-06-22e}
\lim_{n \to \infty} \sqrt{\frac{n}{\pi}}  \int_\mathbb Re^{-n s^2} U^\psi_s\Lambda_{\psi}(\alpha_{x^{-1}}(f(x))) \ \mathrm d s = \Lambda_{\psi}(\alpha_{x^{-1}}(f(x)) )
\end{equation}
for each $x$. Note that
\begin{align*}
&\left\|  \sqrt{\frac{n}{\pi}}  \int_\mathbb Re^{-n s^2} U^\psi_s\Lambda_{\psi}(\alpha_{x^{-1}}(f(x))) \ \mathrm d s     -    \sqrt{\frac{n}{\pi}}  \int_\mathbb Re^{-n s^2} U^\psi_s\Lambda_{\psi}(\alpha_{x_0^{-1}}(f(x_0))) \ \mathrm d s    \right\| \\
&\leq \sqrt{\frac{n}{\pi}}  \int_\mathbb Re^{-n s^2} \left\| \Lambda_{\psi}(\alpha_{x^{-1}}(f(x))) -   \alpha_{x_0^{-1}}(f(x_0))))\right\| \ \mathrm d s \\
& =  \left\| \Lambda_{\psi}(\alpha_{x^{-1}}(f(x)) -   \alpha_{x_0^{-1}}(f(x_0)))\right\|  \\
& = \left\|\Delta(x)^{\frac{1}{2}}e^{\frac{\theta(x)\beta}{2}} W_{x^{-1}}\Lambda_\psi(f(x)) - \Delta(x_0)^{\frac{1}{2}}e^{\frac{\theta(x_0)\beta}{2}} W_{x_0^{-1}}\Lambda_\psi(f(x_0))\right\| .
\end{align*} 
This estimate and an application of Lemma \ref{16-01-22} shows that the sequence of functions
$$
G \ni x \to   \sqrt{\frac{n}{\pi}}  \int_\mathbb Re^{-n s^2} U^\psi_s\Lambda_{\psi}(\alpha_{x^{-1}}(f(x))) \ \mathrm d s, \ n \in \mathbb N ,
$$
is equicontinuous in the sense that for each $x_0 \in G$ and $\epsilon >0$ there is an open neighbourhood $\Omega_{x_0}$ of $x_0$ with the property that
\begin{align*}
&\left\|  \sqrt{\frac{n}{\pi}}  \int_\mathbb Re^{-n s^2} U^\psi_s\Lambda_{\psi}(\alpha_{x^{-1}}(f(x))) \ \mathrm d s     -    \sqrt{\frac{n}{\pi}}  \int_\mathbb Re^{-n s^2} U^\psi_s\Lambda_{\psi}(\alpha_{x_0^{-1}}(f(x_0))) \ \mathrm d s    \right\| \leq \epsilon 
\end{align*} 
for all $n$ and all $x \in \Omega_{x_0}$. Since all the functions vanish outside the compact support of $f$, it follows that the convergence in \eqref{24-06-22e} is in fact uniform in $x$. Using \eqref{24-06-22f} and \eqref{24-06-22e} we find that 
$$
\lim_{n \to \infty} \Lambda_0(\overline{R}_n(f))(x) = \Lambda_\psi( \alpha_{x^{-1}}(f(x))
$$
uniformly and compactly supported in $x$. Hence $\lim_{n \to \infty} \Lambda_0(\overline{R}_n(f)) =\Lambda_0(f)$ in $L^2(G,H_\psi)$.
\end{proof}

From the GNS-triple $(H_\psi,\Lambda_\psi,\pi_\psi)$ of $\psi$, cf. Section \ref{06-02-22}, we get the representation $\widehat{\pi}_\psi$ of $B$ on $L^2(G,H_\psi)$ defined such that
$$
(\widehat{\pi}_\psi(f)\xi)(x) = \int_G \pi_\psi(\alpha_{x^{-1}}(f(y)))\xi(y^{-1}x) \ \mathrm d y  
$$
for $f \in C_c(G,A)$ and $\xi \in L^2(G,H_\psi)$, cf. \cite{Pe} or \cite{Wi}.

\begin{lemma}\label{16-01-22dx} $B_0 \cap B_0^\sharp$ is dense in $B$ and $\Lambda_0(B_0 \cap B_0^\sharp)$ is dense in $L^2(G,H_\psi)$.
\end{lemma}
\begin{proof} For $f \in C_c(G), \ a \in A$, set
$$
(f \otimes a)(x) := f(x)a .
$$
Since $\mathcal N_\psi \cap \mathcal N_\psi^*$ is dense in $A$ by Lemma \ref{07-12-21c} it follows that
$$
\left\{ f \otimes a : \   a \in \mathcal N_\psi \cap \mathcal N_\psi^*, \ f\in C_c(G)\right\} 
$$
spans a subspace $S$ of $C_c(G,A)$ such that every every element of $C_c(G,A)$ can be approximated in $L^1(G,A)$, and hence also in $B$, by elements of $S$. Since $S \subseteq C_c(G,A)^\psi$ it follows from Lemma \ref{24-06-22d} that $\overline{R}_n(S) \subseteq B_0 \cap B_0^\sharp$ and that $\lim_{n \to \infty} \overline{R}_n(f) = f$ in $B$ for all $f \in S$. This shows that $B_0 \cap B_0^\sharp$ is dense in $B$. Let $\xi \in L^2(G,H_\psi)$ and $\epsilon >0$. Since $\Lambda_\psi(\mathcal N_\psi \cap \mathcal N_\psi^*)$ is dense in $H_\psi$ by Lemma \ref{07-12-21c} it follows that there is an element $g \in S$  such that
$$
\int_G \left\| \xi(x) - \Lambda_\psi(g(x))\right\|^2 \ \mathrm d x \leq \epsilon .
$$
Set $f(x) := \alpha_{x}(g(x))$. Then $f \in C_c(G,A)^\psi$ and
$$
\Lambda_0(f)(x) = \Lambda_\psi(g(x))
$$
for all $x$. Hence $\left\|\xi - \Lambda_0(f)\right\| \leq \sqrt{\epsilon}$. Since $\overline{R}_n(f) \in B_0\cap B_0^\sharp$ and
$$
\lim_{n \to \infty} \Lambda_0(\overline{R}_n(f)) = \Lambda_0(f)
$$
by Lemma \ref{24-06-22d}, we conclude that  $\Lambda_0(B_0 \cap B_0^\sharp)$ is dense in $L^2(G,H_\psi)$.
\end{proof}

\begin{lemma}\label{16-01-22f} $\widehat{\pi}_\psi(a)\Lambda_0(b) = \Lambda_0(a \star b)$ for all $a,b \in B_0$.
\end{lemma}
\begin{proof} By using Lemma \ref{12-02-22b} it follows that
\begin{align*}
&\Lambda_\psi\left( \int_G \alpha_{x^{-1}}(a(y)) \alpha_{x^{-1}y}(b(y^{-1}x))\ \mathrm d y \right) 
=  \int_G \Lambda_\psi(\alpha_{x^{-1}}(a(y)) \alpha_{x^{-1}y}(b(y^{-1}x)))  \ \mathrm d y .
\end{align*}
Hence
\begin{align*}
&\widehat{\pi}_\psi(a)\Lambda_0(b)(x) \\
&=  \int_G \pi_\psi(\alpha_{x^{-1}}(a(y))) \Lambda_0(b)(y^{-1}x)  \ \mathrm d y  \\
& =  \int_G \Lambda_\psi(\alpha_{x^{-1}}(a(y)) \alpha_{x^{-1}y}(b(y^{-1}x)))\ \mathrm d y  \\
& = \Lambda_\psi\left( \int_G \alpha_{x^{-1}}(a(y)) \alpha_{x^{-1}y}(b(y^{-1}x)) \ \mathrm d y \right) \\
& =\Lambda_\psi\left(  \alpha_{x^{-1}}\left(\int_G a(y) \alpha_{y}(b(y^{-1}x)) \ \mathrm d y\right) \right) \\
& = \Lambda_0(a \star b)(x). \ \ 
\end{align*} 
\end{proof}

Define unitaries $T_g$ on $L^2(G,H_\psi)$ such that
$$
(T_g\xi)(x) := \Delta(g)^{\frac{1}{2}} W_g\xi(xg) .
$$
Let $\{q_j\}_{j \in J}$ be a net in $C_c(G)$ such that 
\begin{itemize}
\item $\supp q_{j'} \subseteq \supp q_j$ when $j \leq j'$, 
\item $  \bigcap_j \supp q_j = \{e\}$, 
\item $0 \leq q_j(x) = q_j(x^{-1})$ for all $j$ and $x$, and 
\item $\int_G q_j(x) \ \mathrm dx = 1$ for all $j$.
\end{itemize} 
Thanks to Proposition \ref{08-02-22a} we get from the GNS representation $(H_\psi,\Lambda_\psi,\pi_\psi)$ of $\psi$ the nets $\{u_j\}_{j \in I}$ in $A$ and $\{\rho_j\}_{j \in I}$ in $B(H_\psi)$ with  the properties specified in Lemma \ref{09-02-22x}. For $(j,i) \in I \times I$ define $\xi_{j,i} \in C_c(G,H_\psi)$ such that
$$
\xi_{j,i}(x) :=  q_j(x) \Delta(x)e^{\frac{\theta(x)\beta}{2}}\Lambda_\psi(u_i) .
$$
Define $\tilde{\rho}_i \in B(L^2(G,H_\psi))$ by
$$
(\tilde{\rho}_i\xi)(x) := \rho_i\xi(x) .
$$
Since $G \ni g \mapsto q_j(g)\tilde{\rho}_iT_g\eta$ is continuous from $G$ into $L^2(G,H_\psi)$ and 
$$
\int_G \left\|q_j(s) \tilde{\rho}_i T_s \eta \right\| \ \mathrm d s \leq \|\eta\|
$$
for all $\eta \in L^2(G,H_\psi)$, we can define
$$
\rho'_{j,i} := \int_G  q_j(s) \tilde{\rho}_i T_s \ \mathrm d s  \ \in \ B(L^2(G,H_\psi)) ,
$$
cf. Lemma \ref{12-02-22} and \eqref{28-02-22a}.
We consider $I\times I$ as a directed set such that $(j,i)\leq (j',i')$ when $j \leq j'$ and $i \leq i'$ in $I$. Then $\{\xi_{j,i}\}$ and $\left\{\rho'_{j,i}\right\}$ are both nets indexed by $I\times I$.

\begin{lemma}\label{17-01-22} The following hold.
\begin{itemize}
\item[(a)] $\left\|\rho'_{j,i}\right\| \leq 1$ for all $(j,i)$.
\item[(b)] $\rho'_{j,i}\Lambda_0(b) = \widehat{\pi}_\psi(b)\xi'_{j,i}$ for all $b \in B_0$ and all $(j,i)$, where $\xi'_{j,i}(t) = \xi_{j,i}(t^{-1})$.
\item[(c)] $\lim_{(j,i) \to \infty} \rho'_{j,i} = 1$ in the strong operator topology.
\end{itemize}
\end{lemma}
\begin{proof} (a) Since $\left\|\tilde{\rho}_i\right\| \leq 1$ and $\left\|T_s\right\| \leq 1$ we find that
\begin{align*}
& \left\|\rho'_{j,i}\right\| \leq \int_G \left\| q_j(s) \tilde{\rho}_i T_s\right\| \ \mathrm d s  \leq \int_G q_j(s) \ \mathrm d s = 1.
\end{align*}

(b) This is a direct calculation:
\begin{align*}
&(\rho'_{j,i}\Lambda_0(b))(x) = \int_G  q_j(s) (\tilde{\rho}_i T_s\Lambda_0(b))(x) \ \mathrm d s\\
& =  \int_G \Delta(s)^{\frac{1}{2}} q_j(s) {\rho}_i W_s \Lambda_0(b)(xs) \ \mathrm d s\\
& = \int_G \Delta(s)^{\frac{1}{2}} q_j(s) {\rho}_i W_s \Lambda_\psi( \alpha_{s^{-1}x^{-1}}(b(xs))) \ \mathrm d s\\
& = \int_G   q_j(s) \Delta(s) e^{\frac{\theta(s)\beta}{2}} {\rho}_i \Lambda_\psi( \alpha_{x^{-1}}(b(xs))) \ \mathrm d s \\
& = \int_G   q_j(s)  \Delta(s) e^{\frac{\theta(s)\beta}{2}} \pi_\psi( \alpha_{x^{-1}}(b(xs)))\Lambda_\psi(u_i) \ \mathrm d s\\
& = \int_G   \pi_\psi( \alpha_{x^{-1}}(b(xs)))\xi_{j,i}(s) \ \mathrm d s \\
& = \int_G   \pi_\psi( \alpha_{x^{-1}}(b(w)))\xi_{j,i}(x^{-1}w) \ \mathrm d w  \\
& = \int_G   \pi_\psi( \alpha_{x^{-1}}(b(w)))\xi'_{j,i}(w^{-1}x) \ \mathrm d w \\
& = (\widehat{\pi}_\psi(b)\xi'_{j,i})(x).
\end{align*}

(c) Let $\xi \in C_c(G,H_\psi)$. Since $C_c(G,H_\psi)$ is dense in $L^2(G,H_\psi)$ and $\left\|\rho'_{j,i}\right\| \leq 1$ for all $j,i$, it suffices to show that $\lim_{(j,i) \to \infty} \rho'_{j,i}\xi = \xi$ in $L^2(G,H_\psi)$. Fix $j_0 \in I$. Then
\begin{align*}
& \left\|\rho'_{j,i}\xi - \xi\right\| = \left\| \int_G q_j(s) \tilde{\rho}_i T_s \xi - q_j(s) \xi \ \mathrm d s\right\| \\
& \leq \sup_{s \in \supp q_j} \left\| \tilde{\rho}_i T_s\xi - \xi\right\| \mu(\supp q_{j_0})
\end{align*}
when $j \geq j_0$.
It is therefore enough to show that
\begin{equation*}
\lim_{(j,i) \to \infty} \sup_{s \in \supp q_j} \left\| \tilde{\rho}_i T_s\xi - \xi\right\| = 0.
\end{equation*}
Set $M = (\supp \xi) (\supp q_{j_0})$. For $s \in \supp q_j, \ j \geq j_0$,
\begin{align*}
&\left\|\tilde{\rho}_iT_s\xi -\xi\right\|^2 = \int_G \left\| \Delta(s)^{\frac{1}{2}} \rho_i W_s \xi(xs) - \xi(x)\right\|^2 \ \mathrm d x \\
& \leq \mu(M) \sup_{s \in \supp q_j}\sup_{x \in M} \left\| \Delta(s)^{\frac{1}{2}} \rho_i W_s \xi(xs) - \xi(x)\right\|^2
\end{align*}
and it is therefore enough to show that
$$
 \lim_{(j,i) \to \infty} \sup_{s \in \supp q_j}\sup_{x \in M} \left\| \Delta(s)^{\frac{1}{2}} \rho_i W_s \xi(xs) - \xi(x)\right\| = 0.
 $$
Let $\epsilon > 0$. Note that
$$
\left\{ \Delta(s)^{\frac{1}{2}}W_s \xi(xs) : \ s \in \supp q_{j_0}, x \in (\supp \xi)(\supp q_{j_0}) \right\}
$$
is a compact subset of $H_\psi$. Since $\lim_{i \to \infty} \rho_i = 1$ in the strong operator topology there is a $i_0$ such that
$$
\sup_{s \in \supp q_j}\sup_{x \in M} \left\| \Delta(s)^{\frac{1}{2}} \rho_i W_s \xi(xs) -  \Delta(s)^{\frac{1}{2}} W_s \xi(xs)  \right\| \leq \epsilon
$$
when $i\geq i_0$ and $j \geq j_0$. Similarly, since $\lim_{s \to e} \Delta(s)^{\frac{1}{2}}W_s = 1$ in the strong operator topology and $\bigcap_j \supp q_j = \{e\}$ there is a $j_1 \geq j_0$ such that
$$
\sup_{s \in \supp q_j}\sup_{x \in M} \left\|\Delta(s)^{\frac{1}{2}} W_s \xi(xs)  - \xi(xs) \right\| \leq \epsilon
$$
when $i\geq i_0$ and $j \geq j_1$. It follows that
$$
 \sup_{s \in \supp q_j}\sup_{x \in M} \left\| \Delta(s)^{\frac{1}{2}} \rho_i W_s \xi(xs) - \xi(x)\right\| \leq 2 \epsilon +  \sup_{s \in \supp q_j}\sup_{x \in M} \left\| \xi(xs) - \xi(x)\right\|
 $$
 when $i \geq i_0$ and $j \geq j_1$. The uniform continuity of $\xi$ gives then a $j_2 \geq j_1$ such that
 $$
\sup_{s \in \supp q_j}\sup_{x \in M} \left\| \Delta(s)^{\frac{1}{2}} \rho_i W_s \xi(xs) - \xi(x)\right\| \leq 3 \epsilon
$$
when $i \geq i_0$ and $j \geq j_2$.

\end{proof}

\begin{lemma}\label{17-01-22a} $\Lambda_0 : B_0\cap B_0^\sharp \to L^2(G,H_\psi)$ is closable.
\end{lemma}
\begin{proof} Let $\{b_n\}$ be a sequence in $B_0\cap B_0^\sharp $ such that $\lim_{n \to \infty} b_n = 0$ in $B$ and $\lim_{n \to \infty} \Lambda_0(b_n) = \eta$ in $L^2(G,H_\psi)$. It follows from (a) and (b) in Lemma \ref{17-01-22} that
$$
\rho'_{j,i}\eta = \lim_{n \to \infty} \rho'_{j,i} \Lambda_0(b_n) = \lim_{n \to \infty} \widehat{\pi}_\psi(b_n) \xi'_{j,i} = 0 
$$
for all $(j,i)$. Thanks to (c) in Lemma \ref{17-01-22} this implies that $\eta = 0$.
\end{proof}

Let $\Lambda : D(\Lambda) \to L^2(G,H_\psi)$ be the closure of $\Lambda_0 : B_0\cap B_0^\sharp \to  L^2(G,H_\psi)$.


\begin{lemma}\label{09-02-22i} $(L^2(G,H_\psi),\Lambda, \widehat{\pi}_\psi)$ is a GNS representation of $A \rtimes_{\alpha ,r} G$ with the properties (A), (B), (C), (D) from Section \ref{GNS-KMS} relative to the flow $\sigma^\theta$.
\end{lemma}
\begin{proof} It follows from Lemma \ref{16-01-22f} and Lemma \ref{16-01-22dx} that $(L^2(G,H_\psi),\Lambda, \widehat{\pi}_\psi)$ is a GNS representation of $A \rtimes_{\alpha ,r} G$.  To see that $\widehat{\pi}_\psi$ is non-degenerate, let $b \in B_0\cap B_0^\sharp$, let $\{a_i\}_{i \in I}$ be an approximate unit in $A$ and let $\{q_j\}_{j \in J}$ the net in $C_c(G)$ used above to define $\xi_{j,i}$. Define $h_{j,i} \in C_c(G,A)$ by
$$
h_{j,i}(x) := q_j(x)\alpha_x(a_i) .
$$
Then
\begin{align*}
& \widehat{\pi}_\psi(h_{j,i})\Lambda_0(b)(x) = 
\int_G q_j(y)\pi_\psi(a_i)\Lambda_\psi(\alpha_{x^{-1}y}(b(y^{-1}x))) \ \mathrm d y .
\end{align*}
Let $\epsilon >0$. Since $\pi_{\psi}$ is non-degenerate and $b$ is compactly supported, it follows that 
$$
\lim_{i \to \infty} \pi_\psi(a_i)\Lambda_\psi(\alpha_{x^{-1}y}(b(y^{-1}x))) =\Lambda_\psi(\alpha_{x^{-1}y}(b(y^{-1}x)))
$$ 
uniformly for $y \in \supp q_1$ and $x\in (\supp q_1)(\supp b)$. There is therefore an $i_0$ such that
$$ 
\left\|\widehat{\pi}_\psi(h_{j,i})\Lambda_0(b)(x)  - \int_G q_j(y)\Lambda_\psi(\alpha_{x^{-1}y}(b(y^{-1}x))) \ \mathrm d y\right\| \leq \epsilon
$$
for all $x,j$ when $i \geq i_0$. It follows from the properties of $\{q_j\}$, and because $b \in B_0$, that
$$
\left\| \int_G q_j(y)\Lambda_\psi(\alpha_{x^{-1}y}(b(y^{-1}x)))  \ \mathrm d y - \Lambda_\psi(\alpha_{x^{-1}}(b(x)))\right\| \leq \epsilon
$$
for all $x\in G$ when $j$ is large enough. We conclude therefore that 
$$
\widehat{\pi}_\psi(h_{j,i})\Lambda_0(b)(x)
$$ 
converges, for $(j,i) \to \infty$, to $\Lambda_\psi\left( \alpha_{x^{-1}}(b(x))\right) = \Lambda_0(b)(x)$ in $H_\psi$, uniformly and compactly supported in $x$. Hence 
$$
\lim_{(j,i) \to \infty} \widehat{\pi}_\psi(h_{j,i})\Lambda_0(b) = \Lambda_0(b)
$$
in $L^2(G,H_\psi)$. This shows that $\widehat{\pi}_\psi$ is non-degenerate because $\Lambda_0(B_0\cap B_0^\sharp)$ is dense in $L^2(G,H_\psi)$ by Lemma \ref{16-01-22dx}. $\Lambda$ is closed by construction, and hence $(L^2(G,H_\psi),\Lambda, \widehat{\pi}_\psi)$ has properties (A) and (B) from Section \ref{GNS-KMS}. We define unitaries $V^\psi_t \in B(L^2(G,H_\psi))$ such that
$$
(V^\psi_t\xi)(x) := e^{i\theta(x)t} U^\psi_{t}\xi(x) ,
$$
where $U^\psi_t$ is the unitary from Lemma \ref{17-11-21e}. Since $V^\psi_t\circ\Lambda_0 = \Lambda_0 \circ \sigma^\theta_t$ it follows by continuity that $\sigma^\theta_t(D(\Lambda)) = D(\Lambda)$ and $V^\psi_t\Lambda(b) = \Lambda(\sigma^\theta_t(b))$ for all $t \in \mathbb R$ and all $b \in D(\Lambda)$. This shows that $(L^2(G,H_\psi),\Lambda, \widehat{\pi}_\psi)$ also has properties (C) and (D) from Section \ref{GNS-KMS}.
\end{proof}

\begin{lemma}\label{11-02-22d}  For $f,g \in B_0 \cap B_0^\sharp$,
$$
\left<\Lambda_0(f),\Lambda_0(g)\right> =  \left< \Lambda_0\left( \sigma^\theta_{-i \frac{\beta}{2}}(g)^\sharp  \right), \Lambda_0\left( \sigma^\theta_{-i \frac{\beta}{2}}(f)^\sharp \right)    \right> .
$$
\end{lemma}
\begin{proof} This is a direct calculation. 
\begin{equation}\label{26-06-22}
\begin{split}
& \left<\Lambda_0(f),\Lambda_0(g)\right>  = \int_G \left<\Lambda_\psi(\alpha_{x^{-1}}(f(x))), \Lambda_\psi(\alpha_{x^{-1}}(g(x)))\right> \ \mathrm d x \\
& = \int_G \psi( \alpha_{x^{-1}}(g(x)^*) \alpha_{x^{-1}}(f(x)) \ \mathrm d x \\
&= \int_G  \Delta(x) e^{\theta(x)\beta}\psi(g(x)^*f(x)) \ \mathrm d x 
\end{split}
\end{equation}
where the last identity follows from \eqref{04-02-22a}. Set $h := \sigma^\theta_{-i \frac{\beta}{2}}(f)$ and $ k := \sigma^\theta_{-i \frac{\beta}{2}}(g)$, and note that $h,k \in B_0$ by (e) of Lemma \ref{23-06-22}. Since $f,g \in B_0^\sharp$ and $h^\sharp =  \sigma^\theta_{i \frac{\beta}{2}}(f^\sharp)$, it follows also that $h \in B_0^\sharp$. Similarly, $k \in B_0^\sharp$. Note that by (f) of Lemma \ref{23-06-22},
\begin{align*}
&h(x) =\sigma^\theta_{-i \frac{\beta}{2}}(f)(x) = e^{i \theta(x) (-i\frac{\beta}{2})} \sigma_{-i\frac{\beta}{2}}(f(x)) = e^{\frac{\theta(x)\beta}{2}} \sigma_{-i\frac{\beta}{2}}(f(x))  ,
\end{align*}
and similarly $k(x) = e^{\frac{\theta(x)\beta}{2}} \sigma_{-i\frac{\beta}{2}}(g(x))$. Hence
\begin{align*}
&h^\sharp (x) = \Delta(x)^{-1}\alpha_x(h(x^{-1})^*) = \Delta(x)^{-1} \alpha_x\left( e^{-\frac{\theta(x)\beta}{2}} \sigma_{-i \frac{\beta}{2}}(f(x^{-1}))^*\right) \\
&=\Delta(x)^{-1} e^{-\frac{\theta(x)\beta}{2}} \alpha_x\left( \sigma_{i \frac{\beta}{2}}(f(x^{-1})^*)\right) .
\end{align*}
Similarly,
$$
k^\sharp (x) =\Delta(x)^{-1} e^{-\frac{\theta(x)\beta}{2}} \alpha_x\left( \sigma_{i \frac{\beta}{2}}(g(x^{-1})^*)\right).
$$
Thus
\begin{align*}
& \left< \Lambda_0\left( \sigma^\theta_{-i \frac{\beta}{2}}(g)^\sharp \right), \Lambda_0\left( \sigma^\theta_{-i \frac{\beta}{2}}(f)^\sharp \right)    \right>  = \left< \Lambda_0(k^\sharp),\Lambda_0(h^\sharp)\right> \\
& = \int_G \left<\Lambda_{\psi}(\alpha_{x^{-1}}(k^\sharp(x))) , \Lambda_{\psi}(\alpha_{x^{-1}}(h^\sharp(x)))\right> \ \mathrm d x\\
 & = \int_G \psi(\alpha_{x^{-1}}(h^\sharp(x))^*\alpha_{x^{-1}}(k^\sharp(x))) \ \mathrm d x  \\
 & = \int_G \Delta(x)e^{\theta(x)\beta} \psi\left( h^\sharp(x)^*k^\sharp(x)\right) \ \mathrm dx \\
& = \int_G \Delta(x)^{-2} e^{-\theta(x)\beta}  \psi\left(  \sigma_{i \frac{\beta}{2}}(f(x^{-1})^*)^* \sigma_{i \frac{\beta}{2}}(g(x^{-1})^*) \right) \ \mathrm d x .
\end{align*}
Now we use that $\psi$ by assumption is a $\beta$-KMS weight for $\sigma$ to conclude from Corollary \ref{24-06-22b} that
$$
 \psi\left(  \sigma_{i \frac{\beta}{2}}(f(x^{-1})^*)^* \sigma_{i \frac{\beta}{2}}(g(x^{-1})^*)\right) = \psi(g(x^{-1})^*f(x^{-1})).
 $$
 Hence 
\begin{align*}
& \left< \Lambda_0\left( \sigma^\theta_{-i \frac{\beta}{2}}(g)^\sharp \right), \Lambda_0\left( \sigma^\theta_{-i \frac{\beta}{2}}(f)^\sharp \right)    \right>  \\
& =  \int_G \Delta(x)^{-2} e^{-\theta(x)\beta} \psi(g(x^{-1})^*f(x^{-1}))  \ \mathrm d x \\
& = \int_G \Delta(x^{-1}) e^{\theta(x^{-1})\beta}\psi(g(x^{-1})^*f(x^{-1}))  \ \Delta(x^{-1}) \mathrm d x  \\
&= \int_G \Delta(y) e^{\theta(y)\beta} \psi(g(y)^*f(y))  \ \mathrm d y\\
& = \left<\Lambda_0(f),\Lambda_0(g)\right> ,
\end{align*}
where the last equality sign comes from \eqref{26-06-22}.
\end{proof}

\begin{lemma}\label{19-08-22} The GNS representation $(L^2(G,H_\psi),\Lambda, \widehat{\pi}_\psi)$ of $A \rtimes_{r,\alpha} G$ has property (F) of Lemma \ref{09-02-22d} with respect to $\beta$ and the flow $\sigma^\theta$.
\end{lemma}
\begin{proof} Set $S = B_0 \cap B_0^\sharp$. Then $S \subseteq D(\Lambda) \cap D(\Lambda)^*$ by definition of $\Lambda$ and $S \subseteq \mathcal A_{\sigma^\theta}$ by (a) of Lemma \ref{23-06-22}. Since $\sigma^\theta_z(f^\sharp) = \sigma^\theta_{\overline{z}}(f)^\sharp$ when $f \in B_0 \cap B_0^\sharp$ it follows from (e) of Lemma \ref{23-06-22} that $\sigma^\theta_z(S) \subseteq S$ for all $z \in \mathbb C$. Thus $S \subseteq \mathcal M^{\sigma^\theta}_\Lambda$. Since $\Lambda$ is the closure of $\Lambda_0 : B_0 \cap B_0^\sharp \to L^2(G,H_\psi)$ it follows that the first item of property (F) holds. The second holds by Lemma \ref{11-02-22d}.
\end{proof}

It follows from Lemma \ref{19-08-22}, Lemma \ref{09-02-22i} and Lemma \ref{09-02-22d} that the GNS representation $(L^2(G,H_\psi),\Lambda, \widehat{\pi}_\psi)$ has all the properties (A) - (E) required in Proposition \ref{08-02-22a}. It follows therefore from Proposition \ref{08-02-22a} that $(L^2(G,H_\psi),\Lambda, \widehat{\pi}_\psi)$ is isomorphic to the GNS-triple of a $\beta$-KMS weight for $\sigma^\theta$. We denote this weight by $\widehat{\psi}$ and call it \emph{the dual weight} of $\psi$. Thus, by definition,
$$
\widehat{\psi}(x) = \sup_{\omega \in \mathcal F_\Lambda} \omega(x)
$$
for all $x \in (A\rtimes_{r,\alpha} G)^+$,
where
$$
\mathcal F_\Lambda = \left\{ \omega \in (A\rtimes_{r,\alpha} G)^*_+ : \ \omega(a^*a) \leq \left< \Lambda(a),\Lambda(a)\right> \ \forall a \in D(\Lambda) \right\} .
$$
We note that $\mathcal N_{\widehat{\psi}} = D(\Lambda)$ and $\widehat{\psi}(f^\sharp \star f) = \left<\Lambda_0(f),\Lambda_0(f)\right>$ for $f \in B_0 \cap B_0^\sharp$ by \eqref{29-03-23}. By polarization the calculation from  the proof of Lemma \ref{11-02-22d} shows
\begin{equation}\label{11-02-22e}
\widehat{\psi}(g^\sharp \star f) = \int_G \Delta(x) e^{\theta(x)\beta} \psi(g(x)^*f(x)) \ \mathrm d x
\end{equation}
when $f,g \in B_0 \cap B_0^\sharp$. 

To summarize the fundamental properties of the dual weight we let $C_c(G,A)_\psi$ denote the set of elements $f \in C_c(G,A)$ such that $f(s) \in \mathcal N_\psi$ for all $s \in G$ and the function $G \ni s \mapsto \Lambda_\psi(f(s))$ is continuous. Thus
$$
C_c(G,A)^\psi \subseteq C_c(G,A)_\psi.
$$

\begin{thm}\label{19-08-22a} $\widehat{\psi}$ is a $\beta$-KMS weight for $\sigma^\theta$ with the following properties:
\begin{itemize}
\item[(a)] $\mathcal N_{\widehat{\psi}} = D(\Lambda)$, 
\item[(b)]  $C_c(G,A)_\psi \subseteq \mathcal N_{\widehat{\psi}}$, 
\item[(c)] $C_c(G,A)^\psi$ is a core for $\Lambda_{\widehat{\psi}}$,
\item[(d)] $\widehat{\psi}(f^\sharp \star g) = \int_G \Delta(s)e^{\theta(s)\beta} \psi(f(s)^*g(s)) \ \mathrm ds$ for all $f,g \in C_c(G,A)_\psi$, and
\item[(e)] $\widehat{\psi}(f^\sharp \star f) = \psi\left( (f^\sharp \star f)(0)\right)$ for all $f \in G_c(G,A)_\psi$.
\end{itemize}
\end{thm}
\begin{proof} (a) holds by construction. To prove (b) and (d) we consider first the case where $f,g \in C_c(G,A)^\psi$. It follows then from Lemma \ref{24-06-22d} that $\overline{R}_n(f), \overline{R}_n(g) \in B_0 \cap B_0^\sharp$ and hence from \eqref{11-02-22e} that
\begin{equation}\label{25-06-22x}
\begin{split}
& \left< \Lambda_{\widehat{\psi}}(\overline{R}_n(f)), \Lambda_{\widehat{\psi}}(\overline{R}_n(g))\right> =  \left< \Lambda_{0}(\overline{R}_n(f)), \Lambda_{0}(\overline{R}_n(g))\right> \\
& = \int_G \Delta(s)e^{\theta(s) \beta} \psi(\overline{R}_n(g)(s)^*\overline{R}_n(f)(s)) \mathrm d s \\
& = \int_G \Delta(s)e^{\theta(s) \beta} \left< \Lambda_\psi(R_n(f(s))),\Lambda_\psi(R_n(g(s)))\right> \ \mathrm d s .
\end{split}
\end{equation}
By Lemma \ref{24-06-22d} $\lim_{n \to \infty} \Lambda_{0}(\overline{R}_n(f)) = \Lambda_0(f)$ and $\lim_{n \to \infty} \overline{R}_n(f) = f$ in $B$. Since $\Lambda_{\widehat{\psi}}$ is closed it follows that $f \in \mathcal N_{\widehat{\psi}}$ and $\lim_{n \to \infty} \Lambda_{0}(\overline{R}_n(f)) = \Lambda_{\widehat{\psi}}(f)$. Similarly, $g \in \mathcal N_{\widehat{\psi}}$ and 
$$
\lim_{n \to \infty} \left< \Lambda_{\widehat{\psi}}(\overline{R}_n(f)), \Lambda_{\widehat{\psi}}(\overline{R}_n(g))\right> = \left< \Lambda_{\widehat{\psi}}(f), \Lambda_{\widehat{\psi}}(g)\right> .
$$ 
 On the other hand it follows from Lemma \ref{24-11-21g} via an application of Lebesgue's theorem on dominated convergence that 
\begin{align*}
&\lim_{n \to \infty} \int_G \Delta(s)e^{\theta(s) \beta} \left< \Lambda_\psi(R_n(f(s))),\Lambda_\psi(R_n(g(s)))\right> \ \mathrm d s \\
&= \int_G \Delta(s)e^{\theta(s) \beta} \left< \Lambda_\psi(f(s)),\Lambda_\psi(g(s))\right> \ \mathrm d s \\
& =  \int_G \Delta(s)e^{\theta(s)\beta} \psi(g(s)^*f(s)) \ \mathrm ds .
\end{align*}
This proves that (d) holds when $f,g \in C_c(G,A)^\psi$. 

Assume that $f,g \in C_c(G,A)_\psi$. Let $\{e_m\}_{m \in \mathcal I}$ be the net in $\mathcal N_\psi \cap \mathcal N_\psi^*$ from Lemma \ref{16-12-21a}. We can then construct a sequence $\{m_k\}_{k=1}^\infty$ in $\mathcal I$ such that
$$
\lim_{k \to \infty} e_{m_k}d(s) = d(s)
$$
uniformly in $s$ for $d \in \left\{f,g\right\}$. Set $d_k(s) := e_{m_k}d(s)$ when $d \in \{f,g\}$. Since $d_k(s)^* = d(s)^*e_{m_k}\in \mathcal N_\psi$ and
$$
\Lambda_\psi(d_k(s)^*) = \pi_\psi(d(s)^*) \Lambda_\psi(e_{m_k}),
$$
it follows that $d_k \in C_c(G,A)_\psi$. In particular, $f_k -f_l \in C_c(G,A)_\psi$ and from what we have just shown it follows that

\begin{equation}\label{31-03-23}
\begin{split}
&\widehat{\psi}\left( (f_k-f_l)^\sharp \star (f_k-f_l)\right) = \widehat{\psi}\left(f_k^\sharp \star (f_k -f_l)\right) + \widehat{\psi}\left( f_l^\sharp \star (f_l -f_k)\right) \\
& = \int_G \Delta(s)e^{\theta(s)\beta} \psi(f_k(s)^*(f_k(s)-f_l(s))) \ \mathrm ds \\
& \ \ \ \ \ \ \ \ \ \ \ +  \int_G \Delta(s)e^{\theta(s)\beta} \psi(f_l(s)^*(f_l(s)-f_k(s))) \ \mathrm ds .
\end{split}
\end{equation}
When $k \geq l$ an application of the Cauchy-Schwarz inequality gives 
\begin{align*}
& \left|\int_G \Delta(s)e^{\theta(s)\beta} \psi(f_k(s)^*(f_k(s)-f_l(s))) \ \mathrm ds\right|^2 \\
&\leq \int_G \Delta(s)e^{\theta(s)\beta} \psi(f(s)^*e_{m_k}^2f(s))\ \mathrm ds   \\
& \ \ \ \ \ \ \ \ \ \ \ \ \ \times  \int_G \Delta(s)e^{\theta(s)\beta}  \psi(f(s)^*(e_{m_k} -e_{m_l})^2 f(s)) \ \mathrm ds \\
&\leq \int_G \Delta(s)e^{\theta(s)\beta} \psi(f(s)^*f(s))\ \mathrm ds   \\
& \ \ \ \ \ \ \ \ \ \ \ \ \ \times  \int_G \Delta(s)e^{\theta(s)\beta}  \psi(f(s)^*(e_{m_k} -e_{m_l}) f(s)) \ \mathrm ds  . 
\end{align*}
An application of Lebesques theorem on dominated (or monotone) convergence implies that 
\begin{align*}
&\lim_{k \to \infty}  \int_G \Delta(s)e^{\theta(s)\beta} \psi(f(s)^*e_{m_k} f(s)) \ \mathrm ds  = \int_G \Delta(s)e^{\theta(s)\beta} \psi(f(s)^*f(s)) \ \mathrm ds .
\end{align*}
It follows that 
$$
\int_G \Delta(s)e^{\theta(s)\beta} \psi(f_k(s)^*(f_k(s)-f_l(s))) \ \mathrm ds
$$ 
can be made arbitrarily small by choosing $k,l$ large enough. The same argument shows that this is also true for $\int_G \Delta(s)e^{\theta(s)\beta} \psi(f_l(s)^*(f_l(s)-f_k(s)) ) \ \mathrm ds$, and it follows then from \eqref{31-03-23} that $\{\Lambda_{\widehat{\psi}}(f_k)\}$ is a Cauchy sequence in $\mathcal N_{\widehat{\psi}}$. Since $\lim_{ k \to \infty} f_k = f$ in $B$ and $\Lambda_{\widehat{\psi}}$ is closed, it follows that $f \in \mathcal N_{\widehat{\psi}}$ and $\lim_{k \to \infty} \Lambda_{\widehat{\psi}}(f_k) = \Lambda_{\widehat{\psi}}(f)$. In particular, this proves (b), and (d) follows because
\begin{align*}
& \widehat{\psi}(f^\sharp \star g) = \left< \Lambda_{\widehat{\psi}}(g), \Lambda_{\widehat{\psi}}(f)\right> = \lim_{k \to \infty} \left< \Lambda_{\widehat{\psi}}(g_k), \Lambda_{\widehat{\psi}}(f_k)\right> \\
& = \lim_{k \to \infty}  \widehat{\psi}(f_k^\sharp \star g_k) \\
& = \lim_{k \to \infty}\int_G \Delta(s)e^{\theta(s)\beta} \psi(f_k(s)^*g_k(s)) \ \mathrm ds \\
& = \lim_{k \to \infty}\int_G \Delta(s)e^{\theta(s)\beta} \psi(f(s)^*(e_{m_k})^2g(s)) \ \mathrm ds .
\end{align*}
Then (d) follows from Lebesgues theorem on dominated convergence.

 To see that (c) holds note that it follows (b), (c) and (d) of Lemma \ref{23-06-22}, in combination with Lemma \ref{20-01-22e}, that $B_0 \cap B_0^\sharp \subseteq C_c(G,A)^\psi$. Combined with (b) it follows that $B_0 \cap B_0^\sharp \subseteq C_c(G,A)^\psi \subseteq \mathcal N_{\widehat{\psi}}$. Since $B_0 \cap B_0^\sharp$ by definition is a core for $\Lambda_{\widehat{\psi}}$ it follows that so is $C_c(G,A)^\psi$.

Finally, to establish (e), let $f \in C_c(G,A)_\psi$. Then
\begin{align*}
& \widehat{\psi}(f^\sharp \star f) = \int_G \Delta(s)e^{\theta(s)\beta} \psi(f(s)^*f(s)) \ \mathrm ds\\
& = \int_G \Delta(s)^2e^{\theta(s)\beta} \psi(f(s)^*f(s)) \ \mathrm ds^{-1} \\
& =  \int_G \Delta(y)^{-2}e^{-\theta(y)\beta} \psi(f(y^{-1})^*f(y^{-1})) \ \mathrm dy \\
& = \int_G \Delta(y)^{-1}\psi(\alpha_y(f(y^{-1}))^*) \alpha_y(f(y^{-1}))) \ \mathrm dy .
\end{align*}
Let $E$ be a separable $C^*$-subalgebra of $A$ containing the compact set
$$
\left\{ \alpha_y(f(y^{-1})^*) \alpha_y(f(y^{-1})): \ y \in G\right\} .
$$
The restriction of $\psi$ to $E^+$ is also a weight and it follows therefore from Theorem \ref{09-11-21h} that there is a sequence $\{\omega_m\}_{n=1}^\infty$ in $E^*_+$ such that 
$$
\psi(d) = \sum_{n=1}^\infty \omega_n(d) \ \ \forall d \in E^+ .
$$
We can then apply Lebesgues theorem on monotone convergence to conclude that
\begin{align*}
&  \int_G \Delta(y)^{-1}\psi(\alpha_y(f(y^{-1})^*) \alpha_y(f(y^{-1})) \ \mathrm dy \\
& = \int_G \Delta(y)^{-1} \sum_{n=1}^\infty \omega_n(\alpha_y(f(y^{-1})^*) \alpha_y(f(y^{-1})) \ \mathrm dy \\
& =  \sum_{n=1}^\infty \omega_n\left( \int_G \Delta(y)^{-1}\alpha_y(f(y^{-1})^*) \alpha_y(f(y^{-1})) \ \mathrm dy \right) \\
& = \psi\left( (f^\sharp \star f)(0)\right) .
\end{align*}

\end{proof}

\begin{notes} This section is based on the work of Quaegebeur and Verding in \cite{QV} where a somewhat simpler version of the construction is applied to a more general class of weights, called regular weights. As shown in \cite{QV} KMS weights are regular, and one can therefore start with a KMS weight in the construction described in \cite{QV}. It follows from (e) of Theorem \ref{19-08-22a} that the weight $\widehat{\psi}$ constructed here is the dual weight constructed in \cite{QV}. However, in \cite{QV} no flows are considered in the construction of the dual weight and some modifications are necessary to arrange that by starting with a KMS weight, the dual weight is also KMS. In this section we have presented one possible elaboration of the method of Quaegebeur and Verding to this effect. 
\end{notes}

\section{The map defined by the dual KMS weight}\label{13-04-23}

In this section we retain the setting from Section \ref{07-08-22e} and investigate the map of KMS weights resulting from the construction $\psi \mapsto \hat{\psi}$ obtained in that section.

\bigskip

Let $\KMS(\sigma,\beta)_{\theta}$ denote the set of $\beta$-KMS weights for $\sigma$ with the scaling property \eqref{04-02-22a} with respect to $\alpha$. Theorem \ref{19-08-22a} gives us a map
\begin{equation}\label{08-06-22}
D : \KMS(\sigma,\beta)_{\theta} \to \KMS(\sigma^\theta,\beta),
\end{equation}
defined such that $D(\psi) = \widehat{\psi}$. For the study of $D$ we collect here first a few general observations before we move on to special cases.

\begin{lemma}\label{07-06-22a} The map $D$ in \eqref{08-06-22} is injective.
\end{lemma}
\begin{proof} Assume that $\psi, \phi \in  \KMS(\sigma,\beta)_{\theta}$ and that $\widehat{\psi} = \widehat{\phi}$. Let $a \in \mathcal N_\psi \cap \mathcal N_\phi$, and let $h \in C_c(G)$ be a function such that
$$
\int_G \Delta(s) e^{\theta(s)\beta} \left|h(s)\right|^2 \ \mathrm ds = 1 .
$$
In the notation from Lemma \ref{16-01-22dx}, set $f := h \otimes a$.
It follows then from (d) of Theorem \ref{19-08-22a} that
\begin{align*}
& \psi(a^*a) = \widehat{\psi}(f^\sharp \star f) = \widehat{\phi}(f^\sharp \star f) = \phi(a^*a) .
\end{align*} 
Hence $\psi = \phi$ by Corollary \ref{31-03-22d}.
\end{proof}

To see one reason why $D$ is generally not surjective, let $l : G \to \mathbb T$ be a continuous homomorphism. There is an automorphism $\gamma_l \in \Aut (A\rtimes_{r,\alpha} G)$ such that $\gamma_l(f) \in C_c(G,A)$ and
$$
\gamma_l(f)(x) := l(x)f(x)
$$
when $f \in C_c(G,A)$. 

\begin{lemma}\label{20-03-22} Let $\psi \in \KMS(\sigma,\beta)_{\theta}$. Then $\widehat{\psi} \circ \gamma_l = \widehat{\psi}$.
\end{lemma}
\begin{proof} Note that $\gamma_l$ commutes with $\sigma^{\theta}$. It follows therefore from Corollary \ref{01-10-23} that $\widehat{\psi} \circ \gamma_l$ is a $\beta$-KMS weight for $\sigma^\theta$ since $\widehat{\psi}$ is. Note that $\gamma_l(C_c(G,A)^\psi) = C_c(G,A)^\psi$ and that (d) of Theorem \ref{19-08-22a} shows that $\widehat{\psi}\circ \gamma_l(a^*a) = \widehat{\psi}(a^*a)$ when $a \in C_c(G,A)_\psi$. It follows therefore from (c) of Theorem \ref{19-08-22a} that $\widehat{\psi}\circ \gamma_l(a^*a) = \widehat{\psi}(a^*a)$ for all $a \in  \mathcal N_{\widehat{\psi}}$, and then from Corollary \ref{31-03-22d} that $\widehat{\psi} \circ \gamma_l = \widehat{\psi}$. 
 \end{proof}

In general, not all KMS weights for $\sigma^\theta$ are invariant under $\gamma_l$ for all $l : G \to \mathbb T$. The most elementary example of this is sketched in Example \ref{09-10-23c}. Hence Lemma \ref{20-03-22} shows that $D$ does not in general map onto $\KMS(\sigma^\theta,\beta)$. 

\begin{example}\label{09-10-23c} \rm Let $G = \mathbb Z_2$, $A = \mathbb C$ and let $\alpha$ be the trivial action of $\mathbb Z_2$ on $\mathbb C$. With these choices the crossed product $A \rtimes_\alpha G$ is a copy of $\mathbb C^2$. All flows on $\mathbb C^2$ are trivial and all states on $\mathbb C^2$ are KMS states for the trivial action on $\mathbb C^2$. Let $l : \mathbb Z_2 \to \mathbb T$ be the only nontrivial character on $\mathbb Z_2$. The corresponding automorphism $\gamma_l$ of $\mathbb C \rtimes \mathbb Z_2$ is non-trivial and must therefore be the automorphism of $\mathbb C^2$ which interchanges the two copies of $\mathbb C$. Not all states of $\mathbb C^2$ are invariant under this automorphism. 
\end{example}

When $G$ is abelian there is a continuous action 
$$
\hat{\alpha} : \widehat{G} \to \Aut (A \rtimes_\alpha G) 
$$
of the dual group $\widehat{G}$ on $A \rtimes_\alpha G$, known as \emph{the dual action}. It is defined such that
$$
\hat{\alpha}_\gamma(f)(x) := \overline{\gamma(x)}f(x)
$$
when $\gamma \in \widehat{G}$ and $f \in C_c(G,A)$, cf. \cite{Wi}. It follows from Lemma \ref{20-03-22} that $\widehat{\psi}$ is $\hat{\alpha}$-invariant:

\begin{lemma}\label{08-06-22b} Assume that $G$ is abelian and let $\psi \in \KMS(\sigma,\beta)_{\theta}$. Then $\widehat{\psi} \circ \hat{\alpha}_\gamma = \widehat{\psi}$ for all $\gamma \in \widehat{G}$.
\end{lemma}

Later we shall need the following slight strengthening of (c) in Theorem \ref{19-08-22a}.

\begin{lemma}\label{21-06-22f} Let $S$ be a subspace of $\mathcal N_\psi$ which is a core for $\Lambda_\psi$. Then $C_c(G)\otimes S$ is a core for $ \Lambda_{\widehat{\psi}}$.
\end{lemma}
\begin{proof} Let $f \in C_c(G,A)_\psi$. Let $n \in \mathbb N$ and choose an open subset $U$ with compact closure in $G$ containing the support of $f$. For each $s \in G$ there is $a_s \in S$ such that $\left\|a_s -f(s)\right\| < \frac{1}{n}$ and $\left\|\Lambda_\psi(a_s) - \Lambda_\psi(f(s))\right\| < \frac{1}{n}$. A partition of unity argument gives us a finite collection $\{\varphi_i\}_{i\in I}$ of functions in $C_c(G)$, all with supports in $U$, and elements $a_i \in S$, $i \in I$, such that $0 \leq \varphi_i \leq 1$ for all $i$, $f(s) = \sum_{i\in I} \varphi_i(s)f(s)$ and $\sum_{i \in I} \varphi_i(s) \leq 1$ for all $s \in G$, and
$$
\left\|a_i - f(x)\right\| \leq \frac{1}{n}
$$
and
$$
\left\|\Lambda_\psi(a_i) - \Lambda_\psi(f(x))\right\| \leq \frac{1}{n}
$$
for all $i$ and all $x \in \supp \varphi_i$. Set $X_n := \sum_{i \in I} \varphi_i \otimes a_i \in C_c(G) \otimes S$. Then $X_n \in C_c(G,A)_\psi$ and $$
\left\|X_n - f\right\| \leq \left\|X_n -f\right\|_{L^1(G,A)} \leq  \sum_{i \in I} \int_G \varphi_i(s) \left\|a_i - f(s)\right\| \ \mathrm d s \leq \frac{\mu(U)}{n} .
$$
Furthermore, (d) of Theorem \ref{19-08-22a} gives the estimate
\begin{align*}
&\left\|\Lambda_{\widehat{\psi}}(X_n - f)\right\|^2 \\
& = \int_G \Delta(s)e^{\theta(s)\beta} \left\|\Lambda_\psi(X_n(s)) - \Lambda_\psi(f(s))\right\|^2 \ \mathrm d s\\
&  \leq \int_U \Delta(s)e^{\theta(s)\beta} \left(\sum_{i \in I} \varphi_i(s) \left\|\Lambda_\psi(a_i) - \Lambda_\psi(f(s))\right\|\right)^2 \ \mathrm d s\\ 
& \leq K \frac{1}{n^2} ,
\end{align*}
where $K = \int_U \Delta(s)e^{\theta(s)\beta} \ \mathrm d s$. It follows that $\lim_{n \to \infty} X_n = f$ in $A \rtimes_{\alpha,r} G$ and $\lim_{n \to \infty} \Lambda_{\widehat{\psi}}(X_n) = \Lambda_{\widehat{\psi}}(f)$. This completes the proof because $C_c(G,A)_\psi$ is a core for $\Lambda_{\widehat{\psi}}$ by (b) and (c) of Theorem \ref{19-08-22a}.
\end{proof}

\subsection{$G$ abelian. An application of Takai duality}\label{section8.2.1}

We consider in this section the map \eqref{08-06-22} under the additional assumption that $G$ is abelian. While $D$ is injective by Lemma \ref{07-06-22a} it is generally not surjective since it follows from Lemma \ref{08-06-22b} that it maps into the subset 
$$
\KMS(\sigma^\theta,\beta)^{\hat{\alpha}}
$$
of $\KMS(\sigma^\theta,\beta)$ consisting of the $\beta$-KMS weights for $\sigma^\theta$ that are invariant under the dual action $\hat{\alpha}$. We denote the resulting map by $D_1$. In this section we prove

\begin{thm}\label{08-06-22aA} Assume that $G$ is abelian. Let $\beta \in \mathbb R$. The map
$$
D_1:  \KMS(\sigma,\beta)_{\theta} \to  \KMS(\sigma^\theta,\beta)^{\hat{\alpha}}
$$ 
is a bijection.
\end{thm}

For the proof of Theorem \ref{08-06-22aA} we use the map $D$ from \eqref{08-06-22} with a different choice of the ingredients. Namely, we take the  $A$ from Section \ref{07-08-22e} to be $ A \rtimes_{\alpha} G$, the $\alpha$ from Section \ref{07-08-22e} to be $\hat{\alpha}$, the $\sigma$ of Section \ref{07-08-22e} to be $\sigma^\theta$ and the $\theta$ of Section \ref{07-08-22e} to be zero.\footnote{We have dropped the $r$ from $A \rtimes_{r,\alpha} G$ in the notation because $G$ is abelian and hence amenable so that the reduced and the full crossed products are isomorphic, cf. Theorem 7.7.7 in \cite{Pe}.} This is possible because ${\sigma}^\theta$ and $\hat{\alpha}$ commute: ${\sigma}^\theta_t \circ \hat{\alpha}_\gamma = \hat{\alpha}_\gamma\circ {\sigma}^\theta_t$ for all $t \in \mathbb R, \ \gamma \in \hat{G}$. With these choices and because $G$ is abelian and hence unimodular, the relation \eqref{04-02-22a} holds when $\psi$ is $\hat{\alpha}$-invariant. The construction in Section \ref{07-08-22e} gives therefore rise to a map
$$
D_2 : \KMS({\sigma}^\theta,\beta)^{\hat{\alpha}} \to \KMS(\overline{\sigma^\theta},\beta)^{\hat{\hat{\alpha}}} ,
$$
where $\overline{\sigma^\theta}$ is the flow on $(A \rtimes_\alpha G) \rtimes_{\hat{\alpha}} \widehat{G}$ defined such that
$$
\overline{\sigma^\theta}_t(f)(\gamma) = \sigma^\theta_t(f(\gamma)) 
$$
when $f \in C_c(\widehat{G}, A \rtimes_\alpha G)$ and $\hat{\hat{\alpha}}$ denotes the double dual action; the dual of the dual action $\hat{\alpha}$. Let $\mathbb K(L^2(G))$ be the $C^*$-algebra of compact operators on the Hilbert space $L^2(G)$. By Takai duality there is an isomorphism
$$
\Psi : (A\rtimes_\alpha G) \rtimes_{\hat{\alpha}} \widehat{G} \to A \otimes \mathbb K(L^2(G)) ,
$$
such that
\begin{equation}\label{22-06-22h}
\Psi \circ \hat{\hat{\alpha}}_g =  (\alpha_g \otimes \Ad  \rho_g) \circ \Psi
\end{equation}
 for all $g$, where $\rho$ denotes the right-regular representation of $G$ on $L^2(G)$, cf. Theorem 7.1 on page 190 of \cite{Wi}. We are going to use the explicit description of $\Psi$ as it is presented by Williams in the proof of Theorem 7.1 of \cite{Wi}, and while we will explicitly spell out the selected facts from Williams' proof we shall depend on, we will leave the verification of these items to the reader. The first fact we shall need is the observation that
\begin{equation}\label{22-06-22i}
\Psi \circ \overline{\sigma^\theta}_t = (\sigma_{t} \otimes \Ad W_t) \circ \Psi 
\end{equation}
for all $t \in \mathbb R$, where $W$ is the unitary representation of $\mathbb R$ on $L^2(G)$ such that
$$
(W_t\eta)(x) := e^{i \theta(x)t}\eta(x)
$$ 
for $\eta \in L^2(G)$. This implies that the map $\phi \mapsto \phi\circ \Psi^{-1}$ is a bijection 
$$
\KMS(\overline{\sigma^\theta},\beta)^{\hat{\hat{\alpha}}}  \to \KMS(\sigma \otimes \Ad W, \beta)^{\alpha \otimes \Ad \rho} .
$$
By Stone's theorem $W_t = e^{it H_\theta}$ where $H_\theta$ is the self-adjoint operator on $L^2(G)$ given by multiplication by $\theta$. For each $\beta \in \mathbb R$ we denote by $\kappa_\beta$ the $\beta$-KMS weight for $\Ad W$ on $\mathbb K(L^2(G))$ given by the formula
\begin{equation*}\label{30-06-22(2)}
\kappa_\beta(a) = \sup \left\{\Tr\left(e^{\frac{-\beta H_\theta}{2}} f(H_\theta) a f(H_\theta) e^{-\frac{\beta H_\theta}{2}}\right) : \ f \in C_c(\mathbb R), \ 0 \leq f \leq 1 \right\} ,
\end{equation*}
cf. Theorem \ref{17-08-22c}. The crucial point of the proof is the observation that
\begin{equation}\label{28-08-22}
D_2 \circ D_1(\psi) \circ \Psi^{-1} = \psi \otimes \kappa_\beta
\end{equation}
for all $\psi \in \KMS(\sigma,\beta)_{\theta}$. It is for the verification of \eqref{28-08-22} we really have to go close to the arguments for Theorem 7.1 of \cite{Wi}, but before doing so let's see how to use \eqref{28-08-22} to complete the proof of Theorem \ref{08-06-22aA}. Let $\phi \in  \KMS(\sigma \otimes \Ad W, \beta)^{\alpha \otimes \Ad \rho}$. By Lemma \ref{09-06-22x} there is a $\beta$-KMS weight $\psi$ for $\sigma$ such that $\phi = \psi \otimes \kappa_\beta$. We claim that $\psi$ has the scaling property \eqref{04-02-22a} with respect to $\alpha$. Since the modular function is trivial when $G$ is abelian this means that
\begin{equation}\label{12-04-23}
\psi \circ \alpha_x = e^{-\theta(x)\beta} \psi  \ \ \forall x \in G .
\end{equation}
To establish \eqref{12-04-23} note first of all that $\psi \circ \alpha_x$ is a $\beta$-KMS weight for $\sigma$ because $\psi$ is and $\sigma$ commutes with $\alpha$, cf. Corollary \ref{01-10-23}. Similarly, since $\Ad W$ commutes with $\Ad \rho$, because $W_t$ and $\rho_x$ commute up to multiplication by a scalar, the weight $\kappa_\beta \circ \Ad \rho_x$ is a $\beta$-KMS weight for $\Ad W$.
Let $h \in C_c(G)$. Then the compact operator $\theta_{h,h} \in \mathbb K(L^2(G))$ given by
$$
L^2(G) \ni \psi \mapsto \left< \psi,h\right>h
$$
is in $\mathcal M_{\kappa_\beta \circ \Ad_{\rho_x}}^+ \cap \mathcal M_{\kappa_\beta}^+$ and a direct calculation shows that
\begin{equation}\label{09-10-23f}
\kappa_\beta \circ \Ad \rho_x\left(\theta_{h,h}\right) = e^{\beta \theta(x)} \kappa_\beta\left(\theta_{h,h}\right) .
\end{equation}
Indeed,
\begin{align*}
& \kappa_\beta \circ \Ad \rho_x(\theta_{h,h}) = \kappa_\beta (\theta_{\rho_x h,\rho_x h})  = \sup \left\{ \Tr (\theta_{k,k}) : \ f \in C_c(\mathbb R), \ 0 \leq f \leq 1 \right\}
\end{align*}
where $k \in C_c(G)$ is the function 
$$
k(g) : =  e^{-\frac{\beta}{2} \theta(g)} f(\theta(g)) h(g+x) .
$$
Since $k(g) = (\rho_x k')(g)$, where 
$$
k'(g) := e^{-\frac{\beta}{2} \theta(g-x)} f(\theta(g-x)) h(g),
$$
we find that
\begin{align*}
&  \sup \left\{ \Tr (\theta_{k,k}) : \ f \in C_c(\mathbb R), \ 0 \leq f \leq 1 \right\} \\
&= \sup \left\{ \Tr (\rho_x\theta_{k',k'}\rho_x^*) : \ f \in C_c(\mathbb R), \ 0 \leq f \leq 1 \right\} \\
&  =\sup \left\{ \Tr (\theta_{k',k'}) : \ f \in C_c(\mathbb R), \ 0 \leq f \leq 1 \right\} \\
& =  e^{\beta \theta(x)} \sup \left\{ \Tr (\theta_{k'',k''}) : \ f \in C_c(\mathbb R), \ 0 \leq f \leq 1 \right\} ,
\end{align*}
where $k''(g) =  e^{-\frac{\beta}{2} \theta(g)} f'(\theta(g)) h(g)$ with $f'(t) := f(t - \theta(x))$. Since
\begin{align*}
&  \sup \left\{ \Tr (\theta_{k'',k''}) : \ f \in C_c(\mathbb R), \ 0 \leq f \leq 1 \right\} \\
&=  \sup \left\{\Tr\left(e^{\frac{-\beta H_\theta}{2}} f'(H_\theta) \theta_{h,h} f'(H_\theta) e^{-\frac{\beta H_\theta}{2}}\right) : \ f \in C_c(\mathbb R), \ 0 \leq f \leq 1 \right\} \\
& =  \sup \left\{\Tr\left(e^{\frac{-\beta H_\theta}{2}} f(H_\theta) \theta_{h,h} f(H_\theta) e^{-\frac{\beta H_\theta}{2}}\right) : \ f \in C_c(\mathbb R), \ 0 \leq f \leq 1 \right\} \\
&= \kappa_\beta(\theta_{h,h})
\end{align*}
we get \eqref{09-10-23f}.
 
Let $a \in \mathcal N_{\psi \circ \alpha_x} \cap \mathcal N_\psi = \alpha_x^{-1}(\mathcal N_\psi) \cap \mathcal N_\psi$. It follows from Theorem \ref{08-06-22f} that $a^*a \otimes \theta_{h,h} \in \mathcal M_{(\psi \otimes \kappa_\beta) \circ (\alpha_x \otimes \Ad \rho_x)}$ and
\begin{align*}
& (\psi \otimes \kappa_\beta) \circ (\alpha_x \otimes \Ad \rho_x)(a^*a \otimes \theta_{h,h}) = \psi \circ \alpha_x(a^*a)\kappa_\beta \circ \Ad \rho_x(\theta_{h,h})\\
&= e^{\beta \theta(x)}\psi \circ \alpha_x(a^*a) \kappa_\beta(\theta_{h,h}) .
\end{align*}
Since $\psi \otimes \kappa_\beta = \phi$ is $\alpha \otimes \Ad \rho$-invariant we have also that
\begin{align*}
&(\psi \otimes \kappa_\beta) \circ (\alpha_x \otimes \Ad \rho_x)(a^*a \otimes \theta_{h,h}) \\
&= \psi \otimes \kappa_\beta (a^*a \otimes \theta_{h,h}) = \psi(a^*a)\kappa_\beta(\theta_{h,h})
\end{align*}
and by comparing the two expressions we find that $\psi \circ \alpha_x(a^*a) = e^{-\beta \theta(x)}\psi(a^*a)$. By Corollary \ref{31-03-22d} this means that \eqref{12-04-23} holds, i.e. $\psi \in \KMS(\sigma,\beta)_{\theta}$. It follows then from \eqref{28-08-22} that
$$
D_2 \circ D_1(\psi) \circ \Psi^{-1} = \psi \otimes \kappa_\beta = \phi .
$$
Since $\phi \in \KMS(\sigma \otimes \Ad W,\beta)^{\alpha \circ \Ad \rho}$ was arbitrary, this shows that the map 
$$
\KMS(\sigma, \beta)_{\theta} \ni \psi \mapsto D_2 \circ D_1(\psi) \circ \Psi^{-1} 
$$
is surjective onto $\KMS(\sigma \otimes \Ad W, \beta)^{\alpha \otimes \Ad \rho}$. Hence $D_2 \circ D_1$ is also surjective from $\KMS(\sigma, \beta)_{\theta}$ onto $\KMS(\overline{\sigma^\theta},\beta)^{\hat{\hat{\alpha}}}$, and we can therefore complete the proof of Theorem \ref{08-06-22aA} as follows: Since $D_1$ is injective by Lemma \ref{07-06-22a} it remains only to show that $D_1$ is surjective. Let $\phi \in \KMS(\sigma^\theta,\beta)^{\hat{\alpha}}$. Then $D_2(\phi) \in \KMS(\overline{\sigma^\theta},\beta)^{\hat{\hat{\alpha}}}$ and since $D_2 \circ D_1$ is surjective there is $\psi \in \KMS(\sigma, \beta)_{\theta}$ such that $D_2 \circ D_1(\psi) = D_2(\phi)$. Since $D_2$ is injective by Lemma \ref{07-06-22a} this means that $\phi = D_1(\psi)$, completing the proof of Theorem \ref{08-06-22aA} contingent to

\begin{lemma}\label{29-08-22} \eqref{28-08-22} holds for all $\psi \in \KMS(\sigma,\beta)_\theta$.
\end{lemma}
\begin{proof} As in the proof of Theorem 7.1 in \cite{Wi} we write $\Psi$ as the composition of maps, but here only of two. As in \cite{Wi} we denote by $lt \otimes \alpha$ the action of $G$ on $C_0(G,A)$ such that
$$
(lt \otimes \alpha)_y(f)(x) = \alpha_y(f(x-y)) 
$$
for $f \in C_0(G,A)$. Then $\Psi = \Psi_2 \circ \Psi_1$ where
$$
\Psi_1 : (A\rtimes_\alpha G) \rtimes_{\hat{\alpha}} \widehat{G} \to C_0(G,A) \rtimes_{lt \otimes \alpha} G
$$
and 
$$
\Psi_2 :
 C_0(G,A) \rtimes_{lt \otimes \alpha} G \to A \otimes \mathbb K(L^2(G)) .
 $$
are isomorphisms described explicitly in \cite{Wi}. In the notation of \cite{Wi} $\Psi_1 = \Phi_2 \circ \Phi_1$ and $\Psi_2 = \Phi_4 \circ \Phi_3$. 

 We note that when $f \in C_c(\widehat{G}), \ g \in C_c(G)$ and $a \in A$,
\begin{equation}\label{22-06-22a}
\Psi_1( f \otimes g \otimes a)(s,r) = g(s) \hat{f}(r-s) a 
\end{equation}
where 
$$
\hat{f}(x) = \int_{\widehat{G}} \overline{\gamma(x)} f(\gamma) \ \mathrm d \gamma .
$$
 To describe the property of $\Psi_2$ we shall need, consider $f,g \in L^2(G)$ and let $\theta_{f,g} \in \mathbb K(L^2(G))$ be the rank 1 operator on $L^2(G)$ defined such that
$$
\theta_{f,g}(\eta) = \left(\int_{G} \overline{g (s)} \eta(s)  \ \mathrm ds \right) \ f
$$
for $\eta \in L^2(G)$. Given $f \in C_c(G), \ g\in C_0(G) \cap L^2(G), \ a \in A$, define $F \in C_c(G,C_0(G,A))$ such that
\begin{equation}\label{22-06-22}
F(s)(r) = f(r)\overline{g(r -s)}\alpha_{r}(a) .
\end{equation}
Then 
$$
\Psi_2(F) = a \otimes \theta_{f,g} ,
$$
cf. \cite{Wi}. We need to see how $\psi$ transforms under $\Psi_1$ and $\Psi_2$. For this define a weight $\psi_1$ on $C_0(G,A)$ such that
$$
\psi_1(f) = \int_G  \psi(f(x)) \ \mathrm d x
$$
when $f \in C_0(G,A)^+$. Then $\psi_1$ is a $\beta$-KMS weight for $\id_{C_0(G)}\otimes \sigma$. Indeed, if $f \in C_0(G,A)$ is analytic for the flow $\id_{C_0(G)}\otimes \sigma$ it follows that $f(x)$ is analytic for $\sigma$ and
$$
(\id_{C_0(G)}\otimes \sigma)_{-i \frac{\beta}{2}}(f)(x) = \sigma_{-i \frac{\beta}{2}}(f(x))
$$
for all $x \in G$. Therefore $\psi_1$ satisfies condition (2) in Kustermans' theorem, Theorem \ref{24-11-21d}, relative to $\id_{C_0(G)}\otimes \sigma$. Furthermore, $\psi_1$ is scaled by $lt \otimes \alpha$ in the same way as $\psi$ is scaled by $\alpha$:
$$
\psi_1 \circ  (lt \otimes \alpha)_x  = e^{-\theta(x)\beta} \psi_1 .
$$
By Theorem \ref{19-08-22a} the dual weight of $\psi_1$ is a $\beta$-KMS weight $\phi_2$ on 
$$
C_0(G,A) \rtimes_{lt \otimes \alpha} G
$$
for the flow $\sigma'$ which is defined such that 
$$
\sigma'_t(F)(x,x') := e^{i\theta(x)t} \sigma_{t}(F(x,x'))
$$ when 
$$
F \in C_c( G, C_0(G,A)) \subseteq C_0(G,A) \rtimes_{lt \otimes \alpha} G.
$$ 
Note that $\Psi_2 \circ \sigma'_t =  (\sigma_{t} \otimes \Ad W_t) \circ \Psi_2$, which implies that $\phi_2 \circ \Psi_2^{-1}$ is a $\beta$-KMS weight for the flow $\sigma \otimes \Ad W$. By Lemma \ref{09-06-22x} there is therefore $\chi \in \KMS(\sigma,\beta)$ such that
$$
\phi_2 \circ \Psi_2^{-1} = \chi \otimes \kappa_\beta .
$$
Let $a \in \mathcal N_\psi \cap \mathcal N_\chi$ and $f,g \in C_c(G)$. By using that
$$
\Tr(\theta_{h,h}) = \left\|h\right\|^2_{L^2(G)}
$$
when $h \in L^2(G)$, we find that 
\begin{equation}\label{29-06-22x}
\begin{split}
&\chi(a^*a) \left\|f\right\|_{L^2(G)}^2\left\|e^{\frac{-\theta \beta}{2}} g\right\|_{L^2(G)}^2 = \chi\otimes \kappa_\beta \left((a \otimes \theta_{f,g})^*(a \otimes \theta_{f,g})\right) \\
& =\phi_2 \circ \Psi_2^{-1}\left((a \otimes \theta_{f,g})^*(a \otimes \theta_{f,g})\right) = \phi_2(F^*F) ,
\end{split}
\end{equation}
where $F$ is the function \eqref{22-06-22}. To continue the calculation we need the observation that $F$ is in $C_c(G,C_0(G,A))_{\psi_1}$; a fact which boils down to the observation that $F \in C_c(G,C_0(G,A))$ since $f,g \in C_c(G)$, that the function $r \mapsto F(s,r)$ is in $\mathcal N_{\psi_1}$ for each $s \in G$ because $f,g \in C_c(G)$ and $\mathcal N_\psi$ is $\alpha$-invariant, and finally that
$$
\lim_{ s \to s_0} \int_G \psi\left( (F(s,r)- F(s_0,r))^*(F(s,r)-F(s_0,r))\right) \ \mathrm d r = 0
$$
for all $s_0 \in G$. We leave the reader to fill in the details, but by using the fact we conclude from (d) of Theorem \ref{19-08-22a} that
\begin{equation}\label{29-06-22a}
\begin{split}
&\phi_2(F^*F) = \int_G e^{\theta(s) \beta}\psi_1(F(s, \cdot )^*F(s,\cdot)) \ \mathrm ds \\
& = \int_G e^{\theta(s)\beta} \int_G \left|f(r)\right|^2 \left|g(r-s)\right|^2 \psi\left(\alpha_r(a^*a)\right) \ \mathrm dr \ \mathrm d s \\
&= \psi(a^*a) \int_G \int_G e^{\theta(s) \beta}  e^{-\theta(r)\beta} |f(r)|^2|g(r-s)|^2 \ \mathrm d r \ \mathrm d s \\
&= \psi(a^*a)  \left\|f\right\|_{L^2(G)}^2\left\|e^{-\frac{\theta \beta}{2}}g\right\|_{L^2(G)}^2 .
\end{split}
\end{equation}
By comparing \eqref{29-06-22a} and \eqref{29-06-22x} we find that $\chi(a^*a) = \psi(a^*a)$. Since $a \in \mathcal N_\psi \cap \mathcal N_\chi$ was arbitrary it follows from Corollary \ref{31-03-22d} that $\chi = \psi$, and hence 
\begin{equation}\label{22-06-22g}
\phi_2 \circ \Psi_2^{-1} = \psi \otimes \kappa_\beta .
\end{equation}

We claim that $\phi_2 \circ \Psi_1 = D_2\circ D_1(\psi)$. To see this, let $f_1,f_2 \in C_c(\widehat{G}), \ g_1,g_2 \in C_c(G)$ and $a_1,a_2 \in \mathcal N_\psi$. Two application of (d) of Theorem \ref{19-08-22a} show that
\begin{align*}
& D_2 \circ D_1(\psi)((f_1 \otimes g_1 \otimes a_1)^* (f_2 \otimes g_2 \otimes a_2)) \\
&= \psi(a_1^*a_2) \int_{\widehat{G}} \int_G e^{\theta (s)\beta} \overline{f_1(\gamma)}f_2(\gamma) \overline{g_1(s)}g_2(s) \ \mathrm d s \ \mathrm d \gamma  .
\end{align*} 
To compare this with $\phi_2\circ \Psi_1((f_1 \otimes g_1 \otimes a_1)^* (f_2 \otimes g_2 \otimes a_2))$, note that $\hat{f_i} \in C_0(G) \cap L^2(G)$ which implies that 
$$
\Psi_1(f_i \otimes g_i \otimes a_i) \in C_c(G,C_0(G,A))_{\psi_1} ,
$$
$i =1,2$. Using \eqref{22-06-22a} it follows therefore from (d) of Theorem \ref{19-08-22a} that
\begin{align*}
& \phi_2\circ \Psi_1((f_1 \otimes g_1 \otimes a_1)^* (f_2\otimes g_2 \otimes a_2)) \\
&= \psi(a_1^*a_2) \int_{{G}} \int_G e^{\theta(s) \beta} \overline{\hat{f_1}(r-s)} \hat{f_2}(r-s)  \overline{g_1(s)}g_2(s)  \ \mathrm d r \ \mathrm d s\\
&=   \psi(a_1^*a_2) \int_{{G}} \int_G e^{\theta(s) \beta}  \overline{\hat{f_1}(r)} \hat{f_2}(r) \overline{g_1(s)}g_2(s)  \ \mathrm d s \ \mathrm d r.
\end{align*}
Since\footnote{For this formula to hold we assume, as we can, that the Haar measures on $G$ and $\hat{G}$ are properly normalized. See Theorem 1.6.1 in \cite{Ru3}.}
$$
\int_{\widehat{G}} \overline{f_1(\gamma)}f_2(\gamma) \ \mathrm d \gamma =  \int_G \overline{\hat{f_1}(r)} \hat{f_2}(r)  \ \mathrm d r
$$
we conclude that
\begin{align*}
& D_2 \circ D_1(\psi)((f_1 \otimes g_1 \otimes a_1)^* (f_2 \otimes g_2 \otimes a_2)) \\
&=\phi_2\circ \Psi_1((f_1 \otimes g_1 \otimes a_1)^* (f_2\otimes g_2 \otimes a_2)) ,
\end{align*}
 and hence by linearity that $D_2 \circ D_1(\psi)$ and $\phi_2 \circ \Psi_1$ agree on $X^*X$ for all $X \in C_c(\widehat{G}) \otimes C_c(G) \otimes \mathcal N_\psi$. Since $C_c(\widehat{G}) \otimes C_c(G) \otimes \mathcal N_\psi$ is a core for $\Lambda_{D_2 \circ D_1(\psi)}$ by Lemma \ref{21-06-22f}, it  follows from Corollary \ref{21-06-22e} that $D_2 \circ D_1(\psi) =\phi_2 \circ \Psi_1$. In combination with \eqref{22-06-22g} this implies that
\begin{align*}
& \psi \otimes \kappa_\beta = \phi_2 \circ \Psi_2^{-1} = D_2\circ D_1(\psi) \circ \Psi_1^{-1} \circ \Psi_2^{-1} = D_2\circ D_1(\psi) \circ \Psi^{-1} ,
\end{align*}
which is the desired equality \eqref{28-08-22}.
\end{proof}

The proof of Theorem \ref{08-06-22aA} is complete.

\subsection{$G =\mathbb R$, $\alpha = \sigma$ and $\theta = 0$. Proof of an assertion by Kishimoto and Kumjian}\label{8.2.2}

In this section we focus on the version of Theorem \ref{08-06-22aA} where $G = \mathbb R$, $\alpha = \sigma$ and $\theta = 0$. It then takes the following form.

\begin{cor}\label{09-06-22d} Let $\sigma$ be a flow on the $C^*$-algebra $A$ and let $\beta \in \mathbb R$. The map $D$ from \eqref{08-06-22} specialized to the case where $G =\mathbb R$, $\alpha = \sigma$ and $\theta = 0$, is a bijection
$$
D_3 : \KMS(\sigma,\beta) \to \KMS(\underline{\sigma}, \beta)^{\hat{\sigma}} ,
$$
where $\underline{\sigma}$ is the flow on $A \rtimes_\sigma \mathbb R$ defined such that
$$
\underline{\sigma}_t(f)(x) = \sigma_t(f(x))
$$
when $f \in C_c(\mathbb R,A) \subseteq A \rtimes_\sigma \mathbb R$, and $\hat{\sigma}$ is the dual action of $\sigma$.
\end{cor}

We want to compose the map $D_3$ with one of the maps constructed in Section \ref{geninner}. As in Section \ref{05-09-22g} we consider here $A$ as a non-degenerate $C^*$-subalgebra of $B(\mathbb H)$ for some Hilbert space $\mathbb H$. Then $A \rtimes_\sigma \mathbb R$ is represented on the Hilbert space $L^2(\mathbb R,\mathbb H)$ in the following way. For $f \in L^1(\mathbb R,A)$ define $\pi(f) \in B(L^2(\mathbb R,\mathbb H))$ such that
$$
(\pi(f)\xi)(x) = \int_\mathbb R \sigma_{ -x}(f(y))\xi(x-y) \ \mathrm d y  
$$
for $\xi \in L^2(\mathbb R,\mathbb H)$. This is a representation of $L^1(\mathbb R, A)$ and it extends to an injective non-degenerate representation $\pi : A \rtimes_\sigma \mathbb R \to   B(L^2(\mathbb R,\mathbb H))$, cf. \cite{Pe}. As above we will often suppress $\pi$ in the notation and instead identify $ A \rtimes_\sigma \mathbb R$ with its image under $\pi$. In particular, this means that the multiplier algebra $M( A \rtimes_\sigma \mathbb R)$ can and will be identified with
$$
\left\{ m \in B(L^2(\mathbb R,\mathbb H)) : \ m  (A \rtimes_\sigma \mathbb R) \subseteq  A \rtimes_\sigma \mathbb R, \  m^*(A \rtimes_\sigma \mathbb R) \subseteq  A \rtimes_\sigma \mathbb R \right\} ,
$$
cf. Lemma \ref{26-09-22} in Appendix \ref{multipliers}. There is an injective $*$-homomorphism $\iota : A \to M(A\rtimes_\sigma \mathbb R)$ defined such that
\begin{equation}\label{06-08-22a}
(\iota(a)\xi)(t) = \sigma_{-t}(a)\xi(t)
\end{equation}
for all $\xi \in L^2(\mathbb R,\mathbb H)$. Let $f \in C_c(\mathbb R,A)$. Then
$$
\iota(a)\pi(f) = \pi(af),
$$
where
$$
(af)(t) := af(t),
$$
and
$$
\pi(f)\iota(a) = \pi(f \otimes \sigma(a)),
$$
where
$$
(f \otimes \sigma(a))(t) := f(t)\sigma_t(a).
$$
Define unitaries $\lambda_s, s \in \mathbb R$, on $L^2(\mathbb R,\mathbb H)$ such that
$$
\lambda_s \xi(t) := \xi(t-s).
$$

\begin{lemma}\label{20-05-22x} $\lambda_s \in M(A \rtimes_\sigma \mathbb R)$ for all $s \in \mathbb R$ and $\mathbb R \ni s \mapsto \lambda_s$ is continuous for the strict topology.
\end{lemma} 
\begin{proof}

Let $f \in C_c(\mathbb R,A)$. Then
\begin{align*}
& (\lambda_s\pi(f)\xi)(t) = (\pi(f)\xi)(t-s) \\
& = \int_\mathbb R \sigma_{-t +s}(f(y))\xi(t-s-y) \ \mathrm d y\\
& =\int_\mathbb R \sigma_{-t+s}(f(y-s))\xi(t-y) \ \mathrm dy \\
& = (\pi(f_s)\xi)(t),
\end{align*}
where $f_s \in C_c(\mathbb R,A) \subseteq A \rtimes_\sigma \mathbb R$ is the function
$$
f_s(y) := \sigma_s(f(y-s)).
$$ 
Furthermore,
\begin{align*}
& (\pi(f)\lambda_s\xi)(t) = \int_\mathbb R \sigma_{-t}(f(y))(\lambda_s\xi)(t-y) \ \mathrm d y \\
& =  \int_\mathbb R \sigma_{-t}(f(y))\xi(t-y-s)) \ \mathrm d y \\
& =  \int_\mathbb R \sigma_{-t}(f(y-s))\xi(t-y) \ \mathrm d y \\
& = (\pi(f^s)\xi)(t)
\end{align*}
where $f^s \in C_c(\mathbb R,A)$ is the function
$$
f^s(y) := f(y-s).
$$
It follows that $\lambda_s \pi(C_c(\mathbb R,A)) \subseteq \pi(C_c(\mathbb R,A))$ and $ \pi(C_c(\mathbb R,A)) \lambda_s \subseteq \pi(C_c(\mathbb R,A))$, and hence by continuity that
$$
\lambda_s (A \rtimes_\sigma \mathbb R) \subseteq A \rtimes_\sigma \mathbb R
$$
and
$$
(A \rtimes_\sigma \mathbb R)\lambda_s \subseteq A \rtimes_\sigma \mathbb R ;
$$
that is, $\lambda_s \in M(A \rtimes_\sigma \mathbb R)$. To prove that $\lambda_s$ depends continuously on $s$ in the strict topology it suffices to check that $\mathbb R \ni s \mapsto \lambda_s\pi(f)$ is norm-continuous when $f \in C_c(\mathbb R,A)$. For this, note that
\begin{align*}
\lambda_s \pi(f) - \pi(f) = \pi(f_s -f) ,
\end{align*}
and 
\begin{align*}
& \left\|\pi(f_s -f)\right\| \leq \left\| f_s -f\right\|_{L^1(\mathbb R,A)} \leq \int_{\mathbb R}  \left\| {\sigma}_s(f(y-s)) -f(y)\right\| \ \mathrm d y .
\end{align*}
An application of Lebesgues theorem on dominated convergence shows that 
$$
\lim_{s \to 0}\int_{\mathbb R}  \left\| {\sigma}_s(f(y-s))  -f(y)\right\| \ \mathrm d y = 0 .
$$
Hence $\lim_{ s \to 0} \left\|\lambda_s \pi(f) - \pi(f)\right\| = 0$, which is all we need to prove.
\end{proof}

Each of the automorphisms $\underline{\sigma}_t$ extend uniquely to an automorphism of $M(A \rtimes_\sigma \mathbb R)$, cf. Appendix \ref{multipliers}, which we also denote by $\underline{\sigma}_t$.

\begin{lemma}\label{20-05-22ax} $\lambda_s b \lambda_{-s} = \underline{\sigma}_s(b)$ for all $s \in \mathbb R$ and all $b \in M(A \rtimes_\sigma \mathbb R)$, and $\underline{\sigma}_s(\iota(a)) = \iota(\sigma_s(a))$ for all $s \in \mathbb R$ and all $a \in A$.
\end{lemma}
\begin{proof} Let $f \in C_c(\mathbb R,A)$. In the notation from the proof of Lemma \ref{20-05-22x} we have 
\begin{align*}
& \lambda_s\pi(f)\lambda_{-s} = \pi(f_s)\lambda_{-s} = \pi((f_s)^{-s}) .
\end{align*}
The first part of the statement of the lemma follows from this because $(f_s)^{-s}(t) = f_s(t +s) = \sigma_s(f(t)) = \underline{\sigma}_s(f)(t)$, and $\pi(C_c(\mathbb R,A))$ is dense in $M(A \rtimes_\sigma \mathbb R)$ with respect to the strict topology by Lemma \ref{08-02-23} in Appendix \ref{multipliers}. To establish the second part, simply observe that
\begin{align*}
&(\lambda_s\iota(a)\lambda_{-s}\xi)(t) = (\iota(a)\lambda_{-s}\xi)(t-s) = \sigma_{s-t}(a)\lambda_{-s}\xi(t-s) \\
& = \sigma_{-t}(\sigma_s(a))\xi(t) = (\iota(\sigma_s(a))\xi)(t)
\end{align*}
for all $\xi \in L^2(\mathbb R,\mathbb H)$.
\end{proof} 

\begin{cor}\label{30-05-22x} $\underline{\sigma}$ is an inner flow on $A \rtimes_\sigma \mathbb R$.
\end{cor}
\begin{proof} This follows from Lemma \ref{20-05-22ax} and Lemma \ref{20-05-22x}.
\end{proof}

For any $C^*$-algebra $B$ we denote by $\tr(B)$ the set of lower semi-continuous traces on $B$, cf. Definition \ref{03-02-22f}. Since $\underline{\sigma}$ is inner there is by Stone's theorem a self-adjoint operator $H$ on $L^2(\mathbb R,\mathbb H)$ such that 
$$
\underline{\sigma}_t = \Ad e^{it H} 
$$
for all $t \in \mathbb R$. By Lemma \ref{23-01-22} $H$ is an unbounded multiplier of $A \rtimes_\sigma \mathbb R$ in the sense that $f(H) \in M(A\rtimes_\sigma \mathbb R)$ for all $f \in C_0(\mathbb R)$. It follows from Theorem \ref{17-08-22c} that there is a bijective map
$$
\mathcal T : \KMS(\underline{\sigma},\beta) \to \tr(A \rtimes_\sigma \mathbb R) 
$$
defined such that $\mathcal T(\psi) = \tau_\psi$, where 
\begin{equation}\label{10-10-23}
\tau_\psi(a) = \sup \left\{\psi\left(e^{\frac{\beta H}{2}} f(H) a f(H) e^{\frac{\beta H}{2}}\right) : \ f \in C_c(\mathbb R), \ 0 \leq f \leq 1 \right\}.
\end{equation}
Let $\tr(A \rtimes_\sigma \mathbb R)_{\beta}$ denote the set of lower semi-continuous traces $\tau \in \tr(A \rtimes_\sigma\mathbb R)$ that are scaled as follows by the dual action $\hat{\sigma}$:
$$
\tau \circ \hat{\sigma}_t = e^{-\beta t} \tau 
$$
for all $t \in \mathbb R$. Recall that $\hat{\sigma}$ is defined such that $\hat{\sigma}_t(C_c(\mathbb R,A)) \subseteq C_c(\mathbb R,A)$ and
$$
\hat{\sigma}_t(f)(x) = e^{ixt}f(x)
$$
for all $f \in C_c(\mathbb R,A)$. 

 \begin{lemma}\label{06-07-23} Let $g \in C_0(\mathbb R)$. Then $\hat{\sigma}_t(g(H)) = g^{-t}(H)$.
 \end{lemma}
\begin{proof}  As observed in the proof of Lemma \ref{20-05-22x},
$$
\lambda_s f(x) =  \sigma_s(f(x-s))
$$
when $f \in C_c(\mathbb R,A)$ and hence 
\begin{align*}
&(\hat{\sigma}_t(\lambda_s) f)(x) = \hat{\sigma}_t\left( \lambda_s \hat{\sigma}_{-t}(f)\right)(x) \\
& = e^{i xt}(\lambda_s \hat{\sigma}_{-t}(f))(x)   = e^{i xt}\sigma_s(\hat{\sigma}_{-t}(f)(x-s)) \\
& =e^{i xt}e^{ i t(s-x)} \sigma_s(f(x -s)) = e^{i t s}(\lambda_sf)(x) 
\end{align*}
for all $f \in C_c(\mathbb R,A)$. It follows that $\hat{\sigma}_t(\lambda_s) = e^{its}\lambda_s$. Let $h \in L^1(\mathbb R)$. Using Theorem 5.6.36 in \cite{KR}, 
\begin{align*}
&\hat{\sigma}_t(\hat{h}(H)) = \hat{\sigma}_t\left(\int_\mathbb R h(s)\lambda_s \ \mathrm d s\right) \\
&= \int_\mathbb R h(s)\hat{\sigma}_t(\lambda_s) \ \mathrm d s = \int_\mathbb R e^{its} h(s) \lambda_s \ \mathrm d s = \widehat{k}(H) 
\end{align*}
where $k(s):= e^{its} h(s)$. Since $\hat{k} = (\hat{h})^{-t}$ this establishes the desired identity when $g \in \left\{\hat{h} : \ h \in L^1(\mathbb R)\right\}$. The latter space is dense in $C_0(\mathbb R)$ by Lemma \ref{11-10-23} and hence the identity holds for all $g \in C_0(\mathbb R)$ by continuity.
\end{proof}

\begin{lemma}\label{12-06-22x} For all $\beta \in \mathbb R$, 
$$
\mathcal T\left(\KMS(\underline{\sigma},\beta)^{\hat{\sigma}}\right) = \tr(A \rtimes_\sigma \mathbb R)_{\beta} .
$$
\end{lemma}

\begin{proof} Let $\psi \in \KMS (\underline{\sigma},\beta)^{\hat{\sigma}}$ and let $f \in C_c(\mathbb R)$, $0 \leq f \leq 1$. For $a \in (A\rtimes_\sigma \mathbb R)^+$,
\begin{align*}
& \psi \left(e^{\frac{\beta H}{2}}f(H) \hat{\sigma}_t(a)f(H) e^{\frac{\beta H}{2}}\right) \\
& = \psi\left( \hat{\sigma}_t\left( \hat{\sigma}_{-t} (e^{\frac{\beta H}{2}}f(H)) a \hat{\sigma}_{-t} (e^{\frac{\beta H}{2}}f(H))\right)\right)\\
& = \psi\left( \hat{\sigma}_{-t} (e^{\frac{\beta H}{2}}f(H)) a \hat{\sigma}_{-t} (e^{\frac{\beta H}{2}}f(H))\right)\\
& = e^{-t \beta} \psi \left(e^{\frac{\beta H}{2}}f^{t}(H) af^{t}(H) e^{\frac{\beta H}{2}}\right) 
\end{align*}
where the last identity follows from Lemma \ref{06-07-23}.
It follows therefore from the way $\tau_\psi$ is defined, see \eqref{10-10-23}, that $\mathcal T(\psi) \circ \hat{\sigma}_{t} = \tau_\psi\circ \hat{\sigma}_{t} = e^{-\beta t} \mathcal T(\psi)$, i.e. $ \mathcal T(\psi) \in  T(A \rtimes_\sigma \mathbb R)_{\beta}$. 

Let next $\tau \in  T(A \rtimes_\sigma \mathbb R)_{\beta}$. Then
\begin{equation}\label{12-06-22a}
\begin{split}
& \tau \left(e^{\frac{-\beta H}{2}}f(H) \hat{\sigma}_t(a)f(H) e^{-\frac{\beta H}{2}}\right) \\
& = \tau\left( \hat{\sigma}_t \left(\hat{\sigma}_{-t}(e^{-\frac{\beta H}{2}}f(H)) a  \hat{\sigma}_{-t}(f(H)e^{\frac{-\beta H}{2}}) \right)\right)\\
& = e^{-\beta t}\tau\left( \hat{\sigma}_{-t}(e^{-\frac{\beta H}{2}}f(H)) a  \hat{\sigma}_{-t}(f(H)e^{\frac{-\beta H}{2}})\right)\\
& = \tau\left(e^{-\frac{\beta H}{2}}f^{-t}(H)) a  f^{-t}(H)e^{\frac{-\beta H}{2}}\right)\\
\end{split}
\end{equation}
where the last identity follows from Lemma \ref{12-06-22a}.
 It follows from \eqref{12-06-22a} and the way $\mathcal T^{-1}$ is defined, cf. Theorem \ref{17-08-22c}, that $\mathcal T^{-1}(\tau) = \tau_\beta$ is $\hat{\sigma}$-invariant. 
\end{proof}

Combining Corollary \ref{09-06-22d}, Lemma \ref{12-06-22x} and Theorem \ref{17-08-22c} we obtain the following

\begin{thm}\label{12-06-22bx} Let $\sigma$ be a flow on the $C^*$-algebra $A$. The composition $\mathcal T \circ D_3$ is a bijection
$$
\mathcal T \circ D_3 : \KMS(\sigma,\beta) \to \tr(A \rtimes_\sigma \mathbb R)_\beta .
$$
\end{thm}

\begin{notes}\label{12-06-22b} The existence of a bijection between $\KMS(\sigma,\beta)$ and $\tr(A \rtimes_\sigma \mathbb R)_\beta$ was announced by Kishimoto and Kumjian in Remark 3.3 of \cite{KK}, and Theorem 3.2 of \cite{KK} contains a proof for the case where $A$ is unital. 

When $A$ is separable the cone $T(A \rtimes_\sigma \mathbb R)$ can be identified with the cone of unitarily invariant, positive linear functionals on the Pedersen ideal of $A \rtimes_\sigma \mathbb R$ by Proposition 5.6.7 of \cite{Pe}. The space of linear functionals on the Pedersen ideal is a locally convex vector space in a natural way and and $T(A\rtimes_\sigma \mathbb R)$ is a closed cone in this space, and the same is $\tr(A \rtimes_\sigma \mathbb R)_\beta$ for each $\beta \in \mathbb R$. In this way Theorem \ref{12-06-22bx} gives a way to realize $\KMS(\sigma,\beta) \cup \{0\}$ as a closed cone in a locally convex vector space. 

\end{notes}

\subsection{A theorem of Vigand Pedersen}\label{vigandx}

In this section we first use the constructions from Section \ref{20-01-22g} with the following choices:
$$
G \ \text{abelian and}\ \sigma = \id_A.
$$
For $\beta \in \mathbb R$, let $T(A)_{\theta \beta}$ denote the set of lower semi-continuous traces $\tau$ on $A$ with the scaling property
$$
\tau \circ \alpha_g = e^{-\theta(g)\beta } \tau
$$
for all $g \in G$. With these choices the map $D$ from \eqref{08-06-22} becomes a map
$$
D_4 : T(A)_{\theta\beta} \to \KMS(\mu^\theta, \beta)^{\hat{\alpha}} ,
$$
where $\mu^{\theta}$ is the flow on $A \rtimes_\alpha G$ defined such that $\mu^\theta_t(C_c(G,A)) \subseteq C_c(G,A)$ and
$$
\mu^\theta_t(f)(x) := e^{i \theta(x)t} f(x)
$$
for all $f \in C_c(G,A)$. Such a map, $T(A)_{\theta\beta} \to \KMS(\mu^\theta, \beta)^{\hat{\alpha}}$, was constructed by Vigand Pedersen in \cite{VP} and it is not difficult to see that $D_4$ is the same map as the one constructed in \cite{VP}. By Theorem 5.1 in \cite{VP} the map is a bijection. This follows now also from Theorem \ref{08-06-22aA}.

\begin{thm}\label{13-06-22} (Vigand Pedersen, \cite{VP}) For all $\beta \in \mathbb R$ the map
$$
D_4:  T(A)_{\theta\beta} \to \KMS(\mu^\theta, \beta)^{\hat{\alpha}} 
$$
is a bijection.
\end{thm}

If we specialize further to the case $G = \mathbb R$ and $\theta(x) = x$, and change the notation such that $\alpha$ becomes $\sigma$, we get the following

\begin{cor}\label{13-06-22a} Let $\sigma$ be a flow on $A$ and let $T(A)_{\beta}$ denote the set of lower semi-continuous traces $\tau$ on $A$ such that $\tau \circ \sigma_t = e^{-\beta t}\tau$ for all $t \in \mathbb R$. For all $\beta \in \mathbb R$ the map
$$
D_4 :  T(A)_{\beta} \to \KMS(\hat{\sigma}, \beta)
$$
is a bijection.
\end{cor}

\begin{cor}\label{10-10-23b} Let $\sigma$ be a flow on the unital $C^*$-algebra $A$. The flow $\hat{\sigma}$ on $A \rtimes_\sigma \mathbb R$ has no $\beta$-weights for $\beta \neq 0$.
\end{cor}
\begin{proof} This follows from Corollary \ref{13-06-22a} because $T(A)_\beta = \emptyset$ when $\beta \neq 0$. Indeed, for any trace $\tau$ on $A$ we have that $0 < \tau(1) < \infty$ by Lemma \ref{24-09-23b} and hence $\tau \circ \sigma_t(1) = \tau(1)\neq e^{-\beta t} \tau(1)$ when $\beta t \neq 0$.  
\end{proof}

\begin{example}\label{10-10-23d} \rm To show that the conclusion in Corollary \ref{10-10-23b} fails in general when $A$ is not unital let $A= C_0(\mathbb R)$ and let $\sigma$ be the flow on $C_0(\mathbb R)$ given by translation, i.e.
$$
\sigma_t(f)(x) = f(x-t) .
$$
By Corollary \ref{13-06-22a} the set of $\beta$-KMS weights for the flow $\hat{\sigma}$ on $C_0(\mathbb R) \rtimes_\sigma \mathbb R$ is in bijective correspondence with $T(C_0(\mathbb R))_\beta$. To determine the latter set there are at least two different ways to proceed. In the first we notice that the map
$$
C_0(\mathbb R)^+ \ni f \mapsto \int_\mathbb R f(x) e^{-\beta x} \ \mathrm d x
$$
is an element of $T(C_0(\mathbb R))_\beta$. Since $C_0(\mathbb R) \rtimes_\sigma \mathbb R \simeq \mathbb K$, cf. e.g. \cite{Pe} or \cite{Wi}, it follows from Theorem \ref{02-01-22a} that there is exactly one $\beta$-KMS weight for $\hat{\sigma}$ for each $\beta$ and it follows therefore from Corollary \ref{10-10-23b} that the trace above is the unique element of $T(C_0(\mathbb R))_\beta$ up multiplication by scalars. 

The second method uses only measure theory and goes as follows. By Lemma \ref{03-01-22a} the set $T(A)$ can be identified with the set of regular Borel measures on $\mathbb R$; the trace $\tau_\mu$ on $C_0(\mathbb R)$ given by such a measure $\mu$ is given by
$$
\tau_\mu(f) = \int_\mathbb R f \ \mathrm d\mu .
$$
Then $\tau_\mu \in T(C_0(\mathbb R))_\beta$ if and only if $\mu(B + t) = e^{-\beta t}\mu(B)$ for all Borel sets $B \subseteq \mathbb R$ and all $t \in \mathbb R$. The regular measures satisfying this condition are the scalar multiples of the measure $e^{-\beta x} \mathrm d x$. Hence by Corollary \ref{10-10-23b} there is exactly one ray of $\beta$-KMS weights for the flow $\hat{\sigma}$ for all $\beta \in \mathbb R$. 
\end{example}

\begin{example}\label{02-08-22} \textnormal{Given a real number $\lambda > 1$ we define a homeomorphism $\varphi_\lambda$ of $\mathbb R$ by
$$
\varphi_\lambda(x) = \lambda x ,
$$
and we consider the corresponding representation $\alpha = (\alpha_z)_{z \in \mathbb Z}$ of $\mathbb Z$ by automorphisms of $C_0(\mathbb R)$:
$$
\alpha_z(f)(t) = f \circ \varphi_\lambda^{-z}(t) = f(\lambda^{-z}t) .
$$
Let $\theta : \mathbb Z \to \mathbb R$ be the canonical inclusion $\mathbb Z \subseteq \mathbb R$. The flow $\mu^\theta$ on $C_0(\mathbb R) \rtimes_\alpha \mathbb Z$ is then related to the dual action $\hat{\alpha}$ of $\mathbb T = \widehat{\mathbb Z}$ by the formula
$$
\mu^\theta_t = \hat{\alpha}_{e^{-it}} ,
$$ 
and hence KMS weights for $\mu^\theta$ are automatically invariant under $\hat{\alpha}$ so that the map $D_4$ of Theorem \ref{13-06-22} gives a bijection
$$
D_4 : T(C_0(\mathbb R))_{\beta} \to KMS(\mu^\theta,\beta)
$$
for all $\beta \in \mathbb R$. We seek therefore here to determine the set $ T(C_0(\mathbb R))_{\beta}$. By Lemma \ref{03-01-22a} a densely defined weight $\psi$ on $C_0(\mathbb R)$ is given by a regular measure on $\mathbb R$ by the formula \eqref{03-01-22c}. It follows therefore that an element $\psi \in  T(C_0(\mathbb R))_{\beta}$ is given by integration with respect to a non-zero regular Borel measure $\mu$ which satisfies the equality
\begin{equation*}\label{02-08-22a}
\int_\mathbb R f(\lambda^{-1}x) \ \mathrm d \mu(x) = e^{-\beta} \int_\mathbb R f(x) \ \mathrm d \mu(x)
\end{equation*}
for all $f \in C_0(\mathbb R)$. This holds if and only if
\begin{equation}\label{02-08-22b}
\mu(\lambda B) = e^{-\beta} \mu(B)
\end{equation}
for all Borel sets $B \subseteq \mathbb R$. Let $\mu$ be such a measure. Since
$$
\bigcup_{k \geq 1}[-\lambda^k,\lambda^k] = \mathbb R 
$$
and
$\mu([-\lambda^k,\lambda^k]) = e^{-k\beta}\mu([-1,1])$ for all $k \geq 1$, we conclude that $\mu([-1,1]) > 0$ since $\mu \neq 0$. By using that
$[-\lambda,\lambda] \supseteq [-1,1]$
and $\mu([-\lambda,\lambda]) = e^{-\beta}\mu([-1,1])$ we find that $e^{-\beta} \geq 1$, i.e. $\beta \leq 0$. Consider first the case $\beta =0$. For all $\epsilon > 0$ and $k \geq 1$ we then have that $\mu([-\lambda^k\epsilon,\lambda^k\epsilon]) = \mu([-\epsilon,\epsilon])$ and $[-\epsilon,\epsilon] \subseteq [-\lambda^k\epsilon,\lambda^k\epsilon]$, implying that $\mu$ is concentrated on $[-\epsilon,\epsilon]$. Since $\epsilon > 0$ is arbitrary, this means that $\mu$ is a scalar multiple of the Dirac measure $\delta_0$ at $0$. For $\beta < 0$, on the other hand, there are both atomic and non-atomic measures satisfying \eqref{02-08-22b}. Among the regular purely atomic measures satisfying \eqref{02-08-22b} are the measures 
$$
\sum_{k \in \mathbb Z} e^{k \beta} \delta_{\lambda^{-k}x} ,
$$
where $x \in \mathbb R\backslash \{0\}$ and $ \delta_{\lambda^{-k}x}$ denotes the Dirac measure concentrated at $\lambda^{-k}x$, and among the non-atomic measures is the measure
$$
B \mapsto \int_{B\cap ]0,\infty[} \ t^\alpha \ \mathrm d t ,
$$
where $\alpha = -\frac{\beta}{\log \lambda} - 1$. }

\textnormal{We refrain from a further study of the regular Borel measures satisfying \eqref{02-08-22b} and summarize the qualitative conclusions we have already obtained: The flow $\mu^\theta$ on $C_0(\mathbb R) \rtimes_\alpha \mathbb Z$ has no $\beta$-KMS weights for $\beta > 0$, a unique ray of $0$-KMS weights and infinitely many different rays of $\beta$-KMS weights when $\beta < 0$.}

\textnormal{Note that the preceding considerations, based on Theorem \ref{13-06-22}, only give a bijection between certain measures on $\mathbb R$ and KMS weights for $\mu^\theta$; it does not provide a formula relating the measures to the KMS weights. Such a formula will come out of results from the next section. Specifically, it will follow that if $\mu$ is a non-zero regular Borel measure on $\mathbb R$ such that \eqref{02-08-22b} holds, the corresponding KMS weight $\psi_\mu$ on $C_0(\mathbb R) \rtimes_\alpha \mathbb Z$ is given by the formula
$$
\psi_\mu(a) = \int_{\mathbb R} P(a) \ \mathrm d \mu,
$$
where $P : C_0(\mathbb R) \rtimes_\alpha \mathbb Z \to C_0(\mathbb R)$ is the canonical conditional expectation, cf. Theorem \ref{12-03-22}. Once we have established this explicit relation between the measure on $\mathbb R$ and the KMS weights for $\mu^\theta$ we can see that there are no bounded $\beta$-KMS weights for $\mu^\theta$ when $\beta \neq 0$.}
\end{example}

\begin{notes} Theorem \ref{13-06-22} was obtained by Vigand Pedersen in 1979, cf. Theorem 5.1 in \cite{VP}. It generalizes Lemma 3.1 in \cite{Th3}; a fact I regrettably missed when writing the latter.
\end{notes}

\section{Crossed products by discrete groups}\label{crosseddiscrete}
We consider now a setting similar to the cases from the previous sections of this chapter, but where the construction of the dual weight does not enter. Specifically, we consider the same situation as in Section \ref{vigandx}, with the difference that we assume the group is discrete rather than abelian. Thus we are dealing with a discrete group $G$ and a representation $\alpha : G \to \Aut A$ of $G$ by automorphisms of $A$. Let $\mathbb H$ be a Hilbert space such that $A \subseteq B(\mathbb H)$. By definition $A \rtimes_{r,\alpha}G$ is the $C^*$-algebra generated by $\pi(C_c(G,A))$, where $\pi$ is the representation $\pi : C_c(G,A) \to B(l^2(G,\mathbb H))$ defined such that
$$
\pi(f)\xi(x) = \sum_{y \in G} \alpha_{x^{-1}}(f(y))\xi(y^{-1}x), 
$$
when $f : G \to A$ is a finitely supported and $\xi \in l^2(G,\mathbb H)$. (Compare with Section \ref{07-08-22e}.)\footnote{$l^2(G,\mathbb H)$ is the same as $L^2(G,\mathbb H)$, but since the Haahr measure is the counting measure it is also the same as the Hilbert space of functions $\psi: G \to H$ for which $\sum_{g \in G} \|\psi(g)\|^2 < \infty$. Therefore the change in notation.} Consider the representation $\pi_0 : A \to B(l^2(G,\mathbb H))$ of $A$ given by
$$
(\pi_0(a)\xi)(x) = \alpha_{x^{-1}}(a)\xi(x). 
$$
Then $\pi_0(a) = \pi(f)$, where $f\in C_c(G,A)$ is given by
$$
f(g) = \begin{cases} a, & \ g = e \\ 0, & \ g \neq 0. \end{cases}
$$
Hence $\pi_0(A) \subseteq A \rtimes_{r,\alpha}G$. Consider also the unitary representation $\lambda_g,  \ g \in G$, of $G$ on $l^2(G,\mathbb H)$ given by
$$
(\lambda_g\xi)(x) := \xi(g^{-1}x) .
$$
Then $\Ad \lambda_g \circ \pi_0 = \pi_0 \circ \alpha_g$ and
$$
\pi(f) = \sum_{g \in G} \pi_0(f(g))\lambda_g 
$$
for all $f \in C_c(G,A)$. Thus $ A \rtimes_{r,\alpha}G$ is generated as a $C^*$-algebra by the set
$$
\left\{\pi_0(a)\lambda_g : \ a \in A, \ g \in G\right\} .
$$

\subsection{The canonical conditional expectation}\label{canonical}

\begin{defn}\label{20-12-22a} Let $E$ be a $C^*$-algebra and $D \subseteq E$ a $C^*$-subalgebra of $E$. A linear map $P : E \to D$ is a \emph{conditional expectation} when
\begin{enumerate}
\item[(a)] $P$ is positive, i.e $a \geq 0 \Rightarrow P(a) \geq 0$,
\item[(b)] $\|P\| = 1$, and
\item[(c)]$P(ab)= P(a)b$ for all $a \in  E$ and all $b \in D$,
\item[(d)] $P(b) = b$ for all $b \in D$.
\end{enumerate}
$P$ is \emph{faithful} when $a \in E^+, \ P(a) = 0 \Rightarrow a = 0$.
\end{defn}

\begin{lemma}\label{31-08-23} Let $G$ be a discrete group and $\alpha : G \to \Aut A$ a representation of $G$ by automorphisms of $A$. There is a faithful conditional expectation $P : A \rtimes_{r,\alpha} G \to \pi_0(A)$ such that 
$$
P(\sum_{g \in G} \pi_0(f(g))\lambda_g) = \pi_0(f(e))
$$ 
for every finitely supported function $f : G \to A$. 
\end{lemma}
\begin{proof} Define an isometry $V : \mathbb H \to l^2(G,\mathbb H)$ such that
$$
(V\eta)(g) = \begin{cases} \eta, \ g = e \\ 0, \ g \neq e .\end{cases}
$$
The adjoint $V^*: l^2(G,\mathbb H) \to \mathbb H$ is given by
$$
V^*\xi = \xi(e) .
$$
If $a : G \to A$ is a finitely supported function,
$$
V^*(\sum_{g \in G} \pi_0(a(g))\lambda_g) V = a(e),
$$
and it follows therefore that $V^*(A \rtimes_{r,\alpha}G)V \subseteq A$. Set $P(m) := \pi_0(V^*mV)$. It is straightforward to check that $P$ is a conditional expectation. To see that $P$ is faithful, note that
$$
P(\lambda_g m\lambda_g^*) = \lambda_gP(m)\lambda_g^*
$$
for all $m \in A\rtimes_{r,\alpha} G$. Thus, if $m \in (A\rtimes_{r,\alpha} G)^+$ and $P(m)=0$, we find first that $P(\lambda_g m\lambda_g^*) = 0$ for all $g \in G$, and then because $\pi_0$ is faithful, that $V^*\lambda_gm\lambda_g^*V = 0$ for all $g \in G$. Since $m$ is positive this implies that $m\lambda_g^*V \eta = 0$ for all $\eta \in \mathbb H$ and all $g \in \mathbb H$. Since $\left\{\lambda_g^*V\eta: \ g \in G, \ \eta\in \mathbb H\right\}$ spans a dense subspace in $l^2(G,\mathbb H)$ it follows that $m = 0$.
\end{proof}

The conditional expectation in Lemma \ref{31-08-23} will be referred to as \emph{the canonical conditional expectation}. 
When we suppress the representation $\pi_0$ in the notation and identify $A$ with its image $\pi_0(A)$ in $A \rtimes_{r,\alpha} G$ we see that $A \rtimes_{r,\alpha} G$ is generated by a copy of $A$ and the elements of the form $au_g, \ g \in G$, where $u$ is a unitary representation $G$ such 
\begin{equation}\label{02-09-23}
u_gau_g^* = \alpha_g(a).
\end{equation} 
Of course, the representation $u$ is the unitary representation $\lambda$, but we have changed notation because the actual formula for $u$ will rarely be important; only the relation \eqref{02-09-23}. Then 
$$
\Span \left\{au_g: \ a \in A, \ g \in G\right\}
$$
is a dense $*$-algebra in $A \rtimes_{r,\alpha} G$ and the canonical conditional expectation is the unique continuous linear map $A \rtimes_{r,\alpha} G \to A$ with the property that
$$
P(au_g) = \begin{cases} a, & \ g = e, \\ 0, & \ g \neq e.\end{cases}
$$

\begin{lemma}\label{08-09-23} Let $\{u_i\}_{i \in I}$ be a net in $A$ such that
\begin{itemize}
\item[$\cdot$] $u_i^* = u_i$ for all $i \in I$,
\item[$\cdot$] $\sup_{i \in I} \left\|u_i\right\| < \infty$, and
\item[$\cdot$] $\lim_{i \to \infty} u_i a = a$ for all $a \in A$.
\end{itemize}
It follows that $\lim_{i\to \infty} u_ix = \lim_{i\to \infty} xu_i =x$ for all $x \in A\rtimes_{r,\alpha} G$.
\end{lemma}
\begin{proof} Let $\epsilon > 0$. There is a finitely supported function $a : G \to A$ such that 
$$
\left( \sup_{i \in I} \left\|u_i\right\| +1\right) \left\|x - \sum_{g \in G} a(g)u_g\right\| \leq \epsilon .
$$
Since $\lim_{i \to \infty} u_ia(g) = a(g)$ for all $g \in G$ there is an $i_0 \in I$ such that 
$$
\left\|u_i \sum_{g \in G} a(g)u_g - \sum_{g \in G} a(g)u_g\right\| \leq \epsilon
$$
when $i_0 \leq i$. It follows that $\left\|u_ix -x\right\| \leq 2\epsilon$ when $i_0 \leq i$, showing that $\lim_{i\to \infty} u_ix  =x$. Applying this conclusion to $x^*$ we find that $\lim_{i\to \infty} xu_i = \left(\lim_{i\to \infty} u_ix^*\right)^* = x$.
\end{proof}

We consider now a group homomorphism $\theta : G \to \mathbb R$. This gives rise to the flow $\gamma^\theta$ on $A \rtimes_{r,\alpha}G$ defined such that
$$
\gamma^\theta_t \left( \sum_{g \in G} a(g) u_g\right) = \sum_{g \in G} e^{i \theta(g) t} a(g) u_g  
$$
for every finitely supported function $a : G \to A$. This is a special case of the flows defined in \eqref{30-01-22e}. We aim now to determine the KMS weights of $\gamma^\theta$. Some information can be obtained from Section \ref{07-08-22e} and Section \ref{13-04-23}, but by using that the group $G$ is discrete we can get a more complete picture.

Set
$$
\ker \theta := \left\{g \in G: \ \theta(g) = 0 \right\}
$$
which is a normal subgroup of $G$. The elements $u_g , g \in \ker \theta$, and $A$ generate together a $C^*$-subalgebra of $A\rtimes_{r,\alpha} G$ which we denote by
$$
A \rtimes_{r,\alpha} \ker \theta
$$
since it is $*$-isomorphic to the reduced crossed product of $A$ by the action of $\ker \theta$ obtained from $\alpha$ by restriction.

\begin{lemma}\label{31-08-22} $A \rtimes_{r,\alpha} \ker \theta$ is equal to the fixed point algebra $(A \rtimes_{r,\alpha} G)^{\gamma^\theta}$ of $\gamma^\theta$, and there is a conditional expectation 
$$
Q : A \rtimes_{r,\alpha} G \to A \rtimes_{r,\alpha} \ker \theta
$$
such that
$$
Q(x) = \lim_{R \to \infty} \frac{1}{R} \int_0^R \gamma^\theta_t(x) \ \mathrm dt .
$$
for all $x \in  A \rtimes_{r,\alpha} G$.
\end{lemma} 
\begin{proof} Let $a \in A$. When $g \notin \ker \theta$, 
$$
 \frac{1}{R} \int_0^R \gamma^\theta_t(au_g) \ \mathrm dt = \frac{1}{iR\theta(g)} (e^{i\theta(g)R} - 1)au_g
 $$
while $\frac{1}{R} \int_0^R \gamma^\theta_t(au_g) \ \mathrm d t = au_g$ when $g \in \ker \theta$. It follows therefore that
$$
\lim_{R \to \infty} \frac{1}{R} \int_0^R \gamma^\theta_t( \sum_{g \in G} a(g) u_g) \ \mathrm dt = \sum_{g \in \ker \theta} a(g)u_g
$$
when $a : G \to A$ is finitely supported. Let $x \in A \rtimes_{r,\alpha} G$ and $\epsilon > 0$ be given. Choose a finitely supported function $a: G \to A$ such that
$$
\left\| x - \sum_{g \in G}a(g)u_g \right\| \leq \epsilon .
$$
There is an $N \in \mathbb N$ such that
$$
 \left\|\frac{1}{n} \int_0^n \gamma^\theta_t( \sum_{g \in G} a(g) u_g) \ \mathrm dt -  \frac{1}{m} \int_0^m \gamma^\theta_t( \sum_{g \in G} a(g) u_g) \ \mathrm dt\right\| \leq \epsilon
 $$
 when $n,m \geq N$. Then
 \begin{align*}
 & \left\|\frac{1}{n} \int_0^n \gamma^\theta_t( x) \ \mathrm dt -  \frac{1}{m} \int_0^m \gamma^\theta_t(x) \ \mathrm dt\right\| \\
 & \leq  \left\|\frac{1}{n} \int_0^n \gamma^\theta_t( \sum_{g \in G} a(g) u_g) \ \mathrm dt -  \frac{1}{m} \int_0^m \gamma^\theta_t( \sum_{g \in G} a(g) u_g) \ \mathrm dt\right\| + 2\epsilon  \leq 3 \epsilon
 \end{align*}
when $n,m \geq N$. This shows that $\left\{\frac{1}{n} \int_0^n \gamma^\theta_t( x) \ \mathrm dt\right\}_{n \in \mathbb N}$ is Cauchy and hence convergent in $A \rtimes_{r,\alpha} G$. Note that the limit lies in $A \rtimes_{r,\alpha} \ker \theta$ since this is true when $x$ is replaced by $\sum_{g \in G} a(g)u_g$. Because
$$
\left\| \frac{1}{R} \int_0^R \gamma^\theta_t(x) \ \mathrm d t - \frac{1}{n}  \int_0^R \gamma^\theta_t(x) \ \mathrm d t \right\| \leq \frac{2\|x\|}{n}
$$
when $n \leq R \leq n+1$, it follows that
$$
\lim_{R \to \infty} \frac{1}{R} \int_0^R \gamma^\theta_t(x) \ \mathrm dt  = \lim_{n \to \infty} \frac{1}{n} \int_0^n \gamma^\theta_t(x) \ \mathrm dt  .
$$
The resulting map
$$
Q(x) := \lim_{R \to \infty} \frac{1}{R} \int_0^R \gamma^\theta_t(x) \ \mathrm dt 
$$
is easily seen to be a conditional expectation onto $A \rtimes_{r,\alpha} \ker \theta$. Since $Q(x) = x$ when $x \in (A \rtimes_{r,\alpha} G)^{\gamma^\theta}$ it follows that $(A \rtimes_{r,\alpha} G)^{\gamma^\theta} \subseteq A \rtimes_{r,\alpha} \ker \theta$. This completes the proof because the reverse inclusion is trivial. 
\end{proof}

When $\theta$ is injective and hence $\ker \theta =0$, the conditional expectation of Lemma \ref{31-08-22} is the canonical conditional expectation we introduced above.

We define an action $\alpha'$ of $G$ on $A\rtimes_{r,\alpha} \ker \theta$ such that
$$
\alpha'_g(x) := u_g xu_g^* .
$$

\begin{lemma}\label{31-08-22c} Let $\tau$ be a lower semi-continuous trace on $A\rtimes_{r,\alpha} \ker \theta$ such that $\tau \circ \alpha'_g = e^{-\beta \theta(g)}\tau$ for all $g \in G$. Then $\tau \circ Q$ is a $\beta$-KMS weight for $\gamma^\theta$.
\end{lemma}
\begin{proof} For the proof we will use Theorem \ref{12-12-13} with
$$
S:= \left\{\sum_{g \in F} b_gu_g: \ F \subseteq G \ \text{finite}, \ b_g \in \mathcal M_\tau \right\} ,
$$ 
and the proof consists of a check that $\tau \circ Q$ and $S$ have the properties required in that theorem. To see that $S \subseteq \mathcal M_{\tau\circ Q}^{\gamma^\theta}$ it suffices to show that $bu_g \in\mathcal M_{\tau\circ Q}^{\gamma^\theta}$ when $b \in \mathcal M_\tau^+$ and $g \in G$. For this note that $\sqrt{b} \in  \mathcal N_{\tau \circ Q}^*$ since $\tau\circ Q(\sqrt{b}\sqrt{b}^*) = \tau(b) < \infty$ and that $\sqrt{b}u_g \in \mathcal N_{\tau \circ Q}$ since $\tau \circ Q((\sqrt{b}u_g)^*\sqrt{b}u_g) = \tau(u_g^* bu_g) = e^{\beta \theta(g)}\tau(b) < \infty$. Hence $bu_g = \sqrt{b}\sqrt{b}u_g \in \mathcal N_{\tau \circ Q}^*\mathcal N_{\tau \circ Q} \subseteq \mathcal M_{\tau \circ Q}$. Note then that
$$
\sqrt{\frac{k}{\pi}} \int_\mathbb Re^{-k t^2}\gamma^\theta_t(bu_g) \ \mathrm d g = \left( \sqrt{\frac{k}{\pi}}\int_\mathbb R e^{-k t^2}  e^{i\theta(g) t} \ \mathrm d t\right) bu_g .
$$
By applying Lemma \ref{24-11-21} with $X = \mathbb C$ and $\sigma_t(c) = e^{i\theta(g)t}c$ it follows that 
$$
\lim_{k \to \infty}  \sqrt{\frac{k}{\pi}} \int_\mathbb R e^{-k t^2}  e^{i\theta(g) t} \ \mathrm d t = 1 .
$$ 
Therefore, for some $k$ large,
$$
bu_g = \lambda \sqrt{\frac{k}{\pi}}\int_\mathbb R  e^{-k t^2}\gamma^\theta_t(bu_g) \ \mathrm d t .
$$
for some $\lambda \in \mathbb C \backslash \{0\}$.
Since $bu_g \in \mathcal M_{\tau \circ Q}$ it follows then from Lemma \ref{07-12-21} that $bu_g \in \mathcal M_{\tau \circ Q}^\sigma$, and hence that $S \subseteq \mathcal M_{\tau\circ Q}^{\gamma^\theta}$.

Let $\{e_m\}_{m \in I}$ be the net in $A$ from Lemma \ref{16-12-21a} applied with $D(\Lambda) = \mathcal N_\tau$. Set $v_m := e_m^2$. Then $0 \leq v_m \leq 1$, $\tau(v_m) < \infty$ and $\{v_m\}_{m \in  I}$ is an approximate unit for $A\rtimes_{r,\alpha} G$ by Lemma \ref{08-09-23}. Let $b \in (A\rtimes_{r,\alpha} G)^+$. Then $\tau \circ Q(v_mbv_m) = \tau (v_mQ(b)v_m) \leq \|Q(b)\|\|v_m\| \tau(v_m) < \infty$ and we conclude therefore that $\tau \circ Q$ is densely defined. It is non-zero since $\tau$ is. Since
$$
Q\circ \gamma^\theta_s (x) = \lim_{R \to \infty} \frac{1}{R} \int_0^R \gamma^\theta_{t+s}(x) \ \mathrm dt = \lim_{R \to \infty} \frac{1}{R} \int_s^{R+s} \gamma^\theta_{t}(x) \ \mathrm dt = Q(x) ,
$$
we see that $Q$ is $\gamma^\theta$-invariant and hence so is $\tau \circ Q$. 

Let $x \in \mathcal N_{\tau \circ Q}$ and $\epsilon > 0$ be given. Using that $\tau$ is trace, that $v_m \in \mathcal M_\tau$ and  $Q(x^*x) \in \mathcal M_\tau$, we find that
\begin{align*}
& \tau \circ Q((xv_m-x)^*(xv_m-x)) \\
&= \tau(Q(x^*x)) + \tau(v_mQ(x^*x)v_m) - \tau(v_mQ(x^*x)) - \tau(Q(x^*x)v_m) \\
& = \tau(Q(x^*x)) + \tau(\sqrt{Q(x^*x)}v_m^2\sqrt{Q(x^*x)}) - \tau(v_mQ(x^*x)) - \tau(Q(x^*x)v_m).
\end{align*}
Since $\sqrt{Q(x^*x)}$ and $v_m \sqrt{Q(x^*x)}$ are both elements of $\mathcal N_\tau = \mathcal N_\tau^*$ it follows from Lemma \ref{24-06-22b} (applied with $\beta = 0$) that 
$$
\tau(v_mQ(x^*x)) = \tau(Q(x^*x)v_m) = \tau(\sqrt{v_m}Q(x^*x)\sqrt{v_m}) .
$$
Hence
\begin{align*}
& \tau \circ Q((xv_m-x)^*(xv_m-x)) \\
& =\tau(Q(x^*x)) + \tau\left(\sqrt{Q(x^*x)}v_m^2\sqrt{Q(x^*x)}\right) - 2\tau\left(\sqrt{Q(x^*x)}v_m\sqrt{Q(x^*x)}\right) .
\end{align*}
The lower semi-continuity of $\tau$ implies therefore that 
$$
\lim_{m \to \infty} \tau \circ Q((xv_m-x)^*(xv_m-x)) = 0;
$$ 
i.e. $\lim_{m \to \infty} \Lambda_{\tau \circ Q}(xv_m) = \Lambda_{\tau \circ Q}(x)$. We can therefore find $m \in I$ such that $\|\Lambda_{\tau \circ Q}(xv_m) - \Lambda_{\tau \circ Q}(x)
\| \leq \epsilon$ and $\left\|xv_m -x\right\| \leq \epsilon$. Choose $F \subseteq G$ finite and elements $a_g \in \mathcal N_\tau^* = \mathcal N_\tau, \ g \in F$, such that $\left\|x-y\right\| \leq \epsilon$ and 
$$
\left\|x-y\right\|^2 \tau(v_m^2) \leq \epsilon^2 
$$
when we set
$$
y:= \sum_{g \in F} u_ga_g .
$$
Then
\begin{align*}
&\left\| \Lambda_{\tau \circ Q}(yv_m) -  \Lambda_{\tau \circ Q}(xv_m)\right\|^2 = \tau(v_mQ((y-x)^*(y-x))v_m) \\
& \leq \| y-x\|^2 \tau (v_m^2) \leq \epsilon^2 . 
\end{align*}
and $\|yv_m - xv_m\|\leq \epsilon$. Hence
$$
\|\Lambda_{\tau \circ Q}(yv_m) - \Lambda_{\tau \circ Q}(x)
\| \leq 2 \epsilon
$$
and
$$
\|yv_m - x\| \leq 2 \epsilon .
$$
Note that 
$$
yv_m = \sum_{g \in F} u_g a_g v_m = \sum_{g \in F} u_g a_g v_m u_g^* u_g 
$$ 
and that $a_gv_m \in \mathcal N_\tau^* \mathcal N_\tau \subseteq \mathcal M_\tau$. The scaling property of $\tau$ ensures that $u_g \mathcal M_\tau u_g^* \in \mathcal M_\tau$ for all $g$, and we conclude therefore that $u_ga_gv_mu_g^* \in \mathcal M_\tau$ for all $g$ and hence also that $yv_m \in S$.

 Thanks to Theorem \ref{12-12-13} we can now complete the proof by showing that
$$
\tau \circ  Q(bu_h \gamma^\theta_{i\beta}(au_g)) =\tau \circ Q(au_gbu_h)
$$
when $a,b \in \mathcal M_\tau$ and $g,h \in G$. To this end note that $Q(bu_h \gamma^\theta_{i\beta}(au_g)) = Q(au_gbu_h) = 0$ unless $gh \in \ker \theta$, and when $gh \in \ker \theta$ we find that 
$$
Q(bu_h \gamma^\theta_{i\beta}(au_g)) = e^{-\beta \theta(g)}bu_hau_g .
$$ 
Since $\tau\circ \alpha'_g = e^{-\beta \theta(g)}\tau$ by assumption we find that
$$\tau \circ  Q(bu_h \gamma^\theta_{i\beta}(au_g)) = e^{-\beta \theta(g)}\tau(bu_hau_g)= \tau (u_g bu_ha).
$$
To continue the calculation note that $u_gbu_g^* \in \mathcal M_\tau$ thanks to the assumed scaling property of $\tau$ and that therefore $u_gbu_h = u_gbu_g^*u_{gh} \in S \subseteq \mathcal N_{\tau \circ Q}$. Since $u_gbu_h = u_gbu_g^*u_{gh} \in  A\rtimes_{r,\alpha} \ker \theta$ it follows that $u_gbu_h \in \mathcal N_\tau$. Since also $a \in \mathcal N_\tau = \mathcal N_\tau^*$ it follows from Lemma \ref{24-06-22b} (again applied with $\beta =0$) that $ \tau (u_g bu_ha) = \tau(au_gbu_h)= \tau \circ Q(au_gbu_h)$.
\end{proof}

The \emph{fixed point algebra} $D^\sigma$ of the flow $\sigma$ on the $C^*$-algebra $D$ is the set
$$
D^{\sigma} := \left\{ a \in D: \ \sigma_t(a) = a \ \ \forall t \in \mathbb R \right\} .
$$
Note that $D^\sigma$ is a $C^*$-subalgebra of $D$. While $D^\sigma$ always contains
 the unit when $D$ is unital, in the non-unital case the fixed point algebra may consist only of $0$.

\begin{lemma}\label{13-12-21axx} Let $\sigma$ be a flow on the $C^*$-algebra $D$. Assume that
\begin{enumerate}
\item[(1)] the fixed point algebra $D^\sigma$ of $\sigma$ contains an approximate unit for $D$, and
\item[(2)] the limit $Q'(a) := \lim_{R \to \infty} \frac{1}{R} \int_0^R \sigma_t(a) \ \mathrm dt$
exists in norm for all $a \in D$. 
\end{enumerate}
Let $\phi$ be a $\beta$-KMS weight for $\sigma$. Then
$\phi(a) = \phi \circ Q'(a)$ for all $a \in D^+$.
\end{lemma}
\begin{proof} Let $a \in D^+$. We want to show that $\phi(a) = \phi \circ Q'(a)$. Consider first the case $\phi(a) =0$. Since $\phi$ is lower semi-continuous,
$$
\phi \circ Q'(a) \leq \liminf_n \phi\left( \frac{1}{n} \int_0^n \sigma_t(a) \ \mathrm d t \right) .
$$
Since $\phi(\sigma_t(a)) = 0$ for all $t \in \mathbb R$ it follows that $\phi(I) =0$ for every Riemann sum approximating the integral $ \int_0^n \sigma_t(a) \ \mathrm d t $ and then the lower semi-continuity of $\phi$ shows that $\phi\left( \frac{1}{n} \int_0^n \sigma_t(a) \ \mathrm d t \right) = 0$ for each $n$. Hence $\phi \circ Q'(a) =0$, establishing the desired conclusion when $\phi(a) =0$. Assume therefore now that $\phi(a) > 0$. It follows from condition (1) that $D^\sigma$ contains a sequence $\{u_n\}$ such that $\lim_{n \to \infty} u_n a = a$. Let $E$ be the $C^*$-algebra generated by $\{u_n : \ n \in \mathbb N\}$. This is a separable $C^*$-subalgebra of $D^\sigma$ and it follows therefore from Lemma \ref{02-01-23x} and Lemma \ref{19-11-22fx} in Appendix \ref{AppD} that there is a sequence $\{h_n\}_{n=0}^\infty$ in $E$ such that $0 \leq h_n  \leq 1$ and $h_{n+1}h_n = h_n$ for all $n$ and $\lim_{n \to \infty} h_n d = d$ for all $d \in E$. Since $u_n \in E$ for all $n$, it follows that $\lim_{n \to \infty} h_na = a$. Set 
$$
D_0 := \left\{ b \in D: \ \lim_{n \to \infty} h_nb = \lim_{n \to \infty} bh_n = b\right\} .
$$
$D_0$ is a $\sigma$-invariant $C^*$-subalgebra of $D$ and $\{h_n\}$ is an approximate unit for $D_0$. Since $h_{n+1}h_n = h_n$ it follows from Proposition \ref{02-12-21x} that $\phi(h_n) < \infty$, implying that $\phi$ is densely defined on $D_0$. Since it is also non-zero as $\phi(a) > 0$, it follows from Kustermans' theorem, Theorem \ref{24-11-21d}, that $\phi$ restricts to be a $\beta$-KMS on $D_0$. Set $\delta_0 = \sqrt{h_0}$ and $\delta_n = \sqrt{h_n - h_{n-1}}, \ n \geq 1$. We claim that
\begin{equation}\label{31-08-22d}
\sum_{n=0}^\infty \phi(\delta_n d \delta_n) = \phi(d)  \ \ \ \ \forall d \in D_0^+.
\end{equation}
To see this, let $b \in \mathcal A_\sigma \cap D_0$. Then
\begin{align*}
&\sum_{n=0}^N \phi(\delta_n b^*b \delta_n) = \sum_{n=0}^N \phi\left( \sigma_{-i \frac{\beta}{2}}(b)\delta_n^2  \sigma_{-i \frac{\beta}{2}}(b)^*\right)  = \phi\left( \sigma_{-i \frac{\beta}{2}}(b) h_N  \sigma_{-i \frac{\beta}{2}}(b)^*\right).
\end{align*}
Note that $\sigma_{-i \frac{\beta}{2}}(b) \in D_0$ since $D_0$ is $\sigma$-invariant, cf. Remark \ref{31-08-22h}. It follows therefore from the lower semi-continuity of $\phi$ that
\begin{equation}\label{13-12-21gx}
\begin{split}
&\sum_{n=0}^\infty \phi(\delta_n b^*b \delta_n) = \lim_{N \to \infty} \phi\left( \sigma_{-i \frac{\beta}{2}}(b) h_N  \sigma_{-i \frac{\beta}{2}}(b)^*\right)\\
&
=
\phi\left( \sigma_{-i \frac{\beta}{2}}(b)  \sigma_{-i \frac{\beta}{2}}(b)^*\right) = \phi(b^*b).
\end{split}
\end{equation}
We define a weight $\phi'$ on $D_0$ by
$$
\phi'(d) := \sum_{n=0}^\infty \phi(\delta_n d \delta_n) \ \ \ \ \forall d \in D_0^+.
$$
It follows from \eqref{13-12-21gx} that $\phi'(a^*a) < \infty$ when $a \in \mathcal N_\phi \cap \mathcal A_\sigma \cap D_0$ and we conclude therefore that $\phi'$ is densely defined since $\phi|_{D_0}$ is. It is straightforward to see that $\phi'$ is $\sigma$-invariant since $\phi$ is. Hence \eqref{13-12-21gx} and Lemma \ref{01-03-22b} imply that $\phi' = \phi$; i.e. \eqref{31-08-22d} holds as claimed.  Since $h_{n+1} \delta_n = \delta_n$ it follows from  Propostion \ref{02-12-21x} that $\delta_n \in \mathcal N_\phi$. Hence the functional 
$$
D \ni d \mapsto \phi(\delta_n d \delta_n)
$$ 
is defined and bounded for each $n$. It follows therefore that
\begin{align*}
&\phi(\delta_n Q'(d) \delta_n) = \lim_{R \to \infty} \frac{1}{R}\phi\left( \delta_n \left(\int_0^R \sigma_t(a) \ \mathrm d t\right)  \delta_n\right) \\
& = \lim_{R \to \infty} \frac{1}{R}\int_0^R \phi\left( \delta_n\sigma_t(a) \delta_n\right) \ \mathrm d t\\
& = \lim_{R \to \infty} \frac{1}{R}\int_0^R \phi \left(\sigma_t(\delta_n a\delta_n)\right) \ \mathrm d t  = \lim_{R \to \infty} \frac{1}{R}\int_0^R \phi \left(\delta_n a\delta_n\right) \ \mathrm d t\\
& = \phi \left(\delta_n a\delta_n\right) 
\end{align*}
for all $d \in D_0^+$ since $\delta_n$ and $\phi$ are $\sigma$-invariant. It follows therefore from \eqref{31-08-22d} that $\phi \circ Q'(a) = \phi(a)$.
\end{proof}

\begin{lemma}\label{31-08-22e} Let $\psi$ be a $\beta$-KMS weight for $\gamma^\theta$. It follows that $\psi|_{A \rtimes_{r,\alpha} \ker \theta}$ is a lower semi-continuous trace on $A \rtimes_{r,\alpha} \ker \theta$ such that
$$
\psi \circ \alpha'_g = e^{-\beta \theta(g)} \psi
$$
on $(A \rtimes_{r,\alpha} \ker \theta)^+$.
\end{lemma}
\begin{proof}  To show that $\psi|_{A \rtimes_{r,\alpha} \ker \theta}$ is densely defined, let $x \in \left(A \rtimes_{r,\alpha} \ker \theta\right)^+$. Let $\{u_i\}_{i \in I}$ be a positive approximate unit for $A$ such that $0 \leq u_i \leq 1$ for all $i$. By Lemma \ref{08-09-23} there is an $i_0 \in I$ such that $\left\|u_{i_0}xu_{i_0} - x\right\|\leq \epsilon$. Using Lemma \ref{19-11-22fx} in Appendix \ref{AppD} we get a sequence $\{h_n\}$ in $A$ such that $0 \leq h_n \leq 1$ and $h_{n+1}h_n = h_n$ for all $n$, and $\lim_{n \to \infty} h_{n}u_{i_0} = u_{i_0}$. There is then an $n_0 \in \mathbb N$ such that $\left\|h_{n_0}u_{i_0} x u_{i_0}h_{n_0} - x\right\| \leq 2 \epsilon$. Since $h_{n_0+1}h_{n_0 } = h_{n_0}$ it follows from Proposition \ref{02-12-21x} that $h_{n_0} \in \mathcal N_\psi$ and hence that
$$
\psi(h_{n_0}u_{i_0} x u_{i_0}h_{n_0}) \leq \|x\| \left\|u_{i_0}\right\|^2 \psi((h_{n_0})^2) < \infty .
$$
This shows that $\psi|_{A \rtimes_{r,\alpha} \ker \theta}$ is densely defined. To see that it is non-zero, let $x \in \left(A \rtimes_{r,\alpha} G\right)^+$ such that $\psi(x) > 0$. By Lemma \ref{08-09-23} there is a sequence $\{i_n\}$ in $I$ such that $\lim_{n \to \infty} u_{i_n}xu_{i_n} = x$ and hence 
$$
0 < \psi(x) \leq \liminf_n \psi(u_{i_n}xu_{i_n}) \leq \|x\|\liminf_n \psi((u_{i_n})^2) .
$$
It follows that $ \psi((u_{i_n})^2) > 0$ for most $n$. Since $u_{i_n} \in A \subseteq A \rtimes_{r,\alpha} \ker \theta$ it follows that $\psi|_{A \rtimes_{r,\alpha} \ker \theta}$ is not zero.

 Since $\gamma^\theta$ is the trivial flow om $A \rtimes_{r,\alpha} \ker \theta$, it follows now from condition (1) in Kustermans' theorem, Theorem \ref{24-11-21d}, that 
$$
\tau : =\psi|_{A\rtimes_{r,\alpha} \ker \theta}
$$
 is a lower semi-continuous trace on on $A\rtimes_{r,\alpha} \ker \theta$. Let $a \in A\rtimes_{r,\alpha} \ker \theta$ and $g \in G$. Then $au_g \in \mathcal A_{\gamma^\theta}$ and it follows therefore from condition (1) in Kustermans' theorem that 
$$
\psi(u_g^*a^*au_g) = \psi\left( \gamma^\theta_{-i \frac{\beta}{2}}(au_g) \gamma^\theta_{-i \frac{\beta}{2}}(au_g)^*\right) .
$$
Since $ \gamma^\theta_{-i \frac{\beta}{2}}(au_g) =e^{\frac{\theta(g)\beta}{2}} au_g$ it follows that $\tau \circ \alpha'_g(a^*a) = e^{-\theta(g)\beta}\tau(aa^*) =  e^{-\theta(g)\beta}\tau(a^*a)$ for all $g \in G$. 
\end{proof}

\begin{thm}\label{12-03-22} Define $\alpha' : G \to \Aut (A\rtimes_{r,\alpha} \ker \theta)$ by
$$
\alpha'_g(x) := u_g xu_g^* .
$$
The map $\tau \mapsto \tau \circ Q$ a bijection from the set of lower semi-continuous traces $\tau$ on $A\rtimes_{r,\alpha} \ker \theta$ with the property that $\tau \circ \alpha'_g = e^{-\beta \theta(g)}\tau$ for all $g \in G$ onto the set of $\beta$-KMS weights for $\gamma^\theta$. The inverse is the map $\psi \mapsto \psi|_{A\rtimes_{r,\alpha} \ker \theta}$ given by restriction.
\end{thm}
\begin{proof} The map is well-defined by Lemma \ref{31-08-22c} and it is injective because $\tau \circ Q|_{A \rtimes_{r,\alpha} \ker \theta} = \tau$. Let $\psi$ be a $\beta$-KMS weight for $\gamma^\theta$. It follows from Lemma \ref{08-09-23} and Lemma \ref{31-08-22} that Lemma \ref{13-12-21axx} applies to show that 
$$
\psi = \psi\circ Q = \psi|_{A\rtimes_{r,\alpha} \ker \theta} \circ Q.
$$ 
Thanks to Lemma \ref{31-08-22e} this shows that the map under consideration is also surjective and that the inverse is the map $\psi \mapsto \psi|_{A\rtimes_{r,\alpha} \ker \theta}$.
\end{proof}

\begin{notes} The main result of this section, Theorem \ref{12-03-22}, is new. As pointed out in the introduction to this section, the flow $\gamma^\theta$ is a special case of the flows $\sigma^\theta$ considered in Section \ref{07-08-22e} and hence the map $D$ from \eqref{08-06-22} gives a map from traces on $A$ with the scaling property of Lemma \ref{31-08-22e} (on $A$) to the $\beta$-KMS weights for $\gamma^\theta$. While $D$ is injective by Lemma \ref{07-06-22a}, it is not always surjective. What Theorem \ref{12-03-22} achieves in the present setting, is an extension of $D$ which is both injective and surjective.
\end{notes}



\chapter{On tracial representations
of KMS weights}

The KMS weights of the trivial action on a $C^*$-algebra are the lower semi-continuous traces on the algebra, and traces appear in this way as trivial examples of KMS weights. While general KMS weights are not traces, in all the examples we have presented so far the KMS weights have appeared to be related to traces in various ways. Theorem \ref{12-06-22bx} gives a partial explanation for this by showing that the set of KMS weights for a given flow is in bijection with a subset of the traces on the crossed product by the flow. In this chapter we use the dual KMS weight from the setting of Section \ref{8.2.2} to obtain a formula for general KMS weights which resembles the formulae obtained for inner flows in Section \ref{veryinner}, the most famous of which is the formula
\begin{equation*}\label{20-08-22}
\frac{\Tr (e^{-\beta H} a)}{\Tr(e^{-\beta H})}
\end{equation*}
for the $\beta$-KMS state of the flow $\Ad e^{itH}$ on the matrix algebra $M_n(\mathbb C)$, cf. Example \ref{31-07-23}. The new formula, valid for arbitrary KMS weights, is displayed in the box of Theorem \ref{07-08-22bx}.

\section{A general trace formula for KMS weights}\label{04-09-22c}

Let $\sigma$ be a flow on $A$. A \emph{covariant representation} of $\sigma$ is a pair $(\rho,H)$ where $\rho$ is a non-degenerate representation of $A$ on a Hilbert space $\mathbb H'$ and $H$ is a (possibly unbounded) self-adjoint operator on $\mathbb H'$ such that
\begin{equation}\label{23-05-22}
e^{it H}\rho(a)e^{-itH} = \rho(\sigma_t(a)) 
\end{equation}
for all $t \in \mathbb R$ and all $a \in A$.

For any covariant representation $(\rho,H)$ of $\sigma$ there is a non-degenerate representation $\rho \times H : A \rtimes_\sigma \mathbb R \to B(\mathbb H')$ of $A \rtimes_\sigma \mathbb R$ such that
$$
\rho \times H(f) = \int_{\mathbb R} \rho(f(t))e^{i t H} \ \mathrm d t 
$$
for all $f \in L^1(\mathbb R,A)$, cf. Proposition 7.6.4 of \cite{Pe}. We will refer to $\rho \times H$ as the integrated representation of $(\rho,H)$. Since $\rho$ and $\rho \times H$ are non-degenerate they extend to representations of the multiplier algebras
$$
\overline{\rho} : M(A) \to B(\mathbb H')
$$
and
$$
\overline{\rho \times H} : M(A \rtimes_\sigma \mathbb R) \to B(\mathbb H'),
$$
cf. Lemma \ref{26-09-22} in Appendix \ref{multipliers}.

\begin{lemma}\label{06-08-22d} $\overline{\rho \times H}(M(A \rtimes_\sigma \mathbb R)) \subseteq (\rho \times H)(A \rtimes_\sigma \mathbb R)''$.
\end{lemma}
\begin{proof} Let $m \in M(A \rtimes_\sigma \mathbb R)$. Let $\{u_i\}$ be a bounded approximate unit in $A \rtimes_\sigma \mathbb R$. Then $a_i :=u_im$ is a bounded net in $A \rtimes_\sigma \mathbb R$ such that $\lim_{i \to \infty} a_i = m$ in the strict topology. Then
$$
\lim_{i \to \infty} (\rho \times H)(a_i)(\rho \times H)(x) = (\rho \times H)(mx) = \overline{\rho \times H}(m)(\rho \times H)(x) 
$$
in norm for all $x \in A \rtimes_\sigma \mathbb R$. Since $\rho \times H$ is non-degenerate this implies that $\lim_{i \to \infty} (\rho \times H)(a_i) = \overline{\rho \times H}(m)$ in the strong operator topology. Since $(\rho \times H)(A \rtimes_\sigma \mathbb R)''$ is closed in the strong operator topology this shows that $\overline{\rho \times H}(m) \in (\rho \times H)(A \rtimes_\sigma \mathbb R)''$.
\end{proof}

\begin{lemma}\label{24-05-22} Let $\iota : A \to M(A \rtimes_\sigma \mathbb R)$ be the canonical embedding defined by \eqref{06-08-22a}. Then $(\overline{\rho \times H})\circ \iota = {\rho}$.
\end{lemma}
\begin{proof} Let $a \in A, \ f \in L^1(\mathbb R,A)$. Then
\begin{align*}
&\overline{\rho \times H}(\iota(a)) \rho \times H(f) = \rho\times H(\iota(a) f) = \int_\mathbb R \rho(af(t))e^{i t H} \ \mathrm dt \\
& = {\rho}(a)\int_\mathbb R \rho(f(t))e^{i tH} \ \mathrm dt ={\rho}(a) \rho \times H(f).
\end{align*}
The lemma follows from this since $\rho \times H$ is a non-degenerate representation.
\end{proof}

\begin{lemma}\label{24-05-22a} Let $a \in \mathcal A_\sigma$. Then
\begin{equation}\label{24-05-22b}
e^{i zH}{\rho}(a)e^{-izH} = {\rho}(\sigma_z(a))
\end{equation}
for all $z \in \mathbb C$.
\end{lemma}
\begin{proof} Let $\mathcal A_u$ denote that set of entire analytic elements for the flow $u_t := e^{i t H}$ on $\mathbb H'$. Then $\mathcal A_u  \subseteq D(e^{zH})$ for all $z \in \mathbb C$ by \eqref{15-06-22a} in Example \ref{15-06-22}. For $\eta, \eta' \in \mathcal A_u$ the function
$$
\mathbb C \ni z \mapsto \left<{\rho}(a)e^{-i zH}\eta, e^{-i \overline{z} H}\eta'\right>
$$
is entire analytic. The same is true for the function
$$
\mathbb C \ni z \mapsto \left< {\rho}(\sigma_z(a))\eta,\eta'\right> .
$$
The two functions agree on $\mathbb R$ and hence on all of $\mathbb C$. That is,
\begin{equation}\label{19-08-22i}
 \left<{\rho}(a)e^{-i zH}\eta, e^{-i \overline{z} H}\eta'\right>  = \left< {\rho}(\sigma_z(a))\eta,\eta'\right>
\end{equation}
 for all $z \in \mathbb C$ and all $\eta,  \eta' \in \mathcal A_u$.  It follows from Example \ref{15-06-22} that $\mathcal A_u$ is a core for $e^{-i \overline{z}H}$ and hence \eqref{19-08-22i} implies that ${\rho}(a)e^{-i zH}\eta \in D((e^{-i\overline{z}H})^*) = D(e^{i zH})$ and
 $$
 e^{i zH}{\rho}(a)e^{-izH} \eta = {\rho}(\sigma_z(a))\eta
 $$
 for all $z \in \mathbb C$ and all $\eta \in \mathcal A_u$. Since $\mathcal A_u$ is dense in $\mathbb H'$ this implies \eqref{24-05-22b}.
 \end{proof}

\begin{lemma}\label{06-08-22b} $e^{i t H} \in  (\rho \times H)(A \rtimes_\sigma \mathbb R)''$ for all $t \in \mathbb R$ and  $h(H) \in (\rho \times H)(A \rtimes_\sigma \mathbb R)''$ for all $h \in C_0(\mathbb R)$.
\end{lemma}
\begin{proof} Let $\lambda_s  \in M(A \rtimes_\sigma \mathbb R)$ be the unitaries from Lemma \ref{20-05-22x}. Then
\begin{equation}\label{06-08-22e}
\overline{\rho \times H}(\lambda_s) = e^{is H}.
\end{equation}
To see this, let $f \in C_c(\mathbb R,A)$. It follows from a calculation in the proof of Lemma \ref{20-05-22x} that $\lambda_sf = f_s$, where $f_s(y) = \sigma_s(f(y-s))$. Hence
\begin{align*}
&\overline{\rho \times H}(\lambda_s)(\rho \times H)(f) = (\rho \times H)(\lambda_sf) = (\rho \times H)(f_s) \\
&= \int_\mathbb R \rho(\sigma_s(f(y-s)))e^{i y H} \ \mathrm d y = \int_\mathbb R  e^{is H}\rho(f(y-s))e^{i(y-s)H} \ \mathrm d y \\
&= \int_\mathbb R  e^{is H}\rho(f(x))e^{i x H} \ \mathrm d x = e^{is H}(\rho \times H)(f). 
\end{align*}
Since $\rho \times H$ is non-degenerate and $C_c(\mathbb R,A)$ is dense in $A \rtimes_\sigma \mathbb R$ we obtain \eqref{06-08-22e}. In combination with Lemma \ref{06-08-22d} this implies that $e^{i sH} \in (\rho \times H)(A\rtimes_\sigma \mathbb R)''$. Let $h \in C_0(\mathbb R)$ such that $h = \hat{g}$ for some $g \in L^1(\mathbb R) \cap C_0(\mathbb R)$. Then 
\begin{equation}\label{15-08-22}
h(H) = \int_\mathbb R g(t) e^{ i t H} \ \mathrm d t ,
\end{equation}
cf. e.g. Theorem 5.6.36 of \cite{KR}.
As we have shown, $e^{i t H} \in (\rho \times H)(A\rtimes_\sigma \mathbb R)''$ for all $t$ and it follows therefore from \eqref{15-08-22} that
$$
 \int_\mathbb R g(t) e^{i t H} \ \mathrm d t \in (\rho \times H)(A\rtimes_\sigma \mathbb R)''.
 $$
Indeed, the integral is a limit of Riemann sums in the strong operator topology, cf. Appendix \ref{integration}, and hence the limit is in $(\rho \times H)(A\rtimes_\sigma \mathbb R)''$ since the Riemann sums are. Consider then a general element $h \in C_0(\mathbb R)$. It follows from Lemma \ref{11-10-23} that we can pick a sequence $\{h_n\}$ in $C_0(\mathbb R)$ such that $h_n(H)\in (\rho \times H)(A\rtimes_\sigma \mathbb R)''$ for all $n$, and 
 $$
\lim_{n \to \infty} \sup_{t \in \mathbb R} |h_n(t)-h(t)| = 0. 
$$
Then $\lim_{n \to \infty} h_n(H) = h(H)$ in norm and hence $h(H) \in (\rho \times H)(A\rtimes_\sigma \mathbb R)''$.
\end{proof}

\emph{A semi-finite covariant representation} $(\rho,H,\tau)$ of $\sigma$ is a covariant representation $(\rho,H)$ of $\sigma$ together with a normal faithful semi-finite trace $\tau$ on $\rho \times H(A \rtimes_\sigma \mathbb R)''$. We fix now such a semi-finite covariant representation $(\rho,H,\tau)$.\footnote{There is no saying that such triple $(\rho,H,\tau)$ necessarily exists; we just assume that it does.} To simplify notation we set
$$
M := (\rho \times H)(A\rtimes_\sigma \mathbb R)'' 
$$
and we let $\alpha = (\alpha_t)_{t \in \mathbb R}$ be the normal flow on $M$ induced by $H$; viz.
$$
\alpha_t(a) := \Ad e^{i t H}(a) = e^{i t H}ae^{-i t H} .
$$
As in Section \ref{modular2} we denote by $M_a$ the elements of $M$ that are entire analytic for $\alpha$.

\begin{lemma}\label{15-08-22a} 

\begin{itemize}
\item[(1)] Let $x \in M$ and  and $f,g \in C_c(\mathbb R)$. Then $x f(H)\in M$, $f(H)xg(H) \in M_a$ and
$$
\alpha_z(f(H)xg(H)) = e^{izH}f(H)xg(H)e^{-izH} 
$$
for all $z \in \mathbb C$.
\item[(2)] Let $x \in M_a$ and $f,g \in C_c(\mathbb R)$. Then
$$
\alpha_z(f(H)xg(H)) = f(H)\alpha_z(x)g(H)
$$
for all $z \in \mathbb C$.
\item[(3)] Let $x \in M_a$ and $f \in C_c(\mathbb R)$. Then $xf(H) \in M_a$ and
$$
\alpha_z(xf(H)) = \alpha_z(x)f(H) = e^{izH}xe^{-izH}f(H) 
$$
for all $z \in \mathbb C$
\end{itemize}
\end{lemma}
\begin{proof}  (1): Write $x =x_1x_2$ where $x_j \in M, j =1,2$. Note that
$$
e^{iz H}f(H)x_1  = \sum_{n=0}^\infty \frac{(iH)^nf(H)}{n!}x_1 z^n ,
$$
for all $z \in \mathbb C$ where $\left\|(iH)^nf(H)\right\| = \sup_{t \in \supp f}\left|t^nf(t)\right|$. It follows from this expression and Lemma \ref{06-08-22b} that
$$
\mathbb C \ni z \mapsto e^{iz H}f(H)x_1 \in M
$$
is entire analytic. Similarly, 
$$
\mathbb C \ni z \mapsto x_2g(H)e^{-iz H} \in M
$$
is entire analytic, and consequently
$$
\mathbb C \ni z \mapsto e^{iz H}f(H)xg(H)e^{-izH} = e^{iz H}f(H)x_1 x_2g(H)e^{-iz H}
$$
is entire holomorphic.

(2): This follows from (1) and the observation that $\mathbb C \ni z \mapsto f(H)\alpha_z(x)g(H)$ is entire analytic and agrees with $\alpha_t(f(H)xg(H))$ when $z =t \in \mathbb R$.

(3): Let $\{g_n\}$ be a sequence in $C_c(\mathbb R)$ such that $0 \leq g_n \leq g_{n+1} \leq 1$ and $g_n(t) = 1$ for all $t \in [-n,n]$ and all $n$. Let $\eta \in \mathbb H', \ \eta' \in D( e^{-i\overline{z}H})$. Then
\begin{align*}
&\left< xe^{-izH}f(H)\eta, e^{-i\overline{z}H}\eta'\right> \\
& \overset{(i)}{=} \lim_{k \to \infty}\left< xe^{-izH}f(H)\eta, e^{-i\overline{z}H}g_k(H)\eta'\right> \\
&= \lim_{k \to \infty} \left< e^{izH}g_k(H)xf(H)e^{-izH}\eta, \eta'\right> \\
&  \overset{(ii)}{=}\lim_{k \to \infty} \left< g_k(H)\alpha_z(x)f(H)\eta, \eta'\right> \\
& \overset{(iii)}{=}\left< \alpha_z(x)f(H)\eta, \eta'\right> 
\end{align*}
where (i) follows from spectral theory, (ii) from (1) and (2) and (iii) follows because $\lim_{k \to \infty} g_k(H) =1$ in the weak operator topology. As a consequence, 
$$
 xe^{-izH}f(H)\eta \in D(( e^{-i\overline{z}H})^*) = D(e^{izH}) 
 $$
 and 
 $$
 e^{izH}xe^{-izH}f(H)\eta = \alpha_z(x)f(H)\eta .
 $$
 We conclude therefore that 
$$
e^{izH}xe^{-izH}f(H) = \alpha_z(x)f(H) .
$$
Since $\mathbb C \ni z \mapsto \alpha_z(x)f(H)$ is entire analytic and agrees with $\alpha_t(xf(H))$ when $z = t \in \mathbb R$, this completes the proof.

\end{proof}

It follows from Lemma \ref{06-08-22b} that $f(H)e^{z H} \in  M$ for any $z \in \mathbb C$ when $f \in C_c(\mathbb R)$, and we can therefore define $\tau_\beta : M^+ \to [0,\infty]$ by
$$
\tau_\beta(a) := \sup \left\{ \tau(f(H)e^{-\frac{\beta H}{2}} a e^{-\frac{\beta H}{2}}f(H)) : \ f \in C_c(\mathbb R), \ 0 \leq f \leq 1 \right\} 
$$
for any $\beta \in \mathbb R$.

\begin{lemma}\label{06-08-22f} $\tau_\beta$ is a faithful normal $\alpha$-invariant weight and 
$$
\tau_\beta(a^*a) = \tau_\beta\left(\alpha_{-i \frac{\beta}{2}}(a)\alpha_{-i \frac{\beta}{2}}(a)^*\right)
$$
when $a \in M_a$ is analytic for $\alpha$.
\end{lemma}
\begin{proof} Note that since $\tau$ is lower semi-continuous for the $\sigma$-weak topology the same is true for $\tau_\beta$. If $a \in M^+$ and $\tau_\beta(a) =0$ it follows that 
\begin{equation}\label{20-08-22a}
f(H)e^{\frac{-\beta H}{2}}af(H)e^{-\frac{\beta H}{2}} = 0
\end{equation} 
for all $f \in C_c(\mathbb R), \ f \geq 0$, since $\tau$ is faithful. Let $\{g_n\}$ be a sequence in $C_c(\mathbb R)$ such that $0 \leq g_n \leq g_{n+1} \leq 1$ and $g_n(t) = 1$ for all $t \in [-n,n]$ and all $n$. Taking $f(t) = e^{\frac{\beta t}{2}}g_n(t)$ in \eqref{20-08-22a} we see that $g_n(H)ag_n(H) = 0$ which implies that $a =0$ since $\lim_{n \to \infty} g_n(H) = 1$ strongly. Hence $\tau_\beta$ is faithful. Since $e^{i tH} \in M$ by Lemma \ref{06-08-22b} and $\tau$ is a trace on $M$, we find that
\begin{align*}
& \tau_\beta(\alpha_t(a)) =  \sup \left\{ \tau(e^{i tH}f(H)e^{-\frac{\beta H}{2}} a e^{-\frac{\beta H}{2}}f(H)e^{-i tH}) : \ f \in C_c(\mathbb R), \ 0 \leq f \leq 1 \right\} \\
& =  \sup \left\{ \tau(f(H)e^{-\frac{\beta H}{2}} a e^{-\frac{\beta H}{2}}f(H)) : \ f \in C_c(\mathbb R), \ 0 \leq f \leq 1 \right\} = \tau_\beta(a) ,
\end{align*}
showing that $\tau_\beta$ is $\alpha$-invariant. To see that $\tau_\beta$ is additive, note that if $0 \leq f \leq g$ in $C_c(\mathbb R)$ and $a = a^* \in M$, then
\begin{align*} 
&\tau(f(H)e^{-\frac{\beta H}{2}}a^2e^{-\frac{\beta H}{2}}f(H)) = \tau(a f^2(H)e^{-{\beta H}}a) \\
& \leq \tau(a g^2(H)e^{-{\beta H}}a) = \tau(g(H)e^{-\frac{\beta H}{2}}a^2e^{-\frac{\beta H}{2}}g(H)) .
\end{align*} 
It follows that
\begin{equation}\label{15-08-22fx}
\tau_\beta(b) = \lim_{n \to \infty} \tau(g_n(H)e^{-\frac{\beta H}{2}}be^{-\frac{\beta H}{2}}g_n(H))
\end{equation}
for all $b \in M^+$. This implies that $\tau_\beta$ is a additive since $\tau$ is. When $a \in M_a$ we find that
\begin{align*}
& \tau_\beta(a^*a)  =  \lim_{n \to \infty}  \tau(g_n(H)e^{-\frac{\beta H}{2}}a^*ae^{-\frac{\beta H}{2}}g_n(H)) \\
& \overset{(i)}{=}  \lim_{n \to \infty} \lim_{k \to \infty}  \tau(g_n(H)e^{-\frac{\beta H}{2}}a^*g_k^2(H)ae^{-\frac{\beta H}{2}}g_n(H)) \\
 &=    \lim_{n \to \infty} \lim_{k \to \infty}   \tau(  g_k(H)a g_n^2(H)e^{-{\beta H}}a^*g_k(H))\\
&=    \lim_{n \to \infty} \lim_{k \to \infty}   \tau(  g_k(H)e^{-\frac{\beta H}{2}} e^{\frac{\beta H}{2}}a g_n(H)e^{-\frac{\beta H}{2}}e^{-\frac{\beta H}{2}}g_n(H)a^*e^{\frac{\beta H}{2}}e^{-\frac{\beta H}{2}}g_k(H))\\ 
 &\overset{(ii)}{=} \lim_{n \to \infty} \lim_{k \to \infty} \tau\left( g_k(H)e^{-\frac{\beta H}{2}} \alpha_{-i \frac{\beta}{2}}(ag_n(H))\alpha_{-i \frac{\beta}{2}}(ag_n(H))^* e^{-\frac{\beta H}{2}} g_k(H)\right)\\
 &=  \lim_{n \to \infty}  \tau_\beta\left(\alpha_{-i \frac{\beta}{2}}(ag_n(H))\alpha_{-i \frac{\beta}{2}}(a g_n(H))^*\right)\\
& \overset{(ii)}{=} \lim_{n \to \infty}\tau_\beta\left(\alpha_{-i \frac{\beta}{2}}(a)g_n^2(H)\alpha_{-i \frac{\beta}{2}}(a)^*\right) \overset{(iii)}{=}  \tau_\beta\left(\alpha_{-i \frac{\beta}{2}}(a)\alpha_{-i \frac{\beta}{2}}(a)^*\right) ,
 \end{align*}
where (i) and (iii) follow from the lower semi-continuity of $\tau_\beta$, using that $\lim_{n \to \infty} g_n^2(H) = 1$ strongly, and (ii) follows from (3) of Lemma \ref{15-08-22a}.
\end{proof}

It follows from Lemma \ref{06-08-22f} that the map $\tau_\beta \circ \rho: A^+ \to [0,\infty]$ is a weight on $A$ with the property specified in (2) of Kustermans' theorem, Theorem \ref{24-11-21d}. It is easy to see that $\tau_\beta \circ \rho$ is also $\sigma$-invariant, and the reason that this weight is not, in general, a $\beta$-KMS weight for $\sigma$ is that it may not be densely defined.

\begin{defn}\label{07-08-22} An element $m  \in M$ is a \emph{$\beta$-mediator} for $(\rho,H,\tau)$ when
\begin{itemize}
\item[(a)] $e^{itH} m = m e^{itH}$ for all $t \in \mathbb R$,
\item[(b)] $\left\{a \in A:  \tau_\beta(m^* \rho(a^*a)m) < \infty\right\}$ is dense in $A$, and
\item[(c)] $\tau_\beta(\rho(a)mm^*\rho(a)^*) = \tau_{\beta}(m^*\rho(aa^*)m)$ for all $a \in \mathcal A_\sigma$.
\end{itemize}
 \end{defn}

\begin{lemma}\label{07-08-22a} Let $m$ be a non-zero $\beta$-mediator for the semi-finite covariant representation $(\rho,H,\tau)$ of $\sigma$. The map
$$
A^+ \ni a \mapsto \tau_\beta(m^*\rho(a)m)
$$
is a $\beta$-KMS weight for $\sigma$.
\end{lemma}
\begin{proof} Set $\psi(a) :=  \tau_\beta(m^*\rho(a)m)$.
It follows from Lemma \ref{06-08-22f} that $\psi$ is a weight on $A$. Since $m \neq 0$ and $\rho$ is a non-degenerate representation there is an $a \in A^+$ and an $f \in C_c(\mathbb R)$, $0 \leq f \leq 1$, such that
$$
e^{-\frac{\beta H}{2}}f(H)m^*\rho(a) \neq 0 .
$$
Since $\tau$ is faithful this implies that $\psi(a) > 0$. Hence $\psi \neq 0$. It follows from condition (b) of Definition \ref{07-08-22} that $\psi$ is densely defined. Using condition (a) of Definition \ref{07-08-22} and that $\tau_\beta$ is $\alpha$-invariant by Lemma \ref{06-08-22f} we find 
\begin{align*}
& \psi(\sigma_t(a)) =  \tau_\beta(m^*e^{itH}\rho(a)e^{-itH}m) = \tau_\beta(e^{itH}m^*\rho(a)me^{-itH}) \\
& = \tau_\beta(\alpha_t(m^*\rho(a)m))) =  \tau_\beta(m^*\rho(a)m) = \psi(a),
\end{align*}
proving that $\psi$ is $\sigma$-invariant. When $a \in \mathcal A_\sigma$ the element $\rho(a)m$ is entire analytic for $\alpha$ since $m$ is fixed by $\alpha$ by condition (a) in Definition \ref{07-08-22}. It follows therefore from Lemma \ref{06-08-22f} that
\begin{align*}
& \psi(a^*a) = \tau_\beta(m^*\rho(a^*a)m) = \tau_\beta(\alpha_{-i \frac{\beta}{2}}(\rho(a)m) \alpha_{-i \frac{\beta}{2}}(\rho(a)m)^*) .
\end{align*}
Since $m$ is fixed by $\alpha$,
\begin{align*}
&\tau_\beta(\alpha_{-i \frac{\beta}{2}}(\rho(a)m) \alpha_{-i \frac{\beta}{2}}(\rho(a)m)^*) = \tau_\beta(\alpha_{-i \frac{\beta}{2}}(\rho(a))mm^* \alpha_{-i \frac{\beta}{2}}(\rho(a))^*).
\end{align*}
Since $\alpha_{-i \frac{\beta}{2}}(\rho(a)) = \rho(\sigma_{-i\frac{\beta}{2}}(a))$ and $\sigma_{-i\frac{\beta}{2}}(a) \in \mathcal A_\sigma$, we can use condition (c) of Definition \ref{07-08-22} to conclude that
\begin{align*}
& \tau_\beta(\alpha_{-i \frac{\beta}{2}}(\rho(a))mm^* \alpha_{-i \frac{\beta}{2}}(\rho(a))^*) =  \tau_\beta(\rho(\sigma_{-i\frac{\beta}{2}}(a))mm^* \rho(\sigma_{-i\frac{\beta}{2}}(a)^*))\\
& = \tau_\beta(m^*\rho(\sigma_{-i\frac{\beta}{2}}(a)\sigma_{-i\frac{\beta}{2}}(a)^*)m) = \psi(\sigma_{-i\frac{\beta}{2}}(a)\sigma_{-i\frac{\beta}{2}}(a)^*))  .
\end{align*}
Hence $\psi$ is a $\beta$-KMS weight for $\sigma$ by Kustermans' theorem, Theorem \ref{24-11-21d}.
\end{proof}

We aim now for a proof of the following theorem which can be seen as a converse to Lemma \ref{07-08-22a}.

\begin{thm}\label{07-08-22bx} Let $\beta \in \mathbb R$ and let $\psi$ be a $\beta$-KMS weight for $\sigma$. There is a semi-finite covariant representation $(\rho,H,\tau)$ of $\sigma$ such that $g(H)$ is a $\beta$-mediator for $(\rho,H,\tau)$ when $g \in C_c(\mathbb R)$ is infinitely differentiable, in which case
$$
\boxed{\psi(a) = \tau\left(e^{-\frac{\beta H}{2}}g(H)^* \rho(a)g(H)e^{-\frac{\beta H}{2}}\right) }
$$
for all $a \in A^+$ when, in addition, $\int_\mathbb R |g(t)|^2 \ \mathrm dt = 2\pi$. More generally, $\widehat{f}(H)$ is a $\beta$-mediator for $(\rho,H,\tau)$ when $f \in L^1(\mathbb R) \cap L^2(\mathbb R)$, in which case
\begin{equation}\label{09-09-23holdnuop}
\psi(a) = \tau_\beta(\widehat{f}(H)^*\rho(a)\widehat{f}(H))
\end{equation}
for all $a \in A^+$ when, in addition, $\int_\mathbb R |f(t)|^2 \ \mathrm dt = 1$.
\end{thm}

The crucial property of $g$ is not really that it is smooth; what matters is that it has compact support and is the Fourier transform of a function from $L^1(\mathbb R) \cap L^2(\mathbb R)$. This property is guaranteed when $g$ is infinitely differentiable and has compact support since it implies that $g$ is a Schwartz function.

The proof of Theorem \ref{07-08-22bx} requires some preparation. Fix the notation and setting of Section \ref{8.2.2}. Define a representation $\rho_\psi$ of $A$ on $L^2(\mathbb R,H_\psi)$ such that
$$
(\rho_\psi(a)\xi)(x) := \pi_{\psi}(\sigma_{-x}(a))\xi(x) 
$$
when $\xi \in L^2(\mathbb R,H_\psi)$. Define a unitary representation $l$ of $\mathbb R$ on $L^2(\mathbb R,H_\psi)$ such that
$$
(l_t\xi)(x) := \xi(x-t) 
$$
when $t,x \in \mathbb R$ and $\xi \in L^2(\mathbb R,H_\psi)$. Note that
\begin{equation}\label{kerteminde}
\Ad l_t \circ \rho_\psi(a) = \rho_\psi(\sigma_t(a)) 
\end{equation}
for all $t,a$. By Stone's theorem there is a self-adjoint operator $D$ such that $l_t = e^{it D}$ for all $t \in \mathbb R$. This $D$ is the self-adjoint operator 
$$
D := i \frac{d}{dx} \otimes \id_{H_\psi},
$$ 
but we shall not need this formula. The pair $(\rho_\psi,D)$ is a covariant representation of $(A,\sigma)$ and we can therefore consider the integrated representation
$$
\rho_\psi \times D:  A \rtimes_\sigma \mathbb R \to B(L^2(\mathbb R,H_\psi)).
$$
We consider now the dual KMS weight 
$$
\widehat{\psi} = D_3(\psi) \in \KMS(\underline{\sigma},\beta)^{\hat{\sigma}}
$$ 
on $A \rtimes_\sigma \mathbb R$ arising from Corollary \ref{09-06-22d}. By construction the GNS-triple 
$$
(H_{\widehat{\psi}},\Lambda_{\widehat{\psi}},\pi_{\widehat{\psi}})
$$ 
of $\widehat{\psi}$ is isomorphic to 
$$
(L^2(\mathbb R,H_\psi),\Lambda, \widehat{\pi}_\psi),
$$
and the unitary $W$ implementing the isomorphism has the property that $W \Lambda_{\widehat{\psi}}(f) = \Lambda(f)$ when $f \in D(\Lambda)$, and 
$\Ad W \circ \pi_{\widehat{\psi}} = \widehat{\pi}_\psi$, cf.  Section \ref{07-08-22e} and Section \ref{GNS-KMS}.

\begin{lemma}\label{16-06-22x}
$$
\Ad W \circ \pi_{\widehat{\psi}} = \widehat{\pi}_\psi = \rho_\psi \times D .
$$
\end{lemma}
\begin{proof}  Let $f \in C_c(\mathbb R), \ a \in A$ and consider the element $f \otimes a \in C_c(\mathbb R,A)$ defined such that $f \otimes a(t) = f(t)a$. Then
\begin{align*}
&((\rho_\psi \times D)(f \otimes a)\xi)(x) = \int_{\mathbb R} f(t) (\rho_{\psi}(a) l_t\xi)(x) \ \mathrm d t \\
&= \int_{\mathbb R} f(t) \pi_\psi(\sigma_{-x}(a))\xi(x-t) \ \mathrm d t = (\widehat{\pi}_\psi(f \otimes a)\xi)(x)
\end{align*} 
for all $\xi \in L^2(\mathbb R,H_\psi)$. Since $\{f \otimes a: \ f \in C_c(\mathbb R), \ a \in A \}$ spans a dense subspace in $A \rtimes_\sigma \mathbb R$, it follows that $\rho_\psi \times D = \widehat{\pi}_\psi$, completing the proof.
\end{proof}

Since $\lambda_t \in M(A \rtimes_\sigma \mathbb R)$ by Lemma \ref{20-05-22x} there is a $*$-homomorphism $\kappa: C^*(\mathbb R) \to M(A \rtimes_\sigma \mathbb R)$ such that $\kappa(f) = m_f$, where
$$
m_f : = \int_\mathbb R f(t) \lambda_t \ \mathrm d t 
$$ 
when $f \in L^1(\mathbb R)$. \footnote{The $C^*$-algebra $C^*(\mathbb R)$ is the group $C^*$-algebra of $\mathbb R$ and equal to the crossed product $C^*$-algebra of the trivial action of $\mathbb R$ on $\mathbb C$.}
We note that
$$
(m_f\xi)(x) = \int_\mathbb R f(t)\xi(x-t) \ \mathrm d t
$$
for all $\xi \in L^2(\mathbb R,\mathbb H)$, and that
\begin{equation}\label{04-09-22}
\kappa(C^*(\mathbb R)) \subseteq M(A \rtimes_\sigma \mathbb R)^{\underline{\sigma}},
\end{equation}
the fixed point algebra of the action $\underline{\sigma}$ on $M(A \rtimes_\sigma \mathbb R)$, since $\underline{\sigma}_t = \Ad \lambda_t$ by Lemma \ref{20-05-22ax}.

Note that when $f \in L^1(\mathbb R)$ and $a \in A$ we can consider the element $f \otimes a \in L^1(\mathbb R,A) \subseteq A \rtimes_\sigma \mathbb R$ defined such that $f \otimes a(t) = f(t)a$ for all $t \in \mathbb R$.

\begin{lemma}\label{20-05-22g} Let $f \in L^1(\mathbb R), \ a \in A$. Then $f \otimes a = \iota(a)m_f$ and
$$
( f\otimes a)^*(f\otimes a) = m_f^*\iota(a^*a) m_f 
$$
in $M(A \rtimes_\sigma \mathbb R)$.
\end{lemma}
\begin{proof} Since the second identity follows from the first it suffices to observe that
\begin{align*}
& (\iota(a)m_f\xi)(x) = \sigma_{-x}(a)\left( m_f\xi(x)\right) = \sigma_{-x}(a) \int_{\mathbb R} f(y) \xi(x-y) \ \mathrm d y \\
& = (\pi(f \otimes a)\xi)(x)
 \end{align*}
 for all $\xi \in L^2(\mathbb R,\mathbb H)$.
\end{proof}

\begin{lemma}\label{21-08-22} Let $f \in L^1(\mathbb R) \cap L^2(\mathbb R)$. Then
\begin{equation}\label{20-05-22ix} \widehat{\psi}(m_f^*\iota(a^*a)m_f) = \psi(a^*a) \int_\mathbb R |f(t)|^2 \ \mathrm d t 
\end{equation}
for all $a \in A$.
\end{lemma}
\begin{proof} Assume first that $f \in C_c(\mathbb R)$ and $a \in \mathcal N_\psi$. Then $f \otimes a \in C_c(\mathbb R,A)_\psi$ and it follows therefore from (d) of Theorem \ref{19-08-22a} that
$$
\widehat{\psi}((f \otimes a)^\sharp \star (f\otimes a)) = \psi(a^*a)\int_\mathbb R |f(t)|^2 \ \mathrm d t .
$$
By Lemma \ref{20-05-22g} this means that we have established \eqref{20-05-22ix} when $f \in C_c(\mathbb R)$ and $a \in \mathcal N_\psi$. Consider then a function $f \in L^1(\mathbb R) \cap L^2(\mathbb R)$. Let $\{g_n\}$ be a sequence in $C_c(\mathbb R)$ such that $\lim_{n \to \infty} g_n = f$ both in $L^1(\mathbb R)$ and in $L^2(\mathbb R)$. \footnote{For the construction of such a sequence see Theorems 3.13 and 3.14 in \cite{Ru0}, for example.} Then $g_n \otimes a \in C_c(\mathbb R,A)_\psi$ and from (d) of Theorem \ref{19-08-22a} we find that
$$
\left\|\Lambda_{\widehat{\psi}}(g_n \otimes a) - \Lambda_{\widehat{\psi}}(g_m \otimes a)\right\|^2 = \psi(a^*a)\int_\mathbb R \left|g_n(t) - g_m(t)\right|^2  \ \mathrm d t 
$$
for all $n,m$. It follows that $\left\{\Lambda_{\widehat{\psi}}(g_n \otimes a)\right\}$ is Cauchy and hence convergent in $H_{\widehat{\psi}}$. Since $\lim_{n \to \infty} g_n \otimes a = f\otimes a$ in $L^1(\mathbb R,A)$ and hence in $A \rtimes_\sigma \mathbb R$, it follows that $f \otimes a \in \mathcal N_{\widehat{\psi}}$ and $\lim_{n \to \infty} \Lambda_{\widehat{\psi}}(g_n \otimes a) = \Lambda_{\widehat{\psi}}(f \otimes a) $ in $H_{\widehat{\psi}}$ because $\Lambda_{\widehat{\psi}}$ is closed. Therefore
\begin{align*}
&   \widehat{\psi}((f \otimes a)^*(f \otimes a)) = \lim_{n \to \infty} \widehat{\psi}((g_n \otimes a)^\sharp \star (g_n \otimes a))\\
& =  \psi(a^*a)\lim_{n \to \infty}\int_\mathbb R |g_n(t)|^2  \ \mathrm d t  = \psi(a^*a)\int_\mathbb R |f(t)|^2  \ \mathrm d t .
\end{align*}
We have established \eqref{20-05-22ix} when $a \in \mathcal N_\psi$. But then it holds for all $a \in A$ by the following reasoning. Fix $f \in L^1(\mathbb R) \cap L^2(\mathbb R)$, $f \neq 0$. Since $m_f^*\iota(b)m_f \in A\rtimes_\sigma \mathbb R$ by Lemma \ref{20-05-22g}, we can define $\psi' : A^+ \to [0,\infty]$ such that
$$
\psi'(b) :=  \widehat{\psi}(m_f^*\iota(b)m_f) \ \ \ \forall b \in A^+.
$$
Then $\psi'$ is a weight since $\widehat{\psi}$ is, and $\psi'$ agrees with $(\int_\mathbb R |f(t)|^2  \ \mathrm d t)\psi$ on $\mathcal M_\psi^+$. In particular, $\psi'$ is non-zero and densely defined.
Let $a \in \mathcal M_\psi^\sigma$. Then $\sigma_{-i\frac{\beta}{2}}(a) \in \mathcal M_\psi^\sigma$ and hence
\begin{align*}
&\psi'(a^*a) = \psi(a^*a)\int_\mathbb R |f(t)|^2  \ \mathrm d t = \psi(\sigma_{-i\frac{\beta}{2}}(a)\sigma_{-i\frac{\beta}{2}}(a)^* )\int_\mathbb R |f(t)|^2  \ \mathrm d t \\
&= \psi'(\sigma_{-i\frac{\beta}{2}}(a)\sigma_{-i\frac{\beta}{2}}(a)^* ) .
\end{align*}
It follows therefore from Kustermans' theorem, Theorem \ref{24-11-21d}, that $\psi'$ is a $\beta$-KMS weight for $\sigma$, and then Lemma \ref{31-03-22c} shows that $\psi' = \psi\int_\mathbb R |f(t)|^2  \ \mathrm d t$.
\end{proof}

\begin{lemma}\label{10-08-22a}$W \overline{\pi_{\widehat{\psi}}}(\lambda_t)W^*  = l_t$ for all $t \in \mathbb R$.
\end{lemma}
\begin{proof} Let $f \in C_c(\mathbb R, A)$. From the first calculation in the proof of Lemma \ref{20-05-22x} we deduce that $\lambda_tf = f_t$ in $C_c(G,A) \subseteq A \rtimes_\sigma \mathbb R$, where
$$
 f_t(y) = \sigma_t(f(y-t)) .
 $$ 
It follows therefore from Lemma \ref{16-06-22x} that
\begin{align*} & W \overline{\pi_{\widehat{\psi}}}(\lambda_t) \pi_{\widehat{\psi}}(f)
 =W\pi_{\widehat{\psi}}(\lambda_t f) =W\pi_{\widehat{\psi}}(f_t) = \widehat{\pi}_\psi(f_t)W .
 \end{align*}
Let $\xi \in L^2(\mathbb R,H_\psi)$. Then
 \begin{align*}
 & (l_t\widehat{\pi}_\psi(f) \xi)(x) = (\widehat{\pi}_\psi(f) \xi)(x-t) \\
 & = \int_{\mathbb R} \pi_\psi(\sigma_{t-x}(f(y))) \xi(x-t-y) \ \mathrm d y
 \end{align*}
and hence
 \begin{align*}
 &(\widehat{\pi}_\psi(f_t)\xi)(x) =\int_{\mathbb R} \pi_\psi(\sigma_{-x}(f_t(y)))\xi(x-y) \ \mathrm d y\\
 & =\int_{\mathbb R} \pi_\psi( \sigma_{t-x}(f(y-t)))\xi(x-y) \ \mathrm d y \\
 & =  \int_{\mathbb R} \pi_\psi(\sigma_{t-x}(f(y'))) \xi(x-t-y') \ \mathrm d y' \\
 & = (l_t \widehat{\pi}_\psi(f)\xi)(x).
 \end{align*}
It follows that $W \overline{\pi_{\widehat{\psi}}}(\lambda_t) \pi_{\widehat{\psi}}(f) =\pi_{\widehat{\psi}}(f_t)W = l_t \widehat{\pi}_\psi(f) W = l_tW  \pi_{\widehat{\psi}}(f)$. The desired equality follows from this since $\pi_{\widehat{\psi}}$ is non-degenerate and $C_c(\mathbb R,A)$ is dense in $A \rtimes_\sigma \mathbb R$.
\end{proof}

Set
$$
N := (\rho_\psi \times D(A \rtimes_\sigma \mathbb R))'' .
$$ 
It follows from Lemma \ref{16-06-22x} that
$$
W\pi_{\widehat{\psi}}(A \rtimes_\sigma \mathbb R)''W^* = N.
$$
We let $\mu$ denote the normal flow on $N$ defined as conjugation by $l$, i.e.
$$
\mu_t := \Ad l_t = \Ad e^{i t D} ,
$$
cf. Lemma \ref{06-08-22b}.
\begin{lemma}\label{21-08-22a} Let $\underline{\sigma}''$ be the normal flow on $\pi_{\widehat{\psi}}(A \rtimes_\sigma \mathbb R)''$ such that $\underline{\sigma}''_t \circ \pi_{\widehat{\psi}} = \pi_{\widehat{\psi}} \circ \underline{\sigma}_t$ for all $t \in \mathbb R$, cf. \eqref{08-03-22b}. Then
$$
\Ad W \circ \underline{\sigma}''_t = \mu_t \circ \Ad W
$$
for all $t \in \mathbb R$.
\end{lemma}
\begin{proof}
 It is enough to show that $\Ad W \circ \underline{\sigma}''_t \circ \pi_{\widehat{\psi}} = \mu_t \circ \Ad W \circ \pi_{\widehat{\psi}}$. For this note that,
\begin{align*}
& \Ad W \circ \underline{\sigma}''_t \circ \pi_{\widehat{\psi}} = \Ad W\circ \pi_{\widehat{\psi}}\circ \underline{\sigma}_t \\
& \overset{(i)}{=}  \Ad W\circ \pi_{\widehat{\psi}}\circ \Ad \lambda_t \\
& = \Ad W \circ \Ad \overline{\pi_{\widehat{\psi}}}(\lambda_t)\circ \pi_{\widehat{\psi}}\\
& \overset{(ii)}{=} \Ad l_t \circ \Ad W\circ \pi_{\widehat{\psi}} \\
& = \mu_t \circ \Ad W\circ \pi_{\widehat{\psi}},
\end{align*}
where (i) follows from Lemma \ref{20-05-22ax} and (ii) from Lemma \ref{10-08-22a}.
\end{proof}

Consider now the faithful normal semi-finite weight $\widehat{\psi}''$ on $\pi_{\widehat{\psi}}(A \rtimes_\sigma \mathbb R)''$ obtained by applying Theorem \ref{21-02-22dx} to $\widehat{\psi}$. Then
$$
\phi : = \widehat{\psi}'' \circ \Ad W^*
$$
is a  faithful normal semi-finite weight $\phi$ on $N$.  It follows from Lemma \ref{02-03-22g} and Lemma \ref{21-08-22a} that
\begin{equation}\label{13-08-22ax}
\phi(m^*m) = \phi(\mu_{-i \frac{\beta}{2}}(m) \mu_{-i \frac{\beta}{2}}(m)^*)
\end{equation}
when $m \in N_a$ is analytic for $\mu$. It follows from Lemma \ref{06-08-22b} that $f(D) \in N$ for all $f \in C_0(\mathbb R)$ and we can therefore define $\tau : N^+ \to [0,\infty]$ by
$$
\tau(a) := \sup \left\{ \phi(f(D)e^{\frac{\beta D}{2}} a e^{\frac{\beta D}{2}}f(D)) : \ f \in C_c (\mathbb R), \ 0 \leq f \leq 1 \right\} .
$$

\begin{lemma}\label{21-08-22c}
 $\tau$ is a normal faithful semi-finite trace.
\end{lemma}
\begin{proof} It is clear that that $\tau$ is homogeneous and lower semi-continuous for the $\sigma$-weak topology since $\phi$ is. Assume that $x \in N^+$ and that $\tau(x) = 0$. Since $\phi$ is faithful it follows that 
\begin{equation}\label{21-08-22d}
f(D)e^{\frac{\beta D}{2}}xe^{\frac{\beta D}{2}}f(D) =0
\end{equation} 
for all $f \in C_c(\mathbb R)$ between $0$ and $1$, and hence in fact for all elements $f \in C_c(\mathbb R)$. Let $\{g_n\}$ be a sequence in $C_c(\mathbb R)$ such that $0 \leq g_n \leq g_{n+1} \leq 1$, $g_n(t) = 1$ for all $t \in [-n,n]$ and $g_n(t) = 0$ for all $t \notin [-n-1,n+1]$. Substituting $g_n(t)e^{-\frac{\beta t}{2}}$ for all $f$ in \eqref{21-08-22d} it follows that $g_n(D)xg_n(D) =0$ for all $n$, and hence that $x =0$ since $\lim_{n \to \infty} g_n(D) =1$ strongly. This shows that $\tau$ is faithful.

Let $0 \leq f \leq g$ in $C_c(\mathbb R)$ and let $a \in N^+$. Since $\phi$ is lower semi-continuous for the $\sigma$-weak topology and $\lim_{k \to \infty} g_k(D) =1$ strongly,
\begin{align*}
&\phi\left(e^{\frac{\beta D}{2}} f(D) a f(D) e^{\frac{\beta D}{2}}\right)  = \lim_{k \to \infty} \phi\left(e^{\frac{\beta D}{2}} f(D) \sqrt{a}g_k(D)^2\sqrt{a} f(D) e^{\frac{\beta D}{2}}\right) . 
\end{align*}
We can then combine (1) of Lemma \ref{15-08-22a} with \eqref{13-08-22ax} to conclude that
\begin{align*}
& \phi\left(e^{\frac{\beta D}{2}} f(D) \sqrt{a}g_k(D)^2\sqrt{a} f(D) e^{\frac{\beta D}{2}}\right) \\
&= \phi\left( \mu_{-i\frac{\beta}{2}}\left(g_k(D)\sqrt{a} f(D) e^{\frac{\beta D}{2}}\right) \mu_{-i\frac{\beta}{2}}\left(g_k(D)\sqrt{a} f(D)e^{\frac{\beta D}{2}} \right)^*\right) \\
& = \phi\left(e^{\frac{\beta D}{2}}g_k(D)\sqrt{a}f^2(D)\sqrt{a}g_k(D) e^{\frac{\beta D}{2}}\right).
\end{align*}
Similarly,
\begin{align*}
& \phi\left(e^{\frac{\beta D}{2}} g(D) \sqrt{a}g_k(D)^2\sqrt{a} g(D) e^{\frac{\beta D}{2}}\right) = \phi\left(e^{\frac{\beta D}{2}}g_k(D)\sqrt{a}g^2(D)\sqrt{a}g_k(D) e^{\frac{\beta D}{2}}\right),
\end{align*}
and hence
\begin{align*}
& \phi\left(e^{\frac{\beta D}{2}} f(D) \sqrt{a}g_k(D)^2\sqrt{a} f(D) e^{\frac{\beta D}{2}}\right) \leq \phi\left(e^{\frac{\beta D}{2}} g(D) \sqrt{a}g_k(D)^2\sqrt{a} g(D) e^{\frac{\beta D}{2}}\right) 
\end{align*}
since $f^2(D) \leq g^2(D)$. By letting $k$ tend to infinity we conclude that
$$
\phi\left(e^{\frac{\beta K}{2}} f(D) af(D) e^{\frac{\beta D}{2}}\right) \leq \phi\left(e^{\frac{\beta D}{2}} g(D) a g(D) e^{\frac{\beta D}{2}}\right)  ,
$$
implying that the expression $\phi\left(e^{\frac{\beta D}{2}} f(D) a f(D) e^{\frac{\beta D}{2}}\right)$ increases with $f$. It follows that
\begin{equation}\label{15-08-22ex}
\tau(a) = \lim_{k \to \infty} \phi\left(e^{\frac{\beta D}{2}} g_k(D) a g_k(D) e^{\frac{\beta D}{2}}\right) ,
\end{equation}
and therefore also that $\tau$ is additive and hence a weight.

Let $a \in N$ and $k,l \in \mathbb N$. It follows from (1) of Lemma \ref{15-08-22a} that
\begin{align*}
& \phi\left(e^{\frac{\beta D}{2}} g_k(D)a^* g_l(D)^2a  g_k(D) e^{\frac{\beta D}{2}}\right)\\
& =  \phi\left(\mu_{-i\frac{\beta}{2}} \left(g_k(D)a^*  g_l(D) e^{\frac{\beta D}{2}}\right)\mu_{-i\frac{\beta}{2}} \left(g_k(D)a^*  g_l(D) e^{\frac{\beta D}{2}}\right)^*\right).
\end{align*}
Thanks to \eqref{13-08-22ax} this implies that
\begin{align*}
& \phi\left(e^{\frac{\beta D}{2}} g_k(D)a^* g_l(D)^2a  g_k(D) e^{\frac{\beta D}{2}}\right)=  \phi\left( e^{\frac{\beta D}{2}}g_l(D)a  g_k^2(D) a^*g_l(D) e^{\frac{\beta D}{2}}\right) .
\end{align*}
Using the lower semi-continuity of $\phi$ it follows that
\begin{align*}
&  \phi\left(e^{\frac{\beta D}{2}} g_k(D)a^*a  g_k(D) e^{\frac{\beta D}{2}}\right)= \lim_{l \to \infty} \phi\left(e^{\frac{\beta D}{2}} g_k(D)a^* g_l(D)^2a  g_k(D) e^{\frac{\beta D}{2}}\right)\\
& = \lim_{l \to \infty} \phi\left( e^{\frac{\beta D}{2}}g_l(D)a  g_k^2(D) a^*g_l(D) e^{\frac{\beta D}{2}}\right)  = \tau\left(a g_k(D)^2a^*\right) .
\end{align*}
Letting $k$ tend to infinity and using the lower semi-continuity of $\tau$ we conclude that $\tau(a^*a) = \tau(aa^*)$.

It remains only to show that $\tau$ is semi-finite. Let $x \in N^+$. Since $\phi$ is semi-finite, $\mathcal M_\phi$ is $\sigma$-weakly dense in $N$ and hence so is $\mathcal N_\phi \cap \mathcal N_\phi^*$. It follows then from Lemma \ref{02-03-22bx} and Lemma \ref{02-03-22d} that the $*$-algebra $N_a \cap \mathcal N_\phi \cap \mathcal N_\phi^*$ is $\sigma$-weakly dense in $N$. By Kaplansky's density theorem this implies that $x$ can be approximated in the strong operator topology by elements of the form $b^*b$ for some $b \in N_a \cap \mathcal N_\phi$. It suffices therefore to approximate such an element $b^*b$ in the strong operator topology by elements from $\mathcal M_\tau^+$. For this, set $b_n := {b}g_n(D)$. When $k > n$, $g_k(D)g_n(D) = g_n(D)$ and hence
 \begin{align*}
& \phi\left(e^{\frac{\beta D}{2}}g_k(D) b_n^*b_ng_k(D) e^{\frac{\beta D}{2}}\right)   = \phi\left(e^{\frac{\beta D}{2}} b_n^*b_n e^{\frac{\beta D}{2}}\right) .
\end{align*}
Thanks to (3) of Lemma \ref{15-08-22a} we can use \eqref{13-08-22ax} to conclude that
 \begin{align*}
 & \phi\left(e^{\frac{\beta D}{2}} b_n^*b_n e^{\frac{\beta D}{2}}\right) = \phi\left( \mu_{-i\frac{\beta}{2}} (b_n e^{\frac{\beta D}{2}}) \mu_{-i\frac{\beta}{2}} (b_n e^{\frac{\beta D}{2}})^*\right) .
\end{align*}
Since
\begin{align*}
& \phi\left( \mu_{-i\frac{\beta}{2}} (b_n e^{\frac{\beta D}{2}}) \mu_{-i\frac{\beta}{2}} (b_n e^{\frac{\beta D}{2}})^*\right)  = \phi(\mu_{-i\frac{\beta}{2}}(b) g_n(D)^2e^{\beta D} \mu_{-i\frac{\beta}{2}}(b)^* ) \\
& \leq \left\| g_n(D)^2e^{\beta D}\right\|\phi(\mu_{-i\frac{\beta}{2}}(b) \sigma_{-i\frac{\beta}{2}}(b)^* ) = \left\| g_n(D)^2e^{\beta D}\right\|\phi(b^*b) ,
\end{align*}
it follows that 
$$
\tau(b_n^*b_n) = \lim_{k \to \infty}\phi\left(e^{\frac{\beta D}{2}}g_k(D) b_n^*b_ng_k(D) e^{\frac{\beta D}{2}}\right) \leq \left\| g_n(D)^2e^{\beta D}\right\|\phi(b^*b)  < \infty
$$ 
for all $n$. This completes the proof because $\lim_{n \to \infty} b_n^*b_n = b^*b$ strongly.
\end{proof}

We are now ready for the \emph{proof of Theorem \ref{07-08-22bx}:} 
It follows from \eqref{kerteminde} that $({\rho_\psi}, D)$ is a covariant representation of $\sigma$ and hence by Lemma \ref{21-08-22c} that $({\rho_\psi}, D,\tau)$ is a semi-finite covariant representation of $\sigma$. The proof will be completed by showing that $(\rho_\psi, D, \tau)$ has the properties required in Theorem \ref{07-08-22bx}.

We claim that
\begin{equation}\label{15-08-22dx}
\phi = \tau_\beta ;
\end{equation}
the weight from Lemma \ref{06-08-22f} defined using the trace $\tau$ of Lemma \ref{21-08-22c}. To show this let $a \in N$. By using \eqref{15-08-22fx} and \eqref{15-08-22ex} we find that
\begin{align*}
& \tau_\beta(a^*a) = \lim_{n \to \infty} \tau\left(g_n(D)e^{-\frac{\beta D}{2}}a^*a e^{-\frac{\beta D}{2}}g_n(D)\right) \\
& =  \lim_{n \to \infty} \lim_{k \to \infty} \phi\left(g_k(D)e^{\frac{\beta D}{2}}g_n(D)e^{-\frac{\beta D}{2}}a^*a e^{-\frac{\beta D}{2}}g_n(D)g_k(D)e^{\frac{\beta D}{2}}\right) \\
& = \lim_{n \to \infty}  \phi\left(g_n(D)a^*a g_n(D)\right) 
\end{align*}
since $g_kg_n = g_n$ when $k \geq n+1$.
 If we assume that $a \in N_a$ is entire analytic for $\mu$ we can combine (3) of Lemma \ref{15-08-22a} and \eqref{13-08-22ax} to get
\begin{align*}
&\lim_{n \to \infty} \phi\left(g_n(D)a^*a g_n(D)\right) = \lim_{n \to \infty} \phi(\mu_{-i\frac{\beta}{2}}(a)g_n^2(D)\mu_{-i\frac{\beta}{2}}(a)^*) \\
& = \phi(\mu_{-i\frac{\beta}{2}}(a)\mu_{-i\frac{\beta}{2}}(a)^*) = \phi(a^*a) .
\end{align*}
It follows that $\tau_\beta(a^*a) = \phi(a^*a)$ when $a$ is entire analytic for $\mu$. Let then $a\in N$ be a general element, $a \geq 0$. Using the smoothing operators $\overline{R}_k$ defined from the normal flow $\mu$, cf. Section \ref{smoothing}, we have
$$
\tau_\beta(\overline{R}_k(\sqrt{a})\overline{R}_k(\sqrt{a})) =\phi(\overline{R}_k(\sqrt{a})\overline{R}_k(\sqrt{a})) 
$$
for all $k$ because $\overline{R}_k(\sqrt{a})$ is entire analytic for $\mu$ by Lemma \ref{02-03-22bx}. Since 
$$
\overline{R}_k(\sqrt{a})\overline{R}_k(\sqrt{a})) \leq \overline{R}_k(a)
$$ 
by Kadisons inequality, Proposition 3.2.4 in \cite{BR}, and $\lim_{k \to \infty} \overline{R}_k(\sqrt{a}) = \sqrt{a}$ in the $\sigma$-weak topology by Lemma \ref{02-03-22bx}, it follows from the lower semi-continuity of $\tau_\beta$ that
\begin{equation}\label{16-08-22x}
\tau_\beta(a) \leq \liminf_{k \to \infty} \tau_\beta(\overline{R}_k(\sqrt{a})\overline{R}_k(\sqrt{a})) \leq \liminf_{k \to \infty} \phi(\overline{R}_k(a)) .
\end{equation}
Let $\omega \in N_*$ be a positive normal functional such that $\omega \leq \phi$. Then
\begin{align*}
& \omega\left( \overline{R}_k(a)\right) = \sqrt{\frac{k}{\pi}}\int_{\mathbb R} e^{-kt^2} \omega(\mu_t(a)) \ \mathrm d t  \leq \sqrt{\frac{k}{\pi}}\int_{\mathbb R} e^{-kt^2} \phi(\mu_t(a)) \ \mathrm d t = \phi(a)
\end{align*}
since $\phi$ is $\mu$-invariant. By Haagerups theorem, Theorem \ref{06-04-22a}, this shows that $\phi\left( \overline{R}_k(a)\right) \leq \phi(a)$ for all $k$. Inserted into \eqref{16-08-22x} this implies that $\tau_\beta(a) \leq \phi(a)$. The roles of $\tau_\beta$ and $\phi$ can be interchanged in this argument and we obtain in this way \eqref{15-08-22dx}.

Let $f \in L^1(\mathbb R) \cap L^2(\mathbb R)$ and $a \in A$. Since $m_f^*\iota(a^*a)m_f \in A\rtimes_\sigma \mathbb R$ by Lemma \ref{20-05-22g} we have
\begin{equation}\label{08-08-22a}
\begin{split}
& \widehat{\psi}(m_f^*\iota(a^*a)m_f)= \widehat{\psi}'' \circ \pi_{\widehat{\psi}}(m_f^*\iota(a^*a)m_f) \\
&= \widehat{\psi}''( \overline{\pi_{\widehat{\psi}}}(m_f^*)\overline{\pi_{\widehat{\psi}}} \circ \iota(a^*a) \overline{\pi_{\widehat{\psi}}}(m_f)   ).
\end{split}
\end{equation} 
It follows from Lemma \ref{16-06-22x} that $\Ad W \circ \overline{\pi_{\widehat{\psi}}}(m_f) = \overline{\rho_\psi \times D}(m_f)$ and $\Ad W \circ \overline{\pi_{\widehat{\psi}}} \circ \iota(a^*a)= \overline{\rho_\psi \times D} \circ \iota(a^*a)$. Since $\overline{\rho_\psi \times D} \circ \iota(a^*a) = {\rho_\psi}(a^*a)$ by Lemma \ref{24-05-22} we conclude from \eqref{08-08-22a} that
\begin{equation}\label{21-08-22f}
 \widehat{\psi}(m_f^*\iota(a^*a)m_f) =  \phi( n_f^*{\rho_\psi}(a^*a) n_f)
\end{equation}  
when we set
$$
n_f := \overline{\rho_\psi \times D}(m_f) .
$$
Using Lemma \ref{16-06-22x} and Lemma \ref{10-08-22a} we find that
\begin{equation}\label{21-08-22i}
\begin{split}
&n_f = W\overline{\pi_{\widehat{\psi}}}\left( \int_\mathbb R f(t) \lambda_t \ \mathrm dt\right)W^* \\
&=  \int_\mathbb R f(t) W\overline{\pi_{\widehat{\psi}}}(\lambda_t)W^* \ \mathrm d t = \int_\mathbb R f(t) l_t \ \mathrm dt = \widehat{f}(D).
\end{split}
\end{equation} 
We can therefore combine \eqref{21-08-22f} with \eqref{20-05-22ix} to conclude that
\begin{equation}\label{21-08-22j}
\psi(a^*a) \int_\mathbb R |f(t)|^2 \ \mathrm dt = \phi( \widehat{f}(D)^* {\rho_\psi}(a^*a)  \widehat{f}(D)) .
\end{equation}
Since $\phi = \tau_\beta$ this gives the identity \eqref{09-09-23holdnuop} in the statement of Theorem \ref{07-08-22bx}. Let $a \in \mathcal A_\sigma$. It follows from Lemma \ref{20-05-22g} that $ \iota(a)m_f = f \otimes a$ and we see therefore that $\iota(a)m_f \in \mathcal A_{\underline{\sigma}}$ and, thanks to \eqref{04-09-22} and Lemma \ref{20-05-22ax}, 
$$
\underline{\sigma}_{-i \frac{\beta}{2}}(\iota(a)m_f) = \iota(\sigma_{-i\frac{\beta}{2}}(a))m_f .
$$
Since $\widehat{\psi}$ is a $\beta$-KMS weight for $\underline{\sigma}$ by Corollary \ref{09-06-22d} we have therefore
\begin{align*}
&\widehat{\psi}(m_f^* \iota(a^*a) m_f) = \widehat{\psi}( (\underline{\sigma}_{-i \frac{\beta}{2}}(\iota(a)m_f)(\underline{\sigma}_{-i \frac{\beta}{2}}(\iota(a)m_f)^*)) \\
& = \widehat{\psi}(\iota(\sigma_{-i\frac{\beta}{2}}(a))m_fm_f^*\iota(\sigma_{-i\frac{\beta}{2}}(a)^*) ).
\end{align*}
On the other hand, $\psi$ is a $\beta$-KMS weight for $\sigma$ and it follows therefore by use of \eqref{20-05-22ix} that
\begin{align*}
&\widehat{\psi}(m_f^* \iota(a^*a) m_f) 
 = \psi(a^*a) \int_{\mathbb R} |f(t)|^2 \ \mathrm d t \\
&= \psi\left(\sigma_{-i\frac{\beta}{2}}(a)\sigma_{-i\frac{\beta}{2}}(a)^* \right)\int_\mathbb R|f(t)|^2 \ \mathrm d t \\
 & = \widehat{\psi}\left( m_f^* \iota(\sigma_{-i\frac{\beta}{2}}(a))\iota(\sigma_{-i\frac{\beta}{2}}(a)^*) m_f\right)  .
\end{align*}
By comparing the two expressions for $\widehat{\psi}(m_f^* \iota(a^*a) m_f)$ we find that
$$
\widehat{\psi}(\iota(\sigma_{-i\frac{\beta}{2}}(a))m_fm_f^*\iota(\sigma_{-i\frac{\beta}{2}}(a)^*) ) = \widehat{\psi}\left( m_f^* \iota(\sigma_{-i\frac{\beta}{2}}(a))\iota(\sigma_{-i\frac{\beta}{2}}(a)) m_f\right) .
$$
Inserting $\sigma_{i\frac{\beta}{2}}(a)$ for $a$ it follows that
\begin{equation}\label{09-09-23a}
\widehat{\psi}(\iota(a)m_fm_f^*\iota(a^*) ) = \widehat{\psi}\left( m_f^* \iota(a)\iota(a)^*m_f\right)  \ \ \ \forall a \in \mathcal A_\sigma .
\end{equation}
Note that
\begin{align*}
& \widehat{\psi}(\iota(a)m_fm_f^*\iota(a^*) ) = \hat{\psi}'' \circ \pi_{\widehat{\psi}}(\iota(a)m_fm_f^*\iota(a^*) )  \ \ \ \ \text{(by Lemma \ref{03-03-22fx})} \\
& =  \hat{\psi}''\left( \overline{ \pi_{\widehat{\psi}}}(\iota(a)) \overline{ \pi_{\widehat{\psi}}}(m_f)\overline{ \pi_{\widehat{\psi}}}(m_f)^*\overline{ \pi_{\widehat{\psi}}}(\iota(a^*)\right)\\
& =\phi \circ \Ad W\left( \overline{ \pi_{\widehat{\psi}}}(\iota(a)) \overline{ \pi_{\widehat{\psi}}}(m_f)\overline{ \pi_{\widehat{\psi}}}(m_f)^*\overline{ \pi_{\widehat{\psi}}}(\iota(a^*)\right) \\
& = \phi\left(\overline{\rho_\psi \times D}(\iota(a)) n_fn_f^* \overline{\rho_\psi \times D}(\iota(a^*))\right) \ \ \ \ \text{(by Lemma \ref{16-06-22x})} \\
& = \phi\left(\rho_\psi(a) n_fn_f^* \rho_\psi(a^*)\right) \ \ \ \ \text{(by Lemma \ref{24-05-22})}. \\
\end{align*}
Combine this with \eqref{09-09-23a} and \eqref{21-08-22f} to find that
\begin{equation}\label{21-08-22g}
\phi\left({\rho_\psi}(a)n_fn_f^* {\rho_\psi}(a^*)\right) = \phi\left(n_f^* {\rho_\psi}(a){\rho_\psi}(a)^*n_f\right) \ \ \ \forall a \in \mathcal A_\sigma .
\end{equation}
Since $\phi = \tau_\beta$ by \eqref{15-08-22dx} and $n_f = \widehat{f}(D)$ by \eqref{21-08-22i} this shows that $\widehat{f}(D)$ satisfies condition (c) in Definition \ref{07-08-22}. It follows from \eqref{21-08-22j} that 
$$
\tau_\beta(\widehat{f}(D){\rho_\psi} (a^*a) \widehat{f}(D)) < \infty
$$
when $a \in \mathcal N_\psi$ and hence $\widehat{f}(D)$ satisfies condition (b) in Definition \ref{07-08-22}. Since $\widehat{f}(D)$ clearly also satisfies condition (a) we conclude that $\widehat{f}(D)$ is a $\beta$-mediator for $({\rho_\psi}, D,\tau)$.

Assume then that $g$ is an infinitely differentiable function of compact support with $\int_\mathbb R |g(t)|^2 \ \mathrm d t = 2\pi$. Then $g = \widehat{f}$ for some $f \in L^1(\mathbb R) \cap L^2(\mathbb R)$. In fact, since $g$ is a Schwartz function $f$ is a Scwartz function too, cf. Theorem 7.7 (b) of \cite{Ru2}. By Plancherels formula $\int_\mathbb R|f(t)|^2 \ \mathrm d t =  (2\pi)^{-1}\int_\mathbb R|g(t)|^2 \ \mathrm d t = 1$. It follows that
\begin{align*}
&\psi(a) = \tau_\beta(g(D)^* {\rho_\psi}(a)g(D))   \\
& = \lim_{n \to \infty} \tau(g_n(D) e^{-\frac{\beta D}{2}}g(D)^* {\rho_\psi}(a)g(D)g_n(D)  e^{-\frac{\beta D}{2}} ) \ \ \ \ \ \ \ \ \text{(by \eqref{15-08-22fx})}\\
&= \tau(e^{-\frac{\beta D}{2}}g(D)^* {\rho_\psi}(a)g(D)  e^{-\frac{\beta D}{2}}) ,
\end{align*}
since $g_n(D)g(D)^* = g(D)^*$ and $g(D)g_n(D) = g(D)$ for all large $n$.
\qed

\begin{notes} The main result of this section, Theorem \ref{07-08-22bx}, is new.
\end{notes}

\section{KMS weights for periodic flows}\label{periodicflows}
The trace formula in Theorem \ref{07-08-22bx}
for a general KMS weight was obtained via the crossed product $C^*$-algebra of the given flow. In this section we obtain a different, and in many ways simpler trace formula for the KMS weights of periodic flows by using the fixed point algebra of the flow. 

\begin{lemma}\label{22-02-23} Let $\sigma$ be a flow on the $C^*$-algebra $A$ and $\psi$ a KMS weight for $\sigma$. The restriction $\psi|_{A^\sigma}$ of $\psi$ to the fixed point algebra $A^\sigma$ of $\sigma$ is either the zero map or it is a lower semi-continuous trace on $A^\sigma$.
\end{lemma}
\begin{proof} $\psi|_{A^\sigma}$ is a lower semi-continuous weight since $\psi$ is, and because $A^\sigma \subseteq \mathcal A_\sigma$ it follows from (1) in Theorem \ref{24-11-21d} that $\psi(a^*a) = \psi(aa^*)$ for all $a \in A^\sigma$. It remains therefore only to show that $\psi|_{A^\sigma}$ is densely defined. Let $a \in \left(A^\sigma\right)^+$. Define functions $f_n : [0,\infty) \to [0,\infty)$ and $g_n : [0,\infty) \to [0,1]$  such that
 $$
f_n(t) = \begin{cases} 0, & \ t \in [0,\frac{1}{n}]\\  t - \frac{1}{n}, & \ t \geq \frac{1}{n} ,  \end{cases}
$$
and 
$$
g_n(t) = \begin{cases} nt, & \ t \in [0,\frac{1}{n}]\\  1, & \ t \geq \frac{1}{n} .  \end{cases}
$$
Then $g_{n}(a)\sqrt{f_n(a)} = \sqrt{f_n(a)}$ and $\sqrt{f_n(a)} \in A^\sigma \subseteq \mathcal A_\sigma$. It follows therefore from Proposition \ref{02-12-21x} that $f_n(a) \in \mathcal M^+_\psi$. Since $\lim_{n \to \infty} f_n(a) = a$, it follows that $\psi|_{A^\sigma}$ is densely defined. 
\end{proof}

\begin{spec}\label{13-10-23} Let $\sigma$ be a flow on a unital $C^*$-algebra $A$ and $\sigma^1$ the flow on $C_0(\mathbb R) \oplus A$ given by 
$$
\sigma^1_t(f,a) = (f(\cdot - t),\sigma_t(a)) .
$$
The map $(f,a) \mapsto \int_\mathbb R f(x) \ \mathrm d x$ gives rise to a $0$-KMS weight for $\sigma^1$ whose restriction to the fixed point algebra of $\sigma^1$ is the zero map despite that the fixed point algebra itself is not zero. Hence the the first alternative in Lemma \ref{22-02-23} can occur also when the fixed point algebra is non-zero. I do not know, however, if this can occur when $\psi$ is a $\beta$-KMS weight for $\beta \neq 0$.
\end{spec}

\subsection{Periodic flows and Fej\'ers theorem}

A flow $\sigma$ on a Banach space $X$ is \emph{periodic} when there is a $p > 0$ such that $\sigma_p = \id_X$, or equivalently when $\sigma_{t + p} = \sigma_t$ for all $t \in \mathbb R$. When we need to emphasize the period we say  that $\sigma$ is \emph{$p$-periodic} when this holds.

Given a $p$-periodic flow $\sigma$ and $k \in \mathbb Z$, set
$$
X(k) := \left\{ a \in X: \ \sigma_t(a) = e^{i t k  \frac{2\pi}{p}}a \ \forall t \in \mathbb R \right\} .
$$

\begin{lemma}\label{23-02-23} Let $\sigma$ be a $p$-periodic flow on the Banach space $X$. Then
$$
\lim_{R \to \infty} \frac{1}{R}\int_0^R e^{-itk\frac{2\pi}{p}} \sigma_t(a) \ \mathrm d t = \frac{1}{p}\int_0^p e^{-itk\frac{2\pi}{p}} \sigma_t(a) \ \mathrm d t 
$$
for all $a \in X$ and $k \in \mathbb Z$. 
\end{lemma}
\begin{proof} Let $N \in \mathbb N$. When $ Np \leq R \leq( N+1)p$ we have
\begin{align*}
&\int_0^R e^{-itk\frac{2\pi}{p} } \sigma_t(a) \ \mathrm d t = \int_0^{Np} e^{-itk\frac{2\pi}{p}} \sigma_t(a) \ \mathrm d t + \int_{Np}^R e^{-itk\frac{2\pi}{p}} \sigma_t(a) \ \mathrm d t \\
& =  \int_{Np}^R e^{-itk\frac{2\pi}{p}} \sigma_t(a) \ \mathrm d t + N \int_{0}^{p} e^{-isk\frac{2\pi}{p}} \sigma_s(a) \ \mathrm d s \\
\end{align*}
and hence
\begin{align*}
&  \left\| \frac{1}{R}\int_0^R e^{-itk\frac{2\pi}{p}} \sigma_t(a) \ \mathrm d t - \frac{1}{p}\int_0^p e^{-itk\frac{2\pi}{p}} \sigma_t(a) \ \mathrm d t\right\| \\
& \\ 
&\leq \left\| \frac{1}{R}\int_0^R e^{-itk\frac{2\pi}{p}} \sigma_t(a) \ \mathrm d t - \frac{N}{R}\int_0^p e^{-itk\frac{2\pi}{p}} \sigma_t(a) \ \mathrm d t\right\| \\
& \ \ \ \ \ \ \ \ \ \ \ \ \ \ \ \ \ \ \ \ \ +
\left|\frac{N}{R} - \frac{1}{p}\right|\left\| \int_0^p e^{-itk\frac{2\pi}{p}} \sigma_t(a) \ \mathrm d t\right\|  
\end{align*}
\begin{align*}
& \leq \frac{1}{R} \left\|  \int_{Np}^R e^{-itk\frac{2\pi}{p}} \sigma_t(a) \ \mathrm d t\right\|  +  \|a\| \left|\frac{N}{R} - \frac{1}{p}\right| p\\
& \\
& \leq \|a\| \left( \frac{R-Np}{R} +  \left|\frac{Np}{R} - 1 \right|\right) \leq \frac{2 p \|a\|}{R} .
\end{align*}
The conclusion follows.
\end{proof}

Since
\begin{align*}
& \sigma_s\left( \frac{1}{p}\int_0^p e^{-itk\frac{2\pi}{p}} \sigma_t(a) \ \mathrm d t\right) =  \frac{1}{p}\int_0^p e^{-itk\frac{2\pi}{p}} \sigma_{s+t}(a) \ \mathrm d t \\
& = \frac{1}{p}\int_s^{s+p} e^{-i(t-s)k\frac{2\pi}{p}} \sigma_t(a) \ \mathrm d t = e^{is k \frac{2\pi}{p}}  \frac{1}{p}\int_0^p e^{-itk\frac{2\pi}{p}} \sigma_t(a) \ \mathrm d t,
\end{align*}
we see that
$$
 \frac{1}{p}\int_0^p e^{-itk\frac{2\pi}{p}} \sigma_t(a) \ \mathrm d t \in X(k) .
 $$
It follows that we can define a linear map $Q_k : X \to X(k)$ such that
$$
Q_k(a) :=  \frac{1}{p}\int_0^p e^{-itk\frac{2\pi}{p}} \sigma_t(a) \ \mathrm d t ,
$$
and we note that $\left\|Q_k(a)\right\| \leq \|a\|$ for all $k,a$, and that $Q_k(a) = a$ for all $a \in X(k)$ and all $k \in \mathbb Z$. For $j \in \mathbb N$, set
$$
P_j :=  \sum_{k= -j}^j Q_k . 
$$
The following theorem is a variation of
Fej\'ers theorem on
Fourier series.

\begin{thm}\label{23-02-23a} $\left\| \frac{1}{N+1} \sum_{j=0}^N
P_j(a)\right\| \leq \|a\|$ for all $N \in \mathbb N$ and all $a \in X$, and
$$
\lim_{N \to \infty}  \frac{1}{N+1} \sum_{j=0}^N
P_j(a) = a
$$
for all $a \in X$.

\end{thm}

The proof of Theorem \ref{23-02-23a} uses the following lemma from calculus.

\begin{lemma}\label{23-02-23b}  $\sum_{j=0}^N \sum_{k=-j}^j e^{i kx} \geq 0$ for all $x \in \mathbb R$.
\end{lemma}
\begin{proof} Using that $2 \cos (u)\sin (v) = \sin (u+v) - \sin (u-v)$ we find 
\begin{align*}
&\sin (\frac{x}{2}) \sum_{k=-j}^j e^{i kx} = \sin (\frac{x}{2}) \left(1 + 2\sum_{k=1}^j\cos(kx)\right)  \\
&= \sin (\frac{x}{2}) + \sum_{k=1}^j\sin ((k+\frac{1}{2})x) - \sin ((k - \frac{1}{2})x) = \sin((j+ \frac{1}{2})x) . 
\end{align*}
When $e^{ix} \neq 1$ we find that
\begin{align*}
&\sum_{j=0}^N \sin((j+ \frac{1}{2})x) = \Imag \sum_{j=0}^N e^{i (j+ \frac{1}{2})x} = \Imag \left(e^{i \frac{x}{2}} \sum_{j=0}^N (e^{ix})^j\right) \\
& = \Imag \left(e^{i \frac{x}{2}} \frac{1 -e^{i(N+1)x}}{1 -e^{ix}}\right) =   \Imag \frac{1-e^{i(N+1)x}}{e^{-i\frac{x}{2}} -e^{i\frac{x}{2}} } \\
& =  \Imag \frac{1-e^{i(N+1)x}}{-i \sin \frac{x}{2} } = \frac{1 - \cos((N+1)x)}{\sin (\frac{x}{2})} .
\end{align*}
 It follows that 
 $$
\sum_{j=0}^N \sum_{k=-j}^j e^{i kx} = \frac{1-\cos ((N+1)x)}{(\sin (\frac{x}{2}))^2} \geq 0 .
 $$
for all $x$ with $e^{ix} \neq 1$. Hence the desired inequality holds for all $x$ by continuity. 
\end{proof}

\emph{Proof of Theorem \ref{23-02-23a}:} Let $f : [0,p] \to X$ be a continuous function. Define $\overline{Q_k}(f)\in X$ by
$$
\overline{Q_k}(f) :=  \frac{1}{p}\int_0^p e^{-itk\frac{2\pi}{p}} f(t) \ \mathrm d t ,
$$
and set
$$
\overline{P_j}(f) :=  \sum_{k= -j}^j \overline{Q_k}(f) . 
$$

Then
\begin{align*}
& \frac{1}{N+1} \sum_{j=0}^N \overline{P_j}(f) = \frac{1}{N+1} \sum_{j=0}^N \sum_{k=-j}^j \overline{Q_k}(f) \\
& = \frac{1}{N+1} \sum_{j=0}^N \sum_{k=-j}^j \frac{1}{p} \int_0^p e^{-itk \frac{2\pi}{p}} f(t) \ \mathrm d t \\
& = \frac{1}{p} \int_0^p K_N( - t\frac{2\pi}{p})f(t) \ \mathrm dt
\end{align*}
where $K_N(x) = \frac{1}{N+1} \sum_{j=0}^N \sum_{k=-j}^j e^{i kx}$. It follows from Lemma \ref{23-02-23b} that $K_N$ is non-negative and we conclude therefore that
$$
\left\| \frac{1}{N+1} \sum_{j=0}^N \overline{P_j}(f)\right\| \leq \frac{1}{p} \int_0^p
 K_N( -t \frac{2\pi}{p}) \ \mathrm d t  \ \cdot \  \sup_{t \in [0,p]} \left\|f(t)\right\|.
$$
Since
\begin{align*}
&  \frac{1}{p}
\int_0^p
 K_N( -t \frac{2\pi}{p}) \ \mathrm d t = \frac{1}{2\pi} \int_0^{2\pi} K_N(-t) \ \mathrm d t \\
 & = \frac{1}{N+1} \sum_{j=0}^N \sum_{k=-j}^j \frac{1}{2\pi} \int_0^{2\pi} e^{-i kt} \ \mathrm d t =  \frac{1}{N+1} \sum_{j=0}^N 1 = 1
\end{align*}
we find that
\begin{equation}\label{23-03-23d}
 \left\|\frac{1}{N+1} \sum_{j=0}^N \overline{P_j}(f) \right\| \leq \sup_{t \in [0,p]} \left\|f(t)\right\| .
\end{equation}
Note that for $a \in X$, $P_j(a) = \overline{P_j}(h)$ when $h(t) := \sigma_t(a)$ and hence \eqref{23-03-23d} gives the estimate
$$
\left\| \frac{1}{N+1} \sum_{j=0}^N {P_j}(a) \right\| \leq \|a\|
$$
occurring in Theorem \ref{23-02-23a}.

For $k \in \mathbb Z$ set $q_k(t) := e^{ikt \frac{2\pi}{p}}$ and for $a \in X$, let $q_k \otimes a: [0,p] \to X$ be the function
$$
q_k \otimes a(t) := q_k(t)a .
$$ 
Let $\epsilon > 0$. Since $t \mapsto \sigma_t(a)$ is $p$-periodic there is an element
$$
g \in \Span \{q_k \otimes a: \ k \in \mathbb Z, \ a \in X\}
$$
such that $\sup_{t \in [0,p]} \left\| \sigma_t(a) - g(t)\right\| \leq \epsilon$. Since
\begin{equation}\label{24-02-23c}
\frac{1}{p} \int_0^p e^{i kt \frac{2\pi}{p}}  \ \mathrm d t= \begin{cases} 1, & \ k =0 \\ 0, & \ k \neq 0, \end{cases}
\end{equation}
we find that
$$
\frac{1}{N+1} \sum_{j=0}^N \overline{P_j}(q_l \otimes a) = \frac{a (N +1 -|l|)}{N+1}
$$
when $N \geq |l|$
and conclude that 
$$
\lim_{N \to \infty} \frac{1}{N+1} \sum_{j=0}^N \overline{P_j}(q_l \otimes a) =a = q_l \otimes a(0)
$$ 
for all $l \in \mathbb Z$ and all $a \in X$.
By linearity
\begin{align*}
& \lim_{N\to \infty} \frac{1}{N+1} \sum_{j=0}^N \overline{P_j}(g)  = g(0) ,
\end{align*}
and hence  
\begin{align*}
&\left\|\frac{1}{N+1} \sum_{j=0}^N \overline{P_j}(g) - a\right\| = \left\|\frac{1}{N+1} \sum_{j=0}^N \overline{P_j}(g) - h(0) \right\| \\
& \leq \left\|\frac{1}{N+1} \sum_{j=0}^N \overline{P_j}(g) - g(0) \right\| + \left\|h(0) -g(0)\right\| \leq 2 \epsilon
\end{align*} 
for all large $N$. Using \eqref{23-03-23d} we find now that
\begin{align*}
& \left\| \frac{1}{N+1} \sum_{j=0}^N {P_j}(a) -a \right\| =\left\| \frac{1}{N+1} \sum_{j=0}^N \overline{P_j}(h) -a \right\| \\
&  \leq \left\|\frac{1}{N+1} \sum_{j=0}^N \overline{P_j}(h-g)\right\| + \left\|\frac{1}{N+1} \sum_{j=0}^N \overline{P_j}(g) -a\right\| \\
&\leq \sup_{t \in [0,p]} \left\|\sigma_t(a) - g(t)\right\| + \left\|\frac{1}{N+1} \sum_{j=0}^N \overline{P_j}(g) -a\right\| \leq 3 \epsilon 
\end{align*} 
for all large $N$. \qed

\subsection{KMS weights for periodic flows (continued)} We return now to the setting of a periodic flow $\sigma$ on a $C^*$-algebra $A$. Assuming that $ \sigma$ is $p$-periodic, set
$$
A(k) := \left\{ a \in A: \ \sigma_t(a) = e^{i t k \frac{2\pi}{p}}a \ \forall t \in \mathbb R \right\} 
$$ 
for $k \in \mathbb Z$.

\begin{lemma}\label{23-02-23e} The following hold.
\begin{itemize}
\item[(1)] $
 A(0) = A^\sigma$.
\item[(2)] $ A(k)^* = A(-k)$.
\item[(3)] $A(k)A(n) \subseteq A(k+n)$.
\item[(4)] $A(k) \subseteq  \mathcal A_\sigma$.
\item[(5)] $\sigma_z(a) = e^{izk\frac{2\pi}{p}} a$ for all $a \in A(k)$.
\item[(6)] The map $Q_0 : A \to A(0) = A^\sigma$ is a conditional expectation.
\end{itemize}
\end{lemma}
\begin{proof} Left to the reader.
\end{proof}

\begin{lemma}\label{23-02-23g} Let $\phi$ be a $\sigma$-invariant weight on $A$. Then
$Q_k(\mathcal M_\phi) \subseteq \mathcal M^\sigma_\phi$ for all $k \in \mathbb Z$.
\end{lemma}
\begin{proof} The proof is very similar to the proof of Lemma \ref{07-12-21}. There are four continuous functions $g_j : [0,p] \to \mathbb R^+$ such that
$$
e^{-ikt \frac{2\pi}{p}} = g_1(t) - g_2(t) + i g_3(t) - i g_4(t) , \ \ \forall t \in \mathbb R.
$$ 
To show that $Q_k(a) \in \mathcal M_\phi$ it suffices to show that 
$$
\int_0^p g_j(t) \sigma_t(a) \ \mathrm d t \in \mathcal M_\phi , \ j =1,2,3,4.
$$
For this we may assume that $a \in \mathcal M_\phi^+$. Let $D$ be a separable $\sigma$-invariant
$C^*$-subalgebra of $A$ containing $\sigma_t(a)$ for all $t \in \mathbb R$. By Theorem \ref{09-11-21h} there is a sequence $\{\omega_n\}$ of positive linear functionals on $D$ such that $\phi(d) = \sum_{n=1}^\infty \omega_n(d)$ for all $d \in D^+$. We can then use Lebesgues theorem on monotone convergence to conclude that
\begin{align*}
&\phi\left( \int_0^p g_j(t) \sigma_t(a) \ \mathrm d t\right) = \sum_{n=1}^\infty \omega_n\left(  \int_0^p g_j(t) \sigma_t(a)\ \mathrm d t\right) \\
&=. \sum_{n=1}^\infty  \int_0^p g_j(t) \omega_n(\sigma_t(a))\ \mathrm d t\\
& =  \int_0^p g_j(t) \sum_{n=1}^\infty \omega_n(\sigma_t(a))\ \mathrm d t  =  \int_0^p g_j(t)  \phi(\sigma_t(a))\ \mathrm d t\\
& =  \int_0^p g_j(t)  \phi(a) \  \mathrm d t < \infty .
 \end{align*}
Hence  $Q_k(\mathcal M_\phi) \subseteq \mathcal M_\phi$ and in combination with (4) and (5) in Lemma \ref{23-02-23g} this implies that $Q_k(\mathcal M_\phi) \subseteq \mathcal M^\sigma_\phi$. 
 
\end{proof}

\begin{lemma}\label{23-02-23f}  Let $\phi$ be a $\sigma$-invariant weight.
Let $a \in \mathcal N_\phi$. Then $P_j(a) \in \mathcal N_\phi$ for all $j \in \mathbb N$ and $$
\lim_{N\to \infty}
\frac{1}{N+1} \sum_{j=0}^N \Lambda_\phi(P_j(a)) = \Lambda_\phi(a)
$$
in $H_\psi$.
\end{lemma}
\begin{proof} Since $\Lambda_\phi$ is closed by Lemma \ref{17-11-21a} an application of Lemma \ref{12-02-22b} in Appendix \ref{integration} shows that $Q_k(a) \in \mathcal N_\phi$
 and 
$$
\Lambda_\phi(Q_k(a)) = \frac{1}{p} \int_0^p e^{ikt\frac{2\pi}{p}} \Lambda_\phi(\sigma_t(a)) \ \mathrm d t .
$$
Hence  $P_j(a) = \sum_{k=-j}^j Q_k(a) \in \mathcal N_\phi$ and 
$$
 \Lambda_\phi(P_j(a)) = \sum_{k=-j}^j \Lambda_\phi(Q_k(a)) .
$$
Since $\Lambda_\phi(\sigma_t(a)) = U^\phi_t \Lambda_\psi(a)$ we can apply Theorem \ref{23-02-23a} with $H_\psi$ in the role of $X$ to conclude that $\lim_{N\to \infty}\frac{1}{N+1} \sum_{j=0}^N \Lambda_\phi(P_j(a)) = \Lambda_\phi(a)$.
\end{proof}

\begin{lemma}\label{24-02-23} Let $\psi$ be a $\beta$-KMS weight for $\sigma$. Then $\psi|_{A^\sigma}$ is a lower semi-continuous trace on $A^\sigma$ with the property that 
\begin{equation}\label{24-02-23a}
\psi|_{A^\sigma}(a^*a) = e^{\frac{2k\pi \beta}{p}} \psi|_{A^\sigma}(aa^*) 
\end{equation}
for all $a \in A(k)$ and all $k \in \mathbb Z$.
\end{lemma}
\begin{proof} To see that $\psi|_{A^\sigma}$ is not zero, let $a \in A^+$ such that $\psi(a) > 0$. By using Theorem \ref{09-11-21h} in the same way as in the proof of Lemma \ref{23-02-23g} we find that 
$$
\psi(Q_0(a)) = \frac{1}{p} \int_0^p \psi(\sigma_t(a)) \ \mathrm d t = \psi(a) > 0.
$$
Since $Q_0(a) \in (A^\sigma)^+$ it follows that $\psi|_{A^\sigma}$ is not zero. In view of Lemma \ref{22-02-23} it suffices therefore to establish \eqref{24-02-23a}. Since $A(k) \subseteq \mathcal A_\sigma$ by (4) of Lemma \ref{23-02-23e} this follows from (1) in Kustermans' theorem, Theorem \ref{24-11-21d}, and (5) of Lemma \ref{23-02-23e}:
\begin{align*} 
& \psi(a^*a) = \psi\left( \sigma_{-i\frac{\beta}{2}}(a)  \sigma_{-i\frac{\beta}{2}}(a)^*\right) =
\psi \left(e^{\frac{2\pi k \beta}{2p}} a e^{\frac{2\pi k \beta}{2p}} a^*\right) =  e^{\frac{2k\pi \beta}{p}} \psi(aa^*) .
\end{align*}
\end{proof}

\begin{lemma}\label{24-02-23d} Let $a \in A(k)$. Then $Q_0(a) = 0$ when $k \neq 0$ and $Q_0(a) = a$ when $k=0$.
\end{lemma}
\begin{proof} This is a straightforward calculation:
\begin{align*}
& Q_0(a) = \frac{1}{p} \int_0^p \sigma_t(a) \ \mathrm d t = a  \frac{1}{p} \int_0^p e^{ikt \frac{2\pi}{p}} \ \mathrm d t = \begin{cases} a, \ & k =0 \\ 0, \ & k\neq 0 . \end{cases}
\end{align*}
\end{proof}

\begin{lemma}\label{24-02-23b} Let $\tau$ be a lower semi-continuous trace on $A^\sigma$ such that\begin{equation}\label{24-02-23hx}
\tau(a^*a) = e^{\frac{2k\pi \beta}{p}} \tau(aa^*) 
\end{equation}
for all $a \in A(k)$ and all $k \in \mathbb Z$. Then $\tau \circ Q_0$ is a $\beta$-KMS weight for $\sigma$.
\end{lemma}
\begin{proof} Since $\tau$ is a non-zero lower semi-continuous weight it follows that so is $\tau \circ Q_0$. Since $Q_0 \circ \sigma_s = \sigma_s \circ Q_0 = Q_0$ we see that $\tau \circ Q_0$ is $\sigma$-invariant.
To see that $\tau \circ Q_0$ is densely defined, let $a \in A$ and $\epsilon > 0$ be given. It
 follows from Theorem \ref{23-02-23a} that there is a natural number $N$ and elements $a_k \in A(k), \ -N \leq k \leq N$, such that $\|a-b\| \leq \epsilon $, where
 $$
 b = \sum_{k=-N}^N a_k .
 $$
By Lemma \ref{16-12-21a}
there is an
approximate unit for $A^\sigma$ such that $0\leq u_j \leq 1$ and
$u_j \in \mathcal N_\tau$ for all $j$. Since $a_k^*a_k \in A^\sigma$ we find that
\begin{equation}\label{24-02-23l}
\lim_{j \to \infty} \left\|a_ku_j - a_k\right\|^2 = \lim_{j \to \infty} \left\|(u_j-1)a_k^*a_k(u_j-1)\right\| = 0.
\end{equation}
Set $c_j : = bu_j = \sum_{k=-N}^N a_ku_j$. Using Lemma \ref{24-02-23d}, and (6), (2) and (3) of Lemma \ref{23-02-23e} we find that
\begin{align*}
& Q_0(c_j^*c_j) = \sum_{k,l} Q_0(u_ja_k^* a_lu_j) =  \sum_{k,l} 
u_jQ_0(a_k^* a_l)u_j = u_j \left(\sum_{k=-N}^N a_k^*a_k\right)u_j .
\end{align*}
It follows that
$$
\tau \circ Q_0(c_j^*c_j) \leq \sum_{k=-N}^N \|a_k\|^2 \tau(u_j^2) < \infty .
$$
By \eqref{24-02-23l} $\left\|b-c_j\right\| \leq \epsilon$ and hence $\left\| a -c_j\right\| \leq 2\epsilon$ for $j$ large enough. It follows that $\tau \circ Q_0$ is densely defined.

 To see that $\tau \circ Q_0$ is a $\beta$-KMS weight we apply Theorem \ref{12-12-13} with
 $$
 S:= \Span \left\{ Q_k(a): \ k \in \mathbb Z, \ a \in \mathcal M_{\tau \circ Q_0} \right\}.
 $$
It follows from Lemma \ref{23-02-23g} that $S \subseteq \mathcal M^\sigma_{\tau \circ Q_0}$. For $a \in \mathcal M_{\tau \circ Q_0}$, set
$$
a_n := \frac{1}{n+1} \sum_{j=0}^n \sum_{k=-j}^j Q_k(a) \in S.
$$
It follows from Theorem \ref{23-02-23a} that $\lim_{n \to \infty} a_n = a$ and from Lemma \ref{23-02-23f} that $\lim_{n \to \infty} \Lambda_{\tau \circ Q_0}(a_n) = \Lambda_{\tau \circ Q_0}(a)$, showing that $S$ has the properties required of $S$ in Theorem \ref{12-12-13}. It suffices therefore to check that
\begin{equation}\label{24-02-23i}
\tau \circ Q_0(a^*a) = \tau \circ Q_0 \left(\sigma_{-i\frac{\beta}{2}}(a) \sigma_{-i\frac{\beta}{2}}(a)^*\right)
\end{equation}
when $a \in S$. By definition of $S$ an element $a \in S$ has the form
$$
a = \sum_{k=-N}^N a_k
$$
for some $N \in \mathbb N$ and some elements $a_k \in A(k)$. Then
$$
Q_0(a^*a) = \sum_{k=-N}^N a_k^*a_k
$$
and
$$
 Q_0 \left(\sigma_{-i\frac{\beta}{2}}(a) \sigma_{-i\frac{\beta}{2}}(a)^*\right) = \sum_{k=-N}^N e^{\frac{2k\pi \beta}{p}} a_ka_k^*  .
 $$
Hence \eqref{24-02-23i} follows from \eqref{24-02-23hx}.

\end{proof}

\begin{lemma}\label{24-02-23k} Let $\psi$ be a $\beta$-KMS weight for $\sigma$. Then $\psi = \psi \circ Q_0$.
\end{lemma}
\begin{proof} To conclude this from Lemma \ref{13-12-21axx} we must show that $A^\sigma$ contains an approximate unit for $A$ and that $Q_0(a) = \lim_{R \to \infty} \frac{1}{R}\int_0^R  \sigma_t(a) \ \mathrm d t$ for all $a \in A$. The last equality follows from Lemma \ref{23-02-23}, and since
$$
\Span \cup_{k \in \mathbb Z} A(k)
$$ is dense in $A$ by Theorem \ref{23-02-23a}
the first fact follows from \eqref{24-02-23l} and (2) of Lemma \ref{23-02-23e}.
\end{proof}

\begin{thm}\label{24-02-23j} Let $\sigma$ be a $p$-periodic flow on the $C^*$-algebra $A$. Let $\beta \in \mathbb R$. The map $\psi \mapsto \psi|_{A^\sigma}$ is a bijection from the set of $\beta$-KMS weights for $\sigma$ onto the set of lower semi-continuous traces $\tau$
on $A^\sigma$ with the property that
\begin{equation*}
\tau(a^*a) = e^{\frac{2k\pi \beta}{p}} \tau(aa^*) 
\end{equation*}
for all $a \in A(k)$ and all $k \in \mathbb Z$. The inverse map is given by $\tau \mapsto \tau \circ Q_0$ where $Q_0 : A \to A^\sigma$ is the canonical conditional expectation.
\end{thm}
\begin{proof} The map is defined by Lemma \ref{24-02-23} and it is injective by Lemma \ref{24-02-23k}. The map
is surjective by Lemma \ref{24-02-23b}; a lemma which also gives the formula for its inverse.
\end{proof}

Specializing to the unital case we obtain the following

\begin{cor}\label{24-02-23m} Let $\sigma$ be a $p$-periodic flow on the unital
$C^*$-algebra $A$. Let $\beta \in \mathbb R$. The map $\psi \mapsto \psi|_{A^\sigma}$ is a bijection from the set of $\beta$-KMS states $\psi$ for $\sigma$ onto the set of trace states
$\tau$
on $A^\sigma$ with the property that
\begin{equation*}\label{24-02-23h}
\tau(a^*a) = e^{\frac{2k\pi \beta}{p}} \tau(aa^*) 
\end{equation*}
for all $a \in A(k)$ and all $k \in \mathbb Z$. The inverse map is given by $\tau \mapsto \tau \circ Q_0$ where $Q_0 : A \to A^\sigma$ is the canonical conditional expectation.
\end{cor}

\begin{notes}\label{24-02-23n} Theorem \ref{24-02-23j} is new, but Corollary \ref{24-02-23m} can be deduced from
Proposition 8 in \cite{CP2}, and follows also from the paper \cite{PWY} by Pinzari, Watatani and Yonetani under an additional assumption (fullness). In addition, \cite{PWY} contains a very interesting existence result in this setting, cf. Theorem 2.5 in \cite{PWY}. As Lemma \ref{22-02-23} indicates there is very often a map from the set of $\beta$-KMS weights to the densely defined lower semi-continuous traces on the fixed point algebra of the flow. In general, for flows that are not periodic this map is neither injective nor surjective.

\end{notes}

\begin{example}\label{27-02-23} \rm Let $l^2$ denote the Hilbert space of square summable sequences $\alpha =\left\{\alpha_i\right\}_{k=0}^\infty$ of complex numbers. Define isometries $V_1$ and $V_2$ on $l^2$ by
$$
(V_1\alpha)_{2k+1} = \alpha_k , \ \ (V_1\alpha)_{2k} =0 , 
$$
and
$$
(V_2\alpha)_{2k+1} = 0, \ \ (V_2\alpha)_{2k} = \alpha_k,
$$
for $k = 0,1,2,3,\cdots$. Then 
\begin{equation}\label{27-02-23b}
V_k^*V_k= 1, \ k = 1,2, \ \ V_k^*V_j = 0, \ k \neq j,
\end{equation}
and 
$$
V_1V_1^* + V_2V_2^* = 1.
$$ 
Let $A$ be the $C^*$-subalgebra of $B(l^2)$ generated by $V_1$ and $V_2$. When $a = (a_1,a_2,\cdots, a_n) \in \{1,2\}^n$, set
$$
V_a := V_{a_1}V_{a_2} \cdots V_{a_n} \in A .
$$
Set $\{0,1\}^0 := \emptyset$ and $V_{\emptyset} = 1$.
\begin{obs}\label{27-02-23a} The elements $V_aV_b^*, \ a,b \in \bigcup_{n=0}^\infty \{1,2\}^n$ span a dense $*$-subalgebra $B$ of $A$.
\end{obs}
\begin{proof} It follows easily from \eqref{27-02-23b} that the elements $V_aV_b^*$ span a $*$-algebra $B$ and since $B$ contains $V_1$ and $V_2$ the closure of $B$ must be $A$. 
\end{proof}

To define a flow on $A$, let $x_j, j = 0,1,2,3,\cdots$, be the sequence of natural numbers defined recursively such that $x_0=0$,
$$
x_{2k+1} = x_k +1, \ k \geq 0,
$$
and
$$
x_{2k} = x_k+1, \ k \geq 1.
$$
Let $\rho$ be a real number and define unitaries $U_t, \ t \in \mathbb R$, on $l^2$ such that
$$
(U_t\alpha)_k = e^{i\rho x_k t}\alpha_k, \ k \in \mathbb N \cup \{0\}.
$$ 
Then 
\begin{equation}\label{naestved}
U_tV_k U_t^* = e^{i\rho t}V_k, \ \ k =1,2,
\end{equation}
and therefore $U_tAU_t^* = A$. Let $\sigma^\rho_t := \Ad U_t$. It follows from Observation \ref{27-02-23a} and \eqref{naestved} that $\sigma^\rho$ is a flow on $A$. When $\rho =0$ the flow is trivial and hence the KMS weights for $\sigma^0$ are traces; necessarily finite traces since $A$ is unital, cf. Lemma \ref{24-09-23b}. If $\tau$ is a finite trace on $A$ we have that $\tau(V_iV_i^*) = \tau(V_i^*V_i) = \tau(1), \ i =1,2$, and $\tau(1) = \tau(V_1V_1^*) + \tau(V_2V_2^*) = 2 \tau(1)$, implying that $\tau(1) = 0$ and hence that $\tau = 0$. This shows that there are no KMS weights for $\sigma^\rho$ when $\rho = 0$. When $\rho \neq 0$ we find that 
$$
\sigma^{2\pi}_z = \sigma^\rho_{\frac{2\pi}{\rho}z} \ \ \ \ \forall z \in \mathbb C,
$$
and hence a weight on $A$ is at $\beta$-KMS weight for $\sigma^\rho$ if and only if it is a $\frac{\rho \beta}{2\pi}$-KMS weight for $\sigma^{2\pi}$. To find the $\beta$-KMS weights for $\sigma^\rho$ we may therefore as well seek the $\frac{\rho \beta}{2\pi}$-KMS weight for $\sigma^{2\pi}$. To simplify the notation, set $\sigma:= \sigma^{2\pi}$ and note that $\sigma$ is $1$-periodic. Since $A$ is unital all KMS weights are bounded, and each of them is a scalar multiple of a KMS state, cf. Section \ref{KMSstatesx}. We are therefore seeking the KMS states for $\sigma$.

When $a \in \{1,2\}^n, \ b \in \{0,1\}^m$, we find that
\begin{align*}
&Q_k(V_aV_b^*) = \int_0^1 e^{-i 2\pi t k}\sigma_t(V_aV_b^*) \ \mathrm d t \\
&= \left(\int_0^1 e^{-i2\pi t (k +m -n)}   \mathrm d t \right) V_aV_b^* = \begin{cases} V_aV_b^* , \ & k = n-m \\ 0, \ & k \neq m-n .\end{cases} 
\end{align*}
The fixed point algebra $A^\sigma = A(0)$ is equal to $Q_0(A)$ and it follows therefore that it is the closure of 
$$
\Span \left\{V_aV_b^* : \ a,b \in \{1,2\}^n, \ n =0,1,2, \cdots \right\} .
$$
Similarly, $A(k)$ is the closure of
$$
\Span \left\{ V_aV_b^* : a \in \{0,1\}^{n+k}, \ b \in \{0,1\}^n, \ n =0,1,2, \cdots \right\}
$$
when $k \geq 1$ and the closure of
$$
\Span \left\{ V_aV_b^* : a \in \{0,1\}^{n}, \ b \in \{0,1\}^{n+|k|}, \ n =0,1,2, \cdots \right\}
$$
when $k \leq -1$. When $a,b,a',b' \in \{1,2\}^n$ it follows from \eqref{27-02-23b} that
$$
V_aV_b^*V_{a'}V_{b'}^* = \begin{cases} 0, \ & b \neq a' \\ V_aV_{b'}^* , \ &  b = a' , \end{cases}
$$
showing that $V_aV_b^*, \ a,b \in \{1,2\}^n$, constitute a set of matrix units spanning a copy $A_n^\sigma$ of $M_{2^n}(\mathbb C)$. Note that there is a unit preserving inclusion $A^\sigma_n \subseteq A^\sigma_{n+1}$ since $V_aV_b^* = V_aV_1(V_bV_1)^* +  V_aV_2(V_bV_2)^*$. It follows that 
$$
A^\sigma = \overline{\bigcup_n A_n^\sigma} ,
$$
where $A_1^\sigma \subseteq A_2^\sigma \subseteq A_3^\sigma \subseteq \cdots$ is a sequence of unital subalgebras with $A_n^\sigma \simeq M_{2^n}(\mathbb C)$ and we conclude in this way that $A^\sigma$ is a UHF algebra of type $2^\infty$, see Section 2.2, and in particular pages 78-83 in \cite{Th3}. Thus $A^\sigma$ has a unique trace state $\tau$, cf. Example 3.3.26 in \cite{Th3}. It follows then from Corollary \ref{24-02-23m} that there is at most one KMS state for $\sigma$, namely the state $\tau \circ Q_0$. Note that when $k \geq 1$ and $ a,a' \in \{0,1\}^{n+k}, \ b,b' \in \{0,1\}^n$, we have that
$$
\tau(V_aV_b^*(V_{a'}V_{b'}^*)^*)  = \tau(V_aV_a^*) = \begin{cases}  2^{-n-k} , & \ b=b', \ a= a' \\ 0, \ & \text{otherwise} ,\end{cases}
$$
while
$$
\tau((V_{a'}V_{b'}^*)^*V_aV_b^*) = \tau(V_bV_b^*) =  \begin{cases}  2^{-n} , & \ b=b', \ a= a' \\ 0, \ & \text{otherwise} .\end{cases}
$$
It follows that $\tau(x^*x) = 2^k \tau(xx^*)$ when $x \in A(k), \ k \geq 1$. Similar considerations show that this is also true when $k \leq -1$, and we conclude therefore from Corollary \ref{24-02-23m} that $\tau \circ Q_0$ is a $\beta$-KMS state for $\sigma$ if and only if
$$
\beta = \frac{\log 2}{2 \pi} .
$$
We learn in this way that for a general $\rho \in \mathbb R\backslash \{0\}$ there is a $\beta$-KMS state for $\sigma^\rho$ if and only $\beta = \frac{\log 2}{\rho}$, and it is then unique and equal to $\tau \circ Q_0$. 
\end{example}


\begin{notes}\label{28-02-23} The $C^*$-algebra $A$ in Example \ref{27-02-23} is known as the Cuntz algebra $O_2$ and the identification of the KMS states for the flows in the example was obtained by Olesen and Pedersen in \cite{OP}.
\end{notes}



\chapter{The bundle of KMS states}

In this chapter we describe the general structure of the collection of KMS states for flows on a unital separable $C^*$-algebra.

\section{The Choquet simplex of a $\beta$-KMS state}\label{KMSstates}

This section draws heavily on convexity theory and K-theory. Not all the results we need can be found in standard textbooks on $C^*$-algebras, functional analysis or $K$-theory and we shall therefore often refer directly to the original sources. However, a self-contained presentation of the needed material can be found in \cite{Th5} and in cases where the required facts are hard to dig out from other sources, we refer only to \cite{Th5}.

Let $A$ be a unital $C^*$-algebra and $\sigma$ a flow on $A$. For each $\beta \in \mathbb R$ we denote by $S^\sigma_\beta$ the set of $\beta$-KMS states for $\sigma$.

\begin{thm}\label{08-11-22} $S^\sigma_\beta$ is a Choquet simplex for each $\beta \in \mathbb R$.
\end{thm}

We refer to Section 4.1 in \cite{BR} or Chapter 3 in \cite{Th5} for the basics of the theory of compact convex sets and Choquet simplexes.

 Let $K$ be a weak* closed convex set of states of $A$ and let $A_{sa}$ denote the real vector space consisting of the self-adjoint elements of $A$. As in \cite{BR} we denote by $\Aff K$ the real vector space of continuous real-valued affine functions on $K$. We define $R : A_{sa} \to \Aff K$ such that 
 $$
 R(a)(\omega) := \omega(a), \ \ \omega \in K.
 $$

\begin{lemma}\label{11-11-22} $R(A_{sa})$ is dense in $\Aff K$.
\end{lemma}
\begin{proof} Let $l \in (\Aff K)^*$ such that $l \circ R = 0$. We must show that $l =0$. Let $C_\mathbb R(K)$ denote the real Banach space of continuous real-valued functions on $K$. By the Hahn-Banach extension theorem there is a functional $\overline{l} \in C_{\mathbb R}(K)^*$ extending $l$. By considering $\overline{l}$ as a self-adjoint linear functional on the abelian $C^*$-algebra $C(K)$ we can apply the Hahn-Jordan decomposition theorem, Theorem 4.3.6 of \cite{KR} or Theorem 3.3.6 in \cite{Th5},
 to $\overline{l}$ and combine with the Riesz representation theorem, Theorem 2.14 in \cite{Ru0}, to get non-negative real numbers $s_\pm$ and Borel probability measues $\mu_\pm$ such that
$$
\overline{l}(f) = s_+ \int_K f \ \mathrm d \mu_+ - s_-\int_K f \ \mathrm d \mu_- , \ \ \forall f \in C_\mathbb R(K) .
$$
Then
$$
l(g) = s_+ g(b(\mu_+)) - s_-g(b(\mu_-)), \ \ \forall g \in \Aff K,
$$
where $b(\mu_\pm) \in K$ are the barycenters of $\mu_\pm$, cf. Proposition 4.1.1 on page 317 of \cite{BR} or  Lemma 3.1.11 in \cite{Th5}. Since $l \circ R =0$ by assumption it follows that 
\begin{align*}
&s_+ b(\mu_+)(a) -s_-b(\mu_-)(a) = s_+ R(a)(b(\mu_+)) -s_-R(a)(b(\mu_-)) \\
& =l(R(a)) = 0
\end{align*}
for all $a \in A_{sa}$, implying that
 $s_+ b(\mu_+) - s_-b(\mu_-) = 0$ in $(A_{sa})^*$. Evaluation at $1 \in A$ shows that $s_+ = s_-$. It follows that $b(\mu_+) = b(\mu_-)$ if $s_+ = s_- \neq 0$ and hence that $l =0$, as required.
\end{proof}

\begin{lemma}\label{11-11-22a} There is a linear map $L : \Span_{\mathbb R} K  \to  (\Aff K)^*$ such that
$$
L(sx-ty)(f) = sf(x) -t f(y)
$$
for all $s,t \in \mathbb R, \ x,y \in K, \ f \in \Aff K$. 
\end{lemma}
\begin{proof} It suffices to show that when $s,t, s',t' \in \mathbb R, \ x,y,x',y' \in K$ and $sx-ty = s'x' - t'y'$ in $A^*$, then $sf(x) -t f(y) =  s'f(x') -t' f(y')$ for all $f \in \Aff K$, and this follows from Lemma \ref{11-11-22}.
\end{proof}

\begin{lemma}\label{11-11-22b} The map $L : \Span_{\mathbb R} K  \to  (\Aff K)^*$ of Lemma \ref{11-11-22a} is a linear bijection such that $L(\mathbb R^+ K\cup \{0\}) = (\Aff K)^*_+$ where $\mathbb R^+$ denotes the non-negative real numbers and $(\Aff K)^*_+$ the positive linear functionals on $\Aff K$.
\end{lemma}
\begin{proof} $L$ is injective because $L(\omega)(R(a)) = \omega(a)$ for all $a \in A_{sa}$. To see that $L$ is surjective, let $l \in (\Aff K)^*$. By repeating arguments from the proof of Lemma \ref{11-11-22} we obtain $s_\pm \in \mathbb R_+$ and $\omega_\pm \in K$ such that $l(g) = s_+g(\omega_+) - s_-g(\omega_-)$ for all $g \in \Aff K$. Hence $l =L(s_+\omega_+ - s_-\omega_-)$, showing that $L$ is surjective. 

It follows from Theorem 3.1.2 in \cite{Th5} that every state of the order unit space $\Aff K$ is given by evaluation at a point in $K$; a fact which implies that $(\Aff K)^*_+ \subseteq L(\mathbb R^+ K)$. The reverse inclusion is trivial.

\end{proof}

\begin{lemma}\label{11-11-22c} Let $A$ be a unital $C^*$-algebra and $K$ a weak* closed convex set of states of $A$. Then $K$ is a Choquet simplex if and only if $\mathbb R^+K$ is a lattice in $A^*$, in the sense that when $\omega_1,\omega_2 \in \mathbb R^+K $ there is an element $\omega_1 \vee \omega_2 \in \mathbb R^+K$ such that $\omega_i \leq \omega_1 \vee \omega_2, \ i =1,2$, with respect to the order from $A^*$ and $\omega_1 \vee \omega_2$ is the least element of $\mathbb R^+K$, also with respect to the order from $A^*$, with these properties. \footnote{It \emph{does not} mean that there can not be other positive linear functionals between $\omega_1, \omega_2$ and $\omega_1 \vee \omega_2$; they just can't be scalar multiples of $K$. }
\end{lemma}
\begin{proof} This follows from Lemma \ref{11-11-22b} and Theorem 3.2.7 in \cite{Th5} when we show that $\Aff K^*$ is a lattice if and only if $\mathbb R^+K$ is a lattice in the specified sense. To see this assume first that $\Aff K^*$ is a lattice as defined in Definition 3.2.1 of \cite{Th5} and consider two elements $\omega_1,\omega_2 \in \mathbb R^+K$. There is then an element $\mu \in \Aff K^*$ which is the least upper bound of $L(\omega_1)$ and $L(\omega_2)$ in $\Aff K^*$, where $L$ is the bijection from Lemma \ref{11-11-22b}. Then $\mu \in (\Aff K)^*_+$ since $L(\omega_1) \in (\Aff K)^*_+$ and $L(\omega_1) \leq \mu$ in $\Aff K^*$. It follows from Lemma \ref{11-11-22b} that $\mu = L(\omega_1\vee \omega_2)$ for some $\omega_1\vee \omega_2 \in \mathbb R^+K$. Assume $\omega \in  \mathbb R^+K$ and that $\omega_i \leq \omega,\ i = 1,2$ with respect to the order from $A^*$. It follows then from Lemma \ref{11-11-22} that $L(\omega_i) \leq L(\omega), \ i =1,2$, in $\Aff K^*$ and hence that $L(\omega) \leq \mu$. By Lemma \ref{11-11-22b} this implies that $\omega \leq \omega_1 \vee \omega_2$. For the converse, assume that $\mathbb R^+K$ is a lattice as specified in the statement of the lemma. Let $\mu_1,\mu_2\in \Aff K^*$. Since $\Aff K^* = (\Aff K)_+^* - (\Aff K)_+^*$ by Lemma 3.1.8 in \cite{Th5} there is an element $\mu \in (\Aff K)^*_+$ such that $\mu_i + \mu \in (\Aff K)_+^*, \ i =1,2$. By Lemma \ref{11-11-22c} there are elements $\omega_i \in \mathbb R^+ K$, such that $L(\omega_i) = \mu_i + \mu, \ i =1,2$. Let $\omega_1 \vee \omega_2 \in \mathbb R^+K$ least upper bound in $\mathbb R^+K$ with respect to the order from $A^*$ and set $\mu_1 \vee \mu_2 := L(\omega_1 \vee \omega_2) - \mu$. It is straightforward to check that $\mu_1 \vee \mu_2$ is the least upper bound for $\mu_1$ and $\mu_2$ in $\Aff K^*$.
\end{proof}

\emph{Proof of Theorem \ref{08-11-22}:} Observe first that a convex combination of states that satisfy the conditions in Theorem \ref{21-11-23b} will also satisfy them. Thus $S^\sigma_\beta$ is a convex set of states of $A$. Observe next
that $S_\beta^\sigma$ is closed in the weak* topology of $A^*$. Indeed, if $\{\omega_\alpha\}$ is a net in $S^\sigma_\beta$ which converges in the weak* topology of $A^*$ then $\omega := \lim_\alpha \omega_\alpha$ is a state because the state space is closed in the weak* topology. When $a \in \mathcal A_\sigma$ Theorem \ref{21-11-23b} implies that
$$
\omega(a^*a) = \lim_\alpha \omega_\alpha(a^*a) = \lim_\alpha \omega_\alpha(\sigma_{-i\frac{\beta}{2}}(a) \sigma_{-i\frac{\beta}{2}}(a)^*) = \omega(\sigma_{-i\frac{\beta}{2}}(a) \sigma_{-i\frac{\beta}{2}}(a)^*) ,
$$
and then also that $\omega$ is a $\beta$-KMS state for $\sigma$. Thus $S^\sigma_\beta$ is a compact convex set of states. It follows therefore from Theorem \ref{14-02-22d} and Lemma \ref{11-11-22c} that $S^\sigma_\beta$ is a Choquet simplex. \qed

\section{The KMS bundle of a flow}

  Let $S$ be a second countable locally compact Hausdorff space and $\pi : S \to \mathbb R$ a continuous map. If the inverse image $\pi^{-1}(t)$, equipped with the relative topology inherited from $S$, is homeomorphic to a compact metrizable Choquet simplex for all $t \in \mathbb R$ we say that $(S,\pi)$ is a \emph{simplex bundle}. We emphasize that $\pi$ need not be surjective, and we consider therefore also the empty set as a Choquet simplex. When $(S,\pi)$ is a simplex bundle we denote by $\mathcal A(S,\pi)$ the set of continuous functions $f : S \to \mathbb R$ with the property that the restriction $f|_{\pi^{-1}(t)}$ of $f$ to $\pi^{-1}(t)$ is affine for all $t \in \mathbb R$.

\begin{defn}\label{25-08-21} A simplex bundle $(S,\pi)$ is a \emph{proper simplex bundle} when
\begin{itemize}
\item[(1)] $\pi$ is proper, i.e. $\pi^{-1}(K)$ is compact in $S$ when $K \subseteq \mathbb R$ is compact, and
\item[(2)] $\mathcal A(S,\pi)$ separates points on $S$; i.e. for all $x\neq y$ in $S$ there is an $f \in\mathcal A(S,\pi)$ such that $f(x) \neq f(y)$.
\end{itemize}
\end{defn}

Two proper simplex bundles $(S,\pi)$ and $(S',\pi')$ are \emph{isomorphic} when there is a homeomorphism $\phi : S \to S'$ such that $\pi' \circ \phi = \pi$ and $\phi: \pi^{-1}(\beta) \to {\pi'}^{-1}(\beta)$ is affine for all $\beta \in \mathbb R$.

Consider now a flow $\sigma$ on a unital $C^*$-algebra $A$. Let $E(A)$ denote the state space of $A$ which we consider as a topologocal space in the weak* topology. Set 
$$
S^\sigma := \left\{(\omega,\beta) \in E(A)\times \mathbb R: \ \omega \in S^\sigma_\beta \right\} .
$$
We consider $S^\sigma$ as a topological space in the topology inherited from the product topology of $E(A) \times \mathbb R$. The projection $\pi^\sigma : S^\sigma \to \mathbb R$ to the second coordinate is then continuous and by Theorem \ref{08-11-22} the fibre $(\pi^\sigma)^{-1}(\beta)$ over every $ \beta \in \mathbb R$ is a Choquet simplex. We call $(S^\sigma,\pi^\sigma)$ the \emph{KMS bundle of $\sigma$}.

\begin{thm}\label{12-11-22x} Let $\sigma$ be a flow on a unital separable $C^*$-algebra. Then $(S^\sigma,\pi^\sigma)$ is a proper simplex bundle.
\end{thm}
\begin{proof} Since we assume that $A$ is separable $E(A)$ is metrizable and the topology of $S^\sigma$ is therefore second countable. Let $K \subseteq \mathbb R$ be compact. We will show that $(\pi^\sigma)^{-1}(K)$ is compact in $S^\sigma$ which will be enough to conclude both that $\pi^\sigma$ is proper and that $S^\sigma$ is locally compact. Let $s_n := (\omega_n, t_n)$, be a sequence in ${\pi^\sigma}^{-1}(K)$. We show that $\{s_n\}_{n \in \mathbb N}$ has a convergent subsequence. Since $E(A)$ and $K$ are compact we may assume that $\{\omega_n\}_{n \in \mathbb N}$ and $\{t_n\}_{n \in \mathbb N}$ converge in $E(A)$ and $K$, respectively. Set $t := \lim_{n\to \infty} t_n$ and $\omega := \lim_{n \to \infty} \omega_n$. Let $a \in \mathcal A_\sigma$ be an analytic element for $\sigma$. Then $\lim_n \sigma_{-i\frac{ t_n}{2}}(a) = \sigma_{-i\frac{t}{2}}(a)$ and by using Theorem \ref{21-11-23b}, we find that
$$
\omega(a^*a) = \lim_n \omega_n(a^*a) = \lim_n \omega_n(\sigma_{-i\frac{t_n}{2}}(a) \sigma_{-i\frac{t_n}{2}}(a)^*) = \omega(\sigma_{-i\frac{t}{2}}(a) \sigma_{-i\frac{t}{2}}(a)^*) .
$$
It follows, again by Theorem\ref{21-11-23b}, that $\omega$ is a $t$-KMS state for $\sigma$, showing that $(\omega,t) \in S^\sigma$. Since $t \in K$ we have shown that $(\pi^\sigma)^{-1}(K)$ is compact in $S^\sigma$, and we conclude therefore that $S^\sigma$ is locally compact and $\pi^\sigma$ proper. What remains is now only to show that $\mathcal A(S^\sigma,\pi^\sigma)$ separates the points of $S^\sigma$. Let therefore $(\omega_1,t_1)$ and $(\omega_2,t_2)$ be distinct elements of $S^\sigma$. If $t_1 \neq t_2$, choose a continuous function $f$ on $\mathbb R$ such that $f(t_1) \neq f(t_2)$ and define $F : S^\sigma \to \mathbb R$ by $F(\omega,t) := f(t)$. Then $F \in \mathcal A(S^\sigma,\pi^\sigma)$ and $F(\omega_1,t_1) \neq F(\omega_2,t_2)$. If $t_1 = t_2$ the states $\omega_1$ and $\omega_2$ are different and there is therefore an element $a \in A_{sa}$ such that $\omega_1(a) \neq \omega_2(a)$. Define $\hat{a} \in \mathcal A(S^\sigma,\pi^\sigma)$ such that $\hat{a}(\omega,t) := \omega(a)$. Then $\hat{a}(\omega_1,t_1) \neq \hat{a}(\omega_2,t_2)$.
\end{proof}

By Theorem \ref{12-11-22x} every flow on a unital separable $C^*$-algebra gives rise, in a canonical way, to a proper simplex bundle. We will show below that any proper simplex bundle is isomorphic to the KMS bundle of a flow on a simple unital separable $C^*$-algebra, but first we give an alternative description of the KMS bundle for a large class of flows.

\subsection{The KMS bundle for a non-trivial flow on a simple $C^*$-algebra}
Let $\sigma$ be a flow on the $C^*$-algebra $A$. We say that $A$ is \emph{$\sigma$-simple} when the only $\sigma$-invariant closed two-sided ideals in $A$ are $\{0\}$ and $A$, and that $\sigma$ is \emph{non-trivial} when $\sigma$ is not the trivial flow; the one for which $\sigma_t = \id_A$ for all $t \in \mathbb R$.

\begin{lemma}\label{05-01-23}  Assume that $A$ is unital and $\sigma$-simple. Assume also that $\sigma$ is not the trivial flow. Then $S^\sigma_\beta \cap S^\sigma_{\beta'} =\emptyset$ when $\beta \neq \beta'$.
\end{lemma}

\begin{proof} This follows from Corollary \ref{31-12-22}. 
\end{proof}

 Lemma \ref{05-01-23} implies that under the given assumptions a KMS weight for $\sigma$ remembers the value $\beta\in \mathbb R$ for which it is a $\beta$-KMS weight and this can be used to give an alternative description of the KMS bundle of $\sigma$. 

We assume now that $A$ is unital, separable and $\sigma$-simple, and that $\sigma$ is non-trivial. The map 
\begin{equation}\label{31-12-22a}
S^\sigma \ni (\omega,\beta) \mapsto \omega \in \bigcup_{\beta \in \mathbb R} S^\sigma_\beta
\end{equation}
is then injective by Lemma \ref{05-01-23} and hence a bijection. Thus $\bigcup_{\beta \in \mathbb R} S^\sigma_\beta$ has a topology in which it is a second countable locally compact Hausdorff space since $S^\sigma$ does. To identify this topology, let 
\begin{equation}\label{01-01-23a}
\Phi: \
\bigcup_{\beta \in \mathbb R} S^\sigma_\beta \to \mathbb R
\end{equation}
be the function defined such that $\Phi(\omega) = \beta$ when $\omega \in S^\sigma_\beta$. This is well-defined by Lemma \ref{05-01-23}. 

\begin{lemma}\label{01-01-23} Assume that $A$ is unital, separable and $\sigma$-simple, and that $\sigma$ is non-trivial. Then the map $\Phi$ of \eqref{01-01-23a} is continuous with respect to the weak* topology.
\end{lemma}
\begin{proof} Let $\{\omega_n\}_{n =1}^\infty$ and $\omega$ be elements of $\bigcup_{\beta \in \mathbb R} S^\sigma_\beta $ and assume that 
$$
\lim_{n\to \infty} \omega_n = \omega
$$ 
in the weak* topology. Set $\beta := \Phi(\omega)$ and $\beta_n := \Phi(\omega_n)$. Assume $\sup_n|\beta_n| < \infty$. If $\{\beta_n\}$ does not convergence to $\beta$ there is a $\beta' \neq \beta$ and a subsequence $\{\beta_{n_k}\}$ such that $\lim_{k \to \infty} \beta_{n_k} = \beta'$. Since $\lim_{k \to \infty} \omega_{n_k} = \omega$ it follows that
\begin{align*}
&\omega(a\sigma_{i \beta'}(b)) = \lim_{k \to \infty} \omega_{n_k}(a\sigma_{i \beta_{n_k}}(b))  = \lim_{k \to \infty} \omega_{n_k}(ba) = \omega(ba)\\
\end{align*}
for all $a,b \in \mathcal A_\sigma$, showing that $\omega$ is also a $\beta'$-KMS state. This is not possible under the present assumptions by Lemma \ref{05-01-23}. Thus $\lim_{n \to \infty} \beta_n = \beta$. It suffices therefore now to show that $\{\beta_n\}$ can not be unbounded.

Assume for a contradiction that $\beta_n$ can be arbitrarily large; the case when $-\beta_n$ can be arbitrarily large can be handled by the same argument. 
Let $a,b \in \mathcal A_\sigma$. We claim that 
\begin{equation}\label{01-01-23c}
\left|\omega(a\sigma_z(b))\right| \leq \|a\|\|b\| 
\end{equation}
when $\Imag z \geq 0$. For this note that $\lim_{n \to \infty} \omega_n(a\sigma_z(b)) = \omega(a\sigma_z(b))$. Therefore, to establish \eqref{01-01-23c} under the present assumptions it suffices to show that
\begin{equation}\label{01-01-23d}
\left| \omega_n(a\sigma_z(b))\right| \leq \|a\|\|b\| 
\end{equation}
when $0 \leq \Imag z \leq \beta_n$. Note that $z \mapsto \omega_n(a\sigma_z(b))$ is entire holomorphic and bounded on the strip $0 \leq \Imag z \leq \beta_n$ by Lemma \ref{24-09-23x}. Since $\omega_n$ is a $\beta_n$-KMS state for $\sigma$ we  have that $\left|\omega_n(a\sigma_{t+i\beta_n}(b))\right| = \left|\omega_n(\sigma_t(b)a)\right| \leq \|a\|\|b\|$ for all $t \in \mathbb R$. Since clearly also $\left|\omega_n(a\sigma_t(b))\right| \leq \|a\|\|b\|$ we obtain \eqref{01-01-23d} from Proposition 5.3.5 in \cite{BR} (Phragmen-Lindel\"of). This establishes \eqref{01-01-23c}. Thus the function $z \mapsto \omega(a\sigma_z(b))$ is bounded for $\Imag z \geq 0$ when $a,b \in \mathcal A_\sigma$. Since $\omega$ is a $\beta$-KMS state we have that 
$$
\omega(\sigma_{-i\beta}(a)a^* \sigma_z(b)) = \omega(a^*\sigma_z(b)a)
$$
for all $z \in \mathbb C$ by Theorem \ref{21-11-23b}. In particular,
$$
\omega(\sigma_{-i\beta}(a)a^* \sigma_t(b)) = \omega(a^*\sigma_t(b)a) \in \mathbb R
$$
for $t \in \mathbb R$ when $b=b^* \in \mathcal A_\sigma$. It follows therefore from the Schwarz reflection principle, cf. Theorem 11.14 in \cite{Ru0}, that there is an entire function $F$ such that $F(z) = \omega(\sigma_{-i\beta}(a)a^* \sigma_z(b))$ when $\Imag z \geq 0$ and $F(z) = \overline{F(\overline{z})}$ when $\Imag z \leq 0$. This function is bounded on $\mathbb C$ since $\omega(\sigma_{-i\beta}(a)a^* \sigma_z(b))$ is bounded for $\Imag z \geq 0$, and hence constant by Liouville's theorem. Since $z \mapsto \omega(\sigma_{-i\beta}(a)a^* \sigma_z(b))$ is also entire and agrees with $F$ when $\Imag z \geq 0$, it must be equal to $F$ and hence constant. Thus
$$
\omega(\sigma_{-i\beta}(a)a^* \sigma_z(b)) =\omega(\sigma_{-i\beta}(a)a^* b)
$$
for all $z \in \mathbb C$ and $a,b \in \mathcal A_\sigma$, first when $b=b^*$ and then for all $b \in \mathcal A_\sigma$ since $\mathcal A_\sigma$ is a $*$-algebra. Using the polarisation identity
$$
\sigma_{-i\beta}(x)y^* = \frac{1}{4} \sum_{k=1}^4 i^k \sigma_{-i \beta}(x+i^ky) (x+i^ky)^*
$$
we find that
$$
\omega(\sigma_{-i\beta}(a)b^* \sigma_z(c)) =\omega(\sigma_{-i\beta}(a)b^* c)
$$
when $a,b,c \in \mathcal A_\sigma$ and $z \in \mathbb C$. Taking $a =1$ we find that $\omega(b\sigma_z(c)) = \omega(bc)$ for all $b,c \in \mathcal A_\sigma$ and all $z \in \mathbb C$. Since $\omega$ is a $\beta$-KMS state we have that $\omega(cb) = \omega( b \sigma_{i\beta}(c)) = \sigma(bc)$ for all $b,c \in \mathcal A_\sigma$, and it follows that $\omega$ is a $t$-KMS state for all $t \in \mathbb R$. Under the present assumptions this is impossible by Lemma \ref{05-01-23}.
\end{proof}

\begin{prop}\label{01-01-23e} Assume that $A$ is unital, separable and $\sigma$-simple, and that $\sigma$ is non-trivial. Then
$$
\bigcup_{\beta \in \mathbb R} S^\sigma_\beta 
$$
is a second countable locally compact Hausdorff space in the weak* topology and $$
\left(\bigcup_{\beta \in \mathbb R} S^\sigma_\beta ,\Phi\right)
$$ 
is a proper simplex bundle isomorphic to the KMS bundle of $\sigma$.
\end{prop}
\begin{proof} Thanks to Lemma \ref{01-01-23} we can define a continuous map 
$$
\Psi : \bigcup_{\beta \in \mathbb R} S^\sigma_\beta  \to S^\sigma
$$
by $\Psi(\omega) = (\omega, \Phi(\omega))$. This is clearly the inverse of the map \eqref{31-12-22a}.
\end{proof}

\section{All proper simplex bundles occur}\label{alloccur}

Let $A$ be a $C^*$-algebra and $\alpha$ an automorphism of $A$. On the crossed product $B := A \rtimes_\alpha \mathbb Z$ there is a flow $\hat{\alpha}$ defined such that
\begin{equation}\label{07-10-23}
\hat{\alpha}_t( \sum_{k \in \mathbb Z} a(k)u^k = \sum_{k \in \mathbb Z} e^{ikt} a(k) u^k ,
\end{equation}
when $a : \mathbb Z \to A$ is finitely supported and $u$ is the canonical unitary $u \in M(B)$ such that $uau^* = \alpha(a)$ for all $a \in A$. We shall refer to $\hat{\alpha}$ as the \emph{dual flow of $\alpha$}. Note that the dual flow is a special case of the flows considered in Section \ref{crosseddiscrete}. Although the dual flow may seem quite special (it is periodic, for example) we will show in this section that the KMS bundles they give rise to comprise all non-empty proper simplex bundles. To be slightly more precise, consider an arbitrary non-empty proper simplex bundle $(S,\pi)$. We will show that there is an AF-algebra $A$ and an automorphism $\alpha$ of $A$ such that the dual flow of $\alpha$ restricted to a corner of $ A \rtimes_\alpha \mathbb Z$ given by a projection in $A$ has a KMS bundle isomorphic to $(S,\pi)$. In view of Theorem \ref{12-11-22x} this means that the KMS bundle for any flow on a unital separable $C^*$-algebra can be realized by the dual flow on a crossed product by $\mathbb Z$ in this way. The basis for this is the relation between dimension groups and AF-algebras developed by Elliott, \cite{E1}, and Effros, Handelman and Shen, \cite{EHS}. In addition the proofs also use some of the material we have developed in this text, and it is worth noting that although we are only interested in KMS states on unital $C^*$-algebras the proofs require that we consider KMS weights on non-unital $C^*$-algebras.

 \begin{lemma}\label{26-10-20} Let $A$ be a $C^*$-algebra and $\alpha \in \Aut A$ an automorphism of $A$. Let $\hat{\alpha}$ be the dual flow of $\alpha$.  For $\beta \in \mathbb R$ the restriction map $\omega \mapsto \omega|_A$ is a bijection from the set of $\beta$-KMS weights for $\hat{\alpha}$ onto the lower semi-continuous traces $\tau$ on $A$ with the property that $\tau \circ \alpha = e^{-\beta} \tau$. The inverse of the map $\omega \mapsto \omega|_A$ is the map $\tau \mapsto \tau \circ P$, where $P:  A \times_\alpha \mathbb Z  \to A$ is the canonical conditional expectation.
 \end{lemma}
 \begin{proof} The dual flow is a special case of the flows considered in Section \ref{crosseddiscrete}. It fact, it is the flow $\gamma^\theta$ from that section when the map $\theta : \mathbb Z \to \mathbb R$ there is the canonical inclusion. Since $\ker \theta = \{0\}$ in this case, Lemma \ref{26-10-20} is a special case of Theorem \ref{12-03-22}.
 \end{proof}

By combining Lemma \ref{26-10-20} with Theorem \ref{01-03-22d} we obtain the following.

\begin{lemma}\label{06-08-21} Let $A$ be a $C^*$-algebra, $\alpha  \in \Aut A$ an automorphism of $A$ and $q \in A$ a projection in $A$ which is full in $A\rtimes_\alpha \mathbb Z$. Let $\hat{\alpha}^q$ be the restriction to $q(A\rtimes_\alpha \mathbb Z)q$ of the dual flow of $\alpha$. Let $P : A \rtimes_\alpha \mathbb Z \to A$ be the canonical conditional expectation. For each $\beta \in \mathbb R$ the map 
$$
\tau \mapsto \tau \circ P|_{q(A \rtimes_\alpha \mathbb Z)q}
$$ 
is a bijection from the set of lower semi-continuous traces $\tau$ on $A$ that satisfy 
\begin{itemize}
\item $\tau \circ \alpha = e^{-\beta} \tau$, and
\item $\tau(q) =1$,
\end{itemize}
onto the simplex of $\beta$-KMS states for $\hat{\alpha}^q$.\end{lemma}

By Lemma \ref{26-10-20} the inverse of the bijection $\tau \mapsto \tau \circ P|_{q(A \rtimes_\alpha \mathbb Z)q}$ in Lemma \ref{06-08-21} is the map
\begin{equation}\label{28-11-22a}
\omega \mapsto \tilde{\omega}|_A
\end{equation}
where $\tilde{\omega}$ denotes the unique $\beta$-KMS weight for the dual flow $\hat{\alpha}$ of $\alpha$ on $A \rtimes_\alpha \mathbb Z$ which extends $\omega$. Since the set $S^{\hat{\alpha}^q}_\beta$ of $\beta$-KMS states for $\hat{\alpha}^q$ on $q(A \rtimes_\alpha\mathbb Z)q$ is compact in the weak* topology it follows that the set of lower semi-continuous traces on $A$ satisfying the two conditions in Lemma \ref{06-08-21} is compact in the weakest topology making evaluation at elements of $(qAq)^+$ continuous, and in that topology it is affinely homeomorphic to $S^{\hat{\alpha}^q}_\beta$.

When $A$ is an AF-algebra it follows from Theorem \ref{22-11-22} that the map $\tau \mapsto \tau_*$ is a bijection from the set of lower semi-continuous traces $\tau$ on $A$ onto the set $\Hom^+(K_0(A),\mathbb R)$ of non-zero positive homomorphisms $\phi :  K_0(A) \to \mathbb R$. In this case the map $\tau \mapsto \tau_*$ restricts to an affine homeomorphism from the set of lower semi-continuous traces $\tau$ on $A$ which satisfies the two conditions of Lemma \ref{06-08-21} onto the set of elements $\phi \in \Hom^+(K_0(A),\mathbb R)$ such that
\begin{itemize}
\item $\phi \circ \alpha_* = e^{-\beta} \phi$, and
\item $\phi([q]) =1$,
\end{itemize}
when the latter set is given the topology of pointwise convergence on 
$$
\left\{ x \in K_0(A): \ 0 \leq x \leq [q]\right\}.
$$ 
In this way Lemma \ref{06-08-21} has the following

\begin{cor}\label{06-08-21a} In the setting of Lemma \ref{06-08-21} assume that $A$ is an AF-algebra. For each $\beta \in \mathbb R$ the map 
$$
\omega \mapsto (\tilde{\omega}|_A)_*
$$
is an affine homeomorphism from the set of $\beta$-KMS states for $\hat{\alpha}^q$ on $q(A \rtimes_\alpha\mathbb Z)q$ onto the set of positive homomorphisms $\phi \in \Hom^+(K_0(A),\mathbb R)$ that satisfy 
\begin{itemize} 
\item $\phi \circ \alpha_* = e^{-\beta} \phi$, and
\item $\phi([q]) =1$.
\end{itemize}
\end{cor}

Thanks to this corollary it is possible to determine the KMS simplices and in fact the entire KMS bundle for the flow $\hat{\alpha}^q$ from $K$-theory data alone. In the following example we illustrate this point.

\begin{example}\label{08-12-22x} \textnormal{ Let $\mathbb Q$ denote the additive group of rational numbers and consider the group $\mathbb Q^6$ with the strict ordering; i.e.
$$
\left(\mathbb Q^6\right)^+ : = \left\{ (q_1,q_2,q_3,q_4,q_5,q_6) \in \mathbb Q: \ q_i > 0, \ i = 1,2,3,4,5,6\right\} \cup \{0\}.
$$
Define the automorphism $\rho$ of $\mathbb Q^6$ by
$$
\rho(q_1,q_2,q_3,q_4,q_5,q_6) = (q_1,\frac{q_2}{7},\frac{q_3}{7},\frac{q_4}{3},\frac{q_5}{3},\frac{q_6}{3}) .
$$
This is an automorphism of the ordered group 
$$
\left(\mathbb Q^6, \left(\mathbb Q^6\right)^+\right) .
$$
Since this is a dimension group it follows from the theorem of Effros, Handelman and Shen which is reproduced in  Theorem 2.3.49 in \cite{Th5}
that there is an AF-algebra $A$ such that $\Sigma(A) = K_0(A)^+$ and $(K_0(A),K_0(A)^+)$ is isomorphic to $\left(\mathbb Q^6, \left(\mathbb Q^6\right)^+\right)$. By the work of Elliott, \cite{E1}, reproduced in (3) of Theorem 2.3.37 of \cite{Th5},
there is an  automorphism $\alpha$ of $A$ such that $\alpha_* = \rho$ when we make the identification $(K_0(A),K_0(A)^+) =\left(\mathbb Q^6, \left(\mathbb Q^6\right)^+\right)$. Since every non-zero element of $(\mathbb Q^6)^+$ is an order unit, it
follows from \cite{Br} and \cite{E1}, see Theorem 4.2.10 of \cite{Th5}, that $A$ is simple. An application of Theorem \ref{02-12-22} in Appendix \ref{elliott} shows that the same is true for the crossed product $A \times_\alpha \mathbb Z$. Since $\Sigma(A) = (\mathbb Q^6)^+$ there is a projection $e \in A$ such that $[e] = (1,1,1,1,1,1) \in K_0(A)^+$. By using Corollary \ref{06-08-21a} it is not difficult to show that the restriction $\hat{\alpha}^e$ of the dual flow of $\alpha$ to the corner $e(A \times_\alpha \mathbb Z)e$ has a $\beta$-KMS state if and only if $\beta \in \{0,\log 3,\log 7\}$. The $0$-KMS state is unique while the simplices of $\log 3$-KMS states and $\log 7$-KMS states are affinely homeomorphic to a triangle and an interval, respectively.}
\end{example}

The main result of this section is the following

\begin{thm}\label{12-11-22} Let $(S,\pi)$ be a non-empty proper simplex bundle. There is an AF-algebra $A$, an automorphism $\alpha$ of $A$ and a projection $e \in A$ such that $e(A \rtimes_\alpha \mathbb Z)e$ is simple and the KMS bundle of the dual flow of $\alpha$ restricted to $e(A \rtimes_\alpha \mathbb Z)e$ is isomorphic to $(S,\pi)$.
\end{thm}

The proof of this theorem is quite long and will occupy the rest of this section.

\subsection{The proof of Theorem \ref{12-11-22}}\label{29-12-22a}

\begin{lemma}\label{22-11-22bus1} Let $(S,\pi)$ be a proper simplex bundle. For each $\beta \in \mathbb R$ the set
$$
\left\{ f|_{\pi^{-1}(\beta)} : \ f \in \mathcal A(S,\pi) \right\}
$$
is dense in $\Aff \pi^{-1}(\beta)$.
\end{lemma}
\begin{proof} Let $X \subseteq \Aff \pi^{-1}(\beta)$ be the set in question. Note that $X$ is a subspace of $\Aff \pi^{-1}(\beta)$. Let $l \in \Aff \pi^{-1}(\beta)^*$ such that $l(X) = \{0\}$. We must show that $l =0$. As shown in the proof of Lemma \ref{11-11-22} there are non-negative real numbers $s,t$ and points $x,y \in \pi^{-1}(\beta)$ such that $l(g) = sg(x) -tg(y)$ for all $g\in \Aff \pi^{-1}(\beta)$. (For this one can also use Lemma 3.1.8 and Theorem 3.1.2 in \cite{Th5}.) Since the constant function $1$ is in $X$ it follows that $s=t$, and unless $s=t = 0$ also that $f(x) = f(y)$ for all $f \in \mathcal A(S,\pi)$. Since $\mathcal A(S,\pi)$ separates the points of $S$ the last conclusion implies that $x =y$. Hence $l =0$.
\end{proof}

Given a proper simplex bundle $(S,\pi)$ and a closed subset $F \subseteq \mathbb R$ we denote by $(S_F,\pi_F)$ the proper simplex bundle where $S_F := \pi^{-1}(F)$ and $\pi_F$ is the restriction of $\pi$ to $S_F$. The following lemma relates $\mathcal A(S_F,\pi_F)$ to $\mathcal A(S,\pi)$ and it will be a crucial tool in the following.

\begin{lemma}\label{03-09-21a} Let $(S,\pi)$ be a proper simplex bundle and $F \subseteq \mathbb R$ a closed set. 
\begin{itemize}
\item[(1)] The map $\mathcal A(S,\pi) \to \mathcal A(S_F,\pi_F)$ given by restriction is surjective.
\item[(2)] Let $f_1,f_2,g_1,g_2 \in \mathcal A(S,\pi)$ such that $f_i(x) < g_j(x)$ for all $x \in S$ and all $i,j \in \{1,2\}$. Assume that there is an element $h^F \in \mathcal A(S_F,\pi_F)$ such that 
$$
f_i(x) < h^F(x) <  g_j(x) \ \ \forall x \in S_F, \ \forall i,j \in \{1,2\} \ .
$$ 
There is an element $h \in \mathcal A(S,\pi)$ such that $h(y) = h^F(y)$ for all $y \in S_F$ and 
\begin{equation}\label{23-11-22}
f_i(x) < h(x) < g_j(x) \ \ \forall x \in S, \ \forall i,j \in \{1,2\} \ .
\end{equation}
\end{itemize}
\end{lemma}
\begin{proof} (1).
We prove first (1) under the assumption that $S$ and $F$ are both compact. Under this assumption consider a function $h \in \mathcal A(S_F,\pi_F)$.
 By induction we will construct a sequence $\{g_n\}_{n=1}^\infty$ in $\mathcal A(S,\pi)$ such that
\begin{itemize}
\item[(i)] $\sup_{x \in \pi^{-1}(F)} \left| h(x) - \sum_{j=1}^n g_j(x)\right| < 2^{-n}$ for all $n \geq 1$ and
\item[(ii)] $\sup_{x\in S} \left|g_n(x)\right| \leq 2^{-n} + 2^{-n+1}$ for all $n \geq 2$.
\end{itemize}
Let $\beta \in F$ and $\epsilon >0$. It follows from Lemma \ref{22-11-22bus1} that there is an element $h^\beta \in \mathcal A(S,\pi)$ such that $|h^\beta(x) -h(x)| < \epsilon$ for all $\xi \in \pi^{-1}(\beta)$. By continuity of $h^\beta$ and $h$, as well as the assumed compactness of $S$, there is an open neighbourhood $U(\beta)$ of $\beta$ in $\mathbb R$ such that $|h^\beta(x) -h(x)| < \epsilon$ for all $x \in \pi^{-1}(U(\beta) \cap F)$. By compactness of $F$ we can therefore choose a finite set of real-valued functions $\varphi_i \in C_c(\mathbb R)$ and $g_i \in \mathcal A(S,\pi), \ i = 1,2,\cdots,M$, such that 
$$
\left|\sum_{j=1}^M \varphi_j(\pi(x))g_j(x) - h(x) \right| <  \epsilon
$$
for all $x \in \pi^{-1}(F)$. In particular, $\left|\sum_{j=1}^M \varphi_j(\pi(x))g_j(x)\right| < |h(x)| + \epsilon$ for all $x \in \pi^{-1}(F)$.
We start the induction by taking $\epsilon = 2^{-1}$ and 
$$
g_1(x)  := \sum_{j=1}^M \varphi_j(\pi(x))g_j(x).
$$ 
Then (i) holds for $n=1$. To construct $g_2$ we repeat the above construction with $h$ replaced by $h - g_1$ and $\epsilon$ by $2^{-2}$ to get $g_2 \in \mathcal A(S,\pi)$ with
$$
\sup_{x \in \pi^{-1}(F)} |h(x) - g_1(x) - g_2(x)| < 2^{-2}
$$
and
$$
\sup_{x\in S}\left|g_2(x)\right| \leq 2^{-1} + 2^{-2}.
$$
Then (i) and (ii) hold for $n=2$. When $g_1,g_2, \cdots , g_n$, $n \geq 2$, have been defined we obtain $g_{n+1}$ by using the construction with $h$ replaced by $h-\sum_{j=1}^ng_n$ and $\epsilon$ by $2^{-n-1}$. This completes the construction of the sequence $\{g_n\}_{n=1}^\infty$.  It follows from (ii) that the sum $f(x) := \sum_{j=1}^\infty g_j(x)$ is uniformly convergent on $S$. The sum function $f$ is an element of $\mathcal A(S,\pi)$ since the $g_j$'s are and then (i) implies that $f|_{\pi^{-1}(F)} = h$. 

To prove (1) in the general case, consider again 
 $h \in \mathcal A(S_F,\pi_F)$. For each $n \in \mathbb N$ the pair $(S_{[-n,n]},\pi_{[-n,n]})$ is a compact proper simplex bundle and it follows from the preceding that there are elements $f_n \in \mathcal A(S_{[-n,n]},\pi_{[-n,n]})$ such that the restriction $f_n|_{S_{F \cap [-n,n]}}$ of $f_n$ to $S_{F \cap [-n,n]}$ agrees with $h|_{S_{F \cap [-n,n]}}$. For $n \in \mathbb N \backslash \{0\}$, let $\chi_n : \mathbb R \to [0,1]$ be a continuous function such that $\chi_n (t) = 1$ for $t \leq n-\frac{1}{2}$ and $\chi_n(t) = 0$ for $t \geq n$. Define $f'_n : S_{[-n,n]} \to \mathbb R$ recursively by
$$
f'_1(x) = (1-\chi_1(|\pi(x)|))f_2(x) + \chi_1(|\pi(x)|)f_1(x) \ ,
$$
and then $f'_n$ for $n \geq 2$ such that $f'_{n}(x) = f'_{n-1}(x)$ when $x \in \pi^{-1}([-n+1,n-1])$ and $f'_n(x) =  (1-\chi_n(|\pi(x)|))f_{n+1}(x) + \chi_n(|\pi(x)|)f_n(x)$ when $x \in \pi^{-1}([n-1,n] \cup [-n,-n+1])$. Then $f'_n|_{S_{F \cap [-n,n]}} = h|_{S_{F \cap [-n,n]}}$ and since $f'_{n+1}$ extends $f'_n$, there is an element $f \in \mathcal A(S,\pi)$ such that $f|_{S_{[-n,n]}} = f'_n$ for all $n$. This element $f$ extends $h$.

To establish (2) we assume again first that $S$ and $F$ are compact. By (1) we may assume that $h^F$ is the restriction to $\pi^{-1}(F)$ of an element of $\mathcal A(S,\pi)$ which we again denote by $h^F$. By compactness of $S$ there is an open neighborhood $U(F)$ of $F$ in $\mathbb R$ such that $f_i(x) < h^F(x) < g_j(x)$ for all $i,j \in \{1,2\}$ and $x \in \pi^{-1}(U(F))$. Let $\beta \in \mathbb R \backslash F$. Since $\pi^{-1}(\beta)$ is a Choquet simplex it follows from Lemma 3.1 in \cite{EHS} that $\Aff \pi^{-1}(\beta)$ has the Riesz interpolation property with respect to the strict ordering, see also (2) in Theorem 3.2.7 of \cite{Th5}, and it follows therefore in combination with (1) that there is a $h_\beta \in \mathcal A(S,\pi)$ such that $f_i(x) < h_\beta(x) < g_j(x)$ for all $i,j \in \{1,2\}$ and $x \in \pi^{-1}(\beta)$. By compactness of $S$ there is then an open neighborhood $U(\beta)$ of $\beta$ in $\mathbb R$ such that $f_i(x) < h_\beta(x) < g_j(x)$ for all $i,j \in \{1,2\}$ and $x \in \pi^{-1}(U(\beta))$. By compactness of $\pi(S)$ there is a finite set $V \subseteq \pi(S) \backslash F$ such that $\pi^{-1}(U(\beta)), \ \beta \in V$, and $\pi^{-1}(U(F))$ cover $S$. Then $U(\beta), \ \beta \in V$, together with $U(F)$ cover $\pi(S)$ and we can choose functions $\varphi_i \in C_c(\mathbb R), i \in V$, and $\varphi_F\in C_c(\mathbb R)$, such that $0 \leq \varphi_i(t) \leq 1$ and $0 \leq \varphi_F(t) \leq 1$ for all $i \in V$ and $t \in \mathbb R$, $\sum_{i\in V} \varphi_i(t) + \varphi_F(t) = 1$ for all $t \in \pi(S)$ and $\varphi_F(s) = 1$ for all $s \in F$. Then $h(x) = \varphi_F(\pi(x))h^F(x) + \sum_{i\in V} \varphi_i(\pi(x))h_i(x)$ is a function $h \in \mathcal A(S,\pi)$ with the desired property.

To prove (2) in the general case, consider again 
 $h \in \mathcal A(S_F,\pi_F)$. For each $n \in \mathbb N$ the pair $(S_{[-n,n]},\pi_{[-n,n]})$ is a compact proper simplex bundle and it follows from the preceding that there are elements $h_n \in \mathcal A(S_{[-n,n]},\pi_{[-n,n]})$ such that $h_n(y) = h^F(y)$ for all $y \in \pi^{-1}(F \cap [-n,n])$ and 
$$
f_i(x) < h_n(x) < g_j(x) \ \ \forall x \in  \pi^{-1}([-n,n]), \ \forall i,j \in \{1,2\} \ .
$$
For $n \in \mathbb N$ let $\chi_n : \mathbb R \to [0,1]$ be a continuous function such that $\chi_n (t) = 1$ for $t \leq n-\frac{1}{2}$ and $\chi_n(t) = 0$ for $t \geq n$. Define $h'_n : S_{[-n,n]} \to \mathbb R$ recursively by
$$
h'_1(x) = (1-\chi_1(|\pi(x)|))h_2(x) + \chi_1(|\pi(x)|)h_1(x) \ ,
$$
and then $h'_n$ for $n \geq 2$ such that $h'_{n}(x) = h'_{n-1}(x)$ when $x \in \pi^{-1}([-n+1,n-1])$ and $h'_n(x) =  (1-\chi_n(|\pi(x)|))h_{n+1}(x) + \chi_n(|\pi(x)|)h_n(x)$ when $x \in \pi^{-1}([n-1,n] \cup [-n,-n+1])$. Then $h'_n|_{S_{F \cap [-n,n]}} = h^F|_{S_{F \cap [-n,n]}}$ and since $h'_{n+1}$ extends $h'_n$, there is an element $h \in \mathcal A(S,\pi)$ such that $h|_{S_{[-n,n]}} = h'_n$ for all $n$. This element $h$ extends $h^F$ and satisfies \eqref{23-11-22}.

\end{proof}


 Fix a non-empty proper simplex bundle $(S,\pi)$.  In the following we consider $\mathcal A(S,\pi)$ as an abelian group with addition as composition. Let $\mathcal A_c(S,\pi)$ be the set of elements $f$ from $\mathcal A(S,\pi)$ whose support $\supp f$ is compact.
 
\begin{lemma}\label{07-09-21} There is a countable subgroup $G_0$ of $\mathcal A_c(S,\pi)$ with the following density property: For all $f \in \mathcal A_{c}(S,\pi)$ there is a $g \in G_{0}$ such that 
$$
\sup_{x \in S} |f(x) -g(x)| < \epsilon .
$$ 
\end{lemma}
\begin{proof} Since $S$ is second countable there is a sequence $\{U_n\}$ of open sets in $S$ such that $\overline{U_n}$ is compact, $\overline{U_n} \subseteq U_{n+1}$ for all $n$ and $S = \bigcup_n U_n$. Since $U_n$ is second countable in the relative topology there is a countable dense set $L_n$ in $C_0(U_n)$. For each $g \in L_n$ choose an element $a_g \in \mathcal A(S,\pi) \cap C_0(U_n)$ such that 
$$
\left\| g - a_g\right\| \leq 2 \inf \left\{\| g - a\|: \ a \in  \mathcal A(S,\pi) \cap C_0(U_n) \right\} .
$$
Then $\left\{ a_g : \ g \in L_n\right\}$ is a countable set which is dense in $\mathcal A(S,\pi) \cap C_0(U_n)$. Since
$$
\mathcal A_c(S,\pi) = \bigcup_n \mathcal A(S,\pi) \cap C_0(U_n) 
$$
the additive group generated by $\bigcup_n \left\{ a_g : \ g \in L_n\right\}$ is countable and has the stated density property. 
\end{proof}

 Let $\mathcal Z$ be the additive subgroup of $\mathcal A(S,\pi)$ generated by the functions  
 $$ x \mapsto 
qe^{z\pi(x)}
$$ 
where $q \in \mathbb Q$ and $z \in \mathbb Z$. For $k \in \mathbb N \backslash \{0\}$, choose continuous functions $\psi^{i}_k : \mathbb R \to [0,1], \ i \in \{0,\pm\}$, such that $\psi^{-}_k(t) = 1, \ t \leq -k-1$, $\psi^0_k(t) = 1, \ t \in [-k,k]$, $\psi^+_k(t) = 1, \ t \geq k+1$, and $\psi_k^{-}(t) + \psi^0_k(t) + \psi^+_k(t) = 1$ for all $t \in \mathbb R$. Define countable subgroups $G_1\subseteq G_2 \subseteq G_3 \subseteq \cdots$ of $\mathcal A(S,\pi)$ recursively such that $G_1 := G_0+\mathcal Z$ where $G_0$ is the countable group from Lemma \ref{07-09-21}, and for $n \geq 2$, $G_n$ is the additive subgroup of $\mathcal A(S,\pi)$ generated by the functions
$$
S  \ni x \mapsto \psi^i_k(\pi(x))e^{z\pi(x)}g(x) ,
$$
where $i \in \{0,\pm\}, \ k \in \mathbb N \backslash \{0\}, \ z \in \mathbb Z$ and $g \in G_{n-1}$. Then $G_n \subseteq G_{n+1}$ and 
$$
G := \bigcup_{k=1}^\infty G_k
$$
is a countable subgroup of $\mathcal A(S,\pi)$.
\begin{lemma}\label{X07-12-22} The group $G$ has the following properties:
\begin{enumerate}
\item[(a)] $G_{0} \subseteq G$. \\
\item[(b)] $\mathcal Z \subseteq G$. \\
\item[(c)] $(\psi^i_k\circ \pi) G \subseteq G$ for all $i\in \{0,\pm\}$ and all $k \in \mathbb N \backslash \{0\}$.\\
\item[(d)] $e^{z\pi}G = G$ for all $z \in \mathbb Z$.\\
\item[(e)] For each $g \in G$ there are elements $a_\pm \in \mathcal Z$ and $m_\pm \in \mathbb N$ such that $g(x) = a_+(x)$ when $\pi(x) \geq m_+$ and $g(x) = a_-(x)$ when $\pi(x) \leq -m_-$ .
\item[(f)] For each $g \in G$ there are integers $k_\pm \in \mathbb Z$ such that for every $\epsilon >0$ there is an $m  \in \mathbb N$ with the property that $\left|e^{k_+\pi(x)}g(x)\right| \leq \epsilon$ when $x \in S$ and $\pi(x) \geq m$, and $\left|e^{k_-\pi(x)}g(x)\right| \leq \epsilon$ when $x \in S$ and $\pi(x) \leq -m$.
\end{enumerate}
\end{lemma}
\begin{proof} Left to the reader.
\end{proof}

%

Set
$$
\mathcal A(S,\pi)^+ := \left\{ f \in \mathcal A(S,\pi): \ f(x) > 0 \ \forall x \in S \right\} \cup \{0\} \ 
$$
and
$$
G^+ := G \cap \mathcal A(S,\pi)^+ \ .
$$ 

\begin{lemma}\label{01-09-21} The pair $(G,G^+)$ has the following properties.
\begin{itemize}
\item[(1)] $G^+ \cap (-G^+) = \{0\}$.
\item[(2)] $G = G^+ - G^+$.
\item[(3)] $(G,G^+)$ is unperforated, i.e. $n \in \mathbb N \backslash \{0\}, \ g \in G, \ ng \in G^+ \Rightarrow g \in G^+$. 
\item[(4)] $(G,G^+)$ has the strong Riesz interpolation property, i.e. if $f_1,f_2,g_1,g_2 \in G$ and $f_i(x) < g_j(x)$ for all $i,j \in \{1,2\}$ and all $x \in S$, then there is an element $h \in G$ such that
\begin{equation}\label{07-12-23}
f_i(x) < h(x) < g_j(x)
\end{equation}
for all $i,j \in \{1,2\}$ and all $x \in S$. 
\end{itemize}
\end{lemma}
\begin{proof} (1) and (3) are obvious. (2): Let $f \in G$. It follows from property (f) of Lemma \ref{X07-12-22} that there are natural numbers $m,k  \in \mathbb N \cup \{0\}$ such that $f(x) < m g(x)$ for all $x \in S$, when 
$$
g(x) :=\psi^{-}_k(\pi(x)) e^{-k\pi(x)} + \psi^0_k(\pi(x)) + \psi^+_k(\pi(x))e^{k\pi(x)}.
$$
Then $m g \in G^+$, $m g - f \in G^+$ and $f = mg - (mg-f) \in G^+-G^+$.

(4): By property (e) of Lemma \ref{X07-12-22} there are elements $a_i,b_j \in \mathcal Z$ and a natural number $k$ such that $f_i(x) = a_i(x)$ and $g_j(x) = b_j(x), \ i,j \in \{1,2\}$, when $\pi(x) \geq k$. By definition of $\mathcal Z$ there is an $N \in \mathbb N$ such that $e^{N\pi(x)}(b_j(x) -a_i(x)) \geq 2$ for all $i,j$ and all $x$ with $\pi(x)$ large enough. We may therefore assume that this holds when $\pi(x) \geq k$. Similarly, we have either that $a_1(x)\leq a_2(x)$ or that $a_2(x)\leq a_1(x)$ for all $x$ with $\pi(x)$ large enough. Without loss of generality we assume that 
 $a_1(x)\leq a_2(x)$ when $\pi(x) \geq k$. Set $h_+(x) := a_2(x) + e^{-N\pi(x)}$. Then $h_+ \in G$ and $f_i(x) < h_+(x) < g_j(x)$ for all $i,j$ and all $x\in \pi^{-1}([k,\infty))$. In the same way we find $l \in \mathbb N$ and an element $h_-\in G$ such that $f_i(x) < h_-(x) < g_j(x)$ for all $i,j$ and all $x \in \pi^{-1}(-\infty,-l])$. Set $F t:= \pi^{-1}([k,\infty)) \cup \pi^{-1}((-\infty,-l])$ and $h^F(x) = \psi^{-}_{l}(\pi(x))h_-(x) + \psi^+_{k}(\pi(x))h_+(x)$. Then $h^F \in G \subseteq \mathcal A(S,\pi)$ and 
\begin{equation}\label{07-12-22b}
f_i(x) < h^F(x) < g_j(x)
\end{equation} 
for all $i,j$ and all $x \in \pi^{-1}(F)$. It follows therefore from (2) of Lemma \ref{03-09-21a} that there is a $h' \in \mathcal A(S,\pi)$ such that $f_i(x) < h'(x) < g_j(x)$ for all $i,j$ and all $x \in S$, and such that $h'(x) = h^F(x)$ for all $\pi^{-1}(F)$. Note that $\pi^{-1}([-l-1,k+1])$ is compact in $S$ since $\pi$ is proper. Let $\delta >0$ be smaller than both
 $$
 \inf \left\{ h'(x) -f_i(x) : \ x \in \pi^{-1}([-l-1,k+1]), \ i \in \{1,2\} \right\}
 $$
 and
 $$
 \inf \left\{ g_j(x) - h'(x): \  x \in \pi^{-1}([-l-1,k+1]), \ j \in \{1,2\} \right\} .
 $$
It follows from Lemma \ref{07-09-21} that there is a $h'' \in G_{0}$ such that $\left|h''(x) - h'(x)\right| \leq \frac{\delta}{2}$ for all $x \in \pi^{-1}([-l-1,k+1])$, and hence
\begin{equation}\label{07-12-22d}
f_i(x) < h''(x) < g_j(x)
\end{equation}
for all $i,j$ when $x \in \pi^{-1}([-l-1,k+1])$.
 Set 
\begin{align*}
&h(x) := \\
&\psi^{-}_l(\pi(x))h^F(x) + \psi^+_k(\pi(x))h^F(x) + (1-\psi^{-}_l(\pi(x)))(1-\psi^+_k(\pi(x)))h''(x) .
\end{align*}
Since $h^F,h'' \in G$ it follows from (c) of Lemma \ref{X07-12-22} that $h \in G$. When $\pi(x)$ is in $[-l-1,-l]$ or $[k,k+1]$, $h(x)$ is a convex combination of $h''(x)$ and $h^F(x)$ and hence the estimate \eqref{07-12-23} follows from \eqref{07-12-22b} and \eqref{07-12-22d}.
\end{proof}

\begin{cor}\label{07-12-22g} $(G,G^+)$ is a dimension group.
\end{cor}
\begin{proof} We refer to \cite{EHS} or \cite{Th5} for the basic theory of dimension groups. In view of Lemma \ref{01-09-21} it remains only to prove that $(G,G^+)$ has the Riesz interpolation property, so assume $f_i \leq g_j, \ i,j \in \{1,2\}$, in $G$. If $f_{i'} = g_{j'}$ for some $i'$ and $j'$, set $h = f_{i'}$. Then $f_i \leq h \leq g_j$ for all $i,j$. If $f_i \neq g_j$ for all $i,j$, we get the interpolating element $h$ from (4) of Lemma \ref{01-09-21}.
\end{proof}

Thanks to Corollary \ref{07-12-22g} it follows from the theorem of Effros, Handelman and Shen, \cite{EHS}, which is presented in Theorem 2.3.49 of \cite{Th5} that there is an AF-algebra $A$ such that $(K_0(A),K_0(A)^+) = (G,G^+)$ and $\Sigma(A) = K_0^+(A)$, where 
$$
\Sigma(A):= \left\{[e] \in K_0(A): \ e=e^* = e^2 \right\}
$$
is the scale of $K_0(A)$. Thanks to (d) of Lemma \ref{X07-12-22} we can define an automorphism $\rho$ of $G$ by
$$
\rho(f)(x) = e^{-\pi(x)}f(x), \ \ f \in G.
$$
Since $\rho(G^+) = G^+$, it follows that $\rho$ is an automorphism of the ordered group $(G,G^+)$, and hence by a result of Elliott, \cite{E1}, reproduced in (3) of Theorem 2.3.37 in \cite{Th5}, that there is an automorphism $\alpha$ of $A$ such that $\alpha_* = \rho$ under the identification $(K_0(A),K_0(A)^+) = (G,G^+)$.

\begin{lemma}\label{04-08-21x} The only order ideals $I$ in $(G,G^+)$ such that $\rho(I) = I$ are $I = \{0\}$ and $I = G$.
\end{lemma}
\begin{proof} Recall that an order ideal in $(G,G^+)$ is a subgroup $I$ of $G$ such that
 \begin{itemize}
\item[(a)] $I= I \cap G^+ -  I \cap G^+$, and
\item[(b)] when $ 0 \leq y \leq x$ in $G$ and $x \in I$, then $y \in I$.
\end{itemize}
Let $I$ be a non-zero order ideal such that $\rho(I) = I$. Since $I\neq \{0\}$ there is an element $h \in I \cap G^+ \backslash \{0\}$. It follows from (e) of Lemma \ref{X07-12-22} that there are integers $k_\pm \in \mathbb Z$ and positive numbers $r_\pm > 0$ such that $h(x) > r_+ e^{k_+ \pi(x)}$ for all $x \in S$ with $\pi(x)$ large enough and $h(x) > r_- e^{k_- \pi(x)}$ for all $x \in S$ with $-\pi(x)$ large enough. There is therefore a $k \in \mathbb N$ such that 
$$
e^{k\pi(x)}  < h(x) \ \text{when} \ \pi(x) \leq -k -1,
$$
and
$$
e^{-k\pi(x)} < h(x) \ \text{when} \ \pi(x) \geq k+1 .
$$ 
Using Lemma \ref{07-09-21} we find $h_0 \in G_0$ such that $0 < h_0(x) < h(x)$ for all $x \in \pi^{-1}([-k-1,k+1])$ and we set
$$
h_1 : = (\psi^{-}_k\circ \pi) e^{k\pi} + (\psi^0_k\circ \pi) h_0 + (\psi^+_k\circ\pi) e^{-k\pi} .
$$
Then $h_1 \in G^+$ and 
$$
0 < h_1(x) < h(x)
$$
for all $x \in S$. Hence $h_1 \in I^+$. Let $g \in G^+ \backslash \{0\}$. Considerations similar to the preceding show that there are $N,M \in \mathbb N$ such that
$$
g(x) < M\left(\rho^{-N}(h_1)(x) + \rho^N(h_1)(x)\right) 
$$
for all $x \in S$. Since $I$ is $\rho$-invariant and $h_1 \in I$, it follows that $g\in I$. Hence $G = I$.
\end{proof}

We denote in the sequel by $\mathcal A_\mathbb R(S,\pi)$ the real Banach space consisting of the elements $f$ of $\mathcal A(S,\pi)$ that have a limit at infinity, in the sense that there is a real number $t$ with the property that for each $\epsilon > 0$ there is a compact subset $K$ of $S$ such that $|f(s) -t| \leq \epsilon$ when $s \notin K$.

\begin{lemma}\label{01-09-21e} Let $f \in \mathcal A_\mathbb R(S,\pi)$ and let $\epsilon > 0$ be given. There is an element $g\in G \cap \mathcal A_\mathbb R(S,\pi)$ such that 
$$
\sup_{x \in S} |f(x)- g(x)| \leq \epsilon \ .
$$
\end{lemma}
\begin{proof} An initial approximation gives us an element $f_1 \in \mathcal A_c(S,\pi)$ and a rational number $r \in \mathbb R$ such that 
$$
\sup_{x \in S} |f(x)- f_1(x) - r| \leq \frac{\epsilon}{2} \ .
$$
By Lemma \ref{07-09-21} there is an element $g_0 \in G_{0}$ such that $\sup_{s \in S}\left|f_1(x)-g_0(x)\right| \leq \frac{\epsilon}{2}$. This completes the proof because $g_0+r \in G_{0} + \mathbb Q \subseteq G \cap \mathcal A_\mathbb R(S,\pi)$.
\end{proof}

Let $\beta \in \mathbb R$ and $\omega \in \pi^{-1}(\beta)$.  Define $\omega_\beta : G \to \mathbb R$ such that
$$
\omega_\beta(g) := g(\omega) \ .
$$ 
Then $\omega_\beta(G^+) \subseteq [0,\infty)$, $\omega_\beta(1) = 1$ and $\omega_\beta \circ \rho = e^{-\beta} \omega_\beta$.

\begin{lemma}\label{27-08-21x}
Let $\phi : G\to \mathbb R$ be a positive homomorphism with the properties that $\phi(1) =1$ and $\phi \circ \rho = s \phi$ for some $s > 0$. Set $\beta := -\log s$. There is an element $\omega \in \pi^{-1}(\beta)$ such that $\phi = \omega_\beta$.
\end{lemma}
\begin{proof} Let $g \in G \cap \mathcal A_{\mathbb R}(S,\pi)$ such that $\left|g(x)\right| <\frac{n}{m}$ for all $x \in S$,
where $n,m \in \mathbb N \backslash \{0\}$. Then $ -n 1 \leq m g \leq n 1$ in $(G,G^+)$ and it follows from the properties of $\phi$ that $-n \leq m \phi(g)\leq  n$, or $|\phi(g)|\leq \frac{n}{m}$. It follows that 
\begin{equation}\label{08-12-22}
\left|\phi(g)\right| \leq \|g\|_\infty, \ \ \forall g \in G \cap \mathcal A_{\mathbb R}(S,\pi),
\end{equation}
when we set $\|f\|_\infty := \sup_{x \in S} |f(x)|$ for $f \in \mathcal A_\mathbb R(S,\pi)$. Let $f \in \mathcal A_\mathbb R(S,\pi)$. It follows from Lemma \ref{01-09-21e} that there is a sequence $\{g_n\}$ in $G \cap \mathcal A_{\mathbb R}(S,\pi)$ such that $\lim_{n \to \infty} \left\| f-g_n\right\|_\infty = 0$. It follows therefore from \eqref{08-12-22} that we can extend $\phi$ by continuity to a map $\phi : \mathcal A_\mathbb R(S,\pi) \to \mathbb R$ such that \eqref{08-12-22} holds for all $g \in  \mathcal A_\mathbb R(S,\pi)$. Note that the extension is additive since its restriction to $G \cap  \mathcal A_\mathbb R(S,\pi)$ is, and hence in fact linear on $ \mathcal A_\mathbb R(S,\pi)$. It follows from the Hahn-Banach extension theorem that $\phi$ extends further to a linear map $\phi : C_\mathbb R(S) \to \mathbb R$ when $C_\mathbb R(S)$ denotes the real Banach space of continuous real-valued functions on $S$ with a limit at infinity. Furthermore, the extension can be chosen to have the same norm and hence $|\phi(h)| \leq \|h\|_\infty$ for all $h \in    C_\mathbb R(S)$.
Furthermore, since $\phi(1) = 1$ it follows that $\phi : C_\mathbb R(S) \to \mathbb R$ is positive, i.e. $f \geq 0  \Rightarrow \phi(f) \geq 0$. Since $C_\mathbb R(S)$ contains all the continuous compactly supported functions it follows from the Riesz representation theorem that there is a regular Borel measure $m$ on $S$ such that 
\begin{equation}\label{08-12-22a}
\phi(f) = \int_S f \ \mathrm d m
\end{equation}
for all the compactly supported functions $f$ in $C_\mathbb R(S)$. Since we can choose a sequence of such functions increasing pointwise to $1$ it follows from Lebesgues theorem on monotone convergence that $m$ is a bounded measure and hence that \eqref{08-12-22a} holds for all $f \in C_\mathbb R(S)$; in particular for all
 $f \in \mathcal A_{0}(S,\pi)$, where $\mathcal A_0(S,\pi)$ denotes the space of elements in $\mathcal A_\mathbb R(S,\pi)$ that vanish at infinity. 
Let $C_c(\mathbb R)$ denote the set of continuous real-valued compactly supported functions on $\mathbb R$ and note that $C_c(\mathbb R)$ is mapped into $\mathcal A_0(S,\pi)$ by the formula $F \mapsto F \circ \pi$. Since $\phi \circ \rho = s\phi$ by assumption it follows that the measure $m \circ \pi^{-1}$ on $\mathbb R$ satisfies 
$$
 \int_\mathbb R e^{-t}F(t) \ \mathrm{d}m \circ \pi^{-1}(t)  =  s\int_\mathbb R F(t) \ \mathrm{d}m \circ \pi^{-1}(t)  \ \ \forall F \in C_c(\mathbb R)  .
 $$
It follows that $m \circ \pi^{-1}$ is concentrated at the point $\beta = -\log s$ and hence that $m$ is concentrated on $\pi^{-1}(\beta)$. We can therefore define a linear functional $\phi'' : \Aff \pi^{-1}(\beta) \to \mathbb R$ by
$$
\phi'(f) = \phi(\hat{f}) = \int_S \hat{f}(x) \ \mathrm{d} m(x) \ ,
$$
where $\hat{f} \in \mathcal A_0(S,\pi)$ is any element with $\hat{f}|_{\pi^{-1}(\beta)} = f$, which exists by (1) in Lemma \ref{03-09-21a}. If $f \geq 0$ it follows that $\hat{f}$ can be chosen such that $\hat{f} \geq - \epsilon$ for any $\epsilon > 0$ and we see therefore that $\phi'$ is a positive linear functional. Since every state of $\Aff \pi^{-1}(\beta)$ is given by evaluation at a point in $\pi^{-1}(\beta)$, cf. e.g. Theorem 4.3 in \cite{Th5}, it follows in this way that there is an $\omega \in \pi^{-1}(\beta)$ and a number $\lambda \geq 0$ such that 
\begin{equation}\label{01-09-21h}
\phi'(g) = \lambda g(\omega)
\end{equation}
for all $g  \in \mathcal A_0(S,\pi)$. In particular, this conclusion holds for all $g \in G \cap \mathcal A_0(S,\pi)$. A general element $f \in G$ can be write as a sum
$$
f = f_- + f_0 + f_+ ,
$$
where $f_\pm, f_0 \in G$, $f_0$ has compact support and there are natural numbers $n_{\pm} \in \mathbb N$ such that $e^{n_-\pi}f_- \in \mathcal A_0(S,\pi)$ and $e^{-n_+ \pi}f_+ \in \mathcal A_0(S,\pi)$. Then $\phi'(f_0) = \lambda f_0(\omega)$,
\begin{align*}
&\phi'(f_-) = \phi'(\rho^{n_-}(e^{n_-\pi}f_-)) = s^{n_-}\phi'(e^{n_-\pi}f_-) \\
&=  s^{n_-}\lambda e^{n_-\pi(\omega)}f_-(\omega) = \lambda f_-(\omega),
\end{align*}
and similarly, $\phi'(f_+) = \lambda f_+(\omega)$.
It follows that $\phi(f) = \lambda f(\omega)$. Inserting $f=1$ we find that $\lambda = 1$ and the proof is complete.

\end{proof}

 Let $e\in A$ be a projection in $A$ such that $[e]= 1$ when we make the identification $(K_0(A),K_0(A)^+) =(G,G^+)$.

\begin{lemma}\label{28-11-22b} $e$ is a full projection in $A \rtimes_\alpha \mathbb Z$.
\end{lemma}
\begin{proof} We must show that the closed two-sided ideal
$$
I := \overline{\Span (A\rtimes_\alpha\mathbb Z)e(A\rtimes_\alpha\mathbb Z)}
$$
generated by $e$ is all of $A\rtimes_\alpha\mathbb Z$. To this end note that $I \cap A$ is an $\alpha$-invariant ideal in $A$ which is not zero since it contains $e$. Since $\alpha_* = \rho$ it follows from Lemma \ref{04-08-21x} and \cite{E1} (or Theorem 4.1.14 from \cite{Th5}) that $I \cap A = A$. Since $A$ contains an approximate unit for $A\rtimes_\alpha \mathbb Z$ by Lemma \ref{08-09-23} this implies that $I =A\rtimes_\alpha \mathbb Z$.
\end{proof}

\begin{remark}\label{12-12-22} \rm  As we point out below $A\rtimes_\alpha\mathbb Z$ is simple; a fact which implies that the projection $e$ is full. We have included the preceding lemma since the proof of simplicity of $A\rtimes_\alpha\mathbb Z$ is rather long as it uses Theorem \ref{02-12-22}. Lemma \ref{28-11-22b} offers therefore a shortcut in the proof of Theorem \ref{12-11-22} if one is not interested in the simplicity of $e(A \rtimes_\alpha \mathbb Z)e$ and $A\rtimes_\alpha \mathbb Z$.
\end{remark}

We consider now the dual flow $\hat{\alpha}$ of $\alpha$ on $B := A  \rtimes_\alpha \mathbb Z$ and denote by $\hat{\alpha}^e$ the restriction of this flow to $eBe$.

\begin{remark}\label{31-10-23a} \rm 
A lower semi-continuous trace $\tau$ on a $C^*$-algebra is \emph{extremal} when the lower semi-continuous traces it dominates are scalar multiples of $\tau$, i.e. when it is extremal as a $1$-KMS weight for the trivial action. Later in Chapter \ref{factortypes} we shall need the following observation: 

\begin{obs}\label{31-10-23}
Let $\psi$ be a $\beta$-KMS weight for the dual flow $\hat{\alpha}$ on $A\rtimes_\alpha \mathbb Z$. If $\psi$ is extremal the restriction $\tau_\psi := \psi|_A$ is an extremal lower semi-continuous trace on $A$. 
\end{obs}
To see this note that ${\tau_\psi}_*\circ \rho = e^{-\beta}{\tau_\psi}_*$ by Lemma \ref{26-10-20}. Thanks to Lemma \ref{27-08-21x} there is therefore an $s \in \pi^{-1}(\beta)$ such that ${\tau_\psi}_*$ is a scalar multiple of the map $G \ni g \mapsto g(s)$. To show that $\tau_\psi$ is extremal consider a trace $\mu$ on $A$ such that $\mu \leq \tau_\psi$. We must show that $\mu$ is a scalar multiple of $\tau_\psi$. Since $\mu_* \leq {\tau_\psi}_*$ it follows that $\mu_* : G \to [0,\infty)$ must factor through the restriction map $G \to \Aff \pi^{-1}(\beta)$ and hence be a scalar multiple of the map $ G \ni g \mapsto g(s')$ for some $s' \in \pi^{-1}(\beta)$. In particular, $\mu_* \circ \rho = e^{-\beta}\mu_*$ and it follows therefore from Corollary \ref{06-08-21a} and Theorem \ref{22-11-22} in Appendix \ref{AppD} that there is $\beta$-KMS weight $\phi$ for $\hat{\alpha}$ such that ${\phi|_A} = \mu$. Similarly, there is also $\beta$-KMS weight $\phi'$ for $\hat{\alpha}$ such that ${\phi'|_A} = \tau_\psi -\mu$. It follows then from Lemma \ref{26-10-20} that $\psi = \phi + \phi'$. Since $\psi$ is extremal it follows that $\phi$ is a scalar multiple of $\psi$. Hence $\mu$ is a scalar multiple of $\tau_\psi$. \qed
\end{remark}

\begin{lemma}\label{02-09-21d} The KMS bundle $(S^{\hat{\alpha}^e},\pi^{\hat{\alpha}^e})$ of $\hat{\alpha}^e$ is isomorphic to $(S,\pi)$.
\end{lemma}
\begin{proof} Let $(\omega,\beta) \in S^{\hat{\alpha}^e}$. By Corollary \ref{06-08-21a} $(\tilde{\omega}|_A)_*$ is an element of $\Hom^+(G,\mathbb R)$ such that $(\tilde{\omega}|_A)_*(1) =1$ and $(\tilde{\omega}|_A)_* \circ \rho = e^{-\beta}(\tilde{\omega}|_A)_*$. By Lemma \ref{27-08-21x} there is a $s \in \pi^{-1}(\beta)$ such that $(\tilde{\omega}|_A)_*(g) = g(s)$ for all $g \in G$. $s$ is unique since $G$ separates the points of $S$ by Lemma \ref{07-09-21}. We define $\Phi : S^{\hat{\alpha}^e} \to S$ such that $\Phi(\omega,\beta) = s$. Then $\pi \circ \Phi = \pi^{\hat{\alpha}^e}$. Let $ x\in S$. Set $\beta:= \pi(x)$ and consider the positive homomorphism $x_\beta : G \to \mathbb R$. Since $x_\beta(1) =1$ and $x_\beta \circ \rho = e^{-\beta} x_\beta$ it follows from Corollary \ref{06-08-21a} that there is a $\beta$-KMS state $\omega$ for $\hat{\alpha}^e$ such that $(\tilde{\omega}|_A)_* = x_\beta$. Then $(\omega,\beta) \in S^{\hat{\alpha}^e}$ and $\Phi(\omega,\beta) =x$. This shows that $\Phi$ is surjective. 
If $(\omega_i,\beta_i) \in S^{\hat{\alpha}^e}, \ i =1,2$, are such that $\Phi((\omega_1,\beta_1)) = \Phi((\omega_2,\beta_2))$, it follows that $\beta_1 = \pi\left(\Phi((\omega_1,\beta_1))\right) = \pi\left(\Phi((\omega_2,\beta_2))\right) = \beta_2$. The injectivity of the map in Corollary \ref{06-08-21a} implies now that $\omega_1 = \omega_2$, and we conclude that $\Phi : S^{\hat{\alpha}^e} \to S$ is a bijection. In order to show that $\Phi$ is a homeomorphism, note that since $\pi \circ \Phi = \pi^{\hat{\alpha}^e}$, and $\pi$ and $\pi^{\hat{\alpha}^e}$ are both proper maps, it suffices now to show that $\Phi$ is continuous. Assume therefore that $\{(\omega_n,\beta_n)\}$ is a sequence in $S^{\hat{\alpha}^e}$ and that $\lim_{n \to \infty}(\omega_n,\beta_n) =(\omega,\beta)$ in $S^{\hat{\alpha}^e}$. Then $\lim_{n \to \infty} {(\tilde{\omega_n}|_A)}_*(g) =   {\tilde{\omega}|_A}_*(g)$ for all $g \in G$, implying that 
$\lim_{n\to \infty} g (\Phi((\omega_n,\beta_n))) = g (\Phi((\omega,\beta)))$ for all $g \in G$. Since $G$ separates the points of $S$, it follows that
$\lim_{n\to \infty} \Phi((\omega_n,\beta_n)) =\Phi((\omega,\beta))$ 
in $S$.
\end{proof}

 To complete the proof of Theorem \ref{12-11-22} we need to show
 
 \begin{lemma}\label{06-12-22} $eBe$ is simple.
 \end{lemma}
 \begin{proof} For $k \neq 0$ the automorphism ${\alpha^k}_* = \rho^k$ acts without fixed points on $K_0(A)^+ \backslash \{0\}$, as can be seen from the definition of $\rho$. In combination with Lemma \ref{04-08-21x} this shows that the action $\alpha$ has the properties required of $\alpha$ in Theorem \ref{02-12-22}. It follows therefore from that theorem that $B= A \rtimes_\alpha \mathbb Z$ is simple. Assume that $J$ is a non-zero proper ideal in $eBe$. Then $\overline{\Span BJB} = B$ since $B$ is simple. In particular, $e \in \overline{\Span BJB}$ implying that there is a finite sum $\sum_i a_ij_ib_i$, where $a_i,b_i \in B$ and $j_i \in J$ for all $i$, and $\left\| \sum_i a_ij_ib_i -e\right\| < 1$. This implies that $\left\| \sum_i ea_iej_ieb_ie -e\right\| < 1$. Hence $ \sum_i ea_iej_ieb_ie$ is an element of $ J$ which is invertible in $eBe$ implying that $J =eBe$.
 \end{proof}

 The proof of Theorem \ref{12-11-22} is now complete.

 \begin{remark}\label{26-12-22}
  \textnormal{The flows without KMS states or weights are not represented in the examples of Theorem \ref{12-11-22} because we needed $S$ to be non-empty in order to construct the dimension group $(G,G^+)$. It is not clear to the author if it is impossible to construct an automorphism of an AF-algebra such that the dual flow does not have any KMS weights, but is it nonetheless easy to construct examples of flows without KMS weights based on Theorem \ref{12-11-22}. Indeed, the theorem provides a wealth of simple $C^*$-algebras of the form $e(A \rtimes_\alpha \mathbb Z)e$ for which the dual flow does not have any $0$-KMS states. Since the dual flow is periodic it is easy to see that $e(A \rtimes_\alpha \mathbb Z)e$ does not have any trace states at all. In that case the trivial flow on $e(A \rtimes_\alpha \mathbb Z)e$ will have no KMS states and likewise the trivial flow on $A \rtimes_\alpha \mathbb Z$ will have no KMS weights.}
\end{remark}

\begin{notes}\label{29-12-22} The material of this section is an elaboration of material from \cite{ET} and \cite{EST}.  It should be noted that the results in these papers are much stronger than Theorem \ref{12-11-22}, and that the basic ideas go back to \cite{BEH}, \cite{BEK1} and \cite{BEK2}.

\end{notes}

\section{Examples of KMS bundles}

Instead of presenting a small sample of proper simplex bundles we provide a method by which the reader can produce a wealth of examples himself.

Let $X$ be a locally compact second countable Hausdorff space and $f: X \to \mathbb R$ a continuous and proper map. Let $\mathcal M$ denote set of regular Borel probability measures\footnote{A bounded Borel measure on a locally compact second countable Hausdorff space is automatically regular, but we don't want to go into this technicality.} $m$ on $X$ with the property that $m \circ f^{-1}$ is concentrated on a single point $F(m)$ of $\mathbb R$. In symbols,
$$
m \circ f^{-1} = \delta_{F(m)} 
.
$$
Note that
\begin{equation}\label{01-11-23a}
\delta_x \circ f^{-1} = \delta_{f(x)}, 
\end{equation}
showing that $\delta_x \in \mathcal M$ and $F(\delta_x) = f(x)$ for all $x \in X$.
For each $g \in C_c(X)$ we define a function on $\mathcal M$ by the formula
\begin{equation}\label{28-12-22a}
\mathcal M \ni m \mapsto \int_X g \ \mathrm d m ,
\end{equation}
and we equip $\mathcal M$ with the weakest topology making these maps continuous.

\begin{prop}\label{28-12-22} Then
\begin{itemize}
\item[(a)] $(\mathcal M,F)$ is a proper simplex bundle,
\item[(b)] $F^{-1}(t) \neq \emptyset$ if and only if $t \in f(X)$, and
\item[(c)] $F^{-1}(t)$ consists of the Borel probability measures on $X$ that are concentrated on $f^{-1}(t)$ for $t \in f(X)$.
\end{itemize}
\end{prop}
\begin{proof} (c) is obvious. (b): Let $t \in \mathbb R$. Then $F^{-1}(t) \neq \emptyset $ implies that there is an $m \in \mathcal M$ such that $m \circ f^{-1} = \delta_t$. Then $m(f^{-1}(t)) = 1$ and hence $f^{-1}(t) \neq \emptyset$. Thus $t \in f(X)$. Conversely, when $t = f(x)$ it follows from \eqref{01-11-23a} that $\delta_x \in F^{-1}(t)$. (a): The topology is Hausdorff because $C_c(X)$ separates the regular Borel measures of $X$. It is second countable because the topology of $X$ is second countable so that $C_c(X)$ is separable in the topology of uniform convergence on compact subsets. Let's show that $F : \mathcal M \to \mathbb R$ is continuous at $m \in \mathcal M$. Let $\epsilon > 0$ and let $h : \mathbb R \to [0,1]$ be a continuous function such that $h(t) > 0$ iff $t \in \left]F(m)-\epsilon, F(m) + \epsilon\right[$. For $m' \in \mathcal M$, $|F(m')-F(m)| < \epsilon$ iff $\int_X h \circ f \ \mathrm d m' > 0$. This is an open subset of $\mathcal M$ because $h \circ f \in C_c(X)$, and we have therefore shown that $F : \mathcal M \to \mathbb R$ is continuous. Let $m \in \mathcal M$ and let $\epsilon > 0$. To show that $\mathcal M$ is locally compact and $F$ proper we show
\begin{obs}\label{01-11-23c} Let $K \subseteq \mathbb R$ be a compact subset of the real numbers. Then $F^{-1}(K)$ is compact in $\mathcal M$.
\end{obs}
 Since the topology of $\mathcal M$ is second countable, to establish Observation \ref{01-11-23c} it suffices to show that any sequence $\{m_i\}_{i \in \mathbb N}$ in $F^{-1}(K)$ has a convergent subsequence. For this we can assume that $\lim_{i \to \infty}  F(m_i)$ exists; say, $t:= \lim_{i \to \infty}  F(m_i)$. Set $A:=  f^{-1}(K)$. Note that $t \in K$ and 
$m_i\left(A\right) =1$
for all $i$. Since $f$ is proper $A$ is a compact subset of $X$ and there is therefore a Borel probability measure $\nu$ on $A$ such that, after replacing $\{m_i\}_{i \in \mathbb N}$ by a subsequence, 
$$
\lim_{i \to \infty} \int_{A} \psi \ \mathrm d m_i = \int_{A} \psi \ \mathrm d \nu
$$
for all $\psi \in C\left(A\right)$. Define a Borel probability measure $m'$ on $X$ such that
$$
m'(B) := \nu(B \cap A) .
$$
Since each $m_i$ as well as $m'$ are concentrated on $A$ it follows that 
\begin{equation}\label{01-11-23}
\lim_{i \to \infty} \int_{X} h \ \mathrm d m_i =\lim_{i \to \infty} \int_{A} h|_A \ \mathrm d m_i = \int_{A} h|_A \ \mathrm d m' = \int_X h \ \mathrm d m'
\end{equation}
for all $h \in C_c(X)$. We claim that $m'$ is concentrated on $f^{-1}(t)$. To see this let $\psi$ be any non-negative element of $C_c(\mathbb R \backslash \{t\})$. Since $\lim_{i \to \infty} F(m_i) = t$ there is an $n_0$ such that $F(m_i) \notin \supp \psi$ when $i \geq n_0$ and hence 
$$
\int_\mathbb R \psi \ \mathrm dm_i \circ f^{-1} = 0
$$
for all $i \geq n_0$. It follows that $\int_\mathbb R \psi \ \mathrm dm' \circ f^{-1} = 0$. Since $\psi$ was arbitrary it follows that $m'$ is concentrated on $f^{-1}(t)$. Hence $m' \in \mathcal M$ and $F(m') = t$. It follows from \eqref{01-11-23} that $\lim_i m_i = m'$ in $\mathcal M$.

It follows immediately from Observation \ref{01-11-23c} that $F$ is proper, but also that $\mathcal M$ is locally compact. Indeed, for any $m \in \mathcal M$ the set 
$$
F^{-1}\left(]F(m)-1,F(m)+1\right[)
$$ 
is an open neighborhood of $m$ whose closure is contained in 
$$
F^{-1}\left([F(m)-1,F(m)+1\right])
$$ 
and therefore compact by Observation \ref{01-11-23c}. We conclude therefore that $\mathcal M$ is a second countable locally compact Hausdorff space and that $F : \mathcal M \to \mathbb R$ is both continuous and proper. Since the Borel probablity measures on $X$ that are concentrated on $f^{-1}(t)$ is a Choquet simplex affinely homeomorphic to the Bauer simplex of Borel probability measures on $f^{-1}(t)$ it follows from (b) and (c) that $(\mathcal M,F)$ is a simplex bundle. To complete the proof of (a) and hence of the proposition, it remains only to prove that $\mathcal A(\mathcal M,F)$ separates the points of $\mathcal M$, and this follows from the observation that the functions on $\mathcal M$ given by \eqref{28-12-22a} (with $g$ real-valued) are all elements of $\mathcal A(\mathcal M,F)$.

\end{proof}

With the aid of Proposition \ref{28-12-22} it is easy to construct a wealth of proper simplex bundles and hence examples of KMS bundles, for example as follows. Let $X$ be closed subset of $\mathbb C$ such that
$$
\sup \left\{ \left|\Imag z\right| : \ z \in X\right\} < \infty ,
$$
and let $f : X \to \mathbb R$ be the function $f(z) = \Real z$. By combining Proposition \ref{28-12-22} with Theorem \ref{12-11-22} it follows that there is a periodic flow on a simple separable unital $C^*$-algebra for which there is a $\beta$-KMS state if and only if $\beta \in \Real (X)$, and such that for each $\beta \in \Real (X)$ the simplex of $\beta$-KMS states is affinely homeomorphic to the simplex of Borel probability measures on 
$$
\left\{ z \in X: \ \Real z = \beta \right\} .
$$
In particular, for any closed subset $F$ of $\mathbb R$ we can set $X =F$ in this construction to obtain a simple separable unital $C^*$-algebra for which there is a $\beta$-KMS state if and only if $\beta \in F$, and for each $\beta \in F$ the $\beta$-KMS state is unique. This was the main result in \cite{BEH}.

It is also possible, but a little more cumbersome, to construct a second countable locally compact Hausdorff space $X$ and a continuous proper surjection $f: X \to \mathbb R$ such that $f^{-1}(t)$ is not homeomorphic to $f^{-1}(s)$ when $t \neq s$. One such construction is described in Example \ref{31-10-23x} below. By combining Proposition \ref{28-12-22} with Theorem \ref{12-11-22} it follows that there is a periodic flow $\sigma$ on a simple separable unital $C^*$-algebra for which there is a $\beta$-KMS state for all $\beta \in \mathbb R$, and such that $S^\sigma_\beta $ is not affinely homeomorphic to $S^\sigma_{\beta'}$ when $\beta \neq \beta'$. A more direct construction with a more explicit description of such a flow on UHF algebras was given in \cite{Th2}.

\begin{example}\label{31-10-23x} \rm For each $n \in \mathbb N \backslash \{0\}$ let $S^n$ denote the $n$-sphere; i.e.
$$
S^n = \left\{ (x_i)_{i=1}^{n+1} \in \mathbb R^{n+1}: \ \sum_{i=1}^{n+1} x_i^2 = 1 \right\} .
$$
For each infinite subset $A \subseteq \mathbb N \backslash \{0\}$, let 
$$
X_A : = \bigsqcup_{n \in A} S^n
$$
be the topological disjoint union, which is a locally compact Hausdorff space, and let $X_A^\dagger$ be the one point compactification of $X_A$. Then $X_A^\dagger$ is a compact metrizable space whose connected components are the point at infinity and the sets $S^n, n \in A$. In particular,
$$
Y:= X_{\mathbb N \backslash \{0\}}^\dagger \times \mathbb R
$$
is a second countable locally compact Hausdorff in the product topology. Let $Q$ denote set of pairs $(p,q)$ of rational numbers with $p < q$. Since $Q$ is countable there is an injection
$n : Q \to \mathbb N \backslash \{0\}$. Then
$$
X := Y \backslash \left(\bigcup_{(p,q) \in Q} S^{n(p,q)} \times ]p,q[\right) 
$$
is a closed subset of $Y$. Let $f : X \to \mathbb R$ be the projection to the second coordinate. Then $f$ is a continuous proper map with the property that $f^{-1}(t)$ is a non-empty compact metric space, and $f^{-1}(t)$ is a not homeomorphic to $f^{-1}(s)$ when $t \neq s$ since $f^{-1}(t)$ contains connected components that are not homeomorphic to any of the connected components in $f^{-1}(s)$. 

\end{example}

We invite the reader to use Proposition \ref{28-12-22} and Theorem \ref{12-11-22} to construct flows with other KMS-bundles.

\begin{notes} As alluded to in Notes and remarks \ref{29-12-22} it is possible to elaborate the construction in Section \ref{29-12-22},
and by relying on results from the classification of simple $C^*$-algebras,
 arrange that the examples of proper simplex bundles considered in this section are realized as the KMS bundle of a periodic flow on any given $C^*$-algebra from a very large class of $C^*$-algebras. See \cite{Th3}, \cite{ET} and \cite{EST}.
\end{notes}



\chapter{Factor types of extremal KMS weights}\label{factortypes}

Let $\sigma$ be a flow on the $C^*$-algebra $A$ and let $\psi$ be $\beta$-KMS weight for $\sigma$. If $\beta \neq 0$ and $\psi$ is an extremal $\beta$-KMS weight, as defined in Definition \ref{01-08-23f}, then the von Neumann algebra $\pi_{\psi}(A)''$ is a factor by Corollary \ref{07-03-22}. By work of Murray, von Neumann and Connes the von Neumann algebra factors are divided into the following types:
\begin{align*}
&\text{type} \ I:  \ \ \ \ I_n, \ n \in \mathbb N \cup \{\infty\}. \\
&\text{type} \ II:  \ \ \ \ II_1, \ II_\infty. \\
&\text{type} \ III: \ \ \ \ III_\lambda, \ \lambda \in [0,1].
\end{align*}
See \cite{Co2}, V.1. and V.5. We say that an extremal $\beta$-KMS weight $\psi$ is of \emph{factor type X} when $\pi_\psi(A)''$ is of type X. In this chapter we develop tools for the calculation of the factor types of extremal KMS weights and we end by illustrating an application of the theory to the classification of flows up to cocycle-conjugacy.

\section{The first tools}

\begin{lemma}\label{12-08-23} Let $\psi$ be a $\beta$-KMS weight for $\sigma$. Let $N$ be a von Neumann algebra and $\phi$ a normal faithful semi-finite weight on $N$. Assume that there is a $*$-homomorphism $\pi : A \to N$ such that
\begin{itemize}
\item[(a)] $\pi(A)$ is strongly dense in $N$, and
\item[(b)] $\psi = \phi \circ \pi$.
\end{itemize}
There is central projection $e \in N \cap N'$ such that $\pi_\psi(A)''$ is isomorphic to $Ne$.
\end{lemma}
\begin{proof} Thanks to (b) we can define a linear isometry $W : H_\psi \to H_\phi$ such that $W\Lambda_\psi(a) = \Lambda_\phi(\pi(a))$ for $a \in \mathcal N_\psi$. Let $q = WW^* \in B(H_\phi)$ be the orthogonal projection onto the closure of $\{\Lambda_\phi(\pi(a)): \ a \in A \}$ in $H_\phi$. Note that $\pi_\phi(\pi(a))W = W \pi_\psi(a)$,  and hence $\pi_\phi(\pi(a))q = q\pi_\phi(\pi(a))$ for all $a \in A$. Since $\pi_\phi$ is normal by Theorem 7.5.3 in \cite{KR} and $\pi(A)$ is strongly dense in $N$ by assumption (a), it follows that $q \in  \pi_\phi(N)'$. Since $W\pi_\psi(a)W^* = \pi_\phi(\pi(a))$ on $qH_\phi$ it follows that $W\pi_\psi(A)''W^*|_{qH_\phi}$ is the closure of $\pi_\phi(\pi(A))|_{qH_\phi}$ in $B(qH_\phi)$ in the strong operator topology. It follows from (a) that $\pi_\phi(\pi(A))$ is strongly dense in $\pi_\phi(N)$ which is a von Neumann algebra since $\phi$ is normal. Thus $ x\mapsto WxW^*q$ is an isomorphism of $\pi_\psi(A)''$ onto $\pi_\phi(N)q$. Applying Corollary 2.5.5 of \cite{Pe} to the normal homomorphism $\pi_\phi(N) \to \pi_\phi(N)q$ given by $n \mapsto nq$ we find that $\pi_\psi(A)'' \simeq \pi_\phi(N)q \simeq \pi_\phi(N)e_0$, where $e_0$ is a projection in $\pi_\phi(N) \cap \pi_\phi(N)'$. Since $\pi_\phi : N \to \pi_\phi(N)$ is an isomorphism by Theorem 7.5.3 in \cite{KR} it follows that $\pi_\psi(A)'' \simeq Ne$ where $e := \pi_\phi^{-1}(e_0) \in N \cap N'$.
\end{proof}

\begin{cor}\label{12-08-23a} Let $\psi$ be a $\beta$-KMS weight for $\sigma$. Let $N$ be a von Neumann algebra and $\phi$ a normal faithful semi-finite weight on $N$. Assume that there is a $*$-homomorphism $\pi : A \to N$ such that
\begin{itemize}
\item[(a)] $\pi(A)$ is strongly dense in $N$, 
\item[(b)] $\psi = \phi \circ \pi$, and
\item[(c)] $N$ is a factor.
\end{itemize}
Then $\psi$ is an extremal $\beta$-KMS weight and $\pi_\psi(A)''$ is isomorphic to $N$.
\end{cor}
\begin{proof} It follows from Lemma \ref{12-08-23} that $\pi_\psi(A)''$ is isomorphic to $N$ and then from Corollary \ref{07-03-22}, or Corollary \ref{09-03-22b} when $\beta =0$, that $\psi$ is extremal.
\end{proof}

\begin{example}\label{30-08-23} \textnormal{Let $h = h^* \in M_n(\mathbb C)$ and consider the flow $\sigma_t = \Ad e^{ith}$ on $M_n(\mathbb C)$. By Example \ref{31-07-23} there is a unique $\beta$-KMS state $\omega_\beta$ for $\sigma$ for every $\beta \in \mathbb R$, given by
$$
\omega_\beta(a) := \frac{\Tr_n(e^{-\beta h}a)}{\Tr_n(e^{-\beta h})} .
$$
Taking $N = M_n(\mathbb C)$ and $\pi$ to be the identity map in Corollary \ref{12-08-23a} it follows that $\pi_{\omega_\beta}(M_n(\mathbb C))'' \simeq M_n(\mathbb C)$. Hence $\omega_\beta$ is of factor type $I_n$. In a similar way it follows from a combination of Theorem \ref{12-04-22} in Appendix \ref{compact operators}, Theorem \ref{02-01-22a} and Lemma \ref{12-08-23} above, that all KMS weights of a flow on the $C^*$-algebra of compact operators on an infinite dimensional separable Hilbert space are of factor type $I_\infty$.}
\end{example}

\begin{lemma}\label{13-08-23}
 Let $\psi$ be an extremal $\beta$-KMS weight for $\sigma$, $\beta \neq 0$, and $e \in A$ a projection such that
\begin{itemize}
\item[(a)] $e$ is $\sigma$-invariant and
\item[(b)] $\psi(e) = 1$.
\end{itemize}
Then $\pi_{\psi|_{eAe}}(eAe)''$ is isomorphic $\pi_\psi(e)\pi_\psi(A)''\pi_{\psi}(e)$.
\end{lemma}
 \begin{proof} By Theorem \ref{21-02-22dx} and Lemma \ref{03-03-22fx} there is a faithful normal semi-finite weight $\psi''$ on $\pi_\psi(A)''$ such that $\psi'' \circ \pi_\psi = \psi$ on $A^+$. Then $\psi''(\pi_\psi(e)) = \psi(e)= 1$ and hence $\psi''$ restricts to a faithful state $\omega$ on $ \pi_\psi(e)\pi_\psi(A)''\pi_{\psi}(e)$ such that $\omega \circ \pi_\psi|_{eAe} = \psi|_{eAe}$. Note that $\omega$ is normal since $\psi''$ is. Since $\psi$ is extremal, $\pi_\psi(A)''$ is a factor by Corollary \ref{07-03-22} and hence so is $\pi_\psi(e)\pi_\psi(A)''\pi_{\psi}(e)$, cf. Corollary 9 on page 40 in \cite{To}. Therefore Corollary \ref{12-08-23a} applied with $\psi|_{eAe}$ in the role of $\psi$, $\pi_\psi|_{eAe}$ in the role of $\pi$ and $\omega$ in the role of $\phi$, gives the stated conclusion.
  \end{proof}

We are going to use Corollary \ref{12-08-23a} and Lemma \ref{13-08-23} to find the factor types of the extremal KMS states for the flows constructed for the proof of Theorem \ref{12-11-22}. For this we need to introduce the von Neumann algebra crossed product for a discrete group acting on a von Neumann algebra.

\subsection{On von Neumann algebra crossed products by discrete groups}\label{vNcrossed}

Let $N$ be von Neumann algebra acting on the Hilbert space $\mathbb H$, $G$ a discrete group and $\gamma : G \to \Aut N$ a representation of $G$ by automorphisms of $N$. Define a representation $\pi$ of $N$ on $l^2(G,\mathbb H)$ by 
$$
(\pi(a)\xi)(g) = \gamma_g^{-1}(a)\xi(g), \ \ \xi \in l^2(G,\mathbb H)
$$
and a unitary representation $\lambda$ of $G$ on $l^2(G,\mathbb H)$ by
$$
(\lambda_g\xi)(h) = \xi(g^{-1}h), \ \ \xi \in l^2(G,\mathbb H) .
$$
Then 
and the von Neumann algebra of operators on $L^2(G,\mathbb H)$ generated by $\{\pi(a): \ a \in N\}$ and $\lambda_g, \ g \in G$, is the von Neumann algebra crossed product $N\rtimes_\gamma G$. Note that the elements of the form $\sum_{g \in G} \pi(F(g))\lambda_g$ for some finitely supported function $F : G \to N$ form a $*$-algebra, which is dense in $N \rtimes_\gamma G$ in the $\sigma$-strong* topology, cf. Corollary 2.4.15 in \cite{BR}.

Define an isometry $V : \mathbb H \to l^2(G,\mathbb H)$ such that
$$
(V\eta)(g) = \begin{cases} \eta, \ g = e \\ 0, \ g \neq e .\end{cases}
$$
The dual operator $V^*: l^2(G,\mathbb H) \to \mathbb H$ is given by
$$
V^*\xi = \xi(e) .
$$
If $F : G \to N$ is a finitely supported function,
$$
V^*(\sum_{g \in G} \pi(F(g))\lambda_g)V = F(e),
$$
leading to the conclusion that there is a $\sigma$-weakly continuous map $E_N : N\rtimes_\gamma G \to N$ given by $E_n(x) =V^*xV$ and with the property that 
$$
E_N\left(\sum_{g \in G} \pi(F(g))\lambda_g)\right) = F(e) 
$$
for all such functions $F$. $E_N$ is a conditional expectation, but the only property we need is the following.

\begin{lemma}\label{26-08-23} Let $\phi$ be a faithful normal semi-finite weight on $N$. Then $\phi \circ E_N$ is a faithful normal semi-finite weight on $N \rtimes_\gamma G$.
\end{lemma}
\begin{proof} Since $E_N$ is $\sigma$-weakly continuous it follows that $\phi \circ E_N$ is normal because $\phi$ is. Consider a finitely supported function $F : G \to N$ such that $F(g) \in \mathcal N_\phi$ for all $g \in G$. Since $\phi$ is semi-finite the elements of the form $x =\sum_{g \in G} \pi(F(g))\lambda_g$ for such a function $F$ is $\sigma$-weakly dense in $N \rtimes_\gamma G$, and 
$$
\phi \circ E_N(x^*x) = \sum_{g \in G}\psi(F(g)^*F(g))  < \infty .
$$
It follows that $\phi \circ E_N$ is semi-finite. Let $m \in (N\rtimes_\gamma G)^+$ and assume that $\phi \circ E_N(m) =0$. Since $\phi$ is faithful and $E_N(m) \in N^+$ it follows that $E_N(m) = 0$. This implies that $mV\eta=0$ for all $\eta \in \mathbb H$. Note that $E_N(\lambda_g m \lambda_g^*) = \alpha_g(E_N(m)) = 0$ for all $g \in G$. Hence $m\lambda_g^*V \eta = 0$ for all $\eta \in  \mathbb H$ and all $g \in G$. Since $\left\{\lambda_g^*V\eta: \ g \in G, \ \eta\in \mathbb H\right\}$ spans a dense subspace in $L^2(G,\mathbb H)$ it follows that $m = 0$.
\end{proof}

\section{Factor types of KMS weights for the dual flow}

Let $\alpha \in \Aut A$ be an automorphism of $A$. As in Section \ref{alloccur} we consider here the dual flow $\hat{\alpha} $ on $A \rtimes_\alpha \mathbb Z$; cf. \eqref{07-10-23}. 
Let $\psi$ be a $\beta$-KMS weight for $\hat{\alpha}$ and $\tau_\psi$ the trace on $A$ obtained by restricting $\psi$ to $A$, cf. Lemma \ref{26-10-20}. Since
$$
\left\|\Lambda_{\tau_\psi}( \alpha(a))\right\|^2 = \tau_\psi(\alpha(a^*a)) = e^{-\beta}\tau_\psi(a^*a) =  e^{-\beta}\left\|\Lambda_{\tau_\psi}(a)\right\|^2
$$
we can define a unitary $v$ on $H_{\tau_\psi}$ such that
$$
v\Lambda_{\tau_\psi}(a) := e^{\frac{\beta}{2}} \Lambda_{\tau_\psi}(\alpha(a))
$$
for all $a \in \mathcal N_{\tau_\psi}$. Then 
\begin{equation}\label{06-10-23e}
v \pi_{\tau_\psi}(a)v^* = \pi_{\tau_\psi}(\alpha(a))
\end{equation}
for all $a\in A$ and $\Ad v$ gives us an automorphism $\alpha^{\tau_\psi}$ on $\pi_{\tau_\psi}(A)''$. We form the von Neumann algebra crossed product
$$
\pi_{\tau_\psi}(A)''\rtimes_{\alpha^{\tau_\psi}} \mathbb Z ,
$$
cf. Section \ref{vNcrossed}.

\begin{lemma}\label{28-08-23a} There is a central projection $e$ of $\pi_{\tau_\psi}(A)''\rtimes_{\alpha^{\tau_\psi}} \mathbb Z$ such that
$$
\pi_\psi(A  \rtimes_\alpha \mathbb Z)'' \simeq (\pi_{\tau_\psi}(A)''\rtimes_{\alpha^{\tau_\psi}} \mathbb Z)e .
$$
\end{lemma}
\begin{proof}
 Set $N:= \pi_{\tau_\psi}(A)''$. Note that there is a representation $\widetilde{\pi_{\tau_\psi}} \times \lambda :  A \rtimes_\alpha \mathbb Z \to B(l^2(\mathbb Z,H_{\tau_\psi}))$ defined such that
$$
\widetilde{\pi_{\tau_\psi}} \times \lambda(  \sum_{z \in \mathbb Z} a_zu^z) = \sum_{z \in \mathbb Z} \widetilde{\pi_{\tau_\psi}}(a_z) \lambda_z ,
$$
cf. 7.7.1 in \cite{Pe}, and that $\widetilde{\pi_{\tau_\psi}} \times \lambda (  A \rtimes_\alpha \mathbb Z) $ is strongly dense in $ N\rtimes_{\alpha^{\tau_\psi}} \mathbb Z$. Let $\pi : N \to N \rtimes_{\alpha^{\tau_\psi}} \mathbb Z$ be the canonical embedding. We claim that
\begin{equation}\label{28-08-23bx}
\psi (a) = {\tau_\psi}''\circ E_{N}\circ \left(\widetilde{\pi_\tau} \times \lambda\right) (a) 
\end{equation}
for all $a \in A \rtimes_\alpha \mathbb Z$, where ${\tau_\psi}''$ is the normal faithful semi-finite weight from Section \ref{modular2} obtained by considering $\tau_\psi$ as a KMS weight for the trivial action on $A$. To see this, let $F : \mathbb Z \to A$ be finitely supported. Since $ \widetilde{\pi_{\tau_\psi}}  = \pi \circ \pi_{\tau_\psi}$ we find that
\begin{align*}
& E_{N}\circ(\widetilde{\pi_{\tau_\psi}} \times \lambda) ( \sum_{z \in \mathbb Z} F(z)u^z) = E_{N}(\sum_{z \in \mathbb Z} \widetilde{\pi_{\tau_\psi}}(F(z))\lambda_z)\\
&  = E_{N}(\sum_{z \in \mathbb Z} \pi\circ \pi_{\tau_\psi}(F(z))\lambda_z) = \pi_{\tau_\psi}(F(e)) \\
& = \pi_{\tau_\psi}\circ P( \sum_{z \in \mathbb Z} F(z)u^z) 
\end{align*}
where $P : A \rtimes_\alpha \mathbb Z \to A$ is the canonical conditional expectation. Hence
$$
{\tau_\psi}''\circ E_{N}\circ \left(\widetilde{\pi_\tau} \times \lambda\right)(a) = {\tau_\psi}''  \circ \pi_{\tau_\psi}\circ P (a) = \tau_{\psi}\circ P(a)
$$
for all $a \in A \rtimes_\alpha \mathbb Z$, where we have used Lemma \ref{03-03-22fx} for the last identity. Since $\tau_\psi \circ P =\psi$ by Lemma \ref{26-10-20}, we obtain \eqref{28-08-23bx}. Since ${\tau_\psi}''$ is a normal faithful semi-finite weight it follows from Lemma \ref{26-08-23} that ${\tau_\psi}'' \circ E_{N}$ is a normal faithful semi-finite weight on $N\rtimes_{\alpha^{\tau_\psi}} \mathbb Z$, and therefore also from Lemma \ref{12-08-23} that there is a central projection $e$ in $N\rtimes_{\alpha^{\tau_\psi}} \mathbb Z$ such that
$\pi_\psi(A  \rtimes_\alpha \mathbb Z)'' \simeq (N\rtimes_{\alpha^{\tau_\psi}} \mathbb Z)e$.
\end{proof}

\begin{cor}\label{28-08-23b} Assume $\psi$ is a $\beta$-KMS weight for $\hat{\alpha}$ such that $\pi_{\tau_\psi}(A)''\rtimes_{\alpha^{\tau_\psi}} \mathbb Z$ is a factor. Then $\psi$ is extremal and
$$
\pi_\psi(A  \rtimes_\alpha \mathbb Z)'' \simeq \pi_{\tau_\psi}(A)''\rtimes_{\alpha^{\tau_\psi}} \mathbb Z .
$$
\end{cor}
\begin{proof} It follows from Lemma \ref{28-08-23a} that $ \pi_\psi(A  \rtimes_\alpha \mathbb Z)'' \simeq \pi_{\tau_\psi}(A)''\rtimes_{\alpha^{\tau_\psi}} \mathbb Z$. In particular, $\pi_\psi(A  \rtimes_\alpha \mathbb Z)''$ is a factor and it follows therefore from Corollary \ref{07-03-22}, or Corollary \ref{09-03-22b} if $\beta =0$, that $\psi$ is extremal.
\end{proof}

\begin{prop}\label{30-08-23x} Let $\alpha\in \Aut A$ and $\hat{\alpha}$ the dual flow on $A \rtimes_\alpha \mathbb Z$. Let $\psi$ be a $\beta$-KMS weight for $\hat{\alpha}, \ \beta \neq 0$, and assume that $\psi|_A$ is an extremal lower semi-continuous trace on $A$. Then $\psi$ is extremal of factor type $III_{e^{-|\beta|}}$.
\end{prop} 
\begin{proof} We consider first the case $\beta > 0$. Since $\tau_\psi: =\psi|_A$ is extremal, $\pi_{\tau_\psi}(A)''$ is a factor by Corollary \ref{07-03-22} applied with $\sigma$ the trivial flow; necessarily a semi-finite factor since $\tau_\psi''$ is a faithful normal semi-finite trace on $\pi_{\tau_\psi}(A)''$ by Lemma \ref{06-10-23c}.  It follows from \eqref{06-10-23e} that $\alpha^{\tau_\psi} \circ \pi_{\tau_\psi} = \pi_{\tau_\psi} \circ \alpha$ and from Lemma \ref{06-10-23c} that $\tau''_\psi \circ \pi_{\tau_\psi} = \tau_\psi$. Combined with Lemma \ref{26-10-20} this give we $\tau_\psi'' \circ \alpha^{\tau_\psi} \circ \pi_{\tau_\psi} = \tau_\psi \circ \alpha = e^{-\beta} \tau_\psi =  e^{-\beta} \tau_\psi'' \circ \pi_{\tau_\psi}$. It follows therefore from Theorem \ref{05-04-22}, applied with the trivial flow as $\sigma$, that 
\begin{equation}\label{11-09-23c}
\tau_\psi'' \circ \alpha^{\tau_\psi} = e^{-\beta} \tau_\psi''.
\end{equation} 
This implies that $\pi_{\tau_\psi}(A)''$ is a $II_\infty$ factor. Indeed, if $\pi_{\tau_\psi}(A)''$ was of type $I$ it would contain minimal projections, all of positive finite trace. These projections would have the same trace, implying that if $e$ is one of them, $ \tau_\psi'' \circ \alpha^{\tau_\psi}(e) =\tau_\psi''(e)$, contradicting \eqref{11-09-23c}. Hence $\pi_{\tau_\psi}(A)''$ is a $II_\infty$ factor as claimed. Assuming that $\beta > 0$ it follows then from Th\'eor\`eme 4.4.1 (a) of \cite{Co1} that $\pi_{\tau_\psi}(A)''\rtimes_{\alpha^{\tau_\psi}} \mathbb Z$ is a factor of type $III_{e^{-\beta}}$. By Corollary \ref{28-08-23b} so is $\pi_\psi(A  \rtimes_\alpha \mathbb Z)''$. For $\beta < 0$ we use that 
$$
\tau_\psi'' \circ (\alpha^{\tau_\psi})^{-1} = e^{\beta} \tau_\psi''
$$
to conclude from \cite{Co1} that
$$
\pi_{\tau_\psi}(A)''\rtimes_{(\alpha^{\tau_\psi})^{-1}} \mathbb Z
$$
is a factor of type $III_{e^\beta}$. Let $S$ be the unitary on $l^2(\mathbb Z,H_{\tau_\psi})$ given by $S\xi(z) = \xi(-z)$, and note that
$$
S \left( \pi_{\tau_\psi}(A)''\rtimes_{(\alpha^{\tau_\psi})^{-1}} \mathbb Z\right)S^* = \pi_{\tau_\psi}(A)''\rtimes_{\alpha^{\tau_\psi}} \mathbb Z .
$$
This shows that 
$$
\pi_{\tau_\psi}(A)''\rtimes_{(\alpha^{\tau_\psi})^{-1}}  \mathbb Z \ \simeq \ \pi_{\tau_\psi}(A)''\rtimes_{\alpha^{\tau_\psi}} \mathbb Z  ,
$$
and we conclude from Corollary \ref{28-08-23b} that $\pi_\psi(A  \rtimes_\alpha \mathbb Z)''$ is a factor of type $III_{e^\beta} = III_{e^{-|\beta|}}$.

\end{proof}

\begin{cor}\label{30-08-23a} In the setting of Proposition \ref{30-08-23x} let $e \in A$ be a projection such that $\psi(e) =1$. Then the $\beta$-KMS state $\psi|_{e(A \rtimes_\alpha \mathbb Z)e}$ for the flow $\hat{\alpha}^e$ obtained by restriction $\hat{\alpha}$ to $e(A \rtimes_\alpha \mathbb Z)e$ is of factor type $III_{e^{-|\beta|}}$.
\end{cor}
\begin{proof} By Lemma \ref{13-08-23} $\pi_{\psi|_{e(A \rtimes_\alpha \mathbb Z)e}}(e(A\rtimes_\alpha \mathbb Z)e)''$ is isomorphic to 
$$
\pi_\psi(e)\pi_\psi(A \rtimes_\alpha \mathbb Z)''\pi_\psi(e). 
$$
Hence the conclusion follows by combining Proposition \ref{30-08-23x} with Corollaire 3.2.8 (b) of \cite{Co1}: A corner in a factor of type $III_{e^{-|\beta|}}$ is itself a factor of type $III_{e^{-|\beta|}}$.
\end{proof}

\begin{cor}\label{11-09-23} For the flow on $e(A \rtimes_\alpha \mathbb Z)e$ in Theorem \ref{12-11-22} all extremal $\beta$-KMS states are of factor type $III_{e^{-|\beta|}}$ when $\beta \neq 0$.
\end{cor}
\begin{proof} By Observation \ref{31-10-23} we can apply Corollary \ref{30-08-23a}.
\end{proof}

\section{Factor types determined by Connes' $\Gamma$-invariant} 

One of the important tools for the calculation of factor types is the $\Gamma$-invariant of Connes, \cite{Co1}, which we now describe.
 
Let $M$ be a von Neumann algebra and $\alpha = (\alpha_t)_{t \in \mathbb R}$ a normal flow on $M$.  For $f \in L^1(\mathbb R)$ define a linear map
$\alpha_f : M \to M$ such that
$$
\alpha_f(a) := \int_{\mathbb R} f(t) \alpha_t(a) \ dt.
$$  
Let $M^\alpha$ be the fixed algebra of $\alpha$. For every central projection $e$ of $M^\alpha$, set
\begin{equation}\label{03-10-23d}
\Sp_\alpha (eMe) := \bigcap_{ f \in L^1(\mathbb R), \ \alpha_f(eMe) = \{0\}} \ Z(f)
\end{equation}
where
$$
Z(f) := \left\{ r \in \mathbb R : \ \int_{\mathbb R} e^{itr} f(t) \ dt
  \ = \ 0 \right\} .
$$
Then the invariant $\Gamma(\alpha)$ introduced by Connes in the Sections 2.1 and 2.2 of \cite{Co1} can be
expressed as the intersection
$$
\Gamma(\alpha) := \bigcap_{e} \Sp_\alpha (eMe),
$$
where we take the intersection over
all non-zero central projections $e$ in $M^\alpha$. $\Gamma(\alpha)$ is a closed subgroup of $\mathbb R$ and it can be turned into an invariant of $M$ by applying it to the modular automorphism group $\sigma^\phi$ of any normal faithful semi-finite weight on $M$, since by Lemme 1.2.2 and Th\'eor\`eme 2.2.4 of \cite{Co1} the group $\Gamma(\sigma^\phi)$ is independent of which such weight $\phi$ we choose. Hence
$$
\Gamma(M):= \Gamma(\sigma^\phi) 
$$
is an isomorphism invariant of $M$. This invariant is most informative for factors of type $III_\lambda, \ \lambda > 0$, since $\Gamma(M) = \{0\}$ when $M$ is semi-finite or of type $III_0$, at least when $M$ acts on a separable Hilbert space. Since $\Gamma(M)$ is a closed subgroup of $\mathbb R$ it must be one of the following
\begin{itemize}
\item[(i)] $\Gamma (M) = \{0\}$,
\item[(ii)] $\Gamma(M) = \mathbb Z t$ for some $t > 0$, or
\item[(iii)] $\Gamma(M) = \mathbb R$.
\end{itemize}  
When $\Gamma(M) = \{0\}$, $M$ is either semi-finite or of type $III_0$ and when $\Gamma(M) = \mathbb R$, $M$ is of type $III_1$. If $\Gamma(M) = \mathbb Z t$ for some $t > 0$, $M$ is of type $III_{\lambda}$, where $\lambda = e^{-t}$. See \cite{Co1}.


\subsection{Factor types of extremal KMS states for certain periodic flows}

The following lemma simplifies the calculation of $\Gamma(\pi_\psi(A)'')$ in many cases.

\begin{lemma}\label{13-03-23} Let $\sigma$ be a periodic flow on the $C^*$-algebra $A$ and $\psi$ as $\beta$-KMS weight for $\sigma$, $\beta \neq 0$. Assume that $\psi|_{A^\sigma}$ is an extremal lower semi-continuous trace on the fixed point algebra $A^\sigma$. Then $\Gamma(\pi_\psi(A)'') = \Sp_\alpha\left( \pi_\psi(A)''\right)$ where $\alpha_t := \sigma''_{-\beta t}$.
\end{lemma}
\begin{proof} This follows from the definitions and Theorem \ref{05-03-22x} if we show that the fixed point algebra $N :=(\pi_\psi(A)'')^\alpha$ of $\alpha$ on $\pi_\psi(A)''$ is a factor. Set $\tau_\psi = \psi|_{A^\sigma}$. Consider a projection $e$ in the centre of $N$.  Define $\tau_e : (A^\sigma)^+ \to [0,\infty]$ and $\tau_{1-e} : (A^\sigma)^+ \to [0,\infty]$ by
$$
\tau_e(a) := \psi''(e\pi_\psi(a)) 
$$
and
$$
\tau_{1-e}(a) := \psi''((1-e)\pi_\psi(a)) ,
$$
respectively. Then $\tau_e(a^*a) \leq \psi'' \circ \pi_\psi(a^*a) = \psi(a^*a)$ for all $a \in A^\sigma$ by Lemma \ref{03-03-22fx}. Hence $\tau_e$ is densely defined since $\psi$ is. Since $\psi''$ is lower semi-continuous so is $\tau_e$. Since $e$ is central in $N$ and is fixed by $\sigma''$ it follows from Lemma \ref{02-03-22g} that
\begin{align*}
&\tau_e(a^*a) = \psi''(\pi_\psi(a)^*e\pi_\psi(a)) = \psi''(e\pi_\psi(a)\pi_\psi(a)^*e) = \tau_e(aa^*) 
\end{align*}
for all $a \in A^\sigma$. Thus $\tau_e$ is a trace on $A^\sigma$ if it is not zero. Since it is dominated by $\psi$ on $A^\sigma$ it follows that there is a scalar $c_e\geq 0$ such that $\tau_e = c_e\tau_\psi$. Similarly, there is a scalar $c_{1-e} \geq 0$ such that $\tau_{1-e} = c_{1-e} \tau_\psi$. 

\begin{obs}\label{07-10-23b} $\pi_\psi(A^\sigma \cap \mathcal M_\psi)$ is $\sigma$-weakly dense in $N$.
\end{obs}

Before we prove Observation \ref{07-10-23b}, let us see how it leads to the desired conclusion: Assume $c_e \neq 0$ and $c_{1-e} \neq 0$.
Let $d \in A^\sigma \cap \mathcal M_\psi$. Then $N \ni x \mapsto \psi''(e\pi_\psi(d)x\pi_\psi(d^*))$ is a normal linear functional on $N$ (of norm $\leq c_e \psi(dd^*)$). By Observation \ref{07-10-23b} there is a net $\{a_i\}$ in $A^\sigma \cap \mathcal M_\psi$ such that $\lim_{i \to \infty} \pi_\psi(a_i) = 1-e$ in the $\sigma$-weak topology. It follows that,
\begin{align*}
&0 = \psi''(e\pi_\psi(d)(1-e)\pi_\psi(d^*)) = \lim_{i \to \infty}  \psi''(e\pi_\psi(d)\pi_{\psi}(a_i)\pi_\psi(d^*))\\
& = \lim_{i \to \infty}\tau_e(da_id^*)= \frac{c_{e}}{c_{1-e}}  \lim_{i \to \infty}\tau_{1-e}(da_id^*) \\
& =  \frac{c_{e}}{c_{1-e}}  \psi''((1-e) \pi_\psi(d)(1-e)\pi_{\psi}(d^*)) =\frac{c_{e}}{c_{1-e}}  \psi''((1-e) \pi_\psi(dd^*)).
\end{align*}
This equation and the faithfulness of $\psi''$ implies that $(1-e)\pi_\psi(dd^*) = 0$. Since $d \in A^\sigma \cap \mathcal M_\psi$ was arbitrary it follows from Observation \ref{07-10-23b} that $1-e =0$. By symmetry we find also that $e = 0$, a contradiction. Hence either $c_e = 0$ or $c_{1-e} = 0$. Note that $c_e = c_{1-e} =0$ is impossible since $\tau_e(a) + \tau_{1-e}(a) = \psi''(\pi_\psi(a)) = \psi(a)$ for $a \in (A^{\sigma})^+$ and $\psi|_{A^\sigma}$ is not zero. If $c_e \neq 0$ we have that $c_{1-e} = 0$ and hence $\tau_{1-e} =0$. The faithfulness of $\psi''$ implies then that $(1-e)\pi_\psi(a) = 0$ for all $a \in A^\sigma$ and Observation \ref{07-10-23b} gives the conclusion that $e=1$. Likewise the assumption $c_{1-e} \neq 0$ gives that $e =0$. It follows that $0$ and $1$ are the only central projections in $N$; i.e. $N$ is a factor.

 Now let's establish the observation: Let $x \in N$. Since $\pi_\psi(A)$ is $\sigma$-weakly dense in $\pi_\psi(A)''$ and $\mathcal M_\psi$ is norm-dense in $A$ there is a net $\{b_i\}$ in $\mathcal M_\psi$ such that $\lim_{i \to \infty} \pi_\psi(b_i) = x$ in the $\sigma$-weak topology. Let $Q_0 : A \to A^\sigma$ be the conditional expectation given by
 $$
 Q_0(a) = \frac{1}{p}\int_0^p \sigma_t(a) \ \mathrm d t ,
 $$
 where $p$ is the period of $\sigma$, cf. Section \ref{periodicflows}. Let $a \in \mathcal M_\psi^+$. An application of Theorem \ref{09-11-21h} and Lebesques theorem on monotone convergence shows that
 $$
 \psi\left(Q_0(a)\right) = \frac{1}{p}\int_0^p \psi(\sigma_t(a)) \ \mathrm d t = \psi(a) ,
 $$
 implying that $Q_0(b_i) \in \mathcal M_\psi$ for all $i$. Since $Q_0(b_i) \in A^\sigma \cap \mathcal M_\psi$ it suffices to show that $\lim_{i \to \infty} \pi_\psi(Q_0(b_i)) = x$ in the $\sigma$-topology. For this let $\omega$ be a normal functional. Then
 $$
 \omega(\pi_\psi(Q_0(b_i))) = \frac{1}{p}\int_0^p \omega\left(\sigma''_t(\pi_\psi(b_i))\right) \ \mathrm d t 
 $$
 which converges to $
  \frac{1}{p}\int_0^p \omega\left(\sigma''_t(x)\right) \ \mathrm d t$. Since $\sigma''_t = \alpha_{- \frac{t}{\beta}}$, we have that
  $$\frac{1}{p}\int_0^p \omega\left(\sigma''_t(x)\right) \ \mathrm d t = \omega(x),
  $$ 
  and we conclude that $\lim_{i \to \infty} \omega(\pi_\psi(Q_0(b_i))) = \omega(x)$. Since $\omega$ was arbitrary this means that $\lim_{i \to \infty} \pi_\psi(Q_0(b_i)) = x$ in the $\sigma$-topology.
\end{proof}

\begin{example}\label{03-10-23} \rm (Example \ref{27-02-23} continued.) The flow $\sigma^\rho$ from Example \ref{27-02-23} is periodic and has a fixed point algebra which is a UHF algebra when $\rho \neq 0$. Since the trace state of a UHF algebra is unique, Lemma \ref{13-03-23} applies to show that for the $\frac{\log 2}{\rho}$-KMS state $\psi$ for $\sigma^\rho$, the $\Gamma$-invariant of $\pi_\psi(A)''$ is 
$$
\Gamma(\pi_\psi(A)'') = \Sp_{\alpha}(\pi_\psi(A)''),
$$
where $\alpha_t = {\sigma^\rho}''_{-\beta t}$ and $\beta =\frac{\log 2 }{\rho}$. We seek therefore to calculate $\Sp_{\alpha}(\pi_\psi(A)'')$.

It follows from Observation \ref{27-02-23a} and the description of the eigenspaces $V(k)$, given in Example \ref{27-02-23} that $\bigcup_{k \in \mathbb Z} V(k)$ spans a norm dense subspace of $A$ and it follows therefore that
$\bigcup_{k \in \mathbb Z} \pi_\psi(V(k))$
spans a $\sigma$-weakly dense subspace of $\pi_\psi(A)''$. It follows that when $f \in L^1(\mathbb R)$ the condition
$$
\alpha_f(\pi_\psi(A)'') = \{0\}
$$
will hold if and only if 
$$
\int_\mathbb R f(t){\sigma^\rho}''_{-\beta t}(x) \ \mathrm d t = 0
$$
for all $x \in \pi_\psi(V(k))$ and all $k \in \mathbb Z$. If $y := V_aV_b^*$ where $a \in \{0,1\}^{n+k}$ and $b \in \{0,1\}^n$ where $n,k \geq 0$, we find that
$$
{\sigma^\rho}''_{- \beta t} \circ \pi_\psi (y) = \pi_\psi \circ \sigma^\rho_{-\beta t}(y) = e^{-ik\beta \rho t}\pi_{\psi}(y) .
$$
Note that $\pi_\psi(y) \neq 0$ since $\pi_\psi(V_a^*)\pi_\psi(y)\pi_\psi(V_b) = \pi_\psi(1) = 1$. Hence 
$$
\alpha_f(\pi_\psi(V(k))) = 0
$$
if and only if $\int_\mathbb R f(t) e^{-ik\beta \rho t}  \ \mathrm d t = 0$. A similar argument applies when $k < 0$ and we find therefore that $\alpha_f(\pi_\psi(A)'') = \{0\}$ if and only if $\hat{f}(-k\beta \rho) = 0$ for all $k \in \mathbb Z$, where $\hat{f}$ is the Fourier transform of $f$.
Thus
$$
\mathbb Z \beta \rho \subseteq Z(f)
$$
for all $f \in L^1(\mathbb R)$ such that $\alpha_f(\pi_\psi(A)'') = \{0\}$. If $t \in \mathbb R \backslash \mathbb Z \beta \rho$ there is a $g \in L^1(\mathbb R)$ such that $\hat{g}(\mathbb Z \beta \rho ) = \{0\}$ and $\hat{g}(t) =1$, cf. Theorem 2.6.2 in \cite{Ru3}. Then $\alpha_g(\pi_\psi(A)'') = \{0\}$, but $t \notin Z(g)$. It follows therefore from \eqref{03-10-23d} that $\Sp_{\alpha}(\pi_\psi(A)'') =\mathbb Z \beta \rho = \mathbb Z \log 2$.
It follows that $\psi$ is of type $III_{\frac{1}{2}}$, regardless of which $\rho \neq 0$ we consider; a fact which is not surprising because the KMS state is the same for all $\rho \neq 0$.

\end{example}

\subsection{On the factor type of the KMS states of an ITPFI flow}

Let $m \in \mathbb N, \ m \geq 2$, and let $h =h^*\in M_m(\mathbb C)$ be a selfadjoint $m \times m$ matrix. We denote in the following the tensor product $M_m(\mathbb C) \otimes M_m(\mathbb C) \otimes \cdots \otimes M_m(\mathbb C)$ with $k$ tensor factors by
$$
M_m(\mathbb C)^{\otimes k} .
$$
As in Section \ref{ITPFI} we can consider the UHF algebra $\otimes_{k=1}^\infty M_m(\mathbb C)$ and the flow $\sigma$ defined on $\otimes_{k=1}^\infty M_m(\mathbb C)$ defined such that
\begin{equation}\label{28-09-23}
\sigma_t \circ \phi_{k,\infty} = \phi_{k,\infty} \circ \sigma^k_t,
\end{equation}
where $\phi_{k,\infty}: M_m(\mathbb C)^{\otimes k} \to \otimes_{k=1}^\infty M_m(\mathbb C)$ is  the canonical map and $\sigma^k$ is the flow on $M_m(\mathbb C)^{\otimes k}$ defined such that
$$
\sigma^k_t(a_1\otimes \cdots \otimes a_k) = e^{ith}a_1e^{-ith} \otimes  e^{ith}a_2e^{-ith}\otimes \cdots \otimes  e^{ith}a_ke^{-ith}
$$
on simple tensors.
The flow $\sigma$ is trivial when $h$ is a scalar multiple of $1$. When this is not the case, i.e. $h \notin \mathbb R 1$, we say that $\sigma$ is a \emph{stationary ITPFI flow} generated by $h$.

As shown in Section \ref{ITPFI} there is a unique $\beta$-KMS state $\omega_\beta$ for $\sigma$ for each $\beta \in \mathbb R$. We seek here to determine the factor type of these KMS states by calculating the $\Gamma$-invariant of Connes for the factors 
$$
R_\beta := \pi_{\omega_\beta}(\otimes_{k=1}^\infty M_m(\mathbb C))''.
$$

\begin{lemma}\label{05-10-23} The fixed point algebra $R_\beta^{\sigma''}$ for the flow $\sigma''$ on $R_\beta$ is a factor.
\end{lemma}
\begin{proof} 


Choose an orthonormal basis $\{\xi_i\}_{i=1}^m$ for $\mathbb C^m$ consisting of eigenvectors for $h$. For each $n \in \mathbb N$ define a unitary $V_n \in M_m(\mathbb C)^{2n}$ such that
$$
V_n(\xi_{i_1} \otimes \xi_{i_2} \otimes \cdots \otimes \xi_{i_{2n}}) = \xi_{i_{n+1}} \otimes \xi_{i_{n+2}} \otimes \cdots \ \otimes \xi_{i_{2n}} \otimes \xi_{i_1} \otimes \cdots \otimes \xi_{i_n} 
$$
for $(i_1,i_2, \cdots , i_{2n}) \in \{1,2,\cdots, m\}^{2n}$.
Note that 
\begin{equation}\label{08-12-23}
\Ad V_n(a_1\otimes a_2 \otimes \cdots \otimes a_{2n}) = a_{n+1} \otimes a_{n+2}\otimes  \cdots \otimes a_{2n} \otimes a_1 \otimes \cdots \otimes a_n
\end{equation}
for $a_1,a_2, \cdots, a_n \in M_m(\mathbb C)$ and that $\Ad V_n \circ \sigma^{2n}_t = \sigma^{2n}_t \circ \Ad V_n$.
 Set $U_n := \phi_{2n,\infty}(V_n)$. The decisive properties of $U_n$ are that $U_n$ is $\sigma$-invariant, and
\begin{equation}\label{05-10-23c}
U_nxU_n^*y = y U_nxU_n^*
\end{equation}
and
\begin{equation}\label{05-10-23d}
\omega_\beta(U_n x U_n^* y) = \omega_\beta(x)\omega_\beta(y)
\end{equation}
for all $x,y \in \phi_{n,\infty}(M_m(\mathbb C)^{\otimes n})$. The first property, \eqref{05-10-23d}, follows from \eqref{08-12-23}, and the second, \eqref{05-10-23d}, follows from the description of $\omega_\beta$ given in Section \ref{ITPFI}. 

Let $x \in \otimes_{k=1}^\infty M_m(\mathbb C)$. We shall need the following  
\begin{obs}\label{05-10-23e} $\lim_{n \to \infty} \pi_{\omega_\beta}(U_nxU_n^*) = \omega_\beta(x)1$ in the weak operator topology.
\end{obs}

To establish this note that since $\bigcup_k \Lambda_{\omega_\beta}(\phi_{k,\infty}((M_m(\mathbb C)^{\otimes k})))$ is dense in $H_{\omega_\beta}$ it suffices to show that
$$
\lim_{n \to \infty} \left<\pi_{\omega_\beta}(U_nxU_n^*)\Lambda_{\omega_\beta}(a), \Lambda_{\omega_\beta}(b)\right> = \omega_\beta(x) \omega_\beta(b^*a)
$$
for $a,b \in \phi_{k,\infty}((M_m(\mathbb C)^{\otimes k}))$.  Similarly, $\bigcup_k \phi_{k,\infty}((M_m(\mathbb C)^k)$ is norm-dense in $\otimes_{k=1}^\infty M_m(\mathbb C)$ so it suffices to do this when $x\in \phi_{l,\infty}(M_m(\mathbb C)^{\otimes l}), \ l \geq k$. Using \eqref{05-10-23c} and \eqref{05-10-23d} we find for $n \geq l$,
\begin{align*}
&\left<\pi_{\omega_\beta}(U_nxU_n^*)\Lambda_{\omega_\beta}(a), \Lambda_{\omega_\beta}(b)\right> = \omega_\beta\left(b^*U_nxU_n^*a\right)\\
& =\omega_\beta\left(U_nxU_n^*b^*a\right) = \omega_\beta(x)\omega_\beta(b^*a),
\end{align*}
proving Observation \ref{05-10-23e}. 

Let $\omega''_\beta$ be the normal extension of $\omega_\beta$ to $R_\beta$. Since $\omega_\beta$ is a $\beta$-KMS state and $U_n$ is $\sigma$-invariant we have that $\omega_\beta(U_nxU_n^*) = \omega_\beta(x)$ for all $x \in\otimes_{k=1}^\infty M_m(\mathbb C)$, implying that $\omega''_\beta \circ \Ad \pi_{\omega_\beta}(U_n) = \omega''_{\omega_\beta}$. Hence when $p$ is a central projection in $R_\beta^{\sigma''}$ and $x \in \otimes_{k=1}^\infty M_m(\mathbb C)$ we find that
\begin{align*}
& \omega''_\beta(p\pi_{\omega_\beta}(x)) = \omega''_\beta\left(\pi_{\omega_\beta}(U_n) p \pi_{\omega_\beta}(x)\pi_{\omega_\beta}(U_n)^*\right)\\
& =  \omega''_\beta\left(p\pi_{\omega_\beta}(U_n)  \pi_{\omega_\beta}(x)\pi_{\omega_\beta}(U_n)^*\right)\\ 
\end{align*} 
for all $n$ since $\pi_{\omega_\beta}(U_n)\in R_\beta^{\sigma''}$. By using that $\omega''_\beta$ is continuous for the weak operator topology on norm-bounded sets, it follows from Observation \ref{05-10-23e} that $\lim_{n \to \infty} \omega''_\beta\left(p\pi_{\omega_\beta}(U_n)  \pi_{\omega_\beta}(x)\pi_{\omega_\beta}(U_n)^*\right) = \omega_\beta(x) \omega''_\beta(p)$. Thus
$$
\omega''_\beta(p\pi_{\omega_\beta}(x)) = \omega_\beta(x) \omega''_\beta(p) .
$$
Since $\pi_{\omega_\beta}(\otimes_{k=1}^\infty M_m(\mathbb C))$ is $\sigma$-weakly dense in $R_\beta$ and $\omega'' \circ \pi_{\omega_\beta} = \omega_\beta$ we conclude that $\omega''_\beta(py) = \omega''_\beta(y) \omega''_\beta(p)$ for all $y \in R_\beta$. Inserting $1-p$ for $y$ we see that $\omega''_\beta(1-p) \omega''_\beta(p) =0$ which implies that $p =0$ or $p=1$ since $\omega''_\beta$ is faithful by Lemma \ref{01-03-22f}. It follows that $R_\beta^{\sigma''}$ is a factor. 
\end{proof}

\begin{cor}\label{05-10-23f} For $\beta\neq 0$, $\Gamma(R_\beta) = \Sp_\alpha(R_\beta)$ where $\alpha_t = \sigma''_{-\beta t}$ for all $t \in \mathbb R$.
\end{cor}
\begin{proof} Thanks to Lemma \ref{05-10-23} this follows from the definition of $\Gamma(R_\beta)$ and Theorem \ref{05-03-22x}.
\end{proof}

\begin{lemma}\label{05-10-23g} $\Gamma(R_\beta) = \beta \overline{G_h}$ where ${G_h}$ is the additive group generated by the numbers 
$$
\{\lambda - \mu :  \ \lambda, \mu \ \text{are eigenvalues of $h$} \}
$$ 
and $\overline{G_h}$ denotes it closure in $\mathbb R$.
\end{lemma}
\begin{proof} $\Gamma(R_\beta) = \Sp_\alpha(R_\beta)$ by Corollary \ref{05-10-23f}. To calculate the spectrum $\Sp_\alpha(R_\beta)$ let $f \in L^1(\mathbb R)$. Let $\{e_{i,j}\}_{i,j =1}^m$ be a set of matrix units in $M_m(\mathbb C)$ such that $e_{ii}$ is the orthogonal projection onto $\xi_i$. Since $\xi_i$ is an eigenvector for $h$ there are real numbers
$$
\lambda_i \in \mathbb R, \ i \in \{1,2,\cdots, m\},
$$
such that $
h\xi_i = \lambda_i \xi_i$. Let $k \in \mathbb N$ and $i_l,j_l \in \{1,2,\cdots , m\}$, $l = 1,2,\cdots, k$. Since
$$
\Ad e^{it h}(e_{i,j}) = e^{it (\lambda_i - \lambda_j)} e_{i,j}
$$
we find that
$$
\sigma^k_t( e_{i_1,j_1}\otimes e_{i_2,j_2} \otimes \cdots \otimes e_{i_k,l_k}) = e^{it \chi} e_{i_1,j_1}\otimes e_{i_2,j_2} \otimes \cdots \otimes e_{i_k,l_k},
$$
where $\chi := \sum_{r=1}^k (\lambda_{i_r} - \lambda_{j_r})$. For any normal functional $\mu$ of $R_\beta$,
\begin{align*}
&\mu\left( \alpha_f (\pi_{\omega_\beta}(\phi_{k,\infty}(  e_{i_1,j_1}\otimes e_{i_2,j_2} \otimes \cdots \otimes e_{i_k,l_k})))\right) \\
&= \int_\mathbb R f(t) \mu \circ \pi_{\omega_\beta}\left( \sigma_{-\beta t} (\phi_{k,\infty}(  e_{i_1,j_1}\otimes e_{i_2,j_2} \otimes \cdots \otimes e_{i_k,l_k})\right) \ \mathrm d t\\
& = \left(\int_\mathbb R f(t) e^{-i \beta \chi t} \ \mathrm d t  \right) \mu \circ \pi_{\omega_\beta}(\phi_{k,\infty}(  e_{i_1,j_1}\otimes e_{i_2,j_2} \otimes \cdots \otimes e_{i_k,l_k})) .
\end{align*}
Note that $\pi_{\omega_\beta}(\phi_{k,\infty}(  e_{i_1,j_1}\otimes e_{i_2,j_2} \otimes \cdots \otimes e_{i_k,l_k})) \neq 0$ since $\pi_{\omega_\beta} \circ \phi_{k,\infty}$ is faithful. Hence $\mu\left( \alpha_f (\pi_{\omega_\beta}(\phi_{k,\infty}(  e_{i_1,j_1}\otimes e_{i_2,j_2} \otimes \cdots \otimes e_{i_k,l_k})))\right) = 0$ if and only if $\hat{f}(-\beta \chi) = \int_\mathbb R f(t) e^{-i \beta \chi t} \ \mathrm d t =0$. Since $\bigcup_k \pi_{\omega_\beta}\circ \phi_{k,\infty}(M_m(\mathbb C)^{\otimes k})$ is $\sigma$-weakly dense in $R_\beta$ it follows that $\alpha_f(R_\beta) = \{0\}$ if and only if
$$
\hat{f}(-\beta \chi) = \int_\mathbb R f(t) e^{-i \beta \chi t} \ \mathrm d t  = 0
$$
for all $\chi \in G_h$. Hence $\beta G_h \subseteq  Z(f)$ when $\alpha_f(R_\beta) = \{0\}$, implying that the additive subgroup $\beta G_h$ of $\mathbb R$ is in $\Sp_\alpha(R_\beta)$. Since $\Sp_\alpha(R_\beta) = \Gamma(R_\beta)$ is a closed subset of $\mathbb R$ it follows that the closure $\beta \overline{G_g}$ of $\beta G_h$ is contained in $\Gamma(R_\beta)$. If $r \in \mathbb R \backslash \beta\overline{G_h}$ there is a function $f \in L^1(\mathbb R)$ such that $\hat{f}(r) \neq 0$ and $\hat{f}(\beta G_h) = \{0\}$; see Theorem 2.6.2 in \cite{Ru3}. By the preceding calculation the last property implies $\alpha_f(R_\beta)= \{0\}$ and then the first implies that $r \notin \Sp_\alpha(R_\beta)$. 
\end{proof}

Note the $G_h$ is dense in $\mathbb R$ if and only if there are eigenvalues $\lambda,\mu,\lambda',\mu'$ for $h$ such that $\lambda' \neq \mu'$ and
$$
\frac{\lambda - \mu}{\lambda'-\mu'}
$$ 
is irrational. When this is not the case there is a $\kappa_h > 0$ such that
$$
G_h = \overline{G_h} = \mathbb Z \kappa_h .
$$

\begin{cor}\label{06-10-23} Let $\sigma$ be a stationary $ITPFI$ flow generated by the selfadjoint matrix $h$, cf. \eqref{28-09-23}.  
\begin{itemize}
\item[$\bullet$] The $\beta$-KMS state for $\sigma$ is of type $III$ for all $\beta \neq 0$.
\item[$\bullet$] When $G_h$ is dense in $\mathbb R$ the $\beta$-KMS state for $\sigma$ is of type $III_1$ for all $\beta \neq 0$.
\item[$\bullet$] When $G_h$ is not dense in $\mathbb R$ and $\beta \neq 0$, the $\beta$-KMS state for $\sigma$ is of type $III_{\lambda_\beta}$ where $\lambda_\beta= e^{-|\beta|\kappa_h}$.
\end{itemize}
\end{cor}

\begin{example}\label{06-10-23a} \rm Consider a stationary $ITPFI$ flow generated by the selfadjoint matrix $h \in M_2(\mathbb C)$. $h$ has two eigenvalue $\lambda_1$ and $\lambda_2$ is this case. Set $\kappa_h = |\lambda_1 - \lambda_2|$. By Corollary \ref{06-10-23} the $\beta$-KMS state for $\sigma$ is of type $III_{\lambda_\beta}$ where $\lambda_\beta = e^{-|\beta|\kappa_h}$ when $\beta \neq 0$.
\end{example}

\section{KMS-invariants for cocycle-conjugacy of flows}

Let $A$ and $B$ be unital $C^*$-algebras, $\sigma$ a flow on $A$ and $\sigma'$ a flow on $B$. $\sigma$ is \emph{conjugate} to $\sigma'$ when there is a $*$-isomorphim $\alpha : B \to A$ such that $\sigma_t \circ \alpha = \alpha \circ \sigma'_t$ for all $t \in \mathbb R$.

As in Section \ref{KMSstates} we denote by $S^\sigma_\beta$ the set of $\beta$-KMS states for $\sigma$.

\begin{lemma}\label{08-10-23e} Let $\sigma$ be a flow on the unital $C^*$-algebra $A$ and $\sigma'$ a flow on the unital $C^*$-algebra $B$. Assume $\alpha : B \to A$ is a $*$-isomorphism such that $\sigma_t \circ \alpha = \alpha \circ \sigma'_t$ for all $t \in \mathbb R$. For $\beta \in \mathbb R$ the map $\psi \mapsto \psi \circ \alpha$ is an affine homeomorphism from $S^\sigma_\beta$ onto $S^{\sigma'}_\beta$, and $\pi_{\psi}(A)''$ is isomorphic to $\pi_{\psi \circ \alpha}(B)''$ for all $\psi \in S^{\sigma}_\beta$.
\end{lemma}
\begin{proof}
It follows from Lemma \ref{01-10-23a} the map $\psi \mapsto \psi \circ \alpha$ is a bijection from $S^{\sigma}_\beta$ onto $S^{\sigma'}_\beta$ and it is clearly a homeomorphism with respect to the weak* topologies. Furthermore, for every $\psi \in S^{\sigma}_\beta$ there is a unitary $W : H_{\psi \circ \alpha} \to H_{\psi}$ defined such that $W\Lambda_{\psi \circ \alpha}(b) = \Lambda_{\psi}(\alpha(b))$. It is easy to check that $W \pi_{\psi \circ \alpha}(b) = \pi_{\psi}(\alpha(b))W$ for all $b \in B$ which implies the claim, $\pi_{\psi}(A)'' \ \simeq \ \pi_{\psi \circ \alpha}(B)''$.
\end{proof} 

Let $A$ be a unital $C^*$-algebra. Let $\sigma$ and $\sigma'$ be flows on $A$. Following \cite{Ki1} we say that $\sigma'$ is a \emph{cocycle perturbation} of $\sigma$ when there is a continuous unitary representation $u = (u_t)_{t \in \mathbb R}$ of $\mathbb R$ in $A$ such that $u_s\sigma_s(u_t) = u_{s+t}$ and $\sigma'_t = \Ad u_t \circ \sigma_t$ for all $s,t \in \mathbb R$. Let $B$ be another unital $C^*$-algebra and let $\sigma''$ be a flow on $B$. Still following \cite{Ki1} we say that $\sigma''$ is \emph{cocycle conjugate} to $\sigma$ when $\sigma''$ is conjugate to a cocycle perturbation of $\sigma$, i.e. if when there is an isomorphism $\phi : B \to A$ and a continuous unitary representation $u = (u_t)_{t \in \mathbb R}$ of $\mathbb R$ in $A$ such that $u_s\sigma_s(u_t) = u_{s+t}$ and $\phi \circ \sigma''_t \circ \phi^{-1} = \Ad u_t \circ \sigma_t$ for all $s,t \in \mathbb R$.

\begin{thm}\label{08-10-23} Let $A$ and $B$ be unital $C^*$-algebras, $\sigma$ a flow on $A$ and $\sigma''$ a flow on $B$. Assume that $\sigma''$ is cocycle conjugate to $\sigma$ and let $\beta \in \mathbb R$. It follows that there is a homeomorphism $\eta : S^{\sigma}_\beta \to S^{\sigma''}_\beta$ such that $\eta(\omega)$ is extremal if and only if $\omega $ is extremal. Furthermore, when $\omega \in S^{\sigma}_\beta$ is extremal there is a von Neumann algebra factor $N$ and projections $e,f \in N$ such that $\pi_\omega(A)'' \simeq eNe$ and $\pi_{\eta(\omega)}(B)'' \simeq fNf$.
\end{thm}
\begin{proof} By assumption there is an isomorphism $\phi : B \to A$ such that $\sigma'_t :=\phi \circ \sigma''_t \circ \phi^{-1}$ is a flow $\sigma'$ on $A$ which is a cocycle perturbation of $\sigma$. By Lemma \ref{08-10-23e} there is an affine homeomorphism $\eta':S^{\sigma'}_\beta \to S^{\sigma''}_\beta$ such that $\pi_{\eta'(\omega)}(B)'' \simeq \pi_{\omega}(A)''$ when $\omega \in S^{\sigma'}_\beta$ is extremal. Since $\sigma'$ is a cocycle perturbation of $\sigma$ there is a continuous unitary representation $u = (u_t)_{t \in \mathbb R}$ of $\mathbb R$ in $A$ such that $u_s\sigma_s(u_t) = u_{s+t}$ and $\sigma'_t = \Ad u_t \circ \sigma_t$ for all $s,t \in \mathbb R$. As in the proof of Theorem \ref{06-06-22b} we consider now the flow $\gamma$ on $M_2(A)$ given by
$$
\gamma_t\left( \begin{smallmatrix} a & b \\ c & d \end{smallmatrix} \right) := \left(\begin{matrix} \sigma_t(a) & \sigma_t(b)u_{-t} \\ u_t \sigma_t(c) & \sigma'_t(d) \end{matrix} \right) 
$$
and the embeddings $\iota_i : A \to M_2(A), \ i =1,2$, of $A$ into the canonical orthogonal corners of $M_2(A)$. Set
$$
p_1 = \left( \begin{smallmatrix} 1 & 0 \\ 0 & 0 \end{smallmatrix} \right) \ \text{and} \ p_2 =  \left( \begin{smallmatrix} 0 & 0 \\ 0 & 1 \end{smallmatrix} \right) .
$$
It follows from Theorem \ref{07-06-22e} that $\psi \mapsto \psi(p_1)^{-1}\psi \circ \iota_1$ is a bijection $\mu_1$ from $S^{\gamma}_\beta$ onto $S^{\sigma}_\beta$ and  $\psi \mapsto \psi(p_2)^{-1}\psi \circ \iota_2$ is a bijection $\mu_2$ from $S^{\sigma}_\beta$ onto $S^{\sigma'}_\beta$. Note also that $\mu_1$ and $\mu_2$ both are continuous for the weak* topologies and preserve extremality. Hence 
$$
\eta'' := \mu_2 \circ \mu_1^{-1} : S^{\sigma}_\beta \to S^{\sigma'}_\beta
$$ 
is a continuous bijection which takes extremal $\beta$-KMS states for $\sigma$ to extremal $\beta$-KMS states for $\sigma'$. Define $W_i : H_{\mu_i(\psi)} \to H_\psi$ such that
$$
W_i \eta_{\mu_i(\psi)}(a) = \psi(p_i)^{-\frac{1}{2}}\Lambda_\psi(\iota_i(a))
$$
for all $a \in A$, and note that $W_i$ is a unitary from $H_{\mu_i(\psi)}$ onto $H^i_\psi$. Then
$$
W_i \pi_{\mu_i(\psi)}(x)\Lambda_{\mu_i(\psi)}(a) = \pi_\psi(\iota_i(x))W_i\Lambda_{\mu_i(\psi)}(a)
$$
for all $x,a \in A$ and hence
$$
W_i \pi_{\mu_i(\psi)}(x) W_i^*|_{H_\psi^i} = \pi_{\psi}(\iota_i(x))|_{H_\psi^i} = \pi_{\psi}(p_i)\pi_{\psi}(\iota_i(x))\pi_{\psi}(p_i)|_{H_\psi^i}
$$
for all $x \in A$, where $H_\psi^i := \overline{\Lambda_{\psi}(\iota_i(A))}$. Since 
$$
p_i\iota_i(A)p_i = p_iM_2(A)p_i,
$$
it follows that $\pi_{\mu_i(\psi)}(A)'' \simeq   \pi_{\psi}(p_i)\pi_{\psi}(A)''\pi_{\psi}(p_i)$. Hence $\eta : = \eta' \circ \eta'': S^{\sigma}_\beta \to S^{\sigma''}_\beta$ is a continuous bijection with the stated properties.

\end{proof}

\begin{cor}\label{09-10-23}  Let $A$ and $B$ be unital $C^*$-algebras, $\sigma$ a flow on $A$ and $\sigma'$ a flow on $B$. Assume that $\sigma$ and $\sigma'$ are cocycle conjugate. There is a homeomorphism $\eta : S^{\sigma}_\beta \to S^{\sigma'}_\beta$ such that $\eta(\omega)$ is extremal if and only if $\omega $ is extremal. 

Let $\omega \in S^{\sigma}_\beta$ be extremal. Then $\omega$ is of factor type $I$ if and only if $\eta(\omega)$ is of factor type $I$, $\omega$ is of factor type $II$ if and only if $\eta(\omega)$ is of factor type $II$ and for $\lambda \in [0,1]$,  $\omega$ is of factor type $III_\lambda$ if and only if $\eta(\omega)$ is of factor type $III_\lambda$.
\end{cor}
\begin{proof} This follows from Theorem \ref{08-10-23} plus the following facts. Let $e$ be a non-zero projection in the von Neumann algebra factor. Then
\begin{itemize}
\item[(i)] $N$ is of type $I$ if and only if $eNe$ is of type $I$,
\item[(ii)] $N$ is of type $II$ if and only if $eNe$ is of type $II$, 
\item[(iii)] $N$ is of type $III_\lambda$ if and only if $eNe$ is of type $III_\lambda$ for any $\lambda \in [0,1]$.
\end{itemize}
The two first facts, (i) and (ii), are well-known, although their proofs are hard to find. In textbooks they are mostly relegated to the exercises. See e.g. E.4.18 in \cite{SZ}. The last fact (iii) follows from Corollaire 3.2.8 (b) of \cite{Co1}.
\end{proof}

\begin{example}\label{09-10-23a} \rm Consider the stationary ITPFI flows generated by two selfadjoint matrices $h,h' \in M_2(\mathbb C)$ with eigenvalues $\lambda_1,\lambda_2$ and $\lambda'_1,\lambda'_2$, respectively. It follows by combining Example \ref{06-10-23a} with Corollary \ref{09-10-23} that these two flows are not cocycle conjugate unless $|\lambda_1 - \lambda_2| = |\lambda'_1 - \lambda'_2|$. Note this holds if and only if there is a real number $r \in \mathbb R$ and a unitary $W \in M_2(\mathbb C)$ such that $W(h - r1)W^* \in \{h',-h'\}$. 
\end{example}

\begin{notes} The material in this chapter is new in the sense that the author has not seen it in any book or paper. But most if not all arguments are straightforward applications of well-known facts for normal flows on von Neumann algebras, and the content can safely be considered to be folklore. 
\end{notes}



\begin{appendices}
\chapter{Integration in Banach spaces}\label{integration}

Let $X$ be a Banach space, $M$ a locally compact Hausdorff space and $\mu$ a regular measure on $M$. Let $f : M \to X$ be a continuous function such that
\begin{equation}\label{12-02-22a}
\int_M \left\|f(x)\right\| \ \mathrm d \mu(x) < \infty .
\end{equation}
In this appendix we describe a way to define the integral
$$
\int_M f(x) \ \mathrm d\mu(x) ,
$$
and establish some of its properties that are used in the main text.

\section{The definition}
We consider pairs $(\mathcal U, \epsilon)$ where $\mathcal U$ is a finite collection of mutually disjoint Borel sets $U$ of $M$ and $\epsilon > 0$ is a positive number such that
\begin{itemize}
\item $\bigcup_{U \in \mathcal U}$ is pre-compact; that is, its closure is compact,
\item $\int_{M \backslash (\bigcup_{U \in \mathcal U} U)} \  \left\|f(x)\right\| \ \mathrm d \mu(x) \leq \epsilon$, and
\item $\sup_{x,y \in U} \left\|f(x) -f(y)\right\|\left( \mu\left(\bigcup_{U \in \mathcal U} U\right) + 1\right)  \leq \epsilon$ for all $U \in \mathcal U$.
\end{itemize}
The collection of such pairs will be denoted by $\mathcal I$. We consider $\mathcal I$ as a pre-ordered set where $(\mathcal U,\epsilon) \leq (\mathcal V,\delta)$ means that
\begin{itemize}
\item $\delta \leq \epsilon$, 
\item $\bigcup_{U \in \mathcal U} \subseteq \bigcup_{V \in \mathcal V}V$, and 
\item for every $V \in \mathcal V$ there is an element $U' \in \mathcal U$ such that $V \cap (\bigcup_{U \in \mathcal U}U)  \subseteq U'$.
\end{itemize}

\begin{lemma}\label{22-03-22b} $\mathcal I$ is a directed set. 
\end{lemma} 
\begin{proof} Let $(\mathcal U,\epsilon), (\mathcal V,\delta) \in \mathcal I$. Set $\epsilon' = \min \{\epsilon,\delta\}$. We choose a compact set $K \subseteq M$ such that
$$
\left(\overline{\bigcup_{U \in \mathcal U} U}\right) \cup \left( \overline{\bigcup_{V \in \mathcal V} V}\right)  \subseteq K
$$
and
$$
\int_{M\backslash K} \|f(x)\| \ \mathrm d\mu(x) \leq \epsilon' .
$$ 
The sets
$$
K \backslash \left( {\bigcup_{U \in \mathcal U} U} \cup {\bigcup_{V \in \mathcal V} V}\right),
$$
$$
U \cap V , \ \ \ \ U \in \mathcal U, \ V \in \mathcal V,
$$
$$
U  \backslash \bigcup_{V \in \mathcal V} V , \ \ \ \ U \in \mathcal U
$$
and
$$ 
 V\backslash \bigcup_{U \in \mathcal U} U, \ \ \ \  V \in \mathcal V ,
 $$
constitute a partition $\mathcal P$ of $K$ into Borel sets. Since $f$ is continuous and $K$ is compact there is a finite Borel partition 
$$
K = \bigcup_{W \in \mathcal W} W
$$
of $K$ subordinate to $\mathcal P$ such that
$$
\sup_{x,y \in W} \left\|f(x) -f(y)\right\|(\mu(K)+1) \leq \epsilon' 
$$
for all $W \in \mathcal W$. Then $(\mathcal U,\epsilon) \leq (\mathcal W,\epsilon')$ and $(\mathcal V,\delta) \leq (\mathcal W,\epsilon')$.
\end{proof}

For $(\mathcal U,\epsilon) \in \mathcal I$, we denote by 
$\mathcal S(\mathcal U, \epsilon)$
the set of elements $S$ in $X$ of the form
$$
S = \sum_{U \in \mathcal U} f(i(U))\mu(U)
$$
for some function $i : \mathcal U \to M$ with the property that $i(U) \in U$ for all $U \in \mathcal U$.

\begin{lemma}\label{12-02-22} There is an element $I \in X$ with the property that for every $\delta > 0$ there is a $(\mathcal U,\epsilon) \in \mathcal I$ such that 
$$
\left\| I-S\right\| \leq \delta
$$
for all $S \in S(\mathcal V,\epsilon')$ when $(\mathcal U,\epsilon) \leq (\mathcal V,\epsilon')$.
\end{lemma}
\begin{proof} It follows from \eqref{12-02-22a} that for each $n \in \mathbb N$ there is a compact subset $K_n \subseteq M$ such that 
$$
\int_{M \backslash K_n} \|f(x)\| \ \mathrm d\mu(x) \leq \frac{1}{n} .
$$
Since $f$ is continuous and $K_n$ compact there is a finite partition $\mathcal U_n$ of $K_n$ into mutually disjoint Borel subsets $U$ of $K_n$ such that $\|f(x) -f(y)\|(\mu(K_n) +1) \leq \frac{1}{n}$ for all $x,y \in U$ and all $U \in \mathcal U_n$. Then $(\mathcal U_n,\frac{1}{n}) \in \mathcal I$. We proceed inductively, as we can, to arrange that $(\mathcal U_n,\frac{1}{n}) \leq (\mathcal U_{n+1},\frac{1}{n+1})$ in $\mathcal I$ for all $n$. For each $n$ choose $S_n \in \mathcal S(\mathcal U_n,\frac{1}{n})$. Let $(\mathcal V,\delta) \geq (\mathcal U_n,\frac{1}{n})$ and consider an element 
$$
S = \sum_{V \in \mathcal V} f(j(V))\mu(V) \in \ \mathcal S(\mathcal V,\delta) .
$$
Then
\begin{align*}
& S = \sum_{V \in \mathcal V} f(j(V))\mu\left( V \backslash \bigcup_{U \in \mathcal U_n} U\right) +\sum_{V \in \mathcal V, U \in \mathcal U_n} f(j(V))\mu\left( V \cap U\right) .
\end{align*}
Note that
\begin{align*}
&\left\|\sum_{V \in \mathcal V} f(j(V))\mu\left( V \backslash \bigcup_{U \in \mathcal U_n} U \right)\right\| \leq \sum_{V \in \mathcal V} \left\|f(j(V))\right\|\mu\left( V \backslash \bigcup_{U \in \mathcal U_n} U\right) \\
& = \sum_{V \in \mathcal V} \int_{ V \backslash \bigcup_{U \in \mathcal U_n}U} \left\|f(j(V))\right\| \ \mathrm d \mu(x)\\ 
& \leq  \sum_{V \in \mathcal V} \int_{ V \backslash \bigcup_{U \in \mathcal U_n} U} \left(\left\|f(x)\right\| + \frac{\delta}{\mu\left(\bigcup_{W \in \mathcal V}W\right) +1} \right) \ \mathrm d \mu(x)\\ 
& \leq \delta + \int_{M \backslash \bigcup_{U \in \mathcal U_n} U} \left\|f(x)\right\| \ \mathrm d \mu(x) \leq \delta + \frac{1}{n} \leq \frac{2}{n} . 
\end{align*}
Let $i_n : \mathcal U_n \to X$ be the function defining $S_n$; that is,
$$
S_n = \sum_{U \in \mathcal U_n} f(i_n(U)) \mu(U) .
$$
Then
\begin{align*}
& \left\|S_n - \sum_{V \in \mathcal V, U \in \mathcal U_n} f(j(V))\mu\left( V \cap U\right)\right\| \\
& = \left\|\sum_{V \in \mathcal V, U \in \mathcal U_n} f(i_n(U))\mu\left( V \cap U\right)- \sum_{V \in \mathcal V, U \in \mathcal U_n} f(j(V))\mu\left( V \cap U\right)\right\|\\
& \leq \sum_{V \in \mathcal V, U \in \mathcal U_n} \left\|f(i_n(U))   - f(j(V))\right\|\mu\left( V \cap U\right)\\
& \leq  \sum_{V \in \mathcal V, U \in \mathcal U_n} \left( \frac{1}{n \left(\mu(K_n) +1\right)} + \frac{\delta}{\mu\left(\bigcup_{V \in \mathcal V}V\right) +1} \right) \mu\left( V \cap U\right)\\
& \leq \frac{1}{n} + \delta \leq \frac{2}{n} .
\end{align*}
It follows that $\left\|S - S_n\right\| \leq \frac{4}{n}$. In particular, the sequence $\{S_n\}$ is Cauchy in $X$ and we set $I = \lim_{n \to \infty} S_n$. Then $I$ has stated property: Let $\delta > 0$ be given. There is an $N \in \mathbb N$ such that $\frac{4}{n} \leq \frac{\delta}{2}$ for all $n \geq N$ and $\left\|S_n -I\right\| \leq \frac{\delta}{2}$ for all $n \geq N$. If $(\mathcal V, \epsilon') \geq (\mathcal U_N,\frac{1}{N})$ the calculations above imply that $\left\|S - I\right\| \leq \left\|S-S_N\right\| + \frac{\delta}{2} \leq \frac{4}{N} + \frac{\delta}{2} \leq \delta$ for all $S \in \mathcal S(\mathcal V,\epsilon')$.
\end{proof}

The element $I$ of Lemma \ref{12-02-22} will be denoted by
$$
\int_M f(x) \ \mathrm d \mu(x) .
$$
By construction there is a sequence $\{S_n\}$ in $X$ such that
$$
\lim_{n \to \infty} S_n = \int_M f(x) \ \mathrm d \mu(x)
$$
and for each $n$ there is a finite collection of mutually disjoint Borel sets $\{U_i : \ i=1,2, \cdots, N_n\}$ in $M$ and elements $x_i \in U_i, \ i =1,2,\cdots, N_n$, such that the sum
$$
I_n := \sum_{i=1}^{N_n} f(x_i)\mu(U_i)
$$
has the properties that
$$
\left\| I_n - \int_M f(x) \ \mathrm d \mu(x)\right\| \leq \frac{1}{n}
$$
and
$$
\left\|I_n\right\| \leq \int_M \left\|f(x)\right\| \ \mathrm d\mu(x) + \frac{1}{n} .
$$
In particular, it follows that
\begin{equation}\label{28-02-22a}
\left\|\int_M f(x) \ \mathrm d\mu(x)\right\| \leq \int_M \left\|f(x)\right\| \ \mathrm d\mu(x) .
\end{equation}

\section{Miscellaneous}

\begin{lemma}\label{12-02-22b} Let $X$ and $Y$ be Banach spaces, $D \subseteq X$ a subspace and $L : D \to Y$ a linear map. Let $M$ be a locally compact Hausdorff space and $\mu$ a regular Borel measure on $M$. Let $f: M \to X$ be a continuous function. Assume that
\begin{itemize}
\item $L : D \to Y$ is closed,
\item $f(x) \in D$ for all $x \in M$,
\item $M \ni x \mapsto L(f(x))$ is continuous,
\item $\int_M \left\|f(x)\right\| \ \mathrm dx < \infty$, and
\item $\int_M \left\|L(f(x)\right\| \ \mathrm dx < \infty$. 
\end{itemize}
It follows that $\int_M f(x) \ \mathrm d \mu(x) \in D$ and $L\left(\int_M f(x) \ \mathrm d \mu(x)\right) = \int_M L(f(x))  \ \mathrm d \mu(x)$.
\end{lemma}
\begin{proof} Consider the function $M \to X \oplus Y$ defined by 
$$
M \ni x \mapsto (f(x),L(f(x)) .
$$
By applying Lemma \ref{12-02-22} to this function we obtain sequences $I_n$ in $D$ and $J_n$ in $Y$ such that $L(I_n) = J_n$ and 
$$
\lim_{n \to \infty} (I_n,J_n) = \left( \int_M f(x) \ \mathrm d \mu(x), \int_M L(f(x)) \ \mathrm d \mu(x)\right)
$$
in $X \oplus Y$. This gives the conclusion because $L$ is closed. 
\end{proof}

\begin{lemma}\label{28-02-22} Let $X$ be a Banach space, $M$ a locally compact Hausdorff space and $\mu$ a regular measure on $M$. Let $\{f_n\}$ be a sequence of continuous functions $f_n : M \to X$ and $f:M \to X$ a continuous function such that $\lim_{n \to \infty} f_n(x) = f(x)$ for all $x \in M$. Assume that there is function $h \in L^1(M,\mu)$ such that 
$$
\left\|f_n(x)\right\| \leq h(x) \ \ \forall x \in M.
$$
Then $\lim_{n \to \infty} \int_M f_n(x) \ \mathrm d\mu(x) =  \int_M f(x) \ \mathrm d\mu(x)$ in $X$.
\end{lemma} 
\begin{proof} Note that all the integrals $\int_Mf_n(x) \ \mathrm d\mu(x)$ and $\int_Mf(x) \ \mathrm d\mu(x)$ are defined since $\int_M h(x) \ \mathrm d \mu(x) < \infty$. It follows from \eqref{28-02-22a} that
$$
\left\| \int_Mf_n(x) \ \mathrm d\mu(x) -  \int_M f(x) \ \mathrm d\mu(x)\right\| \leq \int_M \left\| f_n(x) - f(x)\right\| \ \mathrm d \mu(x) .
$$
An application of Lebesgue's theorem on dominated convergence shows that $\lim_{n \to \infty}  \int_M \left\| f_n(x) - f(x)\right\| \ \mathrm d \mu(x) = 0$.
\end{proof}

The setting of the next lemma is that from Section \ref{Section-entire}. In particular, it deals with the smoothing operators defined in \eqref{17-11-21f}.

\begin{lemma}\label{01-03-22} Let $k \in \mathbb N$, $a \in A$. For every $\epsilon > 0$ there are numbers $\lambda_i \in [0,1]$ and $t_i \in \mathbb R$, $i =1,2,\cdots , n$, such that $\sum_{i=1}^n \lambda_i = 1$ and
$$
\left\| R_k(a) - \sum_{i=1}^n \lambda_i \sigma_{t_i}(a)\right\| \leq \epsilon .
$$
\end{lemma}
\begin{proof} Define the continuous function $f : \mathbb R \to A \oplus \mathbb C$ such that
$$
f(t) = \left(\sqrt{\frac{k}{\pi}} e^{-k t^2} \sigma_t(a),  \  \sqrt{\frac{k}{\pi}} e^{-k t^2}\right) .
$$
Let $\delta  \in ]0,1[$ be so small that
$$
 \frac{\delta}{1-\delta} (\|a\| + 1) \leq \epsilon .
 $$
By applying Lemma \ref{12-02-22} to $f$ we get numbers $s_i \geq 0$ and $t_i \in \mathbb R$, $i =1,2,\cdots, n$, such that
$$
\left\|\sum_{i=1}^n \sqrt{\frac{k}{\pi}} e^{-k t_i^2} \sigma_{t_i}(a) s_i - R_k(a)\right\| \leq \delta 
$$
and
$$
\left| \sum_{i=1}^n \sqrt{\frac{k}{\pi}} e^{-k t_i^2} s_i - 1\right| \leq \delta .
$$
Set $x ;
:=  \sum_{i=1}^n \sqrt{\frac{k}{\pi}} e^{-k t_i^2} s_i$ and
$$
\lambda_i := x^{-1}\sqrt{\frac{k}{\pi}} e^{-k t_i^2} s_i .
$$
Then $\sum_{i=1}^n \lambda_i = 1$ and
\begin{align*}
& \left\| R_k(a) - \sum_{i=1}^n \lambda_i \sigma_{t_i}(a)\right\| \\
&\leq \left\| x^{-1}R_k(a) - R_k(a)\right\| + \left\| x^{-1}R_k(a) - x^{-1} \sum_{i=1}^n \sqrt{\frac{k}{\pi}} e^{-k t_i^2} \sigma_{t_i}(a) s_i\right\| \\
& \leq \left|x^{-1} -1 \right|\left\|R_k(a)\right\| + x^{-1} \delta \leq \frac{\delta}{1-\delta}\|a\| + \frac{\delta}{1-\delta} \leq \epsilon.
\end{align*}
\end{proof}

\chapter{Flows on compact operators}\label{compact operators}
Let $\mathbb K$ denote the $C^*$-algebra of compact operators on the infinite dimensional separable Hilbert space $\mathbb H$. In this appendix we prove the following folklore theorem.

\begin{thm}\label{12-04-22}  Let $\sigma$ be a flow on $\mathbb K$. There is a strongly continuous unitary representation $U$ of $\mathbb R$ on $\mathbb H$ such that
$$
\sigma_t(a) = U_taU_t^*
$$
for all $a \in \mathbb K$ and all $t \in \mathbb R$.
\end{thm}

By Stone's theorem, cf. e.g. Theorem 5.6.36 of \cite{KR}, this theorem has the following corollary.

\begin{cor}\label{19-04-22} Let $\sigma$ be a flow on $\mathbb K$. There is a (possible unbounded) self-adjoint operator $H$ on $\mathbb H$ such that
$$
\sigma_t(a) = e^{itH}ae^{-itH}
$$
for all $t \in \mathbb R$ and all $a \in \mathbb K$.
\end{cor}

An important step of the proof of Theorem \ref{12-04-22} depends on a cohomological fact which is isolated in the following subsection.

\subsection{Cohomology}

Let $\mathbb T$ be the circle group; $\mathbb T = \left\{ z \in \mathbb C: \ |z| = 1\right\}$. A continuous function $\lambda : \mathbb R \times \mathbb R \to \mathbb T$ satisfying
\begin{itemize}
\item[(i)] $\lambda(s,t)\lambda(s+t,u) = \lambda(t,u)\lambda(s,t+u)$ for all $s,t,u \in \mathbb R$,
\end{itemize}
is called a $2$-cocycle, and it is said to be normalized when
\begin{itemize}
\item[(ii)] $\lambda(0,t) = \lambda(t,0) = 1$ for all $t \in \mathbb R$.
\end{itemize}
The set of normalized $2$-cocycles is an abelian group under pointwise multiplication. Any continuous function $\mu : \mathbb R \to \mathbb T$ with $\mu(0) = 1$ gives rise to a normalized $2$-cocycle $\lambda_\mu$ such that
$$
\lambda_\mu(s,t) = \mu(s)\mu(t) \overline{\mu(s+t)} ,
$$
and such a $2$-cocycle is called a coboundary. The coboundaries constitute a subgroup of the normalized $2$-cocycles and the quotient group is denoted by $H^2(\mathbb R,\mathbb T)$. For the proof of Theorem \ref{12-04-22} it is crucial that this group is trivial.

\begin{thm}\label{19-04-22b} $H^2(\mathbb R,\mathbb T) = 0$, i.e. every normalized $2$-cocycle is a coboundary.
\end{thm}

\begin{proof}  Let $\lambda$ be a normalized $2$-cocycle. By continuity of $\lambda$ there is $\epsilon >0$ such that $|\lambda(s,t) -1| \leq 2^{-\frac{1}{2}}$ when $s,t \in [0,\epsilon]$. Set
$$
\lambda'(s,t) := \lambda(\epsilon s,\epsilon t) .
$$
Then $\lambda'$ is a normalized $2$-cocycle and $\lambda$ is a coboundary if and only if $\lambda'$ is. We may therefore assume that 
\begin{equation}\label{19-04-22d}
|\lambda(s,t) -1| \leq 2^{-\frac{1}{2}}
\end{equation} 
when $s,t \in [0,1]$. There is a unique function $\mu : \mathbb R \to \mathbb T$ such that $\mu(0) = 1$ and
\begin{equation}\label{18-04-22a}
\mu(t+n) = \mu(n)\lambda(n,t), \ \ \forall n \in \mathbb Z, \ t \in [0,1] .
\end{equation}
Note that $\mu$ is continuous and that 
\begin{equation}\label{18-04-22b}
\mu(t) = \lambda(0,t) = 1, \ \ \forall t \in [0,1] .
\end{equation}
Define $\lambda^1 : \mathbb R \times \mathbb R \to \mathbb T$ such that
$$
\lambda^1(s,t) :=  \overline{\mu(s+t)} \mu(s)\mu(t) \lambda(s,t)  \ .
$$
Then $\lambda^1$ is a normalized $2$-cocycle and it follows from \eqref{18-04-22a} and \eqref{18-04-22b} that
\begin{equation}\label{18-04-22}
\lambda^1(n,t) = 1 , \ \ \forall t \in [0,1], \ \forall n \in \mathbb Z .
\end{equation}
Since $\lambda^1$ is a $2$-cocycle,
$$
\lambda^1(n,1)\lambda^1(n+1,t) = \lambda^1(1,t)\lambda^1(n,1+t)
$$
and it follows therefore by induction starting with \eqref{18-04-22}, that $\lambda^1(n,t) = 1$ for all $n \in \mathbb Z$ and all $t \geq 0$. Similarly, by using that 
$$
\lambda^1(n,-1)\lambda^1(n-1,t)  = \lambda^1(-1,t) \lambda^1(n,t-1) ,
$$
it follows that $\lambda^1(n,t) = 1$ for all $n \in \mathbb Z$ and all $t \leq 0$. It follows therefore that
$$
\lambda^1(n+s,t) = \lambda^1(n,s)\lambda^1(n+s,t) = \lambda^1(s,t)\lambda^1(n,s+t) = \lambda^1(s,t)
$$
for all $s,t \in \mathbb R$, showing that
$$
s \mapsto \lambda^1(s,t)
$$
is $1$-periodic for all $t \in \mathbb R$. When $t \in [1,2]$, it follows from the equations \eqref{19-04-22d}, \eqref{18-04-22a} and \eqref{18-04-22b} that
$$
|\mu(t) - 1| = |\mu(1+(t-1)) - 1| = |\lambda(1,t-1) - 1| \leq  2^{-\frac{1}{2}} ,
$$ 
which combined with \eqref{18-04-22b} shows that
$$
|\mu(t) - 1| \leq  2^{-\frac{1}{2}} , \ \ \forall t \in [0,2].
$$ 
Combined with \eqref{18-04-22b} this shows that
$$
\left|\lambda^1(s,t) - \lambda(s,t)\right| = \left|{\mu(s+t)} - 1\right| \leq  2^{-\frac{1}{2}}
$$
when $s,t \in [0,1]$. By using \eqref{19-04-22d} this implies that
\begin{equation}\label{19-04-22e}
\left| \lambda^1(s,t) - 1\right| \leq 2 \cdot 2^{-\frac{1}{2}} = \sqrt{2} 
\end{equation}
when $s,t \in [0,1]$. Since $\lambda^1$ is $1$-periodic in the first variable it follows that \eqref{19-04-22e} holds for all $s \in \mathbb R$ and all $t \in [0,1]$. Let $\chi : \mathbb T \backslash \{-1\} \to ]-\pi,\pi[$ be the inverse of $]-\pi,\pi[ \ni x \mapsto e^{ix}$ and set
$$
y(s,t) := \chi(\lambda^1(s,t))
$$
when $t \in [0,1]$. It follows from the cocycle identity (i) that
\begin{equation}\label{19-04-22f}
y(s,t) + y(s+t,u) = y(t,u) + y(s,t+u)
\end{equation}
for all $s \in \mathbb R$ and all $t,u \in [0,1]$ such that $t+u \leq 1$. Define $\alpha : [0,1] \to \mathbb R$ such that
$$
\alpha(t) := \int_0^1 y(x,t) \ \mathrm d x .
$$
Then $\alpha$ is continuous. Since $y$ is $1$-periodic in the first variable it follows that
$$
\int_0^1 y(s+t,u) \ \mathrm d s  = \int_0^1 y(s,u) \ \mathrm d s
$$
for all $t \in \mathbb R$ and all $u \in [0,1]$. By integrating the equation \eqref{19-04-22f} with respect to $s$ it follows therefore that
$$
y(t,u) = \alpha(t)+\alpha(u) - \alpha(t+u)
$$
for all $t,u \in [0,1]$ with $t+u \leq 1$. Let $\mu^1 : \mathbb R \to \mathbb T$ be a continuous function such that $\mu^1(x) = e^{i \alpha(x)}$ when $x \in [0,1]$. Then 
$$
\mu^1(t)\mu^1(u)\overline{\mu^1(t+u)} = \lambda^1(t,u)
$$
when $t,u \in [0,1]$ and $t+u \leq 1$. Set
$$
\lambda^2(s,t) := \overline{\mu^1(s)}\overline{\mu^1(t)}\mu^1(s+t)\lambda^1(s,t) 
$$
for all $s,t \in \mathbb R$. Then $\lambda^2$ differ from $\lambda^1$ by a coboundary and it suffices therefore now to show that $\lambda^2$ is a coboundary. For this observe that $\lambda^2(s,t) = 1$ for all $s,t \in [0,\frac{1}{2}]$. Set
$$
\lambda^3(s,t) := \lambda^2\left(\frac{s}{2},\frac{t}{2}\right) .
$$
Then $\lambda^3$ is a normalized $2$-cocycle and it suffices to show that $\lambda^3$ is a coboundary. Note that $\lambda^3(s,t) = 1$ for $s,t \in [0,1]$. By repeating the part of the argument above which constructed $\lambda^1$ from $\lambda$ we get a continuous function $\mu^2 : \mathbb R \to \mathbb T$ such that $\mu^2(0) = 1$ and such that the $2$-cocycle $\lambda^4$ given by
$$
\lambda^4(s,t) := \overline{\mu^2(s+t)} \mu^2(s)\mu^2(t) \lambda^3(s,t)  \ 
$$
is $1$-periodic in the first variable and
\begin{equation*}
\lambda^4(n,t) = 1 , \ \ \forall t \in [0,1], \ \forall n \in \mathbb Z .
\end{equation*}
But this time, because $\lambda^3(1,t) = 1$ for all $t \in [0,1]$ we get in addition that $\mu^2(t) = 1$ for all $t \in [0,2]$, and hence
$$
\lambda^4(s,t) = \lambda^3(s,t) = 1
$$
for all $s,t \in [0,1]$. By periodicity in the first variable this implies that $\lambda^4(s,t) = 1$ for all $s \in \mathbb R$ when $t \in [0,1]$. The cocycle equation (i) implies then, by induction, that $\lambda^4(s,t) = 1$ for all $s \in \mathbb R$ and all $t \geq 0$. Let $t \geq 0$. Inserting $-t$ for $u$ in (i) yields now that  
$$
\lambda^4(s+t,-t) = \lambda^4(t,-t) 
$$
for all $s \in \mathbb R$. Hence $\lambda^4(x,-t)$ does not depend on $x$ and we conclude that $\lambda^4(x,-t) = \lambda^4(0,-t) = 1$ for all $x \in \mathbb R$. Thus $\lambda^4$ is constant $1$ and $\lambda^3$ is therefore a coboundary. 

\end{proof}

\emph{Proof of Theorem \ref{12-04-22}:}

 Let $e_{ij}, \ i,j = 1,2,3, \cdots $, be a generating set of matrix units in $\mathbb K$. That is,
\begin{itemize}
\item $e_{jj}, j = 1,2,3, \cdots$, are one-dimensional projections,
\item $\lim_{N \to \infty} \sum_{j=1}^N e_{jj} =1$ in the strong operator topology,
\item $e_{ij}e_{kl} = \begin{cases} e_{il}, \ & \ k =j \\ 0, \ & \ k \neq j,  \end{cases}$ for all $i,j,k,l$, and
\item $e_{ij}^* = e_{ji}$ for all $i,j$.
\end{itemize} 
Then $\mathbb K$ is the norm-closure of the linear span of $\{e_{ij}\}_{i,j \in \mathbb N}$.

\begin{obs}\label{12-04-22a} There is a norm continuous path $V(t), t \in \mathbb R$, of unitaries in $\mathbb C 1 +\mathbb K$ such that $V(t)e_{11}V(t)^* = \sigma_t(e_{11})$ for all $t \in \mathbb R$.
\end{obs}

\emph{Proof of Observation \ref{12-04-22a}.} Choose $\delta > 0$ such that $\left\|\sigma_t(e_{11}) - e_{11}\right\| < 1$ when $|t| \leq \delta$. Set
$$
x_t = e_{11}\sigma_t(e_{11}) + (1-e_{11})(1-\sigma_t(e_{11})) .
$$
Then
$$
\left\|1-x_t\right\| = \left\| (2e_{11} -1)(e_{11} - \sigma_t(e_{11})\right\| \leq  \left\|e_{11} - \sigma_t(e_{11}\right\| < 1
$$
and hence $x_t$ is invertible when $|t| \leq \delta$. Then 
$$
V(t) := (x_t^*x_t)^{-1/2}x_t^*
$$
is a unitary when $|t| \leq \delta$. Note that $[-\delta,\delta]\ni t \mapsto V(t)$ is norm-continuous and that $V(0) = 1$.
Since $e_{11}x_t = e_{11}\sigma_t(e_{11}) = x_t\sigma_t(e_{11})$ it follows that $x_t^*e_{11} = \sigma_t(e_{11})x_t^*$, and hence $x_t^*x_t$ commutes with $\sigma_t(e_{11})$. So does $(x_t^*x_t)^{-1/2}$ and thus $V(t)e_{11} = \sigma_t(e_{11})V(t)$ for $t \in [-\delta,\delta]$. For $t \in [\delta,2 \delta]$, set 
$$
V(t) = \sigma_{t-\delta}(V(\delta))V(t-\delta).
$$
Then $[0,2 \delta] \ni t \mapsto V(t)$ is norm-continuous, $V(t)$ is a unitary in $\mathbb C 1 + \mathbb K$ for all $t \in [0,2\delta]$ and $V(t)e_{11}V(t)^* = \sigma_t(e_{11})$ for all $t \in [0,2 \delta]$. This construction can be continued. If $[0,n \delta] \ni t \mapsto V(t)$ is norm-continuous, $V(t)$ is a unitary in $\mathbb C 1 + \mathbb K$ for all $t \in [0,n\delta]$ and $V(t)e_{11}V(t)^* = \sigma_t(e_{11})$ for all $t \in [0,n\delta]$, set
$$
V(t) = \sigma_{t-n\delta}(V(n\delta))V(t-n\delta) 
$$
for $t \in [n\delta, (n+1)\delta]$. In this way we get a norm-continuous path $V(t), t \in [0,\infty)$, of unitaries in $\mathbb C 1 + \mathbb K$ such that $V(0) = 1$ and $V(t)e_{11}V(t)^* = \sigma_t(e_{11})$ for all $t \in [0,\infty)$. To define $V(t)$ for $t \leq 0$, proceed in the same way starting with $V(t), \ t \in [-\delta,0]$, or set $V(t) = \tilde{\sigma}_t(V(-t)^*)$, where $\tilde{\sigma}_t \in \Aut(\mathbb C 1 + \mathbb K)$ is the automorphism extending $\tilde{\sigma}_t$.
\qed

Let $\psi \in H$. For $n \in \mathbb N$, set
$$
W_n(t) :=  \sum_{j=1}^n \sigma_t(e_{j1})V(t)e_{1j}.
$$
Since
$$
\left\|\sum_{j=n}^m  \sigma_t(e_{j1})V(t)e_{1j} \psi\right\|^2 = \sum_{j=n}^m \left\| \sigma_t(e_{j1})V(t)e_{1j} \psi\right\|^2 \leq \sum_{j=n}^m \left\|e_{jj} \psi\right\|^2 
$$
when $n \leq m$, we see that $\left\|W_n(t)\right\| \leq 1$, and that the sum
$$
W(t) := \sum_{j=1}^\infty \sigma_t(e_{j1})V(t)e_{1j} = \lim_{n \to \infty} W_n(t)
$$
converges in the strong operator topology, uniformly in $t$, to an operator $W(t)$ of norm $\leq 1$. In fact,
\begin{align*}
&W(t)W(t)^* = \lim_{n \to \infty} \lim_{m \to \infty} W_n(t)W_m(t)^* \\
& = \lim_{n \to \infty}\lim_{m \to \infty}\sum_{j=1}^n\sum_{k=1}^m \sigma_t(e_{j1})V(t)e_{1j}e_{k1}V(t)^*\sigma_t(e_{1k})\\
& =\lim_{n \to \infty}\lim_{m \to \infty} \sum_{j=1}^{\min\{n,m\}} \ \sigma_t(e_{j1})V(t)e_{11}V(t)^*\sigma_t(e_{1j}) \\
& =\lim_{n \to \infty}\lim_{m \to \infty} \sum_{j=1}^{\min\{n,m\}} \ \sigma_t(e_{j1})\sigma_t(e_{11})\sigma_t(e_{1j}) \\
& =  \lim_{n \to \infty}\lim_{m \to \infty} \sum_{j=1}^{\min\{n,m\}}\sigma_t\left( e_{jj}\right) = 1,
\end{align*}
where the last equality follows from the fact that $\sum_{j=1}^\infty \sigma_t(e_{jj}) =1$ with convergence in the strong operator topology.
Similarly, $W(t)^*W(t) = 1$, showing that $W(t)$ is a unitary. Since the sum defining $W(t)$ converges uniformly in $t$ with respect to the strong operator topology it follows from the continuity of $V$ that $\mathbb R \ni t \mapsto W(t)$ is continuous for the strong operator topology. For $k,l \in \mathbb N$,
\begin{align*}
&W(t)e_{kl}W(t)^* =\sum_{j=1}^\infty \sum_{i=1}^\infty \sigma_t(e_{j1})V(t)e_{1j}e_{kl} e_{i1}V(t)^* \sigma_t(e_{1i}) \\
& =\sigma_t(e_{k1})V(t)e_{11}V(t)^*e_{1l} = \sigma_t(e_{kl}) ,
\end{align*}
showing that 
\begin{equation}\label{12-04-22c}
W(t)aW(t)^* = \sigma_t(a) \ \forall a \in \mathbb K, \ \forall t \in \mathbb R.
\end{equation}
In particular, since $\sigma_{t + s} = \sigma_t \circ \sigma_s$, it follows that
$$
W(t+s)aW(t+s)^* = W(t)W(s)aW(s)^*W(t)^*
$$
for all $s,t,a$. Thus $W(s)W(t)W(t+s)^*$ is a unitary which commutes with all elements of $\mathbb K$, and it must therefore be a scalar multiple of $1$. It follows that there is a continuous function $\lambda : \mathbb R \times \mathbb R \to \mathbb T$ such that
\begin{equation}\label{12-04-22e}
W(s)W(t)W(t+s)^* = \lambda(s,t) 1  \ \ \forall t,s \in \mathbb R.
\end{equation}
Note that
\begin{align*}
&\lambda(s,t)\lambda(s+t,u)1 = W(s)W(t)W(u)W(s+t+u)^* \\
& =  W(s)W(t)W(u)W(t+u)^*W(t+u)W(s+t+u)^* \\
&= \lambda(t,u)  W(s)W(t+u)W(s+t+u)^* = \lambda(t,u)\lambda(s,t+u) 1 .
\end{align*}
Furthermore, $\lambda(0,t) = \lambda(t,0) = 1$ for all $t \in \mathbb R$. Thus $\lambda$ is a normalized $2$-cocycle and by Theorem \ref{19-04-22b} there is a continuous function $\mu: \mathbb R \to \mathbb T$ such that $\mu(0) = 1$ and $\lambda(s,t) = \overline{\mu(s)}\overline{\mu(t)}\mu(s+t)$ for all $s,t \in \mathbb R$. Set $U_t = \mu(t)W(t)$. \qed

\begin{notes}\label{19-04-22i} The proof of Theorem \ref{19-04-22b} is taken from \cite{BuR}. According to that source  the result itself, that $H^2(\mathbb R,\mathbb T) = 0$, is due Iwasawa.
\end{notes}

\chapter{Multipliers}\label{multipliers}

 Let $A$ be a $C^*$-algebra. We denote by $M(A)$ the set of maps $m : A
\to A$ (\emph{a priori not even linear}) for which there is another map $m^{\ast} : A \to
A$ such that
\begin{equation}\label{(1.2.5)} m(a)^{\ast}b = a^{\ast}m^{\ast}(b) , \ a,b
\in A.
\end{equation}

\begin{lemma}\label{Lemma 1.2.7} Every element $m$ of $M(A)$ is a
bounded linear operator on $A$ and $m^{\ast} \in M(A)$ with $(m^{\ast})^{\ast} = m$.
\end{lemma}
\begin{proof} We first observe that the following implication holds for
any element $a \in A$ :
\begin{equation}\label{(1.2.6)} ax = 0 \ \forall x \in A \Rightarrow a = 0.
\end{equation}
Indeed, if $ax = 0$ for all $x \in A$, then in particular 
$\|a\|^2 = \|aa^{\ast}\| = 0$ which implies that $a = 0$.

To show that $m$ is linear, let $a,b \in A$ and $\lambda \in
\Bbb C$. Then
\begin{align*} 
&(m(\lambda a + b) - \lambda m(a) - m(b))^{\ast}x =
(m(\lambda a + b)^{\ast} - \overline{\lambda} m(a)^{\ast} -
m(b)^{\ast})x =\\
&(\lambda a + b)^{\ast}m^{\ast}(x) - \overline{\lambda} a^{\ast}
m^{\ast}(x) - b^{\ast}m^{\ast}(x) = 0 
\end{align*}
for all $x \in A$, so by (\ref{(1.2.6)}) we have that $m(\lambda a + b) =
\lambda m(a) + m(b)$; i.e. $m$ is linear.

That $m$ is bounded follows from the principle of uniform boundedness
in the following way. For each $a \in A$ with $\|a\| \leq 1$, 
define $S_a : A \to A$ by $S_a(b) = m(a)^{\ast}b$. Each
$S_a$ is a bounded linear operator ($\|S_a\| \leq \|m(a)\|$) and for every fixed $b \in A$ we have that $\|S_a(b)\| = \|m(a)^{\ast}b\| = \|a^{\ast}
m^{\ast}(b)\| \leq \|m^{\ast}(b)\|$. Therefore the principle of uniform boundedness implies
that there is an $M < \infty$ such that $\|S_a\| \leq M$
for all $a$ with $\|a\| \leq 1$. Hence $\|m(a)^{\ast}b\| = \|S_a(b)\|\leq M$ for all $a,b \in A$ with $\|a\| \leq 1 , \|b\| \leq 1$. In particular, whenever
$\|a\| \leq 1$ and $m(a) \neq 0$ we have that
$$  \|m(a)\| = \|m(a)^{\ast} \frac{m(a)}{\|m(a)\|}\| \leq M.
$$
It follows that $\|m\| \leq M$. Finally, by applying the involution $*$ to the identity $m(a)^{\ast}b =
a^{\ast}m^{\ast}(b)$ we get immediately that $m^{\ast} \in M(A)$ and that in fact $(m^{\ast})^{\ast} = m$.
\end{proof}

\begin{prop}\label{Proposition 1.2.8} $M(A)$ is a $C^{\ast}$-algebra with
the involution $m \mapsto m^{\ast}$ and the norm $\|m\| = \sup \ 
\{\|m(a)\| : \|a\| \leq 1 \}$. The product in the algebra is the
composition of operators on $A$ and the identity operator
is in $M(A)$ (and is therefore a unit of $M(A)$).
\end{prop}
\begin{proof} We leave it to the reader to check that $M(A)$ is
a subalgebra of the algebra of bounded linear operators on $A$.
That is, you should check that $\lambda m + n, mn \in M(A)$ when $m,n \in M(A)$ and $\lambda \in \Bbb C$. It is also 
trivial that the identity operator is in $M(A)$.

Let us next check the $C^{\ast}$-identity. Since $\|m n\| 
\leq \|m\|\|n\|$ because we are dealing with the operator norm,
we see immeditaly that $\|m^{\ast}m\| \leq \|m^{\ast}\|\|m\|$. 
Let $a,b \in A, \|a\| \leq 1, \|b\| \leq 1$. Then
$$
\|m(a)^{\ast}b\| = \|a^{\ast}m^{\ast}(b)\| \leq \|m^{\ast}(b)\|
\leq \|m^{\ast}\|.
$$
When we set $b = \frac{m(a)}{\|m(a)\|}$ in this inequality we
get that $\|m(a)\| \leq \|m^{\ast}\|$ (when $m(a) \neq 0$). The
conclusion is that $\|m\| \leq \|m^{\ast}\| , \ m \in
M(A)$. When we substitute $m^{\ast}$ into this inequality we
get the converse version, $\|m^{\ast}\| \leq \|m\|$, implying that
$\|m^{\ast}\| = \|m\|$. Now, for any $a \in A$ with $\|a\| \leq
1$ we have the estimate
$$
\|m(a)\|^2 = \|m(a)^{\ast}m(a)\| = \|a^{\ast}m^{\ast}m(a)\| \leq
\|m^{\ast}m\|,
$$
implying that
$$
\|m\|^2 \leq \|m^{\ast}m\|.
$$
All in all we have the following inequalities
$$
\|m\|^2 \leq \|m^{\ast}m\| \leq \|m^{\ast}\|\|m\| = \|m\|^2
$$
from which the $C^{\ast}$-identity $\left\|m\right\|^2 = \left\|m^*m\right\|$ follows. 

To show that $M(A)$ is complete with respect to the operator norm,
 let $\{m_k\}$ be a Cauchy sequence in $M(A)$. For each $a \in A, \ k,l \in \Bbb N$, we have that
$$
\|m_ka - m_la\| \leq \|m_k - m_l\|\|a\| \ .
$$
It follows that $\{m_ka\}$ is a Cauchy sequence in $A$ for all $a \in A$. Since $\|m_k^* - m_l^*\| = \|(m_k - m_l)^*\| = \|m_k - m_l\|$ for all $k,l$, we see that also $\{m_k^*\}$ is a Cauchy sequence in $M(A)$. Consequently, $\{m_k^*a\}$ is a Cauchy sequence in $A$ for all $a \in A$. We can therefore define maps $m,n : A \to A$ by $ma = \lim_{k \to \infty} m_ka, \ na = \lim_{k \to \infty}m_k^*a, \ a \in A$. Note that
$$
(ma)^*b \ = \ \lim_{k \to \infty} (m_ka)^*b \ = \ \lim_{k\to \infty} a^*m_k^*b \ = \ a^*nb
$$
for all $a,b \in A$, proving that $m \in M(A)$ (and that $m^* = n$). To show that $\lim_{k \to \infty} m_k = m$ in $M(A)$, let $\epsilon > 0$. There is a $N \in \Bbb N$ such that
\begin{equation}\label{(1.2.7)} \|m_ka - m_la\| \leq \|m_k - m_l\| \leq \epsilon 
\end{equation}
for all $k, l \geq N$ and all $a \in A$ with $\|a\| \leq 1$. If we let $k$ tend to infinity we get from (\ref{(1.2.7)}) that
$$
 \|m_ka - ma\| \leq \epsilon
$$
for all $k \geq N$ and all $a \in A$ with $\|a\| \leq 1$. Hence $\|m_k - m\| \leq \epsilon$ for all $k \geq N$.                                     
\end{proof}

The $C^*$-algebra $M(A)$ is called \it the multiplier algebra \rm \index{multiplier algebra} of $A$. Every element $b \in A$ defines a multiplier $m_b \in M(A)$ such that 
$$
m_b(a) := b a. 
$$

\begin{lemma}\label{10-11-23} The map $b \mapsto m_b$ is an isometric $*$-homomorphism mapping $A$ onto a closed two-sided ideal in $M(A)$.
\end{lemma}
\begin{proof} Left to the reader.
\end{proof}

 As is customary we identify $A$ with its image in $M(A)$ under this embedding.

Besides the norm topology the multiplier algebra carries the \emph{strict topology} which is the topology defined by the semi-norms $\| \ \cdot \ \|_a,  a \in A$, where
$$
\left\|m\right\|_a := \left\|ma\right\| + \left\|m^*a\right\| .
$$

\begin{lemma}\label{08-02-23} $A$ is dense in $M(A)$ with respect to the strict topology. In fact, every element $m \in M(A)$ is the limit in the strict topology of a net $\{a_i\}$ in $A$ with $\|a_i\| \leq \|m\|$ for all $i$.
\end{lemma}
\begin{proof} Let $m \in  M(A)$ and let $\{u_i\}$ be an approximate unit in $A$. Then $mu_i \in A$ and for any $a \in A$ we find that
$$
\lim_{i \to \infty} \left\|ma - mu_ia\right\| \leq \lim_{i \to \infty} \left\|m\right\| \left\|a - u_ia\right\| = 0
$$
and
$$
\lim_{i \to \infty} \left\|am - amu_i\right\| = 0
$$
since $am \in A$.
Hence $\lim_{i \to \infty} mu_i = m$ in the strict topology. Set $a_i :=mu_i$.
\end{proof}

\begin{lemma}\label{26-09-22} Let $A$ be a $C^*$-algebra, $\mathbb H$ a Hilbert space and $\pi : A \to B(\mathbb H)$ a $*$-representation of $A$ on $\mathbb H$. Assume that $\pi$ is non-degenerate; i.e. $\pi(A)\mathbb H$ spans a dense subspace of $\mathbb H$. There is a $*$-homomorphism $\overline{\pi}: M(A) \to B(\mathbb H)$ such that
\begin{itemize}
\item $\overline{\pi}$ extends $\pi$, and
\item when $\{m_i\}_{i \in  I}$ is a norm-bounded net in $M(A)$ converging to $m \in M(A)$ in the strict topology, then $\lim_{i \to \infty} \overline{\pi}(m_i) = \overline{\pi }(m)$ in the strong operator topology.
\item When $\pi$ is also injective, $\overline{\pi}$ is a $*$-isomorphism of $M(A)$ onto
$$
\left\{ m \in B(\mathbb H): \ m \pi(A) \subseteq \pi(A), \ \pi(A)m \subseteq \pi(A) \right\} .
$$
\end{itemize}
\end{lemma}
\begin{proof} Let $m\in M(A)$. It follows from Lemma \ref{08-02-23} that there is a net $\{a_i\}$ in $A$ such that $\lim_{i \to \infty} a_i =m$ in the strict topology and $\|a_i\| \leq \|m\|$ for all $i$. Since $\pi$ is non-degenerate it follows that $\{\pi(a_i)\}$ converges in the strong operator topology and we set $\overline{\pi}(m) := \lim_{i \to \infty} \pi(a_i)$. It is straightforward to check that this gives a well-defined $*$-homomorphism $\overline{\pi}: M(A) \to B(\mathbb H)$ with the properties specified in first two items. Assume then that $\pi$ is injective. If $m \in M(A)$ and $\overline{\pi}(m) = 0$ we have for each $a \in A$ that $\pi(ma) = \overline{\pi}(m)\pi(a) = 0$ and hence that $ma = 0$. It follows that $m =0$, proving that $\overline{\pi}$ is injective. Let $m' \in B(\mathbb H)$ and assume that $  m' \pi(A) \subseteq \pi(A), \ \pi(A)m' \subseteq \pi(A)$. We can then define $m : A \to A$ such that
$$
m(a) := \pi^{-1}\left(m'\pi(a)\right) 
$$ 
and $m^{\#} : A \to A$ such that 
$$
m^{\#}(a) = \pi^{-1}\left( {m'}^* \pi(a)\right) .
$$
Since $m(a)^*b = a^*m^{\#}(b)$ we see that $m \in M(A)$ and $m^* = m^{\#}$. Since $\overline{\pi}(m)\pi(a) = \pi(ma) = m'\pi(a)$ for all $a \in A$, it follows that $\overline{\pi}(m) = m'$.
\end{proof}

\begin{remark}\label{27-09-23}\rm{
Note that any automorphism $\alpha \in \Aut A$ of $A$ extends to an automorphism of $M(A)$ defined such that
$$
\alpha(m)a := \alpha(m \alpha^{-1}(a)) \ \ \forall a \in A.
$$
We denote the extension by $\alpha$ again. By Lemma \ref{08-02-23} and Lemma \ref{26-09-22} the extension is the only extension of $\alpha$ to $M(A)$ which is continuous for the strict topology on norm-bounded sets. In particular, when $\sigma = (\sigma_t)_{t \in \mathbb R}$ is a flow on $A$ we have an associated one-parameter group $(\sigma_t)_{t \in \mathbb R}$ of automorphisms of $M(A)$, but this extension is not in general a flow since the map $t \mapsto \sigma_t(m)$ need not be continuous in norm when $m \in M(A)$. It is however continuous with respect to the strict topology.}
\end{remark}

\chapter{Traces and K-theory}\label{AppD}

\section{Pairing traces with $K_0$}

In the following we fix a trace $\tau: A^+ \to [0,\infty]$ on $A$, as defined in Definition \ref{03-02-22f}. Set
$$
\mathcal M^+_\tau :=  \{ a \in A^+ : \ \tau(a ) < \infty \} 
$$
and
$$
\mathcal M_\tau := \Span \{ a \in A^+ : \ \tau(a ) < \infty \} .
$$

\begin{lemma}\label{16-11-22} $\mathcal M_\tau, \ \mathcal M^+_\tau$ and $\tau$ have the following properties.
\begin{itemize}
\item $\mathcal M_\tau$ is a dense two-sided $*$-invariant ideal in $A$.
\item $\mathcal M_\tau \cap A^+ = \mathcal M_\tau^+$.
\item $\tau : \mathcal M^+_\tau \to [0,\infty)$ has a unique extension to a linear map $\tau|_{\mathcal M_\tau} : \mathcal M_\tau \to \mathbb C$.
\end{itemize}
\end{lemma}
\begin{proof} Set $\mathcal N_\tau = \left\{ a \in A: \ \tau(a^*a) < \infty \right\}$. The arguments from the proof of Lemma \ref{04-11-21n} show that $\mathcal N_\tau$ is a left ideal and $\mathcal M_\tau = \Span \mathcal N_\tau^*\mathcal N_\tau$. The trace property of $\tau$ implies that $\mathcal N_\tau^* = \mathcal N_\tau$ and hence $\mathcal N_\tau$ is also a right ideal. It follows therefore that $\mathcal M_\tau$ is a two-sided ideal in $A$. The second item follows as in the proof of Lemma \ref{04-11-21n} and implies that $\mathcal M_\tau$ is dense in $A$. The third is a straightforward consequence of the additivity of $\tau$, as in Lemma \ref{10-11-21}.
\end{proof}

Let $A^\dagger$ be the $C^*$-algebra obtained from $A$ be adjoining a unit to $A$. Thus as a vector space $A^\dagger$ is just $A \oplus \mathbb C$, and the product and involution are given by
$$
(a,\lambda)(b,\mu) = (ab + \lambda b + \mu a, \lambda \mu) 
$$
and
$$
(a,\lambda)^* = (a^*, \overline{\lambda}) .
$$
For every subset $X \subseteq A$, set
$$
X^\dagger := \left\{(a,\lambda) \in A^\dagger: \ a \in X, \ \lambda \in \mathbb C \right\}.
$$
In particular,
$$
\mathcal M_\tau^\dagger := \left\{(a,\lambda) \in A^\dagger: \ a \in \mathcal M_\tau, \ \lambda \in \mathbb C \right\} .
$$
Define $\tau^\dagger : \mathcal M_\tau^\dagger \to \mathbb C$ such that
$$
\tau^\dagger(a,\lambda) = \tau|_{\mathcal M_\tau}(a) .
$$
Since $\tau^\dagger$ is linear on $\mathcal M_\tau^\dagger$ we can consider the tensor product map
$$
\tau^\dagger \otimes \Tr_n : \  M_n(\mathcal M_\tau^\dagger)\to \mathbb C,
$$
where $\Tr_n$ denotes the standard trace on $M_n(\mathbb C)$; the sum of the diagonal entries. Let $P_n(\mathcal M_\tau^\dagger)$ be the set of projections in the $*$-algebra $ M_n(\mathcal M_\tau^\dagger)$ and set
$$
 P_\infty(\mathcal M_\tau^\dagger) :=  \bigcup_n P_n(\mathcal M_\tau^\dagger) .
 $$
We aim to establish the following

\begin{thm}\label{19-11-22d} 
$$K_0(A^\dagger) = \left\{ [e] - [f] : \ e,f \in P_\infty(\mathcal M_\tau^\dagger) \right\}
$$
and there is a homomorphism $\tau^\dagger_* : \ K_0(A^\dagger) \to \mathbb R$ such that
$$
\tau^\dagger_*([e]-[f]) = \tau^\dagger \otimes \Tr_n(e)  - \tau^\dagger \otimes \Tr_n(f)
$$
when $e,f \in M_n(\mathcal M_\tau^\dagger)$.
\end{thm}

\subsection{Proof of Theorem \ref{19-11-22d}} The proof uses the following series of lemmas. Recall that a \emph{strictly positive} element of a $C^*$-algebra $A$ is a positive element $a \in A^+$ such $\omega(a) > 0$ for all non-zero positive functionals $\omega$ on $A$. A $C^*$-algebra is \emph{$\sigma$-unital} when it contains a strictly positive element.

\begin{lemma}\label{02-01-23x} A separable $C^*$-algebra is $\sigma$-unital.
\end{lemma}
\begin{proof} When $A$ is a separable $C^*$-algebra there is a dense sequence $\{a_n\}_{n=1}^\infty$ in $\{a \in D^+: \ 0 \leq a \leq 1 \}$. Then
$$
a := \sum_{n=1}^\infty 2^{-n}a_n
$$
is strictly positive.
\end{proof}

\begin{lemma}\label{19-11-22fx} Let $D$ be a $\sigma$-unital $C^*$-algebra. There is a sequence $\{d_n\}_{n=1}^\infty$ in $D$ such that
\begin{itemize}
\item $0 \leq d_n \leq 1, \ \ \forall n$,
\item $d_nd_{n+1} = d_n, \ \ \forall n$, and
\item $\lim_{n \to \infty} d_na =a, \ \ \forall a \in D$.
\end{itemize}
\end{lemma}

\begin{proof} Let $a_0$ be a strictly positive element of $A$. Let $f_n$ be the continuous function $f_n : [0,\infty) \to [0,1]$ such that
$$
f_n(t) = \begin{cases} 0, & \ t \in [0,\frac{1}{n+1}]\\ \text{linear}, & \ t \in \left[\frac{1}{n+1}, \frac{1}{n}\right], \\ 1, & \ t \geq \frac{1}{n} . \end{cases}
$$
Set $d_n := f_n(a_0)$. The first item holds since $0 \leq f_n \leq 1$ and the second because $f_nf_{n+1} = f_n$. To establish the third, assume for a contradiction that there is an element $x \in D$ for which $(1-f_n(a_0))x$ does not converge to $0$. Then 
$$
\sup_{\omega } \omega((1-f_n(a_0))xx^*(1-f_n(a_0))
$$
does not converge to $0$ when we take the supremum over all states $\omega$ of $A$. It follows that there is an $\epsilon > 0$ and a sequence $n_1 < n_2 < n_3 < \cdots$ in $\mathbb N$ such that for each $k$ there is a state $\omega_k$ with
$$
\omega_k((1-f_{n_k}(a_0))xx^*(1-f_{n_k}(a_0))) \geq \epsilon .
$$
Since the unit ball of $A^*$ is weak* compact there is a weak* condensation point $\mu$ of the set of functionals defined by
$$
A\ni a \mapsto  \omega_k((1-f_{n_k}(a_0))a (1-f_{n_k}(a_0))), \  \  k \in \mathbb N.
$$
Then $\mu(xx^*) \geq \epsilon$, and hence $\mu$ is a non-zero positive functional on $A$. However, $\lim_{n \to \infty} (1-f_n(t))t =0$ uniformly for $t$ in the spectrum of $a_0$, and hence $\lim_{n \to \infty} (1-f_n(a_0))a_0 = 0$. It follows that $\mu(a_0) = 0$, contradicting the strict positivity of $a_0$.  
\end{proof}

Consider now a separable $C^*$-subalgebra $D$ of $A$, and let
$$
{\bf d} := \{d_n\}_{n=1}^\infty
$$
be a sequence in $D$ with the properties specified in Lemma \ref{19-11-22fx}. Set
$$
D_{d_n} := \left\{ a \in D : \ ad_n = d_na = a \right\} .
$$
Then $D_{d_n}$ is a $C^*$-subalgebra of $D$ and 
$$
D_{d_n} \subseteq d_nDd_n \subseteq D_{d_{n+1}} \subseteq d_{n+1}Dd_{n+1}.
$$ 
In particular,
$$
\bigcup_{n=1}^\infty d_nDd_n = \bigcup_{n=1}^\infty D_{d_n} .
$$
Set
$$
D({\bf d}) := \bigcup_{n=1}^\infty D_{d_n} .
$$
Then $D({\bf d})$ is a $*$-subalgebra of $D$ and it is dense in $D$ since 
$\lim_{n \to \infty} d_nad_n =a$ 
for all $a\in D$.

\begin{lemma}\label{18-11-22gx} $D({\bf d})\subseteq \mathcal M_\tau$.
\end{lemma}
\begin{proof} Let $n \in \mathbb N$ and consider an element $a \in D^+$. Set $d := d_nad_n$. It suffices to show that $d \in \mathcal M_\tau$. Since $\tau$ is densely defined there is an $x \in A^+ \cap \mathcal M_\tau$ such that $\left\|x -d_{n+1}\right\| \leq \frac{1}{2}$.  
Using $d_{n+1}\sqrt{d} =\sqrt{d}$ we find that
$$
\frac{1}{2} d \leq \sqrt{d}(1+(x-d_{n+1}))\sqrt{d} = \sqrt{d}(d_{n+1}+(x-d_{n+1}))\sqrt{d} = \sqrt{d}x\sqrt{d}.
$$
Note that $\sqrt{d}x\sqrt{d} \in \mathcal M_\tau$ since $\mathcal M_\tau$ is a two-sided ideal by 
Lemma \ref{16-11-22}. It follows therefore from the estimate above that $\tau(d) \leq 2 \tau(\sqrt{d}x\sqrt{d}) < \infty$. Hence $d \in \mathcal M_\tau$.
\end{proof}

\begin{lemma}\label{17-11-22bx} $\tau^\dagger(xy) = \tau^\dagger(yx)$ for all $x,y \in \mathcal M_\tau^\dagger$.
\end{lemma}
\begin{proof} Write $x = (a,\lambda), \ y = (b,\mu)$ where $a,b \in \mathcal M_\tau, \ \lambda,\mu \in \mathbb C$. Then $\tau^\dagger(xy) = \tau|_{\mathcal M_\tau}(ab + \lambda b + \mu a)$ while $\tau^\dagger(yx) = \tau|_{\mathcal M_\tau}(ba + \lambda b + \mu a)$. It suffices therefore to show that $\tau|_{\mathcal M_\tau}(ab) = \tau|_{\mathcal M_\tau}(ba)$, which follows from the polarization identities
$$
ab = \frac{1}{4}\sum_{k=1}^4 i^k(b + i^ka^*)^*(b + i^ka^*)
$$
and
$$
ba = \frac{1}{4}\sum_{k=1}^4 i^k(b + i^ka^*)(b + i^ka^*)^* .
$$
\end{proof}

\begin{cor}\label{20-11-22x}$\tau^\dagger \otimes \Tr_n(xy) = \tau^\dagger \otimes \Tr_n(yx)$ for all $x,y \in M_n(\mathcal M_\tau^\dagger)$.
\end{cor}

\begin{lemma}\label{04-11-22ax} Let $I$ be a right ideal (not necessarily closed) in the $C^*$-algebra $A$. If $x \in M_n(I^\dagger)$ is invertible in $M_n(A^\dagger)$ then $x^{-1} \in M_n(I^\dagger)$.
\end{lemma}
\begin{proof} Write $x = (a,k)$ and $x^{-1} = (b,m)$ where $a \in M_n(I)$, $b \in M_n(A)$ and $k,m \in M_n(\mathbb C)$. Since $xx^{-1} = 1$ we get the equations $ab + kb+am =0 $ and $km =1$, implying that $b = k^{-1}(-am-ab) \in M_n(I)$.
\end{proof}

\begin{lemma}\label{19-11-22c} Let $e,f$ be projections in $M_n(\mathcal M_\tau^\dagger)$, and assume that there is a partial isometry $v \in M_n(A^\dagger)$ such that $e = vv^*$ and $f= v^*v$. Then $\tau^\dagger \otimes \Tr_n(e) = \tau^\dagger \otimes \Tr_n(f)$.
\end{lemma}
\begin{proof} Choose a separable $C^*$-algebra $D \subseteq A$ such that $e,f,v \in M_n(D^\dagger)$ and choose in $D$ a sequence $\{d_n\}_{n=1}^\infty$ with the properties specified in Lemma \ref{19-11-22fx}. For each $\epsilon >0$ we can then find $N \in \mathbb N$ and projections $e',f' \in M_n(D_{d_N}^\dagger)$ such that $\left\| e-e'\right\| \leq \epsilon$ and $\|f-f'\| \leq \epsilon$, cf. Lemma 2.2.7 and Lemma 2.2.8 in \cite{Th5}. Set $u := (2e'-1)(2e-1) + 1$. Note that
\begin{align*}
&\left\|1- \frac{1}{2}u\right\| = \left\|(2e'-1)(e'-e)\right\| \leq \|e'-e\|,  
\end{align*} 
implying that $u$ is invertible if $\epsilon < 1$. Since $ue = 2e'e = e'u$, we see that $ueu^{-1} = e'$. Note that $u \in M_n(\mathcal M_\tau^\dagger)$. It follows therefore from Lemma \ref{04-11-22ax} that $u^{-1} \in M_n(\mathcal M_\tau^\dagger)$. Since $e' \in M_n(\mathcal M_\tau^\dagger)$ by Lemma \ref{18-11-22gx} it follows now from Corollary \ref{20-11-22x} that $\tau^\dagger \otimes \Tr_n(e') = \tau^\dagger \otimes \Tr_n(e)$. We observe that $u\in M_n(D^\dagger)$, which implies that $u \in M_n((\mathcal M_\tau \cap D)^\dagger)$ and hence that $u^{-1} \in M_n((\mathcal M_\tau \cap D)^\dagger)$ by Lemma \ref{04-11-22ax}. Similarly, $\tau^\dagger \otimes \Tr_n(f') = \tau^\dagger \otimes \Tr_n(f)$ and $sfs^{-1} = f'$ for some invertible element $s \in M_n((\mathcal M_\tau \cap D)^\dagger)$ with $s^{-1} \in M_n((\mathcal M_\tau \cap D)^\dagger)$. We aim to show that $\tau^\dagger \otimes \Tr_n(e') = \tau^\dagger \otimes \Tr_n(f')$. Set $x:= f'sv^*u^{-1}e', \ y := e'uvs^{-1}f'$. Then $x,y \in M_n(D^\dagger)$ and
\begin{align*}
&xy = f'sv^*u^{-1}e'uvs^{-1}f' = f'sv^*evs^{-1}f' \\
&=  f'sv^*vv^*vs^{-1}f'=f'sfs^{-1}f' =f'.
\end{align*}
Similarly, $yx =e'$. Since $f' = {f'}^*f' = y^*x^*xy \leq \|x\|^2 y^*y$ it follows that $y^*y$ is positive and invertible in $f'M_n(D^\dagger)f'$. Taking the inverse in that algebra, set 
$$
w:= y(y^*y)^{-1/2} .
$$ 
Then $w^*w = (y^*y)^{-1/2}y^*y(y^*y)^{-1/2} = f'$ and hence, in particular, $ww^*$ is a projection. Note that
\begin{align*}
&e' = yxx^*y^* \leq \|x\|^2 yy^* = \|x\|
^2 y(y^*y)^{-1/2}yy^*(y^*y)^{-1/2}y^* \\
&\leq \|x\|^2\|y\|^2  y(y^*y)^{-1/2}(y^*y)^{-1/2}y^* = \|x\|^2\|y\|^2 ww^* .
\end{align*}
Then 
$$
0 \leq (1-ww^*)e'(1-ww^*) \leq \|x\|^2\|y\|^2(1-ww^*)ww^*(1-ww^*) = 0,
$$ 
implying that $e' = ww^*e'$. On the other hand,
\begin{align*}
& e'y = e'uvs^{-1}f' = ueu^{-1}uvs^{-1}f' = uevs^{-1}f'\\
& = uvv^*vs^{-1}f' = uvs^{-1}f' =   e'uvs^{-1}f' = y,
\end{align*} 
 implying that $e'ww^* = ww^*$, and hence that $e' = ww^*$. Since $w \in M_n(D^\dagger)$ and $D({\bf d})$ is dense in $D$ there is a $k > N$ and an element $z \in M_n(D_{d_k}^\dagger)$ such that $\left\|z{z}^* - e'\right\| < \frac{1}{2}$ and $\left\|{z}^*z - f'\right\| < \frac{1}{2}$. Since $e',f',z \in M_n(D_{d_k}^\dagger)$ it follows then from Lemma 2.2.7 and Lemma 2.2.8 in \cite{Th5} that there is a partial isometry $w' \in M_n(D_{d_k}^\dagger)$ such that $w'{w'}^* = e'$ and ${w'}^*w' = f'$. Since $w' \in M_n(\mathcal M_\tau^\dagger)$ by Lemma \ref{18-11-22gx} it follows from Corollary \ref{20-11-22x} that $ \tau^\dagger \otimes \Tr_n(e') = \tau^\dagger \otimes \Tr_n(f')$. 
\end{proof}

\begin{lemma}\label{20-11-22a} Let $e \in M_n(A^\dagger)$ be a projection. There is a projection $e' \in M_n(\mathcal M_\tau^\dagger)$ and a partial isometry $v \in M_n(A^\dagger)$ such that $vv^*= e$ and $v^*v= e'$.
\end{lemma}
\begin{proof}  Let $D \subseteq A$ be a separable $C^*$-subalgebra such that $e \in M_n(D^\dagger)$ and let $\{d_n\}_{n=1}^\infty$ be a sequence in $D$ with the properties specified in Lemma \ref{19-11-22fx}. Since $D({\bf d})$ is a $*$-algebra which is dense in $D$ and since $M_n(D_{d_k}^\dagger)$ is a $C^*$-algebra it follows from Lemma 2.2.7 and Lemma 2.2.8 in \cite{Th5} that there is a $k$, a projection $e' \in M_n(D_{d_k}^\dagger)$ and a partial isometry $v \in M_n(A^\dagger)$ such that $vv^*= e$ and $v^*v= e'$. This completes the proof because $M_n(D_{d_k}^\dagger) \subseteq M_n(\mathcal M_\tau^\dagger)$ by Lemma \ref{18-11-22g}.

\end{proof}

\emph{Proof of Theorem \ref{19-11-22d}:} The identity 
$$
K_0(A^\dagger) = \left\{ [e] - [f] : \ e,f \in P_\infty(\mathcal M_\tau^\dagger) \right\}
$$ follows from the definition of $K_0(A^\dagger)$ and Lemma \ref{20-11-22a}. To prove that $\tau^\dagger_*$ is well-defined assume that $e,f,e',f'$ are projections in $M_n(\mathcal M_\tau^\dagger)$ such that $[e] - [f] = [e']-[f']$ in $K_0(A^\dagger)$. There is then a projection $r \in P_\infty(A^\dagger)$ such that $e\oplus f' \oplus r$ and $e'\oplus f \oplus r$ are Murray-von Neumann equivalent in $M_N(A^\dagger)$ for some $N \in \mathbb N$. By Lemma \ref{20-11-22a} we may assume that $r \in P_\infty(\mathcal M_\tau^\dagger)$. It follows then from Lemma \ref{19-11-22c} that 
$$
\tau^\dagger \otimes \Tr_N(e\oplus f' \oplus r) = \tau^\dagger \otimes \Tr_N(e'\oplus f \oplus r)
$$ 
and hence by linearity that
$$
\tau^\dagger\otimes \Tr_n(e) -\tau^\dagger\otimes \Tr_n(f) = \tau^\dagger\otimes \Tr_n(e') -\tau^\dagger\otimes \Tr_n(f').
$$

\qed

By definition $K_0(A)$ is a subgroup of $K_0(A^\dagger)$ so the map $\tau^\dagger_*$ of Theorem \ref{19-11-22d} gives rise, by restriction, to a homomorphism 
$$
\tau^\dagger_* : K_0(A) \to \mathbb R
$$
such that
$$
\tau^\dagger_*([e]-[f]) = \tau^\dagger \otimes \Tr_n(e) - \tau^\dagger \otimes \Tr_n(f)
$$
when $e,f \in M_n(\mathcal M_\tau^\dagger)$ and $[e]-[f] \in K_0(A)$. Let $\iota : K_{00}(A) \to K_0(A)$ be the canonical homomorphism. Set
$$
\tau_* := \tau^\dagger_* \circ \iota : \ K_{00}(A) \to \mathbb R .
$$
To obtain more information about $\tau_*$ we need the following lemmas.

For each $k \in \mathbb N$ define $\tau_k : M_k(A)^+ \to [0,\infty]$ such that
$$
\tau_k(a) = \sum_{i=1}^k \tau(a_{kk}) 
$$
when $a = (a_{ij})_{i,j=1}^k \in M_k(A)^+$.

\begin{lemma}\label{20-11-22g} $\tau_k$ is a trace on $M_k(A)$.
\end{lemma}
\begin{proof}  The two first items of Definition \ref{03-02-22f} clearly hold. To check the third note that
\begin{align*}
&\tau_k(a^*a) = \sum_{i=1}^k \tau(\sum_{j=1}^k a_{ji}^*a_{ji}) =  \sum_{i=1}^k \sum_{j=1}^k\tau( a_{ji}^*a_{ji})  \\
&=  \sum_{i=1}^k \sum_{j=1}^k\tau( a_{ji}a_{ji}^*) = \sum_{j=1}^k \tau(\sum_{i=1}^k a_{ji}a_{ji}^*) = \tau_k(aa^*).
\end{align*}
To check the last item of Definition \ref{03-02-22f} let $a \in M_k(A)$. Since $\mathcal M_\tau$ is dense in $A$ we can approximate $a$ with $b \in M_k(\mathcal M_\tau)$. Then $b^*b$ approximates $a^*a$ and
$$ 
\tau_k(b^*b) = \sum_{i=1}^k \tau(\sum_{j=1}^k b_{ji}^*b_{ji}) < \infty .
$$
\end{proof}

\begin{lemma}\label{20-11-22hx} Let $e$ be a projection in $M_k(A)$. Then $e \in M_k(\mathcal M_\tau)$.
\end{lemma}
\begin{proof} We use the trace $\tau_k$ from Lemma \ref{20-11-22g}. Since $\tau_k$ is densely defined there is an $x \in \mathcal M_{\tau_k}$ such that $\left\|e -x \right\| < \frac{1}{2}$. Then
$$
\frac{1}{2} e \leq e\left(1 + (x-e)\right)e = exe \in \mathcal M_{\tau_k} .
$$
Hence $\tau_k(e) \leq 2 \tau_k(exe) < \infty$. Let $e = (e_{ij})_{i,j=1}^k$. Since $\tau_k(e^*e) = \tau_k(e) < \infty$ we find that 
$$
\sum_{i=1}^k \sum_{j=1}^k\tau(e_{ji}^*e_{ji}) = \tau(e^*e) < \infty,
$$
showing that $e_{ji} \in \mathcal N_\tau$ for all $i,j$, in the notation from the proof of Lemma \ref{16-11-22}. Then $e_{ij} = \sum_{l=1}^ke_{il}e_{lj} \in \Span\mathcal N_\tau\mathcal N_\tau = \Span \mathcal N_\tau^*\mathcal N_\tau =\mathcal M_\tau$ for all $i,j$. 

\end{proof}

It follows then from Lemma \ref{20-11-22hx} and Theorem \ref{19-11-22d} that
\begin{equation}\label{20-11-22e}
\tau_* ([e]-[f]) = \tau|_{\mathcal M_\tau} \otimes \Tr_n(e) - \tau|_{\mathcal M_\tau} \otimes \Tr_n(f)
\end{equation}
when $e$ and $f$ are projections in $M_n(A)$. In particular, it follows that $\tau_*$ is a positive homomorphism in the sense that
$$
\tau_*([e]) \geq 0
$$
when $e \in P_\infty(A)$.

\begin{notes} This section is based on ideas of Connes and Elliott in the guise given to them in \cite{Th6}.
\end{notes}

\subsection{Pairing traces with $K_0$ for AF-algebras}

\begin{lemma}\label{18-11-22} Let $A$ be an AF-algebra and $\varphi : K_0(A) \to \mathbb R$ a positive homomorphism. There is a densely defined lower semi-continuous trace $\tau$ on $A$ such that $\varphi = \tau_*$.
\end{lemma}
\begin{proof} Let $p_1 \leq p_2 \leq p_3 \leq \cdots$ be a sequence of projections in $A$ such that $\lim_{n \to \infty} p_n a = a$ for all $a \in A$. Then $p_nAp_n$ is a unital AF-algebra and by Corollary 3.3.24 of \cite{Th5} there is a unique bounded positive trace $\tau_n$ on $p_n Ap_n$ such that ${\tau_n}_* = \varphi \circ {j_n}_*$ on $K_0(p_nAp_n)$, where $j_n : p_nAp_n \to A$ is the inclusion. Let $e = e^* = e^2\in p_nAp_n$. Then
$$
\tau_{n+1}(e) = {\tau_{n+1}}_*([e]) = \varphi\circ {j_{n+1}}_*([e]) = \varphi\circ {j_n}_*([e]) = {\tau_n}_*([e]) = \tau_n(e),
$$
implying that $\tau_{n+1}|_{p_nAp_n} = \tau_n$.  Let $a \in A^+$. Then
\begin{align*}
&\tau_n(p_nap_n) =  \tau_{n+1}(p_nap_n)=  \tau_{n+1}(p_np_{n+1}ap_{n+1}p_n) =  \tau_{n+1}(p_{n+1}ap_{n+1}p_n) \\
&=  \tau_{n+1}(\sqrt{p_{n+1}ap_{n+1}}p_n\sqrt{p_{n+1}ap_{n+1}}) \leq  \tau_{n+1}(p_{n+1}ap_{n+1}).
\end{align*}
Define $\tau : A^+ \to [0,\infty]$ such that 
$$
\tau(a) = \lim_{n\to \infty} \tau_n(p_nap_n) = \sup_n\tau_n(p_nap_n) .
$$ 
Then $\tau$ is lower semi-continuous and 
\begin{align*}
&\tau(a^*a) = \lim_{n \to \infty} \tau_n(p_na^*ap_n) = \lim_{n \to \infty} \lim_{l \to \infty}  \tau_n(p_na^*p_lap_n) \\
& = \lim_{n \to \infty} \lim_{l \to \infty}  \tau_l(p_lp_na^*p_lap_np_l) =   \lim_{n \to \infty} \lim_{l \to \infty}  \tau_l(p_lap_na^*p_l)\\
& = \lim_{n \to \infty}  \tau(ap_na^*) = \tau(aa^*),
\end{align*}
showing that $\tau$ is a trace.
When $e =e^*=e^2 \in p_nAp_n$, $\tau(e) = \tau_n(e) = \varphi([e])$. It follows that $\tau_* = \varphi$.

\end{proof}

\begin{lemma}\label{20-11-22h} Let $\tau_i, \ i =1,2$, be lower semi-continuous traces on the AF-algebra $A$. Assume that ${\tau_1}_* = {\tau_2}_*$ on $K_0(A)$. It follows that $\tau_1 = \tau_2$.
\end{lemma}

\begin{proof} Let $\{p_n\}$ be an approximate unit in $A$ consisting of projections. When $e$ is a projection in $p_nAp_n$, $\tau_1(e) = {\tau_1}_*([e]) = {\tau_2}_*([e]) = \tau_2(e)$ and it follows therefore from Corollary 3.3.24 in \cite{Th5} that $\tau_1|_{p_nAp_n} = \tau_2|_{p_nAp_n}$. Let $a \in A$. Since $\tau_1$ and $\tau_2$ are both lower semi-continuous, $\lim_{n \to \infty} \tau_i(a^*p_na) = \tau_i(a^*a), \ i =1,2$. Hence 
\begin{align*}
&\tau_1(a^*a) = \lim_{n \to \infty} \tau_1(a^*p_na) = \lim_{n \to \infty} \tau_1(p_naa^*p_n) \\
&=  \lim_{n \to \infty} \tau_2(p_naa^*p_n)  = \lim_{n \to \infty} \tau_2(a^*p_na) = \tau_2(a^*a).
\end{align*}
\end{proof}

The two last lemmas combine to give the following
\begin{thm}\label{22-11-22} Let $A$ be an AF-algebra. The map $\tau \mapsto \tau_*$ is a bijection from the set of lower semi-continuous traces $\tau$ on $A$ onto the set $\Hom^+(K_0(A),\mathbb R)$ of positive homomorphisms $K_0(A) \to \mathbb R$.
\end{thm}


\chapter{Simple crossed products of AF-algebras}\label{elliott}

The following is a result of Elliott, \cite{E2}.

\begin{thm}\label{02-12-22} Let $G$ be a discrete group and $\alpha: G \to \Aut A$ an action of $G$ by automorphisms of the AF-algebra $A$. Assume that for all $g \in G \backslash \{e\}$ there are no non-zero order ideals $I$ of $K_0(A)$ such that ${\alpha_g}_*(x) =x$ for all $x \in I$, and that there are no non-trivial order ideals of $K_0(A)$ that are invariant under all ${\alpha_g}_* , g \in G$. Then $A \rtimes_{\alpha,r} G$ is simple.
\end{thm}

For the proof we fix the AF-algebra $A$. We'll say that an automorphism $\gamma \in \Aut A$ is \emph{$K$-outer} when $\gamma_*$ is not the identity map on any non-zero order ideal of $K_0(A)$. Thus one of the assumptions in Theorem \ref{02-12-22} is that $\alpha_g$ is $K$-outer when $g \neq e$.

\begin{lemma}\label{21-12-22} Let $p$ be a projection in $A$ and $x\in K_0(A)^+$ such that $0 \leq x \leq [p]$ in $K_0(A)$. It follows that there is a projection $f \in A$ such that $0 \leq f \leq p$ and $[f] = x$.
\end{lemma}

\begin{proof} By definitions there are projections $q,r$ in matrices over $A$, say $q \in M_n(A), \ r \in M_m(A)$, such that $x = [q]$ and $[q \oplus r] = [p]$ in $K_0(A)$. By (1) of Lemma 2.3.34 of \cite{Th5} there is a partial partial isometry $v \in M_{n+m}(A)$ such that $v^*v = q\oplus r$ and $vv^* = p \oplus 0_{n+m-1}$. Then $f := v(q \oplus 0_m)v^*$ is a projection in $A$ such that $[f] =x$ and $f \leq p$.  
\end{proof}

\begin{lemma}\label{03-12-22x} Let $F$ be a finite set of $K$-outer automorphisms of $A$. Let $p$ be a non-zero projection in $A$ and let $\epsilon >0$. There is a non-zero projection $f \in A$ such that $f \leq p$ and $\left\|f\gamma(f)\right\| \leq \epsilon$ for all $\gamma\in F$.
\end{lemma}

\begin{proof} Let $\gamma \in F$. We first prove that ${\gamma}_*([p_0]) \neq [p_0]$ for some projection $p_0 \in pAp$. To this end set $I^+ = \left\{ x\in K_0(A)^+ : \ x \leq n[p] \ \text{for some} \ n \in \mathbb N\right\}$. Then $I := I^+-I^+$ is a non-zero order ideal in $K_0(A)$ and thanks to the assumptions it follows that there is an element $x \in I^+$ such that ${\gamma}_*(x) \neq x$. Since $(K_0(A),K_0(A)^+)$ has the Riesz decomposition property by Proposition 2.3.41 of \cite{Th5} and since $x \leq n [p]$ for some $n \in \mathbb N$ it follows that $x = x_1 + x_2 + \cdots + x_n$ where $0 \leq x_i \leq [p]$ for each $i$. By Lemma \ref{21-12-22} there are then projections $p_1,p_2,\cdots, p_n$ in $A$ such that $p_i \leq p$ and $x_i = [p_i]$ for all $i$. For at least for one $i$ we have that ${\gamma}_*([p_i]) \neq [p_i]$. By taking $p_0:= p_i$ we have established the claim .

Since $-p_0 \leq \gamma(p_0)-p_0 \leq \gamma(p_0)$ we have that $\left\|\gamma(p_0) - p_0\right\| \leq 1$. Since ${\gamma}_*([p_0]) \neq [p_0]$ this implies that $\left\|\gamma(p_0) - p_0\right\| = 1$. Indeed, if not it follows from Exercise 2.2.19 in \cite{Th5} (or Proposition 4.6.6 in \cite{B}) that $\left[\gamma(p_0)\right] = [p_0]$ in $K_0(A)$; a contradiction. Thus  $\left\|\gamma(p_0) - p_0\right\| = 1$ as asserted. Let $\delta > 0$. Since $A$ is AF there is a finite dimensional $C^*$-subalgebra $B \subseteq A$ and projections $f,h  \in B$ such that $\| f - p_0 \| \leq \delta$ and $\|h- \gamma(p_0)\| \leq \delta$. This follows for example by use of Lemma 2.2.7 of \cite{Th5}. Then $\|f-h\| \geq 1 -  2\delta$. Since $B$ is finite dimensional it is a direct sum of matrix algebras and it follows that there is a central projection $z \in B$ such that $zB \simeq M_k(\mathbb C)$ for some $k \in \mathbb N$ and $\left\|zf-zh\right\| = \|f-h\| \geq 1 -  2\delta$. We identify $zB = M_k(\mathbb C)$ and let $zB$ act on the Hilbert space $\mathbb C^k$ in the standard way. Then
$$
-1 \leq -\left<  zh\psi,\psi\right> \leq \left<(zf - zh)\psi,\psi\right> \leq \left< zf\psi,\psi\right> \leq 1
$$
when $\psi \in \mathbb C^k$ and $\|\psi\| = 1$. Since $\left\|zf-zh\right\| \geq 1 -  2\delta$ there is a $\psi \in \mathbb C^k$ such that $\|\psi\| = 1$ and either $1-3 \delta \leq  \left<(zf - zh)\psi,\psi\right> \leq 1$ or $-1 \leq \left<(zf - zh)\psi,\psi\right>  \leq -1 + 3 \delta$.  In the first case
\begin{itemize}
\item[$\bullet$] $1-3\delta \leq \left<  zf\psi,\psi\right> \leq 1 $, and
\item[$\bullet$] $\left<zh\psi,\psi\right> \leq 3 \delta$
\end{itemize}
and in the second case
\begin{itemize}
\item[$\bullet$] $-1 \leq -\left<  zh\psi,\psi\right> \leq -1 + 3 \delta$, and
\item[$\bullet$] $\left<zf\psi,\psi\right> \leq 3 \delta$.
\end{itemize}
In the first case we consider the one dimensional projections $f_1$ and $f_1'$ onto $\mathbb C zf\psi$ and $\mathbb C \psi$, respectively. Note that $f_1 \leq zf$. Furthermore, since
$$
\left\|zf\psi - \psi\right\|^2 \leq 2 - 2\left<zf\psi,\psi\right> \leq 6 \delta
$$
we find that
\begin{align*}
&\left\| f_1 - f_1'\right\| \leq 2 \left\| \frac{zf\psi}{\|zf\psi\|} - \psi \right\| \\
& \leq 2 \left(  \left\| \frac{zf\psi}{\|zf\psi\|} - zf \psi \right\| + \left\|zf\psi - \psi\right\|\right) \\
& \leq 2\left( \left| \frac{1}{\|zf\psi\|} -1 \right| + \sqrt{6\delta}\right) \leq 2 \left( \frac{3\delta}{1-3\delta} + \sqrt{6\delta}\right) := \delta' .
\end{align*}
 For any $\phi \in \mathbb C^k$, $\|\phi\| \leq 1$, we have then that
$$
\left< f_1zhf_1\phi,\phi\right> \leq  \left<f'_1zhf'_1\phi,\phi\right> + 2 \delta' \leq \left<zh\psi,\psi\right> + 2 \delta ' \leq 3 \delta + 2 \delta' ,
$$
proving that $\left\|f_1h\right\| = \left\|f_1zh\right\| \leq 3 \delta + 2 \delta'$. And hence 
$$
\left\|f_1 \gamma(p_0) \right\| \leq  4 \delta + 2 \delta'.
$$ 
Since $f_1 \leq f$ and $\left\|f -p_0\right\| \leq \delta$ it follows that 
$$
\left\|(p_0f_1p_0)^2 - p_0f_1p_0\right\| \leq \|f_1p_0f_1 - f_1\| \leq \delta.
$$ 
Provided $\delta < \frac{1}{4}$ it follows from Lemma 2.2.7 in \cite{Th5} that there is a projection $q \leq p_0$ such that $\|q-p_0f_1p_0\| \leq  2\delta$ and hence $\|q-f_1\| \leq 4\delta$ . In particular, if we choose $\delta$ smaller than $1/4$, $q$ will be non-zero since $f_1 \neq 0$. Furthermore, $\left\|q\gamma(p_0)\right\| \leq 8\delta + 2 \delta'$ and hence
 $$
\left\|q\gamma(q)\right\| = \sqrt{\left\| q \gamma(q)q\right\|} \leq  \sqrt{\left\| q \gamma(p_0)q\right\|} = \left\|q\gamma(p_0)\right\| \leq  8\delta + 2 \delta' .
$$

In the second case we consider the one dimensional projections $h_1$ and $h_1'$ onto $\mathbb C zh\psi$ and $\mathbb C \psi$, respectively. Arguments identical to the preceding show that $h_1 \leq zh$ and $\left\|h_1-h_1'\right\| \leq \delta'$ while $\left\|h_1f\right\| \leq 3 \delta + 2 \delta'$. Using that $\left\|h -\gamma(p_0)\right\|\leq \delta$ and $h_1 \leq h$ we get also a projection $r \leq \gamma(p_0)$ such that $\left\| r -h_1\right\| \leq 4\delta$. In this case we set $q:= \gamma^{-1}(r)$. 
Then $q \leq p_0$ and
\begin{align*}
&\|rp_0\| \leq \|rf\| + \delta \leq \|h_1 f\| + 5\delta \leq 8\delta + 2\delta'.
\end{align*}
Hence $\left\|q\gamma(q)\right\| = \sqrt{\|\gamma(q)q\gamma(q)\|} \leq \sqrt{\|\gamma(q)p_0\gamma(q)\|} = \left\|rp_0\right\|\leq 8\delta + 2\delta'$, as in the first case.

Now take another element $\gamma' \in F$. From what we have just shown, with $p$ replaced by $q$, we get a non-zero projection $q' \leq q$ such that $\left\|q'\gamma'(q')\right\| \leq  8\delta + 2\delta '$. Since
\begin{align*}
&\left\|q'\gamma(q')\right\| = \sqrt{\left\|q'\gamma(q')q'\right\|} \leq \sqrt{\left\|q'\gamma(q)q'\right\|} = \left\|q'\gamma(q)\right\| \\
&= \sqrt{\left\|\gamma(q)q'\gamma(q)\right\|} \leq \sqrt{\left\|\gamma(q)q\gamma(q)\right\|} = \left\|q\gamma(q)\right\|  \leq 8\delta + 2\delta '  ,
\end{align*}
we can continue this reasoning until the finite set $F$ is exhausted. This completes the proof provided only that $\delta >0$ is chosen such that $8\delta + 2\delta ' \leq \epsilon$.

\end{proof}

\begin{lemma}\label{05-12-22x} Let $F$ be a finite set of $K$-outer automorphisms of $A$. Let $B$ be a finite dimensional $C^*$-subalgebra of $A$ and $e \in B$ a non-zero central projection in $B$. Set $B' := \left\{a \in A : \ ab =b a \ \forall b \in B\right\}$ and let $\epsilon >0$. There is a non-zero projection $f \in eB'$
such that $\left\|f \gamma(f)\right\| \leq \epsilon$ for all $\gamma$.
\end{lemma}
\begin{proof} Let $e_0$ be a minimal non-zero central projection  in $B$ such that $e_0 \leq e$. Then $e_0B' \subseteq eB'$ and it suffices therefore to find $f$ when $ e=e_0$, i.e. we may assume that $e$ is a minimal central projection in $B$. This means that $eB$ is a full matrix algebra, say $eB \simeq M_k(\mathbb C)$. Then $eB$ is generated by $k$ mutually orthogonal projections $e_1,e_2, \cdots, e_k$, and a unitary $u$ with the properties that $ue_iu^* = e_{i+1}$ for $i \leq k-1$, $ue_ku^* = e_1$ and $u^k = e$. Set $w := u + (1-e)$, which is a unitary in the algebra $A^\dagger$ obtained from $A$ by adjoining a unit. Then $w^k =1$, and we set
$$
F_1 := \left\{ \Ad w^i \circ \gamma: \ \gamma \in F, \ i = 1,2, \cdots, k \right\}. 
$$
This is a finite set of $K$-outer automorphisms of $A$. Set $\delta := k^{-2}\epsilon$. It follows from Lemma \ref{03-12-22x} that there is non-zero projection $f_1 \in A$ such that $f_1 \leq e_1$ and $\left\|f_1 \gamma(f_1)\right\| \leq \delta$ for all $\gamma \in F_1$. There is also a non-zero projection $f_2 \leq e_2$ such that $\left\|f_2 \gamma(f_2)\right\| \leq \delta$ for all $\gamma \in F_1$ and a non-zero projection $f_1' \leq u^*f_2u$ such that $\left\|f_1' \gamma(f_1')\right\| \leq \delta$ for all $\gamma \in F_1$. Exchange $f_1$ with $f_1'$ and $f_2$ with $uf_1'u^*$. Then $0 \neq f_i \leq e_i$, $\left\|f_i\gamma(f_i)\right\| \leq \delta, \ i = 1,2$, for all $\gamma \in F_1$, and $uf_1u^* = f_2$. By Lemma \ref{03-12-22x} there is a non-zero projection $f_3' \leq uf_2u^*$ such that $\left\|f_3'\gamma(f_3')\right\| \leq \delta$ for all $\gamma \in F_1$, and there is a non-zero projection $f_1'' \leq {u^*}^2f_3'u^2$ such that $\left\|f_1''\gamma(f_1'')\right\| \leq \delta$ for all $\gamma \in F_1$ . Exchange $f_1$ with $f_1''$, $f_2$ with $uf_1''u^*$ and set $f_3 = u^2f_1''{u^2}^*$. Then $0 \neq f_i \leq e_i$, $\left\|f_i\gamma(f_i)\right\| \leq \delta, \ i = 1,2,3$, for all $\gamma \in F_1$, and $uf_iu^* = f_{i+1}, \ i = 1,2$. This process can be continued and successfully completed because $u^k =e$ to get projections $f_i, i = 1,2,\cdots, k$, such that $0 \neq f_i \leq e_i$ and $\left\|f_i\gamma(f_i)\right\| \leq \delta, \ i = 1,2,3, \cdots k$, for all $\gamma \in F_1$, and $uf_iu^* = f_{i+1}, \ i = 1,2, \cdots, k-1$, while $uf_ku^* = f_1$. Let $i,j \in \left\{1,2,3,\cdots, k\right\}$ and $\gamma \in F$. There are then $x,y \in \{1,2,\cdots, k\}$ such that $w^xf_j{w^x}^* = f_i$, $w^y = {w^x}^*$, and then
$$
\left\|f_i\gamma(f_j)\right\| = \left\|w^xf_j{w^x}^*\gamma(f_j)\right\| = \left\|f_j{w^x}^*\gamma(f_j)w^x\right\| = \left\|f_jw^y\gamma(f_j){w^y}^*\right\| \leq \delta
$$
since $\Ad w^y \circ \gamma \in F_1$. Set $f=\sum_{i=1}^k f_i$. Then $0 \neq f \leq e$ and 
$$
\left\|f\gamma(f)\right\| \leq \sum_{i,j=1}^k\left\|f_i\gamma(f_j)\right\| \leq k^2 \delta \leq \epsilon.
$$
By construction $f \leq e$ and $f$ commutes with each $e_i$ and with $u$; thus $f \in eB'$.
\end{proof}

\begin{lemma}\label{05-12-22a} Fix $a \in A$ and let $S$ be a finite subset of $A$ and $F$ be a finite set of $K$-outer automorphisms of $A$. For each $\epsilon > 0$ there is a projection $f \in A$ such that
\begin{itemize}
\item[$\bullet$] $\left\|faf\right\| \geq \|a\| - \epsilon$,
\item[$\bullet$] $\left\|fb-bf\right\| \leq \epsilon$ for all $b \in S$, and
\item[$\bullet$] $\left\|f\gamma(f)\right\| \leq \epsilon$ for all $\gamma \in F$.
\end{itemize}
\end{lemma}
\begin{proof} Set $\delta := \frac{\epsilon}{2}$. Since $A$ is AF there is a finite dimensional $C^*$-subalgebra $B \subseteq A$, an element $b \in B$ and a subset $T \subseteq B$ such that $\left\|a-b\right\| \leq \delta$ and $\dist(s,T) \leq \delta$ for all $s \in S$. There is a minimal non-zero central projection $e$ in $B$ such that $\|eb\| = \|b\|$. By Lemma \ref{05-12-22x} there is a non-zero projection $f \in eB'$ such that $\left\|f\gamma(f)\right\| \leq \delta$ for all $\gamma \in F$. Since $eB$ is a full matrix algebra commuting with $f$, the map $eB \ni d \mapsto df$ is an injective $*$-homomorphism and hence an isometry. Therefore $\left\|fbf\right\| = \left\|febf\right\| = \left\|ebf\right\| = \left\|eb\right\| = \left\|b\right\|$. It follows that $\|faf\| \geq \left\|fbf\right\| - \delta = \left\|b\right\| - \delta \geq \|a\| -2 \delta$. Similarly, since $f \in B'$ commutes with the elements of $T$ we have that $\left\|fs-sf\right\| \leq 2 \delta$ for all $s \in S$. 
\end{proof}

We are now ready for the \emph{proof of Theorem \ref{02-12-22}:} Let $J$ be an ideal in $A \rtimes_{\alpha,r} G$. Then $J \cap A$ is an $\alpha$-invariant ideal in $A$ and by Theorem 4.1.8 of \cite{Th5} it corresponds to an order ideal $I$ in $K_0(A)$ which is left globally invariant by ${\alpha_g}_*$ for all $g \in G$. By assumption this ideal $I$ is either $K_0(A)$ or $\{0\}$. If $I = K_0(A)$ it follows from Theorem 4.1.8 of \cite{Th5} that $A \cap J = A$. Since $A$ contains an approximate unit for $A \rtimes_{\alpha,r} G$ by Lemma \ref{08-09-23}, this implies that $J= A \rtimes_{\alpha,r} G$. The proof will be completed by showing that the second possibility, $I = \{0\}$ implies that $J = \{0\}$. So assume that $I = \{0\}$. By Theorem 4.1.8 of \cite{Th5} this means that $J \cap A = \{0\}$. Let $j \in J$ and let $\delta > 0$. By definition of $A \rtimes_{\alpha,r} G$ there is a finitely supported function $F : G  \to A$ such that 
\begin{equation}\label{16-11-23}
\left\|j - F(e)  - \sum_{g \neq e} F(g)u_g\right\| \leq \delta.
\end{equation} 
Let $P : A \rtimes_{\alpha,r} G \to A$ be the canonical conditional expectation, cf. Section \ref{canonical}. Since $P$ is of norm $1$ and annihilates $F(g)u_g$ when $g \neq e$, we have that $\left\|P(j) - F(e)\right\| \leq \epsilon$. Choose $\epsilon > 0$ such that 
$$
 \epsilon \sum_{g \neq e} \left\|F(g)\right\|  + (n+1)\epsilon \leq \delta,
 $$  
 where $n$ is the number of non-zero elements of $F(G)$. By assumption $\alpha_g$ is K-outer when $g\neq e$, and it follows therefore from Lemma \ref{05-12-22a} that there is a projection $f \in A$ such that $\left\|fF(e)f\right\| \geq \left\|F(e)\right\| - \epsilon$, $\left\| fb-bf\right\| \leq \epsilon$ for all $b \in F(G)$ and $\left\|f \alpha_g(f)\right\| \leq \epsilon$ when $g \neq e$. Let $\pi : A \rtimes_{\alpha,r} G \to (A \rtimes_{\alpha,r} G)/J$ be the quotient map and note that
$$
\left\|\pi\left( \sum_{g \in G} F(g) u_g\right)\right\| \leq \delta ,
$$
thanks to \eqref{16-11-23}.
Since $\pi$ is injective and hence isometric on $A$ because $A \cap J = \{0\}$, it follows that $\left\|\pi(fF(e)f)\right\| = \left\|fF(e)f\right\|$. Hence
\begin{align*}
&\left\|F(e)\right\| \leq \left\|fF(e)f\right\| + \epsilon =  \left\|\pi(fF(e)f)\right\| + \epsilon  \\
& \leq \left\|\pi(f(F(e) + \sum_{g \neq e} F(g)u_g)f)\right\| +\left\|\pi(f(\sum_{g \neq e} F(g)u_g)f)\right\| + \epsilon  \\
& \leq \|\pi(fjf)\| +\left\|\pi(f(\sum_{g \neq e} F(g)u_g)f)\right\| + \delta + \epsilon\\
& \leq  \left\|\sum_{g \neq e} fF(g)u_gf\right\|  + \delta + \epsilon \\
& \leq \left\|\sum_{g \neq e} F(g)fu_gf\right\|  + \delta + \epsilon + n\epsilon \\
& \leq \sum_{g \neq e} \left\|F(g)\right\| \left\|f\alpha_g(f)u_g\right\|  + \delta + (n+1)\epsilon \\
& \leq \epsilon \sum_{g \neq e} \left\|F(g)\right\|  + \delta +(n+1)\epsilon \leq 2\delta.
\end{align*}
Since $\left\|P(j)\right\| \leq  \left\|F(e)\right\| + \delta$ we conclude first that $\left\|P(j)\right\| \leq 3 \delta$, and then since $\delta > 0$ was arbitrary, that $P(j) =0$. Thus $P(J) = \{0\}$ and as $P$ is faithful (by Lemma \ref{31-08-23}) it follows that $J= \{0\}$. \ $\qed$

\begin{notes}\label{06-12-22x} \textnormal{This section is based almost entirely on Elliotts paper \cite{E2}, but the statement of the main result Theorem \ref{02-12-22} differs from that in \cite{E2}. The reason is that we here want a version where the assumptions in the theorem only refer to K-theory data since this fits better with the way the theorem is used in the proof of Theorem \ref{12-11-22}.} 
\end{notes}

\end{appendices}





\begin{thebibliography}{WWWWW} 


\bibitem[Br]{Br} O. Bratteli, {\em Inductive limits of finite dimensional $C^*$-algebras}, Trans. Amer. Math. Soc. {\bf 171} (1972), 195-234.



\bibitem[B]{B} B. Blackadar, {\em K-Theory for operator algebras}, MSRI Publ. 5, Springer Verlag, 1986. 


\bibitem[BEH]{BEH} O. Bratteli, G. A. Elliott and R.H. Herman, {\em On the possible temperatures of a dynamical system}, Comm. Math. Phys. {\bf 74} (1980), 281--295. 

\bibitem[BEK1]{BEK1} O. Bratteli, G. A. Elliott and A. Kishimoto, {\em The temperature state space of a dynamical system I}, J. Yokohama Univ. {\bf 28} (1980), 125--167. 

\bibitem[BEK2]{BEK2} O. Bratteli, G. A. Elliott and A. Kishimoto, {\em The temperature state space of a dynamical system II }, Ann. of Math. {\bf 123} (1986), 205--263. 

\bibitem[BR]{BR} O. Bratteli and D.W. Robinson, {\em Operator Algebras and Quantum Statistical Mechanics I + II}, Texts and Monographs in Physics, Springer Verlag, New York, Heidelberg, Berlin, 1979 and 1981.


\bibitem[BuR]{BuR} D. Bucholz and J. Roberts, {\em Bounded perturbations of dynamics}, Comm. Math. Phys. {\bf 49} (1976), 161-177.


  
 
 

  
    
     



\bibitem[Ch]{Ch} J. Christensen, {\em The structure of KMS weights on \'etale groupoid $C^*$-algebras}, J. Noncommut. Geom. {\bf 17} (2023), no. 2, 663--691.



\bibitem[C1]{C1} F. Combes, {\em Poids sur un $C^*$-alg\`ebre}, J. Math. Pures Appl. {\bf 47} (1968), 57-100.

\bibitem[C2]{C2} F. Combes, {\em Poids associ\'e \`a une alg\`ebre   hilbertienne \`a gauche}, Compos. Math. {\bf 23} (1971), 49--77.

\bibitem[Co1]{Co1} A. Connes, {\em Une classification des facteurs de type III}, Ann. Sci. Ecole Norm. Sup. {\bf 6}
(1973), 133--252.



\bibitem[Co2]{Co2} A. Connes, {\em Noncommutative Geometry}, Academic Press, New York (1994).

\bibitem[CP1]{CP1} J. Cuntz and G. K. Pedersen, {\em Equivalence and traces on $C^*$-algebras}, J. Func. Analysis {\bf 33} (1979), 135-164. 

\bibitem[CP2]{CP2} J. Cuntz and G. K. Pedersen, {\em Equivalence and KMS states for periodic $C^*$-dynamical systems}, J. Func. Analysis {\bf 34} (1979), 79-86. 

\bibitem[E1]{E1} G. A. Elliott, {\em On the classification of inductive limits of sequences of semisimple finite-dimensional algebras}, J. Algebra {\bf 38} (1976), 29--44.

\bibitem[E2]{E2} G.A. Elliott, {\em Some simple $C^*$-algebras constructed as crossed products with discrete outer automorphism groups}, Publ. RIMS, Kyoto Univ. {\bf
16} (1980), 299-311.




\bibitem[ET]{ET} G.A. Elliott and K. Thomsen, {\em The bundle of KMS state spaces for flows on a unital C*-algebra}, C. R. Math. Acad. Sci. Soc. R. Can. {\bf 43} (2022), 103--121, arXiv:2109.06464

 \bibitem[EST]{EST} G. A. Elliott, Y. Sato, K. Thomsen, {\em On the bundle of KMS state spaces for flows on a Z-absorbing $C^*$-algebra},  Comm. Math. Phys. {\bf 393} (2022), 1105--1123.
 
 


\bibitem[EHS]{EHS} E.G. Effros, D.E. Handelman and C.-L. Shen, {\em Dimension groups and their affine representations}, Amer. J. Math. {\bf 102} (1980), 385--407.





















 
 

 



\bibitem[Ha]{Ha} U. Haagerup, {\em Normal weights on $W^*$-algebras},  J. Funct. Analysis {\bf 19} (1975), 302-317.  

\bibitem[KR]{KR} R. V. Kadison and J. R. Ringrose, {\em Fundamentals of the theory of operator algebras I+II}, Academic Press 1986.


\bibitem[KiR]{KiR} A. Kishimoto and D.W. Robinson, {\em Derivations, Dynamical Systems and Spectral Restrictions}, Math. Scand. {\bf 56} (1985), 83-95.

 \bibitem[Ki1]{Ki1} A. Kishimoto, {\em Locally representable   one-parameter automorphism groups of AF algebras and KMS states},  Rep. Math. Phys. {\bf 45} (2000), 333-356. 
 
\bibitem[Ki2]{Ki2} A. Kishimoto, {\em Non-commutative shifts and crossed products}, J. Func. Analysis {\bf 200} (2003), 281--300.
 
 

\bibitem[KK]{KK} A. Kishimoto and A. Kumjian, {\em Simple stably projectionless $C^*$-algebras arising as crossed products}, Can. J. Math. {\bf 48} (1996), 980-996.




  
  
  
 


 
 
    
  
    
    


\bibitem[Ku1]{Ku1} J. Kustermans, {\em KMS-weights on $C^*$-algebras},  arXiv: 9704008v1.
  
  \bibitem[Ku2]{Ku2} J. Kustermans, {\em One-parameter representations on $C^*$-algebras}, arXiv:funct-an/9707009v1  
  

\bibitem[KV1]{KV1} J. Kustermans and S. Vaes, {\em Weight theory for $C^*$-algebraic quantum groups}, arXiv:math/9901063  
  
 
\bibitem[LN]{LN} M. Laca and S. Neshveyev, {\em KMS states of quasi-free dynamics on Pimsner algebras}, J. Funct. Analysis {\bf 211} (2004), 457--482.  
  

\bibitem[MS]{MS} H. Matui and Y. Sato, {\em Decomposition rank of UHF-absorbing $C^*$-algebras}, Duke Math. J. {\bf 163} (2014), 2687--2708.  x
  
 

 
\bibitem[OP]{OP} D. Olesen and G.K. Pedersen, {\em Some $C^*$-dynamical systems with a single KMS state}, Math. Scand. {\bf 42} (1978), 111-118.
 
 
 \bibitem[VP]{VP} N. V. Pedersen, {\em On certain KMS-weights on $C^*$-crossed products}, Proc. London Math. Soc. {\bf 44} (1982), 445--472.
  
\bibitem[Pe]{Pe} G. K. Pedersen, {\em $C^*$-algebras and their automorphism groups}, Academic Press, London, 1979.


\bibitem[PT]{PT} G. K. Pedersen and M. Takesaki, {\em A Radon-Nikodym theorem for weights on von Neumann algebras}, Acta. Math. {\bf 130} (1973), 53-87.


\bibitem[PWY]{PWY} C. Pinzari, Y. Watatani and K. Yonetani, {\em KMS States, Entropy and the Variational Principle
in Full $C^*$-Dynamical Systems}, Commun. Math. Phys. {\bf 213} (2000), 331--379.

  
\bibitem[PS]{PS} R.T. Powers and S. Sakai, {\em Existence of ground states and KMS states for approximately inner dynamics},  Comm. Math. Phys. {\bf 39} (1975), 273-288.
 
 \bibitem[QV]{QV} J. Quaegebeur and J. Verding, {\em A construction for weights on $C^*$-algebras: Dual weights for $C^*$-crossed products}, Intern. J. Math. {\bf 10} (1999), 129--157.

 
\bibitem[RS]{RS} M. Reed and B. Simon, {\em Methods of Modern Mathematical Physics}, Academic Press, New York, San Francisco, London, 1972. 






\bibitem[Ru0]{Ru0} W. Rudin, \it Real and complex analysis \rm, McGraw-Hill, New Delhi, 1967.

\bibitem[Ru1]{Ru1} W. Rudin, \it Principles of Mathematical Analysis \rm, McGraw-Hill, New Delhi, 1976.

\bibitem[Ru2]{Ru2} W. Rudin, \it Functional Analysis \rm, McGraw-Hill, New Delhi, 1973.


\bibitem[Ru3]{Ru3} W. Rudin, {\em Fourier analysis on groups}, Interscience Publishers, New York, London, 1962.

 










\bibitem[SZ]{SZ} S. Str\v{a}til\v{a} and L. Zsid\'{o}, {\em Lectures on von Neumann algebras}, Revision of the 1975 original. Editura Academiei, Bucharest; Abacus Press, Tunbridge Wells, 1979.


\bibitem[Ta1]{Ta1} M. Takesaki, {\em Duality for crossed products and the structure of von Neumann algebras of type III}, Acta Math. {\bf 131} (1973), 249-310.


\bibitem[Ta2]{Ta2} M. Takesaki, {\em
Theory of operator algebras. II} ,
Encyclopaedia of Mathematical Sciences, 125. Operator Algebras and Non-commutative Geometry, 6. Springer-Verlag, Berlin, 2003.


\bibitem[Th1]{Th1} K. Thomsen, {\em KMS weights on graph $C^*$-algebras}, Adv. Math. {\bf 309} (2017) 334--391.



\bibitem[Th2]{Th2} K. Thomsen, {\em Phase transition in the CAR algebra}, Adv. Math. {\bf 372} (2020); arXiv:1810.01828 


\bibitem[Th3]{Th3} K. Thomsen, {\em The possible temperatures for flows on a simple AF algebra}, Comm. Math. Phys. {\bf 386} (2021), 1489--1518. 


\bibitem[Th4]{Th4} K. Thomsen, {\em KMS weights, conformal measures and ends in digraphs}, Adv. Oper. Th. {\bf 5} (2020), 489--607.


\bibitem[Th5]{Th5} K. Thomsen, {\em An introduction to $C^*$-algebras}, A book project in slow progress, preliminary version available at https://pure.au.dk/portal/files/300092930/firstbook.pdf


\bibitem[Th6]{Th6} K. Thomsen, {\em On weights, traces and K-theory}, Preprint, 2022, arXiv:2211.03172 

\bibitem[To]{To} D. M. Topping, {\em Lectures on Von Neumann Algebras}, Van Nostrand Reinhold Company, London, 1971.

\bibitem[Wi]{Wi} D.P. Williams: {\em Crossed products of $C^*$-algebras}, Mathematical Surveys and Monographs, Vol. 134, American Mathematical Society, 2007.






 






























\end{thebibliography}
\end{document}